\documentclass[oneside,a4paper,english, 11pt]{amsart}

\usepackage{adjustbox}
\usepackage[latin9]{inputenc}
\usepackage{hyperref}
\usepackage{lscape}
\usepackage{slashbox}
\usepackage{mathrsfs}  
\usepackage{enumerate}  
\usepackage{xcolor}
\usepackage{bm}
\usepackage{tikz-cd}
\usetikzlibrary{arrows}
\usetikzlibrary{decorations,decorations.text,decorations.markings,decorations.shapes}
\tikzset{ closed/.style = {decoration = {markings, mark = at position 0.5 with { \node[transform shape, xscale = .8, yscale=.4] {/}; } }, postaction = {decorate} },
open/.style = {decoration = {markings, mark = at position .5 with { \node[transform shape, scale =1.2] {$\circ$}; } }, postaction = {decorate} }
}

\makeatletter

 \theoremstyle{plain}
\newtheorem{thm}{Theorem}[subsection]
\theoremstyle{plain}
  \newtheorem{prop}[thm]{Proposition}
\theoremstyle{plain}
 \newtheorem{lemma}[thm]{Lemma}
\theoremstyle{plain}

\theoremstyle{plain}
\newtheorem{cor}[thm]{Corollary}
\newtheorem{conj}[thm]{Conjecture}
\theoremstyle{definition}
  \newtheorem{defn}[thm]{Definition}
  \theoremstyle{definition}
  \newtheorem{setup}[thm]{Setup}
\theoremstyle{definition}
  
 \theoremstyle{definition}
  
  \newtheorem{exam}[thm]{Example}
\theoremstyle{remark}
\newtheorem{rmk}[thm]{Remark}
\theoremstyle{definition}
\newtheorem{warning}[thm]{Warning}
\numberwithin{equation}{section}

\usepackage{amsmath}
\usepackage{amssymb}
\usepackage{amscd}
\usepackage{amsthm}
\usepackage{array,multirow}
\usepackage[arrow,matrix,tips,all,cmtip,poly]{xy}

\usepackage{stmaryrd}
\usepackage{url}
\usepackage{color}
\usepackage{caption}
\usepackage{float}

\pdfpageheight\paperheight
\pdfpagewidth\paperwidth
\newcommand{\Z}{\mathbb{Z}}

\newcommand{\Q}{\mathbb{Q}}

\newcommand{\Qp}{\mathbb{Q}_p}

\newcommand{\R}{\mathbb{R}}

\newcommand{\bG}{\mathbb{G}}

\newcommand{\F}{\mathbb{F}}

\newcommand{\N}{\mathbb{N}}

\newcommand{\A}{\mathbb{A}}

\newcommand{\fA}{\mathfrak{a}}

\newcommand{\fF}{\mathfrak{f}}

\newcommand{\fM}{\mathfrak{M}}

\newcommand{\fp}{\mathfrak{p}}

\newcommand{\fS}{\mathfrak{S}}

\newcommand{\fm}{\mathfrak{m}}

\newcommand{\bA}{\mathbb{A}}

\newcommand{\bD}{\mathbb{D}}

\newcommand{\bF}{\mathbb{F}}

\newcommand{\bQ}{\mathbb{Q}}

\newcommand{\bT}{\mathbb{T}}

\newcommand{\bV}{\mathbb{V}}

\newcommand{\bZ}{\mathbb{Z}}

\newcommand{\cA}{\mathcal{A}}

\newcommand{\cC}{\mathcal{C}}

\newcommand{\cE}{\mathcal{E}}
\newcommand{\cF}{\mathcal{F}}
\newcommand{\cG}{\mathcal{G}}

\newcommand{\cI}{\mathcal{I}}

\newcommand{\cJ}{{\mathcal{J}}}
\newcommand{\cK}{\mathcal{K}}

\newcommand{\cM}{\mathcal{M}}

\newcommand{\cO}{\mathcal{O}}
\newcommand{\cP}{\mathcal{P}}

\newcommand{\cR}{\mathcal{R}}
\newcommand{\cS}{{\mathcal{S}}}

\newcommand{\cU}{\mathcal{U}}

\newcommand{\cX}{\mathcal{X}}
\newcommand{\cY}{\mathcal{Y}}
\newcommand{\cZ}{\mathcal{Z}}

\newcommand{\eps}{\varepsilon}

\newcommand{\phz}{\varphi}

\newcommand{\La}{\Lambda}

\newcommand{\Zp}{\mathbb{Z}_p}

\newcommand{\Id}{\mathrm{id}}
\newcommand{\Gal}{\mathrm{Gal}}
\newcommand{\Hom}{\mathrm{Hom}}

\newcommand{\Res}{\mathrm{Res}}
\newcommand{\Ind}{\mathrm{Ind}}

\newcommand{\End}{\mathrm{End}}
\newcommand{\Aut}{\mathrm{Aut}}

\newcommand{\GL}{\mathrm{GL}}

\newcommand{\ad}{\mathrm{ad}}

\newcommand{\red}{\mathrm{red}}

\newcommand{\Spec}{\mathrm{Spec}\ }

\newcommand{\Frob}{\mathrm{Frob}}

\newcommand{\id}{\mathrm{id}}

\newcommand{\Adm}{\mathrm{Adm}}
\newcommand{\Irr}{\mathrm{Irr}}
\newcommand{\gen}{\mathrm{gen}}

\newcommand{\obv}{{\mathrm{obv}}}

\newcommand{\Conv}{\mathrm{Conv}}
\newcommand{\Fp}{\F_p}

\newcommand{\un}[1]{\underline{#1}}
\renewcommand{\bf}[1]{\mathbf{#1}}
\newcommand{\Rep}{\mathrm{Rep}}
\newcommand{\Diag}{\mathrm{Diag}}

\newcommand{\tld}[1]{\widetilde{#1}}

\newcommand{\JH}{\mathrm{JH}}
\newcommand{\rig}{\mathrm{rig}}

\newcommand{\supp}{\mathrm{Supp}}

\newcommand{\out}{{\mathrm{out}}}
\newcommand{\orient}{\mathrm{or}}
\newcommand{\reg}{\mathrm{reg}}
\newcommand{\Dpst}{\mathrm{D}_{\mathrm{pst}}}
\newcommand{\Dpcris}{\mathrm{D}_{\mathrm{pcris}}}
\newcommand{\DdR}{\mathrm{D}_{\mathrm{dR}}}
\newcommand{\BdR}{\mathrm{B}_{\mathrm{dR}}}

\newcommand{\rbar}{\overline{r}}
\newcommand{\rhobar}{{\overline{\rho}}}
\newcommand{\taubar}{\overline{\tau}}

\newcommand{\Spf}{\mathrm{Spf}\,}
\newcommand{\sing}{\mathrm{sing}}

\newcommand{\rG}{\mathrm{G}}

\newcommand{\defeq}{\stackrel{\textrm{\tiny{def}}}{=}}

\newcommand{\ovl}[1]{\overline{#1}}

\newif\iffinalrun
\iffinalrun
  \newcommand{\mar}[1]{}
\else
  \newcommand{\mar}[1]{\marginpar{\raggedright\tiny #1}}

\DeclareMathOperator{\Mod}{Mod}
\DeclareMathOperator{\Coh}{Mod}

\DeclareMathOperator{\Lie}{Lie}

\DeclareMathOperator{\Ad}{Ad}

\DeclareMathOperator{\Mat}{Mat}
\DeclareMathOperator{\nv}{nv}
\DeclareMathOperator{\nm}{nm}
\DeclareMathOperator{\leql}{\text{$\setlength{\thickmuskip}{0mu}\leq\lambda$}}
\DeclareMathOperator{\ext}{ext}
\DeclareMathOperator{\pr}{pr}

\DeclareMathOperator{\gr}{gr}

\DeclareMathOperator{\Gr}{Gr}
\DeclareMathOperator{\Fl}{Fl}

\DeclareMathOperator{\Iw}{\cI}

\DeclareMathOperator{\tr}{tr}

\newcommand{\ra}{\rightarrow}

\newcommand{\iarrow}{\hookrightarrow}
\newcommand{\into}{\hookrightarrow}

\newcommand{\onto}{\twoheadrightarrow}

\newcommand{\risom}{\buildrel\sim\over\rightarrow} 
\title{Local models for Galois deformation rings and Applications}

\author{Daniel Le}
\address{Department of Mathematics,
University of Toronto,
40 St. George Street,
Toronto, ON M5S 2E4, Canada}
\email{le@math.toronto.edu}

\author{Bao V.~Le Hung}
\address{Department of Mathematics,
Northwestern University, 
2033 Sheridan Road, 
Evanston, Illinois 60208, USA}
\email{lhvietbao@googlemail.com}

\author{Brandon Levin}
\address{Department of Mathematics,
University of Arizona, 
617 N Santa Rita Avenue, 
Tucson, Arizona 85721, USA}
\email{bwlevin@math.arizona.edu}

\author{Stefano Morra}
\address{Universit\'e Paris 8, Laboratoire d'Analyse, G\'eom\'etrie et Applications,  LAGA, Universit\'e Sorbonne Paris Nord, CNRS, UMR 7539,  F-93430, Villetaneuse, France}
\email{morra@math.univ-paris13.fr}

\begin{document}

\begin{abstract} We construct projective varieties in mixed characteristic whose singularities model, in generic cases, those of tamely potentially crystalline Galois deformation rings for unramified extensions of $\Qp$ with small regular Hodge--Tate weights.
We establish several significant facts about their geometry including a unibranch property at special points and a representation theoretic description of the irreducible components of their special fibers.     We derive from these geometric results a number of local and global consequences: the Breuil--M\'ezard conjecture in arbitrary dimension for tamely potentially crystalline deformation rings with small Hodge--Tate weights (with appropriate genericity conditions), the weight part of Serre's conjecture for $U(n)$ as formulated by Herzig (for global Galois representations which satisfy the Taylor--Wiles hypotheses and are sufficiently generic at $p$), and an unconditional formulation of the weight part of Serre's conjecture for wildly ramified representations.
\end{abstract}

\maketitle

\tableofcontents

\clearpage{}%
\section{Introduction}

In this paper, we construct and study local models for stacks of \'etale $(\varphi,\Gamma)$-modules which correspond to tamely potentially crystalline Galois representations (of the absolute Galois group of an unramified extension of $\Q_p$) with small regular Hodge--Tate weights under suitable genericity conditions (see \S \ref{sec:suffgen}).
As a consequence, we deduce a refinement of a conjecture of Breuil--M\'ezard due to Emerton--Gee in this context and a conjecture of Herzig about the weight part of Serre's conjecture for definite unitary groups under genericity hypotheses.

\subsection{Motivation}

Over the last few decades, starting with the work of Wiles and Taylor--Wiles \cite{wiles-FLT,TW}, %
there has been tremendous progress on the modularity of global Galois representations, leading to spectacular consequences such as Fermat's Last Theorem and the Sato--Tate conjecture. %
Early modularity results such as those in \cite{TW} %
require stringent $p$-adic Hodge theoretic hypotheses to guarantee formal smoothness of patched global deformation rings. 
In the early 2000s, Kisin made the crucial observation that all the singularities of the patched deformation ring come from bad places, shifting the focus to local deformation rings, especially those at places dividing the residue characteristic of the coefficient field.  %
He then analyzed the singularities of (two-dimensional) potentially Barsotti--Tate local deformation rings through comparison to local models appearing in the theory of integral models of Shimura varieties, leading to very strong modularity lifting theorems in this setting, cf.~\cite{MFFGS}. Furthermore, Kisin constructed potentially semistable deformation rings in great generality and established their basic properties. However, the finer structure of these rings remain mysterious, and they appear to be intrinsically difficult objects in general. Indeed, the Breuil--M\'ezard conjecture predicts a lower bound for the complexity of the singularities in terms of modular representation theory of finite groups of Lie type.

In a recent advance, Emerton and Gee \cite{EGstack} have constructed $p$-adic formal stacks which interpolate these semistable deformation rings (these deformation rings are versal rings for the stacks), thereby ``globalizing" the above deformation theory and opening up more geometric ways to study it.
In this paper, we construct and analyze local models for a subset of these stacks---those parametrizing generic tamely potentially crystalline representations with small Hodge--Tate weights. A common feature of our work and Kisin's work is that these local models are closed subvarieties of certain Pappas--Zhu local models. 
However, unlike Kisin's situation, this inclusion is proper when the Hodge--Tate cocharacter is non-minuscule.

\subsection{Main results}

All our main results hold under suitable \emph{genericity} hypotheses, whose discussion we postpone to \S \ref{sec:suffgen} to avoid unnecessary distractions.
Fix a positive integer $n$, a rational prime $p$, a finite unramified extension $K/\Q_p$ with residue field $k$, and a (sufficiently large) finite extension $\F$ of $\F_p$. For a Hodge--Tate cocharacter $\lambda$ and an inertial type $\tau$, let $\cX^{\lambda,\tau}$ denote the $p$-adic formal stack over $W(\F)$ corresponding to $n$-dimensional potentially crystalline representations of the absolute Galois group $G_K$ of $K$ with Hodge--Tate weights $\lambda$ and Galois type $\tau$.
We let $M(\lambda)_{/W(\F)}$ be the Pappas--Zhu local model corresponding to the cocharacter $\lambda$ and the standard Iwahori subgroup (see \S \ref{sec:intro:MLM} below for the definition and further details).
Our first main theorem establishes a connection between $\cX^{\lambda,\tau}$ and $M(\lambda)$:
\begin{thm}[Theorem \ref{thm:stack_local_model}]
\label{thm:2intro:LMD} 
Let $\lambda$ be a regular Hodge--Tate cocharacter, and let $\tau$ be a sufficiently generic \emph{(}depending on $\lambda$\emph{)} tame inertial type. 
The $p$-adic completion of an explicit irreducible subvariety of $M(\lambda)$ \emph{(}depending on $\tau$\emph{)} is a smooth modification of $\cX^{\lambda,\tau}$. 
\end{thm}

Theorem \ref{thm:2intro:LMD} gives explicit presentations of potentially crystalline deformation rings, which we expect to have applications to local-global compatibility in the mod $p$ and $p$-adic Langlands programs. 
See \cite{EGS,LLLM2,Dan,DoLe,BHHMS} for applications when $n=2$ and $3$. 

\begin{rmk}\label{rmk:psmall}
The genericity condition implies $n\langle \lambda,\alpha\rangle < p$ for any root $\alpha$ so that $\lambda$ is necessarily ``small" with respect to $p$ and in particular is well within the \emph{Fontaine--Laffaille} range. Thus, for any generic representation to exist, we will need $p$ to be at least $O(n^2)$. 
See \S \ref{sec:suffgen} for more details.
\end{rmk}

One can think of this result as the modular/affine analogue of the work of Breuil--Hellmann--Schraen \cite{BHS}: whereas \cite{BHS} finds local models for moduli of trianguline representations in terms of Steinberg varieties (and thus related to the geometry of flag varieties), our models are found inside the (mixed characteristic) affine flag variety.

\vspace{2mm}
  
With $\lambda$ and $\tau$ as above, our methods also determine the irreducible components of the underlying reduced stack $\cX^{\lambda,\tau}_{\red}$ and construct local models for them. 
Now $\cX^{\lambda,\tau}_{\red}$ is a maximal dimensional substack of the underlying reduced stack of the stack $\cX_n$ of $(\varphi,\Gamma)$-modules of rank $n$, whose irreducible components $\cC_\sigma$ are parametrized by \emph{Serre weights} $\sigma$ (i.e. irreducible $\GL_n(k)$-representations over $\F$). If the highest weight of $\sigma$ is sufficiently deep in its $p$-alcove, we thus obtain a description of $\cC_\sigma$ in terms of certain deformed affine Springer fibers. %

The list of irreducible components of $\cX^{\lambda,\tau}_{\red}$ has a representation theoretic interpretation which is a weak (topological) version of the Breuil--M\'ezard conjecture. 
The usual Breuil--M\'{e}zard conjecture predicts that the special fiber $\cX^{\lambda,\tau}_{\F}$ has a complicated non-reduced structure, which we analyze by combining Theorem \ref{thm:2intro:LMD} with global methods. By taking versal rings, we deduce the following theorem (see Theorem \ref{thm:introBMrhobar} below):

\begin{thm}[Corollary \ref{cor:BM}]
Fix a set $\Lambda$ of regular Hodge--Tate cocharacters. 
The Breuil--M\'ezard conjecture holds for tamely potentially semistable deformation rings of Hodge--Tate weights $\lambda\in \Lambda$ of sufficiently generic \emph{(}depending on $\Lambda$\emph{)} representations $\rhobar: G_K \ra \GL_n(\F)$.
\end{thm}
\begin{rmk} Here and in the rest of the paper, we included all semistable deformation rings to get an overdetermined system of Breuil--M\'ezard equations. However, our genericity hypotheses automatically imply that any nonzero potentially semistable deformation ring that occurs is actually a potentially crystalline deformation ring. In particular, we do not prove any results about genuinely potentially semistable deformation rings.
\end{rmk}

Just as the trianguline local models \cite{BHS} shed light on the constituents of the locally analytic socle of completed cohomology of (unitary type) locally symmetric spaces, the models in Theorem \ref{thm:2intro:LMD} shed light on the constituents of the socle of mod $p$ completed cohomology (the \emph{modular} Serre weights). In more traditional language, this is known as the weight part of Serre's conjecture, which seeks to classify congruences between mod $p$ automorphic forms. Our main result in this direction is the following theorem, which confirms the unitary version of a conjecture of Herzig (\cite[Conjecture 6.9]{herzig-duke}, see also \cite[Conjecture 7.2.7, Theorem 10.2.11]{GHS}).
We refer the reader to \S \ref{sec:intro:WPSC} for undefined notation.

\begin{thm}[Theorem \ref{thm:SWC}]\label{thm:i2i:SWC} Let $F/F^+$ be a CM extension which is split at all places above $p$ and such that $F^+$ is unramified at $p$. 
Assume that $F^+ \neq \Q$.
Let $G_{/F^+}$ be a definite unitary group which splits over $F$. 
For each place $v\mid p$ in $F^+$, fix a place $\tld{v}$ of $F$ lying above $v$.
Let $\rbar: G_F \ra \GL_n(\F)$ be a \emph{(}$G$-\emph{)}modular Galois representation such that $\rbar(G_{F(\zeta_p)})$ is adequate and the local components $\rbar_v \defeq \rbar|_{G_{F_{\tld{v}}}}$ are tame and sufficiently generic for all $v\mid p$. 
Then the set of modular Serre weights $W(\rbar)$ is
\[\Big\{\bigotimes_{v| p}  \sigma_v \mid  \sigma_v \in W^?(\rbar_v)\Big\},\]
where $W^?(\rbar_v)$ is the explicit set defined by \cite{herzig-duke}.
\end{thm}

\subsection{The genericity condition}\label{sec:suffgen}

We expand on the terminology \emph{sufficiently generic}, which we only use in the introduction. 
Let $K/\Q_p$ be a finite extension and write $I_K$ for the inertia subgroup of $G_K$.
Suppose that $\rhobar: G_K \ra \GL_n(\F)$ is tame.
Then its restriction $\rhobar|_{I_K}$ to inertia is classified by the combinatorial data of a pair $(s,\mu) \in S_n^{\Hom_{\Qp}(K, \overline{\Q}_p)} \times (\Z^n)^{\Hom_{\Qp}(K, \overline{\Q}_p)}$ up to an equivalence relation (see Example \ref{ex:data:type} for details).
Indeed, $\rhobar|_{I_K}$ is a sum of characters which are necessarily powers of Serre's fundamental characters.
Then, informally speaking, $s$ determines the niveau of these characters and $\mu$ determines the powers.
For example, if $\rhobar$ is completely reducible, then we can take $s$ to be trivial and $\mu$ defined by
\[
\rhobar|_{I_K} = \bigoplus_{i=1}^n \prod_{j\in \Hom_{\Qp}(K, \overline{\Q}_p)} j\circ \ovl{\omega}_1^{\mu_{j,i}},
\] 
where $\ovl{\omega}_1: I_K \ra k^\times$ is the reduction of Serre's fundamental character of niveau $1$.
We say that $\rhobar$ is \emph{sufficiently generic} if, for an implicit nonzero polynomial $P \in \Z[X_1,\ldots,X_n]$ \emph{independent of} $p$, $P(\mu_j) \neq 0 \pmod p$ for each $j \in \Hom_{\Q_p}(K,\ovl{\Q}_p)$ (for some choice of $(s,\mu)$). 
If $\rhobar$ is not tame, then we say that $\rhobar$ is \emph{sufficiently generic} if its semi-simplification $\rhobar^{\mathrm{ss}}$ (which is tame) is.
The independence from $p$ guarantees that many sufficiently generic $\rhobar$ exist for large enough primes $p$, and in fact the proportion of tame $\rhobar|_{I_K}$ which are sufficiently generic tends to $1$ as $p$ tends to $\infty$. 
For other objects that have similar combinatorial descriptions like tame inertial types $\tau$ (cf.~\S \ref{sec:InertialTypes}) or Serre weights $\sigma$, one has an analogous notion of sufficiently generic, which we will freely use for the remainder of the introduction. 

{There are two sources of genericity in our methods.}
\begin{enumerate}
\item
\label{source:gen:1}
A \emph{combinatorial} genericity which requires that $\mu_j$ is \emph{sufficiently deep} in the base alcove of the standard apartment of $\GL_n$. The role of this condition is to guarantee
\begin{itemize}
\item that various representation theoretic objects (e.g.~decompositions of mod $p$ reductions of Deligne--Lusztig representations) behave according to a ``generic'' pattern; and
\item that the relevant Kisin varieties are trivial.
\end{itemize}
 Some form of this condition is unavoidable for our theorems to be true, as the Galois deformation rings are known to exhibit less uniform behavior in its absence, see \cite[Th\'eor\`eme 2]{CDM}.

This sort of condition also appears in \cite{herzig-duke,GHS} and in our previous work \cite{LLLM}, \cite{LLLM2}, \cite{LLL} and corresponds to a polynomial of the form $P = \prod_{i=1}^{n} \prod_{m=0}^M (X_i - X_{i+1} - m)$ for some positive $M$ where $X_{n+1}$ is understood to be $X_1$. 
In these cases, we make the relevant $M$ explicit. 
In particular, we will always have $M \geq \langle \lambda,\alpha \rangle$ for all roots $\alpha$, which gives the inequality in Remark \ref{rmk:psmall}. 
\item
\label{source:gen:2}
A \emph{geometric} genericity, whose role is to guarantee:
\begin{itemize}
\item that we can apply Elkik's approximation theorem to the local models; and
\item that our local models have the desired geometric properties.
\end{itemize}
The first item leads to a condition similar to the combinatorial genericity condition above, i.e.~ it is guaranteed by a choice of polynomial of the form
$P = \prod_{i=1}^{n} \prod_{m=0}^M (X_i - X_{i+1} - m)$ for some positive $M$ independent of $p$ (which arises from the singularity of local models, and hence is less explicit).

On the other hand, to guarantee the second item, our approach is to deduce geometric properties of the local models by specialization from some universal cases.
Since the properties we are interested in (e.g.~ being unibranch) are not preserved under arbitrary base change but only preserved under ``generic'' base change, we need to ensure that $\mu_j$ avoids a closed locus in $\bA^n_\Z$ which is \emph{independent of} $p$ (see \S \ref{sec:UMLM}). This produces a computable, but hard to make explicit, polynomial $P$. 

The geometric genericity condition is mainly an artifact of our proof of Theorem \ref{thm:intro:unibranch}. While the second source for the geometric genericity condition appears to impose more severe restrictions, we conjecture that that it is in fact unnecessary: in other words, we expect that our main result (Theorem \ref{thm:intro:unibranch}) hold with just a combinatorial genericity condition, but with the caveat that the bounds depend on the singularity of (universal) local models. 

We verified this conjecture in several cases, where we write $\eta$ for the Hodge--Tate cocharacter corresponding to $(n-1,n-2,\ldots,1,0)$ in all embedding $K\into\ovl{\Q}_p$:
\begin{itemize}
\item When $n=2$, $\Lambda=\{\eta\}$, where we can take $M=2$ (this follows from Theorems \ref{thm:Breuil-Kisin_local_model} and \ref{prop:stack_diagram}, noting that the ``monodromy condition'' is vacuous in this case).
\item When $n=3$, $\Lambda=\{\eta\}$, where we can take $M=4$, cf.~\cite{GL3Wild}.
\item When $n=3$, $\Lambda=\{\lambda\}$ where $\lambda$ corresponds to $(3,1,0)$ in all embeddding $K\into \ovl{\Q}_p$, but restricting to a specific open locus in the appropriate potentially crystalline stack, where we can take $M=10$, cf.~ Appendix \ref{ex:failure:UB}.
\item When $n$ is arbitrary, $\Lambda=\{\eta\}$, restricting to specific open loci in the appropriate potentially crystalline stack, where we can take $M$ to be a linear function in $n$, cf.~ \cite{OBW, Colength1}.
\end{itemize}
Unfortunately, beyond these cases, directly verifying the conjecture without extra geometric observations seems prohibitively computationally expensive with current computer algebra systems.
\end{enumerate}
In the introduction, while we omit the implicit polynomial $P$, we will describe exactly what it depends on (aside from $n$). Note that the particular $P$ may be different in different statements, and its precise nature will be spelled out in the body of the paper.

\subsection{Local models for potentially crystalline stacks} 
\label{sec:intro:MLM}

The possibility of studying singularities of potentially semistable deformation rings by means of group theoretic local models was first suggested by Kisin in \cite{MFFGS}. Using his theory of Breuil--Kisin modules, he resolved potentially Barsotti--Tate deformation rings (which correspond to minuscule Hodge--Tate cocharacters) by formal schemes which are certain completions of Pappas--Rapoport local models. 
To generalize this picture to non-minuscule cocharacters, one encounters the essential difficulty that not all Breuil--Kisin modules give rise to crystalline representations; indeed, only those obeying the $p$-adic analogue of Griffiths transversality do. Thus, while local models for the moduli of Breuil--Kisin modules exist quite generally in the form of Pappas--Zhu models, one needs to cut them down suitably to obtain models related to Galois deformation rings. In this section we will explain the construction of the subvariety in Theorem \ref{thm:2intro:LMD} above, which achieves this in certain situations.

Let $E$ be a finite extension of $\Qp$ with ring of integers $\cO$, uniformizer $\varpi$, and residue field $\F$. %
 Let $L\cG$ be the ind-group scheme given by $L\cG(R)=\GL_n(R(\!(v+p)\!))$ for any $\cO$-algebra $R$, the loop group. Consider the positive loop group scheme $L^+\cG$ over $\cO$ sending an
$\cO$-algebra %
$R$ to the subgroup of $\GL_n(R[\![v+p]\!])$ consisting of matrices that
are upper triangular mod~$v$. The %
quotient $L^+\cG\backslash L\cG$ is represented by an ind-proper
$\cO$-ind-scheme $\Gr_{\cG}$. This is a mixed characteristic
version of the degeneration of affine Grassmannians introduced by Gaitsgory. Indeed its generic fiber $\Gr_{\cG, E}$ is isomorphic to an affine Grassmannian, %
while the special fiber $\Gr_{\cG, \F}$ is isomorphic to the affine flag variety $\mathrm{Fl}$ (for the standard Iwahori $\cI$).

For $\lambda \in \Z^n$, let $S^{\circ}_E(\lambda)$ denote the $L^+ \cG_E$-orbit of $(v+p)^{\lambda}$ in $\Gr_{\cG, E}$. The Pappas--Zhu local model $M(\leql)$ is the Zariski closure of $S^{\circ}_E(\lambda)$ in $\Gr_{\cG}$, cf.~\cite{PZ}.   %

Let $\bf{a} \in \cO^n$.  We now consider the condition %
\[
 \tag{$\star$}  v \frac{dA}{dv} A^{-1}  + A \Diag(\bf{a}) A^{-1} \in
 \left(\frac{1}{v+p}\right)\mathrm{Lie} \, L^+\cG \label{eq:monodromy}
\]
for $ A\in L\cG (R)$. This is an approximation to the monodromy condition coming from $p$-adic Hodge theory.
 This condition clearly descends to a closed condition on $\Gr_{\cG}$. 

\begin{defn} 
\label{defn:LMQp}
The local model $M(\lambda, \nabla_{\bf{a}})$ is
  the Zariski closure in $M(\leql)$ of the locus of
    \eqref{eq:monodromy} in $S^{\circ}_E(\lambda)$. %
\end{defn}
Note that condition \eqref{eq:monodromy} is preserved under the right action by the constant diagonal torus $T$.  Thus, $M(\lambda, \nabla_{\bf{a}})$ inherits an action of $T$ compatible with the $T$-action on $M(\leql)$.  

The local models $M(\lambda, \nabla_{\bf{a}})$ turn out to behave very differently from the Pappas--Zhu models $M(\leql)$:
\begin{itemize}
\item
The generic fiber of $M(\lambda, \nabla_{\bf{a}})$ is smooth; it is isomorphic to a partial flag variety (see Proposition \ref{prop:genfiber}). In contrast, the generic fiber of $M(\leql)$ is not smooth unless $\lambda$ is minuscule (cf.~\cite{Haines-Richarz}).

\item A deep theorem of Zhu implies that the special fiber of $M(\leql)$ is reduced, and thus $M(\leql)$ is normal. In contrast, it will follow from the connection between $M(\lambda, \nabla_{\bf{a}})$ and Galois deformation theory that its special fiber fails to be reduced, and $M(\lambda, \nabla_{\bf{a}})$ fails to be normal in general. In fact, this failure is quite severe: one can get lower bounds for the non-reducedness in terms of affine Kazhdan--Lusztig multiplicities.
\end{itemize}
In other words, while our models have nice generic fibers,  they are nevertheless complicated degenerations of partial flag varieties.

Using the standard stratifications on $\Gr_{\cG}$, it is not difficult to analyze the underlying reduced subscheme of $M(\lambda, \nabla_{\bf{a}})$, in particular one sees that it is irreducible, and there is a combinatorial parametrization of the irreducible components of the special fiber. However, in order to establish the connection of our models to Galois deformation theory, we have to understand the behavior of $M(\lambda, \nabla_{\bf{a}})$ under completion. The essential difficulty is that an irreducible variety may break up into formal branches in some complicated way after completions: its singularities may not be \emph{unibranch}. One important sufficient condition to guarantee this unibranch property is normality, and to the best of our knowledge, we are not aware of any other useful general criteria. Worse still, it turns out that $M(\lambda, \nabla_{\bf{a}})$ fails to be unibranch in general! {(See Appendix \ref{ex:failure:UB} for an explicit example.)}
Miraculously, we manage to show that (for generic values of $\bf{a}$) $M(\lambda, \nabla_{\bf{a}})$ is unibranch at special points:
\begin{thm}[Theorem \ref{thm:model_unibranch}] \label{thm:intro:unibranchMLM} There exists a nonzero polynomial $P \in \Z[X_1, \ldots, X_n]$ such that if $P(\bf{a}) \neq 0 \mod \varpi$, then for any $T$-fixed point  $x \in M(\lambda, \nabla_{\bf{a}})(\ovl{\F}_p)$,  the completed local ring $\cO_{ M(\lambda, \nabla_{\bf{a}}), x}^{\wedge}$ is a domain $($i.e., $M(\lambda, \nabla_{\bf{a}})$ is unibranch at its $T$-fixed points).     %
\end{thm} 
This is the deepest geometric fact that we prove about $M(\lambda, \nabla_{\bf{a}})$, and its proof lies at the technical heart of the paper.  A key observation (Proposition \ref{prop:unibranch_equal_char}) is that the theorem holds (under a mild assumption on the characteristic) for the equal characteristic analogues of   $M(\lambda, \nabla_{\bf{a}})$ where $p$ is replaced by a variable $t$. 
In this context, there is more symmetry: there is an extra $\bG_m$-action given by ``loop rotation'' which scales $t$. Thanks to this, the $T$-fixed points are all cone points, in the sense that they are the fixed point of an attracting torus action, and one observes that cone points are unibranch. Unfortunately, we can not mimic this argument in the original mixed characteristic setting, as it doesn't make sense to ``scale'' the prime $p$. Instead, we resort to a soft spreading out argument, by contemplating the universal case where $p$ and $\bf{a}$ are formal variables. The fact that being unibranch can be phrased in terms of the normalization map, and normalization commutes with generic base change, allows us to transfer the unibranch property from equal characteristic to mixed characteristic. It is here that the universal polynomial $P$ appears: its vanishing locus is the obstruction to certain properties being preserved under base change. The actual argument is a bit more involved than this outline, since we do not base change to spectra of fields, but rather spectra of DVRs.

Having proven the important geometric properties of $M(\lambda, \nabla_{\bf{a}})$, we now turn to its connection to Galois theory. 
Let $K/\Qp$ be a finite unramified extension, and let $\cJ$ be the set of embeddings $\Hom_{\Qp}(K, \overline{\Q}_p)$. In \cite{EGstack}, Emerton--Gee constructed the moduli stack $\cX_n$ over $\Spf \cO$ of rank $n$ $(\varphi,\Gamma)$-modules. By its construction, $\cX_n$ interpolates framed deformation rings in the sense that the set $\cX_n(\ovl{\F}_p)$ is in bijection with the set of continuous representations $\rhobar:G_K \ra \GL_n(\overline{\F}_p)$, and framed deformation rings of such $\rhobar$ are versal rings (in the sense of \cite[Definition 2.2.9]{EGschemetheoretic}) for $\cX_n$. Furthermore, for a collection $\lambda \in (\Z^n)^{\cJ}$ and a rank $n$ inertial type $\tau$ defined over $\cO$ (cf.~\S \ref{sec:InertialTypes} for their definition), %
 they construct a $\cO$-flat $p$-adic formal algebraic substack $\cX^{\lambda,\tau}$ which is characterized by the property that its points over any finite flat $\cO$-algebra correspond to potentially crystalline representations $\rho$ of type $(\lambda,\tau)$ (i.e.~the Hodge--Tate weights of $\rho$ are given by $\lambda$ and $\mathrm{WD}(\rho)$ induces the inertial type $\tau$). 

 Now, to any \emph{tame} inertial type $\tau$ for $I_K$, one can associate a collection $\bf{a}_{\tau} = (\bf{a}_{\tau, j})_{j \in \cJ}$, where $\bf{a}_{\tau, j} \in \cO^n$ records the inertial weights of $\tau$ (see \S \ref{sec:local_model_EG}). Set $\lambda = (\lambda_j)_{j \in \cJ} \in (\Z^n)^{\cJ}$.  Define 
\[
M_{\cJ}(\lambda, \nabla_{\bf{a}_{\tau}}) = \prod_{j \in \cJ} M(\lambda_j, \nabla_{ \bf{a}_{\tau, j}})
\] 
where, for each $j\in\cJ$, the local models $M(\lambda_j, \nabla_{ \bf{a}_{\tau, j}})$ are those appearing in  Definition \ref{defn:LMQp}.
Our main result is the following:
\begin{thm}[Theorem \ref{thm:stack_local_model}]
\label{thm:intro:LMD}   If $\tau$ is sufficiently generic \emph{(}depending on $\lambda$\emph{)}, then
 there exist Zariski open covers $%
 \underset{\text{\tiny{$\tld{z} $}}}{\text{\Large{$\bigcup$}}}
 \cX_{\reg}^{\leq \lambda,\tau}(\tld{z}) $ and $\underset{\text{\tiny{$\tld{z} $}}}{\text{\Large{$\bigcup$}}} U_{\reg}(\tld{z},\leql,\nabla_{\bf{a}_\tau})^{\wedge_p}$ of $\underset{\tiny{\substack{\lambda' \leq \lambda\\ \lambda' \text{\emph{reg. dom.}}}}}{\text{\Large{$\bigcup$}}} \cX^{\lambda',\tau} $ and  $ \underset{\tiny{\substack{\lambda' \leq \lambda\\ \lambda' \text{\emph{reg. dom.}}}}}{\text{\Large{$\bigcup$}}}M(\lambda', \nabla_{\bf{a}_\tau})^{\wedge_p}$ respectively such that for each $\tld{z} $, there exists a local model diagram 
\begin{equation}
\xymatrix{
& \tld{\cX}_{\reg}^{\leq \lambda,\tau}(\tld{z}) \ar[dl] \ar[dr] & \\
\cX_{\reg}^{\leq \lambda,\tau}(\tld{z}) & &U_{\reg}(\tld{z},\leql,\nabla_{\bf{a}_\tau})^{\wedge_p} \\
}
\end{equation}
where both arrows are torsors for the torus $T^\cJ$ with respect to different $T^{\cJ}$-actions and the superscript $\wedge_p$ stands for taking $p$-adic completion. 
\end{thm}

\begin{rmk} \begin{enumerate}

\item {In the above statement, when we talk about the scheme-theoretic union of two closed formal algebraic substacks $\cY,\cZ$ of a formal algebraic stack $\cX$, we mean to take the scheme-theoretic image of the map $\cY\sqcup\cZ\ra\cX$ (\cite[Definition A.16]{EGstack}).}
\item The right arrow in the local model diagram is highly non-canonical, as it is produced by a Hensel-type lifting argument (in the form of Elkik's approximation theorem \cite{Elkik}). 
However, the entire diagram is canonical in characteristic $p$. 
\item 
When $\lambda = \eta\defeq (n-1,n-2,\ldots,1,0)_{j\in \cJ}\in (\Z^n)^{\cJ}$, one has
$
\underset{\tiny{\substack{\lambda' \leq \lambda\\ \lambda' \text{reg. dom.}}}}{\text{{$\bigcup$}}} \cX^{\lambda',\tau}=\cX^{\eta,\tau}.$
Since potentially crystalline deformation rings of type $(\eta,\tau)$ are versal rings to $\cX^{\eta, \tau}$, we see that they appear (up to smooth modifications) as the completion of local rings of $M(\eta, \nabla_{\bf{a}_\tau})$ at closed points.
\end{enumerate}
\end{rmk}

We now give a slightly simplified outline of the proof of Theorem \ref{thm:intro:LMD}. The starting point is the theory of Breuil--Kisin modules: The potentially crystalline stacks we consider are closed substacks of the moduli stack of Breuil--Kisin modules  $Y^{\leq \lambda,\tau}$ with tame descent data of type $(\lambda,\tau)$, which is known to have the Pappas--Zhu model $M(\leql)$ as a local model.
More specifically, the natural open affine cover of $\Gr_\cG=\bigcup_{\tld{z}} \cU(\tld{z})$ by translates of the ``big open cell'' induces an open cover of $M(\leql)$. We develop a theory of canonical bases of Breuil--Kisin modules to show that this open cover induces an open cover of $Y^{\leql,\tau}$. Thus we get the analogue of the above local model diagram for $Y^{\leql,\tau}$ and induced open affine covers on every object in sight. These are the open covers featured in Theorem \ref{thm:intro:LMD}.

At this point, we get two closed substack of $Y^{\leql,\tau}(\tld{z})$: the substack $\cX^{\leql,\tau}(\tld{z})$ and the substack $\cX^{\leql,\tau,\star}(\tld{z})$ induced by the $p$-adic completion of $ \bigcup_{\lambda' \leq \lambda} M(\lambda', \nabla_{\bf{a}_\tau})$ along the local model diagram for $Y^{\leql,\tau}$. They are genuinely different substacks, because condition (\ref{eq:monodromy}) is only an approximation to the condition cutting out $\cX^{\leql,\tau}$ inside $Y^{\leql,\tau}$. 
However, the two substacks are $p$-adically close, and using the smoothness of the generic fiber of $M(\lambda, \nabla_{\bf{a}})$, one can produce a non-canonical embedding $\cX^{\leql,\tau}(\tld{z})\into \cX^{\leql,\tau,\star}(\tld{z})$. Since both stacks turn out to have the same dimension, the maximal dimensional part $\cX_\reg^{\leql,\tau}(\tld{z})$ of $\cX^{\leql,\tau}(\tld{z})$ embeds into the maximal dimension part of $\cX^{\leql,\tau,\star}(\tld{z})$. Now, using the results of \cite{LLL} (which ultimately uses Taylor--Wiles patching, and hence automorphic forms), one obtains a lower bound on the number of irreducible components (of the spectrum of the structure sheaf) of the former, while Theorem \ref{thm:intro:unibranchMLM} gives the same upper bound for the number of irreducible components (of the spectrum of the structure sheaf) of the latter. Thus the two maximal dimension parts are (non-canonically) isomorphic to each other, which concludes the proof. 

Theorem \ref{thm:intro:LMD} allows us to study local properties of the potentially crystalline stacks $\cX^{\leql,\tau}$ via the local models, which gives crucial geometric information about potentially crystalline deformation rings needed for our applications below, cf.~ Theorem \ref{thm:intro:unibranch}. In characteristic $p$, one can do even better: the local model diagrams produced by Theorem \ref{thm:intro:LMD} glue together, and thus one can even study \emph{global properties} of the underlying reduced stacks $\cX^{\lambda,\tau}_{\red}$ (which live in characteristic $p$) via the reduced special fiber $\overline{M}(\lambda, \nabla_{\bf{a}_\tau})$ of $M(\lambda, \nabla_{\bf{a}_\tau})$. To state our result, we recall \cite{EGstack} that $\cX_{n,\red}$ is equidimensional, and its irreducible components are in bijection with the irreducible $\F$-representations of $\GL_n(k)$ (which we refer to as \emph{Serre weights}). We write $\cC_{\sigma}$ for the irreducible component of $\cX_{n,\red}$ corresponding to a Serre weight $\sigma$. 
Given $\lambda \in (\Z^n)^{\cJ}$ regular and dominant, let $V(\lambda-\eta)$ be the irreducible $\Res_{K/\Q_p} \GL_n$-representation with highest weight $\lambda - \eta$ (recall that $\eta$ is such that $\eta_j = (n-1,n-2,\ldots,1,0)$ for all $j \in \cJ$).
We also denote the restriction of $V(\lambda-\eta)$ to $\GL_n(\cO_K)$ by $V(\lambda-\eta)$.
As in \S \ref{sec:suffgen}, a tame type $\tau$ corresponds to an equivalence class of pairs $(s,\mu)$.
Then we let $\sigma(\tau)$ be the Deligne--Lusztig representation $R_s(\mu)$ where $s$ defines a rational torus and $\mu$ defines a character (see \S \ref{sec:DLandSW}). 
For a representation $V$ over $E$ of a compact group, let $\ovl{V}$ be the semisimplification of the reduction of any invariant $\cO$-lattice in $V$.
Then we prove:
\begin{thm}[Theorem \ref{thm:EGmodp}] Let $\lambda$ be regular dominant and let $\tau$ be a sufficiently generic tame inertial type.  Then: %

\label{thm:intro:irreducible_components_mod_p}
\begin{enumerate}
\item $\cX^{\lambda,\tau}_\red=\cup_\sigma \cC_{\sigma}$, where the union runs over all Serre weights $\sigma\in \JH(\ovl{\sigma(\tau) \otimes_E V(\lambda-\eta)})$. %
\item There is a natural bijection between the irreducible components of $\overline{M}(\lambda, \nabla_{\bf{a}_\tau})$ and the Jordan--H\"older factors of $\ovl{\sigma(\tau) \otimes_E V(\lambda-\eta)}$.
\item For each $\sigma\in \JH(\ovl{\sigma(\tau) \otimes_E V(\lambda-\eta)})$, we have a mod $p$ local model diagram: 
\begin{equation}\label{eq:mod_p_local_model}
\xymatrix{
& \tld{\cC}_{\sigma} \ar[dl] \ar[dr] & \\
\cC_{\sigma} &   &  \overline{M}(\lambda, \nabla_{\bf{a}_\tau})_{\sigma} \\
}
\end{equation} 
where $\overline{M}(\lambda, \nabla_{\bf{a}_\tau})_{\sigma}$ is the irreducible component of $\overline{M}(\lambda, \nabla_{\bf{a}_\tau})$ labelled by $\sigma$ \emph{(}denoted by $C_\sigma$ in Theorem \ref{thm:EGmodp}\emph{)} and both arrows are torsors for the torus $T^{\cJ}$ with respect to different $T^{\cJ}$-actions.   
\end{enumerate}
\end{thm}
\begin{rmk} 
\label{rmk:intro:MLM:sp:fib}
\begin{enumerate}
\item Our proof of Theorem \ref{thm:intro:irreducible_components_mod_p} does not go through Theorem \ref{thm:intro:LMD}. Because of that it holds under much milder genericity conditions compared to our other theorems: we only need an explicit \emph{combinatorial} genericity condition (see \S \ref{sec:suffgen}).
\item It follows essentially from the definitions that $\overline{M}(\lambda, \nabla_{\bf{a}_\tau})_{\sigma}$ is an irreducible component of a \emph{deformed} affine Springer fiber in the sense of \cite{Frenkel_Zhu}.  In particular, $\cC_\sigma$ is equisingular to an irreducible component of a deformed affine Springer fiber. 
We expect that this connection will be a powerful tool to investigate the internal structure of irreducible components of the Emerton--Gee stack. As a sample application, we deduce Herzig's formulation of the weight part of Serre's conjecture (Theorem \ref{thm:intro:SWC}) from the count of torus-fixed points in the irreducible components of affine Springer fibers obtained by Boixeda Alvarez \cite{Pablo} (see \S \ref{sec:intro:WPSC} for more details).
\end{enumerate}
\end{rmk}

Theorem \ref{thm:intro:irreducible_components_mod_p} follows from analyzing the effect of condition (\ref{eq:monodromy}) on the reduced special fiber of $M(\leql)$, which was determined by Pappas--Zhu \cite{PZ}. Namely, \cite{PZ} shows that it is the reduced union of the affine Schubert cells $S_{\F}^{\circ}(\tld{w})$ for $\tld{w}$ running over the $\lambda$-admissible set $\Adm(\lambda)$, which is defined in terms of combinatorics of the affine Weyl group. A simple computation shows that (\ref{eq:monodromy}) cuts out an affine subspace of the affine space $S_{\F}^{\circ}(\tld{w})$, whose dimension is easily computed. This provides a combinatorial parametrization of the irreducible components of $\overline{M}(\lambda, \nabla_{\bf{a}_\tau})$ in terms of a subset of $\Adm(\lambda)$, which beautifully matches with the parametrization of $\JH(\ovl{\sigma(\tau) \otimes_E V(\lambda-\eta)})$ given by Jantzen's generic decomposition pattern. Finally, one has to show that $\overline{M}(\lambda, \nabla_{\bf{a}_\tau})_{\sigma} \subset \overline{M}(\lambda, \nabla_{\bf{a}_\tau})$ corresponds to $\cC_\sigma$ in the local model diagram, and we achieve this by identifying the Breuil--Kisin modules attached to a generic point of $\overline{M}(\lambda, \nabla_{\bf{a}_\tau})_{\sigma}$.

\subsection{The Breuil--M\'ezard conjecture} 
\label{intro:sec:BMC}

Let $K/\Q_p$ be a finite extension with ring of integers $\cO_K$ and residue field $k$. 
Write $G_K$ for the absolute Galois group of $K$.
(Note that this is a more general setup than in the previous section for now.)
The  Breuil--M\'ezard conjecture quantifies the complexity of the special fibers of potentially semistable Galois deformation rings in terms of $\GL_n(\cO_K)$-representations with mod $p$ coefficients.
These special fibers are especially mysterious because outside of very special cases they do not have known moduli interpretations.
We now describe the ``geometric'' version of the conjecture as formulated by \cite{EG}.    

Let $\tau$ be an inertial Weil--Deligne type for $K$ (see Definition \ref{defn:WDtype}) and let $\lambda \in (\Z^n)^{\Hom_{\Qp}(K, \overline{\Q}_p)}$ be a collection of  regular Hodge--Tate weights.  
For a continuous Galois representation $\rhobar: G_K \ra \GL_n(\F)$, there is a unique reduced quotient $R_\rhobar^{\lambda, \tau}$ of the framed $\cO$-deformation ring $R_\rhobar^\square$ whose $\ovl{\Q}_p$-points correspond to lifts $\rho: G_K \ra \GL_n(\ovl{\Q}_p)$ which are potentially semistable of type $(\lambda,\tau)$ (i.e.~the Hodge--Tate weights of $\rho$ are given by $\lambda$ and $\mathrm{WD}(\rho)$ induces the inertial Weil--Deligne type $\tau$). The dimensions of these rings are independent of $(\lambda,\tau)$, and one can associate to each pair $(\lambda,\tau)$ the cycle $Z( R_\rhobar^{\lambda, \tau}/\varpi)$ in $\Spec R_\rhobar^\square/ \varpi$, which counts the irreducible components of $\Spec R_\rhobar^{\lambda, \tau}/\varpi$ with appropriate multiplicities.

The Breuil--M\'ezard conjecture describes the cycle $Z( R_\rhobar^{\lambda, \tau}/\varpi)$ in representation theoretic terms as $\lambda$ and $\tau$ vary.  
For $V$ a virtual $\GL_n(\cO_K)$-representation over $E$ expressed as the difference $V_1 - V_2$ of two genuine representations, we let $\ovl{V}$ be the virtual $\GL_n(\cO_K)$-representation $\ovl{V}_1-\ovl{V}_2$ over $\F$, where for $i = 1,2$, $\ovl{V}_i$ denotes the semisimplification of the reduction modulo $\varpi$ of any $\GL_n(\cO_K)$-stable $\cO$-lattice in $V_i$. This is independent of the choice of $V_1$ and $V_2$ and $\cO$-lattices therein. %

\begin{conj}  \label{conj:introBM} There exist cycles $\cZ_\sigma(\rhobar)$ in $\Spec R_\rhobar^\square/\varpi$ for each irreducible $\GL_n(\cO_K)$-representation $\sigma$ over $\F$ such that for all $\tau$ and all regular $\lambda$, 
\[
Z(R_\rhobar^{\lambda, \tau}/\varpi) = \sum_\sigma [\ovl{r(\tau) \otimes_E V(\lambda-\eta)}:\sigma] \cZ_\sigma(\rhobar),
\]
where $r(\tau)$ is a virtual representation of $\GL_n(\cO_K)$ over $E$ defined in \cite[\S 4.2]{Shotton} using an inertial local Langlands correspondence, $V(\lambda-\eta)$ is the restriction to $\GL_n(\cO_K)$ of the irreducible algebraic representation of $\Res_{K/\Qp} \GL_n$ with highest weight $\lambda-\eta$, and $[\ovl{r(\tau) \otimes_E V(\lambda-\eta)}:\sigma]$ denotes the \emph{(}possibly negative\emph{)} multiplicity of $\sigma$ in $\ovl{r(\tau) \otimes_E V(\lambda-\eta)}$.
\end{conj} 

\begin{rmk}
\begin{enumerate}
\item The symbol $\tau$ is used in \cite{Shotton} to denote what is called in \emph{loc.~cit.}~an inertial type, which is distinct from, but equivalent to, the notion of a Weil--Deligne inertial type (see \cite{tate}).
We ignore this distinction above.
\item If the monodromy operator of $\tau$ is $0$, then we say that $\tau$ is an inertial type.
Then $r(\tau)$ is a genuine $\GL_n(\cO_K)$-representation associated to $\tau$ via the inertial local Langlands correspondence and denoted $\sigma(\tau)$.
As mentioned in \S \ref{sec:intro:MLM}, when $\tau$ is tame and generic, $\sigma(\tau)$ is a Deligne--Lusztig representation with a simple description (Proposition \ref{prop:tameILL}).
\item The equations in the Conjecture \ref{conj:introBM} massively overdetermine the cycles $\cZ_{\sigma}(\rhobar)$. 
In fact, the cycles are uniquely determined by any collection of $(\lambda, \tau)$ such that $\ovl{r(\tau) \otimes_E V(\lambda-\eta)}$ span the Grothendieck group of finite dimensional $\GL_n(\cO_K)$-representations over $\F$.  
\end{enumerate}
\end{rmk}

Combining the Taylor--Wiles patching method and the $p$-adic local Langlands correspondence for $\GL_2(\Qp)$ of \cite{colmez}, Kisin established the conjecture in a wide range of cases when $n = 2$ and $K= \Qp$ in \cite{kisin-fontaine-mazur}.
(When $n = 2$ and $K= \Qp$, the conjecture is now known in all cases by \cite{paskunas-BM, hu-tan, sander, tung1, tung}.)
While the Taylor--Wiles patching method is available in some generality, the $p$-adic Langlands correspondence is not {known} for $n>2$ or $n=2$ and $K \neq \Qp$. 
Absent a general $p$-adic Langlands correspondence, one can still try to establish this conjecture for classes of pairs $(\lambda,\tau)$. 
For example, \cite{Gee-Kisin} prove Conjecture \ref{conj:introBM} when $n=2$, $\lambda = \eta$, and the monodromy operator of $\tau$ is $0$. 
In \S \ref{sec:BMminweight}, we prove this conjecture when $K/\Qp$ is unramified for sufficiently generic $\rhobar$ and pairs $(\lambda, \tau)$ where $\lambda$ ranges over a finite set and $\tau$ ranges over tame inertial Weil--Deligne types. 

\begin{thm}[Corollary \ref{cor:BM}] \label{thm:introBMrhobar}  Assume $K/\Qp$ is unramified and let $\Lambda$ be a finite set of collections of regular Hodge--Tate weights.
If $\rhobar: G_K \ra \GL_n(\F)$ is sufficiently generic \emph{(}depending on $\Lambda$\emph{)}, then there exist cycles $\cZ_\sigma(\rhobar)$ in $\Spec R_\rhobar^\square/\varpi$ for each irreducible $\GL_n(\cO_K)$-representation $\sigma$ over $\F$ such that 
\[
Z(R_\rhobar^{\lambda,\tau}/\varpi) = \sum_\sigma [\ovl{r(\tau) \otimes_E V(\lambda-\eta)}:\sigma] \cZ_\sigma(\rhobar)
\]
for all $\lambda\in \Lambda$ and tame inertial Weil--Deligne types $\tau$.
\end{thm}

\begin{rmk}\label{rmk:intro:BMcycle}
\begin{enumerate}
\item If $\Lambda$ contains $\eta$, then the cycles $\cZ_\sigma(\rhobar)$ are unique since the set of $\ovl{r}(\tau)$ for tame $\tau$ span the Grothendieck group of finite-dimensional $\GL_n(\cO_K)$-representation over $\F$. If $\Lambda$ contains $\eta$ and at least one other Hodge-Tate weight, then the set of classes $[\ovl{r(\tau) \otimes_E V(\lambda-\eta)}]$ is spanning and linearly dependent in the Grothendieck group of finite-dimensional $\GL_n(\cO_K)$-representation over $\F$, so Theorem \ref{thm:introBMrhobar} produces many non-trivial linear relations among $Z(R_\rhobar^{\lambda,\tau}/\varpi)$.
\item In contrast to \cite{kisin-fontaine-mazur,Gee-Kisin}, we restrict to cases where $\tau$ is tame. 
However, this result is new even for $n=2$ if $K \neq \Q_p$.
Indeed, in contrast to \cite{Gee-Kisin}, $\Lambda$ may contain non-minuscule weights (which are necessarily ``small" relative to $p$; see Remark \ref{rmk:psmall}).

\item  For \label{it:intro:BMcycle:3} a suitable globalization of $\rhobar$ (as defined in \cite[\S 5.1.1]{EG}) 
and a choice of global setup, the cycles $\cZ_{\sigma}(\rhobar)$ are expected to be the support cycle of {any} patched module $M_{\infty}(\sigma)$ associated to the Serre weight $\sigma$, 
thereby connecting the theorem directly to modularity and the weight part of Serre's conjecture. When $\rhobar$ is tame, our proof establishes this expectation---see the discussion around Theorem \ref{thm:intro:unibranch}.
 In particular, for tame $\rhobar$, this compatibility with patching functors gives a \emph{global} characterization of $\cZ_{\sigma}(\rhobar)$. %
\end{enumerate}
\end{rmk}

Our starting point to attack Theorem \ref{thm:introBMrhobar} is the Taylor--Wiles method, following the approach of \cite{kisin-fontaine-mazur,Gee-Kisin}. The Taylor--Wiles method provides a large supply of exact functors $M_\infty$ from $\GL_n(\cO_K)$-representations over $\cO$ to maximal Cohen--Macaulay modules of generic rank at most one over (power series over) framed local deformation rings with $p$-adic Hodge-theoretic conditions.
Given $\GL_n(\cO_K)$-stable $\cO$-lattices $\sigma(\tau)^\circ$,  $V(\lambda-\eta)^\circ$ in $\sigma(\tau)$,  $V(\lambda-\eta)$ respectively, it is a folklore expectation (affirmed under mild assumptions by \cite{FKP})
that the Fontaine--Mazur conjecture implies that $M_\infty(\sigma(\tau)^\circ \otimes_{\cO} V(\lambda-\eta)^\circ)$ has full support on $\Spec R_\rhobar^{\lambda, \tau}$. If this were true, the exactness of $M_\infty$ would imply Conjecture \ref{conj:introBM} holds with $\cZ_\sigma(\rhobar)$ taken to be the support cycle of $M_\infty(\sigma)$. Unfortunately, little seems to be known about $\supp \, M_\infty(\sigma(\tau)^\circ \otimes_{\cO} V(\lambda-\eta)^\circ)$ beyond the $\GL_2(\Q_p)$ case, and we are unable to make this work for all $\rhobar$ (even for generic tame $\tau$).

To prove Theorem \ref{thm:introBMrhobar}, our first step is to establish it when $\rhobar$ is tame, where we can show that indeed $\supp \, M_\infty(\sigma(\tau)^\circ \otimes_{\cO} V(\lambda-\eta)^\circ) = \Spec R_\rhobar^{\lambda, \tau}$.
The key input is the following theorem, which follows from the corresponding result for the local models (Theorems \ref{thm:intro:unibranchMLM} and  \ref{thm:intro:LMD}):
\begin{thm}[Theorem \ref{thm:stack_local_model}] \label{thm:intro:unibranch}
Assume $K/\Qp$ is unramified.  Let $\lambda \in (\Z^n)^{\Hom_{\Qp}(K, \overline{\Q}_p)}$ be a collection of regular Hodge--Tate weights, $\tau$ be a tame inertial type for $K$, and $\rhobar:G_K \ra \GL_n(\F)$ be an $n$-dimensional representation.   If $\tau$ is sufficiently generic  and $\rhobar$ is tame then $R_{\rhobar}^{\lambda, \tau}$ is a domain $($or zero$)$.  
\end{thm}

\begin{rmk}  \begin{enumerate}
\item
Theorem \ref{thm:intro:unibranch} does not hold in general without the tameness assumption, for example when $\lambda$ corresponds to $(3,1,0)$ in all embeddding $K\into \ovl{\Q}_p$,
for every generic $\tau$  there is some wild $\rhobar$ where it fails, see Corollary \ref{cor:numerical_main_theorem} and Corollary \ref{cor:numerical_unibranch}. This is the reason for the tameness assumption here and in the global applications in \S \ref{sec:intro:WPSC}. However, the theorem does hold for possibly wild $\rhobar$ in several situations such as: $n =2$, $n=3$ and $\lambda= \eta$ (see \cite[Theorem 7.2.1]{EGS}, \cite[Corollary 3.3.3]{GL3Wild}),
or when $n$ is arbitrary and $\rhobar$ has specific \emph{shapes} relative to $\tau$.
\item If $R^{\lambda,\tau}_{\rhobar} \neq 0$, then sufficient genericity of $\tau$ implies that of $\rhobar$ and vice versa (generally with different choices of universal polynomial). Because of this, the conclusion of Theorem \ref{thm:intro:unibranch} also holds if we let $\rhobar$ be tame and sufficiently generic but impose no hypothesis on $\tau$ \cite[Theorem 4.0.1]{OBW}.
\end{enumerate}
\end{rmk}
Theorem \ref{thm:intro:unibranch} immediately implies Theorem \ref{thm:introBMrhobar} for tame $\rhobar$: The results of \cite{LLL} imply that if $\tau$ is sufficiently generic, $M_\infty(\sigma(\tau)^\circ\otimes_{\cO} V(\lambda-\eta)^\circ)$ is nonzero if and only if $R_\rhobar^{\lambda,\tau}$ is nonzero. Since the support of $M_\infty(\sigma(\tau)^\circ\otimes_{\cO} V(\lambda-\eta)^\circ)$ must be a union of irreducible components and $R_\rhobar^{\lambda,\tau}$ is a domain, the support is everything.

Having proven Theorem \ref{thm:introBMrhobar} for tame $\rhobar$, the second step is to spread the result to all $\rhobar$. This is achieved through the use of the Emerton--Gee stack $\cX_n$. For regular $\lambda$, the special fiber of the substack $\cX^{\lambda,\tau}$ is supported on a union of irreducible components of $\cX_{n,\red}$. Thus we can associate to it a top-dimensional cycle $\cZ_{\lambda, \tau}$ on $\cX_{n, \mathrm{red}}$ by recording the irreducible components of $\cX^{\lambda,\tau}_{\red}$ weighted by their multiplicities.   In \cite{EGstack}, Emerton and Gee formulate a Breuil--M\'ezard conjecture on the stack $\cX_n$:   

\begin{conj}[Conjecture 8.2.2 \cite{EGstack}]\label{conj:intro:BMstack}  %
For each Serre weight $\sigma$, there exists an effective top-dimensional cycle $\cZ_\sigma$ on $\cX_{n,\mathrm{red}}$ such that for all regular $\lambda$ and inertial types $\tau$, we have
\[
\cZ_{\lambda,\tau} = \sum_{\sigma} [\ovl{\sigma(\tau) \otimes_E V(\lambda-\eta)}:\sigma] \cZ_\sigma.
\]
\end{conj}

We first remark that the potentially crystalline case of  Conjecture \ref{conj:introBM} is a consequence of Conjecture \ref{conj:intro:BMstack} by completing at $\rhobar$ and pulling back the cycles.   While it has been understood by experts that the Breuil--M\'ezard conjecture should behave well as $\rhobar$ varies, the Emerton--Gee stack makes it possible to study the conjecture by interpolation. %
 We refer to the conjectural cycles $\cZ_{\sigma}$ as \emph{Breuil--M\'ezard cycles}.   As in Conjecture \ref{conj:introBM}, the system of equations in Conjecture \ref{conj:intro:BMstack} for varying $\lambda$ and $\tau$ over-determines the Breuil--M\'ezard cycles.  
 
 \begin{rmk} 
Caraiani--Emerton--Gee--Savitt \cite{CEGS} recently proved Conjecture \ref{conj:intro:BMstack} in the potentially Barsotti--Tate case (i.e.~in parallel weight $(1, 0)$) when $n =2$ for any extension $K/\Qp$.  The proof uses both the weight part of Serre's conjecture for $\GL_2$ proved by Gee--Liu--Savitt \cite{GLS, GLS15} and the Breuil--M\'ezard conjecture for potentially Barsotti--Tate representations established by Gee--Kisin \cite{Gee-Kisin}. 
 \end{rmk}

By interpolating from Theorem \ref{thm:introBMrhobar} for tame $\rhobar$, we establish a portion of Conjecture \ref{conj:intro:BMstack}.   

 \begin{thm}[Corollary \ref{cor:PgenBM}]
 \label{thm:intro:BMstack}   Assume $K/\Qp$ is unramified. Fix a finite subset $\Lambda \subset (\Z^n)^{\Hom_{\Qp}(K, \overline{\Q}_p)}$ of regular dominant weights. 
There exists a top-dimensional {effective} cycle $\cZ_{\sigma}$ on $\cX_{n, \red}$ for each 
Serre weight $\sigma$ such that for all $\lambda \in \Lambda$ and all sufficiently generic \emph{(}depending on $\Lambda$\emph{)} tame inertial types $\tau$,  
 \[
 \cZ_{\lambda,\tau} = \sum_{\sigma} [\ovl{\sigma(\tau) \otimes_E V(\lambda-\eta)}:\sigma] \cZ_\sigma.
 \] 
\end{thm}
 
\begin{rmk}\label{rmk:intro:characterize}
\begin{enumerate}
\item In contrast to Theorem \ref{thm:introBMrhobar}, the set of $\ovl{\sigma(\tau) \otimes_E V(\lambda-\eta)}$ appearing in Theorem \ref{thm:intro:BMstack} does not span the Grothendieck group of $\GL_n(\cO_K)$-representations, and so it is not immediately apparent that the cycles $\cZ_\sigma$ are uniquely determined. However, though neither stated in nor implied by the theorem, the $\cZ_\sigma$ we construct satisfy a compatibility with patching functors after localizing at tame $\rhobar$ as in Remark \ref{rmk:intro:BMcycle}\eqref{it:intro:BMcycle:3}. It is this compatibility that characterizes the $\cZ_\sigma$ (see also Remark \ref{rmk:algorithm_Zsigma} for an algorithm to compute it without choosing patching functors).

Furthermore, if we assume an extension of Theorem \ref{thm:intro:BMstack} to a sufficiently large spanning set, then the cycles from this extension must agree with the $\cZ_\sigma$ we construct for sufficiently generic $\sigma$, cf.~ Theorem \ref{thm:genBM}. %
With this understanding, we can freely invoke the cycles $\cZ_\sigma$ for sufficiently generic $\sigma$ in our discussion.
\item Even though the equations in the theorem on their own do not suffice to determine all the $\cZ_\sigma$ that occurs in them, they do determine a subset of $\cZ_{\sigma}$ for which $\sigma$ is sufficiently generic, cf.~ Proposition \ref{prop:BMcycles_underdetermined} and Remark \ref{rmk:BMcycles_underdetermined}.
\end{enumerate}
\end{rmk}

 We now explain the idea of the proof of Theorem \ref{thm:intro:BMstack}. We first note that one can invert the equations in Conjecture \ref{conj:intro:BMstack}, and get a candidate for the cycle $\cZ_{\sigma}$: any expression of $\sigma$ in the Grothendieck group of $\GL_n(\cO_K)$-representations as a linear combination of reductions of $\GL_n(\cO_K)$-representations gives a candidate as a linear combination of the $\cZ_{\lambda,\tau}$.  But there is no a priori reason for these candidates to satisfy all the required cycle equations. However, for tame $\rhobar$, the compatibility of $\cZ_{\sigma}(\rhobar)$ with patching functors, cf.~Remark \ref{rmk:intro:BMcycle}(\ref{it:intro:BMcycle:3}), shows that the candidate cycle $\cZ_\sigma$ recovers the already constructed $\cZ_\sigma(\rhobar)$. This implies the equations hold after completion at tame $\rhobar$, and we conclude because there are enough tame $\rhobar$ to detect equality of cycles in $\cX_n$.  

At this point, the proof of Theorem \ref{thm:introBMrhobar} is almost complete. The subtlety is that Theorem \ref{thm:intro:BMstack} only controls the cycles $\cZ_{\lambda, \tau}$ for $\tau$ sufficiently generic.
To deal with this, we invoke a result of \cite{LLL}, which shows that a sufficiently generic (depending on $\lambda$) $\rhobar$ lies in $\cZ_{\lambda, \tau}$ only if $\tau$ is sufficiently generic. This allows us to check the equations in Theorem \ref{thm:introBMrhobar} not covered by Theorem \ref{thm:intro:BMstack}, by showing that they reduce to $0=0$.   

\begin{rmk} \label{rmk:intro:cyclemult} We can certainly write $\cZ_{\sigma} = \sum_{\sigma'}  b_{\sigma', \sigma} \cC_{\sigma'}$, and it is natural to ask what the coefficients $b_{\sigma', \sigma}$ are.
We prove that $b_{\sigma, \sigma} = 1$, and that $b_{\sigma', \sigma} \neq 0$ implies a restrictive relation between $\sigma$ and $\sigma'$, namely that $\sigma$ \emph{covers} $\sigma'$ (Definition \ref{defn:cover}).   When $n =2, 3$, we have $\cZ_{\sigma} = \cC_{\sigma}$ (with mild genericity assumptions).    In \S \ref{sec:BMbasis}, we describe an inductive algorithm for computing $\cZ_{\sigma}$ if one knows the cycles $\cZ_{\eta, \tau}$ for enough $\tau$. In turn, the cycles $\cZ_{\eta, \tau}$ can be computed using the local model $M_{\cJ}(\eta, \nabla_{\bf{a}_\tau})$, introduced in \S \ref{sec:intro:MLM} above, which is an ``explicit" algebraic variety. This algorithm can in theory be implemented on a computer. We have performed computer experiments when $n = 4$, which indicate that $\cZ_{\sigma}$ is not always irreducible.  
We remark that in the analogous situation of \cite{BHS},  the locally analytic Breuil--M\'ezard cycles are also not irreducible in general, beginning with $n = 8$. 
\end{rmk}

\subsection{The weight part of Serre's conjecture} \label{sec:intro:WPSC}
Serre's conjecture \cite{serre-duke} predicts that any odd irreducible two-dimensional mod $p$ Galois representation arises from a modular form, and moreover predicts the minimal level and weight of such a form.
There has been substantial progress in formulating and proving generalizations of this conjecture in higher rank.
While generalizations of the notion of minimal level are rather straightforward, generalizations of the weight part of Serre's conjecture are far from it.
Herzig \cite{herzig-duke} introduced a representation theoretic generalization in the generic tame case, which we now discuss in the context of definite unitary groups.

Let $F$ be an imaginary CM number field unramified at $p$.  Let $F^+$ be the maximal totally real subfield.  Assume $F^+ \neq \Q$ and that all primes of $F^+$ above $p$ split in $F$.  
Let $G$ be a unitary group over $F^+$ which splits over $F$ and is isomorphic to $U(n)$ at each infinite place. %
Let $K^p\subset G(\A_{F^+}^{\infty,p})$ be a compact open subgroup, and let $S(K^p,\F)$ be the space of $\F$-valued locally constant functions on $G(F^+) \backslash G(\A_{F^+}^\infty)/K^p$.
Then $S(K^p,\F)$ has an action of a spherical Hecke algebra $\bT$ (away from $p$ and finitely many other places).
If $\fm \subset \bT$ is a maximal ideal such that $S(K^p,\F)_{\fm}$ is nonzero, then there is a unique semisimple Galois representation $\rbar: G_F \ra \GL_n(\F)$ up to conjugation which matches $\fm$ via the Satake isomorphism.
We say that $\rbar$ is \emph{automorphic}.

Fix places $\tld{v}$ of $F$ lying over $v$ for each place $v \mid p$ of $F^+$ which together give an isomorphism $G(\cO_{F^+}\otimes_{\Z} \Z_p) \cong \GL_n(\cO_{F^+}\otimes_{\Z} \Z_p)$.
A \emph{global Serre weight} is an irreducible smooth $\F$-representation $V$ of $\GL_n(\cO_{F^+}\otimes_{\Z} \Z_p)$. Any such representation has the form $\otimes_{v \mid p} V_v$ with $V_v$ an irreducible representation of $\GL_n(k_v)$ where $k_v$ is the residue field of $F^+$ at $v$. 
We say $\overline{r}$ is \emph{modular} of (global Serre) weight $V$ if $\Hom_{\GL_n(\cO_{F^+}\otimes_{\Z} \Z_p)}(V,S(K^p,\F)_{\fm})$ is nonzero. %

In 2009, Herzig (\cite{herzig-duke}) conjectured (in the context of locally symmetric spaces for $\GL_n$) what the set $W(\rbar)$ of modular weights should be when $\rbar$ is tame at places above $p$, generalizing conjectures of Serre and Buzzard--Diamond--Jarvis (\cite{serre-duke, BDJ}).  These conjectures are collectively referred to as the \emph{weight part of Serre's conjecture}. 
For the reader's convenience, we restate Theorem \ref{thm:i2i:SWC}, our main result towards (the analog for definite unitary groups of) Herzig's conjecture.

\begin{thm}[Theorem \ref{thm:SWC}]\label{thm:intro:SWC} Suppose that 
\begin{itemize}
\item $\rbar: G_{F} \ra \GL_n(\F)$ is automorphic;
\item $\rbar(G_{F(\zeta_p)})$ is adequate; and that 
\item for each $v \mid p$, $\rbar_v \defeq \overline{r}|_{\Gal(\overline{F}_{\tld{v}}/F_{\tld{v}})}$ is tame and sufficiently generic \emph{(}and in particular that $p\nmid 2n$\emph{)}.
\end{itemize}
Then 
\[
\overline{r} \text{ is modular of weight } \otimes_{v \mid p} V_v   \iff V_v \in W^{?}(\rbar_v) \text{ for all } v \mid p
\]
where $W^?(\rbar_v)$ is the explicit collection of Serre weights defined by \cite{herzig-duke}.  
\end{thm} 

When $n =3$, the theorem with an explicit combinatorial genericity condition was proven in \cite{LLLM2}. For general $n$, the forward direction, known as \emph{weight elimination}, was proven in \cite{LLL}, again with an explicit combinatorial genericity condition. 
 The reverse direction is a statement about mod $p$ modularity, and is much harder. Its content is essentially the construction of all possible congruences between mod $p$ automorphic forms. 
One difficulty is that, in contrast to when $n\leq 2$, Serre weights typically do not admit characteristic zero lifts and so the set $W(\rbar)$ cannot be interpreted in terms of the existence of automorphic lifts of prescribed types.
As a result, $W(\rbar)$ does not have an apparent Galois theoretic meaning, while at the same time its complexity grows rapidly with $n$. 

The tameness hypothesis in Theorem \ref{thm:SWC} is natural because the restrictions to inertia of tame Galois representations can be parametrized combinatorially, and this parametrization plays a central role in Herzig's recipe.
On the contrary, a combinatorial parametrization of all Galois representations
is not possible, as is reflected by the geometry of the Emerton--Gee stack. Thus one cannot expect explicit formulas for $W(\rbar)$, rather, it should depend on the position of the local Galois representations in their moduli. At the same time, the non-liftability of Serre weights to characteristic zero makes it difficult to pin down such a geometric recipe in terms of $p$-adic Hodge theory. For these reasons, there has been no unconditional formulation of the weight part of Serre's conjecture in the wildly ramified case when $n>2$.
However, as observed in \cite{GHS}, the Breuil--Mez\'{a}rd conjecture can be used to resolve the above difficulties. We make the following definition.%

\begin{itemize}
\item Assume Conjecture \ref{conj:introBM} holds. Define $W^{\mathrm{BM}}(\rhobar)$ to be the set of $\sigma$ such that $\cZ_\sigma(\rhobar)\neq 0$.
\end{itemize}
The set $W^{\mathrm{BM}}(\rhobar)$ has some relation to characteristic zero: as one can always lift Serre weights \emph{virtually}, we can as before invert the equations in Conjectures \ref{conj:introBM} and \ref{conj:intro:BMstack}, and understand $\cZ_\sigma(\rhobar)$ in terms of characteristic zero $p$-adic Hodge theoretic conditions.

On the other hand, as little is known about the geometry of the cycle $\cZ_\sigma(\rhobar)$, it is helpful to also define the following set of weights.
\begin{itemize}
\item
We say that $\sigma$ is a \emph{geometric} Serre weight of $\rhobar$ if $\rhobar$ lies on $\cC_\sigma$.
We let $W^{\mathrm{g}}(\rhobar)$ be the set of geometric Serre weights of $\rhobar$. 
\end{itemize}
Observe that both $W^{\mathrm{g}}(\rhobar)$ and $W^{\mathrm{BM}}(\rhobar)$ are geometric in nature, and are also defined for wildly ramified representations.
Unlike $W^{\mathrm{g}}(\rhobar)$ and $W^?(\rhobar)$, for $W^{\mathrm{BM}}(\rhobar)$ to be meaningful, one requires the knowledge of Conjecture \ref{conj:introBM} at least for sufficiently many $(\lambda,\tau)$ to pin down $\cZ_\sigma(\rhobar)$ uniquely. In particular, our result on the Breuil--M\'{e}zard conjecture (Theorem \ref{thm:introBMrhobar}) allows us to formulate the following unconditional version of a conjecture of Gee--Herzig--Savitt \cite[Conjecture 3.2.7]{GHS}: %
\begin{conj}\label{conj:intro:GHS}
Suppose that $\rbar: G_{F} \ra \GL_n(\F)$ is automorphic.  Let $V = \otimes_{v \mid p} V_v$ be a global Serre weight.   Assume that for each $v \mid p$, $\ovl{r}_v$ is sufficiently generic, then
\[
\overline{r} \text{ is modular of weight } V   \iff V_v \in W^{\mathrm{BM}}(\rbar_v) \text{ for all } v \mid p.
\]
\end{conj}      
\begin{rmk} \begin{enumerate}
\item 
For $\rhobar$ sufficiently generic, one sees that $W^{\mathrm{BM}}(\rhobar)$ consists of exactly the (neccessarily sufficiently generic) $\sigma$ such that $\rhobar$ lies in the support of the cycle $\cZ_\sigma$. In other words, the discrepancy between $W^{\mathrm{g}}(\rhobar)$ and $W^{\mathrm{BM}}(\rhobar)$ is exactly the discrepancy between the irreducible component $\cC_\sigma$ of $\cX_n$ and the Breuil--Mez\'{a}rd cycle $\cZ_\sigma$.
\item For sufficiently generic $\rhobar$, one has $W^{\mathrm{g}}(\rhobar) \subset W^{\mathrm{BM}}(\rhobar) \subset W^{?}(\rhobar^{\mathrm{ss}})$. The first inclusion is because $\cC_\sigma$ belongs to the support of $\cZ_\sigma$, and we expect this inclusion to be strict in general (cf.~Remark \ref{rmk:intro:cyclemult}). The second inclusion follows from \cite{LLL}, and is always strict when $\rhobar$ is wildly ramified (this is explored in \cite{OBW}). %
\end{enumerate}
\end{rmk}

We now discuss the proof of Theorem \ref{thm:intro:SWC}. We apply Taylor--Wiles patching in our given global context. The modularity of a global Serre weight $V$ is equivalent to the non-vanishing of the associated patched module $M_{\infty}(V)$.
Recall that $W(\rbar)$ denotes the set of modular global Serre weights and we assumed that each $\rbar_v$ is tame and sufficiently generic. At this point, Theorem \ref{thm:introBMrhobar} (or rather, the compatibility with patching functors, see Remark \ref{rmk:intro:BMcycle}) immediately implies Conjecture \ref{conj:intro:GHS} for our tame $\rbar$. However, this is not sufficient for Theorem \ref{thm:intro:SWC}, because of the very mysterious nature of the Breuil--M\'ezard cycles which makes it difficult to show that  $W^{\mathrm{BM}}(\rbar_v) = W^?(\rbar_v)$. Instead, we observe that by the chain of inclusions 
\[
\otimes_{v \mid p} W^{\mathrm{g}}(\rbar_v) \subset W(\rbar) \subset \otimes_{v \mid p} W^?(\rbar_v),
\]    
it suffices to show $W^{\mathrm{g}}(\rbar_v)= W^?(\rbar_v)$. This is more accessible, since $W^{\mathrm{g}}(\rbar_v)$ is expressed in terms of $\cC_\sigma$, which has a transparent geometric meaning, while $W^?(\rbar_v)$ is combinatorially explicit. Using Theorem \ref{thm:intro:irreducible_components_mod_p}, we relate the relevant components $\cC_{\sigma}$ to irreducible components $C_{\sigma}$ of a deformed affine Springer fiber and tame local Galois representations to torus fixed points in the affine flag variety.  
Showing that $ W^{\mathrm{g}}(\rbar_v) = W^?(\rbar_v)$ turns out to be equivalent to showing that the set of torus fixed points of $C_{\sigma}$ achieves the obvious upper bound. Fortunately, Boixeda Alvarez \cite{Pablo} proved the analogous fact for irreducible components of affine Springer fibers, and a simple spreading out argument allows us to transfer his result back to our deformed affine Springer fibers. 

Finally, we remark that, in contrast to \cite[Theorem 7.8]{LLLM} and \cite[Theorem 5.3.1]{LLLM2} for example, we make no assumptions on the ramification of $\rbar$ outside of $p$ in Theorem \ref{thm:intro:SWC}. 
This is possible because our results on the geometric formulations, rather than the original numerical formulation, of the Breuil--M\'ezard conjecture allow for more robust arguments in the Taylor--Wiles method. 
(In fact, for this reason our arguments are slightly more involved than what we describe above.)
In the literature, the Taylor--Wiles patching construction typically takes place after applying solvable base change theorems. 
While this is convenient for the purposes of modularity lifting theorems, in the interest of reducing the hypotheses on our results on the weight part of Serre's conjecture, we describe in the appendix \S \ref{app:TW:patch} the modifications necessary to apply the Taylor--Wiles method at arbitrary sufficiently small level.

\subsection{Outline of the paper}
\label{subsec:overview}
We give a brief overview of the various sections of this paper.

A reader primarily interested in the geometry of the local model and its relationship to the Emerton--Gee stack can read \S \ref{sec:UMLM} (or perhaps just Theorem  \ref{thm:model_unibranch}), \ref{sec:MLM}, \ref{sec:BK:LM}, and \ref{sec:monodromy}, referring to \S \ref{sec:GM} as desired.
A reader primarily interested in our main applications can read \S \ref{sec:BMC} and \ref{sec:GA}, referring to the main results of \S \ref{sec:monodromy}.
\S  \ref{sec:Prel} is preliminary and can be referred to as needed.

\S  \ref{sec:Prel} establishes various connections between extended affine Weyl groups and representation theory used throughout the paper.

\S \ref{sec:UMLM} is the technical heart of the paper.
We introduce a universal version of the local model (\S \ref{subsec:UMLM}) and describe some of its basic properties. %
The  \emph{unibranch} property at torus fixed points is established in \S \ref{subsec:equal_char_unibranch} and the subsequent sections \S \ref{sec:Spr:Out}, \ref{subsec:Sec}, \ref{subsec:prod:UMLM} deal with the problem of spreading out such properties.  The most important result is Theorem \ref{thm:model_unibranch} on the unibranch property used in the main theorem on Galois deformations (Theorem \ref{thm:stack_local_model}).  

\S \ref{sec:MLM} specializes the universal model to the mixed characteristic situation of interest and then studies the special fiber.  The main result (Theorem \ref{thm:compandSW}) uses reductions of Deligne--Lusztig representations to parametrize the irreducible components of the special fiber compatibly (\S \ref{sec:componentmatching}) over varying parameters. 
Finally, Theorem \ref{thm:Tfixedpts} is the main result on torus fixed points of irreducible components used in the proof of the weight part of Serre's conjecture.

\S \ref{sec:BK:LM} compares stacks of Breuil--Kisin modules and Pappas--Zhu local models in preparation for \S \ref{sec:monodromy}.  The main result (Theorem \ref{thm:Breuil-Kisin_local_model}) is the local model diagram for (Zariski covers) of the moduli stack of Breuil--Kisin modules with tame descent data and a Pappas--Zhu local model via a theory of gauge bases for Breuil--Kisin modules (see \S \ref{sub:GBases}, particularly Proposition \ref{prop:gauge_basis_existence}).  We also establish directly a connection in characteristic $p$ to the moduli stack of \'etale $\phz$-modules in Proposition \ref{prop:phi_local_model}.

\S \ref{sec:GM} is an interlude on patching functors. %
Here, global methods are used to show the existence of local lifts of various types, which provides a key input into component counts in \S \ref{sec:monodromy}. %

\S \ref{sec:monodromy} contains the main result (Theorem \ref{thm:stack_local_model}) on the relation between the local models and Galois deformations used in the proof of both the Breuil--M\'ezard conjecture and the weight part of Serre's conjecture.
The \emph{monodromy condition}, in particular its algebrization (Proposition \ref{prop:monodromy_approximation}), is studied in \S \ref{sec:mon:cond}. %
Theorem \ref{prop:stack_diagram} compares Emerton--Gee stacks of potentially crystalline representations with the moduli stacks of Breuil--Kisin modules with tame descent data studied in \S \ref{sec:BK:LM}.  
Finally, \S \ref{sec:local_model_EG} is the culmination of the earlier sections establishing the comparison between the tame potentially crystalline Emerton--Gee stack and the local models of \S \ref{sec:MLM}.
Theorem \ref{thm:irreducible_components_mod_p} describes a sufficiently generic portion of the reduced special fiber of the Emerton--Gee stack.

\S \ref{sub:statements:BM} introduces versions of the Breuil--M\'ezard conjectures, and \S \ref{sub:rel:BM} describes their relationship.
\S \ref{sec:patchBM} provides an axiomatic framework to prove restricted versions of the Breuil--M\'ezard conjectures using patching functors, which is then applied in \S \ref{sec:BMgen} (see Theorem \ref{thm:genBM} and Corollary \ref{cor:BM}).
In \S \ref{sec:gen:BM:basis}, we describe basic properties of Breuil--M\'ezard cycles and an algorithm to compute them. %

Applications to the Serre weight conjecture (Theorem \ref{thm:SWC}) for certain definite unitary groups and modularity lifting are in \S \ref{sec:unitary} and \ref{subsec:MLT}, respectively.
\S \ref{app:TW:patch} describes routine modifications to the Taylor--Wiles method needed to patch at arbitrary level.

\subsection{Acknowledgments}
The genesis of this article covers several years beginning in 2016 at the University of Chicago.
Part of the work has been carried out during visits at the Institut Henri Poinar\'e (2016), CIRM (2017), IAS (2017), Centro di Ricerca Matematica Ennio de Giorgi (2018), Mathematisches Forschungsinstitut Oberwolfach (2019), The University of Arizona, Northwestern University.
We would like to heartily thank these institutions for the outstanding research conditions they provided, and for their support.

For various discussions related to this work, we thank Toby Gee, Florian Herzig, Matthew Emerton, Pablo Alvarez Boixeda. B. LH. would like to thank Rong Zhou for a conversation that lead to decisive progress in the project. We would also like to thank Christophe Breuil and Toby Gee for comments on an earlier draft of this paper. %
We thank the referee for extensive comments that improved the readability of the paper.

Finally, D.L. was supported by the National Science Foundation under agreements Nos.~DMS-1128155 and DMS-1703182 and an AMS-Simons travel grant. B.LH. acknowledges support from the National Science Foundation under grant Nos.~DMS-1128155, DMS-1802037 and the Alfred P. Sloan Foundation. B.L. was supported by Simons Foundation/SFARI (No.~585753), and S.M. by the ANR-18-CE40-0026 (CLap CLap).

\subsection{Notation}
\label{sec:notations}

We fix once and for all a separable closure $\ovl{K}$ of every field $K$ and let $G_K \defeq \Gal(\ovl{K}/K)$.
If $K$ is defined as a subfield of an algebraically closed field, then we set $\ovl{K}$ to be this field.
If $K$ is a nonarchimedean local field, we let $I_K \subset G_K$ denote the inertial subgroup and $W_K \subset G_K$ denote the Weil group.
We fix a prime $p\in\Z_{>0}$.
Let $E \subset \ovl{\Q}_p$ be a subfield which is finite-dimensional over $\Q_p$.
We write $\cO$ to denote its ring of integers, fix an uniformizer $\varpi\in \cO$ and let $\F$ denote the residue field of $E$.
We will assume throughout that $E$ is sufficiently large.

\subsubsection{Reductive groups}
\label{sec:not:RG}
Let $G$ denote a split connected reductive group (over some ring) together with a Borel $B$, a maximal split torus $T \subset B$, and $Z \subset T$ the center of $G$.  
Let $d = \dim G - \dim B$.  
When $G$ is a product of copies of $\GL_n$, we will take $B$ to be upper triangular Borel and $T$ the diagonal torus. 
Let $\Phi^{+} \subset \Phi$ (resp. $\Phi^{\vee, +} \subset \Phi^{\vee}$) denote the subset of positive roots (resp.~positive coroots) in the set of roots (resp.~coroots) for $(G, B, T)$. 
Let $\Delta$ (resp.~$\Delta^{\vee}$) be the set of simple roots (resp.~coroots).
Let $X^*(T)$ be the group of characters of $T$ and $\Lambda_R \subset X^*(T)$ denote the root lattice for $G$.

For a free $\Z$-module $M$ of finite rank (e.g. $M=X^*(T)$), the duality pairing between $M$ and its $\Z$-linear dual $M^*$ will be denoted by $\langle \ ,\,\rangle$.
If $A$ is any ring, the pairing $\langle \ ,\,\rangle$ extends by $A$-linearity to a pairing between $M\otimes_{\Z}A$ and $M^*\otimes_{\Z}A$. %

We say that a weight $\lambda\in X^*(T)$ is \emph{dominant} (resp.~\emph{regular dominant}) if $0\leq \langle\lambda,\alpha^\vee\rangle$ (resp.~$0< \langle\lambda,\alpha^\vee\rangle$) for all $\alpha\in \Delta$.
For $\lambda\in X^*(\un{T})$, set $h_\lambda\defeq \underset{\alpha\in \Phi}{\max}\{\langle \lambda,\alpha^\vee\rangle\}$.
Set $X^0(T)$ to be the subgroup consisting of characters $\lambda\in X^*(T)$ such that $\langle\lambda,\alpha^\vee\rangle=0$ for all $\alpha\in \Delta$.

Let  $W(G)$ denote the Weyl group of $(G,T)$.  Let $w_0$ denote the longest element of $W(G)$.
We sometimes write $W$ for $W(G)$ when there is no chance for confusion.
Let $W_a$ (resp.~$\tld{W}$) denote the affine Weyl group and extended affine Weyl group 
\[
W_a = \Lambda_R \rtimes W(G), \quad \tld{W} = X^*(T) \rtimes W(G)
\]
for $G$.
We use $t_{\nu} \in \tld{W}$ to denote the image of $\nu \in X^*(T)$. 

The Weyl groups $W(G)$, $\tld{W}$, and $W_a$ act naturally on $X^*(T)$ and on $X^*(T)\otimes_{\Z} A$ by extension of scalars for any ring $A$.
Given $\lambda\in X^*(T)$, we write $\mathrm{Conv}(\lambda)$ for the convex hull of the subset $\big\{w(\lambda) \mid w\in W(G)\big\}\subset X^*(T)$.

We write $G^\vee = G^\vee_{/\Z}$ for the split connected reductive group over $\Z$  defined by the root datum $(X_*(T),X^*(T), \Phi^\vee,\Phi)$. 
This defines a maximal split torus $T^\vee\subseteq G^\vee$ such that we have canonical identifications $X^*(T^\vee)\cong X_*(T)$ and $X_*(T^\vee)\cong X^*(T)$.

For $(\alpha,k)\in \Phi \times \Z$, we have the root hyperplane $H_{\alpha,k}\defeq \{\lambda:\ \langle\lambda,\alpha^\vee\rangle=k\}$. 
An alcove %
is a connected component of $X^*(T)\otimes_{\Z}\R\ \setminus\ \big(\bigcup_{(\alpha,n)}H_{\alpha,n}\big)$. 
We say that an alcove $A$ is \emph{restricted} if $0<\langle\lambda,\alpha^\vee\rangle<1$ for all $\alpha\in \Delta$ and $\lambda\in A$.
We let $A_0$ denote the (dominant) base alcove, i.e.~the set of $\lambda\in X^*(T)\otimes_{\Z}\R$ such that $0<\langle\lambda,\alpha^\vee\rangle<1$ for all $\alpha\in \Phi^+$. 
Let $\cA$ denote the set of alcoves. 
Recall that $\tld{W}$ acts transitively on the set of alcoves, and $\tld{W}\cong\tld{W}_a\rtimes \Omega$ where $\Omega$ is the stabilizer of $A_0$.
We define
\[\tld{W}^+\defeq\{\tld{w}\in \tld{W}:\tld{w}(A_0) \textrm{ is dominant}\}.\]
and
\[\tld{W}^+_1\defeq\{\tld{w}\in \tld{W}^+:\tld{w}(A_0) \textrm{ is restricted}\}.\]
We fix an element $\eta\in X^*(T)$ such that $\langle \eta,\alpha^\vee\rangle = 1$ for all positive simple roots $\alpha$ and let $\tld{w}_h$ be $w_0 t_{-\eta}\in \tld{W}^+_1$.%

When $G = \GL_n$, we fix an isomorphism $X^*(T) \cong \Z^n$ in the standard way, where the standard $i$-th basis element $\eps_i=(0,\ldots, 1,\ldots, 0)$ (with the $1$ in the $i$-th position) of the right-hand side corresponds to extracting the $i$-th diagonal entry of a diagonal matrix. Dually we get a standard isomorphism $X_*(T)\cong \Z^n$, and let $(\eps_1^\vee,\dots,\eps_n^\vee)$ denote the dual basis.

Suppose that $G$ is a split connected reductive group over $\Z_p$.
Let $\cO_p$ be a finite \'etale $\Z_p$-algebra, which is necessarily isomorphic to a product $\prod\limits_{v\in S_p} \cO_{v}$ where $S_p$ is a finite set and $\cO_{v}$ is the ring of integers of a finite unramified extension $F^+_{v}$ of $\Q_p$.
For example, we will take $\cO_p$ to be the ring of integers in an unramified extension of $\Q_p$ or $\cO_{F^+} \otimes_{\Z} \Z_p$ where $F^+$ is a number field in which $p$ is unramified and $\cO_{F^+}$ is its ring of integers.
Let $G_0 = \Res_{\cO_p/\Z_p} (G_{/\cO_p})$ with Borel subgroup $B_0 =  \Res_{\cO_p/\Z_p} (B_{/\cO_p})$, maximal torus $T_0 = \Res_{\cO_p/\Z_p} (T_{/\cO_p})$, and $Z_0 = \Res_{\cO_p/\Z_p} (Z_{/\cO_p})$. 
Assume that $\cO$ contains the image of any ring homomorphism $\cO_p \ra \ovl{\Z}_p$.
Let $\cJ$ be $\Hom_{\Zp}(\cO_p,\cO)$.
Then $\un{G} \defeq (G_0)_{/\cO}$ is naturally identified with the split reductive group $G_{/\cO}^{\cJ}$. 
We similarly define $\un{B}, \un{T},$ and $\un{Z}$. %
Corresponding to $(\un{G}, \un{B}, \un{T})$, we have the set of positive roots $\un{\Phi}^+ \subset \un{\Phi}$ and the set of positive coroots $\un{\Phi}^{\vee, +}\subset \un{\Phi}^{\vee}$.
The notations $\un{\Lambda}_R$, $\un{W}$, $\un{W}_a$, $\tld{\un{W}}$, $\tld{\un{W}}^+$, $\tld{\un{W}}^+_1$, $\un{\Omega}$ should be clear as should the natural isomorphisms $X^*(\un{T}) = X^*(T)^{\cJ}$ and the like.  When $G = \GL_n$, then we fix $\eta \in X^*(\un{T})$  to be the product of the elements $(n-1, n-2, \ldots, 0) \in \Z^n$.

The absolute Frobenius automorphism on $\cO_p/p$ lifts canonically to an automorphism $\varphi$ of $\cO_p$. We define an automorphism $\pi$ of the identified groups $X^*(\un{T})$ and $X_*(\un{T}^\vee)$ by the formula $\pi(\lambda)_\sigma = \lambda_{\sigma \circ \varphi^{-1}}$ for all $\lambda\in X^*(\un{T})$ and $\sigma: \cO_p \ra \cO$.
We assume that, in this case, the element $\eta\in X^*(\un{T})$ we fixed is $\pi$-invariant.
We similarly define an automorphism $\pi$ of $\un{W}$ and $\tld{\un{W}}$.

\subsubsection{Galois theory}\label{sec:not:Gal}

Let $\cO_p$ and $\un{G}_{/\cO}$ be as in \ref{sec:not:RG}.
Let $F^+_p$ be $\cO_p[1/p]$.
Then $F^+_p$ is isomorphic to the (finite) product $\prod\limits_{v \in S_p} F^+_{v}$ where, as above, $F^+_{v} = \cO_{v}[1/p]$ is a finite unramified extension of $\Q_p$ for each $v \in S_p$.
Let 
\[
\un{G}^\vee_{/\Z} \defeq \prod_{F^+_p \ra E} G^\vee_{/\Z} 
\]
be the dual group of $\un{G}$ so that the Langlands dual group of $G_0$ is $^L \un{G}_{/\Z} \defeq \un{G}^\vee\rtimes \Gal(E/\Q_p)$ where $\Gal(E/\Q_p)$ acts on the set of homomorphisms $F^+_p \ra E$ by post-composition.
For a topological $\cO$-algebra $A$, an $L$-homomorphism over $A$ is a continuous homomorphism $W_{\Q_p} \ra$ $^L\un{G}(A)$ with open kernel whose projection to $\Gal(E/\Q_p)$ is the natural one.
An $L$-parameter over $A$ is a ${\un{G}^\vee}(A)$-conjugacy class of $L$-homomorphisms.
An isomorphism $\ovl{F^+_v} \risom \ovl{\Q}_p$ for each $v \in S_p$ determines an embedding $G_{F^+_v} \into G_{\Q_p}$ and the restriction of this isomorphism to $F^+_v \into E$ gives a projection ${\un{G}^\vee}\ra {G}^\vee$.
Fixing isomorphisms for each $v\in S_p$, we get a bijection from the set of $L$-homomorphisms over $A$ to the set of collections of continuous representations $W_{F^+_{v}} \ra {G}^\vee(A)$ indexed by $S_p$.
This induces a bijection from the set of $L$-parameters to the set of collections of ${G}^\vee(A)$-conjugacy classes of representations $W_{F^+_{v}} \ra {G}^\vee(A)$ with open kernel indexed by $S_p$.
Moreover, this latter bijection does not depend on the choices of isomorphisms.
Finally, if $A$ is finite, this latter set is equivalent to the set of collections of ${G}^\vee(A)$-conjugacy classes of continuous representations $G_{F^+_{v}} \ra {G}^\vee(A)$ indexed by $S_p$.

An inertial $L$-homomorphism over $A$ is a continuous homomorphism $I_{\Q_p} \ra {\un{G}^\vee}(A)$ with open kernel which admits an extension to an $L$-homomorphism over $A$.
An inertial $L$-parameter over $A$ is a ${\un{G}^\vee}(A)$-conjugacy class of inertial $L$-homomorphisms.
If $K$ is a finite extension of $\Qp$, then an inertial $A$-type (for $K$) is a ${G}^\vee(A)$-conjugacy class of homomorphisms $I_K \ra {G}^\vee(A)$ with open kernels which admit extensions to homomorphisms $W_K \ra {G}^\vee(A)$.
We refer to an inertial $E$-type as just an inertial type.
We say that an inertial $L$-parameter over $A$ (resp.~inertial $A$-type) is \emph{tame} if a homomorphism (equivalently all homomorphisms) in the conjugacy class factors through the tame quotient of the inertial subgroup.
There is a similar bijection between (tame) inertial $L$-parameters over $A$ and collections of (tame) inertial $A$-types $I_{F^+_{v}} \ra G(A)$ indexed by $S_p$ (not depending on choices of isomorphisms between algebraic closures).

We now specialize to the case that $F^+_p$ is a field $K$ and $G = \GL_n$.
Let $K/\Qp$ be an unramified extension of degree $f$ with ring of integers $\cO_K$ and residue field $k$. 
Let $W(k)$ be ring of Witt vectors of $k$, which is also the ring of integers of $K$.  
We denote the arithmetic Frobenius automorphism on $W(k)$ by $\phz$; it acts as raising to $p$-th power on the residue field.

Recall that we fixed a separable closure $\ovl{K}$ of $K$.
We choose $\pi \in \ovl{K}$ such that $\pi^{p^f-1} = -p$ and define $\omega_K : G_K \ra \cO_K^\times$ to be the character defined by $g(\pi) = \omega_K(g) \pi$.
Since any choice of $\pi$ differs by a $p^f-1$-st root of unity on which $G_K$ acts trivially, $\omega_K$ is independent of the choice of $\pi$.
Given a embedding $\sigma: K \into E$, let $\omega_{K,\sigma}: G_K \ra \cO^\times$ be the character $\sigma \circ \omega_K$.
If we let $K^{\mathrm{ur}} \subset \ovl{K}$ be the maximal unramified subfield, then for any subfield $K' \subset K^{\mathrm{ur}}$ which is of finite degree over $\Q_p$, $I_{K'}$ is canonically identified with $G_{K^{\mathrm{ur}}}$. 
Thus, $I_{K'}$ is identified with $I_K$, and we also denote by $\omega_K$ and $\omega_{K,\sigma}$ the restriction of these characters to $I_{K'}$.
For any integer $r\geq 1$, we let $\Q_{p^r}$ denote the unramified degree $r$ extension of $\Q_p$ in $\ovl{\Q}_p$, which we assume is in $E$ (enlarging $E$ if necessary). 
We write $\omega_r$ for $\omega_{\Q_{p^r},\iota}$ where $\iota$ denotes the inclusion $\Q_{p^r} \subset E$ as subfields of $\ovl{\Q}_p$.
We use the overline notation $\ovl{\omega}_K$, $\ovl{\omega}_{K,\sigma}$, $\ovl{\omega}_{r}$, etc.~to denote the mod $\varpi$ reduction of ${\omega}_K$, ${\omega}_{K,\sigma}$, $\omega_{r}$, etc.
When considering $n$-dimensional representations of $G_K$, we will assume that $E$ contains the image of any morphism $K' \ra \ovl{\Q}_p$ where $K' \subset K^{\mathrm{un}}$ is the subfield of degree $r$ over $K$ where $r$ is the order of some element of $S_n$. 
Fix an embedding $\sigma_0: K \into E$. 
Then we define $\sigma_j = \sigma_0 \circ \phz^{-j}$. This identifies $\cJ = \Hom(k, \F) = \Hom_{\Qp}(K, E)$ with $\Z/f \Z$. 

For $K$ as above, we fix once and for all a sequence $\underline{p} \defeq (p_m)_{m\in\N}$ where $p_m\in \overline{K}$ satisfy $p_{m+1}^{p}= p_m$ and $p_0=-p$. 
We let $K_{\infty} \defeq \underset{m\in\N}{\bigcup}K(p_m)$ and $G_{K_{\infty}} \defeq \Gal(\overline{K}/K_{\infty})$.

Let $\eps$ denote the $p$-adic cyclotomic character.  If $W$ is a de Rham representation of $G_K$ over $E$, then for each $\kappa \in \Hom_{\Qp}(K, E)$, we write $\mathrm{HT}_{\kappa}(W)$ for the multiset of Hodge--Tate weights labelled by embedding $\kappa$ normalized so that the $p$-adic cyclotomic character $\eps$ has Hodge--Tate weight $\{1\}$ for every $\kappa$.  For $\mu = (\mu_j) \in X^*(\un{T})$,
we say that an $n$-dimensional representation $W$ has Hodge--Tate weights $\mu$ if 
\[
\mathrm{HT}_{\sigma_j}(W) = \{ \mu_{1, j}, \mu_{2, j}, \ldots, \mu_{n, j} \}.
\]
Our convention is the opposite of that of \cite{EGstack,CEGGPS}, but agrees with that of \cite{GHS}.
We will always use the covariant functors attached to $W$, for example $\DdR(W)=(W\otimes_{\bQ_p} \BdR)^{G_K}$, and similarly we have $\Dpst(W)$ and $\Dpcris(W)$. Note that under our convention, the jumps in the Hodge filtration of $\DdR(W)$ occur at the \emph{opposites} of the Hodge-Tate weights.
We say that an $n$-dimensional potentially semistable representation $\rho:G_K \ra \GL_n(E)$ has type $(\mu, \tau)$ if $\rho$ has Hodge--Tate weights $\mu$ and the Weil-Deligne representation $\mathrm{WD}(\rho)$ restricted to $I_K$ is isomorphic to the inertial type $\tau$. Note that this differs from the conventions of \cite{GHS} via a shift by $\eta$. The condition on inertial type is also equivalent to $\Dpcris(\rho)=\Dpst(\rho)$ being isomorphic to $\tau$ as $I_K$-representations.

Let $\mathrm{Art}_K: K^{\times} \ra W_K^{\mathrm{ab}}$ denote the Artin map normalized so that uniformizers map to geometric Frobenius elements. For $\tau$ an inertial type, we use $\sigma(\tau)$ to denote the finite dimensional smooth irreducible $\overline{\Q}_p$-representation of $\GL_n(\cO_K)$ associated to $\tau$ by the ``inertial local Langlands correspondence'' (see \S \ref{sec:InertialTypes}). 
In fact, in all situations that arise, $\sigma(\tau)$ will be defined over $E$.

\subsubsection{Miscellaneous}\label{sec:not:mis}

For any ring $S$, we define $M_n(S)$ to be the set of $n\times n$ matrix with entries in $S$. 
If $M\in M_n(S)$ and $A\in \GL_n(S)$ we write
\begin{equation}
\label{def:adj}
\Ad(A)(M)\defeq A\,M\,A^{-1}.
\end{equation}
If $\alpha=\eps_i-\eps_j$ is a root of $\GL_n$, we also call the $(i,j)$-th entry of a matrix $X\in M_n(S)$ the $\alpha$-th entry. We will make use of both notations $X_{ij}$ and $X_{\alpha}$ for this entry.

Let $\Gamma$ be a group.
If $V$ is a finite length $\Gamma$-representation, we let $\JH(V)$ be the (finite) set of Jordan--H\"older factors of $V$.
If $V^\circ$ is a finite $\cO$-module with a $\Gamma$-action, we write $\ovl{V}^\circ$ for the $\Gamma$-representation $V^\circ\otimes_{\cO}\F$ over $\F$.
If $\Gamma$ is a compact topological group and $V$ is a virtual representation of $\Gamma$ which is the difference $V_1 - V_2$ of two genuine continuous finite-dimensional $\Gamma$-representations over $E$, let $\ovl{V}$ be the virtual representation $\ovl{V}_1 - \ovl{V}_2$ where $\ovl{V}_i$ is the semisimplification of $\ovl{V}_i^\circ$ and $V_i^\circ$ is any $\Gamma$-stable $\cO$-lattice in $V_i$ (and $\ovl{V}$ depends only on $V$ and not on any other choices). 
Of course, $\ovl{V}$ is a genuine representation if $V$ is.

If $X$ is an ind-scheme defined over $\cO$, we write $X_E\defeq X\times_{\Spec\cO} \Spec E$ and $X_{\F}\defeq X\times_{\Spec \cO}\Spec \F$ to denote its generic and special fiber, respectively. Similarly, if $R$ is any $\cO$-algebra, we write $R_{\F}$ to denote $R\otimes_{\cO}\F$

If $P$ is a statement, the symbol $\delta_P\in \{0,1\}$ takes value $1$ if $P$ is true, and $0$ if $P$ is false. 
\clearpage{}%
\clearpage{}%
\section{Preliminaries}
\label{sec:Prel}

\subsection{Extended affine Weyl groups}

\subsubsection{Admissible sets, regular elements, and admissible pairs}
\label{sub:AWG:2}

Recall that $G$ is a split reductive group with split maximal torus $T$ and Borel $B$.
Let $V\defeq X^*(T) \otimes \R \cong X_*(T^{\vee}) \otimes \R$ denote the apartment of $(G, T)$ on which $\tld{W}$ acts.   Let $\cC_0$ denote the dominant Weyl chamber in $V$.  For any $w \in W(G)$, let $\cC_w = w(\cC_0)$.  In particular, $\cC_{w_0}$ is the anti-dominant Weyl chamber.

Recall from \S \ref{sec:notations} the set $\cA$ of alcoves of $V$.  
We let $\uparrow$ denote the upper arrow ordering on alcoves as defined in \cite[\S II.6.5]{RAGS}. 
Since $W_a$ acts simply transitively on the set of alcoves, $\uparrow$ induces an upper arrow ordering on $W_a$ which we again denote by $\uparrow$.
The dominant base alcove $A_0$ defines a Bruhat order on $W_a$ which we denote by $\leq$. 
If $\Omega$ is the stabilizer of the base alcove, then $\tld{W} = W_a \rtimes \Omega$ and so $\tld{W}$ inherits a Bruhat and upper arrow order in the standard way: For $\tld{w}_1, \tld{w}_2\in W_a$ and $\delta\in \Omega$, $\tld{w}_1\delta\leq \tld{w}_2\delta$ (resp.~$\tld{w}_1\delta\uparrow \tld{w}_2\delta$) if and only if $\tld{w}_1\leq \tld{w}_2$ (resp.~$\tld{w}_1\uparrow \tld{w}_2$), and elements in different right $W_a$-cosets are incomparable. 
We write $X \uparrow Y$ (resp.~$X \leq Y$) between sets $X$ and $Y$ if $x\uparrow y$ (resp.~$x \leq y$) for all $x\in X$ and $y\in Y$.
For $\tld{w} \in \tld{{W}}$, let 
\[
\tld{{W}}_{\leq\tld{w}} \defeq \{\tld{s}\in \tld{{W}} \mid \tld{s}\leq \tld{w}\}.
\]
We write $\ell$ for the Coxeter length function on $W_a$, which we extend to $\tld{W}$ by letting $\ell(\tld{w}\delta)\defeq \ell(\tld{w})$ for any $\tld{w}\in W_a$, $\delta\in \Omega$.

\begin{defn}
Let $m\geq 1$, $\tld{w}_1,\dots, \tld{w}_m\in\tld{W}$ and set $\tld{w}\defeq \prod\limits_{i=1}^m\tld{w}_i$.
We say $\prod\limits_{i=1}^m\tld{w}_i$ is a \emph{reduced expression} for $\tld{w}$ if $\ell(\tld{w})=\sum\limits_{i=1}^m\ell(\tld{w}_i)$.
\end{defn}

We now recall the definition of the admissible set as introduced by Kottwitz and Rapoport:

\begin{defn} \label{defn:adm} Let $\lambda \in X^*(T)$.  Define 
\[
\Adm(\lambda) \defeq  \bigcup_{w\in {W}} \tld{{W}}_{\leq t_{w(\lambda)}}.
\]
\end{defn}

Recall from \S \ref{sec:notations} the hyperplanes $H_{\alpha, n} = \{ x \in V\, | \, \langle x,\alpha^\vee\rangle = n \}$ and the notation $\Phi^{+}$ (resp.~$\Phi^{-}$) for the set of positive (resp.~negative) roots.   
We use the notation $\alpha > 0$ (resp.~$\alpha < 0$) for a positive (resp.~negative) root. 
For $\alpha \in \Phi$, define the half-hyperplanes $H^{+}_{\alpha, n} = \{ x \in V \mid \langle x, \alpha^\vee\rangle > n \}$ and $H^{-}_{\alpha, n} =  \{ x \in V \mid \langle x, \alpha^\vee\rangle < n \}.$ 
Define the $m$-th $\alpha$-strip to be 
\[
H^{(m, m+1)}_{\alpha} = \{ x \in V \mid m < \langle x, \alpha^\vee\rangle < m+1 \}.
\]
Define the critical strips to be strips $H^{(0, 1)}_{\alpha}$ where $\alpha \in \Phi^+$.

\begin{defn} \label{defn:regular} An alcove $A \in \cA$ is \emph{regular} if $A$ does not lie in any critical strips.   For any $\tld{w} \in \tld{W}$, we say $\tld{w}$ is \emph{regular} if $\tld{w}(A_0)$ is regular.  Define 
\[
\Adm^{\mathrm{reg}}(\lambda) = \{ \tld{w} \in \Adm(\lambda)\mid \tld{w} \text{ is regular} \}. 
\]
\end{defn}

From \cite[Lemma 4.1.9]{LLL} we have:
\begin{lemma}\label{lemma:gallery}
Suppose that $\tld{w}_1$ and $\tld{w}_2\in \tld{W}^+$.
Then $\tld{w}_2^{-1}w_0\tld{w}_1$ is a reduced expression. 
\end{lemma}

\begin{prop}
\label{prop:can:reg}
If $\tld{w} \in \tld{W}$ is regular, then there exist $\tld{w}_1$ and $\tld{w}_2 \in \tld{W}_1^+$ and a dominant weight $\nu \in X^*(T)$ such that $\tld{w} = \tld{w}_2^{-1} w_0 t_\nu \tld{w}_1$.
Moreover, $\tld{w}_1$, $\tld{w}_2$, and $\nu$ as above are unique up to $X^0(T)$. 

Conversely, if $\tld{w}_1$ and $\tld{w}_2$ are elements of $\tld{W}^+$, then $\tld{w}_2^{-1} w_0 \tld{w}_1$ is regular.
\end{prop}
\begin{proof}
Suppose that $\tld{w} \in \tld{W}$ is regular and $w_2 \tld{w}(A_0)$ lies in the anti-dominant Weyl chamber for $w_2\in W$.
Let $\eta_2 \in X^*(T)$ be such that $\tld{w}_2 = t_{\eta_2} w_2\in \tld{W}_1^+$.
Note that $\eta_2$ is unique up to $X^0(T)$.
Let $x$ be in $A_0$.
From the assumption we deduce that $\tld{w}(x)$ and $x$ do not lie in the same $\alpha$-strip for any root $\alpha$. 
(Note that $A_0$ only lies inside critical strips.)
Equivalently, $\tld{w}_2\tld{w}(x)$ and $\tld{w}_2(x)$ do not lie in the same $\alpha$-strip for any root $\alpha$.
In particular:
\begin{equation}\label{eqn:noncrit}
\lfloor \langle \tld{w}_2\tld{w}(x),\alpha^\vee\rangle\rfloor \neq \lfloor \langle\tld{w}_2(x),\alpha^\vee\rangle\rfloor = 0
\end{equation} 
for all simple roots $\alpha$, using that $\tld{w}_2 \in \tld{W}_1^+$ to obtain the last equality.

Now let $\alpha$ be a simple root.
Then, $\langle w_2\tld{w}(x),\alpha^\vee\rangle  < 0$ by assumption.
Moreover, $\langle\eta_2,\alpha^\vee\rangle \leq 1$ since $\eta_2$ is $1$-restricted (a lift of a multiplicity free sum of fundamental weights). 
We conclude that $\langle \tld{w}_2\tld{w}(x),\alpha^\vee\rangle < 1$.
From (\ref{eqn:noncrit}), we deduce that $\langle \tld{w}_2\tld{w}(x), \alpha^\vee\rangle< 0$.
Since $\alpha$ is an arbitrary simple root, $\tld{w}_2\tld{w}(x)$ lies in the anti-dominant Weyl chamber.
Thus, $w_0\tld{w}_2 \tld{w} \in \tld{W}^+$.
We conclude that $w_0\tld{w}_2 \tld{w} = t_\nu\tld{w}_1$ for some dominant $\nu \in X^*(T)$ and $\tld{w}_1 \in \tld{W}_1^+$.
Again, $\nu$ and $\tld{w}_1$ are determined up to $X^0(T)$.

For the converse, let $\tld{w}_1$ and $\tld{w}_2$ be elements of $\tld{W}^+$. Let $x \in A_0$. 
Showing that $\tld{w}_2^{-1}w_0\tld{w}_1$ is regular is equivalent to showing that $w_0\tld{w}_1(x)$ and $\tld{w}_2(x)$ do not lie in the same $\alpha$-strip for any root $\alpha$.
This is clear from the fact that $w_0\tld{w}_1(x)$ lies in the anti-dominant Weyl chamber while $\tld{w}_2(x)$ lies in the dominant Weyl chamber.
\end{proof}

\begin{prop}
\label{prop:can:adm}
Suppose that $\tld{w}_1$ and $\tld{w}_2$ are elements in $\tld{W}^+$.
Let $\lambda \in X^*(T)$ be a dominant weight.
The following are equivalent.
\begin{enumerate}
\item \label{item:uparrow} $\tld{w}_1 \uparrow t_\lambda \tld{w}_h^{-1} \tld{w}_2$;
\item \label{item:uparrow2} $\tld{w}_2 \uparrow \tld{w}_h t_{-\lambda} \tld{w}_1$;
\item \label{item:extreme} $\tld{w}_2^{-1} w_0 \tld{w}_1 \leq t_{w_1^{-1}(\lambda+\eta)}, \, t_{(w_0 w_2)^{-1}(\lambda+\eta)}$ where $w_1, \,w_2 \in W$ are the images of $\tld{w}_1$ and $\tld{w}_2$ in $W$; and
\item \label{item:adm} $\tld{w}_2^{-1} w_0 \tld{w}_1$ is in $\Adm(\lambda+\eta)$.
\end{enumerate}
\end{prop}
\begin{proof}
It is clear that  (\ref{item:uparrow}) is equivalent to (\ref{item:uparrow2}) by \cite[Proposition 4.1.2]{LLL} and \cite[II.6.5(4)]{RAGS}. 
We first show that (\ref{item:uparrow}) implies (\ref{item:extreme}).
Let $\omega \in X^*(T)$ be a dominant weight (unique up to $X^0(T)$) such that $t_{-\omega}\tld{w}_2 \in \tld{W}_1^+$. 
Then $t_{-w_0(\omega)}\tld{w}_1 \in \tld{W}^+$, $t_{-w_0(\omega)}\tld{w}_1 \uparrow t_\lambda \tld{w}_h^{-1} (t_{-\omega}\tld{w}_2)$, and $\tld{w}_2^{-1} w_0 \tld{w}_1 = (t_{-\omega}\tld{w}_2)^{-1} w_0 (t_{-w_0(\omega)}\tld{w}_1)$. 
Without loss of generality, we can assume that $\tld{w}_2$ is an element of $\tld{W}_1^+$.
Then $\tld{w}_h^{-1} \tld{w}_2 \in \tld{W}^+$.
Then {Wang's theorem} \cite[Theorem 4.3]{wang} (see also \cite[Theorem 4.1.1]{LLL}) implies that
$\tld{w}_1 \leq t_\lambda \tld{w}_h^{-1} \tld{w}_2$.
Using that $\tld{w}_2^{-1} w_0 (t_\lambda \tld{w}_h^{-1} \tld{w}_2)$ is a reduced expression by Lemma \ref{lemma:gallery}, we have that
\[
\tld{w}_2^{-1} w_0 \tld{w}_1 \leq \tld{w}_2^{-1} w_0 t_\lambda \tld{w}_h^{-1} \tld{w}_2 = t_{w_2^{-1}w_0(\lambda+\eta)}.
\] 
The inequality $\tld{w}_2^{-1} w_0 \tld{w}_1 \leq t_{w_1^{-1}(\eta)}$ follows from (\ref{item:uparrow2}) using the same argument.

Item (\ref{item:extreme}) immediately implies (\ref{item:adm}).
We now show that (\ref{item:adm}) implies (\ref{item:uparrow}).
As before, we can and do change $\tld{w}_1$ and $\tld{w}_2$ so that $\tld{w}_2 \in\tld{W}_1^+$ without affecting the product $\tld{w}_2^{-1}w_0\tld{w}_1$ or the veracity of the relation in (\ref{item:uparrow}).
By writing $\tld{w}_2 = t_{\eta_2} w_2$ (where $\eta_2$ is dominant), it is easy to see that $\tld{w}_2^{-1}w_0\tld{w}_1(A_0) = w_2^{-1} t_{-\eta_2} w_0 \tld{w}_1(A_0)$ lies in the Weyl chamber $(w_0w_2)^{-1} (\cC_0)$ since $t_{-\eta_2} w_0 \tld{w}_1(A_0)$ lies in $w_0(\cC_0)$.
We conclude from \cite[Corollary 4.4]{HH} and the $(\lambda+\eta)$-admissibility of $\tld{w}_2^{-1}w_0\tld{w}_1$ that 
\[
\tld{w}_2^{-1}w_0\tld{w}_1 \leq t_{(w_0w_2)^{-1}(\lambda+\eta)} = \tld{w}_2^{-1}w_0(t_\lambda \tld{w}_h^{-1} \tld{w}_2).
\]
Noting that $\tld{w}_h^{-1} \tld{w}_2 \in \tld{W}^+$ since $\tld{w}_2\in \tld{W}_1^+$, the above factorizations are reduced by Lemma \ref{lemma:gallery}.
We conclude that $\tld{w}_1 \leq t_\lambda \tld{w}_h^{-1} \tld{w}_2$, which implies that $\tld{w}_1 \uparrow t_\lambda \tld{w}_h^{-1} \tld{w}_2$ by Wang's theorem {\cite[Theorem 4.3]{wang}}.
\end{proof}

For a dominant weight $\lambda \in X^*(T)$, define the collection of admissible pairs
\begin{align} \label{eq:AP}
\mathrm{AP}(\lambda+\eta)&\defeq \left\{
\begin{array}{l}
(\tld{w}_1,\tld{w}_2)\in (\tld{W}_1^+\times \tld{W}^+)\slash X^{0}(T)\Bigm\vert\tld{w}_1\uparrow t_{\lambda}\tld{w}_h^{-1}\tld{w}_2
\end{array}
\right\},
\end{align}
where $X^0(T)$ is embedded diagonally in the natural way.

\begin{cor}
\label{cor:can:reg}
Let $\lambda\in X^*(T)$ be a dominant weight.
Then the map
\begin{align*}
\mathrm{AP}(\lambda+\eta)
&\stackrel{\sim}{\longrightarrow}
\Adm^{\mathrm{reg}}(\lambda+\eta).\\
(\tld{w}_1,\tld{w}_2)&\longmapsto \tld{w}_2^{-1}w_0\tld{w}_1
\end{align*}
is a bijection.
\end{cor}
\begin{proof}
We first show that the image of the map lies in $\Adm^{\mathrm{reg}}(\lambda+\eta)$.
Assume $(\tld{w}_1,\tld{w}_2)\in \tld{W}^+_1\times \tld{W}^+$ is such that $\tld{w}_1\uparrow t_{\lambda}\tld{w}_h^{-1}\tld{w}_2$.
Note that this condition is stable under the diagonal action of $X^{0}(T)$ by \cite[II.6.5(4)]{RAGS}.
Then $\tld{w}_2^{-1}w_0\tld{w}_1\in \Adm(\lambda+\eta)$ by Proposition \ref{prop:can:adm}(\ref{item:adm}), and is regular by Proposition \ref{prop:can:reg}.

To show surjectivity, write $\tld{w}=\tld{w}_2^{-1}w_0t_\nu\tld{w}_1 \in \Adm^{\mathrm{reg}}(\lambda + \eta)$ as in the statement of Proposition \ref{prop:can:reg}.
By Proposition \ref{prop:can:adm}(\ref{item:uparrow}) we have $t_\nu\tld{w}_1\uparrow t_{\lambda}\tld{w}_h^{-1}\tld{w}_2$, which is equivalent to $\tld{w}_1\uparrow t_{\lambda}\tld{w}_h^{-1}t_{-w_0(\nu)}\tld{w}_2$ by \cite[II.6.5(4)]{RAGS}.
Since $t_{-w_0(\nu)}\tld{w}_2\in \tld{W}^+$, $(\tld{w}_1,t_{-w_0(\nu)}\tld{w}_2) \in \mathrm{AP}(\lambda+\eta)$ and has image $\tld{w}$. 

The uniqueness of the decomposition up to translation by $X^0(T)$ from Proposition \ref{prop:can:reg} shows injectivity.
\end{proof}

\begin{rmk}\label{rmk:factor}
The same proof also shows that there is a bijection 
\begin{align*}
\left\{
(\tld{w}_1,\tld{w}_2)\in (\tld{W}^+\times \tld{W}_1^+)\slash X^{0}(T)\Bigm\vert\tld{w}_1\uparrow t_{\lambda}\tld{w}_h^{-1}\tld{w}_2 \right\}
&\stackrel{\sim}{\longrightarrow}
\Adm^{\mathrm{reg}}(\lambda+\eta)\\
(\tld{w}_1,\tld{w}_2)&\longmapsto \tld{w}_2^{-1}w_0\tld{w}_1,
\end{align*}
though this plays a lesser role in what follows.
\end{rmk}

\subsubsection{Genericity} \label{sub:AWG:1}

Let $(\tld{W}^\vee,\leq)$ be the following partially ordered group: $\tld{W}^\vee$ is identified with $\tld{W}$ as a group, and $\leq$ is induced from the Bruhat order on $\tld{W}$ defined by the \emph{antidominant} base alcove.  %

\begin{defn} 
\label{affineadjoint} We define a bijection $\tld{w}\mapsto \tld{w}^*$ between $\tld{W}$ and $\tld{W}^\vee$ as follows: for $\tld{w} = t_{\nu}w \in \tld{W}$, with $w\in W$ and $\nu\in X^*(T) = X_*(T^{\vee})$, then $\tld{w}^*\defeq  w^{-1}t_{\nu} \in \tld{W}^\vee$. %
\end{defn}

\cite[Lemma 2.1.3]{LLL} shows that $(-)^*: \tld{W} \ra \tld{W}^\vee$ is an isomorphism of partially ordered groups.

We now introduce various notions of genericity which will be used throughout the paper.

\begin{defn}
\label{defn:var:gen}
Let $\lambda\in X^*(T)$ be a weight and let $m$ be an integer. 

\begin{enumerate}
\item 
\label{defn:-deep}
We say that $\lambda$ is \emph{$m$-deep} in a ($\eta$-shifted) $p$-alcove $C$ if 
\[
n_\alpha p+m < \langle \lambda+\eta,\alpha^\vee \rangle < (n_\alpha +1)p - m
\]
where $C$ is the $p$-alcove defined by the above inequalities with $m = 0$.
\end{enumerate}

We now assume that $m \geq 0$.

\begin{enumerate}\setcounter{enumi}{1}
\item 
\label{defn:deep}
If $m\geq 0$, we say $\lambda$ is \emph{$m$-deep} if $\lambda$ is $m$-deep in some $p$-alcove $C$.
Equivalently, $m < |\langle \lambda+\eta, \alpha^{\vee} \rangle + pk|$ for all $\alpha\in \Phi^{+}$ and $k\in \Z$.  %

\item
\label{it:gen:weyl}
For $\tld{w} = w t_{\nu}$ in either $\tld{W}$ or $\tld{W}^\vee$, we say that $\tld{w}$ is $m$-generic if $\nu - \eta$ is $m$-deep. %

\item  
\label{it:def:small} 
For $\tld{w} = w t_{\nu}$ in either $\tld{W}$ or $\tld{W}^\vee$,  we say that $\tld{w}$ is \emph{$m$-small} if $h_\nu \leq m$, i.e., $\langle \nu, \alpha^{\vee} \rangle  \leq m$
for all $\alpha \in \Phi$.

\item
\label{defn:P-generic}
Let $P = P(X_1,\ldots,X_n) \in \Z[X_1,\ldots,X_n]$ be a polynomial and let $R$ be a commutative ring.
We say that a tuple $\bf{a} \in R^n$ is \emph{$P$-generic} if $P(\bf{a}) \pmod{p} \in R/p$ is in $(R/p)^\times$. 
For a finite set $\cJ$, we say that $\bf{a} \in (R^n)^{\cJ}$ is \emph{$P$-generic} if $\bf{a}_j$ is $P$-generic for all $j \in \cJ$.
If $G= \GL_n^{\cJ}$, we say that $\lambda \in X^*(T)$ is \emph{$P$-generic} if it is under the standard identification of $X^*(T)$ with $(\Z^n)^\cJ$. %
\end{enumerate} 
\end{defn}

\begin{rmk}\label{rmk:P-gen}
\begin{enumerate}
\item 
\label{it:P-gen:1}
We note that depth is preserved by the ($p$-)dot action, smallness is preserved by the standard $W$-action, but $P$-genericity is typically not preserved by either of these.
\item 
\label{it:P-gen:2}
Suppose that $G= \GL_n^{\cJ}$. If we let $P_m(X_1,\ldots,X_n)$ be $\prod_{i=1}^n \prod_{j=1}^m (X_i-X_{i+1} - j)$ where $X_{n+1}$ is understood to be $X_1$, then $\lambda-\eta\in C_0$ is $m$-deep if and only if $\lambda$ is $P_m$-generic.
\end{enumerate}
\end{rmk}

We record some elementary properties of smallness and genericity. 
\begin{prop} \label{prop:propertiesofsmall} 
Let $\tld{w}, \tld{z}$ be elements in $\tld{W}$ (resp.~in $\tld{W}^{\vee}$) and let $\nu\in X^*(T)$.
\begin{enumerate}
\item the element $t_\nu$ is $m$-generic (resp.~$m$-small) if and only if $t_{s(\nu)}$ is $m$-generic (resp.~$m$-small) for all $s\in W(G)$;
\item if $\tld{w}$ is $m$-small and $\tld{z}$ is $m'$-small, then $\tld{w} \tld{z}$ is $(m + m')$-small;
\item the element $\tld{w}$ is $m$-small if and only if $\tld{w}^{-1}$ is $m$-small if and only if $\tld{w}^*$ is $m$-small; and
\item if $\tld{z}$ is $m'$-generic and $\tld{w}$ is $m$-small with $m \leq m'$, then $\tld{z} \tld{w}$ is $(m' - m)$-generic.  
\end{enumerate}
\end{prop}

\subsection{Serre weights}
\label{sec:SW}
We recall some notation in \ref{sec:notations}.
Let $G_{/\Z_p}$ be a split connected reductive group with extended affine Weyl group $\tld{W}$.
Let $\cO_p$ be a finite \'etale $\Z_p$-algebra.
Let $G_0$ be $\Res_{\cO_p/\Z_p} (G_{/\cO_p})$ and $\un{G}$ be the split group $(G_0)_{/\cO}$. 
Note that the Bruhat order on $\tld{\un{W}}\cong \tld{W}^{\cJ}$ is the product partial order induced from the Bruhat order on $\tld{W}$ {(hence, the partial order $\leq$ on $\tld{W}^{\cJ}$ is taken componentwise)}.  

For a dominant character $\lambda\in X^*(\un{T})$, we define $W(\lambda)_{/\cO}$ to be the $\un{G}$-module $\Ind_{\un{B}}^{\un{G}} w_0 \lambda$.
Then $W(\lambda)_{/E}$ is the unique up to isomorphism irreducible $\un{G}_{/E}$-module of highest weight $\lambda$. 
Let $V(\lambda)$ be the (irreducible) restriction of $W(\lambda)_{/E}(E)$ to $G_0(\Z_p)$.
The socle $L(\lambda)$ of the $\un{G}_{/\F}$-module $W(\lambda)_{/\F} \defeq W(\lambda)_{/\cO} \otimes_{\cO} \F$ is the unique up to isomorphism irreducible $\un{G}_{/\F}$-module of highest weight $\lambda$.
For any character $\lambda\in X^*(\un{T})$, we can extend the above definition by letting $W(\lambda)_{/\cO}$ be the \emph{virtual} $\un{G}$-module 
\begin{equation}\label{eqn:euler}
\sum_i (-1)^i R^i \Ind_{\un{B}}^{\un{G}} w_0 \lambda.
\end{equation}
We similarly define the virtual modules $V(\lambda)$ and $W(\lambda)_{/\F}$.

Let $\rG$ be the group $G_0(\Fp)\cong G_{/\cO_p}(\cO_p/p)$. 
A \emph{Serre weight} (of $\rG$) is an irreducible $\F$-representation of $\rG$.
An irreducible $\rG$-representation over $\F$ is necessarily absolutely irreducible and every irreducible $\rG$-representation over $\ovl{\F}_p$ is defined over $\F$.
Each Serre weight is the restriction $F(\lambda) \defeq L(\lambda)|_\rG$ for some $\lambda\in X_1(\un{T})$ where 
\[
X_1(\un{T})\defeq \{
\lambda\in X^*(\un{T}), 0\leq \langle \lambda,\alpha^\vee\rangle \leq p-1\text{ for all $\alpha \in \un{\Delta}$}
\}
\]
is the set of $p$-restricted dominant weights. 
The map $\lambda \mapsto F(\lambda)$ gives a bijection from $X_1(\un{T})/(p-\pi)X^0(\un{T})$ to the set of isomorphism classes of Serre weights of $\rG$ (see \cite[Lemma 9.2.4]{GHS}). 
For $m\geq 0$, we say that a Serre weight $F(\lambda)$ is $m$-deep if $\lambda$ is $m$-deep.
We say that $\lambda\in X_1(\un{T})$ is \emph{regular $p$-restricted} or $F(\lambda)$ is \emph{regular} if $\langle \lambda,\alpha^\vee\rangle < p-1$ for all $\alpha\in \un{\Delta}$.

For $\lambda \in X^*(\un{T})$, let $W(\lambda)$ be the restriction of $W(\lambda)_{/\F}(\F)$ to $\rG$, which is a genuine representation if $\lambda$ is dominant. 
Then $F(\lambda)$ is an $\rG$-submodule of $W(\lambda)$ for $\lambda \in X_1(\un{T})$.

For the combinatorics of Serre weights it is convenient to introduce the notion of $p$-alcoves and the dot action on them.
A $p$-alcove is a connected component of  
the complement $X^*(\un{T})\otimes_{\Z}\R\ \setminus\ \big(\bigcup_{(\alpha,pn)}(H_{\alpha,pn}-\eta)\big)$.
We say that a $p$-alcove $\un{C}$ is
\emph{dominant} (resp.~\emph{$p$-restricted}) if $0 <  \langle\lambda + \eta,\alpha^\vee\rangle$ (resp.~if $0 <  \langle\lambda + \eta,\alpha^\vee\rangle<p$) for all $\alpha\in \un{\Delta}$ and $\lambda\in \un{C}$.
We let $\un{C}_0$ denote the dominant base $p$-alcove, i.e.~the alcove characterized by $\lambda\in \un{C}_0$ if and only if $0 <\langle\lambda + \eta,\alpha^\vee\rangle<p$ for all $\alpha\in \un{\Phi}^+$.
We define the ($p$-)\emph{dot action} of $\tld{\un{W}}$ on $X^*(\un{T})\otimes_{\Z}\R$ by $\tld{w}\cdot \lambda\defeq w(\lambda+\eta+p\nu)-\eta$ for $\tld{w}=wt_\nu\in \tld{\un{W}}$ and $\lambda\in X^*(\un{T})\otimes_{\Z}\R$.
In particular, $\tld{\un{W}}$ acts transitively via the dot action on the set of $p$-alcoves, and $\un{\Omega}$ is the stabilizer of $\un{C}_0$ for the dot action.
We have
\[\tld{\un{W}}^+=\{\tld{w}\in \tld{\un{W}}:\tld{w}\cdot \un{C}_0 \textrm{ is dominant}\}\]
and
\[\tld{\un{W}}^+_1=\{\tld{w}\in \tld{\un{W}}^+:\tld{w}\cdot \un{C}_0 \textrm{ is } p\textrm{-restricted}\}.\]

\begin{lemma}\label{lemma:nongenWeyl}
If 
\begin{itemize}
\item 
$\mu\in X^*(\un{T})$ is a dominant weight which is not $m$-deep,
\item 
$h\in \Z$ such that $\langle \mu,\alpha^\vee \rangle \leq h$ for all $\alpha \in \un{\Phi}$, and 
\item 
$\sigma \in \JH(W(\mu))$, 
\end{itemize}
then $\sigma$ is not $(m+\lfloor \frac{h}{p-1}\rfloor)$-deep.
\end{lemma}
\begin{proof}
Suppose that $\mu$ is as in the statement of the lemma.
By \cite[II.6.13 Proposition]{RAGS}, if $\sigma\in \JH(W(\mu))$, then either 
\begin{itemize}
\item 
$\sigma \in \JH(W(\tld{w}\cdot \mu))$ for $\tld{w} \in W_a$ with $\tld{w} \cdot \mu \neq \mu$ dominant and $\tld{w}\cdot \mu \uparrow \mu$; or 
\item 
$\sigma \in \JH(L(\mu)|_{\rG})$.
\end{itemize}
We now replace this second condition.
Suppose that $\sigma \in \JH(L(\mu)|_{\rG})$.
If $\mu = \mu_0 + p\mu_1$ for $\mu_0,\mu_1$ dominant and $\mu_0$ $p$-restricted, then $L(\mu) \cong L(\mu_0)\otimes L(p\mu_1)$ by the Steinberg tensor product theorem \cite[II.3.17]{RAGS}.
Since $L(p\mu_1)|_{\rG} \cong L\big(\pi (\mu_1)\big)|_{\rG}$, we have that $\sigma \in \JH\big(W(\mu_0) \otimes W(\pi(\mu_1))\big)$.
\cite[II.5.8 Lemma]{RAGS} implies that $\sigma \in \JH\big(W(\mu-p\mu_1 + \pi(\nu))\big)$ for some $\nu \in \Conv(\mu_1)$ (recall that $W(\mu-p\mu_1 + \pi(\nu))$ is \emph{a priori} a virtual representation). 
By \cite[II.5.5 Corollary (b)]{RAGS}, $\sigma \in \JH\big(W(w\cdot (\mu-p\mu_1 + \pi(\nu)))\big)$ for some $w \in \un{W}$ such that $w\cdot (\mu-p\mu_1 + \pi(\nu))$ is dominant.
By the following lemma (where we take $\lambda$, $\nu$, and $\kappa$ to be $\mu-p\mu_1 + \eta$, $\pi(\nu)$ and $\pi(\mu_1)$, respectively), replacing $\nu$ by \[\pi^{-1}\big(w\cdot (\mu-p\mu_1 + \pi(\nu))-(\mu-p\mu_1)\big),\]
we can assume without loss of generality that $\mu-p\mu_1 + \pi(\nu) \in X^*(\un{T})$ is dominant.
\begin{lemma}
Suppose $\lambda\in X^*(\un{T})$ is dominant, $\nu\in \Conv(\kappa)$ and $w$ such that $w(\lambda+\nu)$ is dominant. Then $w(\lambda+\nu)-\lambda\in \Conv(\kappa)$.
\end{lemma}
\begin{proof} There is a sequence of positive roots $\alpha_1,\cdots, \alpha_k$ such that $w=s_{\alpha_k}\cdots s_{\alpha_1}$, and setting $s_{\alpha_j}\cdots s_{\alpha_1}(\lambda+\nu)=\lambda+\nu_j$ we have $\lambda+\nu_{j}$ is on the positive side of the $\alpha_{j}$-wall while $\lambda+\nu_{j-1}$ is on the negative side of the $\alpha_j$-wall. Thus we get $\lambda+\nu_j=\lambda+\nu_{j-1}+m \alpha_j$ with $m=-\langle \lambda+\nu_{j-1},\alpha_j^{\vee}\rangle\geq 0$. Now
\[\langle \nu_{j-1}, \alpha_j^\vee \rangle \leq \langle \lambda+\nu_{j-1},\alpha_j^\vee \rangle<0\]
hence $0<m\leq -\langle \lambda+\nu_{j-1}, \alpha_j^\vee \rangle \leq -\langle \nu_{j-1}, \alpha_j^\vee \rangle$.
This shows that $\nu_j=\nu_{j-1}+m\alpha_j$ lies in the segment between $\nu_{j-1}$ and $s_{\alpha_j}\nu_{j-1}=\nu_{j-1}-\langle \nu_{j-1}, \alpha_j^\vee \rangle \alpha_j$, hence $\nu_{j}\in \Conv(\kappa)$ by induction.
\end{proof}

Returning to the proof of Lemma \ref{lemma:nongenWeyl}, the upshot is that if $\sigma\in \JH(W(\mu))$, then either 
\begin{itemize}
\item 
$\sigma \in \JH(W(\tld{w}\cdot \mu))$ for $\tld{w} \in W_a$ with $\tld{w} \cdot \mu \neq \mu$ dominant and $\tld{w}\cdot \mu \uparrow \mu$; 
\item 
$\sigma \in \JH\big(W(\mu - p\mu_1 + \pi(\nu))\big)$ where $\mu_1$ is nonzero, $\mu_1,\mu-p\mu_1,\mu-p\mu_1 + \pi(\nu) \in X^*(\un{T})$ are dominant, and $\nu \in \Conv(\mu_1)$; or 
\item
$\mu$ is $p$-restricted and $\sigma = F(\mu)$.
\end{itemize}
In this way, either $\sigma = F(\mu)$ or we can replace $\mu$ with a ``smaller" weight.
For convenience, for $\lambda\in X^*(\un{T})$ and $\nu \in \Conv(\lambda)$, we let $t_{\lambda,\nu}$ be the operator on $X^*(\un{T})$ which translates by $-p\lambda+\pi(\nu)$.
Iterating the above weight reduction process, we see that if $\sigma\in \JH(W(\mu))$, then $\sigma = F(\lambda)$ for $\lambda\in X^*(\un{T})$ of the form
\begin{equation} \label{eqn:weyl_iterate}
t_{\mu_M,\nu_M}\tld{w}_M \cdot (t_{\mu_{M-1},\nu_{M-1}} \tld{w}_{M-1} \cdot (\cdots t_{\mu_1,\nu_1}\tld{w}_1 \cdot \mu)\cdots )
\end{equation}
where $\mu_i \in X^*(\un{T})$ is dominant and nonzero and $\tld{w}_i \in \un{W}_a$ is nontrivial for all $i$ and the weight at each step is dominant.
Indeed, since each iteration strictly reduces the value of $\langle -,(\eta-w_0(\eta))^\vee\rangle$ which must be positive, the iterative process must end (with an upper bound on the number of steps depending on $\mu$).
Then (\ref{eqn:weyl_iterate}) can be rewritten as $\lambda = \tld{w}\cdot \mu + \pi(\nu)$ for some $\tld{w}\in \tld{\un{W}}$ and $\nu \in \Conv(\mu_{\mathrm{sum}})$ where $\mu_{\mathrm{sum}} = \sum_{i=1}^M \mu_i$.
We claim that $\langle \nu,\alpha^\vee\rangle \leq \lfloor \frac{h}{p-1}\rfloor$ for any $\alpha \in \un{\Phi}$.
Since the $p$-dot action preserves depth, $\lambda$ would not be $(m+\lfloor \frac{h}{p-1}\rfloor)$-deep.

To prove our claim, consider $t_{\mu_M,\nu_M} \cdots t_{\mu_1,\nu_1} \mu = \mu - p\mu_{\mathrm{sum}} +\pi(\nu')$ for some $\nu' \in \Conv(\mu_{\mathrm{sum}})$.
Observe that $\lambda \uparrow \mu - p\mu_{\mathrm{sum}} +\pi(\nu')$.
Then
\[
\langle p\mu_{\mathrm{sum}} - \pi(\nu'),\alpha^\vee\rangle \leq \langle \mu - \lambda,\alpha^\vee\rangle \leq h
\]
for any highest root $\alpha$.
Choosing $\alpha\in \un{\Phi}^+$ a highest root so that $h_{\mu_{\mathrm{sum}}} = \langle \mu_{\mathrm{sum}} ,\alpha^\vee\rangle$, 
we have 
\[
(p-1)h_{\mu_{\mathrm{sum}}} \leq \langle p\mu_{\mathrm{sum}} - \pi(\nu'),\alpha^\vee\rangle \leq h.
\]
We conclude that $h_{\mu_{\mathrm{sum}}} \leq \lfloor \frac{h}{p-1}\rfloor$, and the claim follows.
\end{proof}

We will call an element of $X^*(\un{Z})$ an \emph{algebraic central character} and an element of $X^*(\un{Z})/(p-\pi)X^*(\un{Z})$ a \emph{central character}.
Note that the character group $\Hom(Z_0(\F_p),\F^\times)$ is naturally identified with $X^*(\un{Z})/(p-\pi)X^*(\un{Z})$.
An algebraic central character determines a central character by the natural reduction map.
The central character (a character of $Z_0(\F_p)$) of a Serre weight $F(\lambda)$ is $\lambda|_Z \in X^*(\un{Z})/(p-\pi)X^*(\un{Z})$ which does not depend on the choice of element in $\lambda+(p-\pi)X^0(\un{T})$ and gives the action of $Z_0(\F_p)$ on $F(\lambda)$.
Note that there is a natural identification of $X^*(\un{Z})$ with $\tld{\un{W}}/\un{W}_a$, which we will use often.

Let $\omega-\eta\in \un{C}_0\cap X^*(\un{T})$ and $\tld{w}_1\in\tld{\un{W}}^+_1$.
Then $\pi^{-1}(\tld{w}_1)\cdot (\omega-\eta)\in X_1(\un{T})$ and we define
\begin{equation}\label{eqn:LAPwt}
F_{(\tld{w}_1,\omega)}\defeq F(\pi^{-1}(\tld{w}_1)\cdot (\omega-\eta)).
\end{equation}
We consider the equivalence relation $(\tld{w}_1,\omega) \sim (t_\nu\tld{w}_1,\omega-\nu)$ for all $\nu\in X^0(\un{T})$ and note that the map $(\tld{w}_1,\omega) \mapsto F_{(\tld{w}_1,\omega)}$ sends equivalent pairs to the same Serre weight.
We say that the equivalence class of $(\tld{w}_1,\omega)$ is a \emph{lowest alcove presentation} of $F_{(\tld{w}_1,\omega)}$.
(Note that the notion of lowest alcove presentation depends on the choice of $\eta$ in \S \ref{sec:not:RG}.)
We often will choose a pair in the equivalence class of a lowest alcove presentation of a Serre weight, though nothing we do will depend on this choice.
From a lowest alcove presentation $(\tld{w}_1,\omega)$ of a Serre weight, we obtain an algebraic central character 
\begin{align}\label{eqn:wtchar}
\tld{\un{W}}/\un{W}_a &\cong X^*(\un{Z}) \\
t_{\omega-\eta}\tld{w}_1\un{W}_a/\un{W}_a &\mapsto \zeta,
\end{align}
which does not depend on the choice of representative in the equivalence class of $(\tld{w}_1,\omega)$.
Then we say that the lowest alcove presentation $(\tld{w}_1,\omega)$ of $F_{(\tld{w}_1,\omega)}$ is \emph{compatible with $\zeta \in X^*(\un{Z})$}.
{The following lemma shows that $\zeta$ is an algebraic lift of the central character of the Serre weight.}
\begin{lemma}
Let $(\tld{w}_1,\omega)$ be a lowest alcove presentation of a Serre weight $\sigma$.
Let $\zeta$ be the algebraic central character associated to $(\tld{w}_1,\omega)$ by \eqref{eqn:wtchar}.
Then the class of $\zeta$ in $\pmod{(p-\pi)X^*(\un{Z})}=\Hom(Z_0(\F_p),\F)$ is the central character of $\sigma$ as a $\rG$-representation.
\end{lemma}
\begin{proof}
This follows from the description of $\sigma$ as a the restriction to $\rG$ of the irreducible algebraic highest weight module with highest weight $\pi^{-1}(\tld{w}_1)\cdot (\omega-\eta)$.
\end{proof}
We say that two lowest alcove presentations of Serre weights are compatible (with each other) if they are compatible with the same element of $X^*(\un{Z})$.
As the $p$-dot action preserves depth, $F_{(\tld{w}_1,\omega)}$ is $m$-deep if and only if $\omega-\eta$ is $m$-deep (Definition \ref{defn:var:gen}(\ref{defn:deep})) in alcove $\un{C}_0$, i.e.~if $m<\langle \omega, \alpha^\vee\rangle<p-m$ for all $\alpha \in \un{\Phi}^+$.

\begin{lemma}\label{lemma:LAPwt}
If a Serre weight $\sigma$ is $0$-deep, then the map $(\tld{w}_1,\omega) \mapsto \tld{w}_1 t_\omega \un{W}_a/\un{W}_a \in\tld{\un{W}}/\un{W}_a \cong X^*(\un{Z})$ gives a bijection between lowest alcove presentations of $\sigma$ and algebraic central characters lifting the central character of $\sigma$.
\end{lemma}
\begin{proof}
If $(\tld{w}_1,\omega)$ is a lowest alcove presentation for $\sigma$, then the set of lowest alcove presentations of $\sigma$ is  \[\{(\tld{w}_1\pi(\delta^{-1}),\delta \cdot(\omega-\eta)+\eta): \delta \in \un{\Omega}\}\]
(where we write one pair in each equivalence class). 
If $(\tld{w}_1,\omega)$ maps to $\zeta \in X^*(\un{Z})$, then the lowest alcove presentation $(\tld{w}_1\pi(\delta^{-1}),\delta \cdot(\omega-\eta)+\eta)$ maps to $\zeta+(p-\pi)\zeta_\delta$ where $\zeta_\delta\in X^*(\un{Z})$ is the image of $\delta$ under the isomorphisms $\un{\Omega} \cong \tld{\un{W}}/\un{W}_a \cong X^*(\un{Z})$.
\end{proof}

\subsection{Deligne--Lusztig representations and their mod $p$ reductions}
\label{sec:DLandSW}

Let $(s,\mu)\in \un{W}\times X^*(\un{T})$ be a \emph{good pair} (\cite[\S 2.2]{LLL}).
Using \cite[Proposition 9.2.1 and 9.2.2]{GHS}, we can attach to $(s,\mu)$ a Deligne--Lusztig representation $R_s(\mu)$ of $\rG$ defined over $E$. 
We say that $(s,\mu-\eta)$ is a \emph{lowest alcove presentation} of $R_s(\mu)$ if $\mu-\eta \in \un{C}_0$.
(Again, this notion depends on the choice of $\eta$.)

\begin{defn}
\label{defn:LAP:DL}
Let $m\geq 0$ and let $R$ be a Deligne--Lusztig representation.
We say that $R$ is $m$-generic if there exists a lowest alcove presentation $(s,\mu-\eta)\in \un{W} \times \un{C}_0$ for $R$ such that $\mu-\eta$ is $m$-deep (Definition \ref{defn:var:gen}(\ref{defn:deep})). We call such a presentation an \emph{$m$-generic lowest alcove presentation}.
If $R$ has a fixed lowest alcove presentation $(s,\mu-\eta)$, define $\tld{w}(R)\defeq t_\mu s\in \tld{\un{W}}$ and $w(R) \defeq s \in \un{W}$.
Note that $\mu- \eta$ being $m$-deep is equivalent to $\tld{w}(R)$ being $m$-generic in the sense of Definition \ref{defn:var:gen}(\ref{it:gen:weyl}).
\end{defn}
\noindent Note that $(s,\mu)\in\un{W}\times X^*(\un{T})$ is good if $\mu-\eta$ is ($0$-deep) in alcove $\un{C}_0$ by \cite[Lemma 2.2.3]{LLL}.
By \cite[Theorem 6.8]{DeligneLusztig}, we see that a $1$-generic Deligne--Lusztig representation is irreducible.

Let $\lambda \in X^*(\un{T})$ be a character.
We say that a lowest alcove presentation $(s,\mu-\eta)$ of a Deligne--Lusztig representation is \emph{$\lambda$-compatible with an algebraic central character $\zeta\in X^*(\un{Z})$} if the image of the element $t_\lambda t_\mu s \un{W}_a/\un{W}_a \in \tld{\un{W}}/\un{W}_a \cong X^*(\un{Z})$ corresponds to $\zeta$.
Instead of saying $\un{0}$-compatible, we just say compatible.
If $(s,\mu-\eta)$ is a lowest alcove presentation of $R$ compatible with $\zeta$, then $\zeta \mod{(p-\pi)X^*(\un{Z})}$ corresponds to the central character of $R$.
We say that lowest alcove presentations of a Deligne--Lusztig representations are compatible if they are compatible with the same algebraic central character.
We say that lowest alcove presentations $(s,\mu-\eta)$ and $(\tld{w}_1,\omega)$ of a Deligne--Lusztig representation $R$ and a Serre weight $\sigma$ are $\lambda$-compatible if $(s,\mu-\eta)$ and $(\tld{w}_1,\omega)$ are $\lambda$-compatible and compatible, respectively, with some $\zeta\in X^*(Z)$.

\begin{lemma} \label{lemma:LAPbij}
If $R$ is a $1$-generic Deligne--Lusztig representation, the map $(s,\mu-\eta) \mapsto t_\mu s \un{W}_a/\un{W}_a \in \tld{\un{W}}/\un{W}_a \cong X^*(\un{Z})$ gives a bijection between lowest alcove presentations of $R$ and algebraic central characters lifting the reduction of the central character of $R$.
\end{lemma}
\begin{proof}
If $(s,\mu-\eta)$ is a $1$-generic lowest alcove presentation for $R$, then by \cite[Proposition 2.2.15]{LLL} the set of lowest alcove presentations for $R$ is 
\[
\{ \big(ws\pi(w)^{-1},w(\mu+p\nu-s\pi(\nu)) - \eta\big): wt_{\nu} \in \un{\Omega}\}.
\]
Note that each of $w(\mu+p\nu-s\pi(\nu)) - \eta$ is $0$-deep in $\un{C}_0$.
Since the image of $t_{w(\mu+p\nu-s\pi(\nu))}ws\pi(w)^{-1}$ in $X^*(\un{Z})$ is $\mu+p\nu-s\pi(\nu)|_{\un{Z}} = \mu|_{\un{Z}}+(p-\pi)\nu|_{\un{Z}}$, it suffices to note that the image of $wt_{\nu}$ under the isomorphism $\un{\Omega}\cong\tld{\un{W}}/\un{W}_a \cong X^*(\un{Z})$ is $\nu|_{\un{Z}}$.
\end{proof}

\begin{lemma}\label{lemma:genLAP}
If $R$ is a Deligne--Lusztig representation, then $R \cong R_s(\mu)$ for some $(s,\mu) \in \un{W} \times X^*(\un{T})$ such that $\mu$ is dominant and $\langle \mu,\alpha^\vee \rangle \leq p+2$ for all $\alpha \in \un{\Phi}$.
In particular, $\mu-\eta$ is $(-3)$-deep in $\un{C}_0$.
\end{lemma}
\begin{proof}
Suppose that $R = R_s(\mu)$ for $(s,\mu) \in \un{W} \times X^*(\un{T})$.
Then $R = R_s(\mu+p\nu-s\pi (\nu))$ for any $\nu \in X^*(\un{T})$.
Since $X^*(\un{T}) (\un{W} (\ovl{\un{A}}_0)) = \tld{\un{W}}(\ovl{\un{A}}_0) = X^*(\un{T})\otimes_{\Z} \R$ where $\ovl{\un{A}}_0$ denotes the closure of the base alcove $\un{A}_0$, there exists $\nu\in X^*(\un{T})$ such that $h_{\mu+p\nu} \leq p$.
Then $h_{p\nu} \leq h_\mu+ h_{\mu-p\nu} \leq h_\mu+p$ so that $h_\nu \leq \lfloor \frac{h_\mu}{p} \rfloor+1$ and therefore $h_{\mu+p\nu-s\pi(\nu)} \leq p+h_\nu \leq p+\lfloor \frac{h_\mu}{p} \rfloor+1$.
Repeatedly replacing $\mu$ with $\mu+p\nu-s\pi (\nu)$ as above, we eventually have that $h_\mu \leq p+2$.
Finally, we replace $(s,\mu)$ with $(ws\pi(w)^{-1},w(\mu))$ where $w\in \un{W}$ is such that $w(\mu)$ is dominant.
\end{proof}

\begin{lemma}\label{lemma:nongenericDL}
Let $\lambda\in X^*(\un{T})$ be a dominant weight such that $h_{\lambda+\eta}<p-3$.
If $R$ is a Deligne--Lusztig representation such that $\JH(\ovl{R} \otimes W(\lambda))$ contains an $m$-deep Serre weight, then there exists a pair $(s,\mu) \in \un{W} \times X^*(\un{T})$ such that $R = R_s(\mu)$ and $\mu-\eta$ is $(m-h_{\lambda+2\eta})$-deep in $\un{C}_0$.
\end{lemma}
\begin{proof}
Let $R = R_s(\mu)$ for some $(s,\mu)$ as in Lemma \ref{lemma:genLAP} and that $\sigma$ is a Serre weight in $\JH(\ovl{R} \otimes W(\lambda))$.
We assume that that $\mu-\eta$ is not $(m-h_{\lambda+2\eta})$-deep in $\un{C}_0$ and will show that $\sigma$ is not $m$-deep.
Note that $\mu-\eta$ is $(-3)$-deep in $\un{C}_0$ so that in particular $m \geq h_{\lambda+2\eta}-2$.
By \cite[\S A.3.4]{herzig-duke}, $\sigma\in \JH(W(\pi^{-1}(\tld{w}_1) \cdot (\tld{w}(R)\tld{w}_2^{-1}(\eta) - \eta)) \otimes W(\lambda))$ for some $\tld{w}_1,\, \tld{w}_2 \in \tld{\un{W}}^+_1$ (in fact, necessarily $\tld{w}_2\in \tld{w}_1\un{W}_a$) where $\tld{w}(R) \defeq t_\mu s$.
Then by the proof of Lemma \ref{lemma:nongenWeyl}, $\sigma\in \JH(W(\nu))$ for $\nu \in (w\pi^{-1}(\tld{w}_1)) \cdot (\tld{w}(R)\tld{w}_2^{-1}(\eta) - \eta)+\Conv(\lambda)$ for $w\in \un{W}$ with $(w\pi^{-1}(\tld{w}_1)) \cdot (\tld{w}(R)\tld{w}_2^{-1}(\eta) - \eta)$ dominant and $\nu$ dominant.
In particular, the depth assumption on $\mu$ implies that $\nu$ is not $(m-h_\eta)$-deep.
Furthermore, since $\mu-\eta$ is $(-3)$-deep in $\un{C}_0$, $(w\pi^{-1}(\tld{w}_1)) \cdot (\tld{w}(R)\tld{w}_2^{-1}(\eta) - \eta)$ is $(-3-h_\eta)$-deep in a $p$-restricted alcove.
In particular, for all $\alpha \in \un{\Phi}$ (we can assume that $\alpha$ is a highest root by dominance), 
$\langle (w\pi^{-1}(\tld{w}_1)) \cdot (\tld{w}(R)\tld{w}_2^{-1}(\eta) - \eta),\alpha^\vee \rangle \leq ph_\eta+2$ so that $\langle \nu,\alpha^\vee \rangle \leq ph_\eta + h_\lambda+2$.
The result now follows from Lemma \ref{lemma:nongenWeyl}.
\end{proof}

\begin{rmk}
If $\un{G}$ is a product of copies of $\GL_n$, one can show that $\mu-\eta$ can be taken to be $(-1)$-deep.
One can then assume instead that $h_{\lambda+\eta} < p-1$ in Lemma \ref{lemma:nongenericDL}.
\end{rmk}

For $\lambda\in X^*(\un{T})$ dominant, recall from \S \ref{sec:SW} that $W(\lambda)_{/\F}$ denotes the dual Weyl module of highest weight $\lambda$ for the split algebraic group $\un{G}_{/\F}$ and that $W(\lambda)$ is the restriction of $W(\lambda)_{/\F}\,(\F)$ to $\rG\subseteq \un{G}(\F)$.
If $R$ is $\lambda$-compatible with $\zeta\in X^*(\un{Z})$, then $\zeta \mod{(p-\pi)X^*(\un{Z})}$ gives the central character of $\ovl{R}\otimes W(\lambda)\defeq \ovl{R}\otimes_{\F} W(\lambda)$.
The set $\JH(\ovl{R}\otimes W(\lambda))$ has the following combinatorial description in terms of $\tld{\un{W}}$. 
We also use $\uparrow$ to denote the ordering on $X^*(\un{T})$ defined in \cite[II.6.4]{RAGS}.

From (the proof of) \cite[Proposition 4.1.3]{LLL} we have:
\begin{prop}\label{prop:JH}
Let $R$ be a Deligne--Lusztig representation with a $2h_\eta$-generic lowest alcove presentation $(s, \mu-\eta)$. %
Let $\lambda\in X_1(\un{T})$.
Then $F(\lambda)\in \JH(\ovl{R})$ if and only if there exists $\tld{w}=wt_\nu\in \tld{\un{W}}^+$ such that $\tld{w}\cdot (\mu-s\pi(\nu)-\eta)\uparrow \tld{w}_h\cdot \lambda$ and $\tld{w}\cdot \un{C}_0\uparrow \tld{w}_h\cdot \un{C}_0$.
\end{prop}

We have the following parametrization of Jordan--H\"older factors of $\ovl{R}\otimes W(\lambda)$ in terms of admissible pairs from \S \ref{sub:AWG:2}.

\begin{prop}\label{prop:JHbij}
Let $\lambda\in X^*(\un{T})$ be a dominant weight and let $m\geq \max\{2h_\eta,h_{\lambda+\eta}\}$ be an integer.
Let $R$ be a Deligne--Lusztig representation together with an $m$-generic lowest alcove presentation, with corresponding element $\tld{w}(R)\in \tld{\un{W}}$ \emph{(}cf.~Definition \ref{defn:LAP:DL}\emph{)}.

Then the map 
\begin{align} 
\nonumber \mathrm{AP}(\lambda+\eta) &\ra \JH(\ovl{R}\otimes W(\lambda))\\
(\tld{w}_1,\tld{w}_2) &\mapsto F_{(\tld{w}_1,\tld{w}(R)\tld{w}_2^{-1}(0))} \label{eqn:JH}
\end{align}
is a bijection. 
Moreover, these Jordan--H\"older factors are $(m - h_{\lambda+\eta})$-deep and the lowest alcove presentations $(\tld{w}_1,\tld{w}(R)\tld{w}_2^{-1}(0))$ of these Serre weights are $\lambda$-compatible with the lowest alcove presentation %
of $R$.
\end{prop}
\begin{proof}
Since $\tld{w}_1 t_{\tld{w}(R)\tld{w}_2^{-1}(0)-\eta} \un{W}_a/\un{W}_a = \tld{w}_1 \tld{w}(R) \tld{w}_2^{-1}t_{-\eta} \un{W}_a/\un{W}_a = t_\lambda \tld{w}(R) \un{W}_a/\un{W}_a$, the lowest alcove presentations $(\tld{w}_1,\tld{w}(R)\tld{w}_2^{-1}(0))$ of Serre weights are $\lambda$-compatible with the given lowest alcove presentation %
of $R$.
If $(\tld{w}_1,\tld{w}_2)\in \mathrm{AP}(\lambda+\eta)$, then $\tld{w}_2 \uparrow t_{-w_0(\lambda)} \tld{w}_h \tld{w}_1$, so that $\langle \tld{w}_2^{-1}(0),\alpha^\vee\rangle \leq h_{\lambda+\eta}$ for all $\alpha \in \un{\Phi}$.
This implies that $F_{(\tld{w}_1,\tld{w}(R)\tld{w}_2^{-1}(0))}$ is $(m - h_{\lambda+\eta})$-deep.
Lemma \ref{lemma:LAPwt} finally implies that (\ref{eqn:JH}) is injective.

We next show the image of (\ref{eqn:JH}) is $\JH(\ovl{R} \otimes W(\lambda))$.
By the translation principle and Proposition \ref{prop:JH}, every element of $\JH(\ovl{R} \otimes W(\lambda))$ is of the form
\begin{equation}\label{eqn:trans}
F(\pi^{-1}(\tld{w}_1) \cdot (\tld{w}(R)((\tld{w}'_2)^{-1}(0) +\omega) - \eta))
\end{equation}
for some $\tld{w}_1\in \tld{\un{W}}_1^+$ and $\tld{w}_2'\in \tld{\un{W}}^+$ such that $\tld{w}_1 \uparrow \tld{w}_h^{-1} \tld{w}_2'$ and some $\omega\in \mathrm{Conv}(\lambda)$.
Let $\tld{w}_2 = w t_{-\omega-(\tld{w}'_2)^{-1}(0)}$ be the unique element in $\un{W}t_{-\omega-(\tld{w}'_2)^{-1}(0)} \cap \tld{\un{W}}^+$.
Since $\omega\in \mathrm{Conv}(\lambda)$ implies that $t_{w(-\omega)}\uparrow t_{-w_0(\lambda)}$ and $\tld{w}_2'\in \tld{\un{W}}^+$ implies that $wt_{-(\tld{w}'_2)^{-1}(0)} \uparrow w_2' t_{-(\tld{w}'_2)^{-1}(0)} = \tld{w}_2'$ (where $w_2' \in \un{W}$ is the image of $\tld{w}'_2$), we see that $\tld{w}_2 =  t_{w(-\omega)}wt_{-(\tld{w}'_2)^{-1}(0)} \uparrow t_{-w_0(\lambda)} \tld{w}_2'$, which is equivalent to the inequality $\tld{w}_h^{-1} \tld{w}_2' \uparrow t_{\lambda} \tld{w}_h^{-1}\tld{w}_2$ by \cite[Proposition 4.1.2]{LLL}.
This implies the desired inequality $\tld{w}_1 \uparrow t_{\lambda}\tld{w}_h^{-1}\tld{w}_2$.

For the converse, suppose that $\tld{w}_1 \uparrow t_{\lambda}\tld{w}_h^{-1}\tld{w}_2$.
Equivalently, by \cite[Proposition 4.1.2]{LLL} and Wang's theorem (\cite[Theorem 4.1.1]{LLL}), we have $\tld{w}_2 \leq t_{-w_0(\lambda)}(\tld{w}_h\tld{w}_1)$. 
Since the latter factorization is reduced, we have $\tld{w}_2 = \tld{x} \tld{y}$ where $\tld{x} \leq t_{-w_0(\lambda)}$ and $\tld{y} \leq \tld{w}_h\tld{w}_1$.
Then by \cite{KR}, $\tld{x}$ is $-w_0(\lambda)$-permissible, and in particular, $\tld{x}(0) \in \mathrm{Conv}(-w_0(\lambda))$.
Taking $\tld{w}_2' \in \un{W}\tld{y} \cap \un{\tld{W}}^+$, we conclude that $\tld{w}_2 = wt_{-\omega} \tld{w}_2'$ for some $\omega\in \mathrm{Conv}(\lambda)$ and $\tld{w}_2' \in \tld{W}^+$ with $\tld{w}_2' \uparrow \tld{w}_h \tld{w}_1$ (equivalently $\tld{w}_1\uparrow \tld{w}_h^{-1} \tld{w}_2'$).
Then $F(\pi^{-1}(\tld{w}_1) \cdot (\tld{w}(R)\tld{w}_2^{-1}(0) - \eta))$ has the form of (\ref{eqn:trans}).
\end{proof}

We use Proposition \ref{prop:JHbij} to give another description of $\JH(\ovl{R} \otimes W(\lambda))$.

\begin{prop}\label{prop:JHfixed}
Let $\lambda\in X^*(\un{T})$ be dominant and suppose that $(\tld{w}_1,\omega)$ and $(s,\mu-\eta)$ are $\lambda$-compatible lowest alcove presentations of a Serre weight $\sigma$ and a Deligne--Lusztig representation, respectively.
Suppose further that $(s,\mu-\eta)$ is $\max\{2h_\eta,h_{\lambda+\eta}\}$-generic.
Then $\sigma \in \JH(\ovl{R} \otimes W(\lambda))$ if and only if 
\begin{equation}\label{eqn:JHfixed}
t_\omega \tld{\un{W}}_{\leq w_0\tld{w}_1} \subset t_\mu s \Adm(\lambda+\eta).
\end{equation}
\end{prop}
\begin{proof}
As usual we let $\tld{w}(R) \defeq t_\mu s$.
Let $\tld{w}_2 \in \tld{\un{W}}^+$ be the unique element such that $t_{-\omega} \tld{w}(R) \in \un{W}\tld{w}_2$.
Note that $\omega = \tld{w}(R) \tld{w}_2^{-1}(0)$.
By Proposition \ref{prop:JHbij} and Proposition \ref{prop:can:adm}, it suffices to show that $\tld{w}_2^{-1}w_0\tld{w}_1 \in \Adm(\lambda+\eta)$ is equivalent to (\ref{eqn:JHfixed}).
If $\tld{w}_2^{-1}w_0\tld{w}_1 \in \Adm(\lambda+\eta)$, then
\[t_\omega \tld{\un{W}}_{\leq w_0\tld{w}_1} \subset \tld{w}(R) \tld{w}_2^{-1} \tld{\un{W}}_{\leq w_0\tld{w}_1} \subset \tld{w}(R) \Adm(\lambda+\eta),\] where the first inclusion follows from the fact that $\tld{\un{W}}_{\leq w_0\tld{w}_1}$ is $\un{W}$-stable under left multiplication.

For the backwards direction, assume (\ref{eqn:JHfixed}).
Then in particular,
\[
\tld{w}(R)\tld{w}_2^{-1} w_0 \tld{w}_1 \in t_\omega \un{W} \tld{w}_1 \subset \tld{w}(R) \Adm(\lambda+\eta).
\]
\end{proof}

\subsubsection{The covering order}
Having discussed the reductions of Deligne--Lusztig representations, we now use these results to define a partial ordering on Serre weights that arises naturally in \S \ref{sec:BMC}.

For a $\max\{2h_\eta,h_{\lambda+\eta}\}$-generic Deligne--Lusztig representation $R$, let $\JH_\out(\ovl{R}\otimes W(\lambda))$ be the subset of $\JH(\ovl{R}\otimes W(\lambda))$ corresponding by (\ref{eqn:JH}) to elements of $\mathrm{AP}(\lambda+\eta)$ of the form $(\tld{w}_1,\tld{w}_ht_{-\lambda}\tld{w}_1)$.
We begin with the following lemma.

\begin{lemma}\label{lemma:separate}
Suppose that $F_{(\tld{w}',\omega')} \in \JH_\out(\ovl{R})$. 
Fix the compatible lowest alcove presentation of $R$ with corresponding element $\tld{w}(R)$ as in Definition \ref{defn:LAP:DL}.
If $(\tld{w},\omega)$ is a compatible lowest alcove presentation of a weight $F_{(\tld{w},\omega)} \in \JH(\ovl{R})$, then \emph{(}choosing any representatives of the equivalence class of lowest alcove presentations\emph{)}
\[
\tld{w} \uparrow t_{w'w(R)^{-1}(\omega' - \omega)}\tld{w}'
\]
where $w'$ denotes the image of $\tld{w}'$ in $\un{W}$.
\end{lemma}
\begin{proof}
We introduce the following notation for this proof.
If $\tld{s} \in \tld{\un{W}}$, then let $\tld{s}^+$ denote the unique element in $\un{W}\tld{s} \cap \tld{\un{W}}^+$.
Since $F_{(\tld{w}',\omega')} \in \JH_\out(\ovl{R})$, we have that $\omega' = \tld{w}(R)(\tld{w}_h\tld{w}')^{-1}(0)$ so that 
\[
\tld{w}_h\tld{w}' = w_0w't_{-\tld{w}(R)^{-1}(\omega')}.
\]
Since $F_{(\tld{w},\omega)} \in \JH(\ovl{R})$, we have that $\omega = \tld{w}(R)\tld{w}_2^{-1}(0)$ for some $\tld{w} \uparrow \tld{w}_h^{-1}\tld{w}_2$ by (\ref{eqn:JH}).
Then $\tld{w}_2$ is $(t_{-\tld{w}(R)^{-1}(\omega)})^+$, so that $(t_{-\tld{w}(R)^{-1}(\omega)})^+ \uparrow \tld{w}_h \tld{w}$ by \cite[Proposition 4.1.2]{LLL}.
Let $\nu$ be $w(R)^{-1}(\omega' - \omega)$.
On the other hand, 
\begin{align*}
t_{w_0w'(\nu)}\tld{w}_h\tld{w}' &=w_0w' t_{\nu-\tld{w}(R)^{-1}(\omega')} \\
&= w_0w' t_{-\tld{w}(R)^{-1}(\omega)} \\
&\uparrow (t_{-\tld{w}(R)^{-1}(\omega)})^+.
\end{align*}
We conclude that $t_{w_0w'(\nu)}\tld{w}_h\tld{w}' \uparrow \tld{w}_h \tld{w}$, or equivalently, $\tld{w} \uparrow t_{w'(\nu)}\tld{w}'$.
\end{proof}

\begin{defn}\label{defn:cover}
Let $\sigma_0$ be a $3h_\eta$-deep Serre weights.
If 
\[
\sigma \in \underset{\substack{R \, 2h_\eta\textrm{-generic}, \\ \sigma_0 \in \JH(\overline{R})}}{\bigcap} \JH(\overline{R}),
\]
where $R$ runs over $2h_\eta$-generic Deligne--Lusztig representations, then we say that $\sigma_0$ \emph{covers} $\sigma$.
{In other words, $\sigma_0$ cover $\sigma$ if every $2h_\eta$-generic Deligne--Lusztig representation containing $\sigma_0$ also contains $\sigma$.}
\end{defn}

\begin{rmk}\label{rmk:partialorder}
Note that covering is a partial ordering on $3h_\eta$-deep Serre weights.
\end{rmk}

The following alternate criteria for covering are sometimes useful.

\begin{prop}\label{prop:covering}
Suppose that $(\tld{w},\omega)$ and $(\tld{w}',\omega')$ are \emph{(}representatives for\emph{)} compatible lowest alcove presentations of Serre weights and $\omega-\eta$ is $3h_\eta$-deep.
The following are equivalent:
\begin{enumerate}
\item \label{it:covering:1}
$F_{(\tld{w},\omega)}$ covers $F_{(\tld{w}',\omega')}$;
\item \label{it:covering:2}
$\tld{w}' \uparrow t_{\un{W}(\omega-\omega')}\tld{w}$ \emph{(}in particular $\ell(\tld{w}')\leq \ell(\tld{w})$, with equality if and only if $(\tld{w},\omega) \sim (\tld{w}',\omega')$\emph{)}; and
\item \label{it:covering:3}
$t_{\omega'} \tld{\un{W}}_{\leq w_0\tld{w}'} \subset t_{\omega} \tld{\un{W}}_{\leq w_0\tld{w}}$; and
\item \label{it:covering:4}
$F_{(\tld{w}',\omega')}$ is a Jordan--H\"older factor of 
\[
\underset{\substack{\tld{w}_1 \in \tld{\un{W}}^+ \\ \tld{w}_1\uparrow \tld{w}}}{\bigoplus} L(\pi^{-1}(\tld{w}_1)\cdot(\omega-\eta))|_{\rG}.
\]
\end{enumerate}
\end{prop}
\begin{proof}
We first show that (\ref{it:covering:1}) implies (\ref{it:covering:2}).
Let $\nu_0$ be $ (\tld{w}_h\tld{w})^{-1}(0)$.
Consider the set $X$ of $\tld{x}\in \tld{\un{W}}$ such that $\omega = \tld{x}(\nu_0)$.
We claim that $\tld{x}$ is $2h_\eta$-generic.
Indeed, since $\tld{w}_h\tld{w} \in \tld{\un{W}}_1^+$, we have $\langle \nu_0,\alpha^\vee\rangle \leq h_\eta$ for all $\alpha\in \un{\Phi}$.
Then the claim follows from the fact that $\omega-\eta$ is $3h_\eta$-deep and $\tld{x}(0) = \omega - x(\nu_0)$ (where $x\in{\un{W}}$ is the image of $\tld{x}$). 

From the above paragraph, the map taking the set of $2h_\eta$-generic Deligne--Lusztig representations $R$ with a lowest alcove presentation compatible with $(\tld{w},\omega)$ to $\tld{\un{W}}$ sending $R$ to $\tld{w}(R)$ induces a bijection between those $R$ with $F_{(\tld{w},\omega)} \in \JH_\out(\ovl{R})$ and $X$.
Moreover, the map $X \ra \un{W}$, induced by the natural quotient map $\tld{\un{W}} \ra \un{W}$, is a bijection.
If $F_{(\tld{w},\omega)}$ covers $F_{(\tld{w}',\omega')}$, then Lemma \ref{lemma:separate} implies that $\tld{w}'\uparrow t_{\un{W}(\omega - \omega')}\tld{w}$.
We now show the parenthetical.
Fix $s \in \un{W}$ so that $s(\omega'-\omega)$ is dominant.
We have that $t_{s(\omega'-\omega)} \tld{w}' \uparrow \tld{w}$ so that $t_{s(\omega'-\omega)} \tld{w}' \leq \tld{w}$ by Wang's theorem.
Since $t_{s(\omega'-\omega)} \tld{w}'$ is a reduced expression (counting galleries), 
\[
\ell(\tld{w}') \leq \ell(t_{s(\omega'-\omega)}) + \ell(\tld{w}') = \ell(t_{s(\omega'-\omega)} \tld{w}') \leq \ell(\tld{w}).
\]
If this is an equality, then $\omega'-\omega\in X^0(\un{T})$ and $\tld{w}^{-1}\tld{w}' \in \un{\Omega}$.
Compatibility of lowest alcove presentations implies that $(\tld{w},\omega)\sim (\tld{w}'\omega')$.

We next show that (\ref{it:covering:2}) implies (\ref{it:covering:3}).
Assuming (\ref{it:covering:2}), \cite[II.6.5(5)]{RAGS} shows that for any $\tld{s}'\uparrow \tld{w}'$ with $\tld{s}' \in \tld{\un{W}}^+$,
\[
\un{W}\tld{s}'\uparrow \tld{s}' \uparrow t_{\un{W}(\omega-\omega')}\tld{w},
\]
so that
\[
\un{W}t_{\omega'-\omega} \un{W} \tld{s}' \uparrow \tld{w}.
\]
Then Wang's theorem implies that $t_{\omega' -\omega}\un{W} \tld{s}' \leq w_0\tld{w}$. 
As any element of $\tld{\un{W}}_{\leq w_0 \tld{w}'}$ is in $\un{W} \tld{s}'$ for some $\tld{s}'$ as above, the result follows.

That (\ref{it:covering:3}) implies (\ref{it:covering:1}) follows from Proposition \ref{prop:JHfixed}.
Finally, we show the equivalence of (\ref{it:covering:2}) and (\ref{it:covering:4}).
Using the Steinberg tensor product theorem and the translation principle \cite[Lemma 4.2.4(1)]{LLLM2}, if we write $\tld{w}_1\in \tld{\un{W}}^+$ as $t_\nu \tld{w}'_1$ with $\tld{w}'_1\in \tld{\un{W}}^+_1$, then 
\begin{equation}\label{eqn:steinberg}
L(\pi^{-1}(\tld{w}_1)\cdot(\omega-\eta))|_{\rG} \cong \bigoplus_{\nu' \in L(\nu)} F_{(\tld{w}'_1,\omega+\nu')}^{\oplus m(\nu,\nu')}, 
\end{equation}
where $m(\nu,\nu')$ is the multiplicity of the $\nu'$-weight space in $L(\nu)$.
Consideration of the chain of inequalities $t_{\un{W}\nu'}\tld{w}_1'\uparrow \tld{w}_1 \uparrow \tld{w}$ shows that (\ref{it:covering:4}) implies (\ref{it:covering:2}). 
Conversely, if $t_{\un{W}(\omega'-\omega)} \tld{w}' \uparrow \tld{w}$, then $F_{(\tld{w}',\omega')} \in \JH(L(\pi^{-1}(\tld{w}_1)\cdot (\omega-\eta))|_{\rG})$ where $\tld{w}_1 = t_{s(\omega'-\omega)} \tld{w}'$ and $s(\omega'-\omega)$ is the dominant weight in the Weyl orbit of $\omega'-\omega$.
\end{proof}

\begin{rmk}\begin{enumerate}
\item
The equivalence of (\ref{it:covering:1}) and (\ref{it:covering:4}) in Proposition \ref{prop:covering} shows that if $\mu \uparrow \lambda$, then $F(\lambda)$ covers $F(\mu)$.
The converse does not hold: for $\GL_4$, $\tld{w}_1$ could be in $t_{(1,0,0,0)}\Omega$ if $\tld{w} \in \tld{w}_h \Omega$. 
Then $\nu = (1,0,0,0)$ so that $\omega'-\omega$ would be nonzero and the Serre weights on the right hand side of (\ref{eqn:steinberg}) are not in the $\tld{W}_a$ $p$-dot orbit of $\pi(\tld{w})^{-1}\cdot(\omega-\eta)$.
\item
By the same equivalence and the linkage principle \cite[II.6.13]{RAGS}, if $\sigma \in \JH(W(\tld{w}\cdot(\omega-\eta)))$, then $F_{(\tld{w},\omega)}$ covers $\sigma$.
The converse does not hold for $\GL_4$ (see \cite{Jantzen74} or \cite[Proposition 9.3]{florian-thesis}).
\end{enumerate}
\end{rmk}

\subsection{Tame inertial $L$-parameters} \label{sec:InertialTypes}

Recall the fundamental characters $\omega_d: I_{\Qp} \ra \cO^\times$ defined in \S \ref{sec:not:Gal}.
For $(w,\mu) \in \un{W} \times X^*(\un{T})$, we let $\tau(w,\mu)$ be the tame inertial $L$-parameter over $E$ given by
\[
\Big(\sum_{i=0}^{d-1} (F^*\circ w^{-1})^i(\mu) \Big) (\omega_d): I_{\Qp} \ra {\un{T}^\vee}(E)
\]
where we view $\mu$ here as an element of $X_*({\un{T}^\vee})$, $F^*$ is defined to be the endomorphism $p\pi^{-1}$ on $X_*({\un{T}^\vee})$, and $d \geq 1$ is an integer such that $(F^*\circ w^{-1})^d = p^d$.
(The tame inertial $L$-parameter does not depend on the choice of $d$.)
Let $\taubar(w,\mu)$ be the inertial $L$-parameter over $\F$ obtained by reduction modulo $\varpi$.
All tame inertial $L$-parameters over $E$ and $\F$ arise in this way.

\begin{exam}
\label{ex:data:type}
Suppose that $F^+_p$ is the field $K$ and $G$ is $\GL_n$.
As explained in \S \ref{sec:not:Gal}, to the tame inertial $L$-parameter $\tau(s,\mu+\eta)$, there is a corresponding tame inertial type for $K$ which we also denote by $\tau(s,\mu+\eta)$.
Fix an isomorphism $\iota: \ovl{K} \risom \ovl{\Q}_p$.
This gives a homomorphism $I_K \ra \GL_n(E)$, which we will make explicit.
The isomorphism $\iota$ gives an injection $G_K \into G_{\Q_p}$.
Let $d$ be a positive integer.
Then induced composition $I_K \overset{\iota}{\ra} I_{\Q_p} \overset{\omega_d}{\ra} \cO^\times$ is $\omega_{K_d,\sigma_0'}$ where $K_d$ is the subfield of $\ovl{K}$ generated by the $p^d-1$-st roots of unity and $\sigma_0': K_d \ra E$ denotes the restriction of $\iota$ (taking $E$ sufficiently large). 
We denote $\omega_{K_d,\sigma_0'}$ by $\omega_d$ as well.

Let $\sigma_0: K \ra E$ denote the restriction of $\iota$ to $K$.
As in \S \ref{sec:not:Gal}, we let $\sigma_j$ be $\sigma_0 \circ \varphi^{-j}$ for $j \in \Z/f\Z$, identifying $\cJ$ with $\Z/f\Z$.
If $s = (s_0, \ldots, s_{f-1})$, then set $s_{\tau} = s_0 s_1 \cdots s_{f-1} \in W$.
Then $(F^*\circ s^{-1})^f = (F^*)^f \circ (s_\tau^{-1},\ldots) = p^f (s_\tau^{-1},\ldots)$ where the unspecified components are conjugates of $s_\tau$ so that $(F^* \circ s^{-1})^{fr} = p^{fr}$ where $r$ is the order of $s_\tau$.
Let $\bm{\alpha}_j \defeq p^{-j}\Big((F^*\circ s^{-1})^j(\mu+\eta)\Big)_0 \in X^*(T)$ for $0\leq j \leq f-1$, so that $\bm{\alpha}_j = s_{f-1}^{-1} s_{f-2}^{-1} \ldots s_{f-j}^{-1}(\mu_{f-j}+\eta_{f-j})$ (and $\bm{\alpha}_{0} = \mu_0+\eta_0$).   We also define $\bf{a}^{(0)} \defeq \Big(\sum_{j=0}^{f-1} (F^*\circ s^{-1})^j(\mu+\eta) \Big)_0 = \sum_{j=0}^{f-1} p^j \bm{\alpha}_j \in X^*(T)$.  (Note that the conventions here are different from \cite{LLLM, LLL} as explained in detail in Remark \ref{rmk:cmpr:mat}.)

We have
\begin{align*}
\Big(\sum_{i=0}^{fr-1} (F^*\circ s^{-1})^i(\mu+\eta) \Big)_0 &= \Big(\sum_{k=0}^{r-1} \sum_{j=0}^{f-1} (F^*\circ s^{-1})^{fk} (F^*\circ s^{-1})^j(\mu+\eta) \Big)_0\\
&= \sum_{k=0}^{r-1} \sum_{j=0}^{f-1} p^{fk} s_\tau^{-k} p^j \bm{\alpha}_j \\
&= \sum_{k=0}^{r-1} p^{fk} s_\tau^{-k} \bf{a}^{(0)}.
\end{align*}
We conclude that $\tau(s,\mu+\eta)$ is $\sum_{k=0}^{r-1} p^{fk} s_\tau^{-k}\bf{a}^{(0)}(\omega_{fr})$.
More concretely, setting $\chi_i\defeq \omega_{fr}^{\sum_{0 \leq k \leq r-1} \bf{a}^{(0)}_{s_{\tau}^{k}(i)} p^{fk}}$ for $1\leq i \leq n$, we have
\begin{equation} \label{eq:def:type}
\tau(s,\mu+\eta) \cong \bigoplus_{1 \leq i \leq n}\chi_i.
\end{equation}
This inertial type depends only on the inertial $L$-parameter $\tau(s,\mu+\eta)$ and not on the choice of isomorphism $\ovl{K} \ra \ovl{\Q}_p$.
\end{exam}

Base change $-\otimes_{\cO} E$ and $-\otimes_\cO \F$ induce bijections between tame inertial $L$-parameters over $\cO$, $E$, and $\F$ (the inverse to $-\otimes_\cO \F$ is the Teichm\"{u}ller lift). 
If $\tau$ is a tame inertial $L$-parameter over $\cO$ or $E$, we let $\taubar$ denote the corresponding tame inertial $L$-parameter over $\F$.
We say that $(s,\mu)\in \un{W}\times X^*(\un{T})$ is a \emph{lowest alcove presentation} of a tame inertial $L$-parameter $\tau$ (resp.~$\taubar$) over $E$ (resp.~$\F$) if $\mu \in \un{C}_0$ and $\tau \cong \tau(s,\mu+\eta)$ (resp.~$\taubar \cong \taubar(s,\mu+\eta)$).
When $F^+_p = K$ and $G$ is $\GL_n$ we say that $(s,\mu)$ is a lowest alcove presentation of a tame inertial type $\tau$ (resp.~$\taubar$) for $K$ over $E$ (resp.~$\F$) if $(s,\mu)$ is a lowest alcove presentation of the tame inertial $L$-parameter corresponding to it.

Let $\lambda\in X^*(\un{T})$ be a character.
We say that a lowest alcove presentation $(s,\mu)$ of a tame inertial $L$-parameter is \emph{$\lambda$-compatible with $\zeta \in X^*(\un{Z})$} if the image of $t_\lambda t_{\mu+\eta} s \un{W}_a/\un{W}_a \in \tld{\un{W}}/\un{W}_a \cong X^*(\un{Z})$ corresponds to $\zeta$.
(This notion of compatibility depends on the choice of $\eta$.)
When $\lambda=\un{0}$ we just say compatible instead of $\un{0}$-compatible.

We say that a lowest alcove presentation $(s,\mu)$ of a tame inertial $L$-parameter over $\F$ is compatible with $\zeta \in X^*(\un{Z})$ if the image of $t_\mu s \un{W}_a \in \tld{\un{W}}/\un{W}_a \cong X^*(\un{Z})$ is $\zeta$.
We say that lowest alcove presentations of a tame inertial $L$-parameter over $\F$ and a Serre weight are compatible if these lowest alcove presentations are both compatible with a single element of $X^*(\un{Z})$.
We say that lowest alcove presentations of a tame inertial $L$-parameter over $\F$ and a tame inertial $L$-parameter (over $E$) are $\lambda$-compatible if the lowest alcove presentations of the tame inertial $L$-parameters over $\F$ and $E$ are compatible and $\lambda$-compatible, respectively, with a single element of $X^*(\un{Z})$.
We will sometimes say compatible to mean $0$-compatible.

\begin{rmk}
Note that if $(s,\mu)$ is a lowest alcove presentation for a tame inertial $L$-parameter $\tau$ compatible with $\zeta \in X^*(\un{Z})$, then as a lowest alcove presentation of the tame inertial $L$-parameter $\taubar$ over $\F$ obtained by reduction, $(s,\mu)$ is compatible with $\zeta - \eta|_{\un{Z}} \in X^*(\un{Z})$.
This confusing choice is made because $\zeta$ gives the central character of $\sigma(\tau)$ (Proposition \ref{prop:tameILL}) while $\zeta - \eta|_{\un{Z}}$ gives the central character of elements of $W^?(\taubar)$, whose definition (Definition \ref{defn:W?}) involves $\eta$.
\end{rmk}

Let $\det$ be the natural quotient map ${\un{G}^\vee} \ra {\un{G}^\vee}/{\un{G}}^{\vee,\mathrm{der}} \cong {\un{Z}^\vee}$.
If $(s,\mu)$ is a lowest alcove presentation of a tame inertial $L$-parameter $\tau$ (resp.~$\taubar$) compatible with $\zeta$, then, thinking of $\zeta$ as an element of $X_*(\un{Z}^\vee)$, $\zeta\circ \omega_1 = \det\circ\tau$ (resp.~$(\zeta-\eta|_{\un{Z}})\circ \ovl{\omega}_1 = \det\circ\taubar$).

\begin{defn}
\label{defi:gen}
Let $\tau$ be a tame inertial $L$-parameter over $E$. 
If $F^+_p = K$, then we also denote by $\tau$ the corresponding inertial type for $K$.
The following adjectives also apply to inertial ($\F$-)types for $K$.
\begin{enumerate}
\item
\label{def:gen:reg}
We say that $\tau$ (resp.~$\taubar$) is \emph{regular} if $\tau$ (resp.~$\taubar$) is ${\un{G}^\vee}$-conjugate to a homomorphism $I_{\Qp} \ra \un{T}^\vee(E)$ (resp.~$I_{\Qp} \ra \un{T}^\vee(\F)$) such that the composition with $\alpha^\vee: \un{T}^\vee(E)\ra E^\times$ (resp.~$\alpha^\vee: \un{T}^\vee(\F)\ra \F^\times$) is nontrivial for any coroot $\alpha^\vee$.
\item 
\label{def:LApres}
\label{defi:gen:type}
We say that $\tau$ (resp.~$\taubar$) is $m$-generic for an integer $m\geq 0$ if there exists a lowest alcove presentation $(s,\mu)$ for $\tau$ (resp.~$\taubar$) where $\mu$ is $m$-deep in alcove $\un{C}_0$.
We call such a presentation an \emph{$m$-generic lowest alcove presentation}.
If $\tau$ (resp.~$\taubar$) has a fixed lowest alcove presentation $(s,\mu)$, then we let $\tld{w}(\tau)$ (resp.~$\tld{w}(\taubar)$) be $t_{\mu+\eta} s$.
Again, note that $\mu$ is $m$-deep if and only if $\tld{w}(\tau)$ (resp.~$\tld{w}(\taubar)$) is $m$-generic in the sense of Definition \ref{defn:var:gen}(\ref{it:gen:weyl}). 
Note that $1$-generic implies regular (see \cite[Remark 2.2.4]{LLLM2}), and that a lowest alcove presentation for $\tau$ (resp.~$\taubar$) exists exactly when $\tau$ (resp.~$\taubar$) is $0$-generic.
\end{enumerate}
\end{defn}

\cite[Proposition 9.2.1]{GHS} defines an injective map $V_\phi$ from the set of tame inertial $L$-parameters over $\F$ to isomorphism classes of $\rG$-representations over $E$ (taking $E$ sufficiently large) which takes $\taubar(w,\mu)$ to $R_w(\mu)$.
Note that $\tau(s,\mu + \eta)$ is $m$-generic if and only if $R_s(\mu + \eta)$ is $m$-generic.
As $V_\phi$ respects the notion of lowest alcove presentation, the argument of Lemma \ref{lemma:LAPbij} gives the following lemma.

\begin{lemma}
\label{lemma:LAPbij:Ftypes}
If $\tau$ is a $1$-generic tame inertial $L$-parameter, then $(s,\mu) \mapsto t_{\mu+\eta} s \un{W}_a/\un{W}_a\in \tld{\un{W}}/\un{W}_a \cong X^*(\un{Z})$ gives a bijection between lowest alcove presentations of $\tau$ and algebraic central characters $\zeta|_{\un{Z}} \in X^*(\un{Z})$ such that, thinking of $\zeta$ as an element of $X_*(\un{Z}^\vee)$, $\zeta\circ \omega_1 = \det\circ\tau$.

If $\taubar$ is a $1$-generic tame inertial $L$-parameter over $\F$, then $(s,\mu) \mapsto t_\mu s \un{W}_a/\un{W}_a\in \tld{\un{W}}/\un{W}_a \cong X^*(\un{Z})$ gives a bijection between lowest alcove presentations of $\taubar$ and algebraic central characters $\zeta \in X^*(\un{Z})$ such that, thinking of $\zeta$ as an element of $X_*(\un{Z}^\vee)$, $(\zeta-\eta|_{\un{Z}})\circ \ovl{\omega}_1 = \det\circ\taubar$.
\end{lemma}

\begin{prop}\label{prop:PSLAP}
Let $\tld{w} = t_\nu w \in \tld{\un{W}}$, $\omega\in X^*(\un{T})$, and let $\kappa = \pi^{-1}(\tld{w})\cdot (\omega-\eta)$. 
Then the tame inertial $L$-parameter $\tau(1,\kappa)$ is isomorphic to $\tau(\pi^{-1}(w)^{-1}w,\omega+\pi^{-1}(w)^{-1}(\nu-\eta))$.
\end{prop}
\begin{proof}
This follows from the paragraph containing \cite[(10.1.11)]{GHS}.
Indeed, in the notation of \emph{loc.~cit.}, 
\[
^{(\pi^{-1}(\nu),\pi^{-1}(w))}(\pi^{-1}(w)^{-1}w,\omega+\pi^{-1}(w)^{-1}(\nu-\eta)) = (1, \pi^{-1}(w)(\omega)+p\pi^{-1}(\nu)-\eta) = (1,\pi^{-1}(\tld{w})\cdot (\omega-\eta)).
\]
\end{proof}

\subsection{Inertial local Langlands for $\GL_n$}
\label{subsec:ILL}

We recall some results towards inertial local Langlands correspondence for $\GL_n$, before making this explicit in the tame case using the previous two subsections. 
In this section, $K$ is an $\ell$-adic field ($\ell$ a rational prime not necessarily equal to $p$).

\begin{defn}\label{defn:WDtype}
A \emph{Weil--Deligne inertial $L$-homomorphism} $\tau$ is a pair $(\rho_\tau,N_\tau)$ where $\rho_\tau: I_{\Qp} \ra \un{G}^\vee(E)$ is a homomorphism with open kernel, $N_\tau$ is a nilpotent element of $\Lie \un{G}^\vee(E)$, and there exists an $\rho: W_{\Qp} \ra$ $^L \un{G}(E)$ such that the projection to $\Gal(E/\Qp)$ is the natural map, $\rho|_{I_{\Qp}} = \rho_\tau$, and $\rho(g)N\rho(g)^{-1} = \|g\|N$, where $\|\cdot\|: W_{\Qp} \onto W_{\Qp}/I_{\Qp} \risom p^{\Z}$ sends an arithmetic Frobenius element to $p$.
A Weil--Deligne inertial $L$-parameter is a $\un{G}^\vee(E)$-conjugacy class of Weil--Deligne inertial $L$-homomorphisms.

We can similarly define a \emph{Weil--Deligne inertial type} $\tau$ (for $K$) to be a conjugacy class of pairs $(\rho_\tau,N_\tau)$ where $\rho_\tau: I_K \ra G^\vee(E)$ is a homomorphism with open kernel, $N_\tau$ is a nilpotent element of $\Lie G^\vee(E)$, and there exists $\rho: W_K \ra G^\vee(E)$ such that $(\rho,N_\tau)$ is a Weil--Deligne representation, i.e.~$\rho(g)N_\tau\rho(g)^{-1} = \|g\|N_\tau$, where $\|\cdot\|: W_{\Qp} \onto W_{\Qp}/I_{\Qp} \risom p^{f\Z}$ sends an arithmetic Frobenius element to $p^f$. 

We say that a Weil--Deligne inertial $L$-parameter or type is \emph{tame} if $\rho_\tau$ above factors through the tame inertial quotient. 
Finally, there is a natural bijection between Weil--Deligne inertial $L$-parameters $\tau$ and collections of Weil--Deligne inertial types $(\tau_v)_{v\in S_p}$ preserving tameness.
\end{defn}

If $(\rho,N)$ is a Weil--Deligne representation for $K$, we denote by $(\rho,N)|_{I_K}$ the Weil--Deligne inertial type $(\rho|_{I_K},N)$.

\begin{rmk}
\label{rmk:not:WDILP}
We abuse notation by denoting both inertial $L$-parameters and Weil--Deligne inertial $L$-parameters by $\tau$ (and similarly for inertial types).
However, there is a natural inclusion from the set of (tame) inertial $L$-parameters (resp.~inertial types) to the set of (tame) Weil--Deligne $L$-parameters (resp.~Weil--Deligne inertial types) sending an inertial $L$-parameter $\tau$ (resp.~inertial type) to the Weil--Deligne inertial $L$-parameter (resp.~Weil--Deligne inertial type) with $\rho_\tau = \tau$ and $N_\tau=0$.
Through this inclusion, we will think of the set of (tame) inertial $L$-parameters (resp.~inertial types) as a subset of the set of (tame) Weil--Deligne inertial $L$-parameters (resp.~inertial types).

There is also a surjective map in the other direction from the set of (tame) Weil--Deligne $L$-parameters (resp.~Weil--Deligne inertial types) to the set of (tame) inertial $L$-parameters (resp.~inertial types), for which the above inclusion is a section, given by forgetting the nilpotent element.
\end{rmk}

We now specialize our discussion to the case $G = \GL_n$.
Recall that the Jordan normal form of a nilpotent element $N$ of $\mathfrak{gl}_n = M_n(E)$ gives a partition $P_N$ of $n$ by recording the sizes of Jordan blocks, which is a complete conjugation invariant of nilpotent elements of $M_n(E)$.
Viewing a partition as a decreasing function $P: \Z_{>0} \ra \Z_{\geq 0}$ with finite support ($P$ is a partition of $\sum_{i\in \Z_{>0}} P(i)$), we write $P_1 \preceq P_2$ if $\sum_{i=1}^k P_1(i) \leq \sum_{i=1}^k P_2(i)$ for all $k \in \Z_{>0}$. Then $\preceq$ defines a partial ordering on the set of partitions.
We write $N_1 \preceq N_2$ for two nilpotent elements of $M_n(E)$ if $P_{N_1} \preceq P_{N_2}$. Then $\preceq$ defines a partial ordering on the set of conjugacy classes of nilpotent elements of $M_n(E)$.
{Note that with this partial ordering, $0$ is the minimal element.}

For an irreducible inertial type $\tau_0$, let $N_\tau(\tau_0)$ be the restriction of $N_\tau$ to the $\tau_0$-isotypic part of $V_\tau$ (which it preserves).

\begin{defn}\label{defn:preceq}
We write $\tau \preceq \tau'$ for two Weil--Deligne inertial types if $\rho_\tau$ and $\rho_\tau'$ are isomorphic and $N(\tau_0) \preceq N'(\tau_0)$ for all irreducible inertial types $\tau_0$.
{(In particular, the trivial representation is $\preceq$ the Steinberg representation.)}
This defines a partial ordering on the set of Weil--Deligne inertial types.
Thinking of a Weil--Deligne inertial $L$-parameter as a collection of Weil--Deligne inertial types, we say that $\tau \preceq \tau'$ for two Weil--Deligne inertial $L$-parameters if $\tau_v \preceq \tau_v'$ for each $v\in S_p$.
\end{defn}

If $\pi$ is an irreducible admissible representation of $\GL_n(K)$ over $E$, then we let $\mathrm{rec}_K(\pi)$ be the Weil--Deligne representation over $E$ in \cite[Theorem A]{harris-taylor}.

\begin{thm}\label{thm:ILL}
Let $G = \GL_n$.
Let $\tau$ be a Weil--Deligne inertial type for $K$.
Then there is a smooth irreducible $\GL_n(\cO_K)$-representation $\sigma(\tau)$ over $E$ such that for an irreducible admissible representation $\pi$ of $\GL_n(K)$, 
\begin{enumerate}
\item if $\pi|_{\GL_n(\cO_K)}$ contains $\sigma(\tau)$ then $\mathrm{rec}_K(\pi)|_{I_K} \preceq \tau$;
\item if $\mathrm{rec}_K(\pi)|_{I_K} = \tau$, then $\pi|_{\GL_n(\cO_K)}$ contains $\sigma(\tau)$ with multiplicity one; and
\item if $\mathrm{rec}_K(\pi)|_{I_K} \preceq \tau$ and $\pi$ is generic, then $\pi|_{\GL_n(\cO_K)}$ contains $\sigma(\tau)$ and the multiplicity is one if furthermore $\tau$ is maximal with respect to $\preceq$. 
\end{enumerate}
\end{thm}
\begin{proof}
This combines \cite[Theorem 3.7]{Shotton} and \cite[Theorem 1.2]{Pyv}.
\end{proof}

Note that we make no claim of uniqueness for $\sigma(\tau)$.
In what follows, $\sigma(\tau)$ will denote either a particular choice that we have made or any choice that satisfies the properties in Theorem \ref{thm:ILL}.

If $\tau$ is a Weil--Deligne inertial $L$-parameter corresponding to the collection of Weil--Deligne inertial types $(\tau_v)_{v\in S_p}$, we let $\sigma(\tau)$ be the $G_0(\Z_p)$-representation $\otimes_{v\in S_p} \sigma(\tau_v)$.

We now make particular choices of $\sigma(\tau)$ when $\tau$ above is tame.

\begin{prop}\label{prop:tameILL}
Suppose that $G = \GL_n$ and $(s,\mu) \in \un{W}\times X^*(\un{T})$.
We can choose $\sigma(\tau)$ in Theorem \ref{thm:ILL} for tame Weil--Deligne inertial $L$-parameters $\tau$ such that $\{\sigma(\tau) \mid \tau = (\tau(s,\mu),N_\tau) \}$ is the set of all irreducible constituents of $R_s(\mu)$ $($where we view $R_s(\mu)$ as a $\GL_n(\cO_p)$-representation by inflation$)$.
\end{prop}
\begin{proof}
We immediately reduce to the case where $\cO_p$ is a domain, say $\cO_K$.
Then this follows from the construction of $\sigma(\tau)$ in \cite[\S 6]{SZ} as we now explain.
We first specify the Bushnell--Kutzko type $(J,\lambda)$ for the Bernstein component corresponding to $\tau(s,\mu)$.
Let $\sigma_0$ be an embedding $\F_{p^f} \into \F$, and let $r$ be the order of $s_\tau$ as in Example \ref{ex:data:type} (though $r$ does not depend on the choice of $\sigma_0$).
Fix an embedding $\sigma_0': \F_{p^{fr}} \into \F$ extending $\sigma_0$, and let $\tau$ also denote the corresponding tame inertial type for $K$ (see Example \ref{ex:data:type}, though again $\tau$ depends only on $\sigma_0$, but not $\sigma_0'$).

We first suppose that $\tau(s,\mu)$ is cuspidal, so that in particular, the order of the automorphism $(\pi^{-1}s^{-1})$ of $X^*(\un{T})$ is $fn$, and we take $r$ above to be $n$.
Then recall that we can choose $g_s\in N(\un{T})(\ovl{\F}_p)$ such that $g_s^{-1}F(g_s) = s$, and we let $\un{T}_s \defeq$ $^{g_s}\un{T} = g_s \un{T} g_s^{-1}$.
By \cite[Proposition 13.7(ii)]{DM}, the map
\begin{equation}\label{eqn:anisotropic}
\Big(\sum_{i=0}^{fn-1} (s\circ F)^i(\nu)\Big) \circ \sigma_0': \F_{p^{fn}}^\times \ra \, ^{g_s^{-1}}(\un{T}_s^F)
\end{equation}
is surjective for $\nu \in X_*(\un{T})$ with $\nu_i = 0$ if $i \neq 0$ and $\nu_0 = (1,0,\ldots,0)$ (since the $(s\circ F)$-orbit of $\nu$ generates $X_*(\un{T})$).
As the domain and codomain of (\ref{eqn:anisotropic}) have the same cardinality, this map is an isomorphism.

Then $R_s(\mu) = (-1)^{n-1} R_{\un{T}_s^F}^\theta$ where $\theta$ is the character $\F_{p^{fn}}^\times \cong \un{T}_s^F \overset{^{g_s^{-1}} (\cdot)}{\ra}$ $^{g_s^{-1}}(\un{T}_s^F) \subset \un{T}(\F) \overset{\mu}{\ra} \F^\times$.
Thus, $\theta = (\sum_{i=0}^{fn-1} (F^* \circ s^{-1})^i(\mu)) \circ \nu \circ \sigma_0'$.
Since $\tau(s,\mu) \cong \oplus_{k=0}^{n-1} \theta^{p^{fk}} \circ \mathrm{Art}_{K'}$ by Example \ref{ex:data:type} where $K' \subset K^{\mathrm{un}}$ is the subfield of degree $n$ over $K$, the result in this case follows from \cite[Proposition 2.4.1(i)]{EGH}.
(Note that in this case, $\tau' \preceq \tau$, $\tau \preceq \tau'$ and $\tau' = \tau$ are all equivalent for $\tau'$ a Weil--Deligne inertial type. 
The multiplicity one statement comes from the fact that, in the notation of \emph{loc.~cit.}, c-$\Ind_{F^\times \GL_n(\cO_F)}^{\GL_n(F)}\tau$ is irreducible.)
In this case, $(\GL_n(\cO_K),R_s(\mu))$ is a Bushnell--Kutzko type for the Bernstein component corresponding to the inertial type $\tau(s,\mu)$.

The general case follows from the fact that if $M\subset \GL_n$ is a Levi subgroup and $(J_M,\lambda)$ with $J_M \defeq M(K) \cap \GL_n(\cO_K)$ is a Bushnell--Kutzko type for a Bernstein component for $M$ corresponding to the inertial equivalence class $[L,\sigma]$ of some supercuspidal pair $(L,\sigma)$, then $(J,\lambda)$ is a Bushnell--Kutzko type for the Bernstein component for $G$ corresponding to $[L,\sigma]$, where $J$ is a minimal parahoric subgroup of $\GL_n(\cO_K)$ containing $J_M$, and $J$ acts on $\lambda$ through the natural quotient map $J \onto J_M$.
Indeed, $(J,\lambda)$ is a $G$-cover of $(J_M,\lambda)$ in the sense of \cite[Definition 8.1]{BK98}, and so $(J,\lambda)$ is the desired Bushnell--Kutzko type (see \cite{BK99}).
Then if $\lambda = R_s(\mu)$ (as an $M(\F_{p^f})$-representation), then $\Ind_J^{\GL_n(\cO_K)} \lambda$ is $R_s(\mu)$ (as a $\rG$-representation) by \cite[11.5]{DM}.
By construction, $\{\sigma(\tau) \mid \tau = (\tau(s,\mu),N_\tau) \}$ is the set of irreducible constituents of $\Ind_J^{\GL_n(\cO_K)} \lambda \cong R_s(\mu)$.
\end{proof}

\subsection{Herzig's conjecture on modular Serre weights}\label{sec:herzig}

Recall that $\tld{w}_h \defeq w_0 t_{-\eta} \in \tld{\un{W}}$.
For a regular Serre weight $\sigma = F(\lambda)$, let $\cR(\sigma)$ be the Serre weight $F(\tld{w}_h \cdot \lambda)$, which does not depend on the choice of $\lambda$.
The map $\cR$ defines a bijection from the set of regular Serre weights to itself (since $\cR^2$ is a twist by a character).
Note however that $\cR$ (like $\tld{w}_h$) depends on the choice of $\eta$.

\begin{defn}\label{defn:W?}
For a tame inertial $L$-parameter $\taubar$ over $\F$, we define $W^?(\taubar)$ to be the set $\cR\big(\JH\big(\ovl{\sigma([\taubar])}\big)\big)$.
\end{defn}

\begin{prop}\label{prop:W?}
Let $m \geq 2h_\eta$ be an integer.
Let $\taubar$ be a tame inertial $L$-parameter over $\F$, together with an $m$-generic lowest alcove presentation with corresponding element $\tld{w}(\taubar)\in\tld{\un{W}}$.
The map 
\begin{equation}\label{eqn:W?}
(\tld{w},\tld{w}_2) \mapsto F_{(\tld{w},\tld{w}(\taubar)\tld{w}_2^{-1}(0))}
\end{equation}
defines a bijection between
\begin{itemize}
\item pairs $(\tld{w},\tld{w}_2)$ with $\tld{w} \in \tld{\un{W}}^+_1$ and $\tld{w}_2\in \tld{\un{W}}^+$, up to the diagonal $X^0(\un{T})$-action, such that $\tld{w}_2 \uparrow \tld{w}$; and
\item elements of $W^?(\taubar)$.
\end{itemize}
Moreover, these Jordan--H\"older factors are $(m-h_\eta)$-deep and the lowest alcove presentations $(\tld{w},\tld{w}(\taubar)\tld{w}_2^{-1}(0))$ of these Serre weights are compatible with the fixed lowest alcove presentation %
of $\taubar$  (see \S \ref{sec:InertialTypes}).
\end{prop}
\begin{proof}
That the map is a bijection follows from the definition of $W^?(\taubar)$ and Proposition \ref{prop:JHbij}.
If $\tld{w}_2 \uparrow \tld{w}$ and $\tld{w} \in \tld{\un{W}}_1^+$, $\langle \tld{w}_2^{-1}(0),\alpha^\vee\rangle \leq h_\eta$ for all $\alpha\in \un{\Phi}$, which implies that $F(\pi^{-1}(\tld{w}) \cdot (\tld{w}(\taubar)\tld{w}_2^{-1}(0) - \eta)$ is $(m-h_\eta)$-deep. 
The lowest alcove presentation $(\tld{w},\tld{w}(\taubar)\tld{w}_2^{-1}(0))$ is compatible with the image of $\tld{w} t_{\tld{w}(\taubar)\tld{w}_2^{-1}(0)} \un{W}_a/\un{W}_a = \tld{w} \tld{w}(\taubar) \tld{w}_2^{-1} \un{W}_a/\un{W}_a = \tld{w}(\taubar)\un{W}_a/\un{W}_a$ which is compatible with the lowest alcove presentation of $\taubar$ (for the latter equality note that $\tld{w}\equiv \tld{w}_2$ modulo $\un{W}_a$).
\end{proof}

\begin{defn}\label{defn:obv_weight}
We let $W_{\obv}(\taubar)$ be the subset of $W^?(\taubar)$ corresponding via (\ref{eqn:W?}) to pairs of the form $(\tld{w},\tld{w})$.
Note that a Serre weight in $W_{\obv}(\taubar)$ is determined by the image $w$ of $\tld{w}$ in $W$.
We say that this is the obvious weight of $\taubar$ corresponding to $w$.
\end{defn}

\subsubsection{Breuil--M\'ezard intersections}
Let $\rhobar$ and $\tau$ be tame inertial $L$-parameters over $\F$ and $E$, respectively.
Suppose that we can fix $\lambda$-compatible lowest alcove presentations of $\rhobar$ and $\tau$ (with corresponding elements $\tld{w}(\rhobar)$ and $\tld{w}(\tau)$), for some dominant $\lambda\in X^*(\un{T})$.
Then let $\tld{w}(\rhobar,\tau)$ be $\tld{w}(\tau)^{-1}\tld{w}(\rhobar)$.

\begin{prop}\label{prop:intersect}
Let $\lambda\in X^*(\un{T})$ be a dominant weight.
Let $\rhobar$ and $\tau$ be tame inertial $L$-parameters over $\F$ and $E$, respectively.
Suppose that we can fix $\lambda$-compatible $2h_\eta$-generic and $\max\{2h_\eta,h_{\lambda+\eta}\}$-generic lowest alcove presentations of $\rhobar$ and $\tau$, respectively, and let $\tld{w}(\rhobar)$ and $\tld{w}(\tau)$ be the corresponding elements of $\tld{\un{W}}$.
Then $(\tld{w},\omega)$ is a compatible lowest alcove presentation for a Serre weight $\sigma \in W^?(\rhobar) \cap \JH(\ovl{\sigma}(\tau))$ if and only if there are $\tld{w}_1$, $\tld{w}_2 \in \tld{\un{W}}^+$ such that $\tld{w}_1\uparrow \tld{w} \uparrow t_\lambda\tld{w}_h^{-1} \tld{w}_2$ and $\omega = \tld{w}(\rhobar)\tld{w}_1^{-1}(0)=\tld{w}(\tau)\tld{w}_2^{-1}(0)$. 

The equality $\tld{w}(\rhobar)\tld{w}_1^{-1}(0)=\tld{w}(\tau)\tld{w}_2^{-1}(0)$ holds if and only if $\tld{w}(\rhobar,\tau) = \tld{w}_2^{-1} w \tld{w}_1$ for some $w\in \un{W}$. 
\end{prop}
\begin{proof}
The first claim follows from Propositions \ref{prop:JHbij} and \ref{prop:W?}. For the second claim, the equality $\tld{w}(\rhobar)\tld{w}_1^{-1}(0) = \tld{w}(\tau)\tld{w}_2^{-1}(0)$ implies that $\tld{w}(\rhobar,\tau) \tld{w}_1^{-1} \in \tld{w}_2^{-1} W$.
\end{proof}

\begin{cor}\label{cor:extremeintersect}
Let $\lambda\in X^*(\un{T})$ be a dominant weight.
Let $\rhobar$ and $\tau$ be tame inertial $L$-parameters over $\F$ and $E$, respectively.
Suppose that we can fix $\lambda$-compatible $2h_\eta$-generic and $\max\{2h_\eta,h_{\lambda+\eta}\}$-generic lowest alcove presentations of $\rhobar$ and $\tau$, respectively, and that $\tld{w}(\rhobar,\tau) = t_{s^{-1}(\lambda+\eta)}$ for some $s \in \un{W}$.
Then the intersection $W^?(\rhobar) \cap \JH(\ovl{\sigma}(\tau)\otimes W(\lambda))$ contains exactly one weight which is the obvious weight in $W_{\obv}(\rhobar)$ corresponding to $s$.
\end{cor}
\begin{proof}
Suppose that $(\tld{w},\omega)$ is a lowest alcove presentation of $\sigma \in W^?(\rhobar) \cap \JH(\ovl{\sigma}(\tau)\otimes W(\lambda))$ which is compatible with that of $\rhobar$ (equivalently it is $\lambda$-compatible with that of $\tau$).
Proposition \ref{prop:intersect} implies that $t_{s^{-1}(\lambda+\eta)} = \tld{w}(\rhobar,\tau) = \tld{w}_2^{-1} s' \tld{w}_1$ for some $s'\in \un{W}$, and some $\tld{w}_1$, $\tld{w}_2 \in \tld{\un{W}}^+$ with $\tld{w}_1 \uparrow \tld{w}$ and $\tld{w} \uparrow t_\lambda \tld{w}_h^{-1} \tld{w}_2$.
These inequalities imply that 
\[
t_{s^{-1}(\lambda+\eta)} = \tld{w}_2^{-1} s' \tld{w}_1 \leq (t_{-w_0(\lambda)} \tld{w}_h \tld{w})^{-1} w_0 \tld{w} = t_{w^{-1}(\lambda+\eta)}, 
\]
where $w\in \un{W}$ is the image of $\tld{w}$.
This implies that $s=w$ and that $\tld{w} = \tld{w}_2$.
Then $\sigma$ is the obvious weight corresponding to $s$.
\end{proof}

\begin{prop}\label{prop:obvint}
Let $\lambda\in X^*(\un{T})$ be a dominant weight.
Let $\rhobar$ be a $2h_\eta$-generic tame inertial $L$-parameter over $\F$ and let $\tau$ be a $\max\{2h_\eta,h_{\lambda+\eta}\}$-generic tame inertial $L$-parameter.
Assume we can fix $\lambda$-compatible lowest alcove presentations for $\rhobar$ and $\tau$ such that $\tld{w}(\rhobar,\tau) \in \Adm(\lambda+\eta)$.
Then $W_{\obv}(\rhobar) \cap \JH(\ovl{\sigma(\tau)}\otimes W(\lambda))$ is nonempty.
\end{prop}  
\begin{proof}
Since $\tld{w}(\rhobar,\tau) \in \Adm(\lambda+\eta)$, there exists a $w\in \un{W}$ such that $\tld{w}(\rhobar,\tau) \leq t_{w^{-1}(\lambda+\eta)} = (t_{-w_0(\lambda)}\tld{w}_h\tld{w})^{-1} w_0 \tld{w}$ where $\tld{w} \in \tld{\un{W}}^+_1$ has image $w\in \un{W}$.
Since this is a reduced factorization by Lemma \ref{lemma:gallery}, $\tld{w}(\rhobar,\tau) = \tld{w}_2^{-1}w'\tld{w}_1$ for some $\tld{w}_1 \leq \tld{w}$, $\tld{w}_2\leq t_{-w_0(\lambda)}\tld{w}_h\tld{w}$ and $w'\in \un{W}$.
By changing $w'$ and using \cite[Lemma 4.3.4]{LLL}, we can assume without loss of generality that $\tld{w}_1$ and $\tld{w}_2$ are elements of $\tld{\un{W}}^+$.
By Wang's theorem (\cite[Theorem 4.1.1]{LLL}), $\tld{w}_1 \uparrow \tld{w}$ and $\tld{w}_2\uparrow t_{-w_0(\lambda)}\tld{w}_h\tld{w}$, or equivalently by \cite[Proposition 4.1.2]{LLL}, $\tld{w} \uparrow t_\lambda \tld{w}_h^{-1}\tld{w}_2$.

Let $\omega\in X^*(\un{T})$ be the unique (dominant) weight up to $X^0(\un{T})$ such that $t_{-\omega} \tld{w}_1\in \tld{\un{W}}^+_1$.
Let $\tld{w}_3$ be the unique element in $\un{W}t_{-w'(\omega)}\tld{w}_2\cap\tld{\un{W}}^+$.
Then $t_{w_0(\omega)}\tld{w}_3 \uparrow \tld{w}_2$ just as in the proof of \cite[Proposition 4.4.1]{LLL} so that $t_{-\omega}\tld{w}_1\uparrow t_\lambda \tld{w}_h^{-1} t_{-w_0(\omega)} \tld{w}_2 \uparrow t_\lambda \tld{w}_h^{-1} \tld{w}_3$.
Replacing $\tld{w}_1$ by $t_{-\omega} \tld{w}_1$ and $\tld{w}_2$ by $\tld{w}_3$ and changing $w'$, we have that $\tld{w}(\rhobar,\tau) = \tld{w}_2^{-1}w'\tld{w}_1$ with $\tld{w}_1\uparrow t_\lambda\tld{w}_h^{-1}\tld{w}_2$ and $\tld{w}_1\in \tld{\un{W}}_1^+$.

We claim that $F_{(\tld{w}_1,\tld{w}(\rhobar)\tld{w}_1^{-1}(0))} \in W_\obv(\rhobar)$ is in $\JH(\ovl{\sigma(\tau)}\otimes W(\lambda))$.
Indeed, 
\[
\tld{w}(\rhobar)\tld{w}_1^{-1}(0) = \tld{w}(\tau)\tld{w}(\rhobar,\tau)\tld{w}_1^{-1}(0) = \tld{w}(\tau)\tld{w}_2^{-1}(0).
\]
The claim now follows from Proposition \ref{prop:JHbij}.
\end{proof}

\begin{lemma}\label{lemma:connectSW}
Let $\taubar$ be a tame inertial $L$-parameter over $\F$.
Suppose there exists a $3h_\eta$-generic lowest alcove presentation for it and let $\tld{w}(\taubar)$ be the corresponding element of $\tld{\un{W}}$.
Let $R$ be the Deligne--Lusztig representation with the $\eta$-compatible lowest alcove presentation such that $\tld{w}(R) = \tld{w}(\taubar) t_{-\eta-w_0(\eta)}$.
Then $W^?(\taubar) \subset\JH(\overline{R} \otimes W(\eta))$. 
\end{lemma}
\begin{proof}
Suppose that $\sigma \in W^?(\taubar)$ so that $\sigma$ has lowest alcove presentation $(\tld{w},\omega)$ with $\omega = \tld{w}(\taubar)\tld{w}_2^{-1}(0)$ for some $\tld{w}_2 \uparrow \tld{w}$ by Proposition \ref{prop:W?}.
Then $\omega = \tld{w}(R)(t_{-\eta-w_0(\eta)}\tld{w}_2)^{-1}(0)$ (note that $-\eta-w_0(\eta)\in X^0(\un{T})$).
By Proposition \ref{prop:JHbij}, to show that $\sigma \in \JH(\overline{R} \otimes W(\eta))$, it suffices to show that $\tld{w} \uparrow t_\eta\tld{w}_h^{-1} t_{-\eta-w_0(\eta)}\tld{w}_2 = t_{-w_0(\eta)} \tld{w}_h^{-1}\tld{w}_2$.
Since $\tld{w}_h^{-1} \tld{w} \uparrow \tld{w}_h^{-1} \tld{w}_2$, it suffices to show that $\tld{w} \uparrow t_{-w_0(\eta)}\tld{w}_h^{-1} \tld{w}$, or equivalently that $w_0 \tld{w}_h^{-1}\tld{w} = t_{w_0(\eta)} \tld{w} \uparrow \tld{w}_h^{-1} \tld{w} \in \tld{\un{W}}^+$.
This follows from \cite[II 6.5(5)]{RAGS}.
\end{proof}
\clearpage{}%
\clearpage{}%
\section{The universal local model}\label{sec:UMLM}
In this section, we construct and study the universal version of our local models.   This will allow us to show that various properties hold generically for the mixed characteristic local models studied in Section \ref{sec:MLM}.  Unless otherwise specified, all algebraic groups will be over $\Z$.  
Let $X =\bA^1_\Z= \Spec \Z[v]$. For any commutative ring $R$, we identify the $R$-points $X(R)$ with $R$ in the usual way: an algebra map $\Z[v]\to R$ corresponds to the image $t\in R$ of the coordinate variable $v$. 
{(We will eventually consider, in sections \S \ref{sec:MLM}, \S\ref{sec:BK:LM} and \S\ref{sec:monodromy}, Noetherian $p$-adically complete $\cO$-algebras $R$, and take $t$ to be $-p$ in this case.)}
We also let $X^0=\bA^1_\Z\setminus \{0\}=\Spec \Z[v,v^{-1}]$.

\subsection{Loop groups}\label{sec:UMLM:loopgroups}

Let $\cG$ be the Bruhat-Tits group for $\GL_n$ over $\bA^1_{\Z}$ as in \cite[4.b.1]{PZ}, which is a dilatation of the Chevalley group ${\GL_n}_{/\bA^1_{\Z}}$ along a subgroup concentrated in the fiber $t=0$. Concretely, for any $\Z[v]$-algebra $R$ such that $v$ gets sent to $t\in R$, the functor of points of $\cG(R)$ is given by %
\begin{align*}
R&\mapsto\left\{\,(A_0,\dots A_{n-1})\in \big(\GL_n(R)\big)^n\ \bigg| \ 
            \begin{varwidth}{\textwidth}
               \centering
                $\Diag(1,\dots, t, \dots 1)\,A_{i-1}=A_i\,\Diag(1,\dots, t, \dots 1)$ for all $i$,\\
                \vspace{.15cm}
                where $t$ is in the $i$-th entry of the diagonal matrix.
            \end{varwidth}
            \,\right\}
\end{align*}
In the special case that $t$ is regular in $R$, the above data reduces to just the data of a pair $(t,A_0)$ such that $A_0$ mod $t$ is upper triangular.
It is known that $\cG$ is a smooth affine group scheme with connected fibers {(see \cite[Corollary 3.2]{PZ} and \cite[\S 1.2, Theorem]{MRR})}. %

We also get the positive loop group $L^+\cG$ and the loop group $L\cG$ whose functors of points on a $\Z[v]$-algebra $R$ (sending $v$ to $t\in R$) are given by 
\begin{align*}
R&\mapsto \cG(R[\![v-t]\!])
\end{align*}
and 
\begin{align*}
R&\mapsto \cG(R(\!(v-t)\!)),
\end{align*}
respectively {(where $R[\![v-t]\!]$ denotes the $(v-t)$-adic completion of $R[v]$, and $R(\!(v-t)\!)\defeq R[\![v-t]\!][\frac{1}{v-t}])$.} 
Here the values of the functor $\cG$ are computed using the maps $\Z[v]\to R[\![v-t]\!]$ and $\Z[v]\to R(\!(v-t)\!)$ sending $v$ to $v$. 
It is known that $L^+\cG$ is represented by a(n infinite type) scheme and $L\cG$ is an ind-group scheme {(\cite[\S 5.b.1]{PZ})}.
We have a canonical map $T\to L^+\cG$, sending $h\in T(R)$ to the ``constant'' diagonal matrices $(h,\cdots,h)\in \GL_n(R[\![v-t]\!])^n$.
We have a well-defined determinant map of $X$-ind-schemes {$\det:L\cG\to L(\bG_m)_{\slash X}$}.

\begin{rmk}When $R$ is Noetherian, $v$ is regular in  $R[\![v-t]\!]$ and $R(\!(v-t)\!)$, thus we get the simpler description
\begin{align*}
L^+\cG(R)&=\{A\in \GL_n(R[\![v-t]\!]),\,A \text{ is upper triangular modulo $v$}\}\\
L\cG(R)&=\{A\in \GL_n(R(\!(v-t)\!)),\,A \text{ is upper triangular modulo $v$}\}
\end{align*}
In particular, $L\cG(R)$ is a subgroup of $\GL_n(R(\!(v-t)\!))$ for Noetherian $R$.
In what follows we will restrict all our functors to locally Noetherian schemes, and hence we will do our manipulations using these simpler descriptions. We leave it as an exercise to the reader to formulate the right definitions for possibly non-Noetherian input rings.
\end{rmk}

For an integer $d$, let $L\cG^{det=d}$ be the subfunctor of $L\cG$ given by
\begin{align*}
L\cG^{det=d}(R)=\big\{g\in L\cG(R)| \det(g)\in (v-t)^d\,\big(R[\![v-t]\!]\big)^\times\subset \big(R(\!(v-t)\!)\big)^\times\big\},
\end{align*}
which is stable under the left translation action by $L^+\cG$.

We also define $L^+\cM$ to be the functor given by
\begin{align*}
L^+\cM(R)=\{  g \in M_n(R[\![v-t]\!]), \text { $g$ is upper triangular modulo $v$}\},
\end{align*}
so the subfunctor $L^+\cM \cap L\cG$ is stable under the left and right translation action by $L^+\cG$.

By \cite[\S 5.b]{PZ}, the fpqc quotient sheaf {(over the site of affine $\Z[v]$-schemes)} $L^+\cG\backslash L\cG$ is representable by an ind-projective ind-scheme $\Gr_{\cG,X}$, which also has a moduli interpretation in terms of $\cG$-torsors. For any ring $R$, we have an injection $L^+\cG(R)\backslash L\cG(R) \into \Gr_{\cG,X}(R)$. %

By construction, $\Gr_{\cG,X}\times_X X^0$ is the affine Grassmannian for the split group $\GL_n$ over $X^0$, while $\Gr_{\cG,X}\times_X \{0\}$ is the affine flag variety for the standard Iwahori group scheme $\cI$ over $\Z[\![v]\!]$.

For each integer $d$, we let $\Gr_{\cG,X}^{det=d}$ be the fpqc quotient subsheaf $L^+\cG\backslash L\cG^{det=d} \subset L^+\cG\backslash L\cG$.

For each $h\geq 0$, we let $L\cG^{det=d,\leq h}$ be the subfunctor of $L\cG^{det=d}$ given by 
\begin{align*}
L\cG^{det=d,\leq h}\big(R\big) =\left\{A\in L\cG^{det=d}(R)\ \Big| \  A \in \frac{1}{(v-t)^h}L^+\cM(R) \right\}
\end{align*}
Then $L\cG^{det=d,\leq h}$ is $L^+\cG$-stable, and the fpqc quotient subsheaf $\Gr_{\cG,X}^{det=d,\leq h}=L^+\cG\backslash L\cG^{det=d,\leq h}$ of $\Gr_{\cG,X}$ is representable by a projective scheme over $X=\bA^1_\Z$ {(see the argument of \cite[Lemma 1.1.5]{zhu-intro-AG})}.
We clearly have $\Gr_{\cG,X}^{det=d}= \underset{h}{\varinjlim} \Gr_{\cG,X}^{det=d,\leq h}$.

\subsection{Affine charts}
\label{sec:affine:charts}
{Given integers $d,h\geq 0$ we define and describe affine open charts $\cU(\tld{z})^{\det,\leq h}$ for $\Gr^{\det=d,\leq h}_{\cG,X}$, for $\tld{z}\in \tld{W}^\vee$ (see Proposition \ref{prop:explicit_affine_chart} and Corollary \ref{cor:open:immersion}).}

\begin{defn} \label{defn:Lminusminus} We define the negative loop group $L^{--}\cG$ to be the subgroup of $L\cG$ whose values on Noetherian $\bZ[v]$-algebra $R$ (sending $v$ to $t$) is given by
\begin{align*}
L^{--}\cG(R) =\left\{\,A\in \GL_n\Big(R\Big[\frac{1}{v-t}\Big]\Big)\ \Big| \ 
            \begin{varwidth}{\textwidth}
               \centering
                $A$ is unipotent lower triangular mod ${\frac{1}{v-t}R\Big[\frac{1}{v-t}\Big]}$\\
                 and upper triangular mod ${\frac{v}{v-t}R\Big[\frac{1}{v-t}\Big]}$
            \end{varwidth}
            \,\right\}
\end{align*}
{(where $R\Big[\frac{1}{v-t}\Big]$ denotes the ring of polynomials in $\frac{1}{v-t}$ with coefficients in $R$; it is a subring of $R(\!(v-t)\!)$).}
\end{defn}
Note that the groups $L^+\cG, L^{--}\cG$ and $L\cG$ are formally smooth over $X=\bA^1_\Z$.
\begin{lemma} \label{lem:negative_loop_mono} The multiplication map
\[L^+\cG \times_X L^{--}\cG \rightarrow L\cG\]
is a monomorphism.
In particular, the induced map $L^{--}\cG \to \Gr_{\cG,X}$ is a monomorphism.
\end{lemma}
\begin{proof} Suppose we have a Noetherian $\bZ[v]$-algebra $R$, sending $v$ to $t\in R$. Let $g_1, g'_1\in L^+\cG(R)$ and $g_2, g'_2\in L^{--}\cG(R)$ such that $g_1g_2=g'_1g'
_2$. Then $g=(g'_1)^{-1}g_1=g'_2(g_2)^{-1}\in \GL_n(R(\!(v-t)\!))$ satisfies:
\begin{itemize}
\item  The entries of $g$ above the diagonal belong to $R[\![v-t]\!]\cap \frac{1}{v-t}R[\frac{1}{v-t}]=0$.
\item  The entries of $g$ below the diagonal belong to $vR[\![v-t]\!]\cap \frac{v}{v-t}R[\frac{1}{v-t}]=0$, since $v$ is regular in $R(\!(v-t)\!)$.
\item  The diagonal entries of $g$ belong to $R[\![v-t]\!]\cap (1+\frac{1}{v-t}R[\frac{1}{v-t}])$, and hence are equal to $1$.
\end{itemize}  
We conclude that $g=1$, hence $g_1=g'_1$, $g_2=g'_2$.

For the last statement, we observe that the natural map $L^+\cG(R) \backslash L\cG(R)\into \Gr_{\cG,X}(R)$ is an injection for any $\Z[v]$-algebra $R$.
\end{proof}

{We now define various Lie algebras that will appear in \S \ref{sec:sp:fib}, \S\ref{sec:componentmatching}, \S\ref{sub:GBases}}
Let $R\onto S$ be a surjection of $\bZ[v]$-algebra (sending $v$ to $t\in R$), such that $J=\ker(R\onto S)$ is a square-zero ideal. Define the $S$-modules%
\begin{align*}
\Lie L^{--}\cG(J) &=\left\{\, M\in M_n\Big(J\Big[\frac{1}{v-t}\Big]\Big),\quad
            \begin{varwidth}{\textwidth}
               \flushleft
                $M$ is nilpotent lower triangular mod $\frac{1}{v-t}$, \\
                 and is upper triangular mod $\frac{v}{v-t}$
            \end{varwidth}
            \,\right\},
\\
&\\
\Lie L\cG(J) &=\left\{\,
  M\in M_n(J(\!({v-t})\!)), 
\text{\small{ $M$ is upper triangular mod $v$}} 
\right\},
\\
&\\
\Lie L^+\cG(J) &=\left\{\,  M\in M_n(J[\![{v-t}]\!]), 
\text{\small{ $M$ is upper triangular mod $v$}}
\right\}.
\end{align*}
We observe that the map $M\mapsto 1+M$ gives a canonical isomorphism $\Lie L\cG(J)\cong \ker(L\cG(R)\onto L\cG(S))$. This gives an action of $L\cG(R)$ on $\Lie \cG(J)$ by conjugation, which factors through $L\cG(S)$, where we interpret matrix multiplication using the $S=R/J$-module structure on $J=J/J^2$. The same discussion also applies to $L^{--}\cG$ and $L^+\cG$.

\begin{lemma} \label{lem:Lie_algebra_decomp} Assume that we have a square-zero extension $R\onto S$ of $\bZ[v]$-algebras, with kernel $J$. Then inside $\Lie L\cG(J)$, we have a direct sum decomposition
\[\Lie L\cG(J) =\Lie L^{--}\cG(J)\oplus \Lie L^+\cG(J). \] 
\end{lemma}
\begin{proof} This follows from the direct sum decompositions
\[J(\!(v-t)\!)\,=\ J[\![v-t]\!]\  \text{\Large{$\oplus$}}\ \frac{1}{v-t} J\Big[\frac{1}{v-t}\Big],\]
\[vJ(\!(v-t)\!)=vJ[\![v-t]\!]\ \text{\Large{$\oplus$}}\ \frac{v}{v-t}J\Big[\frac{1}{v-t}\Big].\]
\end{proof}

\begin{defn} \label{defn:etale}
Let $f: F\to G$ be a morphism of functors on Noetherian rings. 
We say that $f$ is formally \'{e}tale at $x$ if for every commutative diagram
\[
\xymatrix@=3pc{
\Spec k\ar^-{{x}}[r]\ar[d]&F\ar^-{f}[d]\\
\Spec A\ar[r]\ar@{..>}[ur]&G
}
\]
with $A$ an Artinian ring with residue field $k$, there is a unique dotted arrow that makes the diagram commute.
\end{defn}
\begin{rmk}\label{rmk:standard_etale} \begin{enumerate}
\item The above notion of formally \'{e}tale is slightly weaker than the definition in  \cite[\href{https://stacks.math.columbia.edu/tag/049S}{Tag 049S}]{stacks-project}, since we only consider the lifting problems for thickenings of \emph{Artinian} affine schemes as opposed to general affine schemes.
However, for representable functors $F$, $G$ such that $G$ is locally Noetherian and $f$ is locally of finite type  \cite[\href{https://stacks.math.columbia.edu/tag/02HY}{Tag 02HY}]{stacks-project} shows that $f$ being formally \'{e}tale in the sense of Definition \ref{defn:etale} implies $f$ is \'{e}tale (and hence also formally \'{e}tale) in the sense of \cite[\href{https://stacks.math.columbia.edu/tag/049S}{Tag 049S}]{stacks-project}.
\item It is clear that being formally \'{e}tale in the above sense is preserved by composition and arbitrary base change.
\end{enumerate}
\end{rmk}
\begin{lemma} \label{lem:formally_etale_negative_loop}The multiplication map
\[L^+\cG \times_X L^{--}\cG \rightarrow L\cG\]
is formally \'{e}tale. Hence, the same is true for the natural map $L^{--}\cG \to \Gr_{\cG,X}$.
\end{lemma}

\begin{proof}We consider the commutative diagram:
\[
\xymatrix@=3pc{
\Spec k\ar^-{x}[r]\ar[d]&L^+\cG \times_X L^{--}\cG\ar[d]\\
\Spec A\ar[r]\ar@{..>}[ur]&L\cG
}
\]
where $(A,\fm_A)$ is an Artinian local ring with residue field $A/\fm_A=k$. Composing with the projection $L\cG \to X$ makes $A$ naturally a $\bZ[v]$-algebra, sending $v$ to $t$ lifting $\ovl{t}\in k$. 
The top horizontal arrow corresponds to a pair $\ovl{g}_1\in L^{+}\cG(k)$, $\ovl{g}_2\in L^{--}\cG(k)$.
The bottom horizontal arrow correspond to $g\in L\cG(A)$ lifting $\ovl{g}=\ovl{g}_1\ovl{g}_2$.

We need to show that the dotted arrow exists and is unique. We assume $\fm_A\neq 0$, otherwise there is nothing to prove.
The uniqueness follows from the fact that the right vertical map is a monomorphism, by Lemma \ref{lem:negative_loop_mono}.

We now show the existence of the dotted arrow, that is we need to show that $g$ admits a decomposition $g=g_1g_2$ with $g_1\in L^{+}\cG(A)$, $g_2\in L^{--}\cG(A)$. By inducting on the length of $A$, we may assume that we have the desired decomposition for $g$ mod $\eps$, where $0\neq \eps \in \fm_A$ is annihilated by $\fm_A$. We have the square-zero extension $A\onto A/\eps$. Since $L^+\cG$ and $L^{--}\cG$ are formally smooth, we can find $g'_1 \in L^+\cG(A)$, $g'_2\in L^{--}\cG(A)$ such that $(g'_1)^{-1}g (g'_2)^{-1}\in \ker(L\cG(A)\to L\cG(A/\eps))=1+\eps X$.
By Lemma \ref{lem:Lie_algebra_decomp} (noting that $\ker(L\cG(A)\to L\cG(A/\eps))$ is canonically isomorphic to $\Lie(L\cG)(k\eps)$), we can decompose $\eps X=\eps X_1+\eps X_2$ such that $(1+\eps X_1)\in \ker(L^+\cG(A)\to L^+\cG(A/\eps))$, $(1+\eps X_2)\in \ker(L^{--}\cG(A)\to L^{--}\cG(A/\eps))$. This yields the desired decomposition 
\[g=\big(g'_1(1+\eps X_1)\big)\big( (1+\eps X_2)g'_2\big).\]
\end{proof}

Let $\tld{z}=wt_{\nu} \in \tld{W}^{\vee}$ as defined in \S \ref{sub:AWG:1}. We define $\cU(\tld{z})$ to be the subfunctor of $L\cG$ whose value on a Noetherian $\bZ[v]$-algebra $R$ (sending $v$ to $t\in R$) is given by %
\begin{align*}
\cU(\tld{z})(R)=\left\{\,A\in \GL_n(R(\!(v-t)\!))\ \bigg| \ 
            \begin{varwidth}{\textwidth}
               \centering
                $A (v-t)^{-\nu} w^{-1} \in \GL_n(R[\frac{1}{v-t}])$ is unipotent lower triangular mod $\frac{1}{v-t}$\\
\vspace{.1cm}
                 and $A(v-t)^{-\nu}\in \GL_n(R[\frac{1}{v-t}])$ is upper triangular mod $\frac{v}{v-t}$
            \end{varwidth}
            \,\right\}.
\end{align*}
\begin{lemma}\label{lem:open_chart_mono} Left multiplication by $L^{--}\cG$ in $L\cG$ preserves $\cU(\tld{z})$, and makes $\cU(\tld{z})$ an $L^{--}\cG$-torsor. The natural map $\cU(\tld{z})\to \Gr_{\cG,X}$ is a formally \'{e}tale monomorphism.
\end{lemma}
\begin{proof} The first claim follows immediately from the definitions. For the second claim, note that for any Noetherian $\bZ[v]$-algebra $R$, either $\cU(\tld{z})(R)=\emptyset$, or it is a left coset of $L^{--}\cG(R)$ in $L\cG(R)$. The fact that $\cU(\tld{z})\to \Gr_{\cG,X}$ is a monomorphism then follows from Lemma \ref{lem:negative_loop_mono}. 

To show that $\cU(\tld{z})\to \Gr_{\cG,X}$ is formally \'{e}tale, we first note that $\cU(\tld{z})$ is formally smooth over $X$.
Indeed the condition that $A\in \GL_n(R(\!(v-t)\!))$ belongs to $\cU(\tld{z})(R)$ is that each entry $A_{ij}$ of $A$ has the form $(v-t)^d(\frac{1}{v-t})^{\delta}(\frac{v}{v-t})^{\delta'}R[\frac{1}{v-t}]$, where $d\in \Z$, $\delta,\delta'\in \{0,1\}$ are determined by $i,j$ and $\tld{z}$, and hence it is clear that the map $\cU(\tld{z})(R)\to \cU(\tld{z})(S)$ is surjective for any square-zero nilpotent thickening $R\onto S$. This together with Lemma \ref{lem:formally_etale_negative_loop} shows that $\cU(\tld{z})\to \Gr_{\cG,X}$ is formally \'{e}tale.
\end{proof}

For $h\geq 0$, we define $\cU(\tld{z})^{\det,\leq h}$ to be the intersection $\cU(\tld{z})\cap L\cG^{det=d, \leq h}$, where $d=|\!|\nu|\!|:=\sum_i \nu_i$ if $\nu=(\nu_i)_i\in X_*(T^\vee)=\Z^n$.
We have the following explicit description:
\begin{prop}\label{prop:explicit_affine_chart}
 For a Noetherian $\bZ[v]$-algebra $R$, $\cU(\tld{z})^{\det,\leq h}(R)$ is the set of $n\times n$ matrices $A$ with Laurent polynomial entries $A_{ij}\in R[v-t,\frac{1}{v-t}]$ satisfying the following degree bound and determinant condition:
\begin{itemize}
\item For $1\leq i, j\leq n$, 
\[A_{ij}=v^{\delta_{i>j}}\Bigg(\sum_{k=-h}^{ \nu_j-\delta_{i>j}-\delta_{i<w(j)} }c_{ij,k}(v-t)^k\Bigg),\]
 and $c_{w(j) j, \nu_j-\delta_{w(j)>j}} = 1$.
\item $\det A=\det(w) (v-t)^{|\!|\nu|\!|}$.
\end{itemize}
\end{prop}
\begin{proof} The first item follows from unraveling the definition. For the second item, the condition given in the definition is $\det A\in R[\![v-t]\!]^\times (v-t)^{|\!|\nu|\!|}$. However, a priori $\det A \in \det(w)(v-t)^{|\!|\nu|\!|}(1+\frac{1}{v-t}R[\frac{1}{v-t}])$, hence the determinant condition is equivalent to $\det A=\det(w) (v-t)^{|\!|\nu|\!|}$.
\end{proof}
Thus $\cU(\tld{z})^{\det,\leq h}$ is representable by an affine scheme of finite type over $\bZ$, namely the spectrum of the quotient of the polynomial ring generated by the coefficients $c_{ij,k}$ modulo the relations given by the determinant condition. Note that $\cU(\tld{z})^{\det,\leq h}=\emptyset$ unless $h$ is sufficiently large, namely when $h+\nu_j-\delta_{i>j}-\delta_{i<w(j)}\geq 0$ for all $i, j$. 

\begin{defn} \label{defn:section} When $\cU(\tld{z})^{\det,\leq h} \neq \emptyset$, there is a section $\Spec \Z \into \cU(\tld{z})^{\det,\leq h}\times_X \{0\}$ given by the element $\tld{z}=wv^\nu\in \cU(\tld{z})^{\det,\leq h}(R)\subset \GL_n(R(\!(v)\!))$, for any Noetherian $\Z[v]$ algebra $R$ sending $v$ to $0$. We will abusively denote this section and the corresponding $\Z$-point of $\Gr_{\cG, X} \times_X \{0\}$ by $\tld{z}$.
\end{defn}

\begin{cor} 
\label{cor:open:immersion}
The natural map
\begin{align*}
\iota:\cU(\tld{z})^{\det,\leq h} \rightarrow \Gr^{\det=|\!|\nu|\!|,\leq h}_{\cG,X}
\end{align*}
is an open immersion.
\end{cor}
\begin{proof} We observe that $\cU(\tld{z})^{\det,\leq h}= \cU(\tld{z})\times _ {\Gr_{\cG,X} }\Gr^{\det=|\!|\nu|\!|,\leq h}_{\cG,X}$  
Hence Lemma \ref{lem:open_chart_mono} shows that $\iota$ is a formally \'{e}tale monomorphism. By Remark \ref{rmk:standard_etale}, $\iota$ must then be an \'{e}tale monomorphism, and hence is an open immersion by \cite[\href{https://stacks.math.columbia.edu/tag/025G}{Tag 025G}]{stacks-project}.
\end{proof}

\subsection{Universal local models}
\label{subsec:UMLM}

Let $L\cG^{\nabla}$ be the subfunctor of $L\cG\times _\Z \bA^n$ whose value on a Noetherian $\bZ[v]$-algebra $R$ (sending $v$ to $t\in R$) is given by 
\begin{equation} \label{eq:universalnabla}
\text{\small{$L\cG^{\nabla}(R) \defeq\bigg\{  (g,\bf{a}) |\  g\in L\cG(R),\, \bf{a}\in R^n \text{ and } v \frac{dg}{dv} g^{-1}  + g \Diag(\bf{a}) g^{-1} \in \frac{1}{v-t} L^+\cM(R)\bigg\}$}}
\end{equation}
{(where the symbol $\frac{dg}{dv}$ means we differentiate entry-wise).}
{
\begin{lemma}
The functor $L\cG^{\nabla}$ is stable by left multiplication by $L^+\cG^{\nabla}$.
\end{lemma}
\begin{proof}
Let $R$ be a Noetherian $\bZ[v]$-algebra, sending $v$ to $t\in R$, and let $h\in L^+\cG^{\nabla}(R)$,  $L\cG^{\nabla}(R)$.
The Leibnitz rule gives
\[
v\frac{d(hg)}{dv}(hg)^{-1}+hg \Diag(\bf{a}) (hg)^{-1}=v\frac{d(h)}{dv}h^{-1}+h\Bigg(v\frac{dg}{dv}g^{-1}+g \Diag(\bf{a}) g^{-1}\Bigg)
\]
and the right hand side is manifestly an element in $\frac{1}{v-t} L^+\cM(R)$ since $\frac{1}{v-t} L^+\cM(R)$ is stable by conjugation by $h$, and that $\frac{d(h)}{dv}\in L^+\cM(R)$.
\end{proof}
}

{Thus, $L\cG^{\nabla}$} defines a closed sub-ind-scheme $\Gr^{\nabla}_{\cG,X} \defeq L^{+}\cG\backslash L\cG^{\nabla} \subset \Gr_{\cG,X}\times_{\Z} \bA^n$ which is ind-proper over $X\times_{\Z} \bA^n$.

For $\lambda \in X_*(T^\vee)$, we have a section $s_\lambda: X \to \Gr_{\cG,X}$ induced by the element $(v-t)^\lambda \in L\cG(R)$ for a $\bZ[v]$-algebra $R$ sending $v$ to $t\in R$.
\begin{rmk}\label{rmk:setting_last_0} For $H\cong \bA^{n-1}\subset \bA^n$ a hyperplane where one of the coordinates is $0$, we have a natural isomorphism
\[L\cG^\nabla \cong (L\cG^\nabla \cap (L\cG \times_\Z H))\times_\Z \bA^1\]
For example, if $H=\{\bf{a}\in \bA^n  \mid \bf{a}_n=0\}$, we have an isomorphism given by
\[(g,\bf{a})\mapsto (g,\bf{a}-(\bf{a}_n,\cdots, \bf{a}_n),\bf{a}_n)\] 
Because this observation, we could have always worked under the assumption that $\bf{a}_n=0$ throughout the entire paper. This minor simplification is useful when implementing computer algebra computations, see for example Appendix \ref{ex:failure:UB}
\end{rmk}
We define the global Schubert variety $\cS_X(\lambda)$ to be the minimal irreducible closed subscheme of $\Gr_{\cG,X}$ which contains $s_\lambda$ and is stable under the right multiplication action of $L^+\cG$ (cf.~\cite[Definition 3.1]{Zhu_coherence}). We will also write $\cS_{X^0}(\lambda)=\cS_X(\lambda)\times _X X^0$. The maps $\cS_X(\lambda)\to X$, $\cS_{X^0}(\lambda)\to X^0$ are proper. 
Note that as in \cite[Lemma 3.6]{Zhu_coherence}, we have an isomorphism $\Gr_{\cG,X}\times_X X^0 \cong \Gr_{\GL_n} \times_{\bZ} X^0$, under which $\cS_{X^0}$ corresponds to the constant family of the Schubert variety of $\Gr_{\GL_n}$ for the coweight $\lambda$ over $X^0$. This description makes it clear that for any geometric point $x$ of $X^0$, the fiber $\cS_X(\lambda) \times_X\,x\subset  \Gr_{\cG,X}\times_X\,x\cong \Gr_{\GL_n}$ is the usual Schubert variety for the coweight $\lambda$ in $\Gr_{\GL_n}$. In particular, we have $\cS_{X^0}(\lambda)=\cS_{X^0}(w(\lambda))$ for $w\in W$.  We also have the open Schubert variety $\cS_{X}^\circ(\lambda)=\cS_X(\lambda)\setminus \bigcup_{\lambda'\in \Conv(\lambda), \lambda'\notin W\lambda} \cS_X(\lambda')$. Over $X^0$, $\cS_{X^0}^\circ(\lambda)$ correspond to the constant family of the open Schubert variety for the coweight $\lambda$ in $\Gr_{\GL_n}$.

Given $\lambda\in X_*(T^\vee)$, we have the stabilizer group scheme of $s_\lambda$ whose values on a $\bZ[v]$-algebra $R$ is given by
\[L^+\cG_{\lambda}(R)=L^+\cG(R) \cap \Ad\!\big((v-t)^{-\lambda}\big)\big(L^+\cG(R)\big) \]
Let $P_\lambda$
be the parabolic subgroup of $\GL_n$ determined by the condition that the $\alpha$-th entry vanishes for all roots $\alpha$ such that $\langle \lambda, \alpha^\vee \rangle < 0$. Then there is a natural map $(L^+\cG_\lambda \backslash L^+\cG)_{X^0} \to P_\lambda \backslash \GL_n\times_{\Z} X^0$ given by $g \mapsto g \textrm{ mod } (v-t)$, which makes $L^+\cG_\lambda \backslash L^+\cG$ into an iterated affine space bundle over the partial flag variety $P_\lambda \backslash \GL_n \times_{\Z} X^0$ {(see the discussion after  \cite[Corollary 2.1.11]{zhu-intro-AG} or \cite[\S 2]{mirkovic-vilonen})}. 

Then for sufficiently large $h$, we have a monomorphism $L^+\cG_\lambda \backslash L^+\cG \into  \Gr^{\det=|\!|\lambda|\!|,\leq h}_{\cG,X}$ given by the orbit map $g\mapsto s_\lambda g$, and $\cS_X(\lambda)$ is the scheme-theoretic image of this map. The orbit map induces an isomorphism $(L^+\cG_\lambda \backslash L^+\cG)\times_X X^0 \cong \cS^\circ_{X^0}(\lambda)$. This gives us a map $\pi_\lambda: \cS_{X^0}^\circ(\lambda)\to (P_{\lambda}\backslash \GL_n) \times_{\Z} X^0$.

\begin{defn} We define the naive universal local model to be
\[\cM^{\nv}_X(\leql, \nabla)\defeq \Gr^\nabla_{\cG,X}\cap (\cS_X(\lambda)\times_{\Z} \bA^n).\]
\end{defn}
We will also set $\cM^{\nv}_{X^0}(\leql, \nabla)=\cM^{\nv}_X(\leql, \nabla)\times_X X^0$. It is a proper scheme over $X^0\times_{\Z} \bA^n$.

 For any $\tld{z}\in \tld{W}^{\vee}$ and $h$ sufficiently large, we have
\[\big(\cU(\tld{z})^{\det,\leq h}\times_{\Z} \bA^n\big)\cap \cM^{\nv}_X(\leql, \nabla)=\big(\cU(\tld{z})\times_{\Z} \bA^n\big) \cap \cM^{\nv}_X(\leql, \nabla)\]
is an (possibly empty) open subscheme of $\cM^{\nv}_X(\leql, \nabla)$, and denote this by $\cU^{\nv}(\tld{z},\leql, \nabla)$.

The following Lemma describes the part of $\cM^{\nv}_{X^0}(\leql,\nabla)$ in the open global Schubert variety $\cS^\circ_{X^0}(\lambda)$, away from small positive characteristics:
\begin{prop} \label{lem:intersect_open_schubert_variety}   \label{prop:intersect_open_schubert_variety}Let $\lambda$ be dominant and recall $h_\lambda=\max_{\alpha^\vee}\{\langle \lambda,\alpha^\vee\rangle\}$.
The map $\pi_\lambda$ induces an isomorphism $\pi_\lambda: \big(\cM^{\nv}_{X^0}(\leql, \nabla)\cap (\cS^\circ_{X^0}(\lambda)\times_{\Z} \bA^n)\big)[\frac{1}{h_\lambda !}] \risom (P_{\lambda}\backslash \GL_n) \times_{\Z} X^0\times_{\Z} \bA^n[\frac{1}{h_\lambda!}]$.
\end{prop}
\begin{proof} We first note that we have an open cover of $P_{\lambda}\backslash \GL_n$ by affine spaces given by $\ovl{N}_\lambda w$ where $\ovl{N}_\lambda$ is the unipotent radical of the opposite parabolic to $P_\lambda$, and $w$ runs over $W$.
This pulls back to an open cover $\cS^\circ(\lambda)_{X^0}=\bigcup_{w\in W} \tld{N}_\lambda w$, where $\tld{N}_\lambda$ is the affine scheme over $X^0$ whose points on a $\bZ[v,v^{-1}]$ algebra $R$ consists of the set of matrices $(v-t)^{\lambda}N w$ where $N\in \GL_n(R[\![v-t]\!])$ is a matrix such that
\begin{itemize}
\item The diagonal entries of $N$ are $1$.
\item For a root $\alpha$ such that $\langle -\lambda, \alpha^\vee \rangle\leq 0$, the entry $N_\alpha=0$.
\item For a root $\alpha$ such that $\langle -\lambda, \alpha^\vee \rangle> 0$, the entry $N_\alpha=\sum_{j=0}^{\langle -\lambda, \alpha^\vee \rangle-1} X_{\alpha,j} (v-t)^j$ with $X_{\alpha,j}\in R$.
\end{itemize}
Note that this describes an affine space over $X^0$, whose coordinates are given by the coefficients $X_{\alpha,j}$ of the entries of $N$.
Under these coordinates, the map $\pi_\lambda$ is the map $(v-t)^\lambda N w \mapsto (N \textrm{ mod } (v-t) )w$.

It suffices to show that $\pi_\lambda: \Gr^\nabla_{\cG,X^0}\cap \Big(\tld{N}_\lambda w \times_{\Z} \bA^n[\frac{1}{h_\lambda!}]\Big)\to \ovl{N}_\lambda w \times_{\Z} X^0\times_{\Z} \bA^n[\frac{1}{h_\lambda!}]$ induces an isomorphism for each $w\in W$.
Fix an $R$-point $x$ of $X^0\times_{\Z} \bA^n$ corresponding to $t\in R^\times$ and $\bf{a}\in R^n$. 
The set of $R$ points of $\Gr^\nabla_{\cG,X^0}\cap \Big(\tld{N}_\lambda w \times_{\Z} \bA^n[\frac{1}{h_\lambda!}]\Big)$ above $x$ is the set of matrices $(v-t)^\lambda N w$ with $N\in \GL_n(R[\![v-t]\!])$ as above such that
\begin{equation}
\label{eq:solve}
v \frac{d}{dv} \Big((v-t)^\lambda N w \Big)\big((v-t)^\lambda N w\big)^{-1}  + \Ad \big((v-t)^\lambda N w \big)(\Diag(\bf{a}))  \in \frac{1}{v-t} L^+\cM(R) 
\end{equation}
which is equivalent to
\[\frac{v}{v-t} \lambda + \Ad\big((v-t)^{\lambda}\big)\Big( \big(v\frac{d}{dv} N\big) N^{-1}\Big)+ \Ad\big((v-t)^{\lambda} \big) (\Ad (N) (\Ad(w) ( \Diag(\bf{a})))) \in \frac{1}{v-t} L^+\cM(R). \]
This is in turn equivalent to (noting that $v\in R[\![v-t]\!]^\times$ since $t\in R^\times$)
\[ \Big( v\frac{d}{dv} N +\big[N,\Ad(w)(\Diag(\bf{a}))\big] \Big) N^{-1} \in \frac{1}{v-t}\Ad\big((v-t)^{-\lambda}\big) (M_n(R[\![v-t]\!])). \]
Note that the only entries in the above matrix that can be non-zero are the $\alpha$-th entries where $\langle -\lambda, \alpha^\vee \rangle >0$ (which in particular implies $\alpha <0$), and for such $\alpha$ the above condition is that the $\alpha$-th entry is divisible by $(v-t)^{\langle -\lambda, \alpha^\vee \rangle-1}$.
Now, for $\langle {-}\lambda, \alpha^\vee \rangle >0$, the $\alpha$-th entry of the above matrix has the form
\[\Big(v\frac{d}{dv}- \big\langle \Ad(w)(\Diag(\bf{a})),\alpha^\vee\big\rangle\Big)N_\alpha +\cdots  \]
where the terms in $\cdots$ involves only $N_\beta$ where $\alpha<\beta<0$. On the other hand, since $N_\alpha =\sum_{i=0}^{\langle -\lambda, \alpha^\vee\rangle-1} X_{\alpha,i} (v-t)^i$, we have:
\begin{align*}
\Big(v\frac{d}{dv}- \big\langle \Ad(w)(\Diag(\bf{a})),\alpha^\vee\big\rangle\Big)N_\alpha&=\sum_{i=0}^{\langle -\lambda, \alpha^\vee\rangle-2} t(i+1)X_{\alpha,i+1} (v-t)^i \\
&\qquad\qquad+ \sum_{i=0}^{\langle -\lambda, \alpha^\vee\rangle-1}( i-\big\langle \Ad(w)(\Diag(\bf{a})),\alpha^\vee\big\rangle)X_{\alpha,i} (v-t)^i.
\end{align*}
Since $i+1, t \in R^{\times}$ for all $0 \leq i<h_\lambda - 1$, equation (\ref{eq:solve}) solves each $X_{\alpha,i}$ for $i>0$ uniquely in terms of $X_{\alpha',0}$ for $\alpha \leq \alpha' <0$. As $\pi_\lambda$ is exactly obtained by extracting $X_{\alpha,0}$ for all $\alpha$ such that $\langle -\lambda, \alpha^\vee \rangle >0$, we are done.
\end{proof}
We thus have a description of the underlying reduced scheme of $\cM_{X^0}^{\nv}(\leq \lambda,\nabla)$ away from small positive characteristics:
\begin{cor} \label{cor:monodromy_before_reduced} 
Let $\lambda$ and $h_\lambda$ be as in Lemma \ref{lem:intersect_open_schubert_variety}, then the underlying reduced subscheme of $\cM_{X^0}^{\nv}(\leql, \nabla)[\frac{1}{h_\lambda!}]$
is isomorphic to $\coprod_{\lambda'\leq \lambda, \lambda'\in X_*^+(T^\vee)} (P_{\lambda'}\backslash \GL_n) \times_{\Z} X^0\times_{\Z} \bA^n[\frac{1}{h_\lambda!}] $. %
\end{cor}
\begin{proof}
By {Proposition} \ref{lem:intersect_open_schubert_variety}, $\cM_{X^0}^{\nv}(\leql,\nabla)\cap (\cS^\circ(\lambda)_{X^0}\times_{\Z} \bA^n[\frac{1}{h_{\lambda}!}])$ is isomorphic to $(P_{\lambda}\backslash \GL_n) \times_{\Z} X^0\times_{\Z} \bA^n[\frac{1}{h_{\lambda}!}]$, and hence is proper over $X^0\times_{\Z} \bA^n[\frac{1}{h_{\lambda}!}]$. Thus the inclusion $\cM_{X^0}^{\nv}(\leql,\nabla)\cap (\cS^\circ(\lambda)_{X^0}\times_{\Z} \bA^n[\frac{1}{h_{\lambda}!}])\into \cM_{X^0}^{\nv}(\leql,\nabla)[\frac{1}{h_{\lambda}!}]$ is a proper open immersion, hence is the inclusion of a connected component. The complement of this component has the same support as $\Gr^\nabla_{\cG,X^0}\cap \Big((\cS_{X^0}(\lambda)\setminus \cS^\circ_{X^0}(\lambda))\times_{\Z} \bA^n[\frac{1}{h!}]\Big)$. Since $\cS_{X^0}(\lambda)\setminus \cS_{X^0}^\circ(\lambda)=\bigcup_{\lambda
<\lambda, \lambda' \in X_*^+(T^\vee)} \cS_{X^0}(\lambda')$ set theoretically and $h_{\lambda'} \leq h_{\lambda}$, we can repeat the above argument for $\lambda'<\lambda$ to conclude. \end{proof}
\noindent
{In particular, $\cM_{X^0}^{\nv}(\leql,\nabla)\cap (\cS^\circ_{X^0}(\lambda)\times_{\Z} \bA^n)$
 is a connected component of $\cM_{X^0}^{\nv}(\leql,\nabla)$.}
 
We can now make the following definition:
\begin{defn} Let $\lambda\in X_*(T^\vee)$ be dominant. {The universal local model $\cM_X(\lambda,\nabla)$ is the closure of $\cM_{X^0}^{\nv}(\leql,\nabla)\cap (\cS^\circ_{X^0}(\lambda)\times_{\Z} \bA^n)$ in $\cM_X^{\nv}(\leql,\nabla)$ (equivalently, in $\Gr_{\cG,X}\times_{\Z} \bA^n$).}
\end{defn}
{We note that the difference between $\cM_X(\lambda,\nabla)$ and $\cM_X^{\nv}(\leql,\nabla)$ is that the 
monodromy condition is imposed, respectively, before and after taking Zariski closures of $\cS^\circ_{X^0}(\lambda)\times_{\Z} \bA^n$ in $\Gr_{\cG,X}\times_{\Z} \bA^n$.}

We will now show that the conclusion of Corollary \ref{cor:monodromy_before_reduced} actually holds without taking reduced subscheme, at the expense of removing some more small positive characteristics.

We have an action of the torus $T^\vee\times_{\Z} X$ on the $X$-scheme $\Gr_{\cG,X}$ induced by the right multiplication action of $T^\vee$ on $L\cG$. 
This action evidently preserves $\cS_X(\lambda)$ and $\Gr_{\cG,X}^{\nabla}$. Thus we get an induced $T^\vee$ action on $\cM_{X}^{\nv}(\leql,\nabla)$.

\begin{lemma} The $T^\vee$-fixed point scheme of $\cS_{X^0}(\lambda)$ is supported on the union of the sections $s_{\lambda'}$ as $\lambda'$ runs over the elements of $\mathrm{Conv}(\lambda)$.
\end{lemma}
\begin{proof} This statement can be checked at the level of geometric fibers over $X^0$, where the conclusion is well-known{, see for instance the discussion just before \cite[Example  2.1.12]{zhu-intro-AG}}.
\end{proof}
\begin{prop} \label{prop:smooth_generic_fiber}Let $\lambda$ be dominant. Then $\cM_{X^0}^{\nv}(\leql,\nabla) [\frac{1}{(2h_\lambda)!}]$
 is smooth over $X^0\times_{\Z} \bA^n[\frac{1}{(2h_\lambda)!}]$.
\end{prop}
\begin{proof} To ease notation, in this proof we will abbreviate $Y=\cM_{X^0}^{\nv}(\leql,\nabla)[\frac{1}{(2h_\lambda)!}]$, $S=X^0\times_{\Z} \bA^n[\frac{1}{(2h_\lambda)!}]$ and $pr$ for the natural projection map $Y\to S$.
As $pr$ is a finite type map between finite type $\bZ$-schemes, the locus of points where $pr$ is smooth is open in the domain. The non-smooth locus is thus proper over $S$, and is furthermore $T^\vee$-stable. If it is non-empty, it must have a non-zero geometric fiber over $X^0\times_{\Z} \bA^n[\frac{1}{(2h_\lambda)!}]$. %
Such a fiber will be a proper variety over a field with an action of a torus, and hence must contain a torus-fixed closed point, which must occur in the support of $s_\mu$ for some $\mu\in \mathrm{Conv}(\lambda)$.

 Thus, we only need to show that $pr$ is smooth at any closed point $x:\Spec k \to Y$ lying in the support of $s_\mu$. Let $s=pr(x)\in S$. We will show smoothness by bounding the dimension of the tangent space $\dim T_x Y \leq \dim_x Y$ and the dimension of the relative tangent space $\dim T_x Y/S\leq \dim_x Y_s$. Indeed, granting this, we deduce that the completion $\widehat{\cO}_{Y,x}$ is generated over $\widehat{\cO}_{S,s}$ by $\dim_xY_s$ elements. Using that $S$ is regular at $s$ and comparing dimensions, we then conclude that $\widehat{\cO}_{Y,x}$ is a power series ring over $\widehat{\cO}_{S,s}$ in $\dim_xY_s$ variables.

We may without loss of generality enlarge $k$ and assume that $k$ is algebraically closed. The point $s$ corresponds to the data of a tuple $(t,\bf{a})\in k^\times \times k^n$ and the point $x$ corresponds to the data of $s$ and the element $(v-t)^\mu \in \GL_n(k(\!(v-t)\!))$.
By Corollary \ref{cor:monodromy_before_reduced}, the point $x$ is on the connected component of $Y$ occurring inside $\cS_{X^0}^\circ(\mu)\times_{\Z} \bA^n$, and hence $\dim_x Y = \dim P_\mu \backslash \GL_n+ n+1$.

Let $\cU_{X_0}(t_{\mu}) \defeq \cU(t_{\mu}) \times_X X_0$.  Set $U\defeq\left(\big(\cU_{X^0}(t_{\mu})\cap \cS_{X^0}(\lambda)\big)\times_{\Z} \bA^n\right) \cap \Gr^{\nabla}_{\cG,X^0}$, which is an open neighborhood of $x$ in $Y$.
 We observe that $\cU_{X^0}(t_{\mu})\cap \cS_{X^0}(\lambda)$ occurs inside the closed subscheme of $Z\defeq \Gr_{\cG,X}^{ \det=0,\leq h_{\lambda}}s_\mu\cap \cU_{X^0}(t_\mu)$ of $\cU_{X^0}(t_{\mu}) \subset \Gr_{\cG,X^0}$, since for an element $g \in \cU_{X^0}(t_{\mu})(R)$ to occur in $\cS_{X^0}(\lambda)(R)$, a necessary condition is that $\det g \in R[\![\!v-t]\!]^\times (v-t)^{|\!|\lambda|\!|}$ and that $|\!|\lambda|\!|=|\!|\mu|\!|$, and that each entry of $g(v-t)^{-\mu}$ belongs to 
$(v-t)^{\lambda_{\min}-\mu_{\max}}R[\![v-t]\!]$, where $\lambda_{\min}=\min_{0\leq i \leq n-1 } \lambda_i$ and $\mu_{\max}=\max_{0\leq i \leq n-1 } \mu_i$, and $\lambda_{\min}-\mu_{\max}\geq \min_{\alpha^\vee}\{\langle \lambda,\alpha^\vee\rangle\}=-h_\lambda$.

Thus we conclude that $T_xY_s$ is a subspace of the tangent space of $\left(\Gr^{\nabla}_{\cG,X^0}\cap  \big(Z \times_{\Z} \bA^{n}\big)\right)_{\!\!s}$ at $x$.
This latter space has the following explicit description (using definition of $\cU_X(t_{\mu})$ before Lemma \ref{lem:open_chart_mono}): It is the space of matrices $(1+\eps X)(v-t)^\mu$ with  $X\in M_n(k(\!(v-t)\!))$ such that 
\begin{itemize}
\item For each $i$ the diagonal entry $X_{ii}=\sum_{j=1}^{h_{\lambda}} X_{ii,j}(v-t)^{-j}$ with $X_{ii,j}\in k$.
\item For each root $\alpha$, the $\alpha$-th entry $X_\alpha= v^{\delta_{\alpha<0}}\sum_{j=1}^{h_{\lambda}} X_{\alpha,j} (v-t)^{-j}$.
\item $X$ is subject to the condition that
 \[
 v \frac{d}{dv}\Big((1+\eps X)(v-t)^\mu \Big) (v-t)^{-\mu}(1+\eps X)^{-1}  + \Ad\big(1+\eps X\big) \big(\Ad\big((v-t)^{\mu}\big) (\Diag(\bf{a})) \big) \in \frac{1}{v-t} L^+\cM(k[\eps]/\eps^2).
 \]
\end{itemize}
The last condition is equivalent to
\[\eps v\frac{d}{dv} X -\eps\frac{v}{v-t}[\mu,X]-\eps [\Diag(\bf{a}),X]\in \frac{1}{v-t} L^+\cM(k[\eps]/\eps^2). \]
Hence, we have
\begin{itemize}
\item For each $1\leq i \leq n$, 
\[\sum_{j=1}^{h_\lambda} -t j X_{ii,j}(v-t)^{-(j+1)} +\sum_{j=1}^{h_\lambda} -j X_{ii,j} (v-t)^{-j} \in \frac{1}{v-t} k[\![v-t]\!]. \]
This is equivalent to $tj X_{ii,j}=-(j+1)X_{ii,j+1}$ for all $j\geq 1$ (with the convention that $X_{jj,h_\lambda+1}=0$). Since $h_\lambda!$ and $t$ are invertible in $k$, we conclude that $X_{ii,j}=0$ for all $j$.
\item For any root $\alpha$,
\begin{align*}
\sum_{j=1}^{h_\lambda} -t j X_{\alpha,j}(v-t)^{-(j+1)} &+\sum_{j=1}^{h_\lambda} -j X_{\alpha,j} (v-t)^{-j}
-\sum_{j=1}^{h_\lambda} t \langle \mu, \alpha^\vee\rangle  X_{\alpha,j}(v-t)^{-(j+1)}
-\\
&\hspace{2cm}-\sum_{j=1}^{h_\lambda} (\langle \mu+\bf{a}, \alpha^\vee\rangle -\delta_{\alpha<0}) X_{\alpha,j}(v-t)^{-j}\in \frac{1}{v-t} k[\![v-t]\!]. 
\end{align*}
This is equivalent to 
\[t(j+\langle \mu, \alpha^\vee \rangle)X_{\alpha,j}=-(j+1-\delta_{\alpha<0}+ \langle \mu+\bf{a},\alpha^\vee \rangle )X_{\alpha,j+1} \]
for $j\geq 1$  (with the convention that $X_{\alpha,h_\lambda+1}=0$).

If $\langle \mu,\alpha^\vee \rangle\geq 0$, since $t$ and $(2h_\lambda)!$ are invertible in $k$ and $\langle \mu,\alpha^\vee \rangle \leq h_{\lambda}$, $t(j+\langle \mu, \alpha^{\vee}\rangle)$ is invertible for all $1\leq j\leq h_\lambda$, and hence the above recursion forces $X_{\alpha,j}=0$ for all $j$.

If $\langle \mu,\alpha^\vee \rangle< 0$, $t(j+\langle \mu, \alpha^\vee\rangle)$ is invertible in $k$ unless $j=-\langle \mu, \alpha^\vee \rangle$.
Thus the above recursion shows that $X_{\alpha,j}=0$ for $j>-\langle \mu, \alpha^\vee \rangle$, that $X_{\alpha,j}$ for $1\leq j<-\langle \mu, \alpha^\vee \rangle$ is a multiple of $X_{\alpha,-\langle \mu, \alpha^\vee \rangle}$ by a particular constant, and there are no restrictions on $X_{\alpha,-\langle \mu, \alpha^\vee \rangle}$.
\end{itemize}
The upshot of the above discussion is that $\dim T_x Y_s \leq \#\{ \alpha | \langle \mu, \alpha^\vee \rangle <0\}=\dim_{\Z} P_\mu\backslash \GL_n$. 
Putting everything together, we have 
\[\dim T_x Y \leq \dim_x Y_s +\dim_s S \leq \dim_\Z P_\mu\backslash \GL_n+1+n+1=\dim_x Y,\]
which is what we want.
\end{proof}
{We record the following proposition which is an adaptation of Elkik's approximation theorem to our situation, which will only be used in the proof of Theorem \ref{thm:stack_local_model}.
Roughly, it allows us to lift mod $t^m$-points of $\cU^{\nv}(\tld{z}, \leql,\nabla)$ once $m$ is sufficiently large.}

\begin{prop}
\label{prop:Elkik}
Fix $\lambda$, $\tld{z}$. Choose a finite presentation of the map $\cU^{\nv}(\tld{z}, \leql,\nabla)[\frac{1}{(2h_\lambda)!}]\to X\times_{\Z} \bA^n$. 
Then there exists integers $N$ and $r$ (depending on the chosen finite presentation) such that the following hold:

Let $A$ be a ring and $t\in A$ such that $A$ is $t$-adically complete and $t$-torsion-free, and let $g: \Spec A \to X\times_{\Z} \bA^n$ be a map sending $v$ to $t$.
Then for any integer $m\geq N$, and any commutative diagram
\begin{equation}
\xymatrix{ \Spec A/t^m \ar[r]^-f\ar[d]&  \cU^{\nv}_X(\tld{z}, \leql,\nabla)\Big[\frac{1}{(2h_\lambda)!}\Big] \ar[d] \\
\Spec A\ar[r]^-g &  X\times_{\Z} \bA^n }
\end{equation}
we can find a map $\tld{f}:\Spec A \to \cU^{\nv}(\tld{z}, \leql,\nabla)\Big[\frac{1}{(2h_\lambda)!}\Big]$ which agrees with $f$ modulo $t^{m-r}$ making the following diagram commute:
\[
\xymatrix{ \Spec A \ar[r]^-{\tld{f}}\ar[d] &  \cU^{\nv}(\tld{z}, \leql,\nabla)\Big[\frac{1}{(2h_\lambda)!}\Big] \ar[d] \\
\Spec A\ar[r]^g &  X\times_{\Z} \bA^n }
\]
\end{prop}

\begin{proof} We will apply \cite[Lemme 1]{Elkik}. Let $\Spec B$ be the base change of $ \cU^{\nv}(\tld{z}, \leql,\nabla)[\frac{1}{(2h_\lambda)!}]$ along $g$.
Let $\Spec S=X\times_{\Z} \bA^n$. Let our chosen presentation be $ \cU^{\nv}(\tld{z}, \leql,\nabla)[\frac{1}{(2h_\lambda)!}]=\Spec S[X_1,\ldots, X_k]/J$.
Let $H$ be the ideal of $S[X_1, \ldots, X_k]$ defined in \cite[p.~555]{Elkik}, so the image of $H$ in $S[X_1,\cdots X_k]/J$ is supported on the singular locus of the map $ \cU^{\nv}(\tld{z}, \leql,\nabla)[\frac{1}{(2h_\lambda)!}]\to X\times \bA^n$.
Hence, by Proposition \ref{prop:smooth_generic_fiber}, there exists an integer $r$ such that $v^r \in H+J$. 
We now base change the situation to $A$, and let $B=A[X_1,\cdots X_k]/J$, let $H_B$ be the base change of $H$, and  apply \cite[Lemme 1]{Elkik} to $B$ and $A$ to produce the integer $N>2r$ (note that the $k$ in loc.~cit.~is $0$ in our situation because we assumed $A$ is $t$-torsion free). We check that this choice of $N$ and $r$ works.

 Let $m\geq N>2r$ be as in the statement of the Proposition. The $f$ induces a tuple $\un{a}=(a_1,\cdots a_k)\in A^k$ such that $J(\un{a})\subset t^mA$. On the other hand we know $t^r\in H_B(\un{a})+J(\un{a})\subset H_B(\un{a})+t^mA$, and hence $t^r\in H_B(\un{a})$ because $A$ is $t$-adically complete and $m>r$. Thus \cite[Lemme 1]{Elkik} implies we can find a tuple $\un{\tld{a}}\in A^k$ lifting $\un{a}$ modulo $t^{m-r}$ such that $J(\un{\tld{a}})=0$. But this is exactly the data of the map
$\Spec A\to \Spec B$ that we want.
\end{proof}

\subsection{Equal characteristic and unibranch points}\label{subsec:equal_char_unibranch}
Throughout this section we fix $\lambda\in X_*(T^\vee)$ dominant, a field $k$ and a point $s\in \bA^n(k)$ corresponding to a tuple $\bf{a}\in k^n$.
We will assume that $(2h_{\lambda})!$ is invertible in $k$.

We have the base change $\cM^{\nv}(\leql,\nabla_{\bf{a}})\defeq \cM_X^{\nv}(\leql,\nabla)\times_{\bA^n} s$, and define $\cM(\lambda,\nabla_{\bf{a}})$ to be the Zariski closure of {$\big(\cM_X(\lambda,\nabla) \times_{\bA^n} s\big) \times _{X_k} X^0_k\cong (P_\lambda \backslash \GL_n)_{X^0_k}$ in $\cM^{\nv}(\leql,\nabla_{\bf{a}})$.} In particular, the natural map $\cM(\lambda,\nabla_{\bf{a}})\to X_k=\bA^1_k$ is flat. 

\begin{rmk} There is a natural map $\cM(\lambda,\nabla_{\bf{a}}) \to \cM_{X}(\lambda,\nabla)\times_{ \bA^n} s$ which is an isomorphism over $X^0_k$, and identifies $\cM(\lambda,\nabla_{\bf{a}})$ as the Zariski closure of $\cM_{X^0}(\lambda,\nabla)\times_{\bA^n} s$ in $\cM_{X}(\lambda,\nabla)\times_{\bA^n} s$. 
It is unclear whether it is always an isomorphism, but we will see in Proposition \ref{prop:spreading} that it is an isomorphism for generic choices of $\bf{a}$.
\end{rmk}
We recall the following definition (cf.~\cite[\href{https://stacks.math.columbia.edu/tag/06DT}{Tag 06DT}]{stacks-project}):
\begin{defn}
A point $y\in Y$ of a scheme is called \emph{unibranch} if the normalization of the local ring $(\cO_{Y,y})_{\mathrm{red}}$ is local.
\end{defn}

If $Y$ is an integral scheme, we will write $Y^{\nm}\ra Y$ for the normalization of $Y$. 

\begin{rmk}\label{rmk:unibranch_vs_analytic}
\begin{enumerate}
\item \label{item:onepoint} (\cite[\href{https://stacks.math.columbia.edu/tag/0C3B}{Tag 0C3B}]{stacks-project}) If $Y$ has a finite number of irreducible components and the normalization map $Y^{\nm}\to Y$ is finite (e.g.~when $Y$ is excellent), then the following are equivalent:
\begin{enumerate}[(i)]
\item $y$ is unibranch;
\item the (set-theoretic) fiber above $y$ of the normalization is a single point; and
\item the fiber above $y$ of the normalization is connected.
\end{enumerate}
\item \label{it:unibranch_vs_analytic:2}(\cite[\href{https://stacks.math.columbia.edu/tag/0C2E}{Tag 0C2E}]{stacks-project}) When $Y$ is Noetherian and excellent, $Y$ is unibranch at $y$ if and only if $Y$ is analytically irreducible at $y$, i.e.~the completed local ring $\cO_{Y,y}^\wedge$ is a domain.
\end{enumerate}
\end{rmk}
We now fix $\tld{z}=wt_\nu \in \tld{W}^\vee$.
 Recall we have a subfunctor $\cU(\tld{z})$ of $L\cG$  defined before Lemma \ref{lem:open_chart_mono} which defines an open subfunctor of $\Gr_{\cG,X}$.
We let $\cU(\tld{z},\lambda,\nabla_{\bf{a}})=\cU(\tld{z}) \times _{\Gr_{\cG,X}} \cM(\lambda,\nabla_{\bf{a}})$, which is a Zariski open (possible empty) subscheme of $\cM(\lambda,\nabla_{\bf{a}})$.   

Recall (Definition \ref{defn:section}) that for each $\tld{z} = w t_{\nu} \in \tld{W}^{\vee}$, we have an associated $\Z$-point of $\cU(\tld{z}) \times_X \{0\} \subset  \Gr_{\cG,X} \times_X \{0\}$ given by $w v^{\nu} \in \GL_n(\Z(\!(v)\!))$ which we denote by $\tld{z}$.  For any field $k$, let $\tld{z}_k$ denote the base change to $k$.

The following is the main result of this section:
\begin{prop} \label{prop:unibranch_equal_char} Assume that $\tld{z}_k \in \cM(\lambda,\nabla_{\bf{a}})(k)$. Then for any integer $e>0$, the base change $\cM(\lambda,\nabla_{\bf{a}})\times_{X_k, v\mapsto v^e} X_k$ is unibranch at $\tld{z}_k$. Furthermore, the preimage of $\cU(\tld{z})$ in $(\cM(\lambda,\nabla_{\bf{a}})\times_{X_k, v\mapsto v^e} X_k)^{\nm}\times_X \{0\}$ is connected.
\end{prop}

The Proposition implies the following crucial Corollary{, which underlies the unibranch property (at special points) of the local models we will be interested in (cf.~Theorem \ref{thm:model_unibranch}):}
\begin{cor}\label{cor:unibranch_universal} Let $e>0$ be an integer and $U\subset\bA^n$. Let $\cM_U\defeq \cM_X(\lambda,\nabla)\times_{\bA^n} U\to X\times U$, $\cM_{U,e} \defeq \cM_U\times_{X,v\mapsto v^e} X$ and let $\tld{\cM}_{U,e}$ be the normalization of $\cM_{U,e}$ in $\cM_{U,e}\times _X X^0$. Assume that $\cM_U\to X\times U$ and $\tld{\cM}_{U,e}\to X\times U$ are flat. Suppose we have a geometric point $x$ of $(\cM_U)\times_X \{0\}(k)$ which lies in a section $\tld{z} \in\tld{W}^{\vee}$, with image $0\times s \in (X\times U)(k)$. The the preimage of $x$ in $\tld{\cM}_{U,e}$ is supported at a single point. Furthermore, the preimage of $\cU(\tld{z})$ in $\tld{\cM}_{U,e}\times_{X\times U} (\{0\}\times s)$ is connected.
\end{cor}
\begin{proof} The point $x$ gives rise to a point $s\in U(k)\subset \bA^n(k)$ corresponding to a tuple $\bf{a}\in k^n$. The flatness hypotheses imply that the natural map $\cM(\lambda,\nabla_{\bf{a}}) \to \cM_U\times_{ U} s$ is an isomorphism, that $\cM_{U,e}\times_U s= \cM(\lambda,\nabla_{\bf{a}})\times _{X_k, v\mapsto v^e} X_k$, and that $\tld{\cM}_{U,e}\times_U s\to \cM_{U,e}\times_U s$ is a finite birational map. It follows that $\tld{\cM}_{U,e}\times_U s$ is surjected on by the normalization of $\cM(\lambda,\nabla_{\bf{a}})\times_{X_k,v\mapsto v^e} X_k$, and we conclude by Proposition \ref{prop:unibranch_equal_char} and Remark \ref{rmk:unibranch_vs_analytic}(\ref{item:onepoint}). 
\end{proof}

The rest of this section is devoted to the proof of Proposition \ref{prop:unibranch_equal_char}. 
We first recall some torus actions on $\Gr_{\cG,X}$. 
Let $T^{\vee,\ext}=T^\vee\times \bG_m$. We let $T^{\vee,\ext}$ act on $X=\bA^1$ by letting $T^\vee$ act trivially, and the  $\bG_m$ factor act via scaling the coordinate $v$.

Given $r\in R^\times$ we have a canonical isomorphism $R(\!(v-t)\!)\risom R(\!(v-rt)\!)$ of $R$-algebras given by the change of variable $v\mapsto r^{-1}v$. This induces an action $\bG_m \times_\bZ L\cG \to L\cG$ of $\bG_m$ on $L\cG$ (over $\bZ$) which is equivariant with respect to the scaling action of $\bG_m$ on $\bA^1$.
It commutes with the right-translation action of $T^\vee$ on $L\cG$, and thus we obtain an $T^{\vee,\ext}$ action on $L\cG$ which is equivariant for the map $L\cG\to \bA^1$. It is clear that this action preserves the subgroup $L^+\cG$, and thus we get an action of $T^{\vee,\ext}$ on $\Gr_{\cG,X}$ which is equivariant for the action $T^{\vee,\ext}$ on $\bA^1$. This action preserves $S_X(\lambda)$.

\begin{lemma}\label{lem:torus_action_open_chart} Let $\tld{z}=wt_\nu \in \tld{W}^\vee$. Define an action $T^\vee \times \bG_m \times L\cG \to L\cG$ of $T^{\vee,\ext}$ on $L\cG$ given by the formula
\[(D,r)A(v)\defeq \Ad(w)(D^{-1}) A(r^{-1}v)r^{\nu}D,\]
where $D\in T^\vee(R)$ is a diagonal matrix, $r\in R^\times$, $A(v)\in \GL_n(R(\!(v-t)\!))$ and $A(r^{-1}(v))\in \GL_n(R(\!(v-rt)\!))$ is obtained from $A(v)$ by the change of variable $v\mapsto r^{-1}v$. 
Then this action preserves $\cU(\tld{z})$, $\cU(\tld{z})^{\det,\leq h}$ (for any $h$), and the natural map $\cU(\tld{z})\into \Gr_{\cG,X}$ is $T^{\vee,\ext}$-equivariant.
\end{lemma}
\begin{proof} We have 
\[  \Ad(w)(D^{-1}) A(r^{-1}v) r^\nu D(v-rt)^{-\nu}w^{-1} = \Ad(w)(D^{-1}) A(r^{-1}v) r^\nu D(v-rt)^{-\nu}w^{-1} \]
\[ =\Ad(w)(D^{-1}) A(r^{-1}v) (r^{-1}v-t)^{\nu}w^{-1} \Ad(w)(D)\]
and 
\[ \Ad(w)(D^{-1}) A(r^{-1}v) r^\nu D(v-rt)^{-\nu}=\Ad(w)(D^{-1}) A(r^{-1}v)(r^{-1}v-t)^\nu D\]
The result now follows from the definitions since the first condition defining $\cU(\tld{z})^{\det, \leq h}$ is stable under $T^{\vee}$-conjugation and the second under both right and left multiplication by $T^{\vee}$, and the change of variable $v\mapsto r^{-1}v$ induces an isomorphism $R[\frac{1}{v-t}]\cong R[\frac{1}{v-rt}]$ which sends $\frac{v}{v-t}$ to $\frac{v}{v-rt}$.
\end{proof}
\begin{lemma} \label{lem:contracting_torus_action}There exists a one-parameter subgroup $\bG_m \to T^{\vee,\ext}$ such that for any $h$, the induced action via Lemma \ref{lem:torus_action_open_chart} on $\cU(\tld{z})^{\det,\leq h}$ satisfies the following properties:
\begin{itemize}
\item It is contracting, i.e.~it extends to an action $\bA^1 \times \, \cU(\tld{z})^{\det,\leq h}\to \cU(\tld{z})^{\det,\leq h}$ of the multiplicative monoid $\bA^1$.
\item If $ \cU(\tld{z})^{\det,\leq h}$ is non-empty, the fixed-point subscheme of the action is given by the section $\tld{z}:\Spec \Z \into \cU(\tld{z})^{\det,\leq h}\times _X \{0\}$
\end{itemize}
\end{lemma}
\begin{proof} We choose $\mu\in X_*(T^\vee)$ regular dominant, and choose an integer $N>h_{\mu} =\max_{\alpha^\vee}\{\langle \mu,\alpha^\vee\rangle\}$. 
We claim that the one parameter subgroup $r\mapsto (\Ad(w^{-1})(r^{\mu}),r^N)$  does the job. 
It suffices to verify the statement for Noetherian $\bZ[v]$-algebras.  
For a Noetherian $\bZ[v]$-algebra $R$ (sending $v$ to $t\in R$), recall the explicit description of $\cU_X(\tld{z})^{\det,\leq h}(R)$ from Proposition \ref{prop:explicit_affine_chart}.  Using that description and Lemma \ref{lem:torus_action_open_chart}, we see that for any $A\in \cU_X(\tld{z})^{\det,\leq h}(R)$,  
the action of an element $r\in R^\times$ sends $A$ to $A'$ where $A'$ is given by
\begin{align*}
A'_{ij}&=r^{\mu_i}(r^{-N}v)^{\delta_{i>j}}\Bigg(\sum_{k=-h}^{ \nu_j-\delta_{i>j}-\delta_{i<w(j)} }c_{ij,k}(r^{-N}v-t)^k\Bigg)r^{N\nu_j}r^{-\mu_{w(j)}}\\
&=v^{\delta_{i>j}}\Bigg(\sum_{k=-h}^{ \nu_j-\delta_{i>j}-\delta_{i<w(j)} } r^{N(\nu_j-\delta_{i>j}-k)+\mu_i-\mu_{w(j)}} c_{ij,k}(v-r^N t)^k\Bigg).
\end{align*}
If $k<\nu_j-\delta_{i>j}$, then $N(\nu_j-\delta_{i>j}-k)+\mu_i-\mu_{w(j)}>0$ since $N>h_\mu$.
If $k=\nu_j-\delta_{i>j}$, then necessarily $\delta_{i<w(j)}=0$. 
We have two subcases: If $i=w(j)$, then $N(\nu_j-\delta_{i>j}-k)+\mu_i-\mu_{w(j)}=0$ and $c_{ij,k}=1$.
Otherwise, $i>w(j)$, and $N(\nu_j-\delta_{i>j}-k)+\mu_i-\mu_{w(j)}=\mu_i-\mu_{w(j)}>0$, since $\mu$ was chosen to be regular dominant.
Thus we see that the coordinates $c_{ij,k}$ (for $i\neq w(j)$ are homogenous for our $\bG_m$-action with positive weight, hence the $\bG_m$-action extends to an action of $\bA^1$, and that the fixed point scheme $\cU_X(\tld{z})^{\det,\leq h}$ is exactly given by the section $\tld{z}$. 
\end{proof}

\begin{lemma}\label{lem:attractor} Let $k$ be an algebraically closed field and let $M$ be an irreducible variety over $k$, Suppose there is an action of the multiplicative monoid $\bA^1_k$ on $M$ over $k$ with a unique fixed point $x\in X(k)$. Then $M$ is unibranch at $x$. In particular, the completed local ring $\cO_{M,x}^{\wedge}$ of $M$ at $x$ is a domain. %
\end{lemma}
\begin{proof} Let $\pi:M^{\nm}\to M$ be the normalization map, so $\pi$ is finite.
Since $\bA^1_k \times M^{\nm}$ is normal, the action of $\bA^1_k$ on $M$ extends to an action of $\bA^1_k$ on $M^{\nm}$, by the universal property of the normalization. In particular, we get an induced action of $\bG_m$ on $M^{\nm}$.

We claim that the fixed-point scheme $M^{\nm,\bG_m}$ has underlying reduced scheme $(M^{\nm,\bG_m})^\red=(\pi^{-1}(x))^\red$. Since $\pi((M^{\nm,\bG_m})^{\red})\subset (M^{\bG_m})^{\red}=x$, we have $(M^{\nm,\bG_m})^\red\subset(\pi^{-1}(x))^\red$. On the other hand, $\pi^{-1}(x)$ is a finite scheme with a $\bG_m$-action, hence $(\pi^{-1}(x))^{\red}$ consists of $\bG_m$-fixed points, so $(\pi^{-1}(x))^{\red}\subset (M^{\nm,\bG_m})^\red$.

Now the action map $\bA^1_k\times M^{\nm} \to M^{\nm}$ induces a surjective map $M^{\nm} \to M^{\nm,\bG_m}$, given by $m\mapsto 0\cdot m$. Since $M^{\nm}$ is irreducible, we conclude that $(\pi^{-1}(x))^{\red}=\tld{x}$ is a single point. Hence, $M$ is unibranch at $x$ by  Remark \ref{rmk:unibranch_vs_analytic}(\ref{item:onepoint}).
The last assertion follows from Remark \ref{rmk:unibranch_vs_analytic}(\ref{it:unibranch_vs_analytic:2}).

\end{proof}

\begin{proof}[Proof Proposition \ref{prop:unibranch_equal_char}] When $e=1$, we can directly apply Lemma \ref{lem:attractor} to $\cU(\tld{z},\lambda,\nabla_{\bf{a}})$, with the $\bG_m$-action obtained from Lemma \ref{lem:torus_action_open_chart} via base change. Note that $\cU(\tld{z},\lambda,\nabla_{\bf{a}})$ is irreducible since $\cM(\lambda, \nabla_{\bf{a}})$ is by Corollary \ref{cor:monodromy_before_reduced}. Furthermore, this contracting $\bG_m$ lifts to a contracting $\bG_m$-action on the normalization $\cU(\tld{z},\lambda,\nabla_{\bf{a}})^{nm}$. Now the fiber of this normalization above $0\in X_k$ has a contracting $\bG_m$-action with a unique fixed point (namely, the pre-image of 
$\tld{z}_k$), and hence is connected. 

For general $e>0$, by composing the previous action with the $e$-th power map $\bG_m\to \bG_m$, we can construct a contracting $\bG_m$-action on $\cU(\tld{z},\lambda,\nabla_{\bf{a}})$ which is equivariant for the $e$-th power of the scaling action on $X_k$. This allows us to define a $\bG_m$-action on $\cU(\tld{z},\lambda,\nabla_{\bf{a}})\times_{X_k,v\mapsto v^e} X_k$ which is contracting to $\tld{z}_k$. We can now repeat the same argument as above.
\end{proof}

\subsection{Spreading out normality}
\label{sec:Spr:Out}

We now return to the universal setting. Recall that $X=\bA^1=\Spec \bZ[v]$, with a chosen coordinate $v$. We thus get a zero section $0:\Spec \bZ \to X$ given by $v\mapsto 0$, and $X^0=\Spec \bZ[v,\frac{1}{v}]$. We will abusively think of $v$ as a global function on any $X$-scheme.
We study the following setup:
\begin{setup}\label{setup:spreading} We have an integral finite type $\bZ$-scheme $S$, and a finitely presented map $M\to X \times S$. We assume that the generic point of $S$ has characteristic $0$.  We also assume the following properties:
\begin{itemize}
\item The base changed family $M^0=M\times _X X^{0}\to X^0\times S$ over $X^0$ is smooth.%
\item $M$ is the Zariski closure of $M^0$. In particular $v\in \cO(M)$ is a regular element. 
\item $M$ is normal.
\end{itemize}
Given this setup, we will denote $M_0=M\times_{X\times S,0} S$, the fiber of $M$ above the zero section $0:\Spec  \bZ\to X$. 
\end{setup}

We want to understand the base change of this situation to a complete discrete valuation ring $R$, via a map $f: \Spec R \to X\times S$ which induces a map $\bZ[v] \ra R$ sending $v$ to a uniformizer of $R$. In general, the base change $M_R\defeq M\times _{X\times S, f} \Spec R$ may neither be flat over $\Spec R$, nor be normal. However, the following Proposition will guarantee that both properties will hold for ``generic'' choices of $f$:
\begin{prop} \label{prop:spreading}In Setup \ref{setup:spreading}, there exists a non-empty open subscheme $U\subset S$ such that if $R$ is a complete discrete valuation ring and $f:\Spec R \to X\times S$ factor through $X\times U$,and such that $v$ is sent to a uniformizer of $R$, then the base change $M_R\to \Spec R$ is flat, and $M_R$ is normal.
\begin{rmk} The hypothesis that $v$ is sent to a uniformizer of $R$ is necessary. For example, let $M=\Spec \bZ[x,v]/(x^2-v)\to \Spec \bZ[v]$, $S=\Spec \bZ$, and let $f:\cO\to \Spec\bZ[v]$ be the map sending $v$ to $\varpi^2$ where $\varpi$ uniformizes $R$. Then the base change $M_R=\Spec R[x]/(x^2-\varpi^2)$ is not normal.
\end{rmk}

\begin{lemma}\label{lem:flat_poly} Let $B$ be an $A[v]$-algebra. Assume that $v$ is a regular element in $B$, $B[\frac{1}{v}]$ is flat over $A[v,\frac{1}{v}]$, and $B/v$ is flat over $A[v]/v=A$. Then $B$ is flat over $A[v]$.
\end{lemma}
\begin{proof} Let $x \in \Spec B$, and let $y$ be the image of $x$ in $\Spec A[v]$. We need to show $B_x$ is flat over $A[v]_y$.
If $x\in \Spec B[\frac{1}{v}]$, this is part of our hypothesis.
If $x\in \Spec B/v$, our hypotheses imply $B_x/v$ is flat over $A[v]_y/v$, and that $\mathrm{Tor}_1^{A[v]_y}(A[v]_y/v,B_x)=0$. We conclude by the local criteria of flatness (\cite[Lemma 10.98.10, \href{https://stacks.math.columbia.edu/tag/00MD}{Tag 00MD}]{stacks-project}).
\end{proof}

\begin{lemma} \label{lem:generic_flatness} Assume Setup \ref{setup:spreading}. Then there is a non-empty open subscheme $U\subset S$ such that the base change $M\times_{X\times S} (X\times U)\to X\times U$ is flat. 
 \end{lemma}
\begin{proof}
We already observed that the coordinate $v\in \cO(X)$ is regular in $M$. Let $M_0$ be the $S$-scheme $M\times_{X} \{ 0 \}$.
Since $S$ is integral, by generic flatness (\cite[\href{https://stacks.math.columbia.edu/tag/0529}{Tag 0529}]{stacks-project}), there is a non-empty open $U\subset S$ such that $M_0\times_ S U$ is flat over $U$. On the other hand $M^0\to X^0\times S$ is smooth, hence flat. We conclude by Lemma \ref{lem:flat_poly}.
\end{proof}

\end{prop}
\begin{rmk} \label{rmk:generic_flatness}  The above proof actually shows that Lemma \ref{lem:generic_flatness} holds under much less restrictive conditions than Setup \ref{setup:spreading}: In fact one only needs a finitely presented map $M\to X\times S$ such that $S$ is integral, $v$ is regular in $M$ { and $M^0\to X^0\times S$ is flat.}
\end{rmk}

\begin{lemma} \label{lem:S_2} Assume Setup \ref{setup:spreading}, and furthermore assume that $M\to X\times S$ is flat. Then there is a non-empty open subscheme $U\subset S$ such that the map $M_0\times_S U\to S$ has geometrically $S_1$ fibers.
\end{lemma}
\begin{proof} Our hypotheses imply that the fiber over the generic point of $S$ of the composition $M\to X\times S\to S$ is geometrically normal. By \cite[Proposition 9.9.4]{EGA_IV_3}, there exists a non-empty open subscheme $U\subset S$ such that $M\times_S U\to S$ has geometrically normal fibers. Hence for each geometric point $u$ of $U$, $M\times_S u$ is normal, and in particular $S_2$. Since the fiber $M_0\times_S u$ is the zero subscheme of a regular element $v\in \cO(M\times_S u)$, it is $S_1$.
\end{proof}
\begin{lemma} \label{lem:R_1} Assume Setup \ref{setup:spreading}, and furthermore that $M\to X\times S$ is flat. Then there is a non-empty \'{e}tale $S$-scheme $U\to S$ such that for any discrete valuation ring $R$ and a map $f:\Spec R \to U$ which sends $v$ to a uniformizer of $R$, the base change $M_R=M\times_{X\times S,f} (X\times U)$ is  $R_1$.
\end{lemma}
\begin{proof} It suffices to treat the case $M$ is affine. As in the proof of Lemma \ref{lem:S_2}, by shrinking $S$, we may assume $M\to X\times S$ has geometrically normal fibers.

Let $\eta=\Spec \kappa$ be the generic point of $S$. Then there exists a finite (and necessarily separable) extension $\kappa'$ of $\kappa$ such that all irreducible components (with the reduced scheme structure) of $(M_0)_{\kappa'}\defeq M_0\times_S \Spec \kappa'$ are geometrically integral.
The map $(M_0)_{\kappa'}\to \Spec \kappa'$ extends to a map $(M_0)_V\defeq (M_0)\times_S V \to V$ where $V\to S$ is an irreducible affine \'{e}tale $S$-scheme. By \cite[\href{https://stacks.math.columbia.edu/tag/0553}{Tag 0553}]{stacks-project}, we may replace $V$ by an open subset so that for any irreducible component $Z$ of $(M_0)_V$, the map $Z\to V$ has geometrically integral fibers. This implies implies that for any discrete valuation ring $R$ and a map $f:\Spec R \to X\times V$ sending $v$ to a uniformizer $\varpi$ of $R$, all the irreducible components of the special fiber of the  base change $M_R\to \Spec R$  are obtained by base change from the irreducible components of $(M_0)_V$. (This is because the base change of $Z\to X\times V$ to $\Spec R$ will have geometrically integral fibers over $\Spec R/\varpi$. It is here that we use the assumption that $v$ is sent to a uniformizer of $R$.)

We now denote $\Spec B= M\times_{X\times S} (X\times V)$. Then $B$ is normal, and $(M_0)_V=\Spec B/vB$. Let $\cP$ be the finite set of minimal primes of $B/vB$, which we also view as the height $1$ primes of $B$ containing $v$.

Let $\fp\in \cP$, and we fix a finite set $\{y_{\fp,i}\}_I$ of generators for $\fp$. Since $B$ is regular in codimension $1$, the localizations $B_{\fp}$ is a discrete valuation ring. Hence there is an element $x_\fp\in \fp \subset B$ which generates $\fp/\fp^2 \otimes_B B_\fp$ as a $B_\fp$-module. This implies that the module $\fp/(x_\fp+\fp^2)$ as a module over $B/\fp$ is supported on a proper closed subset of $\Spec B/\fp$. Thus there is $f_\fp \in B$ with $f_\fp \notin \fp$, and $a_{\fp, i}\in B$ for each $i\in I$ such that
\begin{equation}\label{eq:generic_uniformizer}
f_{\fp}\ y_{\fp, i}\,\equiv\, a_{\fp,i} \ x_{\fp} \text{ mod } \fp^2
\end{equation}
We remark that these relations persists on any base change of $B$. Now consider the subscheme $V(f_\fp)=\Spec B/(f_\fp+\fp)\into \Spec B/\fp\to V$. The locus of points in $V$ where the fiber of $V(f_\fp)$ has the same dimension as the fiber of $\Spec B/\fp$ is constructible, and does not contain the generic point of $V$. Hence there is an affine Zariski open $V_\fp\subset V$ over which the fiber of $V(f_{\fp})$ has dimension strictly less than the dimension of the fiber of $\Spec(B/\fp)$.

We finally claim that $U=\bigcap_\cP V_\fp$ satisfies the conclusion of the lemma. Indeed, let $R$ be a discrete valuation ring and $f:\Spec R \to X\times V$ sending $v$ to a uniformizer $\varpi$ of $R$. Let $M_R=\Spec B'$. We already observed that the minimal primes of $B'/v=B'/\varpi$ are $\fp B'/v$ for $\fp\in \cP$. But now equation (\ref{eq:generic_uniformizer}) holds in $B'$, and furthermore our arrangement guarantees that $f_\fp\notin \fp B'$. This implies $\fp B'_{\fp}$ is generated by $x_\fp$, hence $B'_\fp$ is a discrete valuation ring.
\end{proof}

\begin{proof}[Proof of Proposition \ref{prop:spreading}] We first pick a Zariski open $U_1\subset S$ for which conclusion of Lemma \ref{lem:generic_flatness} holds. We then pick a Zariski open $U_2\subset U_1$ for which the conclusion of Lemma \ref{lem:S_2} holds. We let $U_3\to U_2$ be the \'{e}tale map for which the conclusion of Lemma \ref{lem:R_1} holds. Thus the base change 
$M\times_{X\times S} U_3\to U_3$ satisfies the conclusions of Lemmas \ref{lem:generic_flatness}, \ref{lem:S_2}, \ref{lem:R_1}.
We let $U$ be the image of $U_3$ in $S$, then $U$ is an open subscheme of $S$. We claim that this choice of $U$ satisfies the conclusion of the Proposition.

Indeed, let $R$ be a complete discrete valuation ring and $f:\Spec R \to U$ be a map that such that $v$ is sent to a uniformizer of $R$.
Since $R$ is complete, we can lift $f$ to a map $\Spec R \to U_3$, which we will abusively call $f$ again. Then the base change $M_R=M$ is a base change of $M\times_{X\times S} U_3\to X\times U_3$ along $f:\Spec R \to U_3$. Hence $M_R\to \Spec R$ is flat and is $R_1$ at the generic points of its special fiber. Since the generic fiber of $M_R$ is smooth (being a base change of $M^0\to X^0\times S$), $M_R$ is $R_1$. Furthermore, since the special fiber of $M_R$ is $S_1$ and $M_R$ is flat over $\Spec R$, $M_R$ is $S_2$. Thus $M_R$ is normal.
\end{proof}
\subsection{Sections}
\label{subsec:Sec}
\begin{prop} \label{prop:rational_points} Let $M\to X\times S$ be a flat finite type map of finite type $\bZ$-schemes, and $S$ is irreducible with characteristic $0$ generic point. Suppose we have a section $s:S\to M_0$. Then there exists a non-empty Zariski open subscheme $U\subset S$ and a closed subscheme $Z\into M\times_{X\times S} (X\times U)$ such that 
\begin{itemize}
\item $Z\to X\times U$ is flat and quasi-finite.
\item $Z$ contains the section $s|_U:U\to (M_0)\times_S U$
\end{itemize}
{(recall from \S \ref{sec:Spr:Out} that $M_0$ is the fiber of $M$ along the zero section $0:\Spec\Z\ra X$).}
\end{prop}
\begin{proof} Let $\Spec \kappa\to S$ be the generic point of $S$, and consider the base change $M_\kappa\to X_\kappa=\bA^1_{\kappa}$. The section $s$ induces a $\kappa$-point $s_\kappa$ of $(M_0)_{\kappa}$. Since $M_\kappa\to X_\kappa$ is flat, it is generizing, and we can find a point $x$ of $M_\kappa$ lying over the generic point of $X_\kappa$ whose closure contain $s_\kappa$. The closure of $\overline{x}$ in $M_\kappa$ is an irreducible curve in $M_\kappa$ which dominates $X_\kappa$, and is hence is flat and quasi-finite over $X_\kappa$. Now, $\overline{x}\subset M_\kappa$ extends to a closed subscheme $Z\subset M\times_{X\times S} (X\times V)$ for some non-empty Zariski open $V\subset S$. Note that $Z$ contains the generic point of the section $s$, hence also contains $s|_V:V\to (M_0)\times_S V$.
 
Now there is a non-empty Zariski open $W\subset X\times V$ containing $X_\kappa$, over which the map $Z\to X\times V$ is flat and quasi-finite. Then the image of $(X\times V)\setminus W$ in $V$ is constructible and does not contain the generic point of $V$, hence its complement contains a non-empty Zariski open $U\subset V \subset S$. 
Replacing $Z$ by $Z \cap (X \times U)$, $U$ and $Z$ satisfy the desired properties.
\end{proof}
\begin{cor} \label{cor:rational_points} Let $M\to X\times S$, $s:S\to M_0$ as in Proposition \ref{prop:rational_points}. Then there is an integer $e$ and a non-empty Zariski open $U\subset S$ depending on $M, S, s$ with the following property: for any complete discrete valuation ring $R$ and a map $f:\Spec R \to X\times U$ sending $v$ to an element with positive valuation in $R$, there exist a finite $DVR$ extension $R'$ of $R$ of degree $\leq e$, and an $R'$-point of $M_R$ lifting the point induced by $s$ in the special fiber of $M_R$.
\end{cor}
\begin{proof} We take $U$ and $Z\subset M\times_{X\times S} (X\times U)$ as in the conclusion of Proposition \ref{prop:rational_points}, and let $e$ be the maximal degree of a fiber of $Z\to X\times U$. Then the base change $Z_R\to \Spec R$ is flat, quasi-finite and contains the point induced by $s$. Since $R$ is complete, \cite[\href{https://stacks.math.columbia.edu/tag/04GE}{Tag 04GE}]{stacks-project} shows that $Z_R$ must contain a connected component $C$ which is finite flat over $R$ contains the closed point induced by $s$. The normalization of $C$ then breaks into a disjoint union of the spectrum of finite complete DVR extensions of $R$, whose degrees are $\leq e$. One of these components will have its closed point mapping into $s$.
\end{proof}
\subsection{Products}
\label{subsec:prod:UMLM}

Let $\cJ$ be a finite set.
Let $\lambda \in X_*(T^{\vee})^{\cJ} = X^*(T)^{\cJ}$ be dominant.   For $j \in \cJ$, $\lambda_j \in X^*(T^{\vee})$ will denote the $j$-th component.     

We define $\cM_{X,\cJ}(\lambda,\nabla) = \prod_{j\in \cJ} \cM_X(\lambda_j,\nabla)$, where the product means fiber product over $X$. We have $M_{X,\cJ}(\lambda,\nabla) \subset (\Gr_{\cG, X} \times \bA^n)^{\cJ}$.

Let $\tld{z} = (\tld{z}_j)_{j \in \cJ} \in \tld{W}^{\vee,\cJ}$.   As in Definition \ref{defn:section}, we have constant sections $\tld{z}_j: \bA^n\to (\Gr_{\cG,X}\times_X \{0\})\times \bA^n$, and these compile into a section $\tld{z}:  (\bA^n)^{\cJ} \to ((\Gr_{\cG, X}\times_X \{0\}) \times \bA^n)^{\cJ}$. We thus get an induced $\Z$-point on each fiber of $((\Gr_{\cG, X}\times_X \{0\}) \times \bA^n)^{\cJ} \to (\bA^n)^{\cJ}$, which we abusively still call $\tld{z}$. 

The following Theorem is the main result of this section:
\begin{thm}\label{thm:model_unibranch} %
Fix an integer $e>0$. There exists a Zariski open $U =U(\{\lambda_j\},e,n)\subset \bA^n$ which depends only on $e$, $n$ and the subset $\{\lambda_j\}\subset X_*(T^\vee)$, such that:
For any complete discrete valuation ring $R$ and any map $f:\Spec R \to X\times_{\bZ} \prod_{j\in J} U$ such that $v$ is sent to an element of valuation at most $e$, the base change $M_{X,\cJ}(\lambda,\nabla)_R\to \Spec R$ is flat and unibranch at
any point of the special fiber of $M_{X,\cJ}(\lambda,\nabla)_R$ which lies in a section $\tld{z}$. Furthermore, letting $U(\tld{z},\lambda,\nabla)_R=\cU(\tld{z})\cap  M_{X,\cJ}(\lambda,\nabla)_R$, the $\fm_R$-adic completion of $\cO(U(\tld{z},\lambda,\nabla)_R)$ is a domain.
\end{thm}
\begin{proof} To simplify notation, in this proof we will set $M_j=\cM_X(\lambda_j,\nabla)$, and $M_{\cJ}=\prod_{j\in\cJ} M_j =\cM_{X,\cJ}(\lambda,\nabla)$. For $\tld{s}\in W^{\vee}$, we also abbreviate $U_j(\tld{s})=M_j\cap \cU(\tld{s})$.

We first observe that $(M_j)_0 = M_j \times_X \{0\}$ can only meet a section $\tld{s}$ which occurs in the subset $\{wt_\nu\in \tld{W}^{\vee}|\, \nu\in \mathrm{Conv}(\lambda_j)\}$. 
Note that the latter is a finite set depending only on $\lambda_j$. %

Let $\eta$ be the generic point of $\bA^n$. For each $j\in \cJ$, let $Fix_j$ to be the set of $\tld{s} \in \tld{W}^{\vee}$ such that the section $\tld{s}$ meets $(M_j)_0 \times_{\bA^n} \eta$.

We can now find a non-empty Zariski open $S_j\subset \bA^n$ depending only on $\lambda_j$ such that 
\begin{itemize} 
\item $M_j\times_{X\times S} (X\times S_j)\to X\times S_j$ is flat 
\item $(M_j\times_{X\times S} (X\times S_j))_0$ contains the restriction of the sections $(\tld{s})|_{S_j}$ for $\tld{s} \in Fix_j$, and is disjoint from any section $\tld{s}\notin Fix_j$.
\item $(2h_{\lambda_j})!$ is invertible in $S_j$.
\end{itemize}
Indeed the first item can be arranged by Remark \ref{rmk:generic_flatness}, and the second item can be arranged by the standard constructibility argument: { indeed for any $\tld{s}\in \{wt_\nu\in \tld{W}^{\vee}|\, \nu\in \mathrm{Conv}(\lambda_j)\}\setminus Fix_j$, the image of $\tld{s}\cap (M_j)_0$ under the projection to $S$ is a constructible set not containing $\eta$, hence its complement must contain a non-empty open subset of $S$.}
We let $S=\cap_{j\in \cJ} S_j$, so $S$ depends only on the set $\{\lambda_j\}$. For the rest of the proof we replace $M_j$ by its restriction $M_j|_{S}$.

By construction, for each $j\in \cJ$, $M_j\to X\times S$ is flat.
Applying Corollary \ref{cor:rational_points} to this family and the sections $\tld{s}$ with $\tld{s}\in Fix_j$, there is a Zariski open $V\subset S$ and an integer $\tld{e}$ such that: For any complete discrete valuation ring $R$ and a map $f:\Spec R \to X\times S$ sending $v$ to an element with positive valuation of $R$, there is a finite extension $R'/R$ of degree $\leq \tld{e}$ such that $(M_j)_R$ has an $R'$-point for all $j \in \cJ$. 
Since the data we used to apply Corollary \ref{cor:rational_points} depended only on $\{\lambda_j\mid j\in \cJ\}$, so does $V$ and $\tld{e}$.

Now for any integer $1\leq l\leq e$ and $j\in \cJ$, we let $\tld{M}_{j,l\tld{e}!}$ be the normalization of ${M_{j,l\tld{e}!} \defeq} M_j\times_{X\times S, v\mapsto v^{l\tld{e}!}} (X\times S)$. 
By Proposition \ref{prop:smooth_generic_fiber}, each $\tld{M}_{j,k}\to X\times S$ satisfies Setup \ref{setup:spreading}. We now let $U\subset V \subset S$ be the Zariski open which satisfies the conclusion of Proposition \ref{prop:spreading} for all the $\tld{M}_{j,k}$, and furthermore that $e!\tld{e}!$ is invertible in $U$. Clearly $U$ depends only on $\{\lambda_j\mid j\in \cJ\}$ and $e$.

We claim that the $U$ thus constructed satisfies the conclusion of the theorem. Let $R$ be a complete DVR and $f:\Spec R \to X\times _\bZ \prod_{j\in \cJ} U$ such that $v$ is sent to an element $a$ with valuation $l\in [1,e]$. Since all our schemes are excellent and Noetherian, being unibranch at a point is equivalent to being analytically irreducible. Therefore, it suffices to establish the unibranch property after making an unramified extension of $R$. Thus, we may and do assume that $R$ has separably closed residue field. Then for any integer $m$ invertible in the residue field of $R$, there is a unique extension of $R$ of degree $m$, namely the extension obtained by adjoining the $m$-th root of any uniformizer of $R$. Let $R'$ be the unique extension of $R$ of degree $\tld{e}!$, so $R'$ contains all extensions of degree $\leq \tld{e}$ of $R$. We may choose a uniformizer $\varpi$ of $R'$ so that $\varpi^{l\tld{e}!}=a$.
By construction of $V$, for each $j\in \cJ$ and $\tld{s}\in Fix_j$, $(M_j)_R$ admits an $R'$-point lifting the point $\tld{s}$ in special fiber of $(M_j)_R$. 

Our choice of uniformizer $\varpi$ gives rise to a commutative diagram
\begin{equation}\xymatrix{   \tld{M}_{j,l\tld{e}!}\ar[d]& \\
  M_{j,l\tld{e}!} \ar[r]\ar[d]&M_j\ar[d] \\
X\times U \ar[r]^{v\mapsto v^{l\tld{e}!}} &X\times U \\
\Spec R' \ar[r]\ar[u]^{f'} &\Spec R \ar[u]^f\\ }
\end{equation}
Here the map $f'$ is induced from $f$ and the map $\Spec R' \to X$ sending $v$ to the uniformizer $\varpi$ of $R'$.
We note that the base change of $(M_j)_R$ to $R'$ is also the change of $M_{j,l\tld{e}!}\to X\times U$ along $f'$. The construction of $U$ implies that the base change $(\tld{M}_{j,l\tld{e}!})_{R'}$ is normal, and hence is the normalization of $(M_{j,l\tld{e}!})_{R'}=(M_j)_{R'}$. This implies that the preimage of any $U_j(\tld{s})_{R'}$ in $(\tld{M}_{j,l\tld{e}!})_{R'}$ is its normalization. By Corollary \ref{cor:unibranch_universal} (which applies by the construction of $U$), for each $\tld{s}$ occurring in the special fiber of $(M_j)_{R'}$, its pre-image in $(\tld{M}_{j,l\tld{e}!})_{R'}$ is supported at a point. This implies that $(M_j)_{R'}$ is unibranch at $\tld{s}$, and hence the completed local ring $\cO_{(M_j)_{R'},\tld{z}}^\wedge$ is a domain. Furthermore, Corollary \ref{cor:unibranch_universal} also shows that the preimage of the special fiber of $U_j(\tld{s})_{R'}$ in $(\tld{M}_{j,l\tld{e}!})_{R'}$ is connected. By Lemma \ref{lem:normal_domain_criteria} below, the $\varpi$-adic completion $\cO(U_j(\tld{s}))^{\wedge_{\varpi}}$ is a domain.

We now finish the proof. Let $\tld{z}\in \tld{W}^{\vee,\cJ}$ such that $\tld{z}$ occurs in the special fiber of $(M_{X,\cJ})_R$. Then for each $j\in \cJ$, the component $\tld{z}_j\in Fix_j$. The completed local ring of ($M_{X,\cJ})_{R'}$ at $\tld{z}$ is the completed tensor product over  $\underset{j, R'}{\widehat{\bigotimes}} \cO_{(M_j)_{R'},\tld{z}_j}^\wedge$, where the index $j$ runs through the set $\cJ$.
Now each factor $c$ is a complete local Noetherian domain, has an $R'$-point, and $\cO_{(M_j)_{R'},\tld{z}_j}^\wedge[\frac{1}{\varpi}]$ is regular (since the generic fiber of $(M_j)_R$ is smooth). By \cite[Proposition 2.2]{KW_Serre_2} (which was stated for finite extensions of $\bZ_p$, but the proof works for general complete DVRs), the completed tensor product is also a domain. Since the completed local ring of $(M_{X,\cJ})_R$ at $\tld{z}$ embeds into the completed local ring of $(M_{X,\cJ})_{R'}$ at $\tld{z}$, the former must also be a domain. Hence $(M_{X,\cJ})_R$ is analytically irreducible at $\tld{z}$, and so is unibranch at $\tld{z}$ by Remark \ref{rmk:unibranch_vs_analytic}(\ref{it:unibranch_vs_analytic:2}).

Finally, we show the $\varpi$-adic completion of $\cO(U(\tld{z},\lambda,\nabla)_{R'})=\underset{j\quad{}}{\bigotimes_{R'}} \cO(U_j(\tld{z_j})_{R'})$ is a domain. To do this, instead of invoking \cite[Proposition 2.2]{KW_Serre_2}, we use \cite[Lemma A.1.1]{BLGGT}: since each $\cO(U_j(\tld{z_j})_{R'})^{\wedge_\varpi}[\frac{1}{\varpi}]$ is a regular affinoid domain which admits a rational point over $R'[\frac{1}{\varpi}]$, they are geometrically connected. Thus the completed tensor product $\widehat{\bigotimes}_j \cO(U_j(\tld{z_j})_{R'})[\frac{1}{\varpi}]$ is geometrically connected. Since it is also regular, it is a domain. We conclude as before.
\end{proof}

\begin{lemma}\label{lem:normal_domain_criteria} Let $R$ be a complete DVR with uniformizer $\varpi$. Let $A$ be a finite type flat $R$-algebra, and assume $A[\frac{1}{\varpi}]$ is a regular domain. Furthermore, assume that the special fiber $\Spec A^{\nm}/\varpi$ of the normalization of $\Spec A$ is connected. Then the $\varpi$-adic completion $A^{\wedge_\varpi}$ of $A$ is a domain.
\end{lemma}
\begin{proof} Since $A$ is excellent, $A^{\nm}$ is excellent and finite over $A$ {(by \cite[\href{https://stacks.math.columbia.edu/tag/07QV}{Lemma 07QV}]{stacks-project}, \cite[\href{https://stacks.math.columbia.edu/tag/035S}{Lemma 035S}]{stacks-project})}, and we have an inclusion $A^{\wedge_\varpi}\subset A^{\nm,\wedge_\varpi}$ of $\varpi$-adic completions. It thus suffices to show that $B\defeq A^{\nm,\wedge_{\varpi}}$ is a domain. 

Now our hypotheses implies that $B$ is $R$-flat, excellent {(\cite[Main Theorem 2]{kurano-shimomoto}), normal (\cite[\href{https://stacks.math.columbia.edu/tag/0C22}{Lemma 0C22}]{stacks-project}), and $B[\frac{1}{\varpi}]$ is regular (\cite[\href{https://stacks.math.columbia.edu/tag/033A}{Lemma 033A}]{stacks-project})}. Thus if $B$ is not a domain, $\Spec B[\frac{1}{\varpi}]$ must be disconnected, and hence there is a non-trivial idempotent $e\in B[\frac{1}{\varpi}]$. But $B$ is normal, hence $e\in B$. Furthermore, $e\notin \varpi B$, since if $e\in \varpi B$, then $e=e^2$ implies $e$ is infinitely divisible by $\varpi$ in $B$, and hence is $0$ since $B$ is $\varpi$-adically complete {and separated.}
Similarly, $(e-1)\notin \varpi B$. Thus the image of $e$ in $B/\varpi$ is a non-trivial idempotent, contradicting our hypothesis that $\Spec B/\varpi$ is connected. 
\end{proof}

\clearpage{}%
\clearpage{}%
\section{Local models in mixed characteristic} %
\label{sec:MLM}
In this section, we will specialize the universal models from \S \ref{sec:UMLM} to a mixed characteristic DVR.  We introduce naive models which may not be flat but are defined by an explicit condition.   The main result is Theorem \ref{thm:compandSW} which labels the top-dimensional irreducible components of the special fiber by Serre weights.  In fact, this label is ``intrinsic'' to the component in the sense that components with same label which appear in different models can be canonically identified inside (a  suitable subvariety of) the affine flag variety (see Theorem \ref{thm:cmp:match:1}).  Finally, we study the $T^{\vee}$-fixed points on these components and match this with Herzig's conjecture (Definition \ref{defn:W?}) in Theorem \ref{thm:Tfixedpts}.  

Recall that $\cO$ is a finite flat local $\Zp$-algebra with fraction field $E$ and residue field $\F$.  When we decorate an object that occurs in \S \ref{sec:UMLM} with a subscript $\cO$, it means the base change of that object to $\cO$ via the map $\bA^1\to \cO$ sending {$t$} to $-p$.
In particular, we have the objects $L\cG_\cO$, $L^+\cG_\cO$, $L^{--}\cG_\cO$, $L^+ \cM_\cO$, and $\Gr_{\cG,\cO}=L^+\cG_\cO\backslash L\cG_\cO$.   
Similarly, objects decorated with $E$ or $\F$ denote the further base change to $E$ or $\F$ respectively. 
As before, the restrictions of these functors to the category of Noetherian $\cO$-algebras have simple descriptions setting $t=-p$.

\subsection{Mixed characteristic local model} \label{subsec:MLM}

For convenience of the reader, we recall some of the discussion from \S \ref{sec:UMLM} specialized over $\cO$.  As explained at the end of \S \ref{sec:UMLM:loopgroups}, since $v$ is invertible in $E[\![v+p]\!]$, $\Gr_{\cG, E}$ is isomorphic to the affine Grassmannian of $\GL_n$ over $E$.  Similarly, $\Iw_{\F} \defeq L^+\cG_{\F}$ is the usual Iwahori group scheme. (In \S \ref{sec:BKwithdescent}, we introduce a version of $\cI$ over $\cO$ but for now, we only need it over $\F$.)  Then
\begin{align*}
(\Gr_{\cG, \cO})_{\F}=\Fl\defeq \Iw_{\F} \backslash L (\GL_n)_{\F}
\end{align*}
is the affine flag variety over $\F$. 

Let $\lambda \in X_*(T^{\vee})$ be a dominant cocharacter of $T^{\vee} \subset \GL_n$. 
Define $S_E^\circ(\lambda)\subseteq \Gr_{\GL_n, E}$ to be the open affine Schubert cell associated to $(v+p)^{\lambda}\in \Gr_{\GL_n, E}$, $S_E(\lambda)\subseteq \Gr_{\GL_n, E}$ its {Zariski} closure, and $M(\leql)$ the Zariski closure of $S_E(\lambda)$ in $\Gr_{\cG, \cO}$.   
Then $M(\leql)$ is the Pappas--Zhu local model defined in \cite{PZ} associated to the group $\GL_n$, the conjugacy class of $\lambda$, and the Iwahori subgroup.

Let $\bf{a} \in \cO^n$.  Let $R$ be a Noetherian $\cO$-algebra. 
Recall that
\[
L^+ \cM_{\cO} (R) =  \{ M \in \Mat_n(R[\![v+p]\!])  \mid M \text{ is upper triangular mod } v \}.
\]
Define the closed subfunctor $L\cG^{\nabla_{\bf{a}}}_{\cO} \subseteq L\cG_{\cO}$ by the condition that 
 \begin{equation} \label{eq:nablaa}
L\cG^{\nabla_{\bf{a}}}_{\cO}(R)\defeq \left\{A\in L\cG_{\cO}(R) \mid \quad  v \frac{dA}{dv} A^{-1}  + A \Diag(\bf{a}) A^{-1} \in \frac{1}{v+p} L^+ \cM_{\cO} (R) \right\}.
\end{equation}
It is easy to see that $L^+\cG_{\cO}(R) \cdot L\cG^{\nabla_{\bf{a}}}_{\cO}(R) \subseteq L\cG^{\nabla_{\bf{a}}}_{\cO}(R)$ and hence we get a closed {subfunctor}  $\Gr_{\cG, \cO}^{\nabla_{\bf{a}}} \defeq  L^+ \cG_{\cO} \backslash L\cG^{\nabla_{\bf{a}}}_{\cO} \subseteq \Gr_{\cG, \cO}$ {(viewed as functors on Noetherian $\cO$-algebras).}  
Comparing with \eqref{eq:universalnabla}, $\Gr_{\cG, \cO}^{\nabla_{\bf{a}}}$ is clearly the fiber of the universal $\Gr_{\cG, X}^{\nabla}$ over the $\cO$-point $(-p, \bf{a})$.  

\begin{prop}
\label{prop:genfiber}
Let $\bf{a} \in \cO^n$. There is a natural isomorphism
\[
S^\circ_E(\lambda) \cap \Gr_{\cG, \cO}^{\nabla_{\bf{a}}}   \stackrel{\sim}{\longrightarrow} (P_\lambda\backslash \GL_n)_E
\]
where $P_{\lambda}$ is the parabolic subgroup of $\GL_n$ determined by the condition that the $\alpha$-th entry vanishes for all roots $\alpha$ such that $\langle \lambda, \alpha^\vee \rangle < 0$.
In particular, $S^\circ_E(\lambda) \cap \Gr_{\cG, \cO}^{\nabla_{\bf{a}}} $ is a closed, 
irreducible, projective and smooth subscheme of ${S_E(\lambda)}$.
\end{prop}

\begin{proof}  
This is Proposition \ref{lem:intersect_open_schubert_variety} base changed to $E$ and taking the fiber over $\bf{a} \in \mathbb{A}^n(E)$. 
\end{proof}

We can now define the local model associated to $\lambda$ and $\bf{a}$.  
\begin{defn}
\label{defn:alg:MLM}
Let $\bf{a}\in \cO^n$.
Define $M(\lambda, \nabla_{\bf{a}})$ to be the Zariski closure of $S^\circ_E(\lambda) \cap \Gr_{\cG, \cO}^{\nabla_{\bf{a}}}$ in $M(\leql)$.
It is a projective, flat, $\cO$-scheme of relative dimension $\dim (P_\lambda \backslash \GL_n)_E$. 
\end{defn}

\subsubsection{Naive local model}
\label{sec:naivemlm}

\begin{defn} 
\label{def:MLM:1emb}
Let $\bf{a} \in \cO^n$.  Define 
\[
M^{\mathrm{nv}}(\leql, \nabla_{\bf{a}}) = M(\leql) \cap \Gr_{\cG, \cO}^{\nabla_{\bf{a}}}.
\]
\end{defn}

\begin{rmk}
\label{rmk:Mnv:sub}  
There is a natural inclusion of  $M^{\mathrm{nv}}(\leql, \nabla_{\bf{a}})$ into the base change $\cM^{\mathrm{nv}}_X(\leql, \nabla)$ via the map $\Spec \cO  \ra X \times \mathbb{A}^n$ given by $(-p, \bf{a})$ which is an isomorphism on the generic fibers over $E$. 
It is in fact the case that this map is an isomorphism, though we will not need to know this: By \cite[Th\'eor\`eme 1.5]{Lourenco}, the global Schubert variety $S_X(\lambda)$ is flat over $X$, hence the base change of $S_X(\lambda)$ along $\Spec \bZ_p \to X$ induced by ${t}\mapsto -p$ is flat,
hence coincides with $M(\leql)$. Imposing equation \ref{eq:nablaa} yields the result.

\end{rmk}

\begin{prop} \label{prop:compare-with-naive} For any $\lambda' \in X_*(T^{\vee})$ dominant with $\lambda' \leq \lambda$  and $\bf{a} \in \cO^n$, 
\[
M(\lambda^\prime, \nabla_{\bf{a}}) \subset M^{\mathrm{nv}}(\leql, \nabla_{\bf{a}}).
\]   
\end{prop}
\begin{proof}  
Since $S^\circ_E(\lambda') \subset S_E(\lambda)$, $S^\circ_E(\lambda^\prime)^{\nabla_{\bf{a}}} \subset M^{\mathrm{nv}}(\leql, \nabla_{\bf{a}})_E$.   This gives the desired inclusion. 
\end{proof} 

Notice in Proposition \ref{prop:compare-with-naive} that the generic fiber of $M^{\mathrm{nv}}(\leql, \nabla_{\bf{a}})$ contains the generic fibers of the models $M(\lambda', \nabla_{\bf{a}})$ for all $\lambda' \leq \lambda$.  For later applications, we will need an $\cO$-flat model with the same generic fiber. With that in mind, we make the following definition        
\begin{equation}  \label{eq:lessthanorequalmodel}
M(\leql, \nabla_{\bf{a}}) \defeq \bigcup_{\lambda' \leq \lambda} M(\lambda', \nabla_{\bf{a}})
\end{equation} 
which is $\cO$-flat and projective and clearly satisfies $M(\leql, \nabla_{\bf{a}}) \subset M^{\mathrm{nv}}(\leql, \nabla_{\bf{a}})$.

\begin{prop}  \label{prop:compare-with-naive2}  The above inclusion induces an equality
\[
M(\leql, \nabla_{\bf{a}})_E  = M^{\mathrm{nv}}(\leql, \nabla_{\bf{a}})_E.  
\]
\end{prop}
\begin{proof}  Since $E$ is characteristic 0, by Corollary \ref{cor:monodromy_before_reduced} and Proposition \ref{prop:smooth_generic_fiber} $M^{\mathrm{nv}}(\leql, \nabla_{\bf{a}})_E \cong \coprod_{\lambda' \leq \lambda} (P_{\lambda'} \backslash \GL_n)_E$. In particular, the generic fiber is the reduced disjoint union of the generic fibers of $M(\lambda',  \nabla_{\bf{a}})$ for $\lambda' \leq \lambda$.  
\end{proof}

\subsection{Special fiber}
\label{sec:sp:fib}

In this section, we study the special fiber $M^{\mathrm{nv}}(\leql, \nabla_{\bf{a}})_\F$ of $M^{\mathrm{nv}}(\leql, \nabla_{\bf{a}})$ which is a closed subscheme of the affine flag variety $\Fl = \Gr_{\cG, \F}$.  In particular, we study the condition \eqref{eq:nablaa} over $\F$.  To ease notation, we let $\Fl^{\nabla_{\bf{a}}} = \Gr_{\cG, \F}^{\nabla_{\bf{a}}} \subset \Fl$ be the closed subscheme defined by the condition \eqref{eq:nablaa} restricted to $\F$-algebras.   Note that when $R$ is a Noetherian $\F$-algebra, $L^+ \cM_{\cO} (R)$ appearing in \eqref{eq:nablaa} is the same as $\Lie \Iw_{\F}(R) \defeq  \{ M \in \Mat_n(R[\![v]\!]) \mid M \text{ is upper triangular mod } v \}$.

Recall that by \cite[Theorem 9.3]{PZ} (which is a consequence of the coherence conjecture proven in \cite{Zhu_coherence}) the special fiber $M(\leql)_\F$ can be identified with the reduced union of the affine Schubert cells $S^\circ_{\F}(\tld{w}){\defeq \cI_{\F}\backslash\cI_{\F}\tld{w}\cI_{\F}}$  for $\tld{w} \in \Adm^{\vee}(\lambda)$. The goal of this section is to describe $S^\circ_{\F}(\tld{w}) \cap \Fl^{\nabla_{\bf{a}}}$, thereby giving a topological description of $M^{\mathrm{nv}}(\leql, \nabla_{\bf{a}})_\F$. 

\begin{rmk}
The special fibers of $M^{\mathrm{nv}}(\leql, \nabla_{\bf{a}})$ and $M(\lambda, \nabla_{\bf{a}})$ are not reduced in general (see Remark \ref{rmk:BM:nonreduced}). 
\end{rmk}

\begin{defn} \label{defn:modpgeneric}
Let $R$ be an $\F$-algebra and $\overline{\bf{a}} = (a_1, \ldots, a_n) \in R^n$. For any positive integer $m$, we say $\overline{\bf{a}}$ is $m$-generic if for all $i \neq j$, $a_i - a_j \notin \{-m, -m +1, \ldots, m -1, m\}$, where $-m,-m+1,\dots$ are considered as elements of $\Fp\into \F$. 
\end{defn}

\begin{rmk}  \label{rmk:comparegen} Let $\nu \in X^*(T) \cong \Z^n$.  If $t_\nu\in\tld{W}$ is $m$-generic in the sense of Definition \ref{defn:var:gen}(\ref{it:gen:weyl}), then $\nu$ mod $p \in (\F_p)^n$ is $m$-generic in the sense of Definition \ref{defn:modpgeneric}. 
\end{rmk}

Let $d \defeq \dim (B \backslash \GL_n)_{\F}$.  Recall the $\alpha$ critical strip $H^{(0, 1)}_{\alpha} = \{ x \in V \mid 0 < \langle x, \alpha^\vee\rangle < 1 \}$ from \S \ref{sub:AWG:2}.  We now state the main result of this section.

\begin{thm} \label{thm:monodromySchubert} Let $h$ be a positive integer.   Let $\tld{w} \in \tld{W}$ and $\bf{a} \in \cO^n$.  Assume that $\tld{w}$ is $h$-small $($Definition \ref{defn:var:gen}$($\ref{it:def:small}$))$ and that $\overline{\bf{a}} = \bf{a} \mod \varpi \in \F^n$ is $h$-generic. Then  the intersection $S^\circ_{\F}(\tld{w}^*) \cap \Fl^{\nabla_{\bf{a}}}$ is an affine space of dimension $d - \# \{ \alpha \in \Phi^{+} \mid \tld{w}(A_0) \subset H^{(0, 1)}_{\alpha} \}$.
\end{thm}

\begin{lemma} 
\label{bound} Let $\lambda\in X_*(T^{\vee}) = X^*(T)$.   Let $h_{\lambda} =\max_{\alpha^\vee}\{\langle \lambda,\alpha^\vee\rangle\}$. %
If $\tld{w} \in \Adm(\lambda)$, then $\tld{w}^* \in \Adm^{\vee}(\lambda)$ and $\tld{w}^*$ is $h_{\lambda}$-small.  

\end{lemma}
\begin{proof} This follows directly from \cite[Lemmas 2.1.4 and 2.1.5]{LLL}. 
\end{proof}

\begin{cor} \label{cor:compofnaive} Let $\lambda \in X^*(T^{\vee})$ be a dominant cocharacter.  
Assume that $\overline{\bf{a}} = \bf{a} \mod \varpi \in \F^n$ is $h_{\lambda}$-generic.  
Then there is a natural bijection 
\[
\tld{w} \mapsto \overline{(S^\circ_{\F}(\tld{w}^*) \cap \Fl^{\nabla_{\bf{a}}} )}
\]
between $\Adm^{\mathrm{reg}}(\lambda)$ and the top-dimensional irreducible components of $M^{\mathrm{nv}}(\leql, \nabla_{\bf{a}})_\F$.
\end{cor}
\begin{proof}
By \cite[Theorem 9.3]{PZ},  $M(\leql)_\F$ is the reduced union of $S_\F^\circ(\tld{z})$ for $\tld{z} \in \Adm^{\vee}(\lambda)$.  Thus, $(M^{\mathrm{nv}}(\leql, \nabla_{\bf{a}})_\F)_{\mathrm{red}}$ is the reduced union of $S_\F^\circ(\tld{z}) \cap  \Fl^{\nabla_{\bf{a}}}$. 
By \cite[Lemma 2.1.4]{LLL}, any such $\tld{z}$ is of the form $\tld{w}^*$ for $\tld{w} \in \Adm(\lambda)$. By Lemma \ref{bound}, all $\tld{z} \in \Adm^{\vee}(\lambda)$ satisfy the hypotheses of Theorem \ref{thm:monodromySchubert}. Thus, $S_\F^\circ(\tld{z}) \cap  \Fl^{\nabla_{\bf{a}}}$ has maximal dimension $d$ if and only if $\tld{z} = \tld{w}^*$ where $\tld{w}$ is regular.  Moreover, the assignment $\tld{w} \mapsto \overline{(S_\F^\circ(\tld{w}^*) \cap \Fl^{\nabla_{\bf{a}}} )}$ is injective: indeed if $\overline{(S_\F^\circ(\tld{z}) \cap \Fl^{\nabla_{\bf{a}}})}=\overline{(S_\F^\circ(\tld{z}') \cap \Fl^{\nabla_{\bf{a}}})}$ then $S_\F^\circ(\tld{z}) \cap \Fl^{\nabla_{\bf{a}}}$ and $S_\F^\circ(\tld{z}') \cap \Fl^{\nabla_{\bf{a}}}$ are both open and nonempty in an irreducible scheme and so must intersect. 
In particular, $S_\F^\circ(\tld{z})$ and $S_\F^\circ(\tld{z}')$ intersect so that $\tld{z} = \tld{z}'$.   
\end{proof}

The remainder of the section is devoted to the proof of Theorem \ref{thm:monodromySchubert} by studying the $\nabla_{\bf{a}}$- condition \eqref{eq:nablaa} in terms of explicit coordinates for $S_\F^\circ(\tld{w}^*)$.  Let $\tld{w} \in \tld{W}$.  The open affine Schubert cell $S_\F^\circ(\tld{w}^*) \subset \Fl$ is an affine space of dimension $\ell(\tld{w}^*)$.    We now recall explicit coordinates for the affine space using the open cell.     

Recall that the roots $\Phi$ of $G = \GL_n$ are canonically identified with the coroots of $G^{\vee}$ and so we use the same notation for both.  
Thus, for any integer $m$ and any $\alpha \in \Phi$, we have an affine root group $U_{\alpha, m}$ of $G^{\vee}$.
Concretely, if $\alpha = \alpha_{ik}$ with $i \neq k$, then $U_{\alpha, m}$ is the unipotent group with $ik$-entry $c v^m$ for $c$ a constant and all other non-diagonal entries zero.    

Specializing Definition \ref{defn:Lminusminus} to $\F$ (hence $t=0$), we have
\[
L^{--} \cG_{\F} (R) \defeq  \Bigg\{ g \in  \GL_n\Big(R\Big[\frac{1}{v}\Big]\Big)\  \Big|\  \text{$g$ mod  $\frac{1}{v}$  is lower unipotent} \Bigg\}.
\]
In particular, $U_{\alpha, m} \subset L^{--} \cG_{\F}$ if and only if $m  \leq - \delta_{\alpha > 0}$.   (Recall from \S \ref{sec:not:mis} that $\delta_P$ is 1 if $P$ is true and 0 if $P$ if false.)

We first record two easy lemmas. 

\begin{lemma}\label{lem:conjrootgp}
If $\tld{w} = st_\nu \in \tld{W}$ and we set $\tld{z} = \tld{w}^* \in \tld{W}^\vee$, then $\tld{z}^{-1} U_{-\alpha,m} \tld{z} = U_{-s(\alpha),m+\langle \nu,\alpha^\vee\rangle}$.
\end{lemma}

Now fix $x \in A_0$.

\begin{lemma}\label{lem:conjiw}
If $\tld{w} \in \tld{W}$ and we set $\tld{z}=\tld{w}^* \in \tld{W}^\vee$, then $U_{-\alpha,m} \subset \tld{z}^{-1} \cI_{\F} \tld{z}$ if and only if $\langle \tld{w}(x),\alpha^\vee\rangle < m$.
Similarly, $U_{-\alpha,m} \subset \tld{z}^{-1} L^{--} \cG_{\F} \tld{z}$ if and only if $m< \langle \tld{w}(x),\alpha^\vee\rangle$.
\end{lemma}
\begin{proof}
Let $\tld{w}^{-1}$ be $st_\nu$.
By Lemma \ref{lem:conjrootgp}, $U_{-\alpha,m} \subset \tld{z}^{-1} \cI_{\F} \tld{z}$ is equivalent to $U_{-s(\alpha),m+\langle \nu,\alpha^\vee\rangle} = \tld{z} U_{-\alpha,m} \tld{z}^{-1} \subset \cI_{\F}$. 
This is equivalent to the fact that $m+\langle \nu,\alpha^\vee\rangle > \langle x,s(\alpha)^\vee\rangle$ (note that $\lceil\langle x,s(\alpha)^\vee\rangle\rceil = \delta_{s(\alpha)>0}$), or that $m > \langle s^{-1}(x)-\nu, \alpha^\vee\rangle = \langle \tld{w}(x),\alpha^\vee\rangle$.
The proof of the second part is similar.
\end{proof}

\begin{defn}  
\label{defn:Uw}
For any $\tld{z} \in \tld{W}^{\vee}$,  define $N_{\tld{z}} \defeq   \tld{z}^{-1}   L^{--} \cG_{\F}  \tld{z} \cap \cI_\F.$ 
\end{defn}

We can use Lemma \ref{lem:conjiw} to characterize the affine roots which appear in  $N_{\tld{w}}$.

\begin{prop} \label{prop:Nwroot}
If $\tld{w} \in \tld{W}$ and set $\tld{z} = \tld{w}^* \in \tld{W}^\vee$, then $U_{-\alpha,m} \subset N_{\tld{z}}$ if and only if 
\begin{equation} \label{ikey1}
\langle x,\alpha^\vee\rangle < m < \langle \tld{w}(x),\alpha^\vee \rangle.
\end{equation}
\end{prop}
\begin{proof}
This follows from the definition of $N_{\tld{z}}$ and Lemma \ref{lem:conjiw}.
\end{proof}

Let $d_{\alpha,\tld{w}}$ be $\lfloor\langle\tld{w}(x),\alpha^\vee\rangle\rfloor - \lceil\langle x,\alpha^\vee\rangle\rceil$.
Note that $\lceil\langle x,\alpha^\vee\rangle\rceil = \delta_{\alpha>0}$.

\begin{rmk}\label{rmk:msmall}
If $\tld{w}$ is $m$-small, then $d_{\alpha,\tld{w}}\leq m$ for any $\alpha\in \Phi$.
\end{rmk}

The following elementary corollary describes the entries of $N_{\tld{w}}$ in terms of ``polynomials with degree bounds''.

\begin{cor}
\label{cor:deg:bound}
Let $R$ be a Noetherian $\F$-algebra and $\alpha\in \Phi$.
Then $(N_{\tld{w}^*}(R))_{-\alpha}= \{ v^{\delta_{\alpha>0}}f_{\alpha,R} \}$ where $f_{\alpha,R}\in R[v]$ has degree $d_{\alpha,\tld{w}}$ $($with the convention that $f_{\alpha,R}=0$ if $d_{\alpha,\tld{w}}<0$.$)$
\end{cor}

The significance of $N_{\tld{z}}$ lies in the following standard description of the affine Schubert cell over $\F$.
\begin{prop} \label{prop:opencell} Let $\tld{z} \in \tld{W}^{\vee}$. The subgroup scheme $N_{\tld{z}}$ is a  finite-dimensional affine unipotent group scheme over $\F$.  The natural map
\[
\tld{z} N_{\tld{z}} \ra S_\F^\circ(\tld{z})
\]
is an isomorphism of affine spaces of dimension $\ell(\tld{z})$.  
\end{prop}

Before giving the proof of Theorem \ref{thm:monodromySchubert}, we collect a series of preliminary results.

\begin{defn} \label{defn:support} Define the \emph{support} of $N_{\tld{w}}$ (denoted $\mathrm{Supp}(N_{\tld{w}})$) to be the set of $\alpha \in \Phi$ such that $U_{\alpha, m} \subset N_{\tld{w}}$ for some $m$.   %
\end{defn}

\begin{cor} \label{cor:support} Let $\alpha \in \Phi$ and $\tld{w} \in \tld{W}$.  
Then the following are equivalent:
\begin{enumerate}
\item \label{item:rootsupp} $-\alpha \in \mathrm{Supp}(N_{\tld{w}^*})$; 
\item \label{item:higherstrip} $\lfloor\langle\tld{w}(x),\alpha^\vee\rangle \rfloor \geq \lceil\langle x,\alpha^\vee\rangle \rceil$; and 
\item \label{item:critstrip} $\langle\tld{w}(x),\alpha^\vee\rangle > 0$ and $\tld{w}(x)$ and $x$ lie in different $\alpha$-strips.
\end{enumerate}
In particular, $-\mathrm{Supp}(N_{\tld{w}^*})\subset w(\Phi^+)$ where $w\in W$ is the unique element so that $w^{-1}\tld{w} \in \tld{W}^+$ and $\#\mathrm{Supp}(N_{\tld{w}^*})=\# \Phi^{+} - \# \{ \alpha \in \Phi^{+} \mid \tld{w}(A_0) \subset H^{(0,1)}_{\alpha} \}$. 
\end{cor}
\begin{proof}
The equivalence of (\ref{item:rootsupp}) and (\ref{item:higherstrip}) follows from Proposition \ref{prop:Nwroot}.
The equivalence between (\ref{item:higherstrip}) and (\ref{item:critstrip}) is clear.

If $-\alpha \in \mathrm{Supp}(N_{\tld{w}^*})$, then by (\ref{item:critstrip}), $\langle \tld{w}(x),\alpha^\vee\rangle = \langle w^{-1}\tld{w}(x),w^{-1}(\alpha)^\vee\rangle>0$.
Since $w^{-1}\tld{w} \in \tld{W}^+$, we have that $w^{-1}(\alpha) \in \Phi^+$ and that $\alpha \in w(\Phi^+)$.

For each $\alpha \in \Phi^+$, (\ref{item:critstrip}) implies that exactly one of $\{\alpha,-\alpha\}$ is in $\mathrm{Supp}(N_{\tld{w}^*})$ unless $\tld{w}(A_0) \subset H^{(0,1)}_{\alpha}$. 
This gives the desired formula for $\#\mathrm{Supp}(N_{\tld{w}^*})$.
\end{proof}

\begin{cor}
\label{cor:NandB}
Let $w$ be as in Corollary \ref{cor:support}.   Then $N_{\tld{w}^*} \subset  w(L^+ \ovl{N}) w^{-1}$ where $\ovl{N}$ represents the subfunctor of $\GL_n$ of unipotent lower triangular matrices and $L^+ \ovl{N}$ represents the functor on $\F$-algebras $R \mapsto \ovl{N}(R[\![v]\!])$.  
\end{cor} 

\begin{proof}[Proof of Theorem \ref{thm:monodromySchubert}]
Let $w$ be as in Corollary \ref{cor:support} so that $-\mathrm{Supp}(N_{\tld{w}^*}) \subset w(\Phi^+)$.  
Let $\cC = w(\cC_0)$ denote the Weyl chamber corresponding to $w(\Phi^{+})$.  
We use $\leq_{\cC}$ to denote the partial order on $w(\Phi^{+})$ defined by the set of simple roots $w(\Delta)$ (i.e.~$\alpha' \leq_{\cC} \alpha$ if and only if $\alpha - \alpha'$ is a non-negative sum of elements in $w(\Delta)$).    

 For any $\alpha \in \mathrm{Supp}(N_{\tld{w}^*})$, by standard results about unipotent groups, $(N^{-1}_{\tld{w}^*})_{-\alpha} = - v^{\delta_{\alpha  > 0}} f_{\alpha} + G_{\alpha}(<_{\cC} \alpha)$ where $G_{\alpha}(<_{\cC} \alpha)$ is a linear combination of terms of the form $v^{\delta_{\alpha_1 > 0} + \ldots + \delta_{\alpha_k > 0}} f_{\alpha_1} f_{\alpha_2} \ldots f_{\alpha_k}$ where $\alpha = \alpha_1 + \ldots + \alpha_k$ and $-\alpha_i \in \mathrm{Supp}(N_{\tld{w}^*})$.  
Note that if $\alpha > 0$, then at least one of the $\alpha_i$ is also positive and so $v^{\delta_{\alpha > 0}}$ divides $G_{\alpha}(<_{\cC} \alpha)$.

Consider the expression 
\[
L^{\nabla}_{\bf{a}}(N_{\tld{w}^*}) \defeq v \frac{d N_{\tld{w}^*}}{dv} N_{\tld{w}^*}^{-1} + N_{\tld{w}^*} \Diag(\ovl{\bf{a}}) N_{\tld{w}^*}^{-1}. 
\]
Let $-\alpha \in \mathrm{Supp}(N_{\tld{w}^*})$. 

Then,
\begin{equation} \label{e1}
(N_{\tld{w}^*} \Diag(\ovl{\bf{a}}) N_{\tld{w}^*}^{-1})_{-\alpha} = \langle \ovl{\bf{a}},\alpha^\vee\rangle v^{\delta_{\alpha > 0}} f_{\alpha} + G_{\alpha}(<_{\cC} \alpha) + F_{2, \alpha}(<_{\cC} \alpha)
\end{equation}
where $F_{2, \alpha}(<_{\cC} \alpha)$ is a linear combination of terms of the form $v ^{\delta_{\alpha_1 > 0}} f_{\alpha_1} G_{\alpha_2}(<_{\cC} \alpha_2)$ where $\alpha_1 + \alpha_2 = \alpha$ and $-\alpha_1, -\alpha_2 \in  \mathrm{Supp}(N_{\tld{w}^*})$. 
(Recall from \ref{sec:not:RG} that $\langle \ovl{\bf{a}},\alpha^\vee\rangle$ denotes the difference  $a_i - a_k$ if $\alpha=\eps_i-\eps_k$ and $\ovl{\bf{a}} = (a_1, \ldots, a_n)\in\F^n$.)
Note that if $\alpha > 0$, then at least one of the $\alpha_1, \alpha_2$ is also positive and so $v^{\delta_{\alpha > 0}}$ divides $F_{2, \alpha}(<_{\cC} \alpha)$. 

Let $f_{\alpha} = \sum_{i=0}^{d_{\alpha, \tld{w}}} c_{\alpha, i} v^i$.   
Set
\[
f^*_{\alpha} \defeq  (v \frac{d N_{\tld{w}^*}}{dv})_{-\alpha} = v^{\delta_{\alpha > 0}} \sum_{i=0}^{d_{\alpha, \tld{w}}} (i + \delta_{\alpha > 0}) c_{\alpha, i} v^i.
\]
Since the diagonal terms of $v \frac{d N_{\tld{w}^*}}{dv}$ are zero, a direct computation shows that 
\begin{equation} \label{e2}
(v \frac{d N_{\tld{w}^*}}{dv} N_{\tld{w}^*}^{-1})_{-\alpha} = v^{\delta_{\alpha > 0}} \sum_{i=0}^{d_{\alpha, \tld{w}}} (i + \delta_{\alpha > 0 }) c_{\alpha, i} v^i + F_{1, \alpha}(<_{\cC} \alpha)
\end{equation}
where $F_{1, \alpha}(<_{\cC} \alpha)$ is a linear combination of terms of the form $f^*_{\alpha_1} G_{\alpha_2}(<_{\cC} \alpha_2)$ where $\alpha_1 + \alpha_2 = \alpha$ and $-\alpha_1, -\alpha_2 \in  \mathrm{Supp}(N_{\tld{w}^*})$.  
By the same logic as above, $F_{1, \alpha}(<_{\cC} \alpha)$ is always divisible by $v^{\delta_{\alpha > 0}}$.

Combining \eqref{e1} and \eqref{e2}, 
\begin{equation}\label{e3}
L^{\nabla}_{\bf{a}}(N_{\tld{w}^*})_{-\alpha} = v^{\delta_{\alpha > 0}} \sum_{i=0}^{d_{\alpha, \tld{w}}} (i + \delta_{\alpha > 0} + \langle \ovl{\bf{a}},\alpha^\vee\rangle) c_{\alpha, i} v^i + F_{\alpha}(<_{\cC} \alpha)
\end{equation}
where $F_{\alpha}(<_{\cC} \alpha) = a_i G_{\alpha}(<_{\cC} \alpha) + F_{1, \alpha}(<_{\cC} \alpha) + F_{2, \alpha}(<_{\cC} \alpha)$.

Finally, we consider the naive monodromy condition \eqref{eq:nablaa} on the family $\tld{w}^* N_{\tld{w}^*}$. By Leibniz rule, this is the condition that  
\[
 \tld{w}^* L^{\nabla}_{\bf{a}}(N_{\tld{w}^*}) (\tld{w}^*)^{-1} + L^{\nabla}_{\bf{a}}(\tld{w}^*) \in \frac{1}{v} \Lie \Iw_{\F}.   
\]
It is straightforward to check that $L^{\nabla}_{\bf{a}}(\tld{w}^*) \in \frac{1}{v} \Lie \Iw_{\F}$ and so the condition is equivalent to $v L^{\nabla}_{\bf{a}}(N_{\tld{w}^*}) \in (\tld{w}^*)^{-1}\Lie \Iw_{\F}\tld{w}^*$. 
By Lemma \ref{lem:conjiw}, this is equivalent to $v^{d_{\alpha, \tld{w}} + \delta_{\alpha > 0}} = v^{\lfloor \langle \tld{w}(x),\alpha^\vee \rangle \rfloor}$ dividing $L^{\nabla}_{\bf{a}}(N_{\tld{w}^*})_{-\alpha}$ for all $\alpha \in \Phi$.
In other words, all terms but the top degree one in (\ref{e3}) must vanish for all $-\alpha \in \mathrm{Supp}(N_{\tld{w}^*})$.

By Remark \ref{rmk:msmall}, $d_{\alpha, \tld{w}} \leq m$ for all $-\alpha \in \mathrm{Supp}(N_{\tld{w}^*})$.
Since $\overline{\bf{a}}$ is $m$-generic, $i + \delta_{\alpha > 0} + \langle \ovl{\bf{a}},\alpha^\vee\rangle \neq 0$ for all $\alpha$ and all $i < d_{\alpha, \tld{w}}$.
The above condition on (\ref{e3}) solves for $c_{\alpha, i}$ for all $i < d_{\alpha, \tld{w}}$ in terms of the coefficients of $v^{- \delta_{\alpha > 0}} F_{\alpha}(<_{\cC} \alpha)$.  
The coefficients of $v^{- \delta_{\alpha > 0}} F_{\alpha}(<_{\cC} \alpha)$ are expressions in terms of coefficients of $f_{\alpha'}$ for $\alpha' <_{\cC} \alpha$.   
There is no condition on $c_{\alpha, d_{\alpha, \tld{w}}}$ for $-\alpha \in \mathrm{Supp}(N_{\tld{w}})$.   

Thus, if we take $N^{\nabla_{\bf{a}}}_{\tld{w}^*} \subset N_{\tld{w}^*}$ to be the subspace defined by these conditions, then clearly $N^{\nabla_{\bf{a}}}_{\tld{w}^*}$ is an affine space of dimension $\#\mathrm{Supp}(N_{\tld{w}^*})$ with coordinates given by the $c_{\alpha, d_{\alpha, \tld{w}}}$ for all $-\alpha \in \mathrm{Supp}(N_{\tld{w}^*})$.  
Since $\tld{w}^* N^{\nabla_{\bf{a}}}_{\tld{w}^*}$ is isomorphic to $S_\F^\circ(\tld{w}^*) \cap \Fl^{\nabla_{\bf{a}}}$, this proves the theorem by the formula in Corollary \ref{cor:support}.   
\end{proof}

\subsection{Irreducible components in the special fiber} 
\label{sec:componentmatching} 
\label{subsub:cmp:match}
\label{sbsec:MAFL}

We next want to compare the irreducible components of the special fibers $M^{\mathrm{nv}}(\leql,\nabla_{\bf{a}})_\F$ for different pairs $(\lambda, \bf{a})$.   
To do this, we introduce a common space in which they all embed.  

Define $\Fl^{\nabla_0}$ to be the fpqc-sheafification of the sub-presheaf on $\F$-algebras $R$
\begin{equation} \label{eq:nabla0}
R \mapsto \left\{\Iw_{\F} A \in \cI_\F(R) \backslash L \GL_n(R) \, | \, (v\frac{d}{dv} A) A^{-1}\in \frac{1}{v} \Lie \Iw_{\F}(R)\right\}.
\end{equation}
This is a special case of $\Gr_{\cG, \F}^{\nabla_{\bf{a}}}$ where $\bf{a} = (0, 0, \ldots, 0)$, hence the notation.  

{We have and action of
 $\tld{W}^{\vee}$ on $\tld{\Fl}$ by right translation, induced by the standard embedding $\tld{W}^{\vee} \subset  L\GL_n(\Z)$ (which sends an element $w\in W$ to the $n$ by $n$ matrix whose $(i,j)$-th entries are $\delta_{j=w(i)}$, and sends $t_{\mu}$ to $v^{\mu}$).} %
\begin{prop} \label{prop:embinmafv}  Let $\tld{z} = s^{-1} t_{\mu} \in \tld{W}^{\vee}$ acting by right translation on $\Fl$.   
Let $\bf{a} \in \Z^n$ and assume that $\bf{a} \equiv s^{-1}(\mu)$ mod $p$.
Then 
\[
M(\leql)_\F\, \tld{z} \cap \Fl^{\nabla_0} = M^{\mathrm{nv}}(\leql, \nabla_{\bf{a}})_\F\, \tld{z}
\]
Similarly, for any $\tld{w}^*\in \tld{W}^\vee$, we have
\[
\big(S_\F^\circ(\tld{w}^*)\tld{z}\big)\cap \Fl^{\nabla_0} = \big(S_\F^\circ(\tld{w}^*)\cap \Fl^{\nabla_\bf{a}}\big) \tld{z}.
\]

In particular, right translation by $\tld{z}$ induces a closed immersion
\[
r_{\tld{z}}:M^{\mathrm{nv}}(\leql, \nabla_{\bf{a}})_\F \iarrow  \Fl^{\nabla_0}.
\]
\end{prop}
\begin{proof}
We show that the $\nabla_0$-condition on $M(\leql)_\F\, \tld{z}$ induces the $\nabla_{\bf{a}}$-condition \eqref{eq:nablaa} which defines  $M^{\mathrm{nv}}(\leql, \nabla_{ \bf{a}})_\F$. 

Let $R$ be any Noetherian $\F$-algebra and let $A \in L \GL_n (R)$.   We compute the $\nabla_0$-condition on the translate $A \tld{z}$. Namely, $\Iw_{\F} \cdot (A \tld{z}) \in \Fl^{\nabla_0}$ if and only if 
\[
\big(v\frac{d}{dv} (A \tld{z})\big) \tld{z}^{-1} A^{-1} = v \frac{d}{dv}(A) A^{-1} +  A \Diag(s^{-1}(\mu)) A^{-1} \in \frac{1}{v} \Lie \Iw_{\F}(R)
\]
using that $v \frac{d}{dv}(v^{\mu}) = \Diag(\mu) v^{\mu}$.  
This is identical to the condition defining $\Fl^{\nabla_{\bf{a}}} = \Gr_{\cG, \F}^{\nabla_{\bf{a}}}$.   
\end{proof}

Since $M^{\mathrm{nv}}(\leql, \nabla_{\bf{a}})_\F$ is topologically the union of $S_\F^\circ(\tld{w}^*) \cap \Fl^{\nabla_{\bf{a}}}$ for $\tld{w} \in \Adm(\lambda)$,  we consider certain translates of Schubert cells inside $\Fl$ arising from the inclusion in Proposition \ref{prop:embinmafv}. 
{
\begin{defn}
\label{defn:Schubert:var}
Let $\tld{s} \in \tld{W}$ and let $\tld{w}_1, \tld{w}_2 \in \tld{W}^+$.
We define:
\begin{enumerate}
\item 
\label{it:Schubert:var:1}
$S_\F^\circ(\tld{w}_1,\tld{w}_2, \tld{s}) \subset \Fl$ to be the locally closed subvariety $S_\F^\circ((\tld{w}_2^{-1} w_0 \tld{w}_1)^*) \tld{s}^* \subset \Fl$;  
\item
\label{it:Schubert:var:2}
$S_\F^\circ(\tld{w}_1,\tld{w}_2, \tld{s})^{\nabla_0} \defeq  S_\F^\circ(\tld{w}_1,\tld{w}_2, \tld{s}) \cap \Fl^{\nabla_0}$; and
\item
\label{it:Schubert:var:3}
$S_\F^{\nabla_0}(\tld{w}_1,\tld{w}_2, \tld{s})$ to be the closure of $S_\F^\circ(\tld{w}_1,\tld{w}_2, \tld{s})^{\nabla_0}$ in $\Fl^{\nabla_0}$.
\end{enumerate}
\end{defn}
}
{
\begin{rmk} 
\label{FullRmk:strictly_smaller}
\begin{enumerate}
\item
The motivation for considering elements of the form $\tld{w}^{-1}_2 w_0 \tld{w}_1$ is that regular elements are of this form by Proposition \ref{prop:can:reg}.   
\item
\label{rmk:strictly_smaller}
The closure of $S_\F^\circ(\tld{w}_1,\tld{w}_2, \tld{s})^{\nabla_0}$ is usually strictly smaller than the closure of $S_\F^\circ(\tld{w}_1,\tld{w}_2, \tld{s})$ (which is a translate of an affine Schubert variety) intersected with $\Fl^{\nabla_0}$ (see \S \ref{sub:Tfixed}). 
\end{enumerate}
\end{rmk}
}

\begin{prop} \label{prop:mafldim} Let $\tld{s} = t_{\mu} s \in \tld{W}$ and let $\tld{w}_1, \tld{w}_2 \in \tld{W}^+$. 
Let $m$ be a positive integer.  
Assume $\tld{w}_2^{-1} w_0 \tld{w}_1$ is $m$-small and $\tld{s}$ is $m$-generic. 
Then $S_\F^\circ(\tld{w}_1,\tld{w}_2, \tld{s})^{\nabla_0}$ is isomorphic to an affine space of dimension $d$. 
\end{prop}
\begin{proof} By Proposition \ref{prop:embinmafv}, $S_\F^\circ(\tld{w}_1,\tld{w}_2, \tld{s})^{\nabla_0}$ is isomorphic to $S_\F^\circ(\tld{z}) \cap  \Fl^{\nabla_{\bf{a}}}$ where $\tld{z} =(\tld{w}_2^{-1} w_0 \tld{w}_1)^*$  for any $\bf{a} \in \Z^n$ such that $\bf{a} \equiv s^{-1}(\mu)$ mod $p$.  
As $t_\mu$ is $m$-generic we deduce by Remark \ref{rmk:comparegen} that $s^{-1}(\mu)$ mod $p$ is $m$-generic; the result follows now from Theorem \ref{thm:monodromySchubert}. 
\end{proof}

Proposition \ref{prop:mafldim} defines a collection of irreducible closed subvarieties $S_\F^{\nabla_0}(\tld{w}_1,\tld{w}_2, \tld{s})$ of $\Fl^{\nabla_0}$ of dimension $d$, associated to certain triples $(\tld{w}_1, \tld{w}_2, \tld{s})$.   
As we will see, in many cases, we get the same subvariety for different triples $(\tld{w}_1, \tld{w}_2, \tld{s})$ and this is crucial in understanding how the special fibers of different $M(\leql, \nabla_{\bf{a}})$ interact.  

\begin{prop} \label{prop:nosematch} 
 Let $\tld{s} \in \tld{W}$ and let $\tld{w}_1, \tld{w}_2 \in \tld{W}^+$.
Assume that for each $i\in\{1,2\}$ there exists a positive integer $m_i$ such that $\tld{w}_i$ is $m_i$-small. 
There is a closed immersion
\[S_\F^\circ(\tld{w}_1,e, \tld{s} \tld{w}_2^{-1})^{\nabla_0} \subset S_\F^\circ(\tld{w}_1,\tld{w}_2, \tld{s})^{\nabla_0}. \] 
If $\tld{s}$ is $(m_1 + m_2)$-generic, then the two sides are equal, {hence
\[S_\F^{\nabla_0}(\tld{w}_1,e, \tld{s} \tld{w}_2^{-1})=S_\F^{\nabla_0}(\tld{w}_1,\tld{w}_2, \tld{s}).\]} 

\end{prop}
\begin{proof}

Let $\tld{z}_1 = \tld{w}_1^*$ and $\tld{z}_2 = \tld{w}_2^*$. Then, we see that $\tld{z}_1 w_0 \tld{z}_2^{-1}$ is a reduced expression in $\tld{W}^{\vee}$, by Lemma \ref{lemma:gallery} and the proof of \cite[Lemma 2.1.3]{LLL} (which says that the star operation preserves reduced expressions).  
By \cite[Proposition 2.8]{iwahori-matsumoto}, we have
\[
\cI_{\F}\,\tld{z}_1 w_0\,\cI_{\F}\,\tld{z}_2^{-1}\,\cI_{\F}=\cI_{\F}\,\tld{z}_1 w_0 \tld{z}_2^{-1}\,\cI_{\F} 
\]
which in particular implies that $S_\F^\circ(\tld{w}_1,e, \tld{s} \tld{w}_2^{-1}) \defeq  S_\F^\circ(\tld{z}_1 w_0) \tld{z}_2^{-1} \tld{s}^* \subset S_\F^\circ(\tld{z}_1 w_0 \tld{z}_2^{-1}) \tld{s}^* \defeq S_\F^\circ(\tld{w}_1,\tld{w}_2, \tld{s})$.   
This gives the desired inclusion.  

By Proposition \ref{prop:propertiesofsmall}, $\tld{w}_2^{-1} w_0 \tld{w}_1$ is $(m_1 + m_2)$-small and if $\tld{s}$ is $(m_1 + m_2)$-generic, then $ \tld{s}\tld{w}_2^{-1}$ is  $m_1$-generic. 
Thus,  if $\tld{s}$ is $(m_1 + m_2)$-generic, then both sides are affine spaces of the same dimension by Proposition \ref{prop:mafldim} and so inclusion implies equality.  
\end{proof}

\begin{prop} \label{prop:match2}
Let $\tld{s} \in \tld{W}$ and $\tld{w}_1 \in \tld{W}^+_1$.  Assume that $\tld{s}$ is $(n-1)$-generic. Then, for any $w \in W$, \[S_\F^{\nabla_0}(\tld{w}_1,e, \tld{s}) = S_\F^{\nabla_0}(\tld{w}_1,e, \tld{s}w^{-1}).\]
\end{prop}

We prove the Proposition after a couple of Lemmas. 

\begin{lemma} \label{lem:match2}   Let $\tld{s}, w, \tld{w}_1$ be as in Proposition \ref{prop:match2}. Then,
\[
\overline{S_\F^\circ(\tld{w}_1,e, \tld{s})} = \overline{S_\F^\circ(\tld{w}_1,e, \tld{s}w^{-1})} 
\]
where closure is taken in $\Fl$.  
\end{lemma}
\begin{proof}
Translating by $(\tld{s}^*)^{-1}$ inside $\Fl_{\F}$,  it suffices to consider the case where $\tld{s}$ is the identity. 
Recall that $S_\F^\circ(\tld{w}_1, e, e) \defeq  S_\F^\circ(\tld{w}^*_1 w_0) \defeq \cI_{\F} \backslash \cI_{\F}\,\tld{w}^*_1 w_0\, \cI_{\F}$.
It suffices to show that $\overline{S_\F^\circ(\tld{w}^*_1 w_0)}$ is $W^\vee$-stable under right multiplication.
As $\tld{w}^*_1 w'$ is a reduced expression for all $w'\in W^\vee$ (as follows from a gallery argument and the fact that the $*$-involution is length preserving), by \cite[Proposition 2.8]{iwahori-matsumoto}, $\overline{S_\F^\circ(\tld{w}^*_1 w_0)}$ contains $\cI_{\F} \backslash \cI_{\F}\,\tld{w}^*_1\,\cI_{\F}\,w'\,\cI_{\F}$ for any $w'\in W^\vee$.
By the Bruhat decomposition $L^+ \GL_n = \cup_{w'\in W^\vee} \cI_{\F}\,w'\,\cI_{\F}$, so that $\overline{S_\F^\circ(\tld{w}^*_1 w_0)}$ is the closure of $\cI_{\F} \backslash \cI_{\F}\,\tld{w}^*_1\,L^+ \GL_n$, which is evidently $W^\vee$-stable under right multiplication.
The result follows. 
\end{proof}

\begin{lemma} \label{lem:obvious-specialization} Let $\tld{s}, \tld{w}_1$ be as in Proposition \ref{prop:match2}.   If $s_{\alpha}$ is a simple reflection for $\alpha \in \Delta$, then
\[
\tld{w}^*_1 w_0 s_{\alpha} \tld{s}^* \in S_\F^{\nabla_0}(\tld{w}_1,e, \tld{s}).
\]
\end{lemma}
\begin{proof}
Set $\tld{z} = \tld{w}^*_1 w_0$ and $\tld{z}' = \tld{z} s_{\alpha}$.
Since $w_0$ is longest element in $W^\vee$ and $s_{\alpha}$ is a simple reflection, we deduce from Lemma \ref{lemma:gallery} that $\tld{z}' \leq \tld{z}$ and $\ell(\tld{z}') = \ell(\tld{z}) - 1$.  
Let $L_{\alpha} \subset \GL_n$ denote the minimal standard Levi subgroup containing $U_{\alpha, 0}$ and $U_{-\alpha, 0}$.   
Consider the family
\[
X_{\alpha} \defeq  \cI_{\F} \backslash \cI_{\F}\,\tld{z}'  L_{\alpha} \subset \Fl.
\]
We clearly have $\tld{z} = \tld{z}' s_{\alpha} \in X_{\alpha}$, and it is standard result that $X_{\alpha} \cong \mathbb{P}^1_{\F}$ and  $X_{\alpha} \subset  \overline{S^{\circ}_\F(\tld{z})}$  (see, for example, \cite[Proposition 8.8]{PRtwisted}).  

We show that $X_{\alpha} \tld{s}^* \subseteq \Fl^{\nabla_0}$.
This will imply that $X_{\alpha} \tld{s}^* \subset S_\F^{\nabla_0}(\tld{w}_1,e, \tld{s})$, hence the statement.   
For any $A \in L_{\alpha}$, the $\nabla_0$-condition \eqref{eq:nabla0} on $\tld{z}' A \tld{s}^*$ is given by 
\begin{equation} \label{f1}
v \frac{d}{dv} (\tld{z}') (\tld{z}')^{-1} + \tld{z}' A  v \frac{d}{dv} (\tld{s}^*) (\tld{s}^*)^{-1} A^{-1} (\tld{z}')^{-1} \in \frac{1}{v} \Lie \cI_{\F}
\end{equation}
since $\frac{dA}{dv} = 0$. 
If $\tld{s}^* = s^{-1} t_{\mu}$ then $v \frac{d}{dv} (\tld{s}^*) (\tld{s}^*)^{-1} = \mathrm{Diag}(s^{-1}(\mu))$.   
Thus, $A  v \frac{d}{dv} (\tld{s}^*) (\tld{s}^*)^{-1} A^{-1} \in \Lie L_{\alpha}$.   
Then \eqref{f1} is satisfied if $v \Lie L_{\alpha} \subset (\tld{z}')^{-1} \Lie \cI_{\F}\, \tld{z}'$.   

Since $\alpha \in \Delta$ and $\tld{w}_1 \in \tld{W}^+_1$, $0< \langle \tld{w}_1(x),-w_0(\alpha)^\vee\rangle < 1$. 
This implies that $-1 < \langle \tld{z}^*(x),\alpha^\vee\rangle < 0$, so that $\langle (\tld{z}')^*(x),\pm\alpha^\vee\rangle < 1$.
The inclusion $v \Lie L_{\alpha} \subset (\tld{z}')^{-1} \Lie \cI_{\F}\tld{z}'$ now follows from Lemma \ref{lem:conjiw}.
\end{proof} 

\begin{proof}[Proof of Proposition \ref{prop:match2}]
By induction on length of $w$ in $W$, we can assume $w = s_{\alpha}$, a simple reflection for $\alpha \in \Delta$. 

Consider $S_\F^{\nabla_0}(\tld{w}_1,e, \tld{s}) \subset \overline{S_\F^\circ(\tld{w}_1,e, \tld{s})}$, a closed subvariety.   
By Lemma \ref{lem:match2},  the intersection $S_\F^{\nabla_0}(\tld{w}_1,e, \tld{s}) \cap S_\F^\circ(\tld{w}_1,e,\tld{s}  s_{\alpha} )$ is open in $S_\F^{\nabla_0}(\tld{w}_1,e, \tld{s})$.  
If the intersection is non-empty, then since $S_\F^{\nabla_0}(\tld{w}_1,e, \tld{s})$ is irreducible by Proposition \ref{prop:mafldim} (as $\tld{w}_1\in \tld{W}^+_1$ implies that $\tld{w}_1$ is $(n-1)$-small), the intersection  is open and dense and this proves the inclusion $S_\F^{\nabla_0}(\tld{w}_1,e, \tld{s})\subseteq S_\F^{\nabla_0}(\tld{w}_1,e, \tld{s}s_{\alpha})$.   
(Note that $S_\F^{\nabla_0}(\tld{w}_1,e, \tld{s}) \cap S_\F^\circ(\tld{w}_1,e, \tld{s} s_{\alpha})\subseteq S_\F^{\nabla_0}(\tld{w}_1,e, \tld{s} s_{\alpha})$.)
By symmetry, this is enough to prove the proposition.  

Lemma \ref{lem:obvious-specialization} shows that the intersection is non-empty since the point $\tld{w}^*_1 w_0 s_{\alpha} \tld{s}^* \in  S_\F^\circ(\tld{w}_1,e,  \tld{s} s_{\alpha})$ lies in  $S_\F^{\nabla_0}(\tld{w}_1,e, \tld{s})$. This completes the proof.  
\end{proof}

Using Proposition \ref{prop:nosematch} and \ref{prop:match2}, we are able to identify the closed subvarieties of the monodromy affine flag variety $\Fl^{\nabla_0}$ which arise in this way.  Let $\mathrm{Irr}_d(\Fl^{\nabla_0})$ denote the set of irreducible subvarieties of dimension $d$.   We `label' the subvarieties in following way:  

Consider $(\tld{w}_1, \omega) \in \tld{W}_1^+ \times X^*(T)$.   Assume that $t_\omega$ is $(n-1)$-generic (Definition \ref{defn:var:gen}(\ref{it:gen:weyl})).   
Define 
\begin{equation} \label{eq:nameofcomponent}
C_{(\tld{w}_1, \omega)}  \defeq  S_\F^{\nabla_0}(\tld{w}_1, e, \tld{s}) \in \mathrm{Irr}_d(\Fl^{\nabla_0})
\end{equation}
for any choice of $\tld{s} \in \tld{W}$ such that $\tld{s}(0) = \omega$.
By Proposition \ref{prop:match2}, this does not depend on the choice of $\tld{s}$. Note that it also only depends on $(\tld{w}_1, \omega)$ up to equivalence relation $(\tld{w}_1,\omega) \sim (t_\nu\tld{w}_1,\omega-\nu)$ for $\nu\in X^0(T)$ from \S \ref{sec:SW}.  Since $\tld{w}_1$ is $(n-1)$-small, $C_{(\tld{w}_1, \omega)}$ is an irreducible closed subvariety of dimension $d$.  

\begin{thm} 
\label{thm:cpt:match}
Let $\tld{w}_1, \tld{w}_2,  \tld{s} \in \tld{W}$ such that $\tld{w}_1 \in \tld{W}_1^+$ and  $\tld{w}_2 \in \tld{W}^+$. Let $m \geq 1$. Assume that $\tld{w}_2$ is $m$-small and that $\tld{s}$ is $(m + n -1)$-generic.    Then 
\[
S_\F^{\nabla_0}(\tld{w}_1,\tld{w}_2, \tld{s}) = C_{(\tld{w}_1, \tld{s} \tld{w}_2^{-1} (0))}.
\] 
\end{thm}
\begin{proof}
The assumption implies that $t_{\tld{s}\tld{w}_2^{-1} (0)}$ is $(n-1)$-generic by Proposition \ref{prop:propertiesofsmall}, and so $C_{(\tld{w}_1, \tld{s} \tld{w}_2^{-1} (0))}$ is well-defined and equal to  $S_\F^{\nabla_0}(\tld{w}_1, e, \tld{w}^{-1}_2 \tld{s})$ by \eqref{eq:nameofcomponent}.
By Proposition \ref{prop:nosematch},  
\[
S_\F^{\nabla_0}(\tld{w}_1,\tld{w}_2, \tld{s})  = S_\F^{\nabla_0}(\tld{w}_1, e, \tld{s} \tld{w}_2^{-1}). 
\]
\end{proof}

Let $\lambda \in X_*(T^{\vee})$ be a dominant cocharacter.  We now assume that $\lambda$ is regular.   Note then that  $\lambda - \eta$ is dominant where $\eta$ is our choice of lift of half-sum of positive roots.  In Corollary \ref{cor:compofnaive}, we identified the top-dimensional irreducible components of $M^{\mathrm{nv}}(\leql, \nabla_{\bf{a}})_{\F}$ (with a genericity condition on $\bf{a}$).   We now combine this with Theorem \ref{thm:cpt:match} to identify those same components in $\mathrm{Irr}_d(\Fl^{\nabla_0})$.   This will allow us to compare special fibers for various $(\lambda, \bf{a})$.

\begin{thm}
\label{thm:cmp:match:1} Let $\lambda$ be dominant and regular.
Let $h_{\lambda} = \max_{\alpha^\vee}\{\langle \lambda,\alpha^\vee\rangle\}$ 
and let $\bf{a} \in \cO^n$.    
Let $\tld{s} = t_{\mu}s$ be $(h_{\lambda} + n -1)$-generic.
Assume that $\bf{a} \equiv s^{-1}(\mu)$ modulo $\varpi$.   
There is a natural bijection between $\mathrm{AP}(\lambda)$ \emph{(}defined in \eqref{eq:AP}\emph{)} and the $d$-dimensional irreducible components of $M^{\mathrm{nv}}(\leql, \nabla_{\bf{a}})_{\F}) \tld{s}^* \subset \Fl^{\nabla_0}$, given by 
\[
(\tld{w}_1, \tld{w}_2) \mapsto C_{(\tld{w}_1, \tld{s} \tld{w}_2^{-1}(0))}.
\] 
\end{thm} 
\begin{proof}
By Corollary \ref{cor:can:reg}, there is a bijection between $\mathrm{AP}(\lambda)$ and $\Adm^{\mathrm{reg}}(\lambda)$ given by $(\tld{w}_1, \tld{w}_2) \mapsto \tld{w}_2^{-1} w_0 \tld{w}_1$.   
By Corollary \ref{cor:compofnaive},  there is a bijection between regular elements $\tld{w} \defeq \tld{w}_2^{-1} w_0 \tld{w}_1 \in \Adm^{\mathrm{reg}}(\lambda)$ and $\mathrm{Irr}_d(M^{\mathrm{nv}}(\leql, \nabla_{\bf{a}})_{\F})$ sending $\tld{w}$ to 
\[
\overline{S_\F^\circ(\tld{w}_1^* w_0 (\tld{w}^*_2)^{-1}) \cap \Fl^{\nabla_{\bf{a}}}}.
\]
By Proposition \ref{prop:embinmafv} and Theorem \ref{thm:cpt:match},  
\[
(\overline{S_\F^\circ(\tld{w}_1^* w_0 (\tld{w}^*_2)^{-1}) \cap \Fl^{\nabla_{\bf{a}}}}) \tld{s}^* = S_\F^{\nabla_0}(\tld{w}_1, \tld{w}_2, \tld{s}) = C_{(\tld{w}_1, \tld{s} \tld{w}_2^{-1}(0))}
\]
(note that $\tld{w}_2$ is $h_\lambda$-small).
\end{proof}

\subsection{$T^{\vee}$-torsors}  \label{sec:Ttorsors}

Let $\widetilde{\Fl}$ be the ind-scheme representing the fpqc-sheafification of the functor on $\F$-algebras given by $R \mapsto \Iw_{1, \F}(R)\backslash L\GL_n(R)$, where 
\begin{align*}
\Iw_{1,\F}:\, R &\mapsto \{A\in \GL_n(R[\![v]\!]),\, \text{$A$ is upper triangular unipotent modulo $v$}\}.
\end{align*}
The natural quotient map $\Psi: \widetilde{\Fl}\ra \Fl$ is a $T^{\vee}_{\F}$-torsor. We define $\tld{M}(\leql)_{\F}$ via the Cartesian diagram
\[
\xymatrix{
\tld{M}(\leql)_{\F}\ar_{\Psi}[d]\ar@{^{(}->}[r]&\widetilde{\Fl}\ar^{\Psi}[d]
\\
M(\leql)_{\F}\ar@{^{(}->}[r]&\Fl.
}
\]
In particular, $\tld{M}(\leql)_{\F}\ra M(\leql)_{\F}$ is a $T^{\vee}_{\F}$-torsor.  Similarly, for any $\bf{a} \in \cO^n$, we have $T^{\vee}_{\F}$-torsors $\tld{M}(\lambda, \nabla_{\bf{a}})_{\F} \ra M(\lambda, \nabla_{\bf{a}})_{\F}$, $\tld{M}^{\mathrm{nv}}(\leql, \nabla_{\bf{a}})_{\F} \ra M^{\mathrm{nv}}(\leql, \nabla_{\bf{a}})_{\F}$, and $\tld{M}(\leql, \nabla_{\bf{a}})_{\F} \ra M(\leql, \nabla_{\bf{a}})_{\F}$ defined by analogous diagrams.  We abusively use $\Psi$ to denote any of these induced maps. 

\begin{rmk}  
{Despite the notation $\tld{M}(\leql)_{\F}$, we will not define (and will not need) an object $\tld{M}(\leql)$ over $\cO$ whose special fiber is $\tld{M}(\leql)_{\F}$.}
{However, we will construct $\tld{U}(\tld{z}, \leql)$ (cf.~\eqref{eq:chartsforlambda}) which are torus torsors over Zariski opens that cover $M(\leql)$.}
The same is true of the other objects defined above. 
\end{rmk}

Given our choice of embedding $\tld{W}^{\vee} \subset  L\GL_n(\Z)$, $\tld{W}^{\vee}$ acts by right translation on $\tld{\Fl}$ (see the beginning of \S \ref{sec:componentmatching}).    
Hence, we can lift the map $r_{\tld{z}}$ from Proposition \ref{prop:embinmafv} to a Cartesian diagram:
\begin{equation} \label{eq:liftoftranslation} 
\xymatrix{
\tld{M}^{\mathrm{nv}}(\leql, \nabla_{\bf{a}})_{\F}\ar_{\Psi}[d]\ar@{^{(}->}^-{\tld{r}_{\tld{z}}}[r]
&\tld{\Fl}^{\nabla_0}\ar^{\Psi}[d]
\\
M^{\mathrm{nv}}(\leql, \nabla_{\bf{a}})_{\F}
\ar@{^{(}->}_-{r_{\tld{z}}}[r]
&
\Fl^{\nabla_0}
}
\end{equation}
where $\tld{\Fl}^{\nabla_0}$ is the preimage of $\Fl^{\nabla_0}$ in $\tld{\Fl}$.  

Finally, for any $(\tld{w}, \omega) \in \tld{W}^+_1 \times X^*(T)$ where $t_\omega$ is $(n-1)$-generic, let $\tld{C}_{(\tld{w}, \omega)} \subset \tld{\Fl}^{\nabla_0}$ denote the preimage of $C_{(\tld{w}, \omega)}$.  It is a closed irreducible subscheme of dimension $n + d$, i.e.,  $\tld{C}_{(\tld{w}, \omega)} \in \mathrm{Irr}_{d+ n}(\tld{\Fl}^{\nabla_0})$.      

\subsection{Products} \label{subsec:products}

Let $\cJ$ be a finite set as in \S \ref{sec:notations}.   
We take products of all the constructions and results of the previous sections; this will be essential for the connection to Galois representations in \S \ref{sec:local_model_EG}.   
We briefly summarize the necessary notation.

For $\lambda = (\lambda_j)_{j \in \cJ} \in X_*(T^{\vee})^{\cJ} = X^*(T)^{\cJ}$. 
Then
\[
\Adm(\lambda) = \prod_{j \in \cJ} \Adm(\lambda_j) \subset \tld{W}^{\cJ}, \qquad \Adm^{\mathrm{reg}}(\lambda) = \prod_{j \in \cJ} \Adm^{\mathrm{reg}}(\lambda_j).
\]

We can then define a local model $M_{\cJ}(\leql) = \prod_{j \in \cJ} M(\leql_{j})\!\subset\!\Gr_{\cG, \cO}^{\cJ}$  a projective scheme over $\cO$.   Similarly, we define $\tld{M}_{\cJ}(\leql)_{\F} \ra  M_{\cJ}(\leql)_{\F}$ a $T^{\vee, \cJ}_{\F}$-torsor.   

For any $\bf{a} \in (\cO^n)^{\cJ}$,  we define local models $M_{\cJ}(\lambda, \nabla_{\bf{a}}), M_{\cJ}(\leql, \nabla_{\bf{a}}),$ and $M^{\mathrm{nv}}_{\cJ}(\leql, \nabla_{\bf{a}})$ in the natural way.  
We have a closed immersion of the latter inside $\Fl_{\cJ}^{\nabla_0} \defeq (\Fl^{\nabla_0})^{\cJ}$ as in Proposition \ref{prop:embinmafv}. 

We extend the construction in Section \ref{sec:Ttorsors} to get the corresponding  $T^{\vee, \cJ}_{\F}$-torsors $\tld{M}_{\cJ}(\lambda, \nabla_{\bf{a}})_{\F}, \tld{M}_{\cJ}(\leql, \nabla_{\bf{a}})_{\F}$ and $\tld{M}^{\mathrm{nv}}_{\cJ}(\leql, \nabla_{\bf{a}})_{\F}$ over the special fibers.   
For consistency in notation, we define $\tld{\Fl}_{\cJ}^{\nabla_0} \defeq (\tld{\Fl}^{\nabla_0})^{\cJ}$.

We then have analog of Proposition \ref{prop:embinmafv} and \eqref{eq:liftoftranslation}:
\begin{prop} \label{prop:embinmafv2}  Let $\tld{z} = s^{-1} t_{\mu} \in \tld{W}^{\vee, \cJ}$ acting by right translation on $\Fl^{\cJ}$ and $\tld{\Fl}^{\cJ}$ component-wise.   
Let $\bf{a} \in (\cO^n)^{\cJ}$.  
If, for each $j \in \cJ$, $\bf{a}_{j} \equiv s^{-1}_{j}(\mu_{j})$ mod $\varpi$, then right translation by $\tld{z}$ induces a Cartesian diagram
\begin{equation*}
\xymatrix{
\tld{M}^{\mathrm{nv}}_{\cJ}(\leql, \nabla_{\bf{a}})_{\F}\ar@{^{(}->}^-{\tld{r}_{\tld{z}}}[r]\ar_{\Psi}[d]&
\tld{\Fl}_{\cJ}^{\nabla_0}\ar^{\Psi}[d]
\\
M^{\mathrm{nv}}_{\cJ}(\leql, \nabla_{\bf{a}})_{\F}\ar@{^{(}->}_-{r_{\tld{z}}}[r]&
\Fl_{\cJ}^{\nabla_0}
}
\end{equation*}
where the horizontal arrows are closed immersions and the vertical arrows are smooth $T_{\F}^{\vee, \cJ}$-torsors.
\end{prop}

Let $d_{\cJ} \defeq (\# \cJ) d = (\# \cJ) \dim_{\F} (B \backslash \GL_n)_{\F}$.   
Let $\omega = (\omega_{j})_{j \in \cJ}  \in X^*(T)^{\cJ}$ where each $t_{\omega_{j}}$ is $(n-1)$-generic.   
Let $\tld{w} \in (\tld{W}^{+}_1)^{\cJ}$.    
Define  
\begin{equation} \label{Zomegas}
C_{(\tld{w}, \omega)} \defeq  \prod_{j \in \cJ} C_{(\tld{w}_{j}, \omega_{j})} \subset \Fl_{\cJ}^{\nabla_0} \text{ and } \tld{C}_{(\tld{w}, \omega)} \defeq  \prod_{ j \in \cJ} \tld{C}_{(\tld{w}_{j}, \omega_{j})} \subset \tld{\Fl}_{\cJ}^{\nabla_0}
\end{equation}
irreducible closed subschemes of dimension $d_{\cJ}$ and $n (\# \cJ) + d_{\cJ}  $  respectively by Proposition \ref{prop:mafldim} and the results of \S \ref{sec:Ttorsors}.

\subsection{Local models, Deligne--Lusztig representations and Serre weights}
\label{sub:LM:DL:SW}
We now connect up the local models to the representation theory results of \S \ref{sec:DLandSW}.   
Let $\cJ = \Hom(k, \F)$ and let 
$\zeta \in X^*(Z)^{\cJ}$ be an algebraic central character.  Let $\sigma$ denote an $(n-1)$-deep Serre weight for $\rG$ which admits a lowest alcove presentation compatible with $\zeta$ (see \S \ref{sec:SW}) and fix a representative  $(\tld{w}_1, \omega) \in (\tld{W}^+_1)^{\cJ} \times X^*(T)^{\cJ}$ for this lowest alcove presentation.

\begin{defn}  \label{defn:complbbysigma} For $\sigma, (\tld{w}_1, \omega)$ and $\zeta$ as above, define
\[
C^{\zeta}_\sigma \defeq  C_{(\tld{w}_1,\omega)}
\]
as defined in \eqref{eq:nameofcomponent}. 
Note that $C_{(\tld{w}_1,\omega)}$ does not depend on the choice of the representative $(\tld{w}_1,\omega)$ for the $\zeta$-compatible lowest alcove presentation of $\sigma$ (see discussion after \eqref{eq:nameofcomponent}).
\end{defn}

We can now give a representation theoretic parametrization of the irreducible components of the special fiber of the local models using Theorem \ref{thm:cmp:match:1}.

\begin{thm} \label{thm:compandSW} 
Let $\lambda\in X^*(T)^{\cJ}$ be a regular dominant weight and set $h_{\lambda} = \max\{ \langle \lambda, \alpha^{\vee} \rangle \mid \alpha \in \Phi \}$.   Let $R$ be a Deligne--Lusztig representation with  $\max\{2n,h_{\lambda}\}$-generic lowest alcove presentation $(s,\mu)$  which is $(\lambda - \eta)$-compatible with $\zeta \in X^*(\un{Z})$.  Let $\bf{a} \in (\cO^n)^{\cJ}$ such that  $\bf{a} \equiv s^{-1}(\mu + \eta)$ modulo $\varpi$.
Then, 
\[
\Irr_{d_{\cJ}}\Big(\big(M^{\mathrm{nv}}_\cJ(\leql, \nabla_{\bf{a}})\big)_{\F} (s^{-1} t_{\mu + \eta}) \Big) = \big\{ C^{\zeta}_{\sigma} \mid \sigma \in \JH\big(\ovl{R}\otimes W(\lambda - \eta)\big) \big\}.  
\]

\end{thm}
\begin{rmk}\label{rmk:local_model_pure_dim} 
\begin{enumerate}
\item One can show that $M^{\nv}_\cJ(\leql,\nabla_{\bf{a}})_{\F}$ is equidimensional of dimension $d_\cJ$ when $\lambda$ is regular, using arguments similar to that of the proof of Theorem \ref{thm:cpt:match}. As we will not need this information, we will not pursue this here.  %

\item One can ask whether Theorem \ref{thm:compandSW} holds for the flat closure $M_{\cJ}(\leql,\nabla_{\bf{a}}) \subset M^{\mathrm{nv}}_{\cJ}(\leql,\nabla_{\bf{a}})$.   This is true under stronger genericity hypotheses and can be deduced from Theorem \ref{thm:irreducible_components_mod_p}. Note that the proof Theorem \ref{thm:irreducible_components_mod_p} uses global input in order to construct the desired lifts of generic points on the components of the special fiber.    
\end{enumerate}
\end{rmk}
\begin{proof}
We begin with the bijection
\[
\mathrm{AP}(\lambda) \xrightarrow{\sim} \JH\big(\ovl{R}\otimes W(\lambda- \eta)\big)
\]
from Proposition \ref{prop:JHbij}.  In particular, each $\sigma \in \JH\big(\ovl{R}\otimes W(\lambda-\eta)\big)$ is $(n-1)$-deep and there exists a unique element $(\tld{w}_1, \tld{w}_2)+X^0(T)^{\cJ} \in \mathrm{AP}(\lambda)$ such that $(\tld{w}_1, \omega) \defeq  (\tld{w}_1, (t_{\mu + \eta} s) \tld{w}_2^{-1}(0))$ is representative for a lowest alcove presentation for $\sigma$ compatible with $\zeta$. 

If we write $\tld{w}_1 = (\tld{w}_{1, j})_{j \in \cJ}$ and $\omega = (\omega_{j})_{j \in \cJ}$, then by Definitions \ref{defn:complbbysigma} and \eqref{Zomegas}, 
\begin{equation} \label{esigma}
C^{\zeta}_{\sigma} = \prod_{j \in \cJ} C_{(\tld{w}_{1, j}, \,\omega_{j})}.  
\end{equation}

We now examine the top-dimensional irreducible components of $\big(M^{\mathrm{nv}}_\cJ(\leql, \nabla_{\bf{a}})\big)_{\F} (s^{-1} t_{\mu + \eta})$.    We have a product structure 
\[
\Irr_{d_{\cJ}}\Big(\big(M^{\mathrm{nv}}_\cJ(\leql, \nabla_{\bf{a}})\big)_{\F} (s^{-1} t_{\mu + \eta}) \Big) = \prod_{j \in \cJ} \Irr_{d}\Big(\big(M^{\mathrm{nv}}(\leql_j, \nabla_{\bf{a}_{j}})\big)_{\F} (s^{-1}_{j} t_{\mu_{j} + \eta_j}) \Big). 
\]
Theorem \ref{thm:cmp:match:1} says that
\[
 \mathrm{AP}(\lambda_{j}) \xrightarrow{\sim} \Irr_{d}\Big(\big(M^{\mathrm{nv}}(\lambda_j, \nabla_{\bf{a}_j})\big)_{\F} (s^{-1}_{j} t_{\mu_j + \eta_j}) \Big)
\]
such that the $d$-dimensional irreducible components are exactly the $C_{(\tld{w}_{1, j}, \,\omega_{j})}$ appearing in \eqref{esigma}. 
\end{proof}

\subsection{$T^{\vee}$-fixed points and Serre weights}
\label{sub:Tfixed}

In this section, we discuss results about the $T^{\vee, \cJ}$-fixed points on the components $\cC_{\sigma}^{\zeta}$ from Definition \ref{defn:complbbysigma} which will used in the proof of the weight part of Serre's conjecture in Section \ref{sec:unitary}. 

Assume $\sigma$ is an $(n-1)$-deep Serre weight with lowest alcove presentation compatible with $\zeta$.
Fix a representative  $(\tld{w}_1, \omega) \in (\tld{W}^+_1)^{\cJ} \times X^*(T)^{\cJ}$ for this lowest alcove presentation so that $C_{\sigma}^{\zeta} = C_{(\tld{w}_1, \omega)}$.     

Recall that the $T^{\vee}$-fixed point of $\Fl$ under the right translation action are in bijection with $\tld{W}^{\vee}$ under the natural inclusion $\tld{W}^{\vee} \subset \Fl$. It is easy to check directly from condition \eqref{eq:nabla0} that $\tld{W}^{\vee} \subset \Fl^{\nabla_0}$.   If we let $T^{\vee, \cJ}$ act on $\Fl^{\nabla_0}_{\cJ}$ component-wise, then clearly $\tld{W}^{\vee, \cJ} \subset \Fl^{\nabla_0}_{\cJ}$ are exactly the $T^{\vee, \cJ}$-fixed points.   We will abuse notation and use $\tld{z} \in \tld{W}^{\vee, \cJ}$ to also denote the corresponding point of $\Fl^{\nabla_0}_{\cJ}$.  We also recall (cf.~Section \ref{subsec:equal_char_unibranch}) that there is an action of $T^{\vee,\ext}=T^\vee\times \bG_m$ on $\Fl$ where $T^\vee$ acts on $\Fl$ by left translation and the $\bG_m$ factor acts by loop rotation $v\mapsto r^{-1}v$. 

We start with a criteria to detect the torus fixed point of a subvariety of $\Fl$:
\begin{lemma} \label{lem:Tfixedpts_criterion} Let $Y\subset \Fl$ be a finite type irreducible closed subscheme which is stable under the action of $T^{\vee,\ext}$, and let $\tld{z}\in \Fl^{T^{\vee}}$ be a $T^{\vee}$-fixed point. Let $Y^\circ\subset Y$ be an open dense subscheme of $Y$. Then the following are equivalent:
\begin{enumerate}
\item $\tld{z}\subset Y$.
\item $Y\cap L^{--}\cG_\F \tld{z} \neq \emptyset$.
\item $Y^{\circ}\cap L^{--}\cG_\F \tld{z} \neq \emptyset$.
\end{enumerate} %
\end{lemma}
\begin{proof} Specializing the $\bG_m$-action constructed in Lemma \ref{lem:contracting_torus_action} and noting that $L^{--}\cG_\F\tld{z}$ is the specialization of $\cU(\tld{z})$ in \emph{loc.~cit.}~along the map $\bZ[v]\to \bF$ sending $v$ to $0$, we find an one parameter subgroup $\bG_m\subset T^{\vee,\ext}$ which induces a contracting action on $L^{--}\cG_\F\tld{z}$ with unique fixed point $\tld{z}$.
It is clear that the first item implies the second item. Conversely, if the second item holds, then the first item holds, since $\tld{z}$ is the limit of a $\bG_m$-orbit of any point in $Y \cap L^{--}\cG_\F \tld{z}$ and $Y$ is closed and $T^{\vee,\ext}$-stable.

Finally, since $Y\cap L^{--}\cG_\F \tld{z} \neq \emptyset$ is an open subscheme of $Y$, it is either empty or open and dense in $Y$. Thus the second and the third item are equivalent.
\end{proof}

\begin{prop}\label{prop:Tfixedobv}  The set of $T^{\vee, \cJ}$-fixed points of $C_{(\tld{w}_1, \omega)}$  contains $\{ (t_{\omega} w\tld{w}_1)^* \mid w\in W^{\cJ} \}$
\end{prop}
\begin{proof}  Since  $C_{(\tld{w}_1, \omega)} =  \prod_{j \in \cJ} C_{(\tld{w}_{1, j}, \omega_{j})}$, this reduces immediately to a statement about  $C_{(\tld{w}_{1, j}, \omega_{j})}$.  
By Theorem \ref{thm:cpt:match},  $C_{(\tld{w}_{1, j}, \omega_{j})}$ is equal to $S_{\F}^{\nabla_0}(\tld{w}_{1, j}, e,  t_{\omega_j} w w_0)$ {(see Definition \ref{defn:Schubert:var}\eqref{it:Schubert:var:3})} which is easily seen to contain the point $(t_{\omega_j} w\tld{w}_{1, j})^*$.
\end{proof}

As discussed in Remark \ref{FullRmk:strictly_smaller}\eqref{rmk:strictly_smaller}, $ S_\F^{\nabla_0}(\tld{w}_{1, j},e, t_{\omega_j})$ can be much smaller than $S_{\F}( (w_0 \tld{w}_{1, j})^*) t_{\omega_j} \cap \Fl^{\nabla_0}$.   Nevertheless under suitable genericity hypotheses, they have the same $T^{\vee}$-fixed points. 

\begin{prop} \label{prop:Tfixedpts} There exists a polynomial $P_{\tld{w}_{1,j}}\in \Z[X_1,\ldots, X_n]$ depending only on $\tld{w}_{1,j} \in \tld{W}^+_1$ such that if $P_{\tld{w}_{1,j}}(\omega_j) \neq 0 \mod p$ for all $j \in \cJ$, then the set of $T^{\vee, \cJ}$-fixed points of $C_{(\tld{w}_1, \omega)}$ is exactly  $\{ \tld{w}^* t_{\omega} \mid \tld{w} \leq w_0 \tld{w}_1\}$.
\end{prop}

\begin{proof} %

Fix $j\in \cJ$. By Proposition \ref{prop:embinmafv}, $Y\defeq C_{(\tld{w}_{1,j},\omega_j)}t_{-\omega_j}$ is the closure of $Y^{\circ}\defeq S^{\circ}_\F(\tld{w}^*_{1,j} w_0)\cap \Fl^{\nabla_{\omega_j}}$.  
One inclusion then follows from the standard description of the $T^{\vee}$-fixed points of the closure of $S^{\circ}_\F(\tld{w}^*_{1,j} w_0)$ in terms of the Bruhat order.   

Fix $\tld{w}\leq w_0\tld{w}_{1, j}$. We need to show that $\tld{z}\defeq \tld{w}^*$ belongs to $Y$ if $\omega_j \mod p$ avoids the zero locus of a universal polynomial depending only on $\tld{w}_{1,j}$.  We will deduce the result from the main result of \cite{Pablo}, which describes the torus fixed points of certain affine Springer fibers.

We consider the base change of the objects in Section \ref{sec:UMLM:loopgroups} along the map $\Z[v]\to \Z$ sending $v$ to $0$.
In particular we get $\cF l_\Z=\Gr_{\cG,X} \times_X \bZ$ and the ind-group schemes $L\cG_\Z$, $L^{--}\cG_\Z$, etc. Thus $\cF l_\Z$ is the affine flag variety for $\GL_n$ over $\Z$, and we have the open affine Schubert variety $\cS_\Z^{\circ}(\tld{w}^*_{1,j} w_0)\subset \cF l_\Z$ which is isomorphic $\bA^{\ell(w_0 \tld{w}_{1,j})}$. We also have the subfunctor $\cU_\Z(\tld{z})\subset \cF l_\Z$ by base changing $\cU(\tld{z})$, which coincides with $L^{--}\cG_\Z \tld{z}$. 
The closed subfunctor of $L\cG_\Z \times_\Z \bA^1\times_\Z \bA^n$ which classifies triples $(g,b,\bf{a})$ such that
\[ b\frac{vdg}{dv} g^{-1}  + g \Diag(\bf{a}) g^{-1} \in \frac{1}{v} L^+\cM\]
induces a closed subscheme $\cY^{\circ}$ of $\cS_\Z^{\circ}(\tld{w}^*_{1,j} w_0)\times_\Z \bA^1\times_\Z \bA^n$. Let $\pi:\cY^\circ\to \bA^{n+1}$ be the projection map.
 We observe
\begin{itemize}
\item $Y^{\circ}$ is the base change of $\cY^{\circ}$ along the map $\bA^1\times_\Z \bA^n \to \F$ corresponding to the tuple $(1,\omega_j \mod p)\in \F^{n+1}$.
\item Let $V\subset \bA^1\times_\Z\bA^n$ be the open locus of tuples $(b,\bf{a})$ such that $b(i + \delta_{\alpha > 0}) + \langle \bf{a}, \alpha^\vee \rangle)$ is invertible for all roots $\alpha$ and $0\leq i< d_{\alpha,w\tld{w}_{1,j}}$.%
Then the proof of Theorem \ref{thm:monodromySchubert} shows that the restriction $\pi: \cY^{\circ}|_{V}\to V$ is isomorphic to the projection $\bA^d \times _\Z V\to V$ (recall from \S \ref{sec:sp:fib} that $d = \dim (B \backslash \GL_n)_{\F}$). 
\item The $\bG_m$-action on $\cF l_\Z\times_\Z \bA^{n+1}$ induced by the scaling action on $\bA^{n+1}$ and the trivial action on $\cF l_\Z$ preserves $\cY^\circ$.
\item If $k$ is a field and $(0,\bf{a})$ is a $k$-point of $V$, then the reduced fiber of $\cY^\circ$ above $(0,\bf{a})$ is an open dense subset of an irreducible component of the affine Springer fiber in $\Fl_k$ associated to the element $v\bf{a}\in \mathfrak{gl}_n(\!(v)\!)$. 
This is exactly the affine Springer fiber studied in \cite{Pablo}.
\end{itemize}
We now consider the intersection $\cZ\defeq \cY^\circ\cap \cU_\Z(\tld{z})\times_\Z \bA^{n+1}$. Then by the fourth item above and \cite[Theorem 3.1]{Pablo}, this intersection is non-empty. Thus $\cZ$ is a non-empty open subscheme of $\cY^\circ$, hence its image $\pi(\cZ)$ is open in $\bA^{n+1}$. Since $\cZ$ is also stable under the scaling $\bG_m$ action, so is $\pi(\cZ)$. Thus there exists a non-zero homogenous polynomial $\tld{P}\in \bZ[b,a_1,\cdots a_n]$ which vanishes on the complement of $\pi(\cZ)\cap V$. Note that $\cY^{\circ}$, and hence $\tld{P}$ depends only on $\tld{w}_{1, j}$ and $\tld{z}$.
Setting $P_{\tld{w}_{1,j},\tld{z}}(a_1,\cdots, a_n)=\tld{P}(1,a_1,\cdots a_n)\neq 0$, we see that as long as $P_{\tld{w}_{1,j},\tld{z}}(\omega_j)\mod p\neq 0$, the fiber of $\cZ$ at the tuple $(1,\omega_j)$ is non-empty. But this fiber is exactly $Y^{\circ}\cap L^{--}\cG_\F\tld{z}$, so Lemma \ref{lem:Tfixedpts_criterion} shows that $\tld{z}\in Y$ in this situation. The polynomial $P_{\tld{w}_{1,j}}=\prod_{\tld{z}^*\leq w_0\tld{w}_1} P_{\tld{w}_{1,j},\tld{z}}$ thus satisfies the conclusion of the Proposition.
\end{proof}
\begin{rmk}\label{rmk:compTfixedpts}
\begin{enumerate}
\item \label{it:compTfixedpts:1}
In fact, whether $P_{\tld{w}_{1,j}}(\omega_j) \neq 0 \mod p$ for a $j\in \cJ$ with $P_{\tld{w}_{1,j}}$ as in the proof of Proposition \ref{prop:Tfixedpts} does not depend on the choice of representative $(\tld{w}_1,\omega)$ for the lowest alcove presentation of $F_{(\tld{w}_1,\omega)}$.
\item \label{it:compTfixedpts:2}
If $\sigma$, $\sigma'$ are two Serre weights for which Proposition \ref{prop:Tfixedpts} holds, then Proposition \ref{prop:covering} shows that $\sigma$ covers $\sigma'$ if and only if all the $T^{\vee,\cJ}$-fixed points of $C^{\zeta}_{\sigma'}$ lie in $C^{\zeta}_{\sigma}$. 
\end{enumerate}
\end{rmk}

We also record the following, which will be convenient for applications:

\begin{prop}\label{prop:Tfixedpots_unibranch} $C_{(\tld{w}_1,\omega)}$ is unibranch at each of its $T^{\vee,\cJ}$-fixed points. 
\end{prop}
\begin{proof} Let $\tld{z}\in C_{(\tld{w}_1,\omega)}$ is a fixed point. The result follows applying Lemma \ref{lem:attractor} to $C_{(\tld{w}_1,\omega)}\cap L^{--}\cG_\F\tld{z}$, using the (specialization of the) contracting $\bG_m$-action constructed in \ref{lem:contracting_torus_action}.
\end{proof}

We now connect back to the Herzig's conjecture on modular Serre weights \S \ref{sec:herzig}. 

 \begin{thm} \label{thm:Tfixedpts} %
Suppose that $(\tld{w}_1,\omega)$ is a lowest alcove presentation of an $(n-1)$-deep Serre weight $\sigma$ and $(s,\mu)$ is a $2(n-1)$-generic lowest alcove presentation of a tame inertial $L$-parameter $\rhobar$ over $\F$.
Suppose that both lowest alcove presentations are compatible with $\zeta \in X^*(\un{Z})$.
Let $\tld{w}^*(\rhobar) = (t_{\mu+\eta} s)^* = s^{-1} t_{\mu+\eta}$.
 \begin{enumerate}
 \item \label{it:Tfixedpts:1} If $\sigma \in W_{\obv}(\rhobar)$, then $\tld{w}^*(\rhobar) \in C_{\sigma}^{\zeta}$. 
 \item \label{it:Tfixedpts:2} If $\tld{w}^*(\rhobar) \in C_{\sigma}^{\zeta}$, then $\sigma \in W^?(\rhobar)$.
 \item \label{it:Tfixedpts:3} For each $j\in\cJ$, let $P_{\tld{w}_{1 ,j}} \in \Z[X_1,\ldots, X_n]$ be as in Proposition \ref{prop:Tfixedpts}. If $P_{\tld{w}_{1, j}}(\omega_j) \neq 0 \mod p$ for all $j \in \cJ$ and $\sigma \in W^?(\rhobar)$, then $\tld{w}^*(\rhobar) \in C_{\sigma}^{\zeta}$.  
 \end{enumerate}  
 \end{thm}
 \begin{proof}
 The set $\{ \tld{w}^* t_{\omega} \mid \tld{w} \leq w_0 \tld{w}_1\}$ from Proposition \ref{prop:Tfixedpts} can also written $\{ (t_{\omega} w\tld{w}_2)^* \mid w\in W^{\cJ},\, \tld{w}_2 \in \tld{W}^{+, \cJ},\, \tld{w}_2 \leq \tld{w}_1\}$ and taking $\tld{w}_2 = \tld{w}_1$ gives the set from Proposition \ref{prop:Tfixedobv}.
 Let $(\tld{w}_1, \tld{w}_2)$ be the pair as in \eqref{eqn:W?} which gives the presentation for $\sigma$ so that $C^{\zeta}_{\sigma} = C_{(\tld{w}_1, \tld{w}(\rhobar) \tld{w}_2^{-1}(0))}$.  Writing $ \tld{w}(\rhobar) \tld{w}_2^{-1}$ as $t_{\omega} w$, the first item follows from Proposition \ref{prop:Tfixedobv} and the third from Proposition \ref{prop:Tfixedpts}.  
 
For the second item, by the upper bound on $T^{\vee}$-fixed points of $C^{\zeta}_{\sigma}$, if $\tld{w}^*(\rhobar) \in C_{\sigma}^{\zeta}$ then $\tld{w}(\rhobar) = t_{\omega} w' \tld{w}_2$ where $\tld{w}_2 \leq \tld{w}_1$ as above.   
By Proposition \ref{prop:W?}, $\sigma \in W^?(\rhobar)$  since $\tld{w}(\rhobar) \tld{w}_2^{-1} (0) = \omega$.  %
 \end{proof} 
\clearpage{}%
\clearpage{}%
\section{Breuil--Kisin modules and Pappas--Zhu local models}
\label{sec:BK:LM}

\subsection{Breuil--Kisin modules with tame descent} \label{sec:BKwithdescent}

Throughout this section we take $G=\GL_n$ and consider the setting of \S \ref{sec:InertialTypes}.
Let $\tau: I_{\Qp}\ra \widehat{\un{T}}(E)$ be a tame inertial $L$-parameter over $E$, with an associated tame inertial type $\tau:I_K\ra\GL_n(E)$ for $K$ as described in Example \ref{ex:data:type}.
We fix  throughout this section a $1$-generic lowest alcove presentation $(s,\mu)$ for $\tau$.
Let $r$ be the order of $s_\tau$, and let $K'$ be the subfield of $\ovl{K}$ which is unramified of degree $r$ over $K$.
Set $\cJ' = \Hom_{\Qp}(K', E)$ and $\cJ=\Hom_{\Qp}(K,E)$.  
Let $f'\defeq fr$, $e' \defeq p^{f'}-1$.
We fix an isomorphism $\sigma'_0:K' \iarrow E$ which extends $\sigma_0:K \iarrow E$.
The identifications $\cJ'\cong \Z/f' \Z$ and $\cJ\cong \Z/f\Z$ (given by $\sigma_{j'} \defeq \sigma'_0\circ \phz^{-j'} \mapsto j'$ and $\sigma_{j} \defeq \sigma_0 \circ \phz^{-j} \mapsto j$ respectively) are such that  restriction of embeddings from $K'$ to $K$ induces the surjection $\cJ' \onto \cJ$ given by reducing modulo $f$ in the above identifications. 
Write $\tau'$ for the tame inertial type for $K'$ obtained from $\tau$ via the identification $I_{K'}=I_K$ induced by the inclusion $K'\subseteq \ovl{K}$.

We fix an $e'$-th root $\pi'\in \ovl{K}$ of $-p$ and 
set $L' \defeq K'(\pi')$.
Let $\Delta' \defeq \Gal(L'/K') \subset \Delta \defeq \Gal(L'/K)$. 
 We set $\omega_{K'}(g) =  \frac{g(\pi')}{\pi'}$ for $g \in \Delta'$ and note that $\omega_{K'}$ does not depend on the choice of $\pi'$.  We can also think of $\omega_{K'}$ as a character of $I_{K'} = I_{K}$ valued in $\cO_{K'}^{\times}$ (the units in the ring of integers of $K'$). %
 Composing with $\sigma'_0$, we get a character $\omega_{K', \sigma'_0}:\Delta' \ra \cO^{\times}$.  In notation of Example \ref{ex:data:type}, we have $\omega_{K', \sigma'_0} = \omega_{f'}$ and hence $\omega_{K', \sigma_{j'}} = \omega_{f'}^{p^{f'-j'}}$.

Let $R$ be an $\cO$-algebra. Let $\fS_{L'} \defeq W(k')[\![u']\!]$ and $\fS_{L', R} \defeq (W(k') \otimes_{\Zp} R)[\![u']\!]$. 
As usual, $\varphi:\fS_{L', R} \ra \fS_{L', R}$ acts as Frobenius on $W(k')$, trivially on $R$, and sends $u'$ to $(u')^{p}$.  

We endow $\fS_{L', R}$ with an action of $\Delta$ as follows: for any $g$ in $\Delta'$, $g(u') = \frac{g(\pi')}{\pi'} u' = \omega_{K'}(g) u'$   
and $g$ acts trivially on the coefficients; if $\sigma^f \in\Gal(L'/K)$ is the lift of $p^f$-Frobenius on $W(k')$ which fixes $\pi'$, then $\sigma^f$ is a generator for $\Gal(K'/K)$, acting in natural way on $W(k')$ and trivially on both $u'$ and $R$. 
Set $v = (u')^{e'}$, and define
\[
\fS_R \defeq (\fS_{L', R})^{\Delta = 1} = (W(k) \otimes_{\Zp} R)[\![v]\!].
\]
Set $E(v) \defeq v + p = (u')^{e'} + p$.  

We will make use of the group scheme $\cI$ defined over $\cO$, which is the base change of $L^+\cG$ along the map $
\bA^1\to \Spec \cO$ sending $t$ to $0$. In other words, for $R$ a Noetherian $\cO$-algebra,
\[\cI(R)=\{A\in \GL_n(R[\![v]\!]) \mid A \textrm{ is upper triangular mod } v\}.\]
We also have the normal subgroup $\cI_1$ of $\cI$ defined by 
\[\cI_1(R)=\{A\in \GL_n(R[\![v]\!]) \mid A \textrm{ is unipotent upper triangular mod } v\}.\]
Note that $\cI=T^\vee_{\cO} \ltimes \cI_1$, where $T^\vee_{\cO}$ is viewed as the subgroup of $\cI$ consisting of constant diagonal matrices.

As in Section \ref{sec:MLM}, when we decorate an object that occurs in Section \ref{sec:UMLM} with a subscript $\cO$, it means we take the base change of that object to $\cO$ via the map $\bA^1\to \cO$ sending $t$ to $-p$. In particular, we have the objects $L\cG_\cO$, $L^+\cG_\cO$, $L^{--}\cG_\cO$, $\Gr_{\cG,\cO}=L^+\cG_\cO\backslash L\cG_\cO$. %

In general the map $v\mapsto v^p$ does not extend to a homomorphism $R[\![v+p]\!]\to R[\![v+p]\!]$. However, when $R$ is $p$-adically complete, we have $R[\![v+p]\!]=R[\![v]\!]$, and so $\phz$ extends to $R[\![v+p]\!]$. Similarly,
if $p$ is nilpotent in $R$ then $R(\!(v+p)\!)=R(\!(v)\!)$, and so $\phz$ extends to $R(\!(v+p)\!)$. Furthermore, the group functors $\cI$ and $L^+\cG_\cO$ coincide on the category of $p$-adically complete Noetherian $\cO$-algebras.  %
Unless stated otherwise, $R$ will be a $p$-adically complete $\cO$-algebra for the remainder of the section.

For any positive integer $h$, let $Y^{[0, h]}(R)$ be the groupoid of Breuil--Kisin modules of rank $n$ over $\fS_{L', R}$ and height in $[0,h]$:
\begin{defn}
An object of $Y^{[0, h]}(R)$ is a pair $(\fM,\phi_\fM)$ where $\fM$ is a finitely generated projective $\fS_{L', R}$-module, which is locally free of rank $n$, and $\phi_\fM:\phz^*(\fM)\ra\fM$ is an injective $\fS_{L', R}$-linear map whose cokernel is annihilated by $E(v)^h = ((u')^{p^{f'}-1} + p)^h$.
\end{defn}
For any $(\fM,\phi_\fM)\in Y^{[0, h]}(R)$, we have a standard $R[\![u']\!]$-linear decomposition $\fM\cong\bigoplus_{j' \in \cJ'}\fM^{(j')}$, induced by the maps $W(k')\otimes_{\Zp}R\rightarrow R$ defined by $x\otimes r\mapsto \sigma_{j'}(x)r$ for $j' \in \cJ'$.  Note that for the corresponding $R[\![u']\!]$-decomposition $\fS_{L', R} \cong \oplus_{j' \in \cJ'} R[\![u']\!]$, the action of $\Delta'$ on $u'$ in embedding $j'$ is given by $\sigma_{j'} \circ \omega_{K'} = \omega_{f'}^{p^{f' - j'}}$.    The Frobenius $\phi_{\fM}$ induces morphisms $\phi_\fM^{(j')}:\phz^*(\fM^{(j'-1)})\ra \fM^{(j')}$.    

\begin{rmk} There is choice of convention  on whether the domain or target of $\phi_\fM^{(j')}$ should correspond to the $\sigma_{j'}$-embedding.  
We are changing the convention here from our previous works, namely \cite{LLLM, LLL, LLLM2}. 
The convention here makes the connection to Hodge--Tate weights labelled by embedding and representation theory more natural. 
The comparison with \cite{LLLM, LLL, LLLM2} is explained in detail in Remark \ref{rmk:cmpr:mat}.  
\end{rmk}

We let $Y^{[0, h], \tau}(R)$ denote the groupoid of Breuil--Kisin modules of rank $n$, height in $[0,h]$ and descent data of type $\tau$ (cf.~\cite[\S 3]{CL}, \cite[Definition 3.1.3]{LLLM2}):
\begin{defn}\label{defn:Breuil-Kisin} 
An object of $Y^{[0, h], \tau}(R)$ is the datum of $(\fM, \phi_{\fM})\in Y^{[0, h]}(R)$ together with a semilinear action of  $\Delta$ on $\fM$ which commutes with $\phi_{\fM}$, and such that, for each $j' \in \cJ'$, 
\[
\fM^{(j')} \mod u' \cong \tau^{\vee} \otimes_{\cO} R 
\]  
as $\Delta'$-representations. 
In particular, the semilinear action of $\Delta$ induces an isomorphism $\iota_{\fM}:(\sigma^f)^*(\fM) \cong \fM$ (see \cite[\S 6.1]{LLLM}) as elements of $Y^{[0,h], \tau'}(R)$.
\end{defn}

We will often omit the additional data and just write $\fM \in Y^{[0, h], \tau}(R)$. 

\begin{rmk}
\begin{enumerate}
\item
As explained in \cite[\S 6.1]{LLLM}, the data of an extension of the action of $\Delta'$ to an action of $\Delta$ is equivalent to the choice of an isomorphism $\iota_{\fM}:(\sigma^f)^*(\fM) \cong \fM$ satisfying an appropriate cocycle condition. We will use both points of view interchangeably.
\item {It is known (\cite[Corollary 3.1.7]{CEGS}, see also \cite[Theorem 4.7]{CL}) that $Y^{[0,h],\tau}$ is a $p$-adic formal algebraic stack in the sense of \cite[Definition A.2]{CEGS} and therefore it is determined by its values on $\cO/\varpi^a$-algebras of finite type, for $a\geq 1$ (and hence, on $p$-adically complete Noetherian $\cO$-algebras).}
\item  The appearance of $\tau^{\vee}$ in the definition is due to the fact that we are using the contravariant functors to Galois representations to be consistent with \cite{LLL, LLLM}.  In \cite{LLLM}, we didn't use the notation $\tau^{\vee}$. Instead, we included it in our description of descent data by having a minus sign in the equation before Definition 2.1 of \emph{loc. cit.}   The notion of Kisin module with tame descent data of type $\tau$ here is consistent with what appears in both \cite{LLL, LLLM}.  
\end{enumerate}
\end{rmk}
\begin{rmk}
Recall that we have fixed a lowest alcove presentation of the tame inertial $L$-parameter $\tau$.
Definitions \ref{defn:eigenbasis} and \ref{defn:shape} below, as well as the definition of matrix of partial Frobenius $A^{(j)}_{\fM, \beta}$ depend on the choice of presentation. 
\end{rmk}

\begin{defn} \label{defn:eigenbasis} (\cite[Definition 3.1.6]{LLLM2})
Let $\fM \in Y^{[0,h], \tau}(R)$.
An \emph{eigenbasis} of $\fM$ is a collection of (ordered) bases ${\beta}^{(j')}=(f_1^{(j')},f_2^{(j')},\dots,f_n^{(j')})$ for each $\fM^{(j')}$ for $j' \in \cJ'$ such that $\Delta'$ acts on $f_i^{(j')}$ via the character $\chi_i^{-1}$ from (\ref{eq:def:type})  and  which is compatible with the isomorphism $\iota_{\fM}$ in the sense that $\iota_{\fM} ((\sigma^f)^*(\beta^{(j')})) = \beta^{(j' + f)}$. 
\end{defn}

We now define the notion of matrix of partial Frobenius with respect to $\beta$ for an object $\fM \in Y^{[0, h], \tau}(R)$.
Let $\fM \in Y^{[0, h], \tau}(R)$ and let $\beta$ be an eigenbasis for $\fM$.
Define $C^{(j')}_{\fM, \beta}$ to be the matrix of $\phi_\fM^{(j')}:\phz^*(\fM^{(j' - 1)})\ra \fM^{(j')}$ with respect to the bases $\phz^*(\beta^{(j'-1)})$ and $\beta^{(j')}$.  
The height condition on $\fM$ is equivalent to $C^{(j')}_{\fM, \beta} \in \Mat_n(R[\![u']\!])$ and $E(v)^h (C^{(j')}_{\fM, \beta})^{-1} \in \Mat_n(R[\![u']\!])$.  
The fact that $\beta$ is compatible with  $\iota_{\fM}:(\sigma^f)^*(\fM) \cong \fM$ implies that $C^{(j')}_{\fM, \beta}$ only depends on $j'$ mod $f$.   

Because $\phi_{\fM}^{(j')}$ commutes with descent datum, this implies a certain $u'$-divisibility of the entries of $C^{(j')}_{\fM, \beta}$. 
To  ``remove the descent datum," we first recall some data related to the tame inertial type $\tau$. 
Let $(s, \mu)$ be a lowest alcove presentation of $\tau$.  
Recall from Example \ref{ex:data:type} that $\bm{\alpha}_j = s_{f-1}^{-1} s_{f-2}^{-1} \ldots s_{f-j}^{-1}(\mu_{f-j} + \eta_{f-j})$ for $1\leq j\leq f-1$ and $\bm{\alpha}_0 = \mu_0 + \eta_0$, and that $s_{\tau} = s_0 s_1 \ldots s_{f-1} \in W$.   
We have the corresponding data for $\tau'$ which is the tame inertial type for $K'$ obtained as the restriction to $I_{K'}$ of $\tau$.
Namely, for any $j' \in \cJ'$, define 
\begin{equation}
\label{eq:alphaprime}
\bm{\alpha}'_{j + kf} \defeq  s_{\tau}^{-k} (\bm{\alpha}_{ j}) \text{ for } 0 \leq j \leq f-1, \, 0 \leq k \leq r-1.
\end{equation}
Next, for any $j' \in \cJ'$, define %
\begin{equation}
\label{eq:aprime}
\bf{a}^{\prime\,(j')} \defeq \sum_{i =0}^{f'-1} \bm{\alpha}'_{-j' + i} p^i.
\end{equation}
Note that if $\chi_i$ are the characters appearing in $\tau$ as in (\ref{eq:def:type}), then $\chi_i = \omega_{f'}^{\bf{a}_i^{\prime\,(0)}}$. 

We define the \emph{orientation} $s'_{\mathrm{or}} \in W^{\cJ'}$ of $(\bm{\alpha}'_{j'})_{j'\in \cJ'}$ by 
\begin{equation} \label{primeorient}
s'_{\mathrm{or}, j + kf} \defeq s_{\tau}^{k+1} (s^{-1}_{f-1} s^{-1}_{f-2} \ldots s^{-1}_{j+1}) \text{ for } 0 \leq j \leq f-1, \, 0 \leq k \leq r-1
\end{equation} 
where the empty product is interpreted as the identity.  
It is an element of $W^{\cJ'}$ such that $(s'_{\mathrm{or}, j'})^{-1}(\bf{a}^{\prime\,(j')})$ is dominant, and there is a unique such element if $\mu$ is $0$-generic.
This follows from the definitions of $s'_{\mathrm{or}, j'}$, $\bm{\alpha}'_{f'-1-j'}$, noting that $\bf{a}^{\prime\,(j')}$ is dominated by $p^{f'-1} \bm{\alpha}'_{f' -1 -j'}$.

Then
\begin{equation}
\label{eq:CtoA}
A^{(j')}_{\fM,\beta}\defeq 
\Ad\left(
(s'_{\orient,j'})^{-1}  (u^{\prime})^{-\bf{a}^{\prime\,(j')}}
\right)(C^{(j')}_{\fM,\beta})
\end{equation}
{is the \emph{matrix of the $j'$-th partial Frobenius of $\fM$ with respect to $\beta$}. (Note the different meaning of the superscript $(j')$ when comparing with the notion of matrix of partial Frobenii appearing in \cite[\S 3.6.1]{LLLM2}, \cite[\S 3.2]{LLL}, see Remark \ref{rmk:cmpr:mat} below.)}
\begin{rmk}
\label{rmk:cmpr:mat}
In the discussion between Definition 3.2.8 and Proposition 3.2.9 in \cite{LLL}, we find the definition of matrices $A^{(j')}$, attached to an eigenbasis $\beta$ for $\fM\in Y^{[0,h],\tau}(R)$ where $\tau$ is a tame inertial type with a given lowest alcove presentation $(s,\mu)$ in the sense of \cite[Definition 2.2.5(4)]{LLL}.
These matrices differ from those defined in equation (\ref{eq:CtoA}) by a shift due to a change in convention.  %
We now explain in detail the differences between the conventions in this paper, and those in \cite{LLL}, \cite{LLLM2}.

Let $\tau$ be the tame inertial type with lowest alcove presentation $(s,\mu)$, which we fixed at the beginning of this section.
Then the lowest alcove presentation of the tame inertial type \emph{in the sense of \cite[Definition 2.2.5(4)]{LLL}, \cite[Definition 2.2.2(4)]{LLLM2}} is the element $(s_{-},\mu_{-})\in W(G)^{\cJ}\times X^*(T)^{\cJ}$ defined by $s_{-,j}\defeq s_{f-j}$, $\mu_{-,j}\defeq \mu_{f-j}$.
Recall that in \cite[\S 3.2]{LLL} we associate elements $s_\tau$, $s'_{\orient,j'}$, $\bm{\alpha}'_{(s_{-},\mu_{-}),j + kf}$, and $\bf{a}^{\prime\,(j')}_{(s_{-},\mu_{-})}$  to $(s_{-},\mu_{-})$.
The comparison between the two conventions gives the following.
\begin{enumerate} 
\item 
\label{it:rmk:stau}
The element $s_\tau$ defined in Example \ref{ex:data:type} coincides with the element $s_\tau$ defined in \cite[\S 3.2]{LLL};
\item
the element $s'_{\orient,j'}$ defined in (\ref{primeorient}) coincides with the element $s'_{\orient,j'}$ defined in \cite[(3.2)]{LLL};
\item
the elements $\bm{\alpha}'_{j + kf}$, $\bf{a}^{\prime\,(j')}$ defined in (\ref{eq:alphaprime}), (\ref{eq:aprime}) respectively, coincide with the elements $\bm{\alpha}'_{(s_{-},\mu_{-}),j + kf}$, $\bf{a}^{\prime\,(j')}_{(s_{-},\mu_{-})}$, defined in \cite[\S 3.2]{LLL} and \cite[(3.4)]{LLL} respectively;
\item
the characters $\chi_i$ defined in (\ref{eq:def:type}) coincide with the characters $\chi_i$ defined in \cite[(3.1)]{LLL}; and
\item
\label{it:rmk:star}
for any $0\leq j\leq f-1$, we have $(s_{-})_{j}^*t_{(\mu_{-}+\eta)^*_{j}}=s^{-1}_{j+1}t_{\mu_{j+1}+\eta_{j+1}}$ where the $(\cdot)^*$ in the left hand side of the equality denotes the ``star'' operation defined in \cite[Definition 2.1.2]{LLL}, \cite[Definition 3.1.1]{LLLM2} (in particular, $(s_{-})_{j}^*=s_{-,f-1-j}^{-1}$, $(\mu_{-}+\eta)^*_{j}=\mu_{-,f-1-j}+\eta_{f-1-j}$).
\end{enumerate}
Since the partial Frobenius $\phi_\fM^{(j')}$ defined above is denoted as $\phi_\fM^{(j'-1)}$ in \cite[\S 3.2]{LLL}, \cite{LLLM2} we easily deduce from items (\ref{it:rmk:stau})--(\ref{it:rmk:star}) that the matrix $A^{(j')}_{\fM,\beta}$ defined in (\ref{eq:CtoA}) coincides with the matrix  $A^{(j'-1)}$ defined in \cite[\S 3.2]{LLL} (see the discussion after Definition 3.2.8 in \emph{loc.cit.}) with respect to the eigenbasis $\beta$ and $(s_{-},\mu_{-})$ as the fixed lowest alcove presentation (in the sense of \cite[Definition 2.2.5(4)]{LLL}).
\end{rmk}

Because $\tau$ is $1$-generic, the condition that $C^{(j')}_{\fM, \beta}\in \Mat_n(R[\![u']\!])$ is equivalent to $A^{(j')}_{\fM,\beta}  \in \Mat_n(R[\![v]\!])$ and is upper triangular modulo $v$ (equivalently $A^{(j')}_{\fM,\beta}  \in L^+ \cM_\cO(R)$ if $R$ is $p$-adically complete)
(This follows as in \cite[Proposition 2.13]{LLLM}, noticing by Remark \ref{rmk:cmpr:mat} that $C^{(j')}_{\fM, \beta}$ would be denoted as $C^{(j' + 1)}$ in \emph{loc.~cit.})
Similarly, the height condition translates into the condition that $E(v)^h (A^{(j')}_{\fM, \beta})^{-1} \in \Mat_n(R[\![v]\!])$ and is upper triangular modulo $v$. Just as with $C^{(j')}_{\fM, \beta}$, $A^{(j')}_{\fM, \beta}$ only depends on $j' \mod f$. Abusing notation, we occasionally write $A^{(j)}_{\fM, \beta}$ for $j \in \cJ$ with the obvious meaning.

The following Proposition is a reformulation of \cite[Proposition 3.2.9]{LLL} and describes how $A^{(j')}_{\fM, \beta}$ behaves under change of eigenbasis. 
\begin{prop} 
\label{prop:changeofbasis} $($\cite[Proposition 3.2.9]{LLL}$)$
Let $\fM \in Y^{[0,h], \tau}(R)$ together with two eigenbases $\beta_1$ and $\beta_2$ related by
\begin{equation*}
\beta_2^{(j')} D^{(j')} = \beta_1^{(j')}
\end{equation*}
with $D^{(j')} \in \GL_n(R[\![u']\!])$ for $j' \in \cJ'$. Set $I^{(j')} \defeq \Ad \big((s'_{\mathrm{or}, j'})^{-1} (u')^{-\mathbf{a}^{\prime \, (j')}} \big) (D^{(j')})$.

Then $I^{(j')}\in \Iw(R)$ depends only on $j' \mod f$, and for all $j' \in \cJ'$, 
\[
A_{\fM, \beta_2}^{(j')}= I^{(j')}A_{\fM, \beta_1}^{(j')} \big(\Ad(s_{j}^{-1} v^{\mu_{j}+\eta_{j}}) \big(\phz(I^{(j' - 1)})^{-1}\big)\big)
\]
where $j = j' \mod f$.

Furthermore, if $(I^{(j')}) \in \Iw(R)^{\cJ'}$ with $I^{(j')} = I^{(j'+f)}$, then $\Ad \big((u')^{\mathbf{a}^{\prime \, (j')}} s'_{\mathrm{or}, j'} \big) (I^{(j')}) = D^{(j')} \in \GL_n(R[\![u']\!])$ and for any eigenbasis $\beta$, $(\beta^{(j')} D^{(j')} )_{j' \in \cJ'}$ is again an eigenbasis. 
\end{prop}  
\begin{proof} 
By Remark \ref{rmk:cmpr:mat}, we see that the matrix $I^{(j')}$ defined above coincides with the matrix $I^{(j')}$ defined in the statement of \cite[Proposition 3.2.9]{LLL} (for which we use the lowest alcove presentation $(s_{-},\mu_{-})$ in the sense of \cite[Definition 2.2.1(iv)]{LLL} for $\tau$).
From the conclusion of Remark \ref{rmk:cmpr:mat}, and its item (\ref{it:rmk:star}), we see that the statement of the Proposition is just the statement of \cite[Proposition 3.2.9]{LLL} with $j'$ taken to be $j'-1$.
\end{proof}
\begin{defn}
\label{defn:shape} The \emph{shape} of a mod $p$ Breuil--Kisin module $\fM \in Y^{[0,h], \tau}(\F')$ with respect to $\tau$ is the element $\tld{z} = (\widetilde{z}_j) \in \widetilde{W}^{\vee, \cJ}$ such that for any eigenbasis $\beta$ and any $j \in \cJ$, the matrix $A^{(j)}_{\fM, \beta}$ lies in $\Iw(\F') \widetilde{z}_j \Iw(\F')$.
(This doesn't depend on the choice of eigenbasis by Proposition \ref{prop:changeofbasis}.)
\end{defn}

We record a useful and elementary lemma for later computations:
\begin{lemma} \label{lem:changeofbasis}  \label{lem:contracting1} 
Assume that $\tau$ admits an $m$-generic lowest alcove presentation $(s, \mu)$.
Let $R$ be an $\cO$-algebra.
\begin{enumerate}[(a)]
 \item 
\label{it:contracting1:a}  
Let $I \in \Iw_1(R)$. Then $\Ad(s^{-1}_j v^{\mu_j+\eta_j}) (\phz(I)) \equiv 1 \mod v^{m+1}.$   
\item If $I \in  \GL_n(R[\![v]\!])$, $m \geq 1$, and $I \equiv 1 \mod v^k$, then $\Ad(s^{-1}_j v^{\mu_j+\eta_j})(\phz(I)) \equiv 1 \mod v^{(k-1)p + m + 1}.$   
\item If $Y \in  \Mat_n(R[\![v]\!])$ and  is upper triangular mod $v$, then $\Ad(s^{-1}_j v^{\mu_j+\eta_j})(Y) \in v^{m+1} \Mat_n(R[\![v]\!])$.
\item If $Y \in v^k \Mat_n(R[\![v]\!])$, then   $\Ad(s^{-1}_j v^{\mu_j+\eta_j})(\phz(Y)) \in v^{(k-1)p + m + 1} \Mat_n(R[\![v]\!])$. %
\end{enumerate}
\end{lemma}
\begin{proof} 
We provide a proof of item (\ref{it:contracting1:a})  and leave the rest to the reader.  
It suffices to prove that $\Ad(v^{\mu_j+\eta_j}) (\phz(I)) \equiv 1 \mod v^{m+1}$. Recall (cf.~ Definition \ref{defn:var:gen}) that since $(s, \mu)$ is an $m$-generic lowest alcove presentation, for any $\alpha \in \Phi^+$, 
\[
m < \langle \mu + \eta, \alpha^{\vee} \rangle < p - m. 
\] 
In particular, for such a $\mu$ to exist  we need $m + 1 < p$.

The diagonal entries of $\Ad(v^{\mu_j+\eta_j}) (\phz(I))$ are the same as the diagonal entries of $\phz(I)$, which are congruent to 1 mod $v^p$ since $I \in \Iw_1(R)$. For the off-diagonal entries, let $\alpha$ be a positive root.  The $\alpha$-entry of $\Ad(v^{\mu_j+\eta_j}) (\phz(I))$ is divisible by $v^{ \langle \mu + \eta, \alpha^{\vee} \rangle}$.  Similarly, the $-\alpha$-entry of $\phz(I)$ is divisible by $v^p$ and so  the $-\alpha$ entry of $\Ad(v^{\mu_j+\eta_j}) (\phz(I))$ is divisible by $v^{ p - \langle \mu + \eta, \alpha^{\vee} \rangle}$.  This gives the desired divisibility.
\end{proof}

\subsection{Gauge bases}
\label{sub:GBases}  
The goal of this subsection is to discuss the notion of gauge basis, which will provide a normal form for various families of Breuil--Kisin modules of type $\tau$, and which will be our main tool to analyze the structure of the $p$-adic completion of the stack $Y^{[0,h],\tau}$ in the next subsection. For $\tau$ sufficiently generic relative to $h$, we will define the notion of a Breuil--Kisin module admitting a $\tld{z}$-gauge, for $\tld{z}\in \tld{W}^{\vee,\cJ}$. This is an open condition in the moduli of Breuil--Kisin modules, and thus is stable under small deformations. We then show that such Breuil--Kisin modules admit a canonical basis adapted to $\tld{z}$, which is unique up to rescaling. An innovation compared to our previous work \cite{LLL}, \cite{LLLM} is that we do not just consider canonical bases associated to the shape of a Breuil--Kisin module (which is an element of $\tld{W}^{\vee,\cJ}$ canonically attached to each closed point of $Y^{[0,h],\tau}$). The stratification by shape decomposes $Y^{[0,h],\tau}$ into a disjoint union of locally closed substack, and hence the property of having constant shape is not preserved under small deformations. For this reason,  our approach here is better suited for the local study of $Y^{[0,h],\tau}$. 

Recall that we have the twisted loop group $L\cG_{\cO}$, the twisted positive loop group $L^+\cG_{\cO}$ and the space $L^+\cM_\cO$. Let $a\leq b$ be integers. We let $L^{[a,b]}\cG_{\cO}$ be the subfunctor of $L\cG_{\cO}$ whose value on a Noetherian $\cO$-algebra is given by given by 
\[L^{[a,b]}\cG_{\cO}(R)=\{g\in L\cG_\cO(R)\, |\, g\in (v+p)^aL^+\cM_\cO(R) \textrm{ and } (v+p)^bg^{-1}\in L^+\cM_\cO(R)\}.\]
Clearly $L^{[0,h]}\cG_\cO$ is preserved by left and right multiplication by $L^+\cG_\cO$, and we define
\[\Gr_{\cG,\cO}^{[a,b]}\defeq L^+\cG_\cO\backslash L^{[a,b]}\cG_\cO.\]
Let now $L^{[a,b]} (\GL_n)_{\F}$ be the subfunctor of $L(\GL_n)_{\F}$ defined on $\F$-algebras $R$ by
\begin{equation} \label{defn:GLnAB}  
L^{[a,b]} (\GL_n)_{\F}(R) \defeq \{A \in L  (\GL_n)_{\F}(R)\mid v^{-a} A,\, v^b A^{-1}  \in \Mat_n(R[\![v]\!]) \}.
\end{equation}
The fpqc-sheafification of $R \mapsto \cI_{\F}(R) \backslash L^{[a,b]}  (\GL_n)_{\F}(R)$ is a finite type closed subscheme $\Fl^{[a, b]} \subset \Fl$ {(consider the natural projection of $\Fl$ onto the affine Grassmannian for $\GL_n$ over $\F$, and for the latter use \cite[Lemma 1.1.5]{zhu-intro-AG})}.   
Base changing to $\bF$, we get $\Gr_{\cG,\bF}^{[a,b]}= \cI_\bF\backslash L^{[a,b]}\cG_\bF\subset \Fl^{[a,b]}$ (where the containment is strict). We also define $\tld{\Gr}^{[a,b]}_{\cG,\bF}=\cI_{1,\bF}\backslash L^{[a,b]}\cG_\bF\subset \tld{\Fl}^{[a,b]}$, which is the pullback of the previous situation to $\tld{\Fl}$ (as in \S \ref{sec:Ttorsors}). We evidently have $(v+p)^m\Gr_{\cG,\cO}^{[a,b]}=\Gr_{\cG,\cO}^{[a+m,b+m]}$.

We first give a presentation of the $p$-adic completion of the stack $Y^{[0,h],\tau}$ as a quotient stack.
Given a pair $(s,\mu)\in W^\cJ\times X^*(T)^\cJ$, we define the $(s,\mu)$-twisted $\phz$-conjugation action of $(L^+\cG_\cO)^{\cJ}$ on $(L\cG_{\cO})^{\cJ}$ by 
\[(I^{(j)})_j\cdot (A^{(j)})_j=I^{(j)}A^{(j)}\big(\Ad(s^{-1}_jv^{\mu_j+\eta_j})\big(\phz(I^{(j-1)})^{-1}\big)\big)\]
Similarly, we define the $(s,\mu)$-twisted conjugation action by the above formula, but with the $\phz$ dropped. 
The following is essentially a reformulation of Proposition \ref{prop:changeofbasis}:
\begin{prop}\label{prop:Breuil-Kisin_moduli_presentation} Let $(s,\mu)$ be a lowest alcove presentation of $\tau$. Then there is a canonical isomorphism of $p$-adic formal $\cO$-stacks $Y^{[0,h], \tau}\cong [(L^{[0,h]}\cG_\cO)^{\cJ}/_{\phz,(s,\mu)}(L^+\cG_\cO)^{\cJ}]^{\wedge_p}$, where the action is the $(s,\mu)$-twisted $\phz$-conjugation action.
\end{prop}
\begin{proof} Consider the groupoid $Y^{[0,h],\tau,\beta}$ parametrizing pairs $(\fM,\beta)$ where $\fM\in Y^{[0,h],\tau}$ and $\beta$ is an eigenbasis of $\fM$. There is a map $Y^{[0,h],\tau,\beta}\to (L\cG_\cO)^{\cJ}$ given by sending $(\fM,\beta)$ to the collection of matrices of partial Frobenii $(A^{(j)}_{\fM,\beta})_{j\in \cJ}$. %
The condition that $\fM\in Y^{[0,h],\tau}$ is equivalent to the condition $ (A^{(j)}_{\fM,\beta})\in (L^{[0,h]}\cG_\cO)^{\cJ}$. For $R$ a $p$-adically complete Noetherian $\cO$-algebra, Proposition \ref{prop:changeofbasis} shows that the set of eigenbases on a given $\fM\in Y^{[0,h],\tau}(R)$ is a torsor for $\cI(R)^{\cJ}=L^+\cG_\cO(R)^{\cJ}$, and the action of $(L^+\cG_\cO(R))^{\cJ}$ corresponds to the $(s,\mu)$-twisted $\phz$-conjugation action on $L\cG^{\cJ}_\cO(R)$ under the above map. Thus 
$[(L^{[0,h]}\cG_\cO)^{\cJ}/_{(s,\mu),\phz}(L^+\cG_\cO)^{\cJ}]^{\wedge_p}$ is the substack of $Y^{[0,h],\tau}$ consisting of objects which fpqc-locally admits an eigenbasis. However, every object $\fM\in Y^{[0,h],\tau}(R)$ has this property: Zariski locally on $R$, we can find a basis for $\fM/u'\fM$, which furthermore consists of eigenvectors for $\Delta'$ and is compatible with $\iota_{\fM}$ modulo $u'$. Such a basis can be lifted to an eigenbasis of $\fM$, since $\Delta'$ has order prime to $p$.
\end{proof}

The following Lemma shows that over $\bF$, the $(s,\mu)$-twisted $\phz$-conjugation action can sometimes be ``straightened'' to a left translation action, at least on the subgroup $(\cI_1)^{\cJ}$.  %
\begin{lemma} %
 \label{lem:iotawelldef}Let $R$ be an $\F$-algebra and $(A^{(j)}_1)_{j \in \cJ},  (A^{(j)}_2)_{j \in \cJ} \in  L^{[0,h]} \GL_n(R)^{\cJ}$.  Let  $\tld{z} =s^{-1}t_{\mu+\eta}\in \tld{W}^{\vee,\cJ}$ where $\mu$ is $(h+1)$-deep in $\un{C}_0$ and $s \in W^\cJ$.  Then, there is a bijection between the following:
\begin{enumerate}
\item
\label{lem:eq:mat:1}
Tuples $(I^{(j)})_{j \in \cJ} \in \cI_1(R)^{\cJ}$ such $A^{(j)}_2 \tld{z}_j = I^{(j)} A^{(j)}_1 \tld{z}_j \phz(I^{(j-1)})^{-1}$ for all $j \in \cJ$;
\item
\label{lem:eq:mat:2}
Tuples $(X_j)_{j \in \cJ} \in \cI_1(R)^{\cJ}$ such that 
$A^{(j)}_2=X_jA^{(j)}_1$ for all $j \in \cJ$.  
\end{enumerate} 
\end{lemma} 
\begin{proof} 
Throughout the proof, we will use that
\begin{equation} \label{eq:f1}
 \Ad(\tld{z}_j)\big(\phz(I^{(j-1)})^{-1}\big) \equiv 1 \mod v^{h+2}
\end{equation}   
for any $I^{(j-1)} \in \Iw_1(R)$, by Lemma  \ref{lem:changeofbasis}.

We give a map $F$ from (\ref{lem:eq:mat:1}) to (\ref{lem:eq:mat:2}). Given the data in (\ref{lem:eq:mat:1}), we define $F((I^{(j)})_{j \in\cJ})=(X_j)_{j\in\cJ}$, where
\[
X_{j}\defeq  A^{(j)}_{2}(A^{(j)}_{1})^{-1}= I^{(j)}A^{(j)}_{1} \big(\Ad(\tld{z}_j)\big(\phz(I^{(j-1)})^{-1}\big)\big)(A^{(j)}_{1})^{-1}
\]
To check that $(X_j)_{j\in\cJ}$ satisfies (\ref{lem:eq:mat:2}), we only need to check $X_j\in \cI_1(R)$ for all $j\in \cJ$.
By \eqref{eq:f1}, we can write $\Ad(\tld{z}_j)\big(\phz(I^{(j-1)})^{-1}\big) = 1 +v^{h+2}Y_j$ with $Y_j\in \Mat_n(R[\![v]\!])$.  
By the height condition on $A^{(j)}_{1}$, we deduce that
\[X_j= I^{(j)}(1 +v^{h+2}A^{(j)}_{1} Y_j(A^{(j)}_{1})^{-1})\in \cI_1(R)\]
as desired.

Next, we construct a map $G$ from (\ref{lem:eq:mat:2}) to (\ref{lem:eq:mat:1}).
Thus we are given $X_j\in \cI_1(R)$ such that $A_2^{(j)}=X_j A_1^{(j)}$ for all $j\in \cJ$, and we need to construct a solution $I^{(j)}\in\cI_1(R)^{\cJ}$ to the system of equations
\[X_{j} A_1^{(j)} \big(\Ad(\tld{z}_j)\big(\phz(I^{(j-1)})\big)\big) (A_1^{(j)})^{-1} = I^{(j)}. \]

We construct such a solution as a limit of a convergent sequence in $\cI_1(R)$ with the $v$-adic topology.
Let $J^{(j)}_0 = \mathrm{Id}$.   
For $i \geq 0$, set 
\[
J^{(j)}_{i+1} = X_j A^{(j)}_1 \Big(A^{(j)}_1 \Ad(\tld{z}_j)\big(\phz(J^{(j-1)}_i)^{-1}\big)\Big)^{-1}
\]

We prove that $\big(J^{(j)}_i\big)_i$ converges in $\Iw_1(R)$ in the $v$-adic topology.
For $i \geq 1$, we have
\[
J^{(j)}_{i+1} - J^{(j)}_i = X_j A^{(j)}_1 \Big(\Ad(\tld{z}_j) \big(\phz(J^{(j-1)}_{i} - J^{(j-1)}_{i-1})\big)\Big) (A^{(j)}_1)^{-1}.
\]
Since $J^{(j)}_1 = X_j$, by Lemma  \ref{lem:changeofbasis} we have $\Ad(\tld{z}_j)\big(\phz(J^{(j-1)}_{1} - J^{(j-1)}_{0})\big) \in v^{h+2} \Mat_n(R[\![v]\!])$ and hence $A^{(j)}_1 \tld{z}_j \phz(J^{(j-1)}_{1} - J^{(j-1)}_{0}) \tld{z}_j^{-1}  (A^{(j)}_1)^{-1} \in v^2 \Mat_n(R[\![v]\!])$ by the height condition.
Therefore $J^{(j)}_{2} - J^{(j)}_1\in v^2 \Mat_n(R[\![v]\!])$.  
We conclude by induction on $i$ that
\[
J^{(j)}_{i} - J^{(j)}_{i-1} \in v^{p(i-2)} \Mat_n(R[\![v]\!]).
\]
for $i \geq 3$ and hence the sequence $\big(J^{(j)}_i\big)_i$ converges in $\Iw_1(R)$ in the $v$-adic topology.

We construct the map $G$ by $G((X_j)_{j\in\cJ})=(\lim_{i\to \infty} J^{(j)}_i)_{j\in\cJ}$. By construction, the composition $F\circ G$ is the identity, thus $F$ is surjective. To finish the proof, we just need to show that $F$ is injective.

 Suppose that $F((I^{(j)})_{j\in\cJ})=F((J^{(j)})_{j\in\cJ})=(X_j)_{j\in\cJ}$, then $F((J^{(j)})^{-1}I^{(j)})_{j\in\cJ})=(1)_{j\in\cJ}$. Thus we may assume without loss of generality that $J^{(j)}=X_j=1$.

We now have
\[ I^{(j)}=A_1^{(j)}\big(\Ad(\tld{z}_j)(\phz(I^{(j-1)})^{{-1}})\big)\big(A_1^{(j)}\big)^{-1}\]
for all $j$, and we need to show that $I^{(j)}=1$ for all $j$.
By \eqref{eq:f1}, we have $A_1^{(j)}  \big(\Ad(\tld{z}_j) (\phz(I^{(j-1)})^{-1})\big)\big(A^{(j)}\big)^{-1} \equiv 1$ mod $v^2$ for all $j$, thus $I^{(j)} \equiv 1$ mod $v^2$ for all $j$. Suppose that for all $j\in\cJ$, $I^{(j)} \equiv 1$ mod $v^\delta$ for some $\delta\geq 2$. Then, by Lemma \ref{lem:changeofbasis},
\[
\Ad(\tld{z}_j) (\phz(I^{(j-1)})^{-1}) \equiv 1 \text{ mod } v^{p (\delta - 1) + h + 2}.  
\]
  Hence $I^{(j)} \equiv 1$ mod $v^{p(\delta -1)+ 2}$.  Since $p(\delta -1)+ 2>\delta$, this shows that $I^{(j)} \equiv 1$ modulo arbitrary high powers of $v$ for all $j$. This shows $I^{(j)}=1$ for all $j$.
\end{proof}

\begin{cor} \label{cor:Breuil-Kisin_moduli_model}Suppose $(s,\mu)$ is a lowest alcove presentation of $\tau$ such that $\mu$ is $(h+1)$-deep in $\un{C}_0$. 
Then forming matrices of partial Frobenii with respect to an eigenbasis induces an isomorphism of $\bF$-stacks 
\[\pi_{(s,\mu)}:Y^{[0,h],\tau}_{\bF}\cong [(\tld{\Gr}^{[0,h]}_{\cG,\bF})^{\cJ}/_{(s,\mu)}T^{\vee,\cJ}_{\F}]\subset [(\tld{\Fl}^{[0, h]})^{\cJ}/_{(s,\mu)}T^{\vee,\cJ}_{\F}],\]
 where the action of the constant torus $T^{\vee,\cJ}_{\F}\subset (L^+\cG_\bF)^{\cJ}$ is the $(s,\mu)$-twisted conjugation action.
\end{cor}
\begin{proof} This follows immediately from Proposition \ref{prop:Breuil-Kisin_moduli_presentation}, Lemma \ref{lem:iotawelldef}, the fact that $\phz$ acts trivially on $T^{\vee}$ and that $\cI=T^\vee \ltimes \cI_1$.
\end{proof}
We construct an open cover for $Y^{[0,h],\tau}$ using the above isomorphism. Recall from before Lemma \ref{lem:open_chart_mono} that for any $\tld{z} = (\tld{z}_j)_{j \in \cJ} \in \tld{W}^{\vee,\cJ}$, we have a subfunctor $\cU(\tld{z})\defeq \prod_{j\in\cJ} \cU(\tld{z}_j) \subset L\cG^{\cJ}$. 
For any integers $a\leq b$, we define 
\[
U^{[a,b]}(\tld{z}) \defeq \cU(\tld{z})_\cO\cap (L^{[a,b]}\cG_\cO)^\cJ.
\]
It follows from Lemma \ref{lem:open_chart_mono} that the projection map $U^{[a,b]}(\tld{z})\to \Gr_{\cG,\cO}^{[a,b], \cJ}$ is an open immersion, since $\Gr_{\cG,\cO}^{[a,b], \cJ}$ is a finite type $\cO$-scheme. 
Hence $\cU(\tld{z})_\bF \to \Gr_{\cG,\bF}^\cJ=\Fl^\cJ$ is an open immersion with $T^{\vee,\cJ}_\bF$-stable image. Since $\Fl^\cJ$ is ind-proper and all the $T_\bF^{\vee,\cJ}$-fixed points (for the right translation action) are given by $\tld{z}\subset \tld{W}^{\vee,\cJ}$, we conclude that (the images of) $\cU(\tld{z})_\bF$ form an open cover of $\Fl^{\cJ}$. 
We also set $\tld{\cU}(\tld{z})=T^{\vee,\cJ}\cU(\tld{z})\subset L\cG^\cJ$, and similarly define $\tld{U}^{[a,b]}(\tld{z}) = T^{\vee,\cJ}_{\cO} U^{[a,b]}(\tld{z})$. Then (the images of) the $\tld{\cU}(\tld{z})_\bF$ form an open cover of $\tld{\Fl}^\cJ$, in fact this is the pullback of the above open cover of $\Fl^\cJ$ to $\tld{\Fl}^\cJ$. We thus get the open cover $U^{[a,b]}(\tld{z})_\bF$ (resp. $\tld{U}^{[a,b]}(\tld{z})_\bF$) of  $\Gr^{[a,b], \cJ}_{\cG,\bF}$ (resp. $\tld{\Gr}^{[a,b], \cJ}_{\cG,\bF}$).

\begin{defn} 
\label{defn:open:substacks:BK}
Let $\tld{z}\in\tld{W}^{\vee,\cJ}$. %
\begin{enumerate}
\item Define $Y_\bF^{[0,h],\tau}(\tld{z})$ to be the open substack of $Y_\bF^{[0,h],\tau}$ which corresponds via $\pi_{(s,\mu)}$ to the open substack $[\tld{U}^{[0,h]}(\tld{z})_\bF/_{(s,\mu)} T_\bF^{\vee,\cJ}]$ of $[\Gr_{\cG,\bF}^{[0,h], \cJ}/_{(s,\mu)} T_\bF^{\vee,\cJ}]$.
\item More generally, we define the $p$-adic formal stack $Y^{[0,h],\tau}(\tld{z})$  as the open substack of $Y^{[0,h],\tau}$ induced by $Y_\bF^{[0,h],\tau}(\tld{z})$. For $R$ a $p$-adically complete Noetherian $\cO$-algebra,  $\fM\subset Y^{[0,h],\tau}(R)$ is said to admit a $\tld{z}$-gauge if $\fM\in Y^{[0,h],\tau}(\tld{z})(R)$, or equivalently, $\fM/\varpi\fM \in Y_\bF^{[0,h],\tau}(\tld{z})(R/\varpi R)$.
\end{enumerate}
\end{defn}
\begin{rmk} Let $\bF'$ be a field extension of $\bF$. If $\fM\in Y^{[0,h],\tau}(\bF')$ has shape $\tld{z}$ (Definition \ref{defn:shape}), then $\fM$ admits a $\tld{z}$-gauge.
\end{rmk}

Let $R$ be a Noetherian $\bF$-algebra and $\fM\in Y_\bF^{[0,h],\tau}(R)$. Assume that $\fM$ admit an eigenbasis. Then the condition that $\fM$ admits a $\tld{z} = (\tld{z}_j)_{j \in \cJ}$-gauge is equivalent to the condition that for any eigenbasis $\beta$, 
the matrix of partial Frobenius $A^{(j)}_{\fM,\beta}$ belongs to $\cI(R) (\cU(\tld{z}_j)(R)) = T^\vee(R) (\cI_1(R)) (\cU(\tld{z}_j)(R))$ for all $j \in \cJ$. Furthermore, in this case Proposition \ref{lem:iotawelldef} shows that we can adjust the eigenbasis $\beta$ so that $A^{(j)}_{\fM,\beta}\in T^\vee(R)\cU(\tld{z}_j)(R)$. 
 The proof of Proposition \ref{lem:iotawelldef} and the fact that the multiplication map $\cI(R)\times \cU(\tld{w})(R)\to L\cG_\cO(R)$ is an injection shows that an eigenbasis $\beta'$ has this property if and only if it is obtained from $\beta$ by a change of basis given by $\{ (t_{j'})_{j'\in\cJ'}\in {T^{\vee}_{\F}(R)} \mid t_{j'}=t_{k'} \textrm{ for }j'\equiv k' \textrm{ mod } f\} \cong T^{\vee,\cJ}_{\bF}(R)$. Thus the set of eigenbases with this property form a torsor for the group $T_\bF^{\vee,\cJ}$. This motivates the following definition:
\begin{defn}\label{defn:gauge_basis} Let $\tld{z} = (\tld{z}_j)_{j \in \cJ} \in  \tld{W}^{\vee, \cJ}$.  Let $R$ be a Noetherian $p$-adically complete $\cO$-algebra, and assume $\fM\in Y^{[0,h],\tau}(R)$ admits a $\tld{z}$-gauge. An eigenbasis $\beta$ of $\fM$ is called a $\tld{z}$-gauge basis if the matrix of partial Frobenii $A^{(j)}_{\fM,\beta}\in T^{\vee}(R) (\cU(\tld{z}_j)(R))$ for all $j\in \cJ$.
\end{defn} 

\begin{prop}\label{prop:gauge_basis_existence} Let $(s,\mu)$ be a lowest alcove presentation of $\tau$ where $\mu$ is $(h+1)$-deep in $\un{C}_0$. Let $R$ be a Noetherian $p$-adically complete $\cO$-algebra and $\fM\in Y^{[0,h],\tau}(R)$. Assume that $\fM$ admits both a $\tld{z}$-gauge and an eigenbasis. Then $\fM$ admits a $\tld{z}$-gauge basis $\beta$, which is uniquely determined up to scaling by the group $\{ (t_{j'})_{j'\in\cJ'}\in {T^{\vee, \cJ'}(R)} \mid t_{j'}=t_{k'} \textrm{ for }j'\equiv k' \textrm{ mod } f\}=T^{\vee,\cJ}(R)$.
\end{prop}
\begin{proof} It suffices to prove the Proposition for $R$ a Noetherian $\cO/\varpi^a$-algebra, for any $a\geq 1$. We already observed that the Proposition holds when $a=1$. Thus we have a $\tld{z}$-gauge basis $\ovl{\beta}$ of $\fM/\varpi^{a-1}\fM\in Y^{[0,h],\tau}(R/\varpi^{a-1})$. Let $\tld{\beta}$ be any eigenbasis of $\fM$ lifting $\ovl{\beta}$. We also set $R_\bF=R/\varpi$, $\fM_\F=\fM/\varpi \fM$ and $\beta_\F=\tld{\beta}$ mod $\varpi$.

We set $\tld{A}^{j}=A^{(j)}_{\fM,\tld{\beta}}$, $\ovl{A}^{(j)}=A^{(j)}_{\fM/\varpi^{a-1}\fM,\ovl{\beta}}$, $A^{(j)}_\F=A^{(j)}_{\fM_\F,\beta_\F}$. We get the square-zero extension $R\onto R/\varpi^{a-1}$ with kernel $J=\varpi^{a-1}R$. 
As in \S \ref{sec:affine:charts}, %
we have $R/\varpi^{a-1}$-modules $\Lie L\cG_{\cO}(J)$, $\Lie\cI(J) \defeq \Lie L^+\cG_{\cO}(J)$, $\Lie L^{--}\cG_{\cO}(J)$, which are in fact $R_\bF$-modules. We also have the obvious variants $\Lie T^\vee(J)$, $\Lie \cI_1(J)$. Note that $\Lie \cI(J)=\Lie T^\vee(J)\oplus \Lie \cI_1(J)$.
By Proposition \ref{prop:changeofbasis}, given the choice of eigenbasis $\tld{\beta}$, the set of eigenbases $\beta$ of $\fM$ lifting $\ovl{\beta}$ is in bijection with the set of tuples $(X_{j'})_{j'\in\cJ'}\in\cI(R)^{\cJ'}$ such that
\begin{itemize}
\item $X_{j'}$ depends only on the image of $j'$ in $\cJ$. 
\item $X_j\in \ker(L^+\cG_{\cO}(R)\to L^+\cG_{\cO}(R/\varpi^{a-1}))$, i.e.~$Y_j\defeq X_j-1\in \Lie \cI(J)$ for all $j\in \cJ$.
\end{itemize} 
We thus need to analyze the set of tuples $(X_j)_{j\in \cJ}$ as above such that 
\[X_{j} \tld{A}^{(j)} \big(\Ad(s^{-1}_jv^{\mu_j+\eta_j})\big( \phz(X_{j-1})^{-1}\big) \big) \in T^{\vee}(R) ( \cU(\tld{z}_{j})(R)). \]
Note that $\cU(\tld{z}_{j}) (R) = L^{--} \cG_{\cO}(R) \tld{z}_j$ and similarly for $R/\varpi^{a-1}$. 
By construction, we have $\ovl{A}^{(j)}=\ovl{D}_j\ovl{U}_j\tld{z}_j$, where $\ovl{D}_j\in T^\vee(R/\varpi^{a-1})$ and $\ovl{U}_j\in L^{--}\cG_{\cO}(R/\varpi^{a-1})$. Since $T^{\vee}$ and $L^{--}\cG_\cO$ are formally smooth, we can find lifts $\tld{D}_j\in T^\vee(R)$ and $\tld{U}_j\in L^{--}\cG_{\cO}(R)$ of $\ovl{D}_j$ and $\ovl{U}_j$ respectively. Thus, 
\[
\tld{A}^{(j)} = (1 +  \mathfrak{a}_j) \tld{D}_j \tld{U}_j \tld{z}_j 
\]
where $\mathfrak{a}_j \in \Lie L \cG_{\cO}(J)$.   

We record the effect of $(s,\mu)$-twisted $\phz$-conjugation by $(X_j)_{j\in\cJ}=(1+Y_j)_{j\in\cJ}$.  Namely, if we write, 
\[
X_{j} \tld{A}^{(j)} \big(\Ad(s^{-1}_jv^{\mu_j+\eta_j})\big( \phz(X_{j-1})^{-1}\big) \big) = (1+ \mathfrak{a}'_j)  \tld{D}_j \tld{U}_j \tld{z}_j 
\]
then we find that 
\[
\mathfrak{a}'_j = Y_{j} + \mathfrak{a}_j - \Ad( A_\bF^{(j)} s^{-1}_j v^{\mu_j+\eta_j}) (\phz(Y_{j-1}))
\]
in $\Lie L\cG_{\cO}(J)$.

Since the set of elements in $T^\vee(R)\cU(\tld{z})(R)$ lifting $\ovl{D}_j\ovl{U}_j\tld{z}_j$ are exactly those of the form $(1+ \mathfrak{a}'_j)  \tld{D}_j \tld{U}_j \tld{z}_j$ where $\mathfrak{a}'_j\in (\Lie L^{--}\cG_{\cO}(J) \oplus \Lie T^\vee(J))$, our job boils down to analyzing the set of solutions $(Y_j)\in \Lie L^+\cG_{\cO}(J)^\cJ$ to the system of containments
\begin{equation}\label{eq:solution_gauge_basis}
 Y_{j} + \mathfrak{a}_j - \Ad( A_\bF^{(j)} s^{-1}_j v^{\mu_j+\eta_j}) (\phz(Y_{j-1}))\in  (\Lie L^{--}\cG(J)\oplus \Lie T^\vee(J)), \quad j\in \cJ.
\end{equation}
Observe that the set of solutions to the system (\ref{eq:solution_gauge_basis}) is invariant under translation by $\Lie T^\vee(J)^\cJ$: if $Y_j\in \Lie T^\vee(J)$ then $\Ad( A_\bF^{(j)} s^{-1}_j v^{\mu_j+\eta_j}) (\phz(Y_{j-1}))=\Ad(\ovl{D}_j\ovl{U}_j)\big(\!\Ad(s^{-1}_j)(Y_{j-1})\big)\subset \Ad(\ovl{D}_j\ovl{U}_j) \big(\Lie T^\vee(J)\big)\subset \Lie L^{--}\cG_{\cO}(J) \oplus \Lie T^\vee(J)$,
where the last inclusion follows from the fact that $T^\vee_{\cO} \cdot L^{--}\cG_{\cO}$ is a subgroup of $L\cG_{\cO}$. Thus, to finish the induction, we only need to show that the system (\ref{eq:solution_gauge_basis}) has a unique solution in $\Lie \cI_1(J)^\cJ$.

Now, it follows from Lemma \ref{lem:Lie_algebra_decomp} that  $\Lie L \cG_{\cO}(J) = \Lie L^{--}\cG_{\cO}(J)\oplus \Lie T^{\vee}(J) \oplus \Lie \cI_1(J)$ where $\Lie \cI_1(J)=\left\{\,  M\in M_n(J[\![{v}]\!]), 
\text{\small{ $M$ is unipotent upper triangular mod $v$}}
\right\}.$
Consider the map $\Psi:  \Lie \cI_1(J)^{\cJ} \ra  \Lie L \cG(J)^{\cJ}$ given by $(Y_j)_{j \in \cJ} \mapsto (Y_{j}  - \Ad( A_\bF^{(j)} s^{-1}_j v^{\mu_j+\eta_j}) (\phz(Y_{j-1}))$. 
 By Lemma \ref{lem:contracting1},  $\Ad(  s^{-1}_j v^{\mu_j +\eta_j})(\phz(Y_{j-1})) \equiv 0 \mod v^{h+2}$ and so, by the height condition on $\fM_\bF$, we have: 
\[
\Ad( A_\bF^{(j)}s^{-1}_j v^{\mu_j+\eta_j}) (\phz(Y_{j-1})) \equiv 0 \mod v^2. 
\]
Thus the image of $\Psi$ is in $\Lie\cI_1(J)^\cJ$.
 Furthermore, if $Y_j \equiv 0 \mod v^k$ for $k \geq 0$, then 
 \[
 \Ad( A^{(j)}_{\ovl{\fM}, \ovl{\beta}} s^{-1}_j v^{\mu_j+\eta_j}) (\phz(Y_{j-1})) \equiv 0 \mod v^{k+1}
 \] by  Lemma \ref{lem:contracting1}. Hence, as an endomorphism of $\Lie \cI_1(J)^\cJ$, 
 $\Psi$ decomposes as a sum of an automorphism and a topologically nilpotent endomorphism. We conclude then that $\Psi$ itself is an automorphism.  
We thus conclude that the system (\ref{eq:solution_gauge_basis}) has a unique solution in $\cI_1(J)^\cJ$, namely $\Psi^{-1}$ of the the projection of $(-\mathfrak{a}_j)_{j\in \cJ}$ onto $\Lie \cI_1(J)^\cJ$.
\end{proof}

\subsection{Local models for moduli of Breuil--Kisin modules}
\label{subsec:LMofBK}
In this section, we describe the local structure of the $p$-adic formal stack $Y^{[0,h],\tau}$ and its closed substack $Y^{\leql,\tau}$.

We have the following mixed characteristic generalization of Corollary \ref{cor:Breuil-Kisin_moduli_model}:
\begin{thm}\label{thm:fh_local_model} %
Let $(s,\mu)$ be a lowest alcove presentation of $\tau$ such that $\mu$ is $(h+1)$-deep in $\un{C}_0$, and let $\tld{z}\in \tld{W}^{\vee,\cJ}$. 
Then there is a local model diagram \emph{(}depending on $(s,\mu)$\emph{)} of $p$-adic formal $\cO$-stacks
\begin{equation}\label{eq:fh_local_model}
\begin{tikzcd}
\phantom{M} &  \tld{U}^{[0,h]}(\tld{z})^{\wedge_p} \ar{dl}[swap]{T^{\vee,\cJ}_{\cO}} \ar{dr}{T^{\vee,\cJ}_{\cO}}& \phantom{M}\\
Y^{[0,h],\tau}(\tld{z})=\Big[\tld{U}^{[0,h]}(\tld{z})/_{(s,\mu)} T^{\vee,\cJ}_{\cO}\Big]^{\wedge_p} \ar[hook,open]{d} & \phantom{M} & U^{[0,h]}(\tld{z})^{\wedge_p}\ar[hook,open]{d} \\
Y^{[0,h],\tau} & \phantom{M} & \Big(\Gr_{\cG,\cO}^{[0,h], \cJ}\Big)^{\wedge_p}
\end{tikzcd}
\end{equation}
where 
\begin{itemize}
\item The left diagonal arrow corresponds to extracting a $\tld{z}$-gauge basis and taking its matrices of partial Frobenii.
\item The diagonal arrows are torsors for the ($p$-adic completion of) $T^{\vee,\cJ}_{\cO}$ for two different $T^{\vee,\cJ}_{\cO}$-actions (and hence are smooth maps): The left diagonal arrow correspond to quotiening by the $(s,\mu)$-twisted conjugation action while the right diagonal arrow correspond to quotiening by the left translation action.
\item The vertical arrows are open immersion.
\end{itemize}
\end{thm}

\begin{proof} The left side of the diagram follows from Proposition \ref{prop:gauge_basis_existence} and the existence of an eigenbasis Zariski locally. The right side of the diagram is a consequence of Lemma \ref{lem:negative_loop_mono}, Lemma \ref{lem:open_chart_mono} and the fact that $\Gr_{\cG,\cO}^{[0,h]}$ is a finite type $\cO$-scheme. %
\end{proof}

\begin{warning}
 As $\tld{z}$ varies in $\tld{W}^{\vee,\cJ}$, the $Y^{[0,h],\tau}(\tld{z})$ form a Zariski open cover
 of $Y^{[0,h],\tau}$. Lemma \ref{lem:iotawelldef} shows that over $\bF$, these local model diagram glue together to give a local model diagram for $Y^{[0,h],\tau}_{\F}$, cf.~Corollary \ref{cor:Breuil-Kisin_moduli_model}. However, the local model diagrams do not glue together into a local model diagram for $Y^{[0,h],\tau}$ in general. The reason is that Lemma \ref{lem:iotawelldef} fails over test rings $R$ where $p\neq 0$, namely, the $(s,\mu)$-twisted $\phz$-conjugation is not equivalent to the left translation relation by $\cI_1(R)$. For example, let $R=\cO/\varpi^a$, $|\cJ|=1$, $(s,\mu)=(1,(k,0))$, and let
\[ A=\begin{pmatrix} 1 &0 \\ 0 & v+p\end{pmatrix}, X=\begin{pmatrix}1 &1 \\0 &1\end{pmatrix}\]
Then $A$ and $XA\Ad(v^{\mu})\big(\phz(X)^{-1}\big)$ are in the same $(s,\mu)$-twisted $\phz$-equivalence class, but do not differ by a left translation by an element of $\cI_1(R)$. Indeed
\[XA\Ad(v^{\mu})\big(\phz(X)^{-1}\big)A^{-1}=\begin{pmatrix}1 & -\frac{v^k}{v+p} \\ 0 & 1\end{pmatrix}\]
does not belong to $\cI_1(R)$, since 
\[\frac{v^k}{v+p}=v^{k-1}(1+\cdots + (-1)^{a-1} \frac{p^{a-1}}{v^{a-1}})\notin R[\![v]\!].\]
\end{warning}

We now impose bounded $p$-adic Hodge type conditions. Let $\lambda \in X_*(T^{\vee})^{\cJ} = X^*(T)^{\cJ}$. We assume that $\lambda$ is effective and has height $\leq h$, that is each component $\lambda_j \in X_*(T^{\vee})$ satisfies $\lambda_j\in[0,h]^n$.  (Note that if $h_{\lambda} = \max \{ \langle \lambda, \alpha^{\vee} \rangle \mid \alpha \in \Phi \}$ then up to changing $\lambda$ by a central cocharacter we can take $h = h_{\lambda}$.) We now recall from \cite[Theorem 5.3]{CL} the closed $p$-adic formal substack $Y^{\leql,\tau}\subset Y^{[0,h],\tau}$ (denoted $Y^{\lambda, \tau}$ in \emph{loc.~cit.}). It is characterized by the following properties (cf.~\cite[Theorem 5.13]{CL}):
\begin{itemize}
\item $Y^{\leql,\tau}$ is flat over $\cO$, and has reduced versal rings (i.e.~it is analytically unramified in the sense of \cite[Definition 8.22]{Emerton_formal}).
\item For any finite extension $E'$ of $E$ with ring of integers $\cO'$, an element $\fM\in Y^{[0,h],\tau}(\cO')$ belongs to
 $Y^{\leql,\tau}(\cO')$ if and only if $\fM[1/p]$ has $p$-adic Hodge type $\leq \lambda$.  This is a condition on the type of the induced grading on $\fM/E(v) \fM$. Lemma 5.10 in \cite{CL} says that the grading on $\fM/E(v) \fM$ is the base change of a grading on the $\chi$-isotypic piece for $\chi$ appearing in the type $\tau$.  The type of this grading is directly related to the elementary divisors of the matrices of partial Frobenii $A^{(j)}_{\fM, \beta}$ (with respect to any eigenbasis). Because of this, the $p$-adic Hodge type $\leq \lambda$ condition translates to the condition that $A^{(j)}_{\fM, \beta}$ viewed as an element of $\GL_n(E'(\!(v+p)\!))$ has elementary divisors bounded by $(v+p)^{\lambda_j}$. Note that this last condition is exactly the condition imposed by the closed affine Schubert variety $S_E(\lambda)\subset \Gr_{\cG,E}^\cJ$.
\end{itemize}
We wish to identify the object that corresponds to $Y^{\leql,\tau}$ under the local model diagram of Theorem \ref{thm:fh_local_model}. Recall the (finite type over $\cO$) closed subscheme $M_{\cJ}(\leql)\into L^+\cG_\cO\backslash L\cG_\cO=\Gr_{\cG,\cO}^{\cJ}$, which is the Zariski closure of the (reduced) affine Schubert variety $S_E(\lambda)\subset \Gr_{\cG,E}^\cJ$ in $\Gr_{\cG,\cO}^\cJ$. 
For $\tld{z} = (\tld{z}_j)_{j \in \cJ} \in\tld{W}^{\vee,\cJ}$, we set
\begin{equation} \label{eq:chartsforlambda}
U(\tld{z}, \leql) := \prod_{j \in \cJ}  \cU(\tld{z}_j)_{\cO} \cap M(\leql_j), \quad \tld{U}(\tld{z}, \leql) := \prod_{j \in \cJ} T^{\vee}_{\cO} \times( \cU(\tld{z}_j)_{\cO} \cap M(\leql_j))
\end{equation}   
where the intersections are understood to be taken inside $\Gr_{\cG,\cO}$ (which can then be lifted to $L\cG_\cO$ since $\cU(\tld{z}_j)$ is canonically lifted to $L\cG$), {and define $Y^{\leql,\tau}(\tld{z})$ as the intersection $Y^{\leql,\tau}\cap Y^{[0,h],\tau}(\tld{z})$ (taken inside $Y^{[0,h],\tau}$)}.
We have the following: 
\begin{thm}\label{thm:Breuil-Kisin_local_model}

Let $(s,\mu)$ be a lowest alcove presentation of $\tau$ such that $\mu$ is $(h+1)$-deep in $\un{C}_0$, and let $\tld{z}\in \tld{W}^{\vee,\cJ}$. Assume $\lambda=(\lambda_j)_{j\in\cJ}\in X_*(T^\vee)^\cJ$ satisfies $\lambda_j\in [0,h]^n$. Then diagram (\ref{eq:fh_local_model}) induces a local model diagram of $p$-adic formal $\cO$-stacks
\begin{equation}\label{eqn:Hodgetype_local_model}
\begin{tikzcd}
\phantom{M} &  \tld{U}(\tld{z},\leql)^{\wedge_p} \subset L\cG^{\cJ}_\cO  \ar{dl}[swap]{T^{\vee,\cJ}_{\cO}} \ar{dr}{T^{\vee,\cJ}_{\cO}}& \phantom{M}\\
Y^{\leql,\tau}(\tld{z})=\Big[\tld{U}(\tld{z},\leql)/_{(s,\mu)} T^{\vee,\cJ}\Big]^{\wedge_p}  \ar[hook,open]{d}& \phantom{M} &U(\tld{z},\leql)^{\wedge_p}\ar[hook,open]{d} \\
Y^{\leql,\tau} & \phantom{M} & M_\cJ(\leql)^{\wedge_p}
\end{tikzcd}
\end{equation}
where the superscript $\wedge_p$ stands for taking $p$-adic completion.
\end{thm}
\begin{proof} We need to check $Y^{\leql,\tau}$ and $M_\cJ(\leql)^{\wedge_p}$ coincide after pulling back to $\tld{U}^{[0,h]}(\tld{z})^{\wedge_p}$ along diagram (\ref{eq:fh_local_model}). Since both pull-backs are reduced and $\cO$-flat, it suffices to check they have the same $\cO'$-points for $\cO'$ the integers in a finite extension $E'$ of $E$. But this is immediate, since the elementary divisor condition on an element of $\GL_n(E'(\!(v+p)\!))$ is preserved under both left and right multiplication by $\GL_n(E'[\![v+p]\!])$.
\end{proof}
\begin{cor} \label{cor:non_emptyness_pattern} Assume the hypotheses of Theorem \ref{thm:Breuil-Kisin_local_model}. Then $Y^{\leql,\tau}(\tld{z})\neq \emptyset$ if and only if $\tld{z}\in \Adm^\vee(\lambda)$.
\end{cor}
\begin{proof} $Y^{\leql,\tau}(\tld{z})\neq \emptyset$ if and only if $U(\tld{z},\leql)^{\wedge_p}\neq \emptyset$ if and only if $U(\tld{z},\leql )_{\bF}\neq \emptyset$. 
On the other hand, by Theorem \cite[Theorem 9.3]{PZ}, $M_{\cJ}(\leql)_{\bF}$ is the union of affine Schubert varieties $S^\circ_{\bF}(\tld{s})$ where $\tld{s}\in \Adm^{\vee}(\lambda)$. Thus, the set of torus fixed points of $M_{\cJ}(\leql)_{\bF}$ is $\Adm^{\vee}(\lambda)$. 
The result then follows from Lemma \ref{lem:Tfixedpts_criterion}. 

\end{proof}

\begin{cor} \label{cor:non_emptyness_pattern} Assume the hypotheses of Theorem \ref{thm:Breuil-Kisin_local_model}. Let $\F'/\F$ be a finite extension.  Then $\fM \in Y^{\leql,\tau}(\F')$ if and only if the shape of $\fM$ with respect to $\tau$ lies in $\Adm^\vee(\lambda)$.
\end{cor}

\subsection{\'Etale $\phz$-modules} 

\subsubsection{Background}
\label{subsub:etale_with_descent}
Let  $\cO_{\cE,K}$ (resp. $\cO_{\cE,L'}$) be the $p$-adic completion of $(W(k)[\![v]\!])[1/v]$ (resp.~of $(W(k')[\![u']\!])[1/u']$). 
It is endowed with a continuous Frobenius morphism $\phz$ extending the Frobenius on $W(k)$ (resp.~{on $W(k')$, and moreover endowed with an action of $\Delta$, cf.~\cite[\S 6.1]{LLLM} for the explicit definition of this action}) and such that $\phz(v)=v^p$ (resp.~$\phz(u')=(u')^p$).
Let $R$ be a $p$-adically complete Noetherian $\cO$-algebra.
We then have the groupoid
$\Phi\text{-}\Mod^{\text{\'et}, n}_K(R)$ (resp.  $\Phi\text{-}\Mod^{\text{\'et}, n}_{dd,L'}(R)$) of \'etale $(\phz,\cO_{\cE,K}\widehat{\otimes}_{\Zp}R)$-modules (resp. \'etale $(\phz,\cO_{\cE,L'}\widehat{\otimes}_{\Zp}R)$-modules with descent data from $L'$ to $K$).
Its objects are rank $n$ projective modules $\cM$ over $\cO_{\cE,K}\widehat{\otimes}_{\Zp}R$ (resp.~$\cO_{\cE,L'}\widehat{\otimes}_{\Zp}R)$), endowed with a Frobenius semilinear endomorphism $\phi_{\cM}:\cM\ra\cM$ ({resp.~a Frobenius semilinear endomorphism $\phi_{\cM}:\cM\ra\cM$,  and a semilinear action of $\Delta$ commuting with $\phi_{\cM}$}) inducing an isomorphism on the pull-back: $\Id\otimes_{\phz}\phi_{\cM}:\phz^*(\cM)\stackrel{\sim}{\longrightarrow}\cM$. It is known that $\Phi\text{-}\Mod^{\text{\'et}, n}_K(R)$ and  $\Phi\text{-}\Mod^{\text{\'et}, n}_{dd,L'}(R)$
form fppf stacks over $\Spf\,\cO$
(see \cite[\S 3.1]{EGstack}, \cite[\S 5.2]{EGschemetheoretic}, \cite[\S 3.1]{CEGS} where they are denoted $\mathcal{R}_{n},  \mathcal{R}^{dd}_{n, L'}$ respectively). 
We use $\Phi\text{-}\Mod^{\text{\'et}}_K(R)$ (resp. $\Phi\text{-}\Mod^{\text{\'et}}_{dd,L'}(R)$) to denote the category of \'etale $\phi$-modules  over $K$ (resp. over $L'$ with descent)  with coefficients in $R$ and of arbitrary finite rank.  

Given $\fM\in Y^{[0,h],\tau}(R)$, the element $\fM \otimes_{\fS_{L',R}} (\cO_{\cE,L'}\widehat{\otimes}_{\Zp}R)$
is naturally an object $\Phi\text{-}\Mod^{\text{\'et},n}_{dd,L'}(R)$ and we define an \'etale $\phz$-module $\cM \in \Phi\text{-}\Mod^{\text{\'et},n}_{K}(R)$ by 
\[
\cM \defeq (\fM \otimes_{\fS_{L',R}} (\cO_{\cE,L'}\widehat{\otimes}_{\Zp}R))^{\Delta=1}
\]
with the induced Frobenius.  
This construction defines a map of stacks $\eps_\tau: Y^{[0,h],\tau}\ra\Phi\text{-}\Mod^{\text{\'et},n}_{K}$. Note that $\eps_\tau$ is independent of any presentation of $\tau$.

\begin{prop} \label{prop:epstau} The map $\eps_{\tau}$ is representable by algebraic spaces, proper, and of finite presentation.
\end{prop}
\begin{proof}  First, the morphism $ Y^{[0,h],\tau}\to\Phi\text{-}\Mod^{\text{\'et}, n}_{dd,L'}$ is representable by algebraic spaces, proper, and of finite presentation by Corollary 3.1.7(3) and Proposition 3.3.5 of \cite{CEGS}.  {Finally, the map $\Phi\text{-}\Mod^{\text{\'et}, n}_{dd,L'}$ to $\Phi\text{-}\Mod^{\text{\'et},n}_{K}$ defined by taking $\Delta$-invariants is an equivalence} of stacks with quasi-inverse given by $\cM \mapsto \cM \otimes _{\cO_{\cE, K}} \cO_{\cE, L'}$.   
\end{proof}

For any $(\cM,\phi_\cM)\in \Phi\text{-}\Mod^{\text{\'et}}_{K}(R)$, we decompose $\cM=\oplus_{j \in \cJ} \cM^{(j)}$ over the embeddings $\sigma_j: W(k)\ra\cO$, with induced maps $\phi_\cM^{(j)}:\cM^{(j-1)}\ra\cM^{(j)}$.
The following proposition, a direct generalization of \cite[Proposition 3.2.1]{LLLM2}, records the effect of $\eps_\tau$ in terms of eigenbases.
\begin{prop} \label{prop:expeps}  
Let $\fM \in Y^{[0, h], \tau}(R)$ and set $\cM = \eps_{\tau}(\fM)$.  Let $(s, \mu)$ be the fixed lowest alcove presentation of $\tau$. 
 If $\beta$ is an eigenbasis of $\fM$, then there exists a basis $\fF$ \emph{(}determined by $\beta$\emph{)} for $\cM$ such that the matrix of $\phi_{\cM}^{(j)}$ with respect to $\fF$ is given by 
\[
A^{(j)}_{\fM, \beta} s^{-1}_j v^{\mu_j + \eta_j}.  %
\]
\end{prop}
\begin{proof} %
The statement is \cite[Proposition 3.2.1]{LLLM2} whose proof is generalized in the proof of \cite[Corollary 3.2.17]{LLL}.
For the convenience of the reader, we reproduce the argument here.
In particular, the proof below is obtained by a simple relabeling from the proof of \emph{loc.~cit.}, using Remark \ref{rmk:cmpr:mat}.
We define a basis $\tld{\beta}'$ for $\cM' \defeq(\fM \otimes_{\fS_{L'}} \cO_{\cE,L'})^{\Delta' = 1}$ as follows:  for each $0 \leq j' \leq f'-1$, define 
\[
\tld{\beta}^{\prime\,(j')}  \defeq \beta^{(j')} \Big((u')^{\bf{a}^{\prime \, (j')}}\Big)
\]
which is a basis for $\cM^{\prime \, (j')}$.  This uses that the action on $u'$ of $\Delta'$ in embedding $j'$ is through the character $\omega_{f'}^{p^{f' - j'}}$.  
The matrix for $\phi_{{\cM}'}^{(j')}:{\cM}^{\prime \, (j'-1)} \ra {\cM}^{\prime \, (j')}$ with respect to $\tld{\beta}'$ is given by 
\[
s'_{\mathrm{or}, j'} A^{(j')}_{\fM, \beta} (s'_{\mathrm{or}, j'})^{-1} \big(u'\big)^{p \bf{a}^{\prime\, (j'-1)} - \bf{a}^{\prime \, (j')}} = s'_{\mathrm{or}, j'} A^{(j')}_{\fM, \beta} (s'_{\mathrm{or}, j'})^{-1}  v^{\bm{\alpha}'_{f'-j'}}
\]
since $p \bf{a}^{\prime \, (j'-1)}- \bf{a}^{\prime \, (j')}= (p^{f'} -1) \bm{\alpha}'_{ f'-j'}$. 
Define $\tld{\beta}$ by $\tld{\beta}^{(j')} \defeq \tld{\beta}^{\prime \, (j')} s'_{\mathrm{or}, j'}$ for all $0\leq j'\leq f'-1$. 
Let $j' = j + if$ for $0 \leq j \leq f-1$.   Then the matrix for $\phi_{\cM'}^{(j')}$ with respect to $\tld{\beta}$ is given by 
\[
 A^{(j')}_{\fM, \beta} (s'_{\mathrm{or}, j'})^{-1}  s'_{\mathrm{or},j'-1} v^{(s'_{\mathrm{or},j'-1})^{-1} (\bm{\alpha}'_{f'-j'})} =  A^{(j')}_{\fM, \beta} s^{-1}_j v^{\mu_j + \eta_j}.
\]   
 Since $(\sigma^f)^*(\tld{\beta}^{(j')}) = \tld{\beta}^{(j' -f)}$, this descends to a basis $\fF$ of ${\cM} \defeq \eps_\tau(\fM)= (\cM')^{\sigma^f = 1}$, with respect to which the matrix of $\phi_\cM^{(j)}$ has the form described in the statement of the Proposition.

\end{proof}

The following Proposition, {which is the global version of the triviality of Kisin varieties,} shows that $\eps_{\tau}$ does not lose information in generic situations:
\begin{prop}\label{prop:BK_to_phi_mono}
{Let $h$ be a nonnegative integer and} assume $\tau$ is $(h+1)$-generic. Then the {proper} map $\eps_\tau: Y^{[0,h],\tau}\to \Phi\text{-}\Mod^{\text{\emph{\'et}},n}_{K}$ is a monomorphism of stacks over $\Spf \, \cO$, and hence is a closed immersion.
\end{prop}
To prepare for the proof, we record the following Lemmas:

\begin{lemma} \label{lem:BKvariety}  Let $R$ be an $\bF$-algebra, and let $(A^{(j)}_1)_{j \in \cJ},  (A^{(j)}_2)_{j \in \cJ} \in  L^{[0,h]}\GL_n(R)^{\cJ}$.  Assume $\tld{z} = s^{-1} t_{\mu + \eta} \in \tld{W}^{\vee, \cJ}$ such that $\mu$ is $(h+1)$-deep in $\un{C}_0$.  Let $(I^{(j)})_{j \in \cJ} \in \GL_n(R(\!(v)\!))^{\cJ}$ such that $A^{(j)}_2 \tld{z}_j = I^{(j)} A^{(j)}_1 \tld{z}_j \phz(I^{(j-1)})^{-1}$ for all $j \in \cJ$, then $I^{(j)} \in \Iw_{\F}(R)$ for all $j \in \cJ$. 
\end{lemma}
\begin{proof} We essentially repeat the argument in the proof of \cite[Theorem 3.2]{LLLM} for a general $\bF$-algebra $R$.   For all $j \in \cJ$, define $k_{j}\in\Z$ so that $I^{(j),\,+}\defeq v^{k_{j}} I^{(j)}\in\Mat_n(R[\![v]\!])$ and $I^{(j),\,+}\not\equiv 0$ modulo $v$.
Rewriting the equation and multiplying through by $v^{h}$, we have
\begin{equation*}
v^{h+k_{j}-pk_{j-1}}\Ad(\tld{z}_j)\big(\phz(I^{(j-1),\,+})\big)=v^{h}\Big(A_{2}^{(j)}\Big)^{-1}I^{+,\,(j)}A_1^{(j)},
\end{equation*}
where the right side is in $\Mat_n(R[\![v]\!])$ by the height condition. 

As $I^{(j),\,+}\not\equiv 0$ modulo $v$, 
we deduce that 
\begin{equation}
\label{ex:KV:key}
k_{j}\geq pk_{j-1}-\max_{\alpha \in R}\{\langle \mu_{j}+\eta_j,\alpha^{\vee} \rangle\}-h > p k_{j-1} - p + m - h \geq p(k_{j-1} -1) + 1
\end{equation}
since $\mu$ is $(h+1)$-deep in $\un{C}_0$.  
We conclude that $k_{j'}\leq 0$ for all $j \in \cJ$ or, equivalently, $I^{(j)}\in \Mat_n(R[\![v]\!])$ for all $j \in \cJ$.
By exchanging the roles of $A_1$ and $A_2$, we conclude that $I^{(j)}\in \GL_n(R[\![v]\!])$ for all $j \in \cJ$.

We now prove that $I^{(j)}\in \Iw(R)$ for all $j \in \cJ$. 
Let $\alpha$ be a negative root of $\GL_n$. Assume $(I^{(j-1)}\big)_\alpha \not\equiv 0$ mod $v$ for some $j \in \cJ$.  Since  
\[
\big(\Ad\big(v^{\mu_{j}+\eta_{j}}\big)(\phz(I^{(j-1)}))\big)_\alpha=(\phz(I^{(j-1)})\big)_\alpha v^{\langle \mu_{j}+\eta_{j},\alpha^{\vee} \rangle},\] 
$\Ad(\tld{z}_j) (\phz(I^{(j-1)}))$ has a pole of order $- \langle \mu_{j}+\eta_{j},\alpha^{\vee} \rangle  > h$.   This is a contradiction since  $v^h \Ad(\tld{z}_j)(\phz(I^{(j-1)})) = v^h\big(\big(A_{2}^{(j)}\big)^{-1}I^{(j)}A_1^{(j)}\big)$ is in $\Mat_n(R[\![v]\!])$.  
\end{proof}
The same argument also proves the following:
\begin{lemma} \label{lem:phi_mod_tangent}Let $R$ be an $\bF$-algebra and let $J$ be an $R$-module. Let $(A^{(j)})_{j \in \cJ}\in  L^{[0,h]}\GL_n(R)^{\cJ}$.  Assume $\tld{z} = s^{-1} t_{\mu + \eta} \in \tld{W}^{\vee, \cJ}$ such that $\mu$ is $(h+1)$-deep in $\un{C}_0$. 
Let $(Y_j)_{j \in \cJ} \in \Mat_n(J(\!(v)\!))$.
Assume that 
\[Y_{j}-\Ad(A^{(j)}\tld{z}_j)\phz(Y_{j-1})\in \frac{1}{v^h} \Mat_n(J[\![v]\!]),\]
for all $j\in\cJ$.
Then $Y_{j}\in \Lie \cI_{\F}(J)$.
\end{lemma}

\begin{proof}[Proof of Proposition \ref{prop:BK_to_phi_mono}] We need to show that for each $p$-adically complete Noetherian $\cO$-algebra, $\eps_\tau$ induces a fully faithful functor $ Y^{[0,h],\tau}(R)\to \Phi\text{-}\Mod^{\text{\'et},n}_{K}(R)$. It suffices to treat the case where $R$ is a Noetherian $\cO/\varpi^a$-algebra. We choose a lowest alcove presentation $(s,\mu)$ of $\tau$ such that $\mu$ is $(h+1)$-deep in $\un{C}_0$.

Suppose $\fM_1, \fM_2 \in Y^{[0,h],\tau}(R)$, and let $\cM_i=\eps_{\tau}(\fM_i)$ for $i=1,2$. We need to show $\eps_\tau$ induces an isomorphism $\Hom_{Y^{[0,h],\tau}(R)}(\fM_1,\fM_2)\cong \Hom_{\Phi\text{-}\Mod^{\text{\'et},n}_{K}(R)}(\cM_1,\cM_2)$. Since this assertion is local in $R$, we may assume that $\fM_i$ admits eigenbases $\beta_i$. Let $A^{(j)}_i = A^{(j)}_{\fM_i, \beta_i}$. Proposition \ref{prop:expeps} using the bases $\fF_1, \fF_2$ constructed from $\beta_1, \beta_2$ shows that an element $\Hom_{\Phi\text{-}\Mod^{\text{\'et},n}_{K}(R)}(\cM_1,\cM_2)$ is in bijection with the set of tuples $(I^{(j)})_{j\in\cJ}\in L\cG_\cO(R)^\cJ$ such that 
\begin{equation}\label{eq:hom_elements}
I^{(j)}A^{(j)}_1=A^{(j)}_2\Ad(s^{-1}_j v^{\mu_j+\eta_j})(\phz(I^{(j-1)})),
\end{equation}
while $\Hom_{Y^{[0,h],\tau}(R)}(\fM_1,\fM_2)$ is in bijection with the set of tuple $(I^{(j)})_{j\in \cJ}\in \cI_{\F}(R)$ satisfying the same relation by Proposition \ref{prop:changeofbasis}. In other words, we need to show that any solution to (\ref{eq:hom_elements}) in $L\cG_\cO(R)^{\cJ}$ must automatically belong to $\cI(R)^{\cJ}$. We will prove this assertion by induction on $a$.

The case $a=1$ is treated by Lemma \ref{lem:BKvariety}. Suppose our assertion is true up to $a-1$. We may assume that $A^{(j)}_1\equiv A^{(j)}_2$ mod $\varpi^{a-1}$, and $\beta_1\equiv \beta_2$ mod $\varpi^{a-1}$. Let $A^{(j)}_\bF=A^{(j)}_1=A^{(j)}$ mod $\varpi$. We perform a Lie algebra computation similar to the proof of Proposition \ref{prop:gauge_basis_existence}. Set $J=\varpi^{a-1}R$. We can write $A^{(j)}_2=(1+\fA_j)A^{(j)}_1$ and $I^{(j)}=1+Y_j$, with $\fA_j, Y_j\in \Mat_n(J(\!(v)\!))$. 
Equation \ref{eq:hom_elements} translates to %
\[Y_{j}=\fA_j+\Ad(A^{(j)}_\bF s_j^{-1}v^{\mu_j+\eta_j})(\phz(Y_{j-1}))\]
Since $A^{(j)}_i\in L^{[0,h]}\GL_n(R)$, $\fA_j\in \frac{1}{v^h}\Mat_n(J[\![v]\!])$. Lemma \ref{lem:phi_mod_tangent} thus shows that $Y_j\in \Lie\cI(J)$, and thus $I^{(j)}\in \cI(R)$.
\end{proof}

\subsubsection{\'Etale $\phi$-modules and local models} %
\label{subsub:etale_phi_LM}

Fix an $(h+1)$-generic tame inertial type $\tau$ with a lowest alcove presentation $(s, \mu)$ 
such that $\mu$ is $(h+1)$-deep in $\un{C}_0$. 
This gives rises to local model diagrams (\ref{eq:fh_local_model}) and (\ref{eqn:Hodgetype_local_model}) for $Y^{[0,h],\tau}(\tld{z})$ and $Y^{\leql,\tau}(\tld{z})$. Over $\bF$, these diagrams glue together into local model diagrams for $Y^{[0,h],\tau}_\bF$ and $Y^{\leql,\tau}_\bF$. On the other hand, we have the canonical map $\eps_\tau$ which does not depend on the presentation $(s,\mu)$. It is therefore natural to express $\eps_\tau$  in terms of the objects occurring in the local model diagram. This will later be used in conjunction with the results of \S \ref{sec:componentmatching} to describe the irreducible components of the Emerton-Gee stack which occur in $Y_\bF^{\leql,\tau}$ in terms of the local model.

Let $\tld{z} = (\tld{z}_j)_{j \in \cJ} \in \tld{W}^{\vee, \cJ}$ and $a \leq b$ integers, define a closed subscheme 
\[
\tld{\Fl}^{[a,b]}_{\cJ, \tld{z}} =  \prod_{j \in \cJ} (\Iw_{1,\bF} \backslash (L^{[a,b]} \GL_n)_\bF\,\tld{z}_j) \subset \tld{\Fl}^{\cJ},
\]
where $L^{[a,b]} \GL_n$ is as in Definition \ref{defn:GLnAB}. Clearly, $\tld{\Fl}^{[a,b]}_{\cJ,\tld{z}}=\tld{\Fl}^{[a,b]}_\cJ\tld{z}$.
There is a natural map $\iota'_{\tld{z}}:\prod_{j \in \cJ} (L^{[a,b]}\GL_n)_\bF\,\tld{z}_j \ra \Phi\text{-}\Mod^{\text{\'et},n}_{K, \F}$, which for an $\bF$-algebra $R$ is given by sending $(A^{(j)}\tld{z}_j)_{j \in \cJ} \in \prod_{j \in \cJ} L^{[a,b]} \GL_n(R)\,\tld{z}_j$ to the free rank $n$ \'etale $\phz$-module $\cM$ over $R$ such that $\phi^{(j)}_{\cM}$ is given by $A^{(j)} \tld{z}_j$ in the standard basis. Clearly $\iota'_{\tld{z}}$ factors through the quotient by the $\phz$-conjugation action $\big[\prod_{j \in \cJ} (L^{[a,b]} \GL_n)_\bF\,\tld{z}_j /_\phz\;\cI_{\F}\big]$.

Now assume that $\tld{z} = \sigma^{-1} t_{\nu+\eta}\in \tld{W}^{\vee,\cJ}$ where $\nu$ is $(b-a+1)$-deep in $\un{C}_0$. Since right translation by $\tld{z}$ intertwines the $(\sigma,\nu)$-twisted $\phz$-conjugation action with the $\phz$-conjugation action, Lemma \ref{lem:iotawelldef} shows that $\iota'_{\tld{z}}$ descends to a map $\iota_{\tld{z}}:\tld{\Fl}^{[a,b]}_{\cJ,\tld{z}} \ra \Phi\text{-}\Mod^{\text{\'et},n}_{K, \F}$. 
This further factors through the quotient $\big[\tld{\Fl}^{[a,b]}_{\cJ, \tld{z}}/T^{\vee, \cJ}_{\F}\text{-conj}\big]$ where the action of $T^{\vee, \cJ}$ is  given by $ (\cI_1 A^{(j)} \tld{z}_j)_{j \in \cJ} \mapsto (D_{j} \cI_1 A^{(j)} \tld{z}_j D^{-1}_{j-1})_{j \in \cJ}$ for $(D_j)_{j \in \cJ} \in T^{\vee, \cJ}_{\F}$. We will abbreviate this as the shifted $T^{\vee,\cJ}_{\F}$-conjugation action. 

\begin{prop} \label{prop:iotamono}  Assume that $\tld{z} = \sigma^{-1} t_{\nu + \eta}$ where $\nu$ is $(b-a+1)$-deep in $\un{C}_0$.  The map $\iota_{\tld{z}}$ induces a monomorphism of stacks
\[
\iota_{\tld{z}}:\Big[\tld{\Fl}^{[a,b]}_{\cJ, \tld{z}}/T^{\vee, \cJ}_{\F}\text{-\emph{conj}}\Big] \iarrow \Phi\text{-}\Mod^{\text{\'et},n}_{K, \F}.
\]
\end{prop}
\begin{proof}
Unraveling the definitions and twisting by $v^{-a}$, the Proposition boils down to the statement that if $R$ is an $\bF$-algebra, $ (A^{(j)}_1 \tld{z}_j)_{j \in \cJ},(A^{(j)}_2 \tld{z}_j)_{j \in \cJ}\in \prod_{j\in\cJ} L^{[0,b-a]}\GL_n(R)\tld{z}_j$ are $\phz$-conjugate by an element $(I^{(j)})_{j \in \cJ} \in \GL_n(R(\!(v)\!))^{\cJ}$, then $I^{(j)} \in \Iw_{\F}(R)$ for all $j \in \cJ$. This follows from Lemma \ref{lem:BKvariety}. 
\end{proof} 

The following Proposition, obtained by combining Propositions \ref{prop:iotamono}, \ref{prop:BK_to_phi_mono} and Corollary \ref{cor:Breuil-Kisin_moduli_model}, provides our desired description of $\eps_\tau$:
\begin{prop}\label{prop:phi_local_model} Suppose we are given the following data:
\begin{itemize} 
\item Integers $a\leq b$, and $h\geq 0$.
\item An element $\tld{z}=\sigma^{-1} t_{\nu+\eta}\in\tld{W}^{\vee,\cJ}$ such that $\nu$ is $(b-a+1)$-deep in $\un{C}_0$.
\item A tame inertial type $\tau$ with lowest alcove presentation $(s,\mu)$ such that $\mu$ is $(h+1)$-deep in $\un{C}_0$. 
Setting $\tld{w}^{*}(\tau)=s^{-1}t_{\mu+\eta}$, assume that $\big(\!\Gr_{\cG,\bF}^{[0,h], \cJ}\big) \tld{w}^*(\tau)\subset \Fl^{[a,b]}_{\cJ,\tld{z}}$
\item An element $\lambda\in X_*(T^\vee)^\cJ$ such that $\lambda_j\in [0,h]^n$ for all $j \in \cJ$.
\end{itemize} %
Then we have a commutative diagram
\begin{equation}
\label{it:map:iotatau}
\xymatrix{\tld{M}_{\cJ}(\leql)_{\bF} \ar@{^{(}->}[r]\ar[d]
& \tld{\Gr}^{[0,h], \cJ}_{\cG,\bF}   \ar^-{r_{\tld{w}^*(\tau)}}@{^{(}->}[r]\ar[d]^{\pi_{(s, \mu)}} &\tld{\Fl}^{[a,b]}_{\cJ,\tld{z}} \ar[r] &\Big[\tld{\Fl}^{[a,b]}_{\cJ, \tld{z}}/T^{\vee, \cJ}_{\F}\text{-\emph{conj}}\Big]  \ar@{^{(}->}^{\iota_{\tld{z}}}[d]   \\  Y^{\leql,\tau}_\bF \ar@{^{(}->}[r] &Y^{[0,h],\tau}_\bF \ar@{^{(}->}[rru]\ar@{^{(}->}^{\eps_\tau}[rr] &
&\Phi\text{-}\Mod^{\text{\emph{\'et}},n}_{K, \F}
}
\end{equation}
\end{prop}

\begin{rmk}  In Proposition \ref{prop:phi_local_model}, there is a natural choice of $a, b, \tld{z}$, namely, $a=0, b=h$ and $\tld{z} = s^{-1} t_{\mu + \eta} = \tld{w}^*(\tau)$.   However, to compare multiple types, it can be convenient to make other choices. 
\end{rmk}

\subsection{Semisimple Breuil--Kisin modules}
\label{sub:ssBK}

Let $G_{K_{\infty}} \subset G_K$ denote the Galois group of $K_{\infty}$.  Recall that $K_{\infty}/K$ is totally wildly ramified. When $R$ is a complete Noetherian local $\cO$-algebra with finite residue field, from the theory of fields of norms, we have an exact anti-equivalence {(\cite{fontaine-fest})}
\begin{align*}
\bV^*_K:\Phi\text{-}\Mod^{\text{\'et},n}_{K}(R)&\ra\Rep^n_R(G_{K_\infty})
\end{align*}
{(where $\Rep^n_R(G_{K_\infty})$ denotes the groupoid of $G_{K_\infty}$-representations on rank $n$ projective $R$-modules)}
and hence a functor $T^*_{dd}: Y^{[0,h],\tau}(R)\rightarrow \Rep^n_R(G_{K_\infty})$ defined as the composite of $\eps_{\tau}$ followed by $\bV^*_K$.

Since the subgroup $G_{K_{\infty}}$ of $G_K$ projects surjectively to the tame quotient of $G_K$, the restriction map 
\[
\Rep_{\F}^{n}(G_K) \rightarrow \Rep^n_{\F}(G_{K_{\infty}})
\]   
is fully faithful on the subcategory of tame representation of rank $n$. 
We will often implicitly identify representations of $G_{K_{\infty}}$ in the essential image {(of the tame representations)} with their canonical extensions to $G_K$. Note that this essential image contains exactly representations of $G_{K_\infty}$ which are trivial on $G_{K_\infty}\cap G_{K^t}$, where $K^t$ is the maximal tamely ramified extension of $K$.  Note that semisimple representations of $G_{K_{\infty}}$ are necessarily tame and hence extend uniquely to $G_K$.

Given an $n$-dimensional $\F$-representation $\rhobar$ of $G_K$ or $G_{K_{\infty}}$, we denote its semisimplification by $\rhobar^{\mathrm{ss}}$.
If $\rhobar$ is tame, i.e.~if $\rhobar=\rhobar^{\mathrm{ss}}$, then $\rhobar|_{I_K}$ is a tame inertial $\F$-type for $K$ (see \S \ref{sec:InertialTypes}).

\begin{defn}
\label{defn:gen:rhobar}
Let $\rhobar$ be an $n$-dimensional $\F$-representation of $G_K$ or $G_{K_{\infty}}$.
\begin{enumerate}
\item Given an integer $m\geq 0$, we say that $\rhobar$ is $m$-generic if the tame inertial $\F$-type $\rhobar^{\mathrm{ss}}|_{I_K}$ is $m$-generic in the sense of Definition \ref{defi:gen}(\ref{defi:gen:type}).
\item  If $\rhobar$ is tame we say that $(s,\mu)\in \un{W}\times X^*(\un{T})$ is a lowest alcove presentations for $\rhobar$ if $(s,\mu)$ is a lowest alcove presentation for the inertial $\F$-type $\rhobar|_{I_K}$ as defined in the paragraph following Example \ref{ex:data:type}.
We then write $\tld{w}(\rhobar)$ for the element $\tld{w}(\rhobar|_{I_K}) = t_{\mu + \eta} s$ defined in Definition \ref{defi:gen}(\ref{defi:gen:type}) and $\tld{w}^*(\rhobar) \defeq s^{-1} t_{\mu + \eta}$.  A lowest alcove presentation is $m$-generic if $\tld{w}(\rhobar)$ is $m$-generic.
\end{enumerate}
\end{defn}

We can directly relate the lowest alcove presentation to a description of the corresponding \'etale $\phi$-module. 
\begin{prop} \label{prop:ssetphi}  Let  $\tld{z}' \in \Fl_{\cJ, \tld{z}}^{[a, b]}$ for $a, b, \tld{z}$ as in Proposition \ref{prop:phi_local_model}.   Let $\rhobar$ be an $n$-dimensional semisimple  $\F$-representation of either $G_K$ or $G_{K_{\infty}}$.   Then $\rhobar$ admits a lowest alcove presentation $(s, \mu)$ such that $\tld{w}^*(\rhobar) = s^{-1} t_{\mu + \eta} =  \tld{z}'$ if and only if there exists $D \in T^{\vee, \cJ}(\F)$ such that 
\[
\bV^*_K(\iota_{\tld{z}}( D \tld{z}'))  \cong \rhobar|_{G_{K_{\infty}}}. \]  
\end{prop}
\begin{proof}
Let $\tld{z}' = s^{-1} t_{\mu + \eta}$. Then $\mu$ is 1-deep in $\un{C}_0$ by the hypotheses in Proposition \ref{prop:phi_local_model}. The fact that for any such $D$, $\iota_{\tld{z}}( D \tld{z}') \cong \taubar(s, \mu + \eta)$ (and hence has lowest alcove presentation $(s, \mu)$) follows from a direct computation as in Proposition 3.1.2 of \cite{LLL}.  (Note that one needs to use Remark \ref{rmk:cmpr:mat} to translate the conventions of \emph{loc.~cit.}~into conventions of this paper.)   

To show the forward direction, one has to show that the choice of $D$ accounts for all possible extensions from $I_K$ to $G_K$.  This can be done by counting isomorphism classes over $\F$ since by Proposition \ref{prop:phi_local_model}, $\iota_{\tld{z}}( D \tld{z}') \cong \iota_{\tld{z}}( D' \tld{z}')$ if and only if $D \tld{z}$ and $D' \tld{z}'$ are $T^{\vee, \cJ}(\F)$-conjugate by shifted conjugation.   
\end{proof}

\begin{rmk}  In Proposition \ref{prop:ssetphi}, if $\rhobar$ admits a $1$-generic lowest alcove presentation with corresponding element $\tld{w}(\rhobar)$, then one can take $\tld{z} = \tld{w}^*(\rhobar)$ and $a = b =0$ to satisfy the hypotheses of Proposition \ref{prop:ssetphi}. 
\end{rmk}

Fix $\lambda \in X^*(T)^{\cJ}$ dominant.  Assume that $\lambda_j \in [0,h]^n$ for $h \geq 0$.  Let $\tau$ be a tame inertial type together with a $(h + 1)$-generic lowest alcove presentation. %
We now recall notion of shape of $\rhobar$ with respect to $\tau$:
\begin{defn} \label{defn:shaperhobar} Let $\rhobar$ be $n$-dimensional $\F$-representation of $G_K$ or $G_{K_{\infty}}$.   If there exists $\fM \in Y^{[0, h], \tau}(\F)$ such that $T^*_{dd}(\fM) \cong \rhobar|_{G_{K_{\infty}}}$, then define the \emph{shape} $\tld{w}^*(\rhobar, \tau) \in \tld{W}^{\vee, \cJ}$ of $\rhobar$ with respect to $\tau$ to be the shape of $\fM$ (Definition \ref{defn:shape}).    

This is well-defined since $\eps_{\tau}$ is a monomorphism (Proposition \ref{prop:phi_local_model}). 
\end{defn}

We also have the notion of a semisimple Breuil--Kisin module:

\begin{defn} \label{defn:ssBKmod}  Let $\fM \in Y^{[0,h], \tau}(\F')$ for any finite extension $\F'/\F$.   Then $\fM$ is \emph{semisimple} of shape $\tld{z} \in  \tld{W}^{\vee, \cJ}$ if $\fM \in Y^{[0,h], \tau}(\tld{z})$ and for any choice of $\tld{z}$-gauge basis $\beta$, the image of $(\fM, \beta)$
 under the map $\tld{U}^{[0,h]}(\tld{z})_{\F} \ra U^{[0,h]}(\tld{z})_{\F}$ from \eqref{eq:fh_local_model} 
is the $T^{\vee, \cJ}$-fixed point $\tld{z}$ of $\Gr^{[0,h], \cJ}_{\cG, \F}.$  (In this case, $\fM$ clearly has shape $\tld{z}$.)  
\end{defn}

\begin{rmk} Concretely, the condition of being semisimple of shape $\tld{z}$ in Definition \ref{defn:ssBKmod} is that there exists an eigenbasis for $\fM$ such that $A^{(j)}_{\fM, \beta} \in T^{\vee} (\F) \tld{z}_j$ for all $j \in \cJ$.   By \cite[Proposition 3.2.16]{LLL}, this is equivalent to the definition given in Definition 3.2.14 of \emph{loc.~cit}. 
\end{rmk}

Recall that for a fixed lowest presentation $(s, \mu)$ of $\tau$, we define $\tld{w}^*(\tau) = s^{-1} t_{\mu + \eta}$.

\begin{prop} \label{prop:semisimple_admissible} Let $\tau$ be a tame inertial type with lowest alcove presentation $(s, \mu)$ where $\mu$ is $(h + 1)$-deep in $\un{C}_0$.  Let $\rhobar:G_{K_{\infty}} \ra \GL_n(\F)$ be a semisimple representation.  There exists a semisimple $\fM \in Y^{[0,h], \tau}(\F)$ such that $\rhobar \cong T^*_{dd}(\fM)$ if and only if $\rhobar$ admits a lowest alcove presentation such that $\tld{w}^*(\rhobar) (\tld{w}^*(\tau))^{-1} \in \Gr^{[0,h], \cJ}_{\cG, \F}$.     In this case, $\fM$ has shape  $\tld{w}^*(\rhobar) (\tld{w}^*(\tau))^{-1}$.
Furthermore, $\fM \in Y^{\leql, \tau}(\F)$ if and only if $ \tld{w}(\rhobar, \tau) = (\tld{w}(\tau))^{-1}\tld{w}(\rhobar) \in \Adm(\lambda)$.  
\end{prop}
\begin{proof}
The forward direction follows from  Proposition \ref{prop:ssetphi} and the diagram in Proposition  \ref{prop:phi_local_model} with $a= 0, b=h$ and $\tld{z} = \tld{w}^*(\tau)$.  Namely, if $\fM$ is semisimple of shape $\tld{z} \in  \Gr^{[0,h], \cJ}_{\cG, \F}$, then $\eps_{\tau}(\fM) \cong \iota_{\tld{w}^*(\tau)}(D \tld{z} \tld{w}^*(\tau))$ for some $D \in T^{\vee, \cJ}(\F)$.  By Proposition \ref{prop:ssetphi},  $\rhobar = T^*_{dd}(\fM)$ is semisimple and has a lowest alcove presentation $(w, \nu)$ such that $\tld{w}^*(\rhobar) = \tld{z} \tld{w}^*(\tau)$ as desired.  

Similarly, if $\rhobar$ admits a lowest alcove presentation such that $\tld{z} = \tld{w}^*(\rhobar) (\tld{w}^*(\tau))^{-1} \in \Gr^{[0,h], \cJ}_{\cG, \F}$, then $\tld{w}^*(\rhobar) \in \tld{\Fl}_{\cJ, \tld{w}^*(\tau)}^{[0,h]}$.  Thus, by Proposition \ref{prop:ssetphi}, there exists  $D \in T^{\vee, \cJ}(\F)$ such that $ \iota_{\tld{w}^*(\tau)}(D \tld{w}^*(\rho))$ gives rise to the \'{e}tale $\phz$-module corresponding to $\rhobar|_{G_{K_\infty}}$.   
Then $\pi_{(s,\mu)}(D \tld{z})$ is the desired semisimple $\fM$ since the diagram in Proposition \ref{prop:phi_local_model} commutes.  

The final statement follows from Corollary \ref{cor:non_emptyness_pattern}.  
\end{proof}

\begin{cor} \label{cor:admshape} Let $\rhobar:G_K \ra \GL_n(\F)$ be a tame representation and $\tau$ be as in Proposition \ref{prop:semisimple_admissible}.  If $R_{\rhobar}^{\lambda, \tau}$ is non-zero, then $\rhobar$ has a $\lambda$-compatible lowest alcove presentation such that $\tld{w}(\rhobar, \tau) = \tld{w}(\tau)^{-1} \tld{w}(\rhobar) \in \Adm(\lambda)$. 
\end{cor}
\begin{proof}
By the same argument as in case (2) of Theorem \cite[Theorem 3.2.20]{LLL}, there exists {a semisimple} $\fM \in Y^{\leql, \tau}(\F)$ such that $\rhobar|_{G_{K_{\infty}}} \cong T^*_{dd}(\fM)$.   The rest follows from Proposition \ref{prop:semisimple_admissible} noting that $(\tld{w}(\tau))^{-1}\tld{w}(\rhobar) \in \Adm(\lambda)$ implies that the lowest alcove presentation of $\rhobar$ is $\lambda$-compatible (see \S \ref{sec:InertialTypes}).    
\end{proof}

\begin{prop}  \label{prop:redss}  Let $\F'/\F$ be a finite extension. 
 If $\fM \in Y^{\leql, \tau}(\F')$ $($resp. $Y^{[0,h], \tau}(\F'))$, then there exists a semisimple Breuil--Kisin module $\fM' \in Y^{\leql, \tau}(\F')$ $($resp. $Y^{[0,h], \tau}(\F'))$ such that $T^*_{dd}(\fM)^{\mathrm{ss}} \cong T^*_{dd}(\fM')$.
\end{prop} 
\begin{proof}

Let $\rhobar = T^*_{dd}(\fM)$, and let $\cM = \fM[1/u'] \in \Phi\text{-}\Mod^{\text{\'et},n}_{dd, L'}(\F')$.  Let $(\rhobar_i)_{0 \leq i \leq d}$ denote a decreasing filtration on $\rhobar$ such that $\gr_i(\rhobar) := \rhobar_i/\rhobar_{i+1}$ is semisimple for all $i$. Recall the exact anti-equivalence of categories $\bV^*_{dd}$ between $\Mod^{\text{\'et}}_{dd, L'}(\F')$ and $\Rep_{\F'}(G_{K_{\infty}})$ (see pg.~24 in \cite{LLLM} for example). Using this equivalence, there is an increasing filtration $\cM_i \subset \cM$ such that 
\[
\bV^*_{dd}(\cM_i/\cM_{i-1}) \cong \gr_i(\rhobar)
\]
for all $i$.  %
Define $\fM_i = \fM \cap \cM_i$.  By construction, $\fM_i$ is a lattice in $\cM_i$ stable under both $\phi_{\fM}$ and the action of $\Delta$.  Thus, $\fM_i$ is a Breuil--Kisin module over $L'$ with descent datum to $K$ and of rank $\dim(\bV^*_{dd}(\cM_i))$. We can inductively construct a basis $\alpha$ adapted to the filtration $(\fM_i)_i$ and compatible with the descent datum.  That is, we inductively pick bases $\alpha_i =( \alpha_i^{(j)})$ for each $\fM_i$ such that $\Delta'$ acts by characters on individual basis elements and $\iota:(\alpha_i^{(j)}) = \alpha_i^{(j + f)}$. 

Let $\alpha = \alpha_d$.  Define the matrix $C^{(j)}\in \GL_n(\F'(\!(u')\!))$ by the condition
\[
\phi_{\fM}^{(j)}(\phz^*(\alpha^{(j-1)})) = \alpha^{(j)} C^{(j)}.
\]
By construction, $C^{(j)}$ lies in a parabolic subgroup $P(\F'(\!(u')\!)) \subset \GL_n(\F'(\!(u')\!))$ corresponding to the filtration $(\cM_i)$.   
Let $L$ denote the corresponding Levi subgroup which contains the diagonal torus $T$.   
Choose a dominant cocharacter $\nu$ such that $L$ is the centralizer of $\nu$.  

We now construct a family $\cM_x$ of free \'{e}tale $\phz$-module with descent data of rank $n$ over $\bA^1_{\F'}=\Spec \F'[x]$ as follows: we take a basis $\alpha_x$ and let $\Delta$ act on $\alpha_x$ in the same way it acts on $\alpha$, and let Frobenius act by $C^{(j)}_x = \nu(x) C^{(j)} \nu(x)^{-1}$ (with respect to $\alpha_x$). Note that the right-hand side belongs to $\GL_n(\bF'[\![x]\!](\!(u')\!))$, and that $C^{(j)}_0$ lies in the Levi subgroup $L(\F'(\!(u')\!))$. The family $\cM_x$ gives a map $ \bA^1_{\F'} \to \Phi\text{-}\Mod^{\text{\'et},n}_{dd,L'}$.

The family we constructed has the following properties:
\begin{itemize} 
\item For each $x\in \ovl{\F}^\times$, the matrices $C^{(j)}_x$ define a Breuil--Kisin module with descent data $\fM_x\subset \cM_x$. Furthermore, $\fM_x\cong \fM$ as Breuil--Kisin modules with descent data (via scaling the basis by $\nu(x)$), and thus $\fM_x$ gives a point of $Y^{\leql,\tau}(\ovl{\F})$.
\item $\bV^*_{dd}(\cM_0)=\rhobar^{\mathrm{ss}}$.
\end{itemize}
Since the map $Y^{\leql,\tau}\to \Phi\text{-}\Mod^{\text{\'et},n}_{dd,L'}$ is representable and proper (and in fact is a closed immersion in the current situation), the locus of $x$ where $\cM_x$ comes from a Breuil--Kisin module in $Y^{\leql,\tau}$ is closed. Since this locus contains all elements of $\ovl{\F}^\times$, it must contain $x=0$.
We conclude that there is a Breuil--Kisin module $\fM'\in Y^{\leql,\tau}(\F')$ inside $\cM_0$. In particular we have $T^*_{dd}(\fM')\cong \rhobar^{\mathrm{ss}}$, and Proposition \ref{prop:semisimple_admissible} implies furthermore that $\fM'$ is semisimple.

\end{proof} 

\begin{cor} \label{lem:ss_bound} Let $\bF'$ be a finite extension of $\bF$ and let $\fM\in Y^{[0,h],\tau}(\bF')$. Assume that $\tau$ is $m$-generic where $m \geq h+1$. Then $T^*_{dd}(\fM)^{\mathrm{ss}}$ admits a $(m-h)$-generic lowest alcove presentation $(w, \nu)$. %
\end{cor}
\begin{proof} By Proposition \ref{prop:redss}, we can reduce to the case where $T^*_{dd}(\fM)$ is semisimple.   Choose a lowest alcove presentation $(s, \mu)$ of $\tau$ where $\mu$ is $m$-deep in $\un{C}_0$.  By Proposition \ref{prop:semisimple_admissible},  $\fM$ is semisimple of shape $\tld{z} =((s^{-1}_j t_{\nu_j -\mu_{{j}}} w_j)^*)_{j \in \cJ}$ where $T^*_{dd}(\fM)|_{I_K} \cong \taubar(w, \nu + \eta)$.   By the height condition on $\fM$, it is clear that $\tld{z}$ is $h$-small so $|\langle \nu - \mu, \alpha^{\vee} \rangle| \leq h$ for all $\alpha \in \Phi$.  Since $\mu$ is $m$-deep in $\un{C}_0$, we conclude the $\nu$ at least $(m - h)$-deep in $\un{C}_0$.     
\end{proof}

\clearpage{}%
\clearpage{}%
\section{Global methods}
\label{sec:GM}

\subsection{Deformations of representations}\label{sec:deforms}

Let $\mathcal{C}_\cO$ be the category of Noetherian complete local $\cO$-algebras with residue field $\F$ and local $\cO$-algebra homomorphisms.
Let $\cG_{/\cO}$ be a split (possibly disconnected) reductive group.
Given a topological group $\Gamma$, a continuous representation $\rbar: \Gamma \ra \cG(\F)$, and $(A,\fm_A) \in \mathcal{C}_\cO$, an $A$-valued lifting of $\rbar$ is a continuous representation $r_A: \Gamma \ra \cG(A)$ such that $\rbar \equiv r_A \pmod{\fm_A}$.

\subsubsection{Deformations of local Galois representations}\label{sec:localdeforms}

Let $L$ be a nonarchimedean local field of characteristic zero.
For a continuous Galois representation $\rhobar: G_L \ra \cG(\F)$, define the functor $D^\square_\rhobar: \mathcal{C}_\cO\ra\mathrm{Sets}$ by letting $D_\rhobar(A)$ be the set of $A$-valued liftings of $\rhobar$.
Then $D^\square_\rhobar$ is represented by a ring $R_\rhobar^\square$, the $\cO$-lifting ring of $\rhobar$.

Suppose now that $\cG = \GL_n$ and that $L$ is not a $p$-adic field.
If $\tau$ is an inertial type for $L$, then let $R_\rhobar^\tau$ denote the reduced $\cO$-flat quotient of $R_\rhobar^\square$ whose $E'$-points correspond to representations $\rho: G_L \ra \GL_n(E')$ with $\mathrm{WD}(\rho)|_{I_L} \cong \tau \otimes_E E'$ for any $E' \subset \ovl{\Q}_p$ which is finite-dimensional over $E$.

Now suppose that $L$ is a $p$-adic field.
Let $T$ and $B$ be the diagonal maximal torus and upper triangular Borel subgroup, respectively, in $\cG = \GL_n$.
Let $\cJ\defeq \Hom_{\Qp}(L,E)$, let $\lambda \in X_*(\un{T}^\vee)\cong  X_*(T^\vee)^{\cJ}$ be a dominant cocharacter, 
and let $\tau$ be a Weil--Deligne inertial type for $L$.
Then let $R_\rhobar^{\lambda,\tau}$ be the reduced $\cO$-flat quotient of $R_\rhobar^\square$ such that $\Spec R_\rhobar^{\lambda,\tau}$ is the Zariski closure of $E'$-points which correspond to potentially semistable representations $\rho: G_L \ra \GL_n(E')$ of Hodge type $\lambda$ with $\mathrm{WD}(\rho)|_{I_L} \cong \tau \otimes_E E'$ for any subfield $E' \subset \ovl{\Q}_p$ which is of finite degree over $E$.
Let $\Spec R_\rhobar^{\lambda,\preceq\tau}$ denote the reduced union $\cup_{\tau' \preceq \tau} R_\rhobar^{\lambda,\tau'}$ (see Definition \ref{defn:preceq} for the relation $\preceq$). 
We also write $R_\rhobar^\tau$ (resp.~$R_\rhobar^{\preceq\tau}$) for $R_\rhobar^{\eta,\tau}$ (resp.~$R_\rhobar^{\eta,\preceq\tau}$).

\begin{rmk}
If the nilpotent element $N_\tau$ of $\tau$ is zero (i.e.~$\tau$ is minimal with respect to $\preceq$), then $R_\rhobar^{\lambda,\tau}$ is a framed potentially crystalline deformation ring defined in \cite{KisinPSS} and is a versal ring for $\cX^{\lambda,\tau}$ at (the point corresponding to) $\rhobar$; see \S\ref{subsec:tame:EGstacks}.
If $\tau$ is maximal with respect to $\preceq$ then $R_\rhobar^{\lambda,\preceq\tau}$ is a framed potentially semistable deformation ring defined in \cite{KisinPSS}.
\end{rmk}

\subsection{Patching axioms}
\label{sec:patch:ax}

Recall from \S \ref{sec:not:Gal} that $\cO_p$ is a finite \'etale $\Z_p$-algebra, which we write as $\prod\limits_{v\in S_p} \cO_v$ where $S_p$ is a finite set and for each $v\in S_p$, $\cO_v$ is the ring of integers in a finite unramified extension $F^+_v$ of $\Q_p$.
Let $G_{/\Z}$ be a split reductive group.
We let $G_0$ be $\Res_{\cO_p/\Z_p} (G_{/\cO_p})$ and denote the Langlands dual group (defined over $\Z$) of $G_0$ by $^L \un{G}$.
Recall from \S \ref{sec:not:Gal} that 
$^L \un{G}=\un{G}^\vee_{/\Z} \rtimes \Gal(E/\Q_p)$ and that $\un{G}^\vee_{/\Z}\cong G^{\vee,\cJ}_{\slash\Z}$ where  $\cJ=\Hom_{\Zp}(\cO_p,\cO)$. %

We fix isomorphisms $\ovl{F^+_v} \ra \ovl{\Q}_p$ for each $v\in S_p$.
Then we recall that an $L$-homomorphism $W_{\Q_p} \ra$ $^L \un{G}(A)$ over a finite cardinality $\cO$-algebra $A$ is equivalent to a collection $(G_{F^+_v} \ra {G}^\vee(A))_{v\in S_p}$ of continuous homomorphisms.
Similarly, a Weil--Deligne inertial $L$-parameter $\tau$ is equivalent to a collection $(\tau_v)_{v\in S_p}$ of Weil--Deligne inertial types.
We now take $G$ to be $\GL_n$.
Let $\rhobar$ be an $L$-homomorphism over $\F$ with corresponding collection $(\rhobar_v)_{v \in S_p}$.

Let $R_\infty$ be $R_\rhobar \widehat{\otimes}_{\cO} R^p$ where
\[
R_\rhobar \defeq \widehat{\bigotimes}_{v\in S_p,\cO} R_{\rhobar_v}^\square
\]
and $R^p$ is a (nonzero) complete local Noetherian equidimensional flat $\cO$-algebra with residue field $\F$ (we suppress the dependence on $R^p$ below).
(Though we will not use it, $R_\rhobar \cong R_{\rhobar'}^\square$ for $\cG = $ $^L \un{G}$ where $\rhobar'$ denotes the unique extension or $\rhobar$ to $G_{\Qp}$.)
For a Weil--Deligne inertial $L$-parameter $\tau$ and a cocharacter $\lambda\in X_*(\un{T}^\vee)$, let $R_\infty(\lambda, \tau)$ (resp.~$R_\infty(\lambda, \preceq\tau)$) be $R_\infty \otimes_{R_\rhobar} R_\rhobar^{\lambda+\eta,\tau}$ (resp.~$R_\infty \otimes_{R_\rhobar} R_\rhobar^{\lambda+\eta,\preceq\tau}$) where
\[
R_\rhobar^{\lambda+\eta,\tau} \defeq \widehat{\bigotimes}_{v\in S_p,\cO} R_{\rhobar_{v}}^{\lambda_{v}+\eta_{v},\tau_{v}} \quad (\textrm{resp. }R_\rhobar^{\lambda+\eta,\preceq\tau} \defeq \widehat{\bigotimes}_{v\in S_p,\cO} R_{\rhobar_{v}}^{\lambda_{v}+\eta_{v},\preceq\tau_{v}}).
\]
Let $X_\infty$, $X_\infty(\lambda,\tau)$, and $X_\infty(\lambda,\preceq\tau)$ be $\Spec R_\infty$, $\Spec R_\infty(\lambda,\tau)$, and $\Spec R_\infty(\lambda,\preceq\tau)$, respectively. 
Let $\Mod(X_\infty)$ be the category of coherent sheaves over $X_\infty$, and let $\Rep_{\cO}(\GL_n(\cO_p))$ denote the category of topological $\cO[\GL_n(\cO_p)]$-modules which are finitely generated over $\cO$. 
Let $\sigma(\lambda,\tau)$ be the finitely generated $E[\GL_n(\cO_p)]$-module $V(\lambda) \otimes \sigma(\tau)$.

\begin{defn}\label{minimalpatching}
A \emph{weak patching functor} for $\rhobar$ is defined to be a nonzero covariant exact functor $M_\infty:\Rep_{\cO}(\GL_n(\cO_p))\ra \Coh(X_{\infty})$ satisfying the following: if $\sigma^\circ(\lambda,\tau)$ is an $\cO$-lattice in $\sigma(\lambda,\tau)$ then
\begin{enumerate}
\item 
\label{it:minimalpatching:1}
$M_\infty(\sigma^\circ(\lambda,\tau))$ is a maximal Cohen--Macaulay sheaf on $X_\infty(\lambda,\preceq\tau)$; and
\item 
\label{it:minimalpatching:2}
for all $\sigma \in \JH(\ovl{\sigma}^\circ(\lambda,\tau))$, $M_\infty(\sigma)$ is a maximal Cohen--Macaulay sheaf on $\ovl{X}_\infty(\lambda,\preceq\tau)$ (or is $0$).
\end{enumerate}
Moreover, we distinguish the following kind of weak minimal patching functors.
\begin{enumerate}
\item[(I)] A weak patching functor is \emph{minimal} if $R^p$ is formally smooth over $\cO$ and whenever $\tau$ is an inertial $L$-parameter (so $N=0$ as in Remark \ref{rmk:not:WDILP}), $M_\infty(\sigma^\circ(\lambda,\tau))[p^{-1}]$, which is locally free over (the regular scheme) $\Spec R_\infty(\lambda,\tau)[p^{-1}]$, has rank at most one on each connected component.
\item[(II)] A weak patching functor is \emph{potentially diagonalizable} if $M_\infty(\sigma^\circ(\lambda,\tau))$ is nonzero whenever each $\rhobar_v$ for $v\in S_p$ has a potentially diagonalizable lift of type $(\lambda_v+\eta_v,\tau_v)$ (in the sense of \cite[\S 1.4]{BLGGT}).
\item[(III)]  If $\rhobar$ is semisimple and $2n$-generic, we say that a weak patching functor $M_\infty$ is \emph{detectable} if $\sigma\in W_\obv(\rhobar)$ implies that $M_\infty(\sigma)$ is nonzero. 
\item[(IV)] Let $\cS$ be a set of types $(\lambda+\eta,\tau)$ with $\tau$ an inertial $L$-parameter (so $N_\tau=0$).
A \emph{minimal patching functor for} $\rhobar$ and $\cS$ is a minimal weak patching functor for $\rhobar$ such that $M_\infty(\sigma^\circ(\lambda,\tau))[p^{-1}]$ has rank one on $\Spec R_\infty(\lambda,\tau)[p^{-1}]$ whenever $(\lambda+\eta,\tau) \in \cS$ and $\sigma^\circ(\lambda,\tau)$ is as above.
\end{enumerate}
\end{defn}

\begin{rmk}
We essentially consider two contexts: one global and one local.
Correspondingly, in practice, $S_p$ will either be the set of $p$-adic places of a number field or contain a single element.
In the global context, $\rhobar$ will arise from restriction of a global characteristic $p$ Galois representation.
However, in either context, all constructions of patching functors that we use will come from (modifications of) the (global) Taylor--Wiles patching method.

When $S_p$ is the set of $p$-adic places of a number field, then $R^p$ will be a formally smooth algebra over a completed tensor product of local deformation rings at some places away from $p$.
(The extra variables, sometimes called auxiliary, are a byproduct of the global nature of the construction.)

When $S_p$ contains a single element, we globalize the local Galois representation $\rhobar$ i.e.~find a suitable number field $F^+$ whose completion at a place $v$ is $F^+_v$ and a Galois representation whose restriction to the decomposition group at $v$ is isomorphic to $\rhobar$.
We then apply the Taylor--Wiles method to this globalization to obtain a patching functor.
In this case, $R^p$ will be a formally smooth algebra over a completed tensor product of local deformation rings at some places away from $v$ (including all places that divide $p$ except for $v$).
In this local context then, the notation $R^p$ may be misleading, for which we apologize.
\end{rmk}

\begin{prop}\label{prop:WE}
Let $\rhobar$ be as above, $M_{\infty}$ be any weak patching functor, and $\sigma$ be a Serre weight such that $M_\infty(\sigma) \neq 0$.     If either $\rhobar_v^{\mathrm{ss}}$ is $(6n-2)$-generic for all $v \mid p$ or $\sigma$ is $(2n-1)$-deep and $\rhobar_v^{\mathrm{ss}}$ is $4n$-generic for all $v \mid p$,  then $\sigma \in W^?(\rhobar^{\mathrm{ss}})$, where $\rhobar^{\mathrm{ss}}$ denotes the $L$-homomorphism over $\F$ corresponding to the collection $(\rhobar_v^{\mathrm{ss}})_{v\in S_p}$.
\end{prop}
\begin{proof}
Let $\sigma$ be $\otimes_{v\in S_p} \sigma_v$.
Suppose first that $\rhobar_v^{\mathrm{ss}}$ is $(6n-2)$-generic for all $v \mid p$.
Then the axioms for $M_\infty$ imply that for each $v_0\in S_p$, $M_\infty(-\otimes \otimes_{v\neq v_0}\sigma_v): \F[\GL_n(k_{v_0})] \ra \mathrm{Vect}_{/\F}$ is an arithmetic cohomology functor in the sense of \cite[Definition 4.2.1]{LLL}.
Then \cite[Corollary 4.2.4]{LLL} implies that $\sigma_v \in W^?(\rhobar_v^{\mathrm{ss}})$ for each $v\in S_p$.

Now suppose that $\sigma$ is $(2n-1)$-deep and $\rhobar_v^{\mathrm{ss}}$ is $4n$-generic for all $v \mid p$.
Let $\lambda\in X^*(\un{T})$ be such that $F(\lambda) \cong \sigma$.
Then $\tau \defeq \tau(w,\tld{w}_h \cdot \lambda+\eta)$ is an $n$-generic tame inertial type for all $w\in \un{W}$.
Moreover, the proof of \cite[Corollary 4.1.12]{LLL} (and \cite[Lemma 5]{enns}) shows that there is a $w\in \un{W}$ such that $\rhobar$ does not have a potentially crystalline lift of type $(\eta,\tau)$ if $\sigma \notin W^?(\rhobar^{\mathrm{ss}})$.
By the axioms of $M_\infty$, it would suffice to show that $\sigma \in \JH(\ovl{\sigma}(\tau))$.
This follows from the observation that if $(s,\mu-\eta)$ is an $n$-generic lowest alcove presentation of $\tau$ %
then $F(\pi^{-1}(\tld{w})\cdot(t_\mu s \tld{w}^{-1}(\eta) - \eta)) \in \JH(\ovl{\sigma}(\tau))$
(The fact that $\mu-\eta$ is $n$-deep ensures that $t_\mu s \tld{w}^{-1}(\eta) - \eta$ is dominant and $p$-restricted for any $\tld{w} \in \tld{W}^+_1$.)
\end{proof}

\begin{prop}\label{prop:patchexist}
\begin{enumerate}
\item 
\label{it:patchexist:1}
If $p\nmid 2n$ and $\rhobar$ is an $L$-homomorphism over $\F$, then a weak potentially diagonalizable patching functor exists.
\item 
\label{it:patchexist:2}
If furthermore for each $v\in S_p$, $\rhobar_v$ has a potentially diagonalizable lift of type $(\xi_v+\eta_v,\tau_v)$ so that the potentially crystalline lifting ring $R_\rhobar^{\xi+\eta,\tau}$ is formally smooth, then a weak minimal potentially diagonalizable patching functor exists. 
\end{enumerate}
\end{prop}
\begin{proof}
We can assume that $\cO_p$ is a domain $\cO_K$ as the general case follows by taking completed tensor products.
By \cite[Theorem 6.4.4]{EGstack}, a (potentially) crystalline potentially diagonalizable lift always exists, say of type $(\xi+\eta,\tau)$. 
Setting $\xi$ and $\tau$ in \cite[\S 2]{CEGGPS} to be this $\xi$ and $\tau$, the construction in \emph{loc.~cit.}~produces a finitely generated $R_\infty[\![\GL_n(\cO_K)]\!]$-module $M_\infty$.
For $\sigma$ a finite $\cO[\![\GL_n(\cO_K)]\!]$-module, we define $M_\infty(\sigma)$ to be $\Hom^{\mathrm{cont}}_{\cO[\![\GL_n(\cO_K)]\!]}(M_\infty,\sigma^\vee)^\vee$ where $(-)^\vee$ denotes the Pontrjagin dual.
Then $M_\infty(-)$ is a weak patching functor by \cite[Lemma 4.18(1)]{CEGGPS}.
By construction, $\rbar$ in \emph{loc.~cit.}~is potentially diagonalizably automorphic, which implies that $M_\infty$ is potentially diagonalizable by the proof of \cite[Theorem 4.3.8]{LLL}. 
If $R_\rhobar^{\xi+\eta,\tau}$ is formally smooth, then $M_\infty$ is minimal.
\end{proof}

\begin{rmk}
For our purposes, the hypothesis $p \nmid 2n$ is often implicitly assumed since if $p\mid 2n$, then there are no $n$-generic tame inertial $L$-parameters.
\end{rmk}

Let $K$ be a finite unramified extension of $\Q_p$ with ring of integers $\cO_K$, and we now let $\cO_p$ be $\cO_K$.
We assume for the remainder of this section that $p \nmid 2n$ (otherwise there are no $n$-generic tame inertial $L$-parameters).
Then an $L$-homomorphism over $\F$ is equivalent to a representation $\rhobar: G_K \ra \GL_n(\F)$ which we also denote by $\rhobar$.

\begin{prop}\label{prop:obvpatchexist}
If $\rhobar: G_K \ra \GL_n(\F)$ is a semisimple  continuous Galois representation whose restriction $\rhobar|_{I_K}$ corresponds to a $4n$-generic tame inertial $L$-parameter over $\F$, then any weak potentially diagonalizable patching functor for $\rhobar$ is detectable.
Moreover, a weak minimal detectable potentially diagonalizable patching functor exists.
\end{prop}
\begin{proof}
The first part follows from the proof of \cite[Theorem 4.3.8]{LLL} using Proposition \ref{prop:WE} in place of Corollary 4.2.7 in \emph{loc. cit.}.  
For each $v\in S_p$, $\rhobar_v$ is Fontaine--Laffaille and so $\rhobar_v$ has a crystalline potentially diagonalizable lift for some Fontaine--Laffaille Hodge--Tate weights and the corresponding crystalline lifting ring is formally smooth (\cite[Lemma 1.4.3(2)]{BLGGT},\cite[Lemma 2.4.1]{CHT}). 
(Alternatively, one can use \cite[Theorem 3.4.1]{LLL}.)
Proposition \ref{prop:patchexist}(\ref{it:patchexist:2}) implies that a weak minimal potentially diagonalizable patching functor exists, which is then necessarily detectable.
\end{proof}

\begin{prop}\label{prop:patchnonzero}  
Let $\rhobar$ be a semisimple Galois continuous representation whose restriction $\rhobar|_{I_K}$ corresponds to a $4n$-generic tame inertial $L$-parameter over $\F$.
Let $\lambda \in X_*(\un{T}^\vee)$ be dominant with $\lambda_j \in [0, h]^n$.   
Let $\tau$ be a tame inertial type with a fixed $\max \{2n, h+n-1\}$-generic lowest alcove presentation \emph{(}cf.~Definition \ref{defi:gen}(\ref{def:LApres}))\emph{)}.
Let $\sigma^\circ(\lambda,\tau)$ be an $\cO$-lattice in $\sigma(\lambda,\tau)$.
Let $M_\infty$ be a weak detectable minimal patching functor for $\rhobar$ (which exists by Proposition \ref{prop:obvpatchexist} if $\cO_p$ is a domain).
Then the following are equivalent.
\begin{enumerate}
\item 
\label{it:patchnonzero:1}
$M_\infty(\sigma^\circ(\lambda,\tau))$ is nonzero;
\item
\label{it:patchnonzero:2}
$R_\infty(\lambda,\tau)$ is nonzero; 
\item 
\label{it:patchnonzero:3}
$R^{\lambda+\eta,\tau}_\rhobar$ is nonzero; and
\item 
\label{it:patchnonzero:4}
there is a $\lambda$-compatible lowest alcove presentation of $\rhobar$ and $\tld{w}(\rhobar,\tau) \in \Adm(\lambda+\eta)$. %
\end{enumerate}
\end{prop}
\begin{proof}
If $M_\infty(\sigma^\circ(\lambda,\tau))$, which is supported on $X_\infty(\lambda,\tau)$, is nonzero, then $R_\infty(\lambda,\tau)$ must be nonzero.
By definition, $R_\infty(\lambda,\tau)$ is nonzero if and only if $R^{\lambda+\eta,\tau}_\rhobar$ is nonzero.

If $R^{\lambda+\eta,\tau}_\rhobar$ is nonzero, then there is a $\lambda$-compatible lowest alcove presentation of $\rhobar$ such that $\tld{w}(\rhobar,\tau) \in \Adm(\lambda+\eta)$ by Corollary \ref{cor:admshape}.
If $\tld{w}(\rhobar,\tau) \in \Adm(\lambda+\eta)$, then $W_\obv(\rhobar) \cap \JH(\ovl{\sigma}^\circ(\lambda,\tau))$ is nonempty by Proposition \ref{prop:obvint}.
If $\sigma$ is in this intersection then $M_\infty(\sigma)$ is nonzero, which implies that $M_\infty(\sigma^\circ(\lambda,\tau))$ is nonzero by exactness of $M_\infty$.
\end{proof}

We will also need a version of the above result for certain non-semisimple $\rhobar$:

\begin{lemma}\label{lem:obvious_tame_type_lift}
Let $\kappa \in X_1(\un{T})$ be $(n-1)$-deep.
Suppose that $\rhobar:G_K\to \GL_n(\F)$ is of the form
\[ \begin{pmatrix}  \ovl{\chi}_1 &* &\cdots & * \\
0&\ovl{\chi}_2& \cdots& *\\
\vdots &&\ddots&\vdots\\
0&\cdots &0 & \ovl{\chi}_n
\end{pmatrix}\]
where $\ovl{\chi}_i|_{I_K}= \prod_{j\in \cJ} \ovl{\omega}_{K,\sigma_j}^{\kappa_{j,i}+\eta_{j,i}}$.
Then $\rhobar$ can be lifted to a representation $\rho$ of the form
 \[\begin{pmatrix}  \chi_1 &* &\cdots & * \\
0&\chi_2& \cdots& *\\
\vdots &&\ddots&\vdots\\
0&\cdots &0 & \chi_n
\end{pmatrix}\]
where $\chi_i|_{I_K}= \eps^{n-i} \prod_{j\in \cJ} \omega_{K,\sigma_j}^{\kappa_{j,i}}$.
Any such lift is potentially crystalline of type $(\eta,\tau(1,\kappa))$.
\end{lemma}
\begin{proof}
The depth hypothesis implies that $\ovl{\chi}_i\neq \ovl{\chi}_{i'}\ovl{\eps}$ for all $1\leq i<i'\leq n$, so that $H^2(G_K,\ovl{\chi}_i\ovl{\chi}_{i'}^{-1}) = 0$ and there are no obstructions to finding an upper triangular lift $\rho$ of $\rhobar$ with characters $\chi_i$ on the diagonal. 

We now check that such a lift $\rho$ is potentially crystalline of type $(\eta,\tau(1,\kappa))$.
Since for each embedding $j\in \cJ$, the $j$-labelled Hodge-Tate weights of $\rho$ increase along the diagonal, $\rho$ is de Rham, by \cite[Lemme 6.5]{Berger}.
Hence $\rho$ is potentially semistable. Clearly the Hodge-Tate weight of $\rho$ is $\eta$. Now $\Dpst(\rho)$ is a successive extension of $\Dpst(\chi_i)$ as $I_K$-representations, and since $\Dpst(\chi_i)\cong \prod_{j\in \cJ} \omega_{j}^{\kappa_{j,i}}$ as $I_K$-representations, $\rho$ has inertial type $\tau(1,\kappa)$. 
Finally, the depth hypothesis on $\kappa$ implies that $\Dpst(\rho)|_{I_K}$ is a direct sum of $n$ distinct characters, which forces the monodromy operator $N$ on $\Dpst$ to be $0$. Thus $\rho$ is in fact potentially crystalline of type $(\eta,\tau(1,\kappa))$.
\end{proof}

\begin{prop}\label{prop:obvious_wild_lift}
Suppose that $\kappa$ and $\rhobar$ are as in Lemma \ref{lem:obvious_tame_type_lift}.
If $\rhobar$ is $4n$-generic and $M_\infty$ is a weak potentially diagonalizable patching functor (which exists by Proposition \ref{prop:patchexist}), then $M_\infty(F(\kappa)) \neq 0$.
Moreover, if $F(\kappa) \in \JH(\ovl{\sigma}(\lambda,\tau))$, then $R_\rhobar^{\lambda+ \eta,\tau}$ is nonzero.
\end{prop}
\begin{proof}
Since $\rho$ is ordinary in Lemma \ref{lem:obvious_tame_type_lift}, $\rho$ is potentially diagonalizable by \cite[Lemma 1.4.3]{BLGGT}.
Then $M_\infty(\sigma^\circ(\tau(1,\kappa)))$ is nonzero for any $\cO$-lattice $\sigma^\circ(\tau(1,\kappa))$ in $\sigma(\tau(1,\kappa))$.

Let $(\tld{w},\omega)$ be a lowest alcove presentation for $F(\kappa)$ so that $\kappa = \pi^{-1}(\tld{w}) \cdot (\omega-\eta)$.
Since $\rhobar$ is $4n$-generic, $(\tld{w}, \omega)$ is a $3n$-generic lowest alcove presentation. 
We let $\tld{w} \in \tld{\un{W}}^+_1$ be $t_{\eta_w} w$.
Then writing $\tau(1,\kappa) \cong \tau(\pi^{-1}(w)^{-1} w,\omega + \pi^{-1}(w)^{-1}(\eta_w-\eta))$ and $\rhobar{^{\mathrm{ss}}}|_{I_K} \cong \taubar(\pi^{-1}(w)^{-1} w,\omega + \pi^{-1}(w)^{-1}(\eta_w))$ using Proposition \ref{prop:PSLAP} gives compatible lowest alcove presentations of $\tau(1,\kappa)$ and $\rhobar|_{I_K}$.
Since $\tld{w}(\rhobar,\tau(1,\kappa)) = t_{w^{-1}(\eta)}$, $W^?(\rhobar^{\mathrm{ss}}) \cap \JH(\ovl{\sigma}(\tau(1,\kappa))) = \{F(\kappa)\}$ by Corollary \ref{cor:extremeintersect}.
For any $\sigma \in \JH(\ovl{\sigma}(\tau(1,\kappa)))$ with $\sigma \not\cong F(\kappa)$, $ \sigma \notin W^?(\rhobar^{\mathrm{ss}})$ and $\sigma$ is $2n$-deep by Proposition \ref{prop:JHbij}, 
and so $M_\infty(\sigma) = 0$ by Proposition \ref{prop:WE}.
This implies that $M_\infty(F(\kappa))$ is nonzero.
The final part then follows from the axioms satisfied by $M_\infty$.
\end{proof}
\clearpage{}%
\clearpage{}%
\section{Monodromy, potentially crystalline stacks, and local models} \label{sec:monodromy}
  
As we saw earlier, Theorem \ref{thm:fh_local_model} and Theorem \ref{thm:Breuil-Kisin_local_model} gives a Zariski local description of the moduli of Breuil--Kisin module $Y^{[0,h],\tau}$ in terms of certain affine opens of global Schubert varieties.
In this section, we give a similar description for the potentially semistable stacks of type $(\lambda,\tau)$.
This will the main local ingredient for the global applications.

\subsection{The monodromy condition}
\label{sec:mon:cond}
We are in the setup of \S \ref{sec:BKwithdescent}.
We have fixed a tame inertial $L$-parameter $\tau:I_{\Qp}\ra\widehat{\un{T}}(E)$ together with a $1$-generic lowest alcove presentation $(s,\mu)$. %
To the tame inertial $L$-parameter above, we associate a tame inertial type for $K$, denoted by $\tau: I_K\rightarrow \GL_n(E)$, as described in the Example \ref{ex:data:type}.
Let $r$ be the order of $s_\tau$. 
As in \S \ref{sec:BKwithdescent} we write $K'$ for the unramified extension of $K$ of degree $r$, let $k'$ be its residue field and set $f' = fr$, $e' =p^{f'}-1$. 
Finally, recall that we have fixed an identification of $\cJ' = \Hom_{\Qp}(K', E)$ with $\Z/f' \Z$ by the choice of the isomorphism $\iota: \ovl{K}\stackrel{\sim}{\ra}\ovl{\Q}_p$.

\vspace{2mm}

We begin by recalling some notations from \cite{KisinFcrys}.  Let $\cO_{K'}^{\rig}=\big(\varprojlim_{n} W(k')[\![u',\frac{u'^n}{p}]\!]\big)[\frac{1}{p}]$ denote the ring of rigid analytic functions on the open unit disc over $K'$. %
There is a natural injective map $\cO_{K'}^{\rig}\into (W(k')\otimes \bQ_p)[\![u']\!]=K'[\![u']\!]$, which identifies $\cO_{K'}^{\rig}$ as the subring consisting of power series $\sum_{n = 0}^{\infty} a_n (u')^n$ such that $|a_n|_p R^n \ra 0$ for all $R < 1$. Clearly $\fS_{L'}\subset \cO_{K'}^{\rig}$.
Set 
$$
\lambda = \prod_{n=0}^{\infty} \phz^n \left(\frac{E(u')}{p} \right) \in \cO_{K'}^{\rig}.
$$ 
We define a derivation on $\cO_{K'}^{\rig}$ by $N_{\nabla} \defeq - u' \lambda \frac{d}{d(u')}$; the Frobenius $\phz$ on $\fS_{L'}$ extends to a Frobenius $\phz$ on $\cO_{K'}^{\rig}$. If $\La$ is a finite flat $\cO$-algebra, we define  $\cO_{K',\La}^{\rig}\defeq \cO_{K'}^{\rig} \otimes_{\Zp}\La$. For any Kisin module $\fM \in Y^{[0,h],\tau}(\Lambda)$, we define its base change to $\cO^{\rig}_{K',\La}$ as $\fM^{\rig} \defeq \fM \otimes_{\fS_{L'}} \cO_{K'}^{\rig}$.

One has the following important result of Kisin:%
\begin{thm} \label{thmKisin} 
The module $\fM^{\rig}[1/\lambda]$ is equipped with a unique derivation $N_{\fM^{\rig}}$ over $N_\nabla$ such that 
\begin{equation} \label{commrel}
N_{\fM^{\rig}} \phi_{\fM^{\rig}} = E(u') \phi_{\fM^{\rig}} N_{\fM^{\rig}}
\end{equation}
and $N_{\fM^{\rig}}\mod u' = 0$.   The module $\fM^{\rig}$ is stable under $N_{\fM^{\rig}}$ if and only if $T^*_{dd}(\fM)[1/p]$ is the restriction to $G_{K_{\infty}}$ of a potentially crystalline representation of $G_K$ over $\Lambda[\frac{1}{p}]$, of inertial type $\tau$ for $K$ and Hodge-Tate weights in $[0,h]$.  
\end{thm} 
\begin{proof} This is essentially \cite[Corollary 1.3.15]{KisinFcrys}. The result in \emph{loc.~cit.} is stated there without tame descent data, however, using the full faithfulness of the restriction from crystalline $G_{L'}$-representations to $G_{L'_{\infty}}$-representations (Corollary 2.1.14 in \emph{loc.~ cit.}), we see the stability of $\fM^\rig$ under $N_{\fM^\rig}$ is equivalent to $V=T^*_{dd}(\fM)[1/p]$ extending to a potentially crystalline representation of $G_K$, which becomes crystalline over $L'$. The fact that it has inertial type $\tau$ follows from the fact that $D_{pst}(V)$ is isomorphic to $((\fM/u'\fM)[\frac{1}{p}])^\vee$ as an $I_K$-representation.
\end{proof}

\begin{defn} \label{defn:mcond} Let $\Lambda$ be a finite flat $\cO$-algebra.  We say that $\fM \in Y^{[0, h], \tau}(\Lambda)$ satisfies the \emph{monodromy condition} if $N_{\fM^{\rig}}(\fM^{\rig}) \subset \fM^{\rig}$.   
\end{defn}

The significance of the monodromy condition is that by Theorem \ref{thmKisin}, it captures the condition that the $G_{K_{\infty}}$ representation attached to a Breuil--Kisin module comes from potentially crystalline $G_K$-representations, at least on finite $E$-algebras.  We would like to study this condition when one varies the Breuil--Kisin module in a family, and understand it explicitly in terms of the coordinate charts of $Y^{[0,h],\tau}$ produced by Theorem \ref{thm:fh_local_model}. %

 Let $R$ be a $p$-adically complete, topologically of finite type flat $\cO$-algebra.  %
We define $\cO^{\rig}_R=\varprojlim_n R[\![u',\frac{u'^n}{p}]\!][\frac{1}{p}]$, which can be interpreted as the ring of rigid analytic function on the open unit ball $(\Spf R)^{\rig}\times \bD^{\circ}$ over the rigid analytic generic fiber of $\Spf R$.
There is a natural injection $\cO^{\rig}_R\into R[\frac{1}{p}][\![u']\!]$ whose image is stable under $\frac{d}{du'}$, and we will always think of the former as a subring of the latter via this injection.
Note that for each $m\geq 0$, we have map $\cO^{\rig}_R \onto (R[\frac{1}{p}])[u']/\phz^m(E(u'))$, which we can roughly think of as ``evaluation at $(-p)^{\frac{1}{e'p^m}}$'' (in contrast, there is no such map for $R[\frac{1}{p}][\![u']\!]$). If $F\in \cO^{\rig}_R$, we write $F|_{\phz^m(E(u'))=0}$ to mean the image of $F$ under this evaluation map. Note that the condition $F|_{\phz^m(E(u'))=0}=0$ is a Zariski closed condition on $\Spec R[\frac{1}{p}]$.
{Finally, we note that the formation of $\cO^{\rig}_R$ is a Zariski sheaf on $\Spf R$ (and thus we are free to make Zariski localizations on $\Spf R$ in our arguments below): Indeed, a Zariski open cover of $\Spf R$ induces an open cover of the adic space $(\Spf R[\![u']\!])^{\mathrm{ad}}$ whose generic fiber over $(\Spf \Zp)^{\mathrm{ad}}$ is $(\Spf R)^{\rig}\times \bD^{\circ}$, and these adic spaces are sheafy by \cite[Theorem 2.2 ]{huber94}}.

We also define the variant $\cO^{\rig}_{K',R}=\varprojlim_n (W(k')\otimes_{\bZ_p}R)[\![u',\frac{u'^n}{p}]\!][\frac{1}{p}]$, which is a subring of $(K'\otimes_{\bZ_p} R)[\![u']\!]$. Since $R$ is an $\cO$-algebra, we have a decomposition $\cO^{\rig}_{K',R}=\prod_{j\in \cJ'} \cO^{\rig}_R$. The operators $\phz$, $N_\nabla$ continue to make sense on $\cO^{\rig}_{K',R}$.

Given $\fM\in Y^{[0,h],\tau}(R)$, we define $\fM^{\rig}=\fM\otimes_{\fS_{L',R}} \cO^{\rig}_{K',R}$, which decomposes as 
 $\fM^{\rig} =\oplus_{j'\in \cJ'} \fM^{\rig, (j')}$.

\begin{prop} \label{prop:mconverge} Let $R$ be a $p$-adically complete, topologically finite type flat $\cO$-algebra, and $\fM\in Y^{[0,h],\tau}(R)$.
 \begin{enumerate}
\item There exists a unique derivation $N_{\fM^\rig}:\fM^{\rig}[\frac{1}{\lambda}]\to \fM^{\rig}[\frac{1}{\lambda}]$ over $N_\nabla$ such that 
\begin{equation} \label{commrel_general}
N_{\fM^{\rig}} \phi_{\fM^{\rig}} = E(u') \phi_{\fM^{\rig}} N_{\fM^{\rig}}
\end{equation}
and $N_{\fM^{\rig}}\mod u' = 0$.
\item Suppose $\fM$ admits an eigenbasis $\beta=(\beta^{(j')})_{j'\in\cJ'}$, and recall that $C^{(j')}_{\fM, \beta} \in \Mat_n(R[\![u']\!])$ is the matrix of $\phi_\fM:\fM^{(j'-1)}\to \fM^{(j')}$.
Define inductively the sequence $N^{(j')}_{i}\in \Mat_n(R[\frac{1}{p}][\![u']\!])$ for $j'\in \cJ'$ and $i\geq 0$ as follows:
\begin{itemize}
\item $N^{(j')}_0=0$ for all $j'\in\cJ'$.
\item For each $i \geq 1$, define
\[N_{ i}^{(j')} \defeq E(u')C^{(j')}_{\fM, \beta} \phz(N_{i-1}^{(j'-1)}) C^{(j')}_{\fM, \beta})^{-1}  - N_{\nabla}(C^{(j')}_{\fM, \beta}) (C^{(j')}_{\fM, \beta})^{-1}.
\] 
\end{itemize} 
Then for each $j'\in \cJ'$, the sequence $N^{(j')}_i$ converges in $\Mat(R[1/p][\![u']\!])$ to an element $N^{(j')}_\infty$. Furthermore $N^{(j')}_\infty\in \frac{1}{\lambda^{h-1}} \Mat_n(\cO^{\rig}_{R})$, and is the matrix of $N_{\fM^\rig}: \fM^{\rig,(j')}\to \fM^{\rig,(j')}$ with respect to $\beta^{(j')}$.
\end{enumerate}
\end{prop}
\begin{proof} To prove both parts, we can work Zariski locally on $R$ and hence assume that $\fM$ admits a eigenbasis $\beta$. First, assume $N_{\fM^{\rig}}$ exist, then it is $(W(k')\otimes R)$-linear so it preserves $\fM^{\rig,(j')}$. Let $N^{(j')}$ be the matrix of $N_{\fM^{\rig}}$ with respect to $\beta^{(j')}$.  Let $C^{(j')} := C^{(j')}_{\fM, \beta}$.  We can thus write the commutation relation (\ref{commrel_general}) as
\begin{equation}\label{m1}
N^{(j')}C^{(j')} =E(u') C^{(j')} \phz\big(N^{(j'-1)}\big)-N_{\nabla}\big(C^{(j')}\big).
\end{equation}
Then $N_{\fM^{\rig}}$ is unique since this system has at most one solution even in $\Mat_n(R[\frac{1}{p}][\![u']\!])^{\cJ'}$. Indeed, the difference $X_{j'}$ of any two solutions will satisfy (noting that $C^{(j')} \in \GL_n(R[\frac{1}{p}][\![u']\!])$)
\[X_{j'}=E(u')\Ad\big(C^{(j')}\big)(\phz(X_{j'}))\]
and $X_{j'} \mod u'=0$. From this we deduce by induction that $X_{j'}$ is infinitely divisible by $u'$ in $\Mat_n(R[\frac{1}{p}][\![u']\!])$, hence must be $0$.

Thus, we are left with showing the second part of the Proposition, since the limiting $N^{(j')}_\infty$ constructed there will be a solution to the commutation relation (\ref{m1}). 
We show by induction that 
\begin{equation} \label{m3}
\lambda^{h-1}\big(N^{(j')}_{i+1} - N^{(j')}_{i}\big) \in \frac{(u')^{p^{i}}}{p^{1+(h-1)(i+1)}}  \phz^{i+1}(\lambda^h) \Mat(R[\![u']\!]).
\end{equation}
For the base case, we have
\begin{align*}\lambda^{h-1} N_{1}^{(j')} &= \lambda^{h} u'\frac{d}{du'}\big(C^{(j')}\big) \big(C^{(j')}\big)^{-1}\\
&=\phz(\lambda^h)\frac{u'}{p^h}\frac{d}{du'}\big(C^{(j')}\big)E(u')^h\big(C^{(j')}\big)^{-1}\ \in \phz(\lambda^h)\frac{u'}{p^h}\Mat_n(R[\![u']\!]), 
\end{align*}
since $C^{(j')}, E(u')^h(C^{(j')})^{-1}\in \Mat_n(R[\![u']\!])$ by the height condition.
Now suppose we already know (\ref{m3}) up to $i-1\geq 0$. We have
\begin{equation*}
\lambda^{h-1} \left(N_{i+1}^{(j')} - N^{(j')}_i \right) = \frac{E(u')^h}{p^{h-1}} C^{(j')} \phz \left(\lambda^{h-1} \left(N_{i}^{(j'-1)} - N^{(j'-1)}_{i-1} \right) \right) \big(C^{(j')}\big)^{-1}
\end{equation*} 
belongs to $\frac{1}{p^{h-1}}\phz\Big(\frac{(u')^{p^{i-1}}\phz^i(\lambda^h)}{p^{1+(h-1)i}}\Big)\Mat_n(R[\![u']\!])=\frac{(u')^{p^i}}{p^{1+(h-1)(i+1)}}\phz^{i+1}(\lambda^h)\Mat_n(R[\![u']\!])$, since we have $C^{(j')}, E(u')^h\big(C^{(j')}\big)^{-1}\in \Mat_n(R[\![u']\!])$ by the height condition. This finishes the inductive step.

Property (\ref{m3}) shows the convergence of $N^{(j')}_i$ in $\Mat_n(R[\frac{1}{p}][\![u']\!])$, and the limit necessarily is the unique solution of the system (\ref{m1}). It remains to show $N^{(j')}_\infty\in \frac{1}{\lambda^{h-1}}\Mat_n(\cO_R^{\rig})$. From (\ref{m3}), we just need to show an element in $R[\frac{1}{p}][\![u']\!]$ of the form
\[\psi=\sum_{i=0}^\infty \frac{(u')^{p^i}}{p^{1+(h-1)(i+1)}}\phz^{i+1}(\lambda)f_i(u')\]
with $f_i(u')\in R[\![u']\!]$ must belong to $\cO_R^{\rig}$. Equivalently, we need to show that for each fixed $m$, $\psi$ lies in the image of the homomorphism $R[\![x,y]\!][\frac{1}{p}]\to R[\frac{1}{p}][\![u']\!]$ sending $x$ to $u'$ and $y$ to $\frac{(u')^m}{p}$. However this is clear, since $\phz^{i+1}(\lambda)\in \bZ_p[\![\frac{(u')^m}{p}]\!]$ and $\frac{(u')^{p^i}}{p^{1+(h-1)(i+1)}}f_i(u')\in (u')^iR[\![u',\frac{(u')^m}{p}]\!]$ for $i$ sufficiently large relative to $m$.
\end {proof}
\begin{prop} 
\label{prop:Mcond}
Let $\Lambda$ be a finite flat $\cO$-algebra and let $\fM \in Y^{[0,h],\tau}(\Lambda)$ with a eigenbasis $\beta$. Let $N_{\fM, \infty}^{(j)}$ be as in Proposition \ref{prop:mconverge}. 
Then $\fM^{\rig}$ satisfies the monodromy condition if and only if for all $0 \leq t \leq h-2$ and $j' \in \cJ'$,  
$(\frac{d}{du'})^t|_{E(u')=0}(\lambda^{h-1} N_{\infty}^{(j')}) = 0$.
\end{prop} 
\begin{proof}  The forward direction is clear.  For the reverse direction, we deduce from the commutation relation (\ref{m1}) that $(\frac{d}{du'})^t|_{\phz^m(E(u'))=0}(\lambda^{h-1} N_{\infty}^{(j')}) = 0$ for all $m\geq 0$. It follows that $\lambda^{h-1}N^{(j')}_\infty \in \lambda^{h-1}\Mat_n(\cO^{\rig}_\Lambda)$.%
\end{proof}

\begin{cor} \label{cor:ideal_mon}Let $R$ be a $p$-adically complete topologically finite type flat $\cO$-algebra, let $\fM\in Y^{[0,h],\tau}(R)$. Assume $\fM$ is free over $\fS_{L',R}$. Let $\beta$ be a eigenbasis for $\fM$. Let $N^{(j)}_{\infty}$ be the matrix of $N_{\fM^\rig}$ with respect to $\beta^{(j)}$. Then 
\begin{itemize}
\item The condition $(\frac{d}{du'})^t|_{E(u')=0}(\lambda^{h-1} N_{\infty}^{(j')}) = 0$ for all $0\leq t\leq h-2$ and $j'\in\cJ'$ defines a Zariski closed subset $\Spec R[\frac{1}{p}]^{\fM,\nabla_\infty}\subset \Spec R[\frac{1}{p}]$, which is independent of the choice of the eigenbasis $\beta$.
\item The formation of $\Spec R[\frac{1}{p}]^{\fM,\nabla_\infty}$ is compatible with arbitrary base change on the pair $(R,\fM)$ satisfying the above hypotheses.
\end{itemize}
\end{cor}
\begin{proof} By Proposition \ref{prop:mconverge}, each entry of $(\frac{d}{du'})^t|_{E(u')=0}(\lambda^{h-1} N_{\infty}^{(j')})$ is an element of $R[\frac{1}{p}][u']/E(u')$. The Zariski closedness is immediate. A change of the choice of eigenbasis $\beta$ changes $N^{(j')}_{\infty}$ to $\Ad(X^{(j')})(N^{(j')}_{\infty})-X^{(j')}\lambda u'\frac{d}{du'}((X^{(j')})^{-1})$ for $X^{(j')}\in \GL_n(R[\![u']\!])$, and an easy computation shows the independence on the choice of eigenbasis. The last assertion is immediate, as the generators for the ideal cutting out our condition are literally the same if we compute using compatible choice of eigenbases.
\end{proof}

\begin{prop} \label{prop:functorial_mon_condition}Let $R$ be a $p$-adically complete topologically finite type flat $\cO$-algebra, let $\fM\in Y^{[0,h],\tau}(R)$. There is a unique ideal $I_{\fM,\nabla_\infty}$ such that 
\begin{itemize} 
\item $R/I_{\fM,\nabla_\infty}$ is $\cO$-flat; and
\item For any flat map $R\to S$ such that $S$ is a $p$-adically complete topologically finite type flat $\cO$-algebra and the base change $\fM_S$ of $\fM$ to $S$ is free, one has $\Spec S[\frac{1}{p}]/I_{\fM,\nabla_\infty}=\Spec S[\frac{1}{p}]^{\fM_S,\nabla_\infty}$.
\end{itemize}
Furthermore formation of $I_{\fM,\nabla_\infty}$ is compatible with flat base change on the pairs $(R,\fM)$ as above.
\end{prop}
\begin{rmk} The compatibility with base change means that for flat $R\to S$, we have $I_{\fM,\nabla_\infty}S=I_{\fM_S,\nabla_\infty}$. In general, we always have an inclusion $I_{\fM,\nabla_\infty}S\subset I_{\fM_S,\nabla_\infty}$.
\end{rmk}
\begin{proof} The existence when $\fM$ is free follows from Corollary \ref{cor:ideal_mon}, {taking the Zariski closure of $\Spec R[\frac{1}{p}]^{\fM,\nabla_\infty}$ in $\Spec R$.}
The uniqueness then follows, since there is a Zariski cover of $\Spf \, R$ by $p$-adic affine formal schemes topologically of finite type over $\cO$, over which $\fM$ becomes free. Finally the existence in general and the compatibility with flat base change follows from the base change property in Corollary \ref{cor:ideal_mon}.
\end{proof}

We now wish to analyze the ideal $I_{\fM,\nabla_\infty}$ more closely, and in particular find some approximation of it which is more algebraic in nature. To this end, we let $R$ be a $p$-adically complete topologically finite type $\cO$-flat algebra, and fix a pair $(\fM,\beta)$ where $\fM\in Y^{[0,h],\tau}(R)$ and $\beta$ is an eigenbasis for $\fM$. Given this data we get the matrices of partial Frobenii $A^{(j')}=A^{(j')}_{\fM,\beta}$ and $C^{(j')}=C^{(j')}_{\fM,\beta}$ for $j'\in \cJ'$, cf the discussion after Definition \ref{defn:eigenbasis}. On the other hand, Proposition \ref{prop:mconverge} constructs the matrices $N^{(j')}_\infty\in \frac{1}{\lambda^{h-1}}\Mat_n(\cO^{\rig}_R)$ given by the infinite series
\begin{align} \label{eq:Ninf}
N_{\infty}^{(j')}=N_1^{(j')} + \sum^{\infty}_{i = 1} \left(\prod_{k=0}^{i-1}\phz^{k}\big(C^{(j'-k)}\big)\right) \phz^i(N_{1}^{(j'-i)})\left(\prod_{k=i-1}^{0}\phz^{k} \big(E(u') (C^{(j'-k)})^{-1}\big) \right)
\end{align}
where $N_{1}^{(j')} = \lambda u'\frac{d}{du'}(C^{(j')}) (C^{(j')})^{-1}$.

Thus we can write
\[p^{h}\lambda^{h-1}N_{\infty}^{(j')}= \phz(\lambda)^{h} u'\frac{d}{du'}(C^{(j')}) (v+p)^h (C^{(j')})^{-1}+\sum_{i=1}^{\infty} X_i^{(j')}\]
where
\[X_i^{(j')} := \frac{\phz^{i+1}(\lambda)^{h}}{p^{i(h-1)}} \left(\prod_{k=0}^{i-1}\phz^{k}(C^{(j'-k)})\right) \phz^i\left(u' \frac{d}{du'} C^{(j'-i)} \right) \left(\prod_{k=i}^{0}\phz^{k} \left((v+p)^{h} (C^{(j'-k)})^{-1} \right) \right).\]
We can rewrite this in terms of the $A^{(j')}$ by ``removing the descent data'' as in \cite[Page 52]{LLLM}. We obtain (see \eqref{eq:CtoA}):
\begin{equation} 
\label{eq:remove_descent}
p^{h} \Ad \big((s'_{\mathrm{or}, j'})^{-1} (u')^{-\mathbf{a}^{\prime \, (j')}} \big) (\lambda^{h-1} N_{\infty}^{(j')}) =
-\phz(\lambda)^{h}P_N(A^{(j')})+\sum_{i= 1}^{\infty}\phz^{i+1}(\lambda)^{h}Z_i^{(j')} %
\end{equation}
where (cf.~\cite[Lemma 5.4]{LLLM})
\begin{align*}
P_{N}(A^{(j')})&\defeq\text{\small{$\left(- e'v \frac{d}{dv} A^{(j')} - [\mathrm{Diag}((s'_{\mathrm{or}, j'})^{-1}(\mathbf{a}^{\prime \, (j')})), A^{(j')}] \right)  (v+p)^{h}(A^{(j')})^{-1}\in L^+\cM(R)\subset \Mat_n(R[\![v+p]\!])$}},
\\
Z^{(j')}_i&\defeq \Ad \big((s'_{\mathrm{or}, j'})^{-1} (u')^{-\mathbf{a}^{\prime \, (j')}} \big) \left(\frac{1}{\phz^{i+1}(\lambda)^{h}} X_i^{(j')} \right). 
\end{align*}
 We make the following definition:
\begin{defn} Let $R$, $(\fM,\beta)$ be as above, giving rise to the matrices of partial Frobenii $A^{(j')}$.
We define the ideal $I_{\fM,\beta,\nabla_1} \subset R$ to be the ideal generated by the elements
$(\frac{d}{dv})^t (v^{-\delta_{k>l}}P_N(A^{(j')})_{kl})|_{v=-p}$
for $0\leq t\leq h-2$, $j'\in \cJ'$ and $1\leq k,l\leq n$.
\end{defn} 
\begin{rmk} \label{rmk:derivative_vs_vanishing}
When $p>h-2$, the condition that a series $F=\sum_{m=0}^\infty  a_m(v+p)^m$ belongs to $(v+p)^{h-1}R[\![v+p]\!]$ is equivalent to the condition that $(\frac{d}{dv})^{t}F|_{v=-p}=0$ for $0\leq t\leq h-2$. Thus in this case, we see that the ideal $I_{\fM,\beta,\nabla_1}$ cuts out the locus in $\Spec R$ where
\begin{equation}\label{eqn:algmono}
\left(e'v \frac{d}{dv} A^{(j')} - A^{(j')}\mathrm{Diag}((s'_{\mathrm{or}, j'})^{-1}(\mathbf{a}^{\prime \, (j')})) \right) (A^{(j')})^{-1}\in \frac{1}{(v+p)}L^+\cM(R)
\end{equation}
for all $j'\in \cJ'$. 
Note that the condition (\ref{eqn:algmono}) depends only on the image of $j'$ in $\cJ$. {Furthermore, for each fixed embedding $j'\in\cJ$, because $e'$ is invertible in $\cO$, condition \eqref{eqn:algmono} is the same as condition \eqref{eq:nablaa} with $\bf{a}=-\big((s'_{\mathrm{or}, j'})^{-1}(\mathbf{a}^{\prime \, (j')})\big)\big/e'$ (and hence is a specialization of condition \eqref{eq:universalnabla})}.
\end{rmk}

\begin{prop} \label{prop:monodromy_approximation} Let $\tau$ be a tame inertial type with a lowest alcove presentation $(s,\mu)$. Assume that $\mu$ is $m$-deep in $\un{C}_0$. Let $R$ be a $p$-adically complete topologically finite type $\cO$-flat algebra. Let $\fM\in Y^{[0,h],\tau}(R)$ and $\beta$ an eigenbasis of $\fM$. Then
\[ I_{\fM,\beta,\nabla_1}\subset (I_{\fM,\nabla_\infty},p^{m-2h+3})\]
\end{prop}
{\begin{rmk} This Proposition controls the discrepancy between the ``true'' monodromy condition in Proposition \ref{prop:Mcond} and its truncation \eqref{eqn:algmono}. It is a generalization of \cite[Theorem 5.6]{LLLM}, which asserts that the tail/error term of the true monodromy condition is highly divisible by $p$. 
\end{rmk}}
\begin{proof} We continue to use the notations introduced above.
It follows from that definitions that we have the recursion %
\[Z^{(j')}_i=\frac{1}{p^{h-1}}A^{(j')} \Ad(s^{-1}_{j'}v^{\mu_{j'}+\eta_{j'}})\Big(\phz(Z^{(j'-1)}_{i-1})\Big)(v+p)^h \big(A^{(j')}\big)^{-1}\]
for $i\geq 1$, and
\begin{align*}
Z^{(j')}_0 &= \Ad \Big((s'_{\mathrm{or}, j'})^{-1} (u')^{-\mathbf{a}^{\prime \, (j')}} \Big) \left(u'\frac{d}{du'}C^{(j')}\right)(v+p)^{h}(A^{(j')})^{-1} \\
&=%
\left(\left[\mathrm{Diag}\Big((s'_{\mathrm{or}, j'})^{-1}(\mathbf{a}^{\prime \, (j')})\Big), A^{(j')}\right]+e'v\frac{d}{dv}\big(A^{(j')}\big)\right)(v+p)^{h}(A^{(j')})^{-1}\in \frac{1}{(v+p)}L\cM^+(R).
\end{align*}

An easy induction using the fact that $m+1\leq \langle \mu+\eta, \alpha^\vee \rangle \leq p-m-1$ shows that for $i\geq 1$
\begin{equation} \label{eq:divisibility_1}
Z^{(j')}_i\in \frac{1}{p^{i(h-1)}}v^{1+m\frac{p^{i}-1}{p-1}}\Mat_n(R[\![v+p]\!])
\end{equation}
Now over $\Spec R/I_{\fM,\nabla_\infty}[\frac{1}{p}]$, for all $0\leq t\leq h-2$, we have by definition
\[\left(\frac{d}{du'}\right)^{\!\!\!t}\Big|_{E(u')=0} (\lambda^{h-1} N_{\infty}^{(j')})= 0\]
hence also 
\[\left(u'\frac{d}{du'}\right)^{\!\!\!t}\Big|_{E(u')=0} (\lambda^{h-1} N_{\infty}^{(j')})= 0\] and since $u'$ is invertible in $R[u']/(I_{\fM,\nabla_\infty},E(u'))[\frac{1}{p}]$,
\[\left(u'\frac{d}{du'}\right)^{\!\!\!t}\Big|_{E(u')=0} \left(p^{h} \Ad \Big((s'_{\mathrm{or}, j'})^{-1} (u')^{-\mathbf{a}^{\prime \, (j')}} \Big) (\lambda^{h-1} N_{\infty}^{(j')})\right)=0. \]
Equation (\ref{eq:remove_descent}) thus shows 
\[\left(u'\frac{d}{du'}\right)^{\!\!\!t}\Big|_{E(u')=0} \left(-\phz(\lambda)^{h}P_N(A^{(j')})+\sum_{i= 1}^{\infty}\phz^{i+1}(\lambda)^{h}Z_i^{(j')}\right)=0\]
Since the expression inside the derivative belongs to $R[\frac{1}{p}][\![v]\!]$ and $u'\frac{d}{du'}=e'v\frac{d}{dv}$ we get
\begin{equation}\label{eq:divisible_2}
\left(v\frac{d}{dv}\right)^{\!\!\!t}\Big|_{v=-p} \left(-\phz(\lambda)^{h}P_N(A^{(j')})+\sum_{i= 1}^{\infty}\phz^{i+1}(\lambda)^{h}Z_i^{(j')}\right)=0
\end{equation}
Now observe that 
\begin{itemize} 
\item $\big(v\frac{d}{dv}\big)^{\!t}|_{v=-p} \phz^k(\lambda)\in p^{p-1}\bZ_p$ for any $t,k \geq 1$.
\item If $F\in v^M\Mat_n(R[\![v+p]\!])$ then $\big(v\frac{d}{dv}\big)^{\!t}|_{v=-p}\in p^MR$ for any $t\geq 0$.
\end{itemize}
Hence (\ref{eq:divisibility_1}) and (\ref{eq:divisible_2}) imply the equation
\begin{equation} \label{eq:divisibility_3}
\left(v\frac{d}{dv}\right)^{\!\!\!t}\Big|_{v=-p}P_N(A^{(j')})+ O(p^{m+1-(h-1)})=0
\end{equation}
in $R/I_{\fM,\nabla_\infty}[\frac{1}{p}]$, where the symbol $O(p^{M})$ stands for an element in $p^{M}R$.
Since the differential operator $\big(v\frac{d}{dv}\big)^{\!t}-v^t\big(\frac{d}{dv}\big)^{\!t}$ is a $\bZ$-linear combination of differential operators $v^a\big(\frac{d}{dv}\big)^{\!b}$ with $a, b <t$, (\ref{eq:divisibility_3}) implies by induction
\[\left(\frac{d}{dv}\right)^{\!\!\!t}\Big|_{v=-p}P_N(A^{(j')})+ O(p^{m+1-(h-1)-t})=0\]
in $R/I_{\fM,\nabla_\infty}[\frac{1}{p}]$ for all $0\leq t \leq h-2$. Now using the equation $\big(\frac{d}{dv}\big)^{\!t}v=t\big(\frac{d}{dv}\big)^{\!t-1}+\big(\frac{d}{dv}\big)^{\!t}$, we conclude that
\[\left(\frac{d}{dv}\right)^{\!\!\!t}|_{v=-p} v^{-\delta_{k>l}} P_N(A^{(j')})_{kl}+ O(p^{m-(h-1)-t})=0\]
$R/I_{\fM,\nabla_\infty}[\frac{1}{p}]$ for $0\leq t\leq h-2$, $1\leq k,l \leq n$.  Since the left-hand side of the above equation belong to $R$ and $R/I_{\fM,\nabla_\infty}$ is a subring of $R/I_{\fM,\nabla_\infty}[\frac{1}{p}]$, the above equation implies
\[\left(\frac{d}{dv}\right)^{\!\!\!t}|_{v=-p}v^{-\delta_{k>l}} P_N(A^{(j')})_{kl}\in (I_{\fM,\nabla_\infty},p^{m-2h+3})\]

\end{proof}

\subsection{Tame potentially crystalline stacks}
\label{subsec:tame:EGstacks}

In \cite{EGstack}, Emerton and Gee considered the formal stack $\cX_n$ over $\Spf\cO$ 
parametrizing (projective) \'{e}tale $(\phz,\Gamma)$-modules (see \cite[Definition 3.2.1]{EGstack} for the definition) and showed that $\cX_n$ is a Noetherian formal algebraic stack. For any complete local Noetherian $\cO$-algebra $R$ with finite residue field, the groupoid $\cX_n(R)$ is equivalent to the groupoid of $R$-families of $G_K$-representations, i.e. rank $n$ projective $R$-modules equipped with a continuous $G_K$-action.
We will write $\cX^K_n$ for $\cX_n$ if we want to emphasize the dependence on the field $K$.
Similarly, if $\cO_p$ is a finite \'etale $\Z_p$-algebra and $F^+_p \defeq \cO_p \otimes_{\Z_p} \Q_p$ which can be written in the form $\prod_{v\in S_p} F^+_v$, then we write $\cX_n^{F^+_p}$ for the product
\[
\prod_{v\in S_p,\Spf \cO} \cX_n^{F^+_v}.
\]
In this section, we will consider the case $\cO_p = \cO_K$, but the evident generalizations follow by taking products.

Now let $\tau$ be a tame inertial type (for $K$) and $\lambda\in X_*(T^\vee)^\cJ$ dominant. Then \cite[Theorem 4.8.12]{EGstack} shows there is a unique closed formal substack $\cX^{\lambda,\tau}$ of $\cX_n$, which is characterized by the following properties:
\begin{itemize}
\item $\cX^{\lambda,\tau}$ is $\cO$-flat.
\item For any finite flat $\cO$-algebra $\Lambda$, the groupoid $\cX^{\lambda,\tau}(\Lambda)$ is the subgroupoid of $\cX_n(\Lambda)$ consisting of $G_K$-representations on  rank $n$ projective $\Lambda$-modules which (after inverting $p$) are potentially crystalline with Hodge-Tate weight $\lambda$ and inertial type $\tau$.
\end{itemize}
Furthermore, $\cX^{\lambda,\tau}$ is a $p$-adic formal algebraic stack topologically of finite type over $\Spf \cO$.
For any $h\geq 0$, we also have the closed substack $\cX^{[0,h],\tau}\into \cX_n$, characterized by the same properties except in the second item, where we demand the Hodge-Tate weights to belong to $[0,h]$. Then $\cX^{[0,h],\tau}$ is the scheme theoretic union of $\cX^{\lambda,\tau}$ for $\lambda=(\lambda_j)_{j\in\cJ}$ satisfying $\lambda_j\in [0,h]^n$.
Finally, we set $\cX^{\leq \lambda,\tau}\subset \cX^{[0,h],\tau}$ to be the scheme theoretic union of $\cX^{\lambda',\tau}$ for $\lambda'$ dominant and $\lambda'\leq \lambda$.%

Recall from \cite[Definition 8.22]{Emerton_formal} that a $p$-adic formal algebraic stack $\cZ$ topologically of finite type over $\Spf \cO$ (which implies residual Jacobson) is \emph{analytically unramified} if for any smooth chart $\Spf A \to  \cZ$, $A$ is reduced. This is is also equivalent to $\cZ$ having reduced versal rings at all finite type points. Given $\cZ$, \cite[Example 9.10]{Emerton_formal} shows that it admits an associated reduced formal algebraic substack $\cZ'\into \cZ$. It is characterized as the maximal analytically unramified closed substack of $\cZ$. For any smooth chart $\Spf A \to \cZ$, the pullback of $\cZ'$ is $\Spf A^{\red}$, where $A^{\red}$ is the maximal reduced quotient of $A$.
\begin{warning}\begin{enumerate}
\item In \cite{EGstack}, the convention for Hodge-Tate weights is such that the cyclotomic character has weight $-1$. This is \emph{opposite} of our convention, where the cyclotomic character has weight $1$. As a result, a point in $\cX^{\lambda,\tau}(\overline{\bQ}_p)$ gives rise to a $p$-adic Galois representation $\rho$ such that the \emph{covariant} admissible module $\Dpst(\rho)$ is isomorphic to $\tau[\frac{1}{p}]$ as an $I_K$-representation (and $N=0$), and the Hodge filtration has jumps described by $-w_0\lambda$. In other words, our $\cX^{\lambda,\tau}$ would be $\cX^{-w_0(\lambda), \tau}$ in the notation of \cite{EGstack}. 
\item We warn the reader that the notion of associated reduced formal algebraic substack is different from the notion of underlying reduced algebraic stack: For $\cZ=\Spf A$, the former notion gives the formal scheme $\Spf A^{\red}$, while the latter gives the scheme $\Spec (A/I)^{\red}$, for $I$ an ideal of definition for the topology on $A$. In particular, the former notion is usually larger than the latter.
\end{enumerate}
\end{warning}

We now record some basic properties of these stacks established in \cite{EGstack}.
\begin{thm} \label{thm:EG_basic_properties} Let $?\in \{[0,h], \leql, \lambda\}$.
 \begin{enumerate}
\item 
\label{it:EG_basic_properties:1}
The stack $\cX^{?,\tau}$ is a $p$-adic formal algebraic stack, flat and topologically of finite type over $\Spf \cO$. Furthermore, $\cX^{?,\tau}$ is analytically unramified.
\item 
\label{it:EG_basic_properties:2}
For any smooth map $\Spf R \to \cX^{?,\tau}$ from a topologically finite type affine $p$-adic formal algebraic space, the ring $R[\frac{1}{p}]$ is regular.
\item 
\label{it:EG_basic_properties:3}
Let $\rhobar \in \cX^{?,\tau}(\bF)$ corresponding to a mod $p$ representation of $G_K$. Then the potentially crystalline deformation ring $R_\rhobar^{?,\tau}$ is a versal ring to $\cX^{?,\tau}$ at $\rhobar$.
\item 
\label{it:EG_basic_properties:4}
The stack $\cX^{\lambda,\tau}$ is  equidimensional of dimension 
\[1+\sum_{j\in \cJ} \dim_\bZ P_{\lambda_j} \backslash \GL_n. \]
\end{enumerate}
\end{thm}
\begin{proof}
{The first half of part (\ref{it:EG_basic_properties:1}) follows from \cite[Theorem 4.8.12]{EG}. Part (\ref{it:EG_basic_properties:3}) follows from \cite[Proposition 4.8.10]{EG}. Part (\ref{it:EG_basic_properties:4}) follows from \cite[Theorem 4.8.14]{EG}. Finally, part (\ref{it:EG_basic_properties:2}) and the second half of part (\ref{it:EG_basic_properties:1}) follows from part (\ref{it:EG_basic_properties:3}) and \cite[Theorem 3.3.8]{KisinPSS}.}
\end{proof}

By \cite[Proposition 3.7.2]{EGstack}, there is a canonical map $\cX_{n}\to \Phi\text{-}\Mod_K^{\text{\'et},n}$, which when evaluated on complete local Noetherian $\cO$-algebras $A$ corresponds to restricting $G_K$-representations to $G_{K_\infty}$-representations.

Let
\[\cK^{[0,h],\tau}=\cX^{[0,h],\tau}\times_{\Phi\text{-}\Mod_K^{\text{\'et},n}} Y^{[0,h],\tau}\]
be the pullback of $\eps_\tau:Y^{[0,h],\tau}\to \Phi\text{-}\Mod_K^{\text{\'et},n}$ along $\cX^{\lambda,\tau}\to \Phi\text{-}\Mod_K^{\text{\'et},n}.$
Similarly, for $\lambda\in X_*(T^\vee)^\cJ$ such that $\lambda_j\in[0,h]^n$, we have the pullbacks $\cK^{\lambda,\tau}$, $\cK^{\leq\lambda,\tau}$. 

Finally, we let $Y^{[0,h],\tau,\nabla_\infty}\into Y^{[0,h],\tau}$ be the unique $\cO$-flat closed substack characterized by the following property: For any $p$-adically complete topologically finite type flat $\cO$-algebra $R$ and a map $f:\Spf R \to Y^{[0,h],\tau}$ corresponding to $\fM\in Y^{[0,h],\tau}(R)$, $f$ factors through $Y^{[0,h],\tau,\nabla_\infty}$ if and only if the ideal $I_{\fM,\nabla_\infty}\subset R$ constructed in Proposition \ref{prop:functorial_mon_condition} is $0$. The existence of such a substack follows from the general construction of \cite[\S 9]{Emerton_formal}, using the compatibility {of} $I_{\fM,\nabla_\infty}$ with smooth base change established in Proposition \ref{prop:functorial_mon_condition}.
Similarly, we define the $\cO$-flat closed substack $Y^{\leql,\tau,\nabla_\infty}\into Y^{\leql,\tau}$ by imposing the same kind of condition.

The following result is the main result of this section, which summarizes the relationship between the tame potentially crystalline stacks and the moduli stack of Breuil--Kisin modules in generic situations:
\begin{prop} \label{prop:stack_diagram} Let $h\geq 1$, and $\tau$ be an $(h+2)$-generic tame inertial type. Let $\lambda=(\lambda_j)_{j\in\cJ}\in X_*(T^\vee)^\cJ$ be dominant such that $\lambda_j\in [0,h]^n$, and let $\lambda'$ be dominant such that $\lambda'\leq \lambda$.

We then have the following diagram
\begin{equation} \label{eq:stack_diagram}
\xymatrix{\cK^{\lambda',\tau} \ar@{^{(}->}[r] \ar^{\cong}[dd] & \cK^{\leq \lambda,\tau} \ar[rr]^{\cong} \ar^{\cong}[dd] \ar@{^{(}->}[dr] && Y^{\leql,\tau,\nabla_\infty}  \ar@{^{(}->}[r] \ar@{^{(}->}[d] &Y^{\leql,\tau} \ar@{^{(}->}[d] \\
&&\cK^{[0,h],\tau} \ar[r]^{\cong}  \ar^{\cong}[d] &Y^{[0,h],\tau,\nabla_\infty}  \ar@{^{(}->}[r]& Y^{[0,h],\tau} \ar@{^{(}->}^{\eps_\tau}[d] \\
\cX^{\lambda',\tau} \ar@{^{(}->}[r] &\cX^{\leq \lambda,\tau} \ar@{^{(}->}[r] &\cX^{[0,h],\tau} \ar@{^{(}->}[rr]  \ar@{^{(}->}[dr]& & \Phi\text{-}\Mod_K^{\text{\emph{\'et}},n}\\
&&&\cX_n\ar[ur] & }
\end{equation}
such that:
\begin{itemize}
\item All rectangles and trapezoids {except possibly for the top right rectangle} are Cartesian.
\item The arrows decorated with the symbol $\cong$ are isomorphisms.
\item All the hooked arrows are monomorphisms, and except for the {rightmost} bottom horizontal arrow, are even closed immersions.
\end{itemize}
{In particular, $\cX^{[0,h],\tau}\cong Y^{[0,h],\tau,\nabla_\infty}$, and if $\lambda=(\lambda_j)_{j\in\cJ}\in X_*(T^\vee)^\cJ$ is dominant such that $\lambda_j\in [0,h]^n$ for all $j\in\cJ$ then $\cX^{\leql,\tau}\cong Y^{\leql,\tau,\nabla_\infty}$.}
\end{prop}

\begin{rmk} {It's not clear to us if the top right rectangle in diagram \eqref{eq:stack_diagram} is Cartesian: it is Cartesian after inverting $p$, but taking Zariski closure does not commute with base change in general.}
\end{rmk}
The proof of Proposition \ref{prop:stack_diagram} will occupy the rest of this section.
To prepare for the proof, we record some Lemmas which give criteria for maps of schemes or stacks to be isomorphisms using information on special kinds of points.

\begin{lemma} \label{lem:Artinian_point_iso} Let $a\geq 1$ and let $f: Y\to Z$ be a map between finite type $\cO/\varpi^a$-schemes. Assume that for any local Artinian ring $A$ with finite residue field, $f$ induces a bijection $Y(A)\cong Z(A)$. Then $f$ is an isomorphism. 
\end{lemma}
\begin{proof} We note that for any finite type $\cO/\varpi^a$-scheme, the set of closed points is dense, and the residue field at the closed points are finite fields.

By \cite[\href{https://stacks.math.columbia.edu/tag/02HY}{Tag 02HY}]{stacks-project}, $f$ is a smooth map. Since $f$ is also quasi-finite, $f$ is \'{e}tale. Thus the diagonal $\Delta_f: Y\to Y\times _Z Y$ is an open immersion. Since $\Delta_f$ is surjective on closed points, it is an isomorphism, hence $f$ is a monomorphism. Thus $f$ is an \'{e}tale monomorphism, hence is an open immersion by \cite[\href{https://stacks.math.columbia.edu/tag/025F}{Tag 025F}]{stacks-project}. Finally $f$ is also surjective on closed points, hence $f$ is an isomorphism.
\end{proof}

\begin{lemma}\label{lem:iso_criteria} Let $f: \cY \to \cZ$ be a monomorphism of $p$-adic formal algebraic stacks topologically of finite type over $\Spf \cO$. Assume that $\cZ$ is flat over $\Spf \cO$.
 Assume that either of the following holds:
\begin{enumerate}
\item $\cZ$ is analytically unramified, and for any finite flat $\cO$-algebra $\Lambda$, $f:\cY(\Lambda)\to \cZ(\Lambda)$ is essentially surjective.
\item $\cZ$ is analytically unramified, $f$ is a closed immersion, and for any finite extension $E'$ of $E$ with ring of integers $\cO'$, $f:\cY(\cO')\to \cZ(\cO')$ is an essentially surjective.
\item $f$ is a closed immersion, and for any finite flat $\cO$-algebra $\Lambda$, $f:\cY(\Lambda)\to \cZ(\Lambda)$ is essentially surjective.
\end{enumerate}
Then $f$ is an isomorphism.
\end{lemma}
\begin{proof} As the problem is local (in the smooth topology) in $\cZ$, we reduce to the case $\cZ=\Spf B$ where $B$ is a $p$-adically complete topologically of finite type $\cO$-flat algebra. Then $\cY$ is a formal algebraic space (in fact, a formal scheme by \cite[\href{https://stacks.math.columbia.edu/tag/0B89}{Tag 0B89}]{stacks-project}).

Suppose that we are in the first case. We claim that for any local Artinian $\cO$-algebra $A$ with finite residue field, $f:\cY(A)\to \cZ(A)$ is an equivalence. Since we already have fully faithfulness (from $f$ being a monomorphism), we only need to show essential surjectivity. Suppose we have an element $x\in \cZ(A)$, which corresponds to a map $B\to A$, which factors through $B/\fm^k\to A$ for some maximal ideal $\fm$ of $A$ and $k\geq 1$. Now our hypotheses on $\cZ$ imply that $B$ is reduced and $\bZ_p$-flat. Furthermore, since $B$ is $p$-adically complete and $B/p$ is Nagata, $B$ is also Nagata \cite{Marot}
Hence \cite[Lemma 4.1.2]{BartlettIrrComp}
implies that $B\to B/\fm^k$ factors through some continuous map $B\to \Lambda$ where $\Lambda$ is a finite flat $\cO$-algebra. Thus $x$ can be lifted to a point $\tld{x}\in \cZ(\Lambda)\cong \cY(\Lambda)$, hence $x$ is in the essential image of $\cY(A)$. But now for each $a\geq 1$, Lemma \ref{lem:Artinian_point_iso} implies that the base change $(\cY)_{\cO/\varpi^a}\to (\cZ)_{\cO/\varpi^a}$ is an isomorphism, hence $f$ itself is an isomorphism.

Suppose now that we are in the second case. Then $\cY=\Spf B/J$. Since the residue fields at maximal ideals of $B[\frac{1}{p}]$ are finite extensions of $E$, and any map $B[\frac{1}{p}]\to E'$ where $E'$ is a finite extension of $E$ comes from a map $B\to \cO'$, our hypothesis implies that $J[\frac{1}{p}]$ is in the intersection of all the maximal ideals of $B[\frac{1}{p}]$. Since $B[\frac{1}{p}]$ is Jacobson, $J[\frac{1}{p}]=0$, and hence $J=0$ since $B$ is $\cO$-flat.

Finally, suppose that we are in the third case. Then $\cY=\Spf B/J$. For any maximal ideal $\fm$ of $B[\frac{1}{p}]$ and any $a\geq 1$, $B[\frac{1}{p}]/\fm^a$ is finite dimensional over $E$, and the map $B\to B[\frac{1}{p}]/\fm^a$ factors through some finite flat $\cO$-algebra $\Lambda$ such that $\Lambda[\frac{1}{p}]=B[\frac{1}{p}]/\fm^a$. Our hypothesis implies that the map $B\to \Lambda$ factors through $B/J$. It follows that $J[\frac{1}{p}]\subset \cap_{a=1}^\infty \fm^a$, hence $J[\frac{1}{p}]_{\fm}=0$ Since this is true for any maximal ideal $\fm$, we have $J[\frac{1}{p}]=0$, and hence $J=0$ since $B$ is $\cO$-flat.
\end{proof}

We can now deal with the vertical isomorphisms occurring in diagram (\ref{eq:stack_diagram}):
\begin{prop} \label{prop:unique_BK}Assume that $\tau$ is $(h+1)$-generic. Then the natural map $\cK^{[0,h],\tau}\to \cX^{[0,h],\tau}$ is an isomorphism. %
\end{prop}
\begin{proof}It follows from Proposition \ref{prop:BK_to_phi_mono} that our map is a closed immersion. By Lemma \ref{lem:iso_criteria}, we only need to check that for any finite extension $E'$ of $E$ with ring of integers $\cO'$, the natural functor $\cK^{[0,h],\tau}(\cO')\to \cX^{[0,h],\tau}(\cO')$ is essentially surjective. Let $V\in \cX^{[0,h],\tau}(\cO')$ be an $\cO'$-lattice in a potentially crystalline representation over $E'$ with Hodge-Tate weights in $[0,h]$, and let $\cM\in \Phi\text{-}\Mod_{dd,L'}^{\text{\'et},n}(\cO')$ be the associated \'{e}tale $\phz$-module with descent data from $L'$ to $K$. By \cite[Corollary (1.3.15), Proposition (2.1.5) and Lemma (2.1.15)]{KisinFcrys},
there is a unique projective $\fS_{L',\cO'}$-submodule $\fM\subset \cM$ which is $\phi_\cM$-stable (projectivity follows from \cite[Remark 2.2.16(2)]{BartlettIrrComp}), such that $\cM=\fM\otimes_{\fS_{L'}} \cO_{\cE,L'}$ and the cokernel of $\phi_\cM$ on $\fM$ is killed by $E(u')^h$. The uniqueness implies that $\fM'$ is stable under the semi-linear action of $\Delta'$. As $\fM/u'\fM[\frac{1}{p}]\cong D_{pst}(V^\vee)\cong \tau^\vee \otimes_{\cO} (K'\otimes_{\bZ_p} \cO')$ 
 as projective $K'\otimes_{\bZ_p} \cO'$-modules with $\Delta=I(L'/K)$-action, we deduce that $\fM\in Y^{[0,h],\tau}(\cO')$. Thus $V\in \cX^{[0,h],\tau}(\cO')$ is isomorphic to the image of $(V,\fM)\in \cK^{[0,h],\tau}(\cO')$.
\end{proof}
\begin{rmk}  For a general finite flat $\cO$-algebra $\Lambda$ which is not the ring of integers of a finite extension of $E$ and $V\in \cX^{[0,h]}(\Lambda)$, the unique Breuil--Kisin module $\fM$ associated to $V$ viewed as an $\cO$-lattice in potentially crystalline representation over $E$ is a priori only an $\fS_{L',\Lambda}$-module.  However, it follows from the above Proposition that in this case, it is actually $\fS_{L',\Lambda}$-projective. 
\end{rmk}

We now analyze the bottom horizontal map of diagram (\ref{eq:stack_diagram}). We recall the following definition \cite[Definition 3.8]{LLLM}.  Recall that $\eps$ denotes the $p$-adic cyclotomic character. 
\begin{defn} Let $\rhobar:G_K\to \GL_m(\overline{\bF}_p)$. We say $\rhobar$ is cyclotomic free if there is an unramified extension $M/K$ of degree prime to $p$ such that $\rhobar|_{G_M}^{\mathrm{ss}}$ is a direct sum of characters, and 
\[H^0(G_M,\rhobar|_{G_M}^{\mathrm{ss}}\otimes \varepsilon^{-1})=0.\]

\end{defn}
The main feature about this notion that is relevant to us is the following
\begin{lemma}\label{lem:cyclo_free}\begin{enumerate}
\item Suppose $\rhobar$ is cyclotomic free. Then the natural inclusion induces an isomorphism
\[H^0(G_K,\rhobar)\cong H^0(G_{K_\infty},\rhobar)\]
\item If $\rhobar^{\mathrm{ss}}|_{I_K}$ is $2$-generic, 
then $\ad(\rhobar)$ is cyclotomic free.
\item \label{cyclo_free_3} Suppose $V$, $W$ are two $\cO[G_K]$-modules of finite length. Assume there exists a semisimple $G_K$-representation $\rhobar$ such that $\ad(\rhobar)$ is cyclotomic free, and such that $V^{\mathrm{ss}}$, $W^{\mathrm{ss}}$ are direct summands of a direct sum of finitely many copies of $\rhobar$. Then the natural restriction map induces an isomorphism
\[\Hom_{G_K}(V,W)\cong \Hom_{G_{K_{\infty}}}(V,W)\]
\item \label{cyclo_free_4} Let $W$ as in (\ref{cyclo_free_3}). Then any $G_{K_{\infty}}$-submodule $V\subset W$ is $G_K$-stable.
\end{enumerate}
\end{lemma}
\begin{proof}
\begin{enumerate} 
\item This follows from (the proof of) \cite[Lemma 3.11]{LLLM}.
\item This is \cite[Proposition 3.9]{LLLM}. Note the proof in loc.~cit.~was written for $n=3$, but works in general.  Also, $2$-generic in the sense of this paper is stronger than $2$-generic in loc.~cit.~(see \cite[Remark 2.2.8]{LLL}). 
\item %
Since $\ad(\rhobar)=\rhobar\otimes \rhobar^{\vee}$ is cyclotomic free, the same is true for any finite direct sum of $\ad(\rhobar)$.
Now $V^{\mathrm{ss}}$, $W^{\mathrm{ss}}$ are direct summands of a finite direct sum of $\rhobar$, hence $W^{\mathrm{ss}}\otimes (V^{\mathrm{ss}})^\vee$ is a direct summand of a finite direct sum of $\ad(\rhobar)$, and thus is cyclotomic free. Since $(W\otimes V^\vee)^{\mathrm{ss}}=W^{\mathrm{ss}}\otimes (V^{\mathrm{ss}})^\vee$, $W\otimes V^\vee$ is also cyclotomic free. The result now follows from the first part. 
\item We first assume that $V$ is irreducible. Then $V$ extends uniquely to a $G_K$-module. By the previous part, the $G_{K_\infty}$-equivariant inclusion map $V\into W$ is $G_K$-equivariant, thus finishing the proof in this case.

For general $V$, we let $V_0$ be non-zero irreducible $G_{K_\infty}$-submodule of $W$. Then the argument above shows that $V_0$ is a $G_K$-submodule of $W$. We repeat the argument for $V/V_0\into W/V_0$ to conclude.
  
\end{enumerate}

\end{proof}
\begin{prop} \label{prop:restriction_mono}Suppose $\tau$ is $(h+2)$-generic. Then the composition $\cX^{[0,h],\tau}\to \cX_n\to \Phi\text{-}\Mod_K^{\text{\'et},n}$ is a monomorphism.
\end{prop}
\begin{proof} It suffices to show that for any $a\geq 1$ and $A$ a finite type $\cO/\varpi^a$-algebra, the functor $\cX^{[0,h],\tau}(A)\to \Phi\text{-}\Mod_K^{\text{\'et},n}(A)$ is fully faithful.

Suppose first that $A$ is local Artinian $\cO$-algebra with finite residue field $\bF'$. Then $\cX_n(A)$ is equivalent to the groupoid of $G_K$-representation on projective $A$-modules of rank $n$, $\Phi\text{-}\Mod_K^{\text{\'et},n}(A)$ is equivalent to the groupoid of $G_{K_\infty}$-representation on projective $A$-modules of rank $n$. Suppose we have two such $G_K$-representations $V_A$, $W_A$. We need to show the restriction map induces a bijection between the set of isomorphisms $\mathrm{Isom}_{G_K}(V_A,W_A)\cong\mathrm{Isom}_{G_{K_\infty}}(V_A,W_A)$. 

We first observe that if either set is non-empty, then $V_{\bF'}^{\mathrm{ss}}\cong W_{\bF'}^{\mathrm{ss}}$: indeed, the restriction map identifies the semisimple representations of $G_K$ and $G_{K_\infty}$ over $\bF'$.
We can thus assume that $V_{\bF'}^{\mathrm{ss}}\cong W_{\bF'}^{\mathrm{ss}}$, and denote this common representation by $\rhobar$. But now $V_{\bF'}^{\mathrm{ss}}|_{G_{K_\infty}}$ comes from an object of $Y^{[0,h],\tau}(\bF')$, so Lemma \ref{lem:ss_bound}
shows that $\rhobar$ is $2$-generic. Finally, since $V_A^{\mathrm{ss}}$ and $W_A^{\mathrm{ss}}$ are direct summands of a finite direct sum of $\rhobar$, we conclude by Lemma \ref{lem:cyclo_free}.

Suppose now that $A$ is a general finite type $\cO/\varpi^a$-algebra. Let $x_1, x_2 \in \cX^{[0,h],\tau}(A)$ and let $y_1$, $y_2$ be their images in $\Phi\text{-}\Mod_K^{\text{\'et},n}(A)$. Let $Y=\mathrm{Isom}(x_1,x_2)$ and $Z=\mathrm{Isom}(y_1,y_2)$ be the functor over $\Spec A$ which represents isomorphisms between $x_1$, $x_2$ and $y_1$, $y_2$. By \cite[Proposition 5.4.8]{EGschemetheoretic}, $Y$, $Z$ are representable by finite type $A$-schemes, and hence are finite type $\cO/\varpi^a$-schemes. 
The composition in the statement of the proposition 
induces a natural map $Y\to Z$. By the Artinian case above, for any local Artinian ring $B$ with finite residue field, the natural map $Y(B)\to Z(B)$ is a bijection. Lemma \ref{lem:Artinian_point_iso} then shows that $Y\to Z$ is an isomorphism of $A$-schemes, hence in particular $Y(A)=Z(A)$, and hence the subsets of $Y(A)$ and $Z(A)$ which maps to the identity via the structure maps $Y\to \Spec A$, $Z\to \Spec A$ also coincide. But these sets are exactly the Hom space between $x_1$, $x_2$ in $\cX^{[0,h],\tau}(A)$ and the Hom space between $y_1, y_2$ in $\Phi\text{-}\Mod_K^{\text{\'et},n}(A)$.
\end{proof}
Finally, we deal with the middle and top horizontal maps of diagram (\ref{eq:stack_diagram}).
\begin{prop} \label{prop:stack_vs_monodromy}
The natural map $\cK^{[0,h],\tau}\to Y^{[0,h],\tau}$ factors through the substack $Y^{[0,h],\tau,\nabla_\infty}\subset Y^{[0,h],\tau}$, and the natural map $K^{\leq \lambda,\tau}\into \cK^{[0,h],\tau}\to Y^{[0,h],\tau}$ factors through $Y^{\leql,\tau,\nabla_\infty}\subset Y^{[0,h],\tau}$. The induced maps $\cK^{[0,h],\tau}\ra Y^{[0,h],\tau,\nabla_\infty}$ and $\cK^{\leq \lambda,\tau}\ra Y^{\leql,\tau,\nabla_\infty}$ are isomorphisms. 
\end{prop}
\begin{proof} We first show that $\cK^{[0,h],\tau}\to Y^{[0,h],\tau}$ factors through $Y^{[0,h],\tau,\nabla_\infty}$. To do this, it suffices to show that for some smooth cover $\Spf R \to \cK^{[0,h],\tau}$, the induced map $\Spf R \to Y^{[0,h],\tau}$ factors through $Y^{[0,h],\tau,\nabla_\infty}$. Since $\cK^{[0,h],\tau}\cong \cX^{[0,h],\tau}$ is $\cO$-flat, analytically unramified and topologically of finite type over $\Spf \cO$, $R$ is also $\cO$-flat, reduced and topologically of finite type over $\cO$. 
The induced map $\Spf R \to Y^{[0,h],\tau}$ corresponds to an object $\fM\in Y^{[0,h],\tau}(R)$, and the existence of the desired factorization is equivalent to $I_{\fM,\nabla_\infty}=0$. Now for any finite extension $E'/E$ with ring of integers $\cO'$, and any map $x:R \to \cO'$, the base change $\fM_x$ of $\fM$ along $x$ is the Breuil--Kisin module associated to an $\cO'$-lattice in a potentially crystalline representation with inertial type $\tau$, and thus $\fM_x$ satisfies the monodromy condition (cf.~Definition \ref{defn:mcond}). Thus $I_{\fM_x,\nabla_\infty}=0$ in $\cO'$ by Proposition \ref{prop:Mcond}, so $I_{\fM,\nabla_\infty}\subset \ker x$.
This shows that $I_{\fM,\nabla_\infty}[\frac{1}{p}]$ lies in the intersection of all the maximal ideals of $R[\frac{1}{p}]$. Since $R[\frac{1}{p}]$ is reduced and Jacobson, this intersection is $0$, and hence $I_{\fM,\nabla_\infty}=0$ since $R$ is $\cO$-flat. We note that this argument actually shows that $\Spf R \to Y^{[0,h],\tau}$ factors through the associated reduced formal algebraic substack $\cZ$ of $Y^{[0,h],\tau,\nabla_\infty}$.

We have a sequence of monomorphisms $\cK^{[0,h],\tau} \into \cZ \into Y^{[0,h],\tau,\nabla_\infty}$, since $\cK^{[0,h],\tau}\to Y^{[0,h],\tau}$ is a monomorphism, being the base change of the monomorphism $\cX^{[0,h],\tau}\into \Phi\text{-}\Mod_K^{\text{\'et},n}$ (see Proposition \ref{prop:restriction_mono}). Note that the second monomorphism is a closed immersion. %
We now show that for any finite flat $\cO$-algebra $\Lambda$, the composition $\cK^{[0,h],\tau}(\Lambda)\to \cZ(\Lambda)\to Y^{[0,h],\tau}(\Lambda)$ is essentially surjective. %
Let $x\in Y^{[0,h],\tau,\nabla_\infty}(\Lambda)$. Then $x$ corresponds to an object $\fM_x\in Y^{[0,h],\tau}(\Lambda)$. Since $x\in Y^{[0,h],\tau,\nabla_\infty}(\Lambda)$ and $\fS_{L',\Lambda}$ is semilocal, $\fM_x$ satisfies the hypothesis of Proposition \ref{prop:Mcond}. Thus $V=T^*_{dd}(\fM_x)$ is a $G_{K_\infty}$-representation on a free $\Lambda$-module of rank $n$, and $V[\frac{1}{p}]$ extends to a potentially crystalline representation of $G_K$ over $\Lambda[\frac{1}{p}]$ with Hodge-Tate weights in $[0,h]$ and inertial type $\tau$. By Lemma \ref{lem:G_K_lattices} below, $V\subset V[\frac{1}{p}]$ is actually $G_K$-stable, and hence $x$ indeed comes an object $V\in \cX^{[0,h],\tau}(\Lambda)=\cK^{[0,h],\tau}(\Lambda)$.
The upshot of this argument is that on the one hand, $\cK^{[0,h],\tau}\cong \cZ$ by the first criterion of Lemma \ref{lem:iso_criteria}, and on the other hand $\cZ\cong Y^{[0,h],\tau,\nabla_\infty}$ by the third criterion of Lemma \ref{lem:iso_criteria}.

We have thus proved the result for $\cK^{[0,h],\tau}$. To show the result for $\cK^{\leq \lambda,\tau}$, we use the following observations
\begin{itemize}
\item If $V\in \cX^{[0,h],\tau}(\cO')$ with associated Breuil--Kisin module $\fM\in Y^{[0,h],\tau}(\cO')$, then $\DdR(V^\vee)$ is identified with $(\phz^*\fM/E(u')\phz^*\fM) [\frac{1}{p}]$, cf.~\cite[\S 4.7]{EGstack} (the appearance of the dual is because we use the contravariant functor $T^*_{dd}$ on Breuil--Kisin modules, in contrast to \cite{EGstack}). Thus the jumps in the Hodge filtration of $\DdR(V^\vee)$ occur at the components of relative position of $\phz^*\fM$ with respect to $\fM$ (which is given by the elementary divisors of the matrices of partial Frobenii).
This implies that $V$ has Hodge-Tate weights $\leq \lambda$ if and only if $\fM\in Y^{\leql,\tau}(\cO')$, cf.~the discussion above Theorem \ref{thm:Breuil-Kisin_local_model}. 
\item Let $\Lambda$ is a finite flat $\cO$-algebra and and $V \in \cX^{[0,h],\tau}(\Lambda)$ with associated Breuil--Kisin module $\fM\in Y^{[0,h],\tau}(\Lambda)$. Then $\fM\in Y^{\leq \lambda,\tau}(\Lambda)$ implies $V\in \cX^{\leq \lambda,\tau}(\Lambda)$. This is due to the fact that the Hodge filtration on $\DdR(V^{\vee})$ are given by projective $\Lambda[\frac{1}{p}]$ modules, hence one can check the Hodge-Tate weight $\leq \lambda$ condition by passing to $\Lambda^{\red}$, which is dealt with by the item above.
\end{itemize}
The first item shows the existence of the factorization $\cK^{\leq \lambda,\tau}\into Y^{\leq \lambda,\tau,\nabla_\infty}\into Y^{\leq \lambda,\tau}\into Y^{[0,h],\tau}$, and the second item allows us to carry out the above  argument to conclude $\cK^{\leq \lambda,\tau}\cong Y^{\leq \lambda,\tau,\nabla_\infty}$.
\end{proof}

\begin{lemma}\label{lem:G_K_lattices} Let $\tau$ be $(h+2)$-generic. Let $\Lambda$ be a finite flat $\cO$-algebra. Suppose $\fM\in Y^{[0,h],\tau}(\Lambda)$ and $V=T^*_{dd}(\fM)$, a $G_{K_\infty}$-representation on a free $\Lambda$-module of rank $n$. Suppose $V[\frac{1}{p}]$ extends to a $G_K$-representation. Then $V\subset V[\frac{1}{p}]$ is $G_K$-stable.
\end{lemma}
\begin{proof} It suffices to treat the case $\Lambda$ local. Let $\bF'$ be the residue field of $\Lambda$. Note that the semisimplified reduction $V_{\bF'}^{\mathrm{ss}}$ extends uniquely to a semisimple $G_K$-representation $\rhobar$. By Lemma \ref{lem:ss_bound}, $\rhobar$ is $2$-generic.
We choose a $\Lambda[G_K]$-stable $\cO$-lattice $W$ in $V[\frac{1}{p}]$ such that $V\subset W$. Then $(W/\varpi)^{\mathrm{ss}}$ is isomorphic to a finite direct sum of $\rhobar$ as $G_K$-representations. Choose $N$ large enough so that $p^NW\subset V \subset W$. Applying Lemma \ref{lem:cyclo_free}(\ref{cyclo_free_4}) to $V/p^NW\subset W/p^NW$, we conclude that $V/p^NW$, and hence $V$ is $G_K$-stable.
\end{proof}

\begin{proof}[Proof of Proposition \ref{prop:stack_diagram}] The Cartesian-ness of the rectangles and trapezoids follows from the definitions. Proposition \ref{prop:unique_BK} shows that the vertical maps labelled with $\cong$ are isomorphisms. Proposition \ref{prop:restriction_mono} and the definitions shows that all the hooked arrows are monomorphisms. Finally, Proposition \ref{prop:stack_vs_monodromy} show that the maps $K^{[0,h],\tau}\into Y^{[0,h],\tau}$ and $K^{\leq \lambda,\tau}\into Y^{\leq \lambda,\tau}$ are closed immersions with images as claimed.
\end{proof}
\subsection{Local models for potentially crystalline stacks} \label{sec:local_model_EG}
Throughout this section, we fix $\lambda\in X_*(T^\vee)^\cJ$ regular dominant such that $\lambda_j \in [0,h]^n$, and a tame inertial type $\tau$ with a lowest alcove presentation $(s,\mu)$. 
By Proposition \ref{prop:stack_diagram}, if $\mu$ is $(h+2)$-deep in $\un{C}_0$,  we have $\cX^{\leql,\tau}\cong Y^{\leql,\tau,\nabla_\infty}$, which is obtained from the stack of Breuil--Kisin modules $Y^{\leql,\tau}$ by imposing an explicit list of equations. On the other hand, Theorem \ref{thm:Breuil-Kisin_local_model} relates the local structure of $Y^{\leql, \tau}$ to the $p$-adic completion of the local model $M_\cJ(\leql)$. Thus we wish to analyze the effect of imposing the $\nabla_\infty$ equations on the local model diagram of Theorem \ref{thm:Breuil-Kisin_local_model}.

To our lowest alcove presentation $(s, \mu)$ of $\tau$, we get the data $\bf{a}^{\prime\, (j')} \in \Z^n$ for any $j' \in \cJ'$, cf.~equation (\ref{eq:aprime}). 
For each integer $j' \in\cJ'$, define %
\[
\bf{a}_{\tau, j'} \defeq (s'_{\mathrm{or},  j'})^{-1} (\bf{a}^{\prime\, (j')})/ (1 - p^{f'}) 
\]   
so $\bf{a}_{\tau} \in (\cO^n)^{\cJ'}$.
We caution that $\bf{a}_\tau$ depends on the choice of presentation $(s,\mu)$, and not just on $\tau$.
A direct computation gives:
\begin{lemma} 
Let $\bf{a}_\tau\in (\cO^n)^{\cJ'}$ be as above.
Then 
\[
\bf{a}_{\tau,j} \equiv s_{j}^{-1}(\mu_j + \eta_j) \mod \varpi.
\]
for any $j\in\{0,\dots,f-1\}$. 
\end{lemma} 
\begin{proof}
The proof is obtained by unraveling the definitions.
First of all, we notice that $\bf{a}_{\tau,j} \equiv (s'_{\mathrm{or},  j})^{-1} (\bf{a}^{\prime\, (j)})$ modulo $\varpi$, hence:
\begin{align}
\nonumber
\bf{a}_{\tau,j} &\equiv  (s'_{\mathrm{or},  j})^{-1}(\bm{\alpha}'_{- j}) \text{ modulo $\varpi$}\\
\label{eq:atau:modp}
&= (s'_{\mathrm{or},  j})^{-1}(s_\tau)^{\delta_{j>0}}(\bm{\alpha}_{f-j})
\end{align}
for $0\leq j\leq f-1$, using (\ref{eq:aprime}) and (\ref{eq:alphaprime}) above (where we set  $\bm{\alpha}_{f}\defeq \bm{\alpha}_{0}$).
Recall from (\ref{primeorient}) that $s'_{\mathrm{or},  0} =s_0$, that $s'_{\mathrm{or},  j} = s_{\tau}(s_{f-1}^{-1} \ldots s_{j+1}^{-1})$ for $0< j < f-1$ and that $s'_{\mathrm{or},  f-1} = s_{\tau}$.   
Thus the expression (\ref{eq:atau:modp}) equals $s_0^{-1}(\bm{\alpha}_{0})$ for $j=0$, $(s_{j+1} \ldots s_{f-1})(\bm{\alpha}_{f-j})$ for $0 < j < f-1$ and $\bm{\alpha}_{1}$ if $j = f-1$. 
As $\bm{\alpha}_{f-j} = s^{-1}_{f-1} \ldots s^{-1}_j (\mu_j + \eta_j)$ for $0<j\leq f-1$ and $\bm{\alpha}_0=\mu_0+\eta_0$ (see Example \ref{ex:data:type}), the conclusion follows.
\end{proof}
Recall from \S \ref{sec:naivemlm}, \S \ref{subsec:products} the projective $\cO$-scheme $M_{\cJ}^{\nv}(\leql, \nabla_{\bf{a}_{\tau}})=\prod_{j\in \cJ} M^{\nv}(\leql_j, \nabla_{\bf{a}_{\tau,j}})$. Here, for each $j\in \cJ$, we defined $M^{\nv}(\leql_j, \nabla_{\bf{a}_{\tau,j}})$ as the intersection $M(\lambda)\cap \Gr_{\cG,\cO}^{\nabla_{\bf{a}_{\tau,j}}}$ inside $\Gr_{\cG,\cO}$. In other words if $R$ is a Noetherian $\cO$-algebra and $x\in M(\lambda_j)(R)$ is represented by $A\in L\cG(R)$, then $x\in M^{\nv}(\leql_j,\nabla_{\bf{a}_{\tau,j}})$ if and only if  
\[  v \frac{d}{dv}(A) A^{-1}  + A \Diag(\bf{a}_{\tau,j}) A^{-1} \in \frac{1}{v+p}L^+\cM(R).\]
We have $\dim M_{\cJ}^{\nv}(\leql, \nabla_{\bf{a}_{\tau}})\leq 1+ \#\cJ\dim_E (B\backslash \GL_n)_E =1+\#\cJ\frac{n(n-1)}{2}$
 
Recall from Definition \ref{defn:alg:MLM} and \S \ref{subsec:products} the $\cO$-flat subscheme $M_{\cJ}(\lambda,\nabla_{\bf{a}_\tau})\subset M_{\cJ}^{\nv}(\leql, \nabla_{\bf{a}_{\tau}})$. By Proposition \ref{prop:compare-with-naive2}, the scheme theoretic union of $M_{\cJ}(\lambda',\nabla_{\bf{a}_\tau})$ over $\lambda'$ dominant and $\lambda'\leq \lambda$ is the $\cO$-flat part of $M_{\cJ}^{\nv}(\leql, \nabla_{\bf{a}_{\tau}})$. We define  $M_{\cJ,\reg}(\leql, \nabla_{\bf{a}_{\tau}}) $ to be the scheme theoretic union of $M_{\cJ}(\lambda',\nabla_{\bf{a}_\tau})$ over $\lambda'$ \emph{regular} dominant and $\lambda'\leq \lambda$. {Because $\dim M_{\cJ}(\lambda',\nabla_{\bf{a}_\tau})\leq 1+\#\cJ\dim_E (B\backslash \GL_n)_E$ with equality if and only if $\lambda'$ is regular,} $M_{\cJ,\reg}(\leql, \nabla_{\bf{a}_{\tau}}) $ is characterized as the maximal $\cO$-flat closed subscheme of $M_{\cJ}(\leql, \nabla_{\bf{a}_{\tau}}) $ which is equidimensional of dimension
\[ 1+\#\cJ\dim_E (B\backslash \GL_n)_E.\]

Let $\tld{z}=(\tld{z}_j)_{j\in \cJ}$. Recall from (\ref{eq:chartsforlambda}) the open affine subscheme $U(\tld{z},\leql)\into M_\cJ(\leql)$ and the trivial $T^{\vee,\cJ}_\cO$-torsor $\tld{U}(\tld{z},\leql)=T^{\vee,\cJ}_\cO\times U(\tld{z},\leql)$. Intersecting with these affine opens, we get the objects $U(\tld{z},\lambda,\nabla_{\bf{a}_\tau})$, $\tld{U}(\tld{z},\lambda,\nabla_{\bf{a}_\tau})$, $U_\reg(\tld{z},\leql,\nabla_{\bf{a}_\tau})$ and $\tld{U}_\reg(\tld{z},\leql,\nabla_{\bf{a}_\tau})$.

Our main result is the following
\begin{thm} \label{thm:stack_local_model}
{Let $(s,\mu)$ be a $(h+2)$-lowest alcove presentation for $\tau$ and} consider the commutative diagram of $p$-adic formal algebraic stacks
\begin{equation} %
\begin{tikzcd}[column sep=small]
\tld{U}_\reg(\tld{z},\leql, \nabla_{\bf{a}_\tau})^{\wedge_p}\ar[hook]{rr} & \phantom{M} & \tld{U}^{\nv}(\tld{z}, \leql,\nabla_{\bf{a}_\tau})^{\wedge_p} \ar{r}{T^{\vee,\cJ}_\cO}\ar[hook]{d}& U^{\nv}(\tld{z},\leql,\nabla_{\bf{a}_\tau})^{\wedge_p} \ar[hook,open]{r}\ar[hook]{d}& M_{\cJ}^{\nv}(\leql, \nabla_{\bf{a}_\tau})^{\wedge_p}\ar[hook]{d} \\
\tld{\cX}_{\reg}^{\leq\lambda,\tau}(\tld{z}) \ar[hook,dotted]{u}{\cong}\ar[hook]{r}\ar{d}{T^{\vee,\cJ}_\cO}&\tld{U}(\tld{z},\leql,\nabla_{\tau,\infty})\ar[hook]{r}\ar{d}{T^{\vee,\cJ}_\cO}\ar[hook,dotted]{ur}&\tld{U}(\tld{z},\leql)^{\wedge_p}\ar{r}{T^{\vee,\cJ}_\cO}\ar{d}{T^{\vee,\cJ}_\cO} & U(\tld{z},\leql)^{\wedge_p}\ar[hook,open]{r}& M_{\cJ}(\leql)^{\wedge_p}\\
\cX_{\reg}^{\leq \lambda,\tau}(\tld{z})\ar[hook]{r}\ar[hook,open]{d}&Y^{\leql,\tau,\nabla_\infty}(\tld{z}) \ar[hook]{r}\ar[hook,open]{d}&Y^{\leql,\tau}(\tld{z})\ar[hook,open]{d} & \phantom{M} & \phantom{M} \\
  \cX^{\leq \lambda,\tau}_{\reg}\ar[hook]{r}&Y^{\leql,\tau,\nabla_\infty} \ar[hook]{r}&Y^{\leql,\tau}&\phantom{M} & \phantom{M}
\end{tikzcd}
\end{equation}
where:
\begin{itemize}
\item {The objects $M_{\cJ}(\leql), M_{\cJ}^{\nv}(\leql, \nabla_{\bf{a}_\tau})$ are defined in \S \ref{subsec:products};
$Y^{\leql,\tau}$, $Y^{\leql,\tau}(\tld{z})$, $U(\tld{z},\leql)$, and $\tld{U}(\tld{z},\leql)$ are defined in \S \ref{subsec:LMofBK},$Y^{\leql,\tau,\nabla_\infty}$ is defined in \S \ref{subsec:tame:EGstacks} and finally $\tld{U}_\reg(\tld{z},\leql, \nabla_{\bf{a}_\tau})$ is defined in the paragraph above;
}
\item $\cX^{\leq \lambda,\tau}_\reg$ is the scheme theoretic union $\bigcup_{\lambda'} \cX^{ \lambda',\tau}$, where $\lambda'$ runs over all regular dominant coweights $\leq \lambda$.
\item All solid rectangles are Cartesian. This defines any previously undefined object in the diagram, {namely $\tld{U}^{\nv}(\tld{z}, \leql,\nabla_{\bf{a}_\tau})^{\wedge_p}$, $\ U^{\nv}(\tld{z},\leql,\nabla_{\bf{a}_\tau})^{\wedge_p}$, $\tld{\cX}_{\reg}^{\leq\lambda,\tau}(\tld{z})$, $\tld{U}(\tld{z},\leql,\nabla_{\tau,\infty})$, $\cX_{\reg}^{\leq \lambda,\tau}(\tld{z})$ and $Y^{\leql,\tau,\nabla_\infty}(\tld{z})$}. 
{(Note that the first two objects are $p$-adic completions of $\tld{U}^{\nv}(\tld{z}, \leql,\nabla_{\bf{a}_\tau})$, $\ U^{\nv}(\tld{z},\leql,\nabla_{\bf{a}_\tau})$, which are defined by the same pullbacks but without the $p$-adic completion.)}
\item All undecorated hooked arrows are closed immersions.
\item All circled hooked arrows are open immersions. %
\item All arrows decorated with $T^{\vee,\cJ}_\cO$ are $T^{\vee,\cJ}_\cO$-torsors.
\end{itemize}
Then:
\begin{enumerate}
\item 
\label{it:stack_local_model:1}
There exists an integer $N_{\sing}=N(\{\lambda_j\}_{j\in\cJ})$ which depends only on the subset $\{\lambda_j\}\subset \bZ^n$ $($and not on $p)$ such that if $\mu$ is $N_{\sing}$-deep in $\un{C}_0$, then the {diagonal} dotted arrow exists for all $\tld{z}\in\tld{W}^{\vee,\cJ}$.
\item 
\label{it:stack_local_model:2}
There exists a polynomial $P=P_{\{\lambda_j\}_{j\in\cJ},e}(X_1,\cdots, X_n)\in \bZ[X_1,\cdots, X_n]$ depending only on the subset $\{\lambda_j\}\subset \bZ^n$ and the ramification index $e$ of $\cO$ $($and not on $p)$, such that if  $\mu$ is $P-$generic, then%
\begin{itemize}
\item The dotted arrows exist and the vertical dotted arrow is an isomorphism.
\item %
For any $\lambda'\leq \lambda$ regular dominant and any semisimple $\rhobar\in \cX_n(\overline{\bF}_p)$, the versal rings to $\cX^{\lambda',\tau}$ at $\rhobar$ are domains $($or $0)$. In other words, $\cX^{\lambda',\tau}$ is analytically irreducible at the $\overline{\bF}_p$-points corresponding to semisimple $G_K$-representations.
\end{itemize}
\end{enumerate}
\end{thm}
\begin{warning} \begin{enumerate} 
\item
Unlike $\tld{U}^{\nv}(\tld{z},\leql,\nabla_{\bf{a}_\tau})$ which is a scheme, $\tld{U}(\tld{z},\leql, \nabla_{\tau,\infty})$ is only a $p$-adic formal scheme. This is because the equations imposed by the $\nabla_{\tau,\infty}$ condition involve infinite series which only make sense over $p$-adically complete test rings. This is why we need to put the $p$-adic completion on some objects in the diagram.
\item 
{The way we will produce the dotted map is by invoking Proposition \ref{prop:Elkik}, which appeals to Elkik's approximation theorem, and hence produces non-canonical liftings. We do not expect that there is a choice which makes the triangle commute. Looking at explicit formulas, we suspect, but have not tried to show, that the rigid generic fibers of $\tld{U}(\tld{z},\leql, \nabla_{\tau,\infty})$ and $\tld{U}^{\nv}(\tld{z},\leql,\nabla_{\bf{a}_\tau})^{\wedge_p}$ do not coincide as subspaces of the rigid generic fiber of $\tld{U}(\tld{z},\leql)^{\wedge_p}$ in general.}

However, we will see that once the diagonal dotted arrow has been constructed, it will induce the vertical dotted arrow to make the top left trapezoid commute.
\end{enumerate}
\end{warning}

\begin{rmk}\label{rmk:product_stack_local_model} Theorem \ref{thm:stack_local_model} is stated for the stacks $\cX_{\reg}^{\leq \lambda,\tau}$ parametrizing representations of $G_K$ where $K$ is a given unramified extension of $\Q_p$ which we fixed at the beginning of this section.
More generally, if we have a finite collection $(F^+_v)_{v\in S_p}$ of such, we have analogous objects $\cX_{\reg}^{\leq \lambda,\tau}$, $Y^{\leql,\tau}$ etc.~by taking products over $S_p$. Then the proof given below carries over verbatim and shows that Theorem \ref{thm:stack_local_model} continues to hold in this more general setting. 
\end{rmk}
\begin{lemma} \label{lem:local_model_lower_bound} Let $\lambda' \in X_*(T^\vee)^\cJ$ regular dominant such that $\lambda'_j \in [0,h']^n$ for all $j \in\cJ$ and $\tau$ a tame inertial type together with a fixed  $(4n + h')$-generic lowest alcove presentation.
Then $\cX^{\lambda',\tau}(\tld{z})=\cX^{\lambda',\tau} \cap Y^{\leq\lambda',\tau}(\tld{z})$ is non-empty if and only if $\tld{z}\in \Adm^\vee(\lambda')$. %
\end{lemma}
\begin{proof} For one direction, we consider the Breuil--Kisin module $\fM_{\tld{z}}\in Y^{\leql,\tau}(\tld{z})(\bF)$ which has matrices of partial Frobenii with respect to an eigenbasis given by $A^{(j)}=\tld{z}_j$ where $\tld{z} \in \Adm^{\vee}(\lambda')$. Then $\rhobar=T^*_{dd}(\fM_{\tld{z}})$ is a semisimple representation of $G_{K_\infty}$, hence is also a semisimple representation of $G_K$.  By Corollary \ref{lem:ss_bound}, $\rhobar$ is $4n$-generic.  Furthermore, by Proposition \ref{prop:semisimple_admissible} based on \cite[Corollary 3.2.17]{LLL}, $\rhobar|_{I_K}$ admits a lowest alcove presentation $(w, \nu)$ which is $(\lambda' - \eta)$-compatible such that $\tld{w}(\rhobar, \tau) \in \Adm(\lambda')$.   Then $\rhobar$ satisfies the fourth item of Proposition \ref{prop:patchnonzero} with $\lambda = \lambda' - \eta$ and hence $\rhobar$ admits a potentially crystalline lift with Hodge-Tate weight $\lambda'$ and inertial type $\tau$, and this produces a point of $\cX^{\lambda',\tau}(\tld{z})$. Conversely, if $\cX^{\lambda',\tau}(\tld{z})\neq \emptyset$, then $Y^{\leq \lambda',\tau}(\tld{z})\neq \emptyset$, hence $\tld{z}\in \Adm^{\vee}(\lambda')$ by Corollary \ref{cor:non_emptyness_pattern}.
\end{proof}

\begin{proof}[Proof of Theorem \ref{thm:stack_local_model}] 
\begin{enumerate}
\item
Suppose $\mu$ is $M$-deep in $\un{C}_0$ where $M \geq 2 h - 3$.
From Theorem \ref{thm:Breuil-Kisin_local_model}, $\tld{U}(\tld{z},\leql)^{\wedge_p}=\Spf R$ classifies objects in $Y^{\leql,\tau}$ together with a $\tld{z}$-gauge basis. Thus we get the universal such pair $(\fM,\beta)$ over $R$, and the matrix of partial Frobenii $A^{(j)}=A^{(j)}_{\fM,\beta}$. 
We have $\tld{U}(\tld{z},\leql,\nabla_{\tau,\infty})=\Spf R/I_{\fM,\nabla_\infty}$, while{, by Remark \ref{rmk:derivative_vs_vanishing} and the definition of $\bf{a}_{\tau,j}$,} $\tld{U}^{\nv}(\tld{z},\leql,\nabla_{\bf{a}_\tau})^{\wedge_p}=\Spf R/I_{\fM,\beta,\nabla_1}$ is cut out by the condition  
\[  v \frac{d}{dv}A^{(j)}(A^{(j)})^{-1}  + A^{(j)} \Diag(\bf{a}_{\tau,j}) (A^{(j)})^{-1} \in \frac{1}{v+p}L^+\cM(R).\]
By Proposition \ref{prop:monodromy_approximation}, $I_{\fM,\beta,\nabla_1} \subset (I_{\fM,\nabla_\infty},p^N)$, where $N=M-2h +3$. 
We also note that 
 this implies that for each $j\in \cJ$ there is a diagram 
\[\xymatrix{\Spec R/(I_{\fM,\nabla_\infty},p^N)\ar@{^{(}->}[r]\ar[d]& \Spec R/(I_{\fM,\beta,\nabla_1},p^N)\ar[r]^-{f_j}& \cU_X^{\nv}(\tld{z}_j,\leql_j,\nabla)\ar[d]\\
\Spec  R/I_{\fM,\nabla_\infty}\ar[rr]^-{(-p,\bf{a}_{\tau,j})} &&\bA^1\times \bA^n}
\] \
Note that $R/I_{\fM,\nabla_\infty}$ is $p$-adically complete and $p$-torsion free. By Proposition \ref{prop:Elkik}, there exists an integer $N'_\sing$ depending only on the set $\{\lambda_j\}$ such that if $M \geq 2\max_{j\in\cJ} h_{\lambda_j}+N'_\sing$ , we can lift $f_j \mod p$ to a map $\tld{f}_j: \Spec R/I_{\fM,\nabla_\infty}\to \cU_X^{\nv}(\tld{z}_j,\leql_j,\nabla)$. Since $T^{\vee,\cJ}$ is smooth over $\bZ$, we can also lift the composite $ \Spec R/(I_{\fM,\nabla_\infty},p) \to \tld{U}(\tld{z},\leql)=T^{\vee,\cJ}_{\cO} \times U(\tld{z},\leql)\to T^{\vee,\cJ}_{\cO}$ to $ \Spec R/I_{\fM,\nabla_\infty} \to T^{\vee,\cJ}_{\cO}$.
{Taking the product of the lifts above produces a map}
\[\tld{f}:\Spec R/I_{\fM,\nabla_\infty}\to (\prod_{j\in \cJ} T^{\vee}\times\cU_X^{\nv}(\tld{z}_j,\leql_j,\nabla))_\cO\]
where the base change is along the map $\Spec \cO \to X\times \bA^{n,\cJ}$ given by the tuple $(-p,\bf{a}_\tau)$. By Remark \ref{rmk:Mnv:sub}, {the fact that $U^{\nv}(\tld{z},\leql, \nabla_{\bf{a}_\tau})=M^{\nv}_{\cJ}(\leql,\nabla_{\bf{a}_\tau})\cap U(\tld{z},\leql)$,} and the fact that $R/I_{\fM,\nabla_\infty}$ is $\cO$-flat and $p$-adically complete, $\tld{f}$ factors through $U^{\nv}(\tld{z},\leql, \nabla_{\bf{a}_\tau})$, and further factors through the $p$-adic completion, which produces the desired dotted arrow. 
{This map is a closed immersion of $p$-adic formal algebraic stacks, since it is a closed immersion modulo $p$.}
This finishes the proof of the first part.
\item

We first claim that if the diagonal dotted arrow exists, then it induces the vertical arrow so that the trapezoid commutes. Indeed by construction $\tld{U}_\reg(\tld{z},\leql, \nabla_{\bf{a}_\tau})^{\wedge_p}$ is the maximal closed $p$-adic formal subscheme of $\tld{U}^{\nv}(\tld{z},\leql, \nabla_{\bf{a}_\tau})^{\wedge_p}$ which is $\cO$-flat and equidimensional of dimension $1+\dim_E (B\backslash \GL_n)_E \#\cJ  +n \#\cJ$. But Theorem \ref{thm:EG_basic_properties} says that $\tld{\cX}_\reg^{\leq \lambda,\tau}(\tld{z})$ has both these properties, proving our claim.

We abbreviate $\Spf B= \tld{U}_\reg(\tld{z},\leql,\nabla_{\bf{a}_\tau})^{\wedge_p}$ and $\Spf A=\tld{\cX}^{\leq \lambda,\tau}_\reg(\tld{z})$. Note that $A$ and $B$ are both $\cO$-flat, $p$-adically complete, reduced, equidimensional of dimension $1+\#\cJ \frac{n(n+1)}{2}$, and there is a surjection $B\onto A$ provided by the existence of the dotted vertical arrow.

We now apply Theorem \ref{thm:model_unibranch} for each $\lambda'\leq \lambda$ regular dominant. It implies there exists a polynomial $P\in \bZ[X_1,\cdots X_n]$ depending only on the ramification index $e$ of $\cO$ and the set $\{\lambda_j\}_{j\in \cJ}$ such that if $P(\bf{a}_{\tau,j})\mod p \neq 0$ for all $j\in \cJ$, then for each $\lambda'\leq \lambda$ regular dominant:
\begin{itemize}
\item {$M(\lambda',\nabla_{\bf{a}_\tau})$ is the base change of $\cM_{X,\cJ}(\lambda',\nabla)$ via the map $\Spec \cO  \ra X \times \mathbb{A}^{n,\cJ}$ given by $(-p, \bf{a}_{\tau})$.}
\item \label{item:unibranch_1}
 $\cO(U(\tld{z},\lambda',\nabla_{\bf{a}_\tau})^{\wedge_p})$ is a domain. Hence the same is also true for $\cO(\tld{U}(\tld{z},\lambda',\nabla_{\bf{a}_\tau})^{\wedge_p})$.
\item \label{item:unibranch_2}$U(\tld{z},\lambda',\nabla_{\bf{a}_\tau})$ is unibranch (equivalently, analytically irreducible) at $\tld{z}$. In other words, any of its versal rings at $\tld{z}$ is a domain. %
\end{itemize}
{(To arrange the first item, we use Remark \ref{rmk:generic_flatness} to the map $\cM_{X,\cJ}(\lambda',\nabla)\to X\times \mathbb{A}^{n,\cJ}$ to guarantee that the base change of $\cM_{X,\cJ}(\lambda',\nabla)$ via the map $\Spec \cO  \ra X \times \mathbb{A}^{n,\cJ}$ given by $(-p, \bf{a}_{\tau})$ is $\cO$-flat for each $\lambda'$.)}

Now if $U(\tld{z},\lambda',\nabla_{\bf{a}_\tau})\neq \emptyset$ then $Y^{\leq\lambda',\tau}(\tld{z})_\bF\neq \emptyset$. Hence by Corollary \ref{cor:non_emptyness_pattern}, $\tld{z}\in \Adm^\vee(\lambda')$.
It follows from this and the first item above that the number of minimal primes of $B$ is at most
\[\# \{\lambda'\leq \lambda | \lambda' \textrm{ regular dominant},\, \tld{z}\in \Adm^{\vee}(\lambda') \}.\]
On the other hand, taking $M \geq \max \{4n + h, 2h +3 \}$, Lemma \ref{lem:local_model_lower_bound} shows that the number of minimal primes of $A$ is at least  
\[\# \{\lambda'\leq \lambda | \lambda' \textrm{ regular dominant},\, \tld{z}\in \Adm^{\vee}(\lambda') \}.\]
This forces the kernel $\ker(B\onto A)$ to lie in the intersection of all the minimal primes of $B$, and hence is $0$. This shows that the surjection induces an isomorphism $B\risom A$.

We now wish to show that given a semisimple $G_K$ representation $\rhobar \in \cX_n(\overline{\bF}_p)$ {and $\lambda'\leq \lambda$ regular dominant, any versal ring to $\cX^{\lambda',\tau}$ at $\rhobar$ is an integral domain or zero.}

We may assume without loss of generality that $\rhobar\in \cX_n(\bF)$. For any $\lambda'\leq \lambda$, let $n_{\lambda'}\geq 0$ be the number of minimal primes of any versal ring to $\cX^{\lambda',\tau}$ at $\rhobar$. 

First suppose $\rhobar \in \cX^{\leq \lambda',\tau}(\bF)$. By Proposition \ref{prop:semisimple_admissible}, the image $\fM_\rhobar$ of $\rhobar$ in $Y^{\leq\lambda',\tau}(\bF)$ is a semisimple Breuil--Kisin module of some shape $\tld{z}$ by the uniqueness of $\fM_{\rhobar}$ (Proposition \ref{prop:BK_to_phi_mono}). Hence,  $\fM_\rhobar\in Y^{\leq\lambda',\tau}(\tld{z})(\bF)$, and can be lifted to an element in $T^{\vee,\cJ}(\bF)\tld{z} \subset \tld{U}(\tld{z},\leql')(\bF)$. Thus we can find a versal ring $R$ to $\cX_\reg^{\leq \lambda',\tau}$ at $\rhobar$ which is also a versal ring to $U_\reg(\tld{z},\leql',\nabla_{\bf{a}_\tau})$ at $\tld{z}$. Now it follows from the second item above that the number of minimal primes of $R$ is exactly $\#\{\lambda'' \leq \lambda' \textrm{ regular dominant}| \tld{z}\in \Adm^{\vee}(\lambda'')\}$.
Since $\cX_\reg^{\leq \lambda',\tau}$ is the scheme theoretic union of $\cX^{\lambda'',\tau}$ over {$\lambda''\leq \lambda'$} regular dominant {(each of which is equidimensional of the same dimension and no two share an irreducible component)}, 
we have thus shown that 
\begin{equation}\label{eq:branch_number}\sum_{{\lambda'' \leq \lambda'} \textrm{ regular dominant }} n_{\lambda''}=\#\{{\lambda'' \leq \lambda'} \textrm{ regular dominant}| \tld{z}\in \Adm^{\vee}(\lambda'')\}
\end{equation}
On the other hand, if $\rhobar \notin \cX^{\leq \lambda',\tau}(\bF)$, then $\sum_{{\lambda''\leq \lambda'}} n_{\lambda''}=0$, and $\#\{\lambda'' \leq \lambda'\textrm{ regular dominant}| \tld{z}\in \Adm^{\vee}(\lambda'')\}=0$ by Lemma \ref{lem:local_model_lower_bound}.

Thus equation (\ref{eq:branch_number}) holds for all $\lambda'\leq \lambda$ regular dominant. This implies by induction on $\#\{\lambda''\leq \lambda'| \lambda'' \textrm{ regular dominant}\}$ that $n_{\lambda'}\in \{0,1\}$ and that $n_{\lambda'}=1$ if and only if $\tld{z}\in \Adm^\vee({\lambda'})$. Thus any versal ring to $\cX^{\lambda',\tau}$ to $\rhobar$ is either the zero ring or a domain.
\end{enumerate}
\end{proof}

\subsection{Structure of potentially crystalline stacks modulo $p$}
\label{subsec:PCS:modp}
In \cite[Theorem 6.5.1]{EGstack}, Emerton and Gee describe a parametrization of the irreducible components of the underlying reduced stack $\cX_{n,\red}$ of the moduli of $(\phz,\Gamma)$-modules $\cX_n$ by Serre weights of $\GL_n(\cO_K)$. 
Taking products, this gives a parametrization of the irreducible components of the underlying reduced stack 
\[
\cX_{n,\red}^{F^+_p} = \prod_{v\in S_p,\F}\cX_{n,\red}^{F^+_v} \subset \cX_n^{F^+_p} \defeq \prod_{v\in S_p,\Spf \cO} \cX_n^{F^+_v}
\]
by Serre weights of $\rG$.
Let $\sigma=F(\kappa)$ be a Serre weight of $\GL_n(\cO_K)$ with $\kappa\in X_1(\un{T})=X_1(T)^{\cJ}$. Then the component $\cX_{EG,n,\red}^{\sigma}$ labelled by $\sigma$ is characterized as the reduced substack of $\cX_n$ which is is the closure of the locus of $\rhobar\in \cX_n(\overline{\bF}_p)$ such that $\rhobar$ has the form %
\[ \rhobar \cong \begin{pmatrix}  \chi_1 &* &\cdots & * \\
0&\chi_2& \cdots& *\\
\vdots &&\ddots&\vdots\\
0&\cdots &0 & \chi_n
\end{pmatrix}
\]
where%
\begin{itemize}
\item $\rhobar$ is maximally non-split niveau $1$, i.e.~it has a unique $G_K$-stable complete flag;
\item $\chi_i|_{I_K} = \prod_{j\in\cJ} \ovl{\omega}_{K,\sigma_j}^{1-i-\kappa_{j,n+1-i}}$;
\item If $\chi_{i+1}\chi_i^{-1}|_{I_K}= \overline{\eps}^{-1}$, then $\langle \kappa_j, \eps^{\vee}_{n-i}-\eps^{\vee}_{n+1-i}\rangle=p-1$ for all $j\in \cJ$ if and only if $\chi_{i+1}\chi^{-1}_i =\ovl{\eps}^{-1}$, and the element $\mathrm{Ext}^1(\chi_i,\chi_{i+1})=H^1(G_K,\ovl{\eps})$ determined by $\rhobar$ is tr\`{e}s ramifi\'{e}e (and otherwise $\langle \kappa_j, \eps^{\vee}_{n-i}-\eps^{\vee}_{n+1-i}\rangle=0$ for all $j\in \cJ$). 
\end{itemize}
We define $\cC_{\sigma}\defeq \cX^{\sigma^\vee\otimes \det^{n-1}}_{EG,n,\red}$. 
 Thus if $\sigma=F(\kappa)$ is $1$-deep and $\rhobar$ is as above, then %
\begin{equation}\label{eq:generic_points_cpt} \rhobar \cong \begin{pmatrix}  \chi_1 &* &\cdots & * \\
0&\chi_2& \cdots& *\\
\vdots &&\ddots&\vdots\\
0&\cdots &0 & \chi_n
\end{pmatrix}
\end{equation}
where $\chi_{i}|_{I_K}=\prod_{j\in \cJ} \ovl{\omega}_{K,\sigma_j}^{(\kappa_j+\eta_j)_i}$, and admits a unique $G_K$-stable flag. %

We now analyze the $\cC_\sigma$ in terms of local models, for sufficiently generic $\sigma$. To do so, we recall the setup of Section \ref{sub:LM:DL:SW}. Thus, we fix $\zeta\in X^*(\un{Z})$ an algebraic central character, a regular dominant weight $\lambda\in X^*(\un{T})$ such that $\lambda_j \in [0, h]^n$ for all $j \in \cJ$, and a tame inertial type $\tau=\tau(s,\mu+\eta)$ with lowest alcove presentation $(s,\mu)$ which is $\lambda$-compatible with $\zeta$. Set $\tld{w}^*(\tau)=s^{-1}t^{\mu+\eta}$. %
We assume that $\mu$ is $(h+2)$-deep in $\un{C}_0$.
We also continue to use notations from Section \ref{sec:local_model_EG}.

Recall the diagram from Proposition \ref{prop:phi_local_model} specialized with $a = 0, b=h$ and $\tld{z} = \tld{w}^*(\tau)$.  

\begin{equation}
\label{eq:local_model_mod_p}
\xymatrix{\tld{M}_{\cJ}(\leql)_{\bF} \ar@{^{(}->}^-{r_{\tld{w}^*(\tau)}}[r]\ar[d]^{\pi_{(s,\mu)}} & \tld{\Fl}^{[0,h]}_{\cJ,\tld{w}^*(\tau)} \ar[d] &\\  
Y^{\leql,\tau}_\bF \ar@{^{(}->}[r]& \Big[\tld{\Fl}^{[0,h]}_{\cJ, \tld{w}^*(\tau)}/T^{\vee, \cJ}_{\F}\text{-conj}\Big]  \ar@{^{(}->}^-{\iota_{\tld{w}^*(\tau)}}[r] & \Phi\text{-}\Mod^{\text{\'et},n}_{K, \F}}
\end{equation}

We have the potentially crystalline substack $\cX^{\leql,\tau}_\bF \into Y^{\leql,\tau}_{\bF}$ by Proposition \ref{prop:stack_diagram}, and define $\tld{\cX}^{\leql,\tau}_\bF$ to be its pullback along $\pi_{(s,\mu)}$. This is compatible with our earlier notation, since for any $\tld{z}\in\tld{W}^{\vee,\cJ}$, $\tld{\cX}^{\leql,\tau}_\bF\cap \tld{U}( \tld{z},\leql)=\tld{\cX}^{\leql,\tau}(\tld{z})_\bF$.

When working over $\F$, we have the following refinement of Theorem \ref{thm:stack_local_model}(\ref{it:stack_local_model:1}):
\begin{prop}\label{prop:mod_p_factor} Assume $\mu$ is $(2h-2)$-deep in $\un{C}_0$.   Then the closed immersion $\tld{\cX}^{\leql,\tau}_{\F} \into \tld{M}_{\cJ}(\leql)_\F$ factors through $\tld{M}_{\cJ}^{\nv}(\leql,\nabla_{\bf{a}_\tau})_\F$. %
\end{prop}
\begin{proof} It suffices to check the factorization after intersecting with each affine open $\tld{U}(\tld{z},\leql)_\F$. But this follows from Proposition \ref{prop:monodromy_approximation}, since $\mu$ is at least $2h-2$-deep in $\un{C}_0$ (cf.~the proof of Theorem \ref{thm:stack_local_model}(\ref{it:stack_local_model:1})).
\end{proof}

We now recall the top dimensional irreducible components of $\tld{M}^{\nv}_\cJ(\leql,\nabla_{\bf{a}_\tau})_{\F}$ constructed in Sections \ref{sec:componentmatching} and \ref{subsec:products} and identified in Theorem \ref{thm:compandSW}. Since $\lambda$ is regular,  $\dim \tld{M}^{\nv}_\cJ(\leql,\nabla_{\bf{a}_\tau})_{\F} =d_\cJ=\#\cJ\dim_{\F} (B\backslash \GL_n)_{\F}$. For each Serre weight $\sigma$ with a lowest alcove presentation $(\tld{w},\omega)$ compatible with $\zeta$, we have a closed $d_\cJ$-dimensional subvariety $C^\zeta_\sigma=C_{(\tld{w},\omega)}$ of $(\Fl^{\nabla_0})^{\cJ}$ defined in \eqref{Zomegas}.   Recall that  
\[ C_{(\tld{w},\omega)}= \prod_{j \in \cJ} S_\F^{\nabla_0}(\tld{w}_{1, j},\tld{w}_{2, j}, \tld{s}_j) \]
for any choices of $\tld{w}_1,\tld{w}_2,\tld{s}$ such that
\[(\tld{w},\omega)=(\tld{w}_1, \tld{s}\tld{w}_2^{-1}(0)),\]
cf.~Theorem \ref{thm:cpt:match}.  
{Recall that $ S_\F^{\nabla_0}(\tld{w}_{1, j},\tld{w}_{2, j}, \tld{s}_j)$ (Definition \ref{defn:Schubert:var}\eqref{it:Schubert:var:3}) is the closure of the intersection $S^\circ_\F((\tld{w}_{2, j}^{-1}w_0\tld{w}_{1, j})^*)\tld{s}_j^*\cap \Fl^{\nabla_0}$.}
Pulling back to $\tld{\Fl}$, we get the subvarieties $\tld{C}^{\zeta}_\sigma \in \tld{\Fl}$.

Theorem \ref{thm:compandSW} shows that the top dimensional irreducible components of $\tld{M}^{\nv}_\cJ(\leql,\nabla_{\bf{a}_\tau})$ are exactly the translates $\tld{C}^\zeta_{\sigma}\tld{w}^{*}(\tau)^{-1}\subset \tld{\Fl}$, where $\sigma$ runs over $\JH(W(\lambda-\eta)\otimes \ovl{\sigma}(\tau))$. %

Our main result in this section is the following:
\begin{thm} 
\label{thm:EGmodp}
Let $\lambda \in X_*(T)^{\cJ}$ be regular dominant and let $\tau$ be a tame inertial type with lowest alcove presentation $(s,\mu)$ which is $(\lambda - \eta)$-compatible with $\zeta \in X^*(\un{Z})$. Assume that $\mu$ is $\max \{2(h+1), 4n + h \}$-deep. %

\label{thm:irreducible_components_mod_p}
\begin{enumerate}
\item 
\label{it:irreducible_components_mod_p:1}
$\cX^{\leq \lambda,\tau}_{\reg,\red}=\cX^{\lambda,\tau}_\red=\cup_\sigma \cC_{\sigma}$, where the union runs over all Serre weights $\sigma\in \JH(W(\lambda-\eta)\otimes \ovl{\sigma}(\tau))$.
\item 
\label{it:irreducible_components_mod_p:2}
For each $\sigma\in \JH(W(\lambda-\eta)\otimes \ovl{\sigma}(\tau))$, we have a local model diagram:
\begin{equation}\label{eq:mod_p_local_model}
\xymatrix{ &&&&\tld{C}^{\zeta}_\sigma   \ar@{^{(}->}[d]\\
\tld{\cC}_\sigma \ar@{^{(}->}[r] \ar@/^2pc/[rrrru]^{\cong} \ar[d]&\tld{\cX}^{\leql,\tau}_\F\ar@{^{(}->}[r]\ar[d]&\tld{M}^{\nv}_{\cJ}(\leql,\nabla_{\bf{a}_\tau})_\F \ar@{^{(}->}[r]& \tld{M}_\cJ(\leql)_\F \ar@{^{(}->}^{r_{\tld{w}^*(\tau)}}[r]\ar[d]& \tld{\Fl}^{[0,h]}_{\cJ, \tld{w}^*(\tau)} \ar[d] \\
\cC_\sigma \ar@{^{(}->}[r]\ar@/_2pc/@{^{(}->}[rrrrd]& \cX^{\leql,\tau}_\F \ar@{^{(}->}[rr]&&Y^{\leql,\tau}_\F \ar@{^{(}->}[r]&\Big[\tld{\Fl}^{[0,h]}_{\cJ, \tld{w}^*(\tau)}/T^{\vee, \cJ}_{\F}\text{-\emph{conj}}\Big]\ar@{^{(}->}[d]\\
&&&&\Phi\text{-}\Mod^{\text{\emph{\'et}},n}_{K, \F}}
\end{equation}
where
\begin{itemize}
\item $\tld{\cC}_\sigma$ is defined so that all rectangles are Cartesian, and all vertical arrows are $T^{\vee,\cJ}_\F$-torsors.
\item All hooked arrows are closed immersions.
\item The bottom diagonal map is the canonical composition $\cC_\sigma \into \cX_n\to \Phi\text{-}\Mod^{\text{\emph{\'et}},n}_{K, \F}$.
\end{itemize}
\end{enumerate}
\end{thm}
\begin{rmk} 
\label{rmk:MLM:sp:fib}
\begin{enumerate}
\item 
\label{it:MLM:sp:fib:1}
The theorem shows that $\cC_\sigma= [\tld{C}^\zeta_\sigma/T^{\vee, \cJ}_{\F}\text{-conj}]$ as subfunctors of $\Phi\text{-}\Mod^{\text{\'et},n}_{K, \F}$. Note that this depends only on $\zeta$ and not on the choice of $\lambda$, $s$, and $\mu$.

On the other hand, making the choices $\lambda,s,\mu$ computes $\cC_\sigma$ as a quotient $[T_\F^{\vee,\cJ}C^{\zeta}_\sigma\tld{w}^{*}(\tau)^{-1}/_{(s,\mu)}T^{\vee,\cJ}]$, where $T_\F^{\vee,\cJ}C^{\zeta}_\sigma\tld{w}^{*}(\tau)^{-1}$ is an irreducible component of a \emph{deformed} affine Springer fiber in the sense of \cite{Frenkel_Zhu}, i.e.~the reduced subvariety of $\tld{\Fl}^\cJ$ cut out by the condition
\[(v\frac{d}{dv}g) g^{-1} + \Ad(g)\big(v^{s^{-1}(\mu+\eta)}\big)\in \Lie \cI\]
In particular, $\cC_\sigma$ is equisingular to an irreducible component of a deformed affine Springer fiber.
\item  \label{it:MLM:compmatch} As the proof shows, the isomorphism between $\tld{\cC}_{\sigma}$ and $\tld{C}_{\sigma}^{\zeta}$ holds as long as there exists $(\lambda, \tau)$ such that $\cC_{\sigma} \subset \cX^{\leql, \tau}$ and $\tau$ is $2(h+1)$-generic (in particular, any irreducible component of the special fiber of $\cX^{\eta, \tau}$ where $\tau$ is $2n$-generic).  As long as $\sigma$ is $(3n-1)$-deep, this can always be arranged (see the proof of Proposition \ref{prop:rhobaroncomponents}).
\item 
\label{it:MLM:sp:fib:3}
Under the weaker hypotheses that $\mu$ is $2(h+1)$-deep, one can still show the upper bound on the components $\cX^{\lambda,\tau}_\red \subset \cX^{\leq \lambda,\tau}_{\reg,\red} \subset \cup_\sigma \cC_{\sigma}$.  
In the proof, the bound $4n + h$ only appears when invoking weight elimination and modularity of obvious weights from \cite{LLL}.  
\item 
\label{it:MLM:sp:fib:2}
Using the fact that $M^{\nv}_\cJ(\leql,\nabla_{\bf{a}_\tau})_\F$ is equidimensional (cf.~Remark \ref{rmk:local_model_pure_dim}), one can strengthen the first part of the Theorem to $\cX_\red^{\leql,\tau}=\cup_\sigma \cC_\sigma$; in particular, $\cX^{\leql,\tau}_\F$ is equidimensional. This is because Lemma \ref{lem:lower_bound_cpt} below shows that $\tld{\cX}^{\leql,\tau}_{\red}$ exhausts all the top dimensional irreducible components of $M^{\nv}_\cJ(\leql,\nabla_{\bf{a}_\tau})_\F$, and has the same underlying reduced scheme.
\item \label{item:prod_mod_p_comp}
By taking products over a finite set $S_p$ indexing finite unramified extensions $F^+_v$ of $\Q_p$, one obtains the evident generalization of this theorem for $\cX^{F^+_p,\lambda,\tau}_{\F}$.

\end{enumerate}
\end{rmk}
The rest of this section is devoted to the proof of Theorem \ref{thm:irreducible_components_mod_p}.

\begin{lemma}\label{lem:lower_bound_cpt} Assume that $\mu$ is $(4n + h)$-deep in $\un{C}_0$.  Let $\sigma\in \JH(W(\lambda-\eta)\otimes \ovl{\sigma}(\tau))$. Then $\cC_\sigma \subset \cX^{\lambda,\tau}\subset \cX^{\leql,\tau}$.
\end{lemma}
\begin{proof} By Proposition \ref{prop:JHbij} (noting the $\eta$-shift),  $\sigma$ is $4n$-deep.  %
Then for any $\rhobar$ of the form (\ref{eq:generic_points_cpt}) for $\sigma$, $\rhobar$ is $4n$-generic.
By Proposition \ref{prop:obvious_wild_lift}, any such $\rhobar$ of lies in $\cX^{\lambda,\tau}(\ovl{\F})$. Since such points are dense in $\cC_\sigma$, we are done.
\end{proof}
\begin{cor} \label{cor:bound_cpt}  Assume that $\mu$ is $\max \{2(h+1), 4n + h \}$-deep in $\un{C}_0$.   Then the $d_\cJ$-dimensional irreducible components of $\cX^{\lambda,\tau}_\red$ and $\cX^{\leql,\tau}_\red$ are exactly the $\cC_\sigma$ with $\sigma\in \JH(W(\lambda-\eta)\otimes \ovl{\sigma}(\tau))$.
\end{cor}
\begin{proof} By Lemma \ref{lem:lower_bound_cpt}, $\cX^{\lambda,\tau}$ has at least $\#\JH(W(\lambda-\eta)\otimes \ovl{\sigma}(\tau))$ $d_\cJ$-dimensional irreducible components. On the other hand, Proposition \ref{prop:mod_p_factor} and Theorem \ref{thm:compandSW} (using that $\mu$ is $2(h+1)$-deep which implies that $\tau$ is $2n$-generic) imply that $\cX^{\leql,\tau}$ has at most $\#\JH(W(\lambda-\eta)\otimes \ovl{\sigma}(\tau))$ $d_\cJ$-dimensional irreducible components. We conclude that equality must be achieved at each stage.
\end{proof}

\begin{lemma}\label{lem:good_open_cpt} Let $\sigma$ be an $(n-1)$-deep Serre weight with a lowest alcove presentation $(\tld{w},\omega)$ compatible with $\zeta$. Let $\kappa=\pi^{-1}(\tld{w})\cdot (\omega-\eta)$ so that $\sigma=F(\kappa)$. Then there is an open dense subset $U^\zeta_\sigma\subset C^\zeta_\sigma$ with the following property: For any point $x\in \tld{C}^\zeta_\sigma$ lying over $U^{\zeta}_\sigma$ with associated \'{e}tale $\phz$-module $\cM_x$, $\bV^*_{K}(\cM_x)$ has the form
\[\begin{pmatrix}  \chi_1 &* &\cdots & * \\
0&\chi_2& \cdots& *\\
\vdots &&\ddots&\vdots\\
0&\cdots &0 & \chi_d
\end{pmatrix}\]
where the (canonical extension to $G_K$ of) character $\chi_i$ satisfies $\chi_{i}|_{I_K}=\prod_{j\in \cJ} \omega_{K,\sigma_j}^{(\kappa_j+\eta_j)_i}$.

\end{lemma}
\begin{proof} We write $\tld{w}=t_{\eta_w}w$, thus $\kappa_j=w_{j-1}(\omega_j)+p\eta_{w_{j-1}}-\eta_j$. %
The set of triples $(\tld{w}_1,\tld{w}_2,\tld{s})$ such that
\begin{equation} \label{eq:l1}
\tld{w}_1=\tld{w},\quad \tld{w}_2=\tld{w}_h\tld{w}_1,\quad \tld{s}\tld{w}_2^{-1}(0)=\omega
\end{equation}
is in bijection with $W^{\cJ}$, since the first two condition determines $\tld{w}_1, \tld{w}_2$, and the third condition uniquely specifies $\tld{s}$ once the image of $\tld{s}$ in $W^\cJ$ is fixed. Thus we can choose the triple $(\tld{w}_1,\tld{w}_2,\tld{s})$
such that the above conditions hold, and furthermore writing $\tld{s}=t_{\nu}s$, we have
\begin{equation} \label{eq:l2}
w_j s_j^{-1}w_{j-1}^{-1}=1
\end{equation} \label{eq:l2}
for all $j\in \cJ$. Note that our choices give $\tld{w}_2^{-1}w_0\tld{w}_1=t_{w^{-1}(\eta)}$.

We now choose $U^{\zeta}_\sigma$ to be the open affine \[
\prod_{j \in \cJ} \left(S^\circ_\F((\tld{w}_{2, j}^{-1}w_0\tld{w}_{1, j})^*)\tld{s}_j^*\cap \Fl^{\nabla_0}\right) \subset C^\zeta_\sigma. 
\]
Let $x\in \tld{C}^\zeta_\sigma(\ovl{\F})$ such that the image of $x$ in $\Fl^\cJ$ is in $U^{\zeta}_\sigma$. By Proposition \ref{prop:opencell}, Corollary \ref{cor:NandB} and the fact that $wt^{w^{-1}\eta}\in \un{\tld{W}}^{+}$, we see that $x$ can be represented by a tuple of matrices $(C^{(j)})_j=(D_jv^{w_j^{-1}(\eta_j)}w_j^{-1}\ovl{N}_jw_j\tld{s}_j^*)_j\in \GL_n(\ovl{\F}(\!(v)\!))^\cJ$ where $\ovl{N}_j$ is unipotent lower triangular and $D_j\in T^{\vee}(\ovl{\F})$. 

Writing $\tld{s}^*_j = s_j^{-1} \nu_j$, $\phz$-conjugation by $(w_j)$ yields
\[C'^{(j)}=w_j C^{(j)}w_{j-1}^{-1}=\Ad(w_j)(D_j)\Ad(v^{\eta_j})(\ovl{N}_j)v^{\eta_j}w_js_j^{-1}v^{\nu_j}w_{j-1}^{-1}=\ovl{B}'_jv^{\eta_j+w_{j-1}\nu_j} \]
where $\ovl{B}'_j$ is lower triangular with constant diagonal entries.
Thus the \'{e}tale $\phz$-module $\cM_x$ associated to $x$ has a filtration by rank $1$ \'{e}tale $\phz$-module, and $V_x=\bV^*_K(\cM_x)$ has a $G_{K_\infty}$-stable complete flag $0=\mathrm{Fil^0}\subset \mathrm{Fil^1}\subset \cdots \subset \mathrm{Fil^n}=V_x$ with associated graded 
\[\mathrm{gr}_iV_x|_{I_K} \cong \prod_{j\in \cJ} \omega_{K,\sigma_j}^{(\eta_j+w_{j-1}\nu_j)_i}.\]
This follows from Proposition 3.1.2 \cite{LLL} noting that in the conventions of this paper as explained in Remark \ref{rmk:cmpr:mat}(\ref{it:rmk:star}), $\taubar(s^*, \mu^*)$ would be replaced by $\taubar(s^{-1}, \mu)$ in the formula.  
Now the relations (\ref{eq:l1}) gives
\[\nu_j+s_jw^{-1}_j(\eta_j-\eta_{w_j})=\omega\]
 and (\ref{eq:l2}) then implies
\[\nu_j=\omega_j-w^{-1}_{j-1}(\eta_j-\eta_{w_j})\]
and thus 
\[\eta_j+w_{j-1}(\nu_j)=w_{j-1}(\omega_j) +\eta_{w_j}.\]
We conclude by observing that
\[\prod_{j\in \cJ} \omega_{K,\sigma_j}^{(\kappa_j+\eta_j)_i}=\prod_{j\in \cJ} \omega_{K,\sigma_j}^{(w_{j-1}(\omega_j)+p\eta_{w_{j-1}})_i}=\prod_{j\in \cJ} \omega_{K,\sigma_j}^{(w_{j-1}(\omega_j) +\eta_{w_j})_i},\]
since $\omega_{K,\sigma_j}^p=\omega_{K,\sigma_{j-1}}$.

\end{proof}
\begin{proof}[Proof of Theorem \ref{thm:irreducible_components_mod_p}] The first part follows from Corollary \ref{cor:bound_cpt} and the fact that $\cX^{\leql,\tau}_{\reg,\red}$ and $\cX^{\lambda,\tau}_{\red}$ are equidimensional of dimension $d_\cJ$.

We now prove the second part. Let $\sigma\in \JH(W(\lambda-\eta)\otimes \ovl{\sigma}(\tau))$ with lowest alcove presentation $(\tld{w},\omega)$ and set $\kappa=\pi^{-1}(\tld{w})\cdot(\omega-\eta)$. Then $\cC_\sigma$ occurs as an irreducible component of $\cX_{\red}^{\leql,\tau}$, and we have the pullback $\tld{\cC}_\sigma$ as in diagram (\ref{eq:mod_p_local_model}). Now $\tld{\cC}_\sigma$ is a top dimensional irreducible component of $\tld{M}_\cJ^{\nv}(\leql,\nabla_{\bf{a}_\tau})$, thus it must be of the form $\tld{C}^{\zeta}_{\sigma'}\tld{w}^{*}(\tau)^{-1}$ for some $\sigma'\in \JH(W(\lambda-\eta)\otimes \ovl{\sigma}(\tau))$. Let $(\tld{w}',\omega')$ be the lowest alcove presentation of $\sigma'$ compatible with $\zeta$ and $\kappa'=\pi^{-1}(\tld{w}')\cdot(\omega'-\eta)$. %

We need to show that $\sigma'=\sigma$. To this end, let $U^\zeta_{\sigma'}$ be the open subscheme of $C^{\zeta}_{\sigma'}$ constructed in Lemma \ref{lem:good_open_cpt}. By the definition of $\cC_\sigma$, we can find a dense set of points $x\in \tld{\cC}_\sigma(\ovl{\F})$ such that the associated Galois representation $\rho_x$ has the form described in (\ref{eq:generic_points_cpt}). We can thus find such a point $x$ which furthermore induces a point in $U^\zeta_{\sigma'}$. Since $\ad(\rhobar)$ is cyclotomic free as $\rhobar$ is at least $4n$-generic, Lemma \ref{lem:cyclo_free}(\ref{cyclo_free_4}) implies that that any $G_{K_\infty}$-stable filtration on $\rho_x|_{G_{K_\infty}}$ is automatically $G_K$-stable. We conclude that the filtration of $\rho_x$ coming from Lemma \ref{lem:good_open_cpt} and the filtration described in (\ref{eq:generic_points_cpt}) coincide. Comparing the associated graded thus shows that 
\[\prod_{j\in \cJ} \omega_{K,\sigma_j}^{(\kappa_{j}+\eta_j)_i}=\prod_{j\in \cJ} \omega_{K,\sigma_j}^{(\kappa'_{j}+\eta_j)_i}\]
for $1\leq i\leq n$. 
The equation above shows that $\kappa - \kappa' \in (p-\pi) X^*(T)$. Since $\kappa$ and $\kappa'$ are both $p$-restricted, $\kappa - \kappa' \in (p-\pi) X^0(T)$ which means that $\sigma \cong \sigma'$. 
\end{proof}

\begin{prop} \label{prop:rhobaroncomponents}
Let $\sigma$ be $(3n-1)$-deep Serre weight with lowest alcove presentation $(\tld{w}_1, \omega)$. Let $\rhobar$ be a tame $n$-dimensional representation of $G_K$ which is $2n$-generic. 
 \begin{enumerate}
 \item \label{item:comp:obv} If $\sigma \in W_{\obv}(\rhobar)$, then $\rhobar \in \cC_{\sigma}$.    
 \item \label{item:comp:fixed} For each $j\in\cJ$, let $P_{\tld{w}_{1 ,j}} \in \Z[X_1,\ldots, X_n]$ be as in Proposition \ref{prop:Tfixedpts}. If $\sigma \in W^?(\rhobar)$ and $P_{\tld{w}_{1, j}}(\omega_j) \neq 0 \mod p$ for all $j \in \cJ$, then $\rhobar \in \cC_{\sigma}$.   
  \item \label{item:comp:WE} If  $\rhobar \in \cC_{\sigma}$, then $\sigma \in W^?(\rhobar)$. 
 \end{enumerate}  
\end{prop}
\begin{proof}
First, we claim there exists a $2n$-generic $\tau$ such that $\cC_{\sigma} \subset \cX^{\eta, \tau}$. Let $\sigma = F(\kappa)$.  The component $\cC_{\sigma}$ is characterized by the fact it contains all $\rhobar$ of the form \eqref{eq:generic_points_cpt} such that $\chi_i|_{I_K} = \prod_{j \in \cJ} \ovl{\omega}_{K, \sigma_j}^{(\kappa_j+ \eta_k)_i}$.   To show $\cC_{\sigma} \subset \cX^{\eta, \tau}$, it suffices to show that all such $\rhobar$ have potentially crystalline lifts of type $(\eta, \tau)$.  By Lemma \ref{lem:obvious_tame_type_lift}, this property holds for $\tau = \tau(1, \kappa)$.  Furthermore, if $\kappa$ is $(3n-1)$-deep in its alcove, then by Proposition \ref{prop:PSLAP}, $\tau(1, \kappa)$ is $2n$-generic.    

Thus, by Remark \ref{rmk:MLM:sp:fib}(\ref{it:MLM:sp:fib:3}), there is a diagram as in  \eqref{eq:mod_p_local_model} such that $\tld{\cC}_{\sigma}$ is isomorphic to $\tld{C}^{\zeta}_{\sigma}$.   By Propositions \ref{prop:ssetphi} and \ref{prop:semisimple_admissible}, $\rhobar \in \cC_{\sigma}$ if and only if $\rhobar$ admits a lowest alcove presentation such that $\tld{w}^*(\rhobar) \in \tld{C}^{\zeta}_{\sigma}$.

Since $\tld{w}^*(\rhobar) \in \tld{C}^{\zeta}_{\sigma}$ exactly when $\tld{w}^*(\rhobar) \in C^{\zeta}_{\sigma}$.   Each item follows directly from corresponding item in Theorem \ref{thm:Tfixedpts}.   
\end{proof}

\begin{rmk}  Proposition \ref{prop:rhobaroncomponents} likely holds for $2n$-deep weights.  However, it requires more work to realize $\cC_{\sigma}$ inside some $\cX^{\eta, \tau}$ in that case. 
\end{rmk}
\clearpage{}%
\clearpage{}%
\section{The Breuil--M\'ezard conjecture}
\label{sec:BMC}

In this section, we let $K/\Q_p$ be a finite extension and $n>0$ an integer.
Let $\cJ$ be $\Hom_{\Qp}(K,E)$.
We let $G_0$ be $\Res_{\cO_K/\Z_p} ({\GL_n}_{/\cO_K})$ so that $\un{G} \cong \prod_{\cJ} ({\GL_n}_{/\cO})$ and $\un{G}^\vee \cong \prod_{\cJ} \GL_n$. %
\subsection{The statement of the conjectures}
\label{sub:statements:BM}

In this section, we recall two conjectures which we call the \emph{geometric} and \emph{versal} Breuil--M\'ezard conjectures.

Let $\Z[\cX_{n,\mathrm{red}}]$ denote the free abelian group on the irreducible components $\cC_\sigma$ of $\cX_{n,\mathrm{red}}$ parametrized by Serre weights $\sigma$.
We call elements of $\Z[\cX_{n,\mathrm{red}}]$ \emph{cycles} and call $\cC_\sigma\in \Z[\cX_{n,\mathrm{red}}]$ for a Serre weight $\sigma$ an \emph{irreducible cycle}.
(One might normally call these \emph{top-dimensional} cycles among cycles of varying dimension, but since we only consider top-dimensional cycles, we omit this adjective.)
A cycle is \emph{effective} if its coefficients are nonnegative.
We say that $Z_1 \in \Z[\cX_{n,\mathrm{red}}]$ is greater than or equal to $Z_2 \in \Z[\cX_{n,\mathrm{red}}]$ (and write $Z_1 \geq Z_2$) if $Z_1 - Z_2$ is effective.
Let $K(\Rep_{\F}(\rG))$ be the Grothendieck group of finitely generated $\F[\rG]$-modules, or equivalently the free abelian group generated by Serre weights for $\rG$.
If $W$ is a finitely generated $\F[\rG]$-module, we write $[W] = \sum_\sigma [W:\sigma][\sigma]$ for its image in $K(\Rep_{\F}(\rG))$ where $[W:\sigma]$ denotes the multiplicity of a Serre weight $\sigma$ as a Jordan--H\"older factor of $W$.
If $V$ is a finitely generated $E[\rG]$-module, then $[\ovl{V}^\circ]$ is independent of the $\rG$-stable $\cO$-lattice $V^\circ \subset V$, and so denote this by $[\ovl{V}]$.
We then also denote $[\ovl{V}^\circ:\sigma]$ by $[\ovl{V}:\sigma]$.

A \emph{type} is a pair $(\lambda+\eta,\tau)$ where $\lambda\in X_*(\un{T}^{\vee})$ is a dominant weight and $\tau$ is a Weil--Deligne inertial type for $K$.
We say that a type is \emph{extremal} if $\tau$ is maximal or minimal with respect to $\preceq$. 
Recall that given an extremal type $(\lambda+\eta,\tau)$, $\cX^{\lambda+\eta,\tau}$ denotes the potentially semistable or the potentially crystalline stack of type $(\lambda+\eta,\tau)$.
Let $\cZ_{\lambda,\tau}$ denote the cycle 
\[
\sum_\sigma \mu_\sigma(\cX^{\lambda+\eta,\tau}_{\F}) \cC_\sigma
\]
in $\Z[\cX_{n,\mathrm{red}}]$ where $\mu_\sigma(\cX^{\lambda+\eta,\tau}_{\F})$ denotes the multiplicity of $\cC_\sigma$ as an irreducible component of $\cX^{\lambda+\eta,\tau}_{\F}$ in the sense of \cite[\href{https://stacks.math.columbia.edu/tag/0DR4}{Tag 0DR4}]{stacks-project}.
We also denote by $\lambda$ the corresponding element in $X^*(\un{T})$. %
For a set $\cS$ of extremal types $(\lambda+\eta,\tau)$, we write 
\[
\JH(\ovl{\sigma}(\cS)) \defeq \cup_{(\lambda+\eta,\tau)\in \cS} \JH(\ovl{\sigma}(\lambda,\tau)).\]
The following conjecture is based on a geometric version of (a generalization of) a conjecture of Breuil--M\'ezard (\cite{BM}).

\begin{conj}[Geometric Breuil--M\'ezard conjecture]\label{conj:S-BM}
Let $\cS$ be a set of extremal types.
Then for each $\sigma \in \JH(\ovl{\sigma}(\cS))$, there exists an effective cycle $\cZ_\sigma \in \Z[\cX_{n,\mathrm{red}}]$ such that for all $(\lambda+\eta,\tau) \in \cS$, we have
\[
\cZ_{\lambda,\tau} = \sum_{\sigma} [\ovl{\sigma}(\lambda,\tau):\sigma] \cZ_\sigma.
\]
\end{conj}

\begin{rmk}\label{rmk:ZgeqC}
Though it is not necessary for our purposes, we further expect that $\cZ_\sigma$ in Conjecture \ref{conj:S-BM} is greater than or equal to $\cC_\sigma$.
\end{rmk}

Recall from \cite[\S 3.3]{GHS} that $\cS$ is called a \emph{Breuil--M\'ezard system} if the map 
\begin{align*}
\Z[\cS] &\ra K(\Rep_{\F}(\rG))\\
(\lambda+\eta,\tau) &\mapsto [\ovl{\sigma}(\lambda,\tau)]
\end{align*}
has finite cokernel.

\begin{rmk}
\begin{enumerate}
\item If we take $\cS$ to contain all extremal types $(\lambda+\eta,\tau)$, then Conjecture \ref{conj:S-BM} combines the potentially crystalline and semistable parts of \cite[Conjecture 8.2.2]{EGstack} with the additional assertion that the cycles $\cZ_\sigma$ are effective.

\item It is not hard to see that if a system of cycles $\cZ_\sigma$ in Conjecture \ref{conj:S-BM} exists for a Breuil--M\'ezard system $\cS$, then it must be unique. 
Of course, for general $\cS$, there may be more than one system of cycles $\cZ_\sigma$ for which Conjecture \ref{conj:S-BM} holds. 
We will show that the cycles $\cZ_\sigma$ can sometimes also be characterized using minimal patching functors even when $\cS$ is not a Breuil--M\'ezard system (see Theorem \ref{thm:genBM}). 
\end{enumerate}
\end{rmk}

\begin{rmk} \label{rmk:BM:nonreduced}  If $[\ovl{\sigma}(\lambda,\tau):\sigma] > 1$ and $\cZ_{\sigma}$ is nonzero for some Serre weight $\sigma$, then Conjecture \ref{conj:S-BM} (with $(\lambda+\eta,\tau) \in \cS$) would imply that $\cX^{\lambda + \eta, \tau}_{\F}$ is necessarily non-reduced.   
It is known that when $n \geq 4$ and $\tau$ is $2n$-generic, $\JH(\ovl{\sigma}(\tau))$ has Jordan--H\"{o}lder factors with higher multiplicity and so the same will be true for $\JH(\ovl{\sigma}(\lambda, \tau))$ for any $\lambda$. %
Under suitable genericity hypotheses, Proposition \ref{prop:mod_p_factor} and Conjecture \ref{conj:S-BM} then imply that the local model $M_{\cJ}(\leql, \nabla_{\bf{a}_\tau})$ will also have non-reduced special fiber when $\ovl{\sigma}(\lambda,\tau)$ has multiplicities.  
\end{rmk}

Taking versal rings for $\cX_n$ (and taking Hilbert--Samuel multiplicities) recovers the original Breuil--M\'ezard conjecture.
Let $\rhobar:G_K\ra \GL_n(\F)$ be a continuous Galois representation.
We also let $\rhobar$ denote the corresponding $\F$-point of $\cX_n$.
Fix a versal ring $R_\rhobar^{\mathrm{ver}}$ for $\cX_n$ at $\rhobar$.
For example, we could take the framed deformation ring $R^\square_\rhobar$.
For a type $(\lambda+\eta,\tau)$, the fiber product $\Spf R_\rhobar^{\mathrm{ver}} \times_{\cX_n} \cX^{\lambda+\eta,\tau}$ is a closed formal subscheme of $\Spf R_\rhobar^{\mathrm{ver}}$, which we denote by $\Spf R_\rhobar^{\mathrm{ver},\lambda+\eta,\tau}$.
Then $R_\rhobar^{\mathrm{ver},\lambda+\eta,\tau}$ is a versal ring for $\cX^{\lambda+\eta,\tau}$ at $\rhobar$.
Similarly, the fiber product $\Spf R_\rhobar^{\mathrm{ver}} \times_{\cX_n} \cX_{n,\mathrm{red}}$ is a closed formal subscheme of $\Spf R_\rhobar^{\mathrm{ver}}$, which we denote by $\Spf R_\rhobar^{\mathrm{alg}}$.
{Since $\cX_{n,\textrm{red}}$ is an algebraic stack, the versal map $\Spf R_\rhobar^{\mathrm{alg}}\ra \cX_{n,\mathrm{red}}$ is effective (\cite[Definition 2.2.9]{EGschemetheoretic}), i.e.~arises from a map}
\[
i_{\rhobar}: \Spec R_\rhobar^{\mathrm{alg}}\ra \cX_{n,\mathrm{red}}.
\]
The map $i_{\rhobar}$ induces a map from the set of irreducible components of $\Spec R_\rhobar^{\mathrm{alg}}$ to the set of irreducible components $\cX_{n,\mathrm{red}}$ whose image is exactly the set of irreducible components of $\cX_{n,\mathrm{red}}$ containing $\rhobar$ (see \cite[\href{https://stacks.math.columbia.edu/tag/0DRB}{Tag 0DRB}]{stacks-project}).
We denote by $\Z[\Spec R_\rhobar^{\mathrm{alg}}]$ the free abelian group generated by irreducible components of $\Spec R_\rhobar^{\mathrm{alg}}$.
We use the terms \emph{cycle}, \emph{irreducible}, and \emph{effective} in this context as well.
Thinking of $\Z[\cX_{n,\mathrm{red}}]$ and $\Z[\Spec R_\rhobar^{\mathrm{alg}}]$ as spaces of functions on sets of irreducible components, pullback gives a map $i_{\rhobar}^*: \Z[\cX_{n,\mathrm{red}}] \ra \Z[\Spec R_\rhobar^{\mathrm{alg}}]$.
Let $\cZ_{\lambda,\tau}(\rhobar)$ denote the cycle $i_{\rhobar}^*(\cZ_{\lambda,\tau})\in \Z[\Spec R_\rhobar^{\mathrm{alg}}]$ which is the cycle corresponding to $\Spec R_{\rhobar,\F}^{\mathrm{ver},\lambda+\eta,\tau}$ using that taking formal fibers preserves multiplicities (see \cite[\href{https://stacks.math.columbia.edu/tag/0DRD}{Tag 0DRD}]{stacks-project}). 
(We suppress in the notation $\cZ_{\lambda,\tau}(\rhobar)$ the dependence on the choice of versal ring.) 

\begin{conj}[Versal Breuil--M\'ezard conjecture]\label{conj:v-S-BM}
Let $\cS$ be a set of extremal types.
For each $\sigma \in \JH(\ovl{\sigma}(\cS))$, there exist effective cycles $\cZ_\sigma(\rhobar)$ in $\Spec R_\rhobar^{\mathrm{alg}}$ such that for all $(\lambda+\eta,\tau) \in \cS$, we have
\[
\cZ_{\lambda,\tau}(\rhobar) = \sum_{\sigma} [\ovl{\sigma}(\lambda,\tau):\sigma] \cZ_\sigma(\rhobar).
\]
\end{conj}

\begin{rmk}
\begin{enumerate}
\item As stated, Conjecture \ref{conj:v-S-BM} depends on the choice of a versal ring.
However, by choosing a common formally smooth covering of any two versal rings and using that a formally smooth covering of an equidimensional scheme induces a bijection between sets of irreducible components and preserves multiplicities of components, we see that Conjecture \ref{conj:v-S-BM} for one choice of versal ring implies the same result for any other choice.

\item Taking $\cS$ to contain all minimal types ($\tau$ is minimal) and $R_\rhobar^{\mathrm{ver}}$ to be the framed deformation ring $R_\rhobar^\square$ recovers \cite[Conjecture 4.2.1]{EG}, with the added assertion that $\cZ_\sigma(\rhobar)$ is effective.
\end{enumerate}
\end{rmk}

\subsubsection{Cycles from modules}
If $M$ is a finitely generated $R_{\rhobar,\F}^{\mathrm{ver},\lambda+\eta,\tau}$-module for some type $(\lambda+\eta,\tau)$, then we will let $Z(M)$ be the cycle
\[
\sum_{\cC} \mu_{\cC}(M) \cC \in \Z[\Spec R_\rhobar^{\mathrm{alg}}]
\]
where $\cC$ ranges over irreducible components of $\Spec R_{\rhobar,\F}^{\mathrm{ver},\lambda+\eta,\tau}$, $\mu_{\cC}(M)$ denotes $\mathrm{length}_{R_{\rhobar,\F,\fp_{\cC}}^{\mathrm{ver},\lambda+\eta,\tau}}(M_{\fp_{\cC}})$, and $\fp_{\cC}\subset R_{\rhobar,\F}^{\mathrm{ver},\lambda+\eta,\tau}$ denotes the prime ideal corresponding to $\cC$.

\subsection{Relations between the two conjectures}
\label{sub:rel:BM}

\begin{prop}\label{prop:BMlocalize}
Let $\cS$ be a set of extremal types.
Then Conjecture \ref{conj:S-BM} (for $\cS$) implies Conjecture \ref{conj:v-S-BM} (for $\cS$) for all $\rhobar \in \cX_n(\F)$.
\end{prop}
\begin{proof}
This follows from the fact that multiplicities of cycles do not change upon passing to versal rings (\cite[\href{https://stacks.math.columbia.edu/tag/0DRD}{Tag 0DRD}]{stacks-project}).
\end{proof}

In fact, the converse of this statement is true (see Remark \ref{rmk:Pset}), but we will need a variation of it.
Let $\cP \subset \cX_n(\F)$ be a subset.
Let 
\[
i_{\cP}^* \defeq \prod_{x\in \cP} i_x^*: \Z[\cX_{n,\mathrm{red}}] \ra \prod_{x\in \cP} \Z[\Spec R_x^{\mathrm{alg}}].
\]

For a set $\cS$ of extremal types, let $\cC(\cS)$ denote the set of irreducible components of $\cX_{n,\mathrm{red}}$ which lie in the support of $\cX^{\lambda+\eta,\tau}$ for some $(\lambda+\eta,\tau) \in \cS$.
We say that $\cP$ \emph{meets all components of $\cS$} if any $\cC \in \cC(\cS)$ intersects $\cP$.

\begin{lemma}\label{lemma:localizeinj}
If $\cP$ meets all components of a set $\cS$ of extremal types, then the restriction of $i_{\cP}^*$ to $\Z[\cC(\cS)]$, the $\Z$-span of $\cC(\cS)$, is injective.
Moreover, $Z\in \Z[\cC(\cS)]$ is effective if and only if $i_{\cP}^*(Z)$ is effective $($i.e.~$i_x^*(Z)$ is effective for all $x\in\cP$$)$. %
\end{lemma}
\begin{proof}
If $\cC$ is in the support of $Z$, a nonzero element of the $\Z$-span of $\cC(\cS)$, then $\cC$ contains some $x \in \cP$.
Then $i_x^*(Z)$ is nonzero by \cite[\href{https://stacks.math.columbia.edu/tag/0DRD}{Tag 0DRD}]{stacks-project}.
Similarly, if the coefficient of $\cC$ is negative, then $i_x^*(Z)$ is not effective.
\end{proof}

\begin{prop}\label{prop:BMglobalize}
Let $\cS$ be a set of extremal types and $\cP \subset\cX_n(\F)$ be a subset.
Assume the following: 
\begin{enumerate} 
\item 
\label{it:BMglobalize:1}
$\cP \subset\cX_n(\F)$ meets all components of $\cS$; and
\item 
\label{it:BMglobalize:2}
Conjecture \ref{conj:v-S-BM} holds for $\cS$ and all $x\in \cP$ with cycles $\cZ_\sigma(x)$ for each $\sigma \in \JH(\ovl{\sigma}(\cS))$ and $x\in \cP$.
\end{enumerate}
Then for each $\sigma \in \JH(\ovl{\sigma}(\cS))$, there is at most one cycle $\cZ_\sigma$ in $\Z[\cX_{n,\mathrm{red}}]$ %
such that the support of $\cZ_\sigma$ is contained in the support of $\cZ_{\lambda,\tau}$ for some $(\lambda+\eta,\tau) \in \cS$ and for each $x\in \cP$, $i_x^*(\cZ_\sigma) = \cZ_\sigma(x)$.
If all such cycles exist, then they satisfy the conclusion of Conjecture \ref{conj:S-BM}.
\end{prop}
\begin{proof}
The uniqueness of the cycles $\cZ_\sigma$ follows from Lemma \ref{lemma:localizeinj}.

We next show that the cycles $\cZ_\sigma$ are effective. 
Indeed, if an irreducible cycle $\cC$ in $\cZ_\sigma$ has a negative coefficient, then there is a point $x\in \cP$ with $x\in \cC$ since $\cP$ meets all components of $\cS$.
Then $i_x^*(\cZ_\sigma)$ is not effective by \cite[\href{https://stacks.math.columbia.edu/tag/0DRD}{Tag 0DRD}]{stacks-project}, which is a contradiction.

From the existence of the cycles, the equality $\cZ_{\lambda,\tau} = \sum_{\sigma} [\ovl{\sigma}(\lambda,\tau):\sigma] \cZ_\sigma$ follows again from Lemma \ref{lemma:localizeinj}.
\end{proof}

\begin{rmk}\label{rmk:Pset}
We recall a strong form of the converse of Proposition \ref{prop:BMlocalize} from \cite[\S 8.3]{EGstack}.
Let $\cS$ be a set of extremal types.
The set $\cP = \{x_{\cC}\}_{\cC}$, where $\cC$ ranges over irreducible components of $\cup_{(\lambda+\eta,\tau)\in \cS}\cX_{\F}^{\lambda+\eta,\tau}$ and $x_{\cC} \in \cC$ is a smooth point of $\cX_{n,\mathrm{red}}$, meets all components of $\cS$.
Moreover, one can define $\cZ_\sigma$ as in \cite[\S 8.3]{EGstack}, so that this collection satisfies the hypothesis of Proposition \ref{prop:BMglobalize}.
In fact, $i_{\cP}^*$ induces an isomorphism on (top-dimensional) cycles.
We conclude that the following strong form of the converse of Proposition \ref{prop:BMlocalize} holds: if Conjecture \ref{conj:v-S-BM} holds at all points in $\cP$, then Conjecture \ref{conj:S-BM} holds.

It is natural to ask if a strong form of the converse holds for any set $\cP$ which meets all components of $\cS$.
In general, it is not as easy to construct the collection of cycles $\cZ_\sigma$ satisfying the properties in Proposition \ref{prop:BMglobalize}.
In \S \ref{sec:patchBM} we show how to use a minimal patching functor to construct $\cZ_\sigma$ so that Conjecture \ref{conj:S-BM} holds for a subset of $\cS$.
In \S \ref{sec:BMgen}, we will take $\cP$ to be the set of semisimple $\rhobar$.
\end{rmk}

\subsection{Patching functors and Breuil--M\'ezard cycles} \label{sec:patchBM}

In this section, we provide an axiomatic framework to show how patching functors (\S \ref{sec:patch:ax}) can be used to deduce versions of Conjectures \ref{conj:S-BM} and \ref{conj:v-S-BM}.
{The idea is to define the cycles $\cZ_\sigma$ in Conjecture \ref{conj:S-BM} by formally inverting the Breuil--M\'ezard equations and ignoring suitably non-generic components. Then one can prove Conjecture \ref{conj:S-BM} using Lemma \ref{lemma:localizeinj} assuming Conjecture \ref{conj:v-S-BM} for each $\rhobar$ in a large enough set $\cP$ (in particular $\cP$ must meet all components of $\cS$). Conjecture \ref{conj:v-S-BM} holds given the existence of a suitable minimal patching functor. The argument requires some intricate definitions, and the reader is invited to consider the context of \S \ref{sec:BMgen}.}

\begin{defn}\label{defn:irrelevant}
If $\rhobar: G_K \ra \GL_n(\F)$ is a continuous representation and $\cS_{\mathrm{elim}}$ is a set of extremal types, 
we say that a Serre weight $\sigma$ is \emph{$(\rhobar,\cS_{\mathrm{elim}})$-irrelevant} if there exists $(\lambda+\eta,\tau) \in \cS_{\mathrm{elim}}$ such that 
\begin{enumerate}
\item 
\label{it:irrelevant:2}
$\sigma \in \JH(\ovl{\sigma}(\lambda,\tau))$; and 
\item 
\label{it:irrelevant:3}
$R_\rhobar^{\lambda+\eta,\preceq\tau} = 0$.
\end{enumerate}
For a set $\cP \subset \cX_n(\F)$, we say that a Serre weight $\sigma$ is $(\cP,\cS_{\mathrm{elim}})$-irrelevant if $\sigma$ is $(\rhobar,\cS_{\mathrm{elim}})$-irrelevant for every $\rhobar \in \cP$.
We say $\sigma$ is $\rhobar$-irrelevant or $\cP$-irrelevant if $\sigma$ is $(\rhobar,\cS_{\mathrm{elim}})$-irrelevant or $(\cP,\cS_{\mathrm{elim}})$-irrelevant for some $\cS_{\mathrm{elim}}$.
\end{defn}

\begin{rmk} \label{rmk:irrelevant}
The significance of Definition \ref{defn:irrelevant} comes from the fact that if $M_\infty$ is a weak patching functor for $\rhobar: G_K \ra \GL_n(\F)$, then $M_\infty(\sigma) = 0$ for all $\rhobar$-irrelevant Serre weights $\sigma$ since $R_\rhobar^{\lambda+\eta,\preceq\tau} = 0$ implies that $M_\infty(\sigma^\circ(\lambda,\tau))=0$ for any $\cO$-lattice $\sigma^\circ(\lambda,\tau) \subset \sigma(\lambda,\tau)$.
Similarly, if Conjecture \ref{conj:v-S-BM} holds for $\cS_{\mathrm{elim}}$, then $\cZ_\sigma(\rhobar) = 0$ for all $(\rhobar,\cS_{\mathrm{elim}})$-irrelevant Serre weights $\sigma$. %
\end{rmk}

\begin{defn}
\label{defn:sets:types}
Below, $\cS$, $\cS'$, $\widehat{\cS}$, and $\widehat{\cS}_{\mathrm{elim}}$ denote sets of extremal types.
\begin{enumerate}
\item 
\label{it:sets:types:1}
We say that $\sigma \in \JH(\ovl{\sigma}(\cS))$ $\cS$-\emph{covers} $\sigma'$ if $(\lambda+\eta,\tau) \in \cS$ and $\sigma \in \JH(\ovl{\sigma}(\lambda,\tau))$ imply that $\cC_{\sigma'}$ lies in $\cX_{\F}^{\lambda+\eta,\tau}$.
\item 
\label{it:sets:types:2}
We say that $\sigma\in \JH(\ovl{\sigma}(\cS))$ is $\cS$-\emph{disjoint} from $\cS'$ if for all Serre weights $\kappa$ such that $\sigma$ $\cS$-covers $\kappa$, $\cC_{\kappa}$ does not lie in $\cX_{\F}^{\lambda+\eta,\tau}$  for any $(\lambda+\eta,\tau) \in \cS'$. 
\item 
\label{it:sets:types:3}
If $\cS \subset \widehat{\cS}$, we say that $\sigma\in \JH(\ovl{\sigma}(\cS))$ is $(\cS,\widehat{\cS})$-\emph{generic} if $\sigma$ is $\widehat{\cS}$-disjoint from $\widehat{\cS} \setminus \cS$.
\end{enumerate}
As before, we let $\cP \subset \cX_n(\F)$ be a subset.
\begin{enumerate}\setcounter{enumi}{3}
\item
\label{it:sets:types:4}
We say that $\widehat{\cS}$ is a $(\cP,\widehat{\cS}_{\mathrm{elim}})$-Breuil--M\'ezard system if for any Serre weight $\sigma$ 
there is a nonzero integer $d_\sigma$ and integers $n_{\lambda,\tau}^\sigma$ such that \[d_\sigma[\sigma] - \sum_{(\lambda+\eta,\tau)\in \widehat{\cS}} n_{\lambda,\tau}^\sigma [\ovl{\sigma}(\lambda,\tau)]\] is supported only at $(\cP,\widehat{\cS}_{\mathrm{elim}})$-irrelevant Serre weights. 
As before, we say that $\cS$ is a $\cP$-Breuil--M\'ezard system if $\widehat{\cS}$ is a $(\cP,\widehat{\cS}_{\mathrm{elim}})$-Breuil--M\'ezard system for some $\widehat{\cS}_{\mathrm{elim}}$.
\item 
\label{it:sets:types:5}
If $\widehat{\cS}$ is a $\cP$-Breuil--M\'ezard system and $\cS \subset \widehat{\cS}$, then we let $\cS_{\cP}\subset \cS$ denote the subset of types $(\lambda+\eta,\tau)$ such that $\JH(\ovl{\sigma}(\lambda,\tau))$ contains only $(\cS,\widehat{\cS})$-generic Serre weights.
(We suppress here the dependence of $\cS_{\cP}$ on $\widehat{\cS}$.)
\end{enumerate}
\end{defn}

\begin{rmk}
It is not clear \emph{a priori} that given a set of extremal types $\cS$, a Serre weight $\cS$-covers itself, {i.e.~$\cC_{\sigma} \leq \cZ_{\lambda,\tau}$ whenever $\sigma \in \JH(\ovl{\sigma}(\lambda,\tau))$,} though we expect this to be true, as would follow from the strengthening of Conjecture \ref{conj:S-BM} in Remark \ref{rmk:ZgeqC}.
Indeed, it will be true in some contexts that we consider in \S \ref{sec:BMgen} (see Proposition \ref{prop:BMcoeff}).
\end{rmk}

\begin{thm}\label{thm:patchBM}
\begin{enumerate}
\item 
\label{it:patchBM:1}
Let $\cP \subset\cX_n(\F)$ and let $\cS$ be a set of extremal types $(\lambda+\eta,\tau)$.
Suppose that for each $x\in \cP$, there exists a minimal patching functor $M_\infty^x$ for $x$ and $\cS$.
Then Conjecture \ref{conj:v-S-BM} holds for each $x\in \cP$ with $\cZ_\sigma(x) \defeq Z(M_\infty^x(\sigma))$.
\item 
\label{it:patchBM:2}
Suppose further that $\cP$ meets all components of $\cS$, and that $\widehat{\cS}$ is a $\cP$-Breuil--M\'ezard system containing $\cS$.
Then for each $(\cS,\widehat{\cS})$-generic $\sigma$, there exists a unique cycle $\cZ_\sigma$ in $\Z[\cX_{n,\mathrm{red}}]$ such that the support of $\cZ_\sigma$ is contained in the support of $\cZ_{\lambda,\tau}$ for some $(\lambda+\eta,\tau) \in \cS$ and for each $x\in \cP$, $i_x^*(\cZ_\sigma) = \cZ_\sigma(x)$.
Moreover, $\cZ_\sigma$ is effective.
In particular, Conjecture \ref{conj:S-BM} holds for $\cS_{\cP}$.
\item 
\label{it:patchBM:3}
For each $(\cS,\widehat{\cS})$-generic $\sigma$, the cycle $\cZ_\sigma$ does not depend on the choice of the patching functors $M_\infty^x$ for $x \in \cP$.
In particular, the cycle $Z(M_\infty^x(\sigma))$ in $R_x^{\mathrm{alg}}$ depends only on the versal ring $R_\infty^x$ of $\cX_n$ at $x$ \emph{(}and not on other data in $M_\infty^x$\emph{)}.
\item
\label{it:patchBM:4}
If furthermore there is a Breuil--M\'ezard system $\widehat{\cS}_{\mathrm{elim}}$ containing $\widehat{\cS}$ such that $\widehat{\cS}$ is a $(\cP,\widehat{\cS}_{\mathrm{elim}})$-Breuil--M\'ezard system and Conjecture \ref{conj:S-BM} holds for $\widehat{\cS}_{\mathrm{elim}}$, then the above cycles $\cZ_\sigma$ coincide with those in Conjecture \ref{conj:S-BM}. 
\end{enumerate}
\end{thm}
\begin{proof}
For item (\ref{it:patchBM:1}), we can assume that $\cP$ contains a single element $\rhobar$.
The data of a minimal patching functor for $\rhobar$ provides a choice of versal ring $R_\infty = R_\rhobar^\square \widehat{\otimes}_{\cO} R^p$ for $\cX_n$ at $\rhobar$ as $R^p$ is a formally smooth $\cO$-algebra.
We need to show that 
\[
\cZ_{\lambda,\tau}(\rhobar) = \sum_\sigma [\ovl{\sigma}(\lambda,\tau):\sigma] Z(M_\infty(\sigma)).
\]
For each $(\lambda+\eta,\tau)\in \cS$, $Z(M_\infty(\ovl{\sigma}(\lambda,\tau))) = \cZ_{\lambda,\tau}(\rhobar)$ by \cite[Lemma 2.2.10]{EG}.
On the other hand, 
\[
Z(M_\infty(\ovl{\sigma}(\lambda,\tau))) = \sum_\sigma [\ovl{\sigma}(\lambda,\tau):\sigma] Z(M_\infty(\sigma))
\]
by \cite[Lemma 2.2.7]{EG}.

We now proceed to items (\ref{it:patchBM:2}) and (\ref{it:patchBM:3}).
We first define $\cZ_\sigma$ for every $(\cS,\widehat{\cS})$-generic Serre weight $\sigma$.
For such a $\sigma$, {we can find $d_\sigma$ and $n_{\lambda,\tau}^\sigma$ as in Definition \ref{defn:sets:types}(\ref{it:sets:types:4})}.
Let $\tr_{\sigma,\cS}$ denote the idempotent endomorphism of $\Z[\cX_{n,\mathrm{red}}]$ which maps $\cC_{\sigma'}$ to itself if $\sigma$ $\widehat{\cS}$-covers $\sigma'$ and to $0$ otherwise. %
We let 
\[
\cZ_\sigma \defeq \tr_{\sigma,\cS}\Big(\frac{1}{d_\sigma} \sum_{(\lambda+\eta,\tau)\in \widehat{\cS}} n_{\lambda,\tau}^\sigma \cZ_{\lambda,\tau}\Big),
\]
which is \emph{a priori} a cycle with rational coefficients.
We will show that $i_x^*(\cZ_\sigma) = \cZ_\sigma(x)$ for all $x\in \cP$, which also implies that $\cZ_\sigma$ is a cycle with integer coefficients by \cite[\href{https://stacks.math.columbia.edu/tag/0DRD}{Tag 0DRD}]{stacks-project}.
Uniqueness and effectivity in (\ref{it:patchBM:2}) follows as in the proof of Proposition \ref{prop:BMglobalize}.
Item (\ref{it:patchBM:3}) follows from the fact that the definition of $\cZ_\sigma$ does not depend on the choices of $M_\infty^x$ for $x \in \cP$.

We need the following lemma, which follows from definitions.

\begin{lemma}\label{lemma:disjoint}
If $(\lambda+\eta,\tau) \in \widehat{\cS}\setminus \cS$, then $\tr_{\sigma,\cS} (\cZ_{\lambda,\tau}) = 0$.
\end{lemma}

Fix an element $x\in \cP$.
Let $\tr_{\sigma,\cS}(x)$ be the idempotent endomorphism of $\Z[\Spec R_x^{\mathrm{alg}}]$ such that $i_x^* \circ \tr_{\sigma,\cS} = \tr_{\sigma,\cS}(x) \circ i_x^*$ which exists and is unique by \cite[\href{https://stacks.math.columbia.edu/tag/0DRB}{Tag 0DRB},\href{https://stacks.math.columbia.edu/tag/0DRD}{Tag 0DRD}]{stacks-project}.

\begin{lemma}\label{lemma:coveridem}
We have $\tr_{\sigma,\cS}(x) (\cZ_\sigma(x)) = \cZ_\sigma(x)$.
\end{lemma}
\begin{proof}
Suppose that $\cC_{\sigma'}(x)$ is an irreducible cycle in the support of $\cZ_\sigma(x)$, which is also in the support of $i_x^*(\cC_{\sigma'})$ for some Serre weight $\sigma'$.
Then for any $(\lambda+\eta,\tau)\in \widehat{\cS}$ such that $\sigma \in \JH(\ovl{\sigma}(\lambda,\tau))$ we have that $\cC_{\sigma'}(x) \leq \cZ_\sigma(x) \leq \cZ_{\lambda,\tau}(x)$ (by (\ref{it:patchBM:1})), which implies that $\cC_{\sigma'} \leq \cZ_{\lambda,\tau}$.
This means that $\sigma$ $\widehat{\cS}$-covers $\sigma'$.
\end{proof}

Note that for any $(\lambda+\eta,\tau) \in \widehat{\cS} \setminus \cS$, $\tr_{\sigma,\cS}(x) (Z(M_\infty^x(\ovl{\sigma}(\lambda,\tau))))=0$ since $\tr_{\sigma,\cS}(x)(\cZ_{\lambda,\tau}(x)) = 0$ by Lemma \ref{lemma:disjoint} and $Z(M_\infty^x(\ovl{\sigma}(\lambda,\tau)))) \leq \cZ_{\lambda,\tau}(x)$.
Then
\begin{align*}
\cZ_\sigma(x) &= Z(M_\infty^x(\sigma)) \\
&= \frac{1}{d_\sigma}\sum_{(\lambda+\eta,\tau) \in \widehat{\cS}} n_{\lambda,\tau}^\sigma Z(M_\infty^x(\ovl{\sigma}(\lambda,\tau))) \\
&= \frac{1}{d_\sigma}\sum_{(\lambda+\eta,\tau) \in \widehat{\cS}} n_{\lambda,\tau}^\sigma \tr_{\sigma,\cS}(x) (Z(M_\infty^x(\ovl{\sigma}(\lambda,\tau)))) \\
&= \frac{1}{d_\sigma}\sum_{(\lambda+\eta,\tau) \in \cS} n_{\lambda,\tau}^\sigma \tr_{\sigma,\cS}(x) (Z(M_\infty^x(\ovl{\sigma}(\lambda,\tau)))) \\
&= \frac{1}{d_\sigma}\sum_{(\lambda+\eta,\tau) \in \cS} n_{\lambda,\tau}^\sigma \tr_{\sigma,\cS}(x) (\cZ_{\lambda,\tau}(x)) \\
&= i_x^*(\cZ_\sigma),
\end{align*}
where the first equality is by definition, the second equality follows from Remark \ref{rmk:irrelevant}, the third equality follows from Lemma \ref{lemma:coveridem}, the fourth equality follows from the previous sentence, the fifth equality is as in the first paragraph of the proof, and the final equality is by definition of $\cZ_\sigma$ and $\tr_{\sigma,\cS}$.

Finally, we turn to (\ref{it:patchBM:4}).
Suppose that $\widehat{\cS}_{\mathrm{elim}}$ is a Breuil--M\'ezard system containing $\widehat{\cS}$ such that $\widehat{\cS}$ is a $(\cP,\widehat{\cS}_{\mathrm{elim}})$-Breuil--M\'ezard system and Conjecture \ref{conj:S-BM} holds for $\widehat{\cS}_{\mathrm{elim}}$ with cycles $\cZ_\sigma^{\mathrm{BM}}$.
We will show that for a $(\cS,\widehat{\cS})$-generic $\sigma$, the cycle $\cZ_\sigma$ coincides with $\cZ_\sigma^{\mathrm{BM}}$.
Suppose that
\[
\sigma_{\mathrm{irr}} \defeq [\sigma] - \frac{1}{d_\sigma}\sum_{(\lambda+\eta,\tau)\in \cS} n_{\lambda,\tau}^\sigma [\ovl{\sigma}(\lambda,\tau)]
\]
is supported only at $(\cP,\widehat{\cS}_{\mathrm{elim}})$-irrelevant weights, 
and define
\[
\cZ_{\sigma_{\mathrm{irr}}}^{\mathrm{BM}} \defeq \cZ_\sigma^{\mathrm{BM}} - \frac{1}{d_\sigma}\sum_{(\lambda+\eta,\tau)\in \cS} n_{\lambda,\tau}^\sigma \cZ_{\lambda,\tau},
\]
which is a rational linear combination of cycles $\cZ_{\kappa}^{\mathrm{BM}}$ for $(\cP,\widehat{\cS}_{\mathrm{elim}})$-irrelevant weights $\kappa$.
We claim that
\begin{itemize}
\item
$\tr_{\sigma,\cS}(\cZ_\sigma^{\mathrm{BM}}) = \cZ_\sigma^{\mathrm{BM}}$ and
\item 
$\tr_{\sigma,\cS}(\cZ_{\sigma_{\mathrm{irr}}}^{\mathrm{BM}}) = 0$.
\end{itemize}
Then 
\begin{align*}
\cZ_\sigma^{\mathrm{BM}} &= \tr_{\sigma,\cS}(\cZ_\sigma^{\mathrm{BM}}) \\
&\defeq \tr_{\sigma,\cS}\Big(\cZ_{\sigma_{\mathrm{irr}}}^{\mathrm{BM}}+\frac{1}{d_\sigma} \sum_{(\lambda+\eta,\tau)\in \widehat{\cS}} n_{\lambda,\tau}^\sigma \cZ_{\lambda,\tau}\Big) \\
&= \tr_{\sigma,\cS}\Big(\frac{1}{d_\sigma} \sum_{(\lambda+\eta,\tau)\in \widehat{\cS}} n_{\lambda,\tau}^\sigma \cZ_{\lambda,\tau}\Big) \\
&\defeq \cZ_\sigma
\end{align*}
where the first and third equalities correspond to the above claims. %
Turning to the claims, the first follows from the proof of Lemma \ref{lemma:coveridem}.
To show the second claim, by linearity we assume without loss of generality that $\sigma_{\mathrm{irr}}$ is a $(\cP,\widehat{\cS}_{\mathrm{elim}})$-irrelevant Serre weight.
Since $\sigma_{\mathrm{irr}}$ is $(\cP,\widehat{\cS}_{\mathrm{elim}})$-irrelevant, for each $x\in \cP$, there exists $(\lambda_x+\eta,\tau_x) \in \widehat{\cS}_{\mathrm{elim}}$ such that $\sigma_{\mathrm{irr}} \in \JH(\ovl{\sigma}(\lambda_x,\tau_x))$ and the support of $\cZ_{\lambda_x,\tau_x}$ does not contain $x$.
Then the support of $\cZ_{\sigma_{\mathrm{irr}}}^{\mathrm{BM}}$, which is less than or equal to $\cZ_{\lambda_x,\tau_x}$ for all $x\in \cP$, contains no elements of $\cP$.
On the other hand, suppose that $\sigma \in \JH(\ovl{\sigma}(\lambda,\tau))$ for $(\lambda+\eta,\tau)\in \cS$.
Since $\cP$ meets all components of $\cS$, the support of $\cZ_{\lambda,\tau}$ (the set of irreducible components) must be disjoint from that of $\cZ_{\sigma_{\mathrm{irr}}}$.
We conclude that $\sigma$ cannot $\widehat{\cS}$-cover any weights corresponding to components in the support of $\cZ_{\sigma_{\mathrm{irr}}}$ so that $\tr_{\sigma,\cS}(\cZ_{\sigma_{\mathrm{irr}}}^{\mathrm{BM}}) = 0$.
\end{proof}

\subsection{Geometric Breuil--M\'ezard for generic tamely potentially crystalline types} \label{sec:BMgen}

We apply the results of the previous section to a context in which we have enough patching functors.
The section begins with a series of lemmas that establish the requisite hypotheses.

Let $\Lambda \subset X_*(\un{T}^{\vee})$ be a finite set of dominant weights containing $0$, and let $\cS_{\Lambda,\mathrm{t}}$ denote the union of the set of extremal types $(\lambda'+\eta,\tau)$ where $\lambda'\leq \lambda$ for some $\lambda \in \Lambda$ {and is dominant} and $\tau$ is a $P_{\lambda+\eta,e}$-generic and $(6n-2+h_{\lambda+\eta})$-generic tame inertial type. %
Let $\cP_{\mathrm{ss}}$ be the set of $x\in \cX_n(\F)$ such that $x|_{I_K}$ is a $(6n-2)$-generic tame inertial $\F$-type for $K$. 
Let $\widehat{\cS}_{\Lambda,\mathrm{t}}$ be the union of $\cS_{\Lambda,\mathrm{t}}$ and the set of types $(\eta,\tau)$ where $\tau$ is a $2n$-generic tame inertial type for $K$.
Let $\widehat{\cS}_{\Lambda,\mathrm{t,elim}}$ be the union of $\cS_{\Lambda,\mathrm{t}}$ and the set of extremal types $(\eta,\tau)$ where $\rho_\tau$ is tame.
\begin{rmk}
In what follows, we could replace $\cP_{\mathrm{ss}}$ by any set $\cP \subset \cX_n(\F)$ so that $\{x|_{I_K}\mid x\in \cP\}$ is the set of $(6n-2)$-generic tame inertial $\F$-types for $K$.
\end{rmk}

\begin{lemma}\label{lemma:W?irrelevant}
Suppose that $\rhobar\in \cX_n(\F)$ is such that $\rhobar|_{I_K}$ is a $(6n-2)$-generic tame inertial $\F$-type for $K$. 
If $\sigma \notin W^?(x|_{I_K})$, then $\sigma$ is $(\rhobar,\widehat{\cS}_{\Lambda,\mathrm{t,elim}})$-irrelevant.
\end{lemma}
\begin{proof}
This follows immediately from the proof of \cite[Corollary 4.2.4]{LLL}.
\end{proof}

\begin{rmk}
We use Lemma \ref{lemma:W?irrelevant} to apply Theorem \ref{thm:patchBM} in the setting of this section.
However, examining the proof of Theorem \ref{thm:patchBM}, we just need that the evaluations of patching functors applied to $\cP_{\mathrm{ss}}$-irrelevant weights vanish, for which Proposition \ref{prop:WE} suffices.
\end{rmk}

\begin{lemma}
The set $\widehat{\cS}_{\Lambda,\mathrm{t}}$ is a $(\cP_{\mathrm{ss}},\widehat{\cS}_{\Lambda,\mathrm{t,elim}})$-Breuil--M\'ezard system.
\end{lemma}
\begin{proof}
Given a Serre weight $\sigma$, we can write $[\sigma] = \sum_R n^\sigma_R[\ovl{R}]$ in the Grothendieck group by \cite[Theorem 33]{serre-book}, where $R$ runs over irreducible $\rG$-representations over $E$.
If $\tau$ is a $2n$-generic tame type, we let $n_{\un{0},\tau}^\sigma$ be $n^\sigma_{\sigma(\tau)}$.
We otherwise let $n_{\lambda,\tau}^\sigma \defeq 0$ for $(\lambda+\eta,\tau) \in \widehat{\cS}_{\Lambda,\mathrm{t}}$.
Since each such $R$ above is a Jordan--H\"older factor of a Deligne--Lusztig representation $R_s(\mu)$ by \cite[Corollary 7.7]{DeligneLusztig}, if a Serre weight is in the support of
\[
[\sigma] - \sum_{(\lambda+\eta,\tau) \in \widehat{\cS}_{\Lambda,\mathrm{t}}} n^\sigma_{\sigma(\tau)}[\ovl{\sigma}(\tau)],
\]
then it is contained in $\JH(\ovl{R})$ for some Deligne--Lusztig representation $R$ which is not $2n$-generic.
By Lemma \ref{lemma:nongenericDL}, such Serre weights are not $(4n-2)$-deep, and so not in $W^?(x|_{I_K})$ for any $x \in \cP_{\mathrm{ss}}$ by Proposition \ref{prop:W?}.
\end{proof}

\begin{lemma}\label{lemma:coveringimply}
If $\sigma$ is $3n-1$-deep and $\widehat{\cS}_{\Lambda,\mathrm{t}}$-covers $\sigma'$, then $\sigma$ covers $\sigma'$ (in the sense of Definition \ref{defn:cover}).
\end{lemma}
\begin{proof}
Suppose that $\sigma$ $\widehat{\cS}_{\Lambda,\mathrm{t}}$-covers $\sigma'$. 
Any $(2n-2)$-generic tame inertial type $\tau$ for $K$ with $\sigma \in \JH(\ovl{\sigma}(\tau))$ must be $2n$-generic by Proposition \ref{prop:JHbij}, so that $\cC_{\sigma'}$ is contained in $\cX^{\eta,\tau}_{\F}$ by assumption. 
Remark \ref{rmk:MLM:sp:fib}(\ref{it:MLM:sp:fib:3}) implies that $\sigma'\in \JH(\ovl{\sigma}(\tau))$. 
The conclusion follows. 
\end{proof}

\begin{defn}\label{defn:genSW}
We say that a Serre weight $\sigma$ is \emph{generic} if $\sigma\in \JH(\ovl{\sigma}(\tau))$ for some $(\eta,\tau) \in \cS_{{\{0\}},\mathrm{t}}$ and $\sigma$ does not cover any Serre weights in $\JH(\ovl{\sigma}(\tau'))$ for all $(\eta,\tau') \in \widehat{\cS}_{{\{0\}},\mathrm{t}} \setminus \cS_{{\{0\}},\mathrm{t}}$. 
\end{defn}

\begin{rmk}\label{rmk:genSW}
\begin{enumerate}
\item 
If $\sigma$ is generic, then $\sigma$ is necessarily $(6n-2)$-deep by Proposition \ref{prop:JHbij} and the fact that $\sigma\in \JH(\ovl{\sigma}(\tau))$ for some $(\eta,\tau) \in \cS_{{\{0\}},\mathrm{t}}$.
\item \label{item:gencover}
If $\sigma$ is generic and covers $\sigma'$, then $\sigma'$ is generic.
\end{enumerate}
\end{rmk}

\begin{lemma}\label{lemma:coverimply}
If $\sigma$ is generic, then $\sigma$ is $(\cS_{\Lambda,\mathrm{t}},\widehat{\cS}_{\Lambda,\mathrm{t}})$-generic (for any set $\Lambda$ as above).
\end{lemma}
\begin{proof}
Suppose that $\sigma$ $\widehat{\cS}_{\Lambda,\mathrm{t}}$-covers a Serre weight $\sigma'$ and that $\cC_{\sigma'}$ is contained in $\cX_{\F}^{\eta,\tau}$ for $(\eta,\tau)\in \widehat{\cS}_{\Lambda,\mathrm{t}}$. 
We need to show that $(\eta,\tau)\in \cS_{\Lambda,\mathrm{t}}$.
Lemma \ref{lemma:coveringimply} implies that $\sigma$ covers $\sigma'$. 
Remark \ref{rmk:MLM:sp:fib}(\ref{it:MLM:sp:fib:3}) implies that $\sigma'\in \JH(\ovl{\sigma}(\tau))$. 
Then the genericity of $\sigma$ implies that $(\eta,\tau)\in \cS_{\Lambda,\mathrm{t}}$. 
\end{proof}

\begin{lemma}
The set $\cP_{\mathrm{ss}}$ meets all components of $\cS_{\Lambda,\tau}$. Any tame $\rhobar\in \cX^{\lambda,\tau}(\F)$ where $(\lambda,\tau)\in \cS_{\Lambda,\tau}$ is $(6n-2)$-generic.
\end{lemma}
\begin{proof}
If $\cC_\sigma$ is a component of $\cX^{\lambda+\eta,\tau}$ for $(\lambda+\eta,\tau) \in \cS_{\Lambda,\tau}$, then $\sigma \in \JH(\ovl{\sigma}(\lambda,\tau))$ by Remark \ref{rmk:MLM:sp:fib}(\ref{it:MLM:sp:fib:3}).
Fix a $6n-2+h_{\lambda+\eta}$-generic lowest alcove presentation for $\tau$.
Then $(\tld{w},\tld{w}(\tau)\tld{w}_2^{-1}(0))$ is a $\lambda$-compatible lowest alcove presentation for $\sigma$ for some $\tld{w}_2 \in \tld{\un{W}}^+$ with $\tld{w}\uparrow t_\lambda \tld{w}_h^{-1} \tld{w}_2$ by Proposition \ref{prop:JHbij}.
Let $\rhobar: G_K\ra \GL_n(\F)$ be a semisimple continuous representation such that $\rhobar|_{I_K}$ has a lowest alcove presentation such that $\tld{w}(\rhobar|_{I_K}) = \tld{w}(\tau) \tld{w}_2^{-1} w \tld{w}$ for some $w\in \un{W}$.
This lowest alcove presentation is $(6n-2)$-generic so that $\rhobar \in \cP_{\mathrm{ss}}$.
Moreover, since $\tld{w}(\tau)\tld{w}_2^{-1}(0) = \tld{w}(\rhobar|_{I_K})\tld{w}^{-1}(0)$, $\sigma \in W_{\mathrm{obv}}(\rhobar|_{I_K})$.
Then $\rhobar \in \cC_\sigma$ by Proposition \ref{prop:rhobaroncomponents}(\ref{item:comp:obv}).
\end{proof}

Let $\cS_{\cP,\Lambda,\mathrm{t}}\subset \cS_{\Lambda,\mathrm{t}}$ be the subset consisting of types $(\lambda+\eta,\tau)$ such that $\JH(\ovl{\sigma}(\lambda,\tau))$ consists only of generic Serre weights $\sigma$. The following result is the main result of the section.

\begin{thm}\label{thm:genBM}
\begin{enumerate}
\item 
\label{it:genBM:1}
For any semisimple $(6n-2)$-generic $\rhobar$, a minimal patching functor $M_\infty$ for $\rhobar$ and $\cS_{\Lambda,\mathrm{t}}$ exists.
In particular, setting $\cZ_\sigma(\rhobar) \defeq Z(M_\infty(\sigma))$, 
\[
\cZ_{\lambda,\tau}(\rhobar) = \sum_{\sigma} [\ovl{\sigma}(\lambda,\tau):\sigma] \cZ_\sigma(\rhobar).
\]
for all $(\lambda+\eta,\tau) \in \cS_{\Lambda,\mathrm{t}}$.
\item 
\label{it:genBM:2}
For each $x\in \cP_{\mathrm{ss}}$, choose a minimal patching functor $M^x_\infty$ for $x$ and $\cS_{\Lambda,\mathrm{t}}$. Then, for each generic Serre weight $\sigma$, there exists a unique cycle $\cZ_\sigma$ in $\Z[\cX_{n,\mathrm{red}}]$ such that the support of $\cZ_\sigma$ is contained in the set $\{\cC_\kappa\mid\sigma \textrm{ covers }\kappa\}$ and for each $x\in \cP_{\mathrm{ss}}$, $i_x^*(\cZ_\sigma) = \cZ_\sigma(x)$, where $\cZ_\sigma(x)\defeq Z(M_\infty^x(\sigma))$.

\item 
\label{it:genBM:3}
For generic $\sigma$, the cycle $\cZ_\sigma$ does not depend on the choice of the patching functors $M_\infty^x$ for $x \in \cP_{\mathrm{ss}}$.
For generic $\sigma$ and semisimple $(6n-2)$-generic $\rhobar$ with minimal patching functor $M_\infty^\rhobar$ for $\rhobar$ and $\cS_{\Lambda,\mathrm{t}}$, $Z(M_\infty^{\rhobar}(\sigma))$ depends only on the versal ring $R_\infty$ (i.e., it is the pullback to $R_\infty$ of a cycle that is independent of $M_\infty^{\rhobar}$).
\item
\label{it:genBM:4}
Assume Conjecture \ref{conj:S-BM} holds for a Breuil-M\'ezard system containing $\widehat{\cS}_{\Lambda,\mathrm{t,elim}}$. 
Then the above cycles $\cZ_\sigma$ (for generic $\sigma$) coincide with those from Conjecture \ref{conj:S-BM}. 
\end{enumerate}
\end{thm}
\begin{proof}
We start with item (\ref{it:genBM:1}).
Let $M_\infty$ be a weak minimal detectable patching functor for $\rhobar$.
We claim that $M_\infty$ is a minimal patching functor for $\rhobar$ and $\cS_{\Lambda,\mathrm{t}}$.
If $(\lambda+\eta,\tau) \in \cS_{\Lambda,\mathrm{t}}$, then $R_\infty(\lambda,\tau)$ is a domain (or zero) by Theorem \ref{thm:stack_local_model}(\ref{it:stack_local_model:2}).
Moreover, $M_\infty(\sigma^\circ(\lambda,\tau))$ is nonzero if and only if $R_\infty(\lambda,\tau)$ is nonzero by Proposition \ref{prop:patchnonzero}.
These facts imply that $M_\infty(\sigma^\circ(\lambda,\tau))[1/p]$, which is locally free of rank at most one over $R_\infty(\lambda,\tau)[1/p]$, is locally free of rank one. 
This proves the first part.
Items (\ref{it:genBM:2}), (\ref{it:genBM:3}), and (\ref{it:genBM:4}) follow from Theorem \ref{thm:patchBM} (and the previous lemmas in this section).
The stronger conclusion that the support of $\cZ_\sigma$ is contained in the set $\{\cC_\kappa\mid\sigma \textrm{ covers }\kappa\}$ follows from the definition of $\cZ_\sigma$ and Lemma \ref{lemma:coveringimply}.
\end{proof}

\subsubsection{Breuil--M\'ezard with polynomial genericity} \label{sec:polygen}

Let $\tld{P}_{\eta,e}$ be the product of $P_{\eta,e}$ and $P_{7n-3}$ (see Theorem \ref{thm:stack_local_model}(\ref{it:stack_local_model:2}) and Remark \ref{rmk:P-gen}(\ref{it:P-gen:2})).
If $f(t_1,\ldots,t_n) \in \Z[t_1,\ldots,t_n]$ and $\omega \in X^*(T) \cong \Z^n$ is dominant, let 
\begin{equation}\label{eqn:polysuper}
f^\omega(t_1,\ldots,t_n) \defeq \prod_{\nu \in \Conv(\omega)} f(t_1-\nu_1,\ldots,t_n-\nu_n) \in \Z[t_1,\ldots,t_n].
\end{equation}
For a finite set $\Lambda \subset X_*(\un{T}^{\vee})$ of dominant weights, we let 
\[
P_{\cP,\Lambda,e} \defeq \prod_{\lambda\in \Lambda} \prod_{j\in \cJ} \tld{P}^{(\lambda + \eta -w_0(\eta))_j}_{\eta,e}.
\]

\begin{lemma}\label{lemma:univpolytau}
Let $\Lambda \subset X_*(\un{T}^{\vee})$ be a finite set of dominant weights containing $0$.
Let $\tau$ be a tame inertial type for $K$ with a lowest alcove presentation $(s,\mu-\eta)$ such that $\mu$ is $P_{\cP,\Lambda,e}$-generic.
Then $(\lambda+\eta,\tau) \in \cS_{\cP,\Lambda,\mathrm{t}}$ for any $\lambda\in \Lambda$. 
\end{lemma}
\begin{proof}
We need to show that for any $\lambda\in \Lambda$, any $\sigma \in \JH(\ovl{\sigma}(\lambda,\tau))$ is generic, i.e.~that if $\sigma$ covers $\sigma'$ and $\sigma' \in \JH(\ovl{\sigma}(\tau'))$ for some tame inertial type $\tau'$, then $\tau'$ has a lowest alcove presentation $(s',\mu'-\eta)$ such that $\mu'$ is $\tld{P}_{\eta,e}$-generic.
In fact, we take $(s',\mu'-\eta)$ to be compatible with $(s,\mu-\eta)$.
Since $\sigma$ covers $\sigma'$, $\sigma' \in\JH(\ovl{\sigma}(\tau))$ so that we can assume without loss of generality that $\sigma = \sigma'$.
Choose a tame inertial $\F$-type $\rhobar$ such that $\sigma \in W^?(\rhobar)$.
Then choosing the compatible lowest alcove presentation for $\rhobar$, we have that $\tld{w}(\rhobar,\tau) \in \Adm(\lambda+\eta)$ and $\tld{w}(\rhobar,\tau') \in \Adm(\eta)$.
Thus $\tld{w}(\tau)^{-1} \tld{w}(\tau') \in \Adm(\lambda+\eta-w_0(\eta))$, so that $\mu'-\mu \in \Conv(\lambda + \eta-w_0(\eta))$.
Then the $P_{\cP,\Lambda,e}$-genericity of $\mu$ implies the $\tld{P}_{\eta,e}$-genericity of $\mu'$.
\end{proof}

\begin{cor}\label{cor:PgenBM}
Let $\Lambda \subset X_*(\un{T}^{\vee})$ be a finite set of dominant weights containing $0$.
Then there exist effective cycles $\cZ_\sigma\in \Z[\cX_{n,\mathrm{red}}]$ for each Serre weight $\sigma$ such that 
\[
\cZ_{\lambda,\tau} = \sum_{\sigma} [\ovl{\sigma}(\lambda,\tau):\sigma] \cZ_\sigma
\]
for any $\lambda \in \Lambda$ and tame inertial type $\tau$ with a lowest alcove presentation $(s,\mu-\eta)$ with $\mu$ $P_{\cP,\Lambda,e}$-generic.
\end{cor}
\begin{proof}
This follows from Lemma \ref{lemma:univpolytau} and Theorem \ref{thm:genBM}.
\end{proof}

\begin{rmk}
If $\sigma$ is not generic, then $\sigma \notin \JH(\ovl{\sigma}(\lambda,\tau))$ for any $(\lambda+\eta,\tau) \in \cS_{\cP,\Lambda,\mathrm{t}}$. Hence any $\sigma$ such that $\cZ_\sigma$ occurs in Corollary \ref{cor:PgenBM} must be generic.
\end{rmk}

\begin{rmk}\label{rmk:BMproduct}
In this entire section, we have restricted ourselves to the case where $\cO_p$ is the ring of integers of a $p$-adic field $K$.
However, the evident generalization of Theorem \ref{thm:genBM} to the general case can be proven in the exact same way.
Moreover, since the completed tensor products of patching functors are again patching functors, the uniqueness statements in Theorem \ref{thm:genBM}(\ref{it:genBM:2}) and (\ref{it:genBM:3}) imply that the cycles $\cZ_\sigma$ have a product structure corresponding to that of $\cO_p$.
\end{rmk}

\subsection{Generic Breuil--M\'ezard for tamely potentially semistable deformation rings in small weight}\label{sec:BMminweight}

In this section, we prove the Breuil--M\'ezard conjecture for sufficiently generic Galois representations and the Breuil--M\'ezard system coming from tame inertial Weil--Deligne types and small regular weight.

\begin{lemma}\label{lemma:univpolyrhobar}
Let $\Lambda\subset X_*(\un{T}^\vee)$ be a finite subset of dominant weights, and let $(s,\mu-\eta)$ be a lowest alcove presentation of a tame inertial $\F$-type $\rhobar$ for $K$.
If $\mu_j$ is $P^\eta_{\cP,\Lambda,e}$-generic (see Lemma \ref{lemma:univpolytau}) for all $j\in \cJ$, then for any tame inertial type $\tau$ for $K$ with $\tld{w}(\rhobar,\tau) \in \Adm(\lambda+\eta)$ for $\lambda \in \Lambda$ \emph{(}and for some lowest alcove presentation for $\tau$\emph{)}, $(\lambda+\eta,\tau) \in \cS_{\cP,\Lambda,\mathrm{t}}$.
\end{lemma}
\begin{proof}
This follows from Lemma \ref{lemma:univpolytau} and a similar argument.
\end{proof}

For a finite subset $\Lambda \subset X^*(\un{T}) = X_*(\un{T}^\vee)$, let 
\[
h_\Lambda \defeq \max_{\lambda \in \Lambda,\alpha \in \Phi} \langle \lambda,\alpha^\vee \rangle.
\]
The following is a corollary of Theorem \ref{thm:genBM}(\ref{it:genBM:2}) and Proposition \ref{prop:BMlocalize}.

\begin{cor}\label{cor:BM}
Let $\Lambda\subset X_*(\un{T}^\vee)$ be a finite subset of dominant weights containing $0$.
Let $\rhobar: G_K \ra \GL_n(\F)$ be a continuous Galois representation such that $\rhobar^{\mathrm{ss}}|_{I_K}$ has a lowest alcove presentation $(s,\mu-\eta)$ where $\mu_j$ is $P^\eta_{\cP,\Lambda,e}P_m$-generic for all $j\in \cJ$ and $m = \max\{2h_\Lambda + 2h_\eta,6n-2\}$.
Then there exist cycles $\cZ_\sigma(\rhobar) \in \Z[\Spec R_\rhobar^{\mathrm{alg}}]$ for each Serre weight $\sigma$ such that 
\begin{equation}\label{eqn:minimalBM}
Z(R_{\rhobar,\F}^{\lambda+\eta,\tau}) = \sum_\sigma [\ovl{r(\tau)\otimes_E V(\lambda)}:\sigma] \cZ_\sigma(\rhobar)
\end{equation}
for all $\lambda\in \Lambda$ and tame inertial Weil--Deligne types $\tau$, where $r(\tau)$ is a virtual representation of $\GL_n(\cO_K)$ over $E$ defined in \cite[\S 4.2]{Shotton}. 
\end{cor}
\begin{proof}
If $\sigma$ is a generic Serre weight, then let $\cZ_\sigma(\rhobar) \defeq i_\rhobar^*(\cZ_\sigma)$ with $\cZ_\sigma$ as in Theorem \ref{thm:genBM}(\ref{it:genBM:2}) with the set $\Lambda$.
Otherwise, let $\cZ_\sigma(\rhobar) \defeq 0$.
Then (\ref{eqn:minimalBM}) for $\tau$ such that $N_\tau = 0$ and $(\lambda+\eta,\tau) \in \cS_{\cP,\Lambda,\mathrm{t}}$ holds by Theorem \ref{thm:genBM}(\ref{it:genBM:2}) and and Proposition \ref{prop:BMlocalize}.
Note that when $N_\tau = 0$, $r(\tau) = \sigma(\tau)$ (the semisimple case in \cite{Shotton}).

Fix $\lambda \in \Lambda$.
It suffices to show that for any other $\tau$ with $(\lambda+\eta,\tau) \notin \cS_{\cP,\Lambda,\mathrm{t}}$, $R_\rhobar^{\lambda+\eta,\preceq\tau}$ is zero and $\cZ_\sigma(\rhobar) = 0$ for any $\sigma \in \JH(\ovl{\sigma}(\lambda,\tau))$.
Then both sides of (\ref{eqn:minimalBM}) would be zero since $r(\tau)$ for any such $\tau$ is a virtual combination of $\sigma(\tau)$ for such $\tau$.
To show that $R_\rhobar^{\lambda+\eta,\preceq\tau}$ is zero, it suffices to show that $R_{\rhobar^{\mathrm{ss}}}^{\preceq\tau}$ is zero by \cite[Lemma 5]{enns}.
We assume without loss of generality that $\rhobar \cong \rhobar^{\mathrm{ss}}$.
If $\tau$ (or really $\rho_\tau$) is $(h_\Lambda+h_\eta+1)$-generic and $(\lambda+\eta,\tau) \notin \cS_{\cP,\Lambda,\mathrm{t}}$, then $\tld{w}(\rhobar,\tau) \notin \Adm(\lambda+\eta)$ by Lemma \ref{lemma:univpolyrhobar} so that $R_\rhobar^{\lambda+\eta,\preceq\tau} = R_\rhobar^{\lambda+\eta,\tau}$ is zero by Corollary \ref{cor:admshape}.

Suppose now that $\rho_\tau$ is not $(h_\Lambda+h_\eta+1)$-generic.
It suffices to show that $R_{\rhobar|_{G_{K'}}}^{\preceq\tau|_{I_{K'}}}$ is zero for any subfield $K' \subset \ovl{K}$ of finite degree over $K$.
Taking $K'$ to be a sufficiently large unramified extension of $K$, we assume without loss of generality that $\tau$ is a principal series type.
Then the claim follows from a mild strengthening of \cite[Proposition 7]{enns} (the same proof works with minor modifications), replacing $[-n+1,0]$ and $a_j^i \in [0,n-1]$ in \emph{loc.~cit.}~with $[-h_\Lambda-h_\eta+1,0]$ with $a_j^i \in [0,h_\Lambda+h_\eta-1]$, respectively.

We now show that if $\sigma \in \JH(\ovl{\sigma}(\lambda,\tau))$ where $(\lambda+\eta,\tau) \notin \cS_{\cP,\Lambda,\mathrm{t}}$, then $\cZ_\sigma(\rhobar) = 0$.
If $\sigma$ is not generic, $\cZ_\sigma = 0$ by definition.
Assume that $\sigma$ is generic.
There exists a tame type $\tau'$ such that $\sigma \in \JH(\ovl{\sigma}(\tau')) \subset \JH(\ovl{\sigma}(\lambda,\tau))$ so that in particular $(\eta,\tau') \in \cS_{{\{0\}},\mathrm{t}}$. 
Then $i^*_x(\cZ_\sigma) \leq i^*_x(\cZ_{\eta,\tau'})$ for all $x \in \cP_{\mathrm{ss}}$ by Theorem \ref{thm:genBM}(\ref{it:genBM:2}). %
If $i^*_x(\cZ_{\eta,\tau'})$ is zero for all $x \in \cP_{\mathrm{ss}}$, we deduce that $\cZ_\sigma$ is zero by Lemma \ref{lemma:localizeinj}.
Suppose that $i^*_x(\cZ_{\eta,\tau'})$ is nonzero for some $x \in \cP_{\mathrm{ss}}$.
Then $i^*_x(\cZ_{\lambda+\tau})$ is nonzero by Theorem \ref{thm:EGmodp}(\ref{it:irreducible_components_mod_p:1}).
Lemma \ref{lemma:univpolyrhobar} implies that $(\lambda+\eta,\tau) \in \cS_{\cP,\Lambda,\mathrm{t}}$, which is a contradiction.
\end{proof}

\subsection{The generic Breuil--M\'ezard basis}
\label{sec:gen:BM:basis}

In this section, we prove some basic results about the Breuil--M\'ezard cycles $\cZ_\sigma$ that appear in Theorem \ref{thm:genBM}.

\begin{prop}\label{prop:BMcoeff}
Let $\sigma$ be generic and $\cZ_\sigma$ be as in Theorem \ref{thm:genBM}.
Then the coefficient of $\cC_\sigma$ in $\cZ_\sigma$ is $1$.
\end{prop}
\begin{proof}
Choose a lowest alcove presentation $(\tld{w},\omega)$ for $\sigma$ compatible with $\zeta\in X^*(\un{Z})$.
Let $\rhobar$ be a semisimple Galois representation such that there exists a lowest alcove presentation of $\rhobar|_{I_K}$ so that $\tld{w}(\rhobar) = t_\omega\tld{w}$.
Then $\sigma \in W_{\obv}(\rhobar)$ (see Definition \ref{defn:obv_weight}) and so $\rhobar \in \cC_{\sigma}$ by Proposition \ref{prop:rhobaroncomponents}(\ref{item:comp:obv}). %

Let $\tau$ be the tame inertial type with lowest alcove presentation so that $\tld{w}(\tau) = t_\omega w_0\tld{w}_h\tld{w}$. 
Then $\tau$ is $2n$-generic and $\sigma\in \JH(\ovl{\sigma}(\tau))$ corresponds to $(\tld{w},\tld{w}_h\tld{w})$ in (\ref{eqn:JH}).
In fact, since $\sigma$ is generic, $(\eta,\tau) \in \cS_{{\{0\}},\mathrm{t}}$ so that $\cC_\sigma$ is a component of $\cX_{\F}^{\eta,\tau}$ by Theorem \ref{thm:irreducible_components_mod_p}.
We conclude that $0 < i_{\rhobar}^*(\cC_\sigma) \leq i_{\rhobar}^*(\cZ_{\eta,\tau})$ for any versal ring at $\rhobar$.
On the other hand, $R_\rhobar^{\eta, \tau}$ is formally smooth by \cite[Theorem 3.4.1]{LLL} since $\tld{w}(\rhobar,\tau) = t_{w^{-1}(\eta)}$ where $w\in \un{W}$ is the image of $\tld{w}$.
This implies that $i_{\rhobar}^*(\cC_\sigma) = i_{\rhobar}^*(\cZ_{\eta,\tau})$ and that both of these are irreducible cycles.

By the proof of Theorem \ref{thm:genBM}, there exists a minimal patching functor $M_\infty$ for $\rhobar$ and $\cS_{\Lambda,\mathrm{t}}$, which is detectable. 
Then $M_\infty(\sigma)$ is nonzero, and hence so is $\cZ_{\sigma}(\rhobar)$.
Theorem \ref{thm:genBM} implies that $0<\cZ_{\sigma}(\rhobar) \leq i_{\rhobar}^*(\cZ_{\eta,\tau}) = i_{\rhobar}^*(\cC_\sigma)$, so that $i_{\rhobar}^*(\cC_\sigma) = i_{\rhobar}^*(\cZ_\sigma)$.
The result follows from \cite[\href{https://stacks.math.columbia.edu/tag/0DRD}{Tag 0DRD}]{stacks-project}.
\end{proof}

\begin{prop}
The cycles $\cZ_\sigma$, for $\sigma$ generic, form a basis for the span of the cycles $\cC_\sigma$, for $\sigma$ generic.
\end{prop}
\begin{proof}
By Theorem \ref{thm:genBM}, 
Remark \ref{rmk:partialorder}, 
Remark \ref{rmk:genSW}(\ref{item:gencover}), and Proposition \ref{prop:BMcoeff}, the ``change-of-basis matrix" relating $(\cC_\sigma)_\sigma$ and $(\cZ_\sigma)_\sigma$ is ``unipotent upper triangular".
\end{proof}

\subsubsection{Computation of the Breuil--M\'ezard basis}\label{sec:BMbasis}

We end this section with an alternative proof of Theorem \ref{thm:genBM}(\ref{it:genBM:3}), which introduces notation and arguments that will be used in \S \ref{sec:unitary}.
Let $\rhobar$ be a $2n$-generic tame inertial $\F$-type and choose a $2n$-generic lowest alcove presentation for $\rhobar$ with corresponding element $\tld{w}(\rhobar)$ (cf.~Definition \ref{defn:gen:rhobar}).
If $\sigma \in W^?(\rhobar)$ corresponds to the pair $(\tld{w},\tld{w}_1)$ in (\ref{eqn:W?}), we say that the $\rhobar$-defect $\delta_\rhobar(\sigma)$ of $\sigma$ is $\ell(t_\eta) - \ell((\tld{w}_h\tld{w})^{-1}w_0\tld{w}_1)$.
{Since a change of lowest alcove presentation corresponds to conjugation by an element of $\underline{\Omega}$ and the latter preserves length, $\delta_\rhobar(\sigma)$ is independent of the choice of lowest alcove presentation of $\rhobar$.}
Wang's theorem implies that $\tld{w}_1 \leq \tld{w}$ so that $(\tld{w}_h\tld{w})^{-1}w_0\tld{w}_1 \leq (\tld{w}_h\tld{w})^{-1}w_0\tld{w} = t_{w^{-1}(\eta)}$. 
Hence, $\delta_\rhobar(\sigma) \geq 0$ with equality if and only if $\sigma \in W_{\obv}(\rhobar)$.

\begin{prop}\label{prop:defectlower}
Let $\rhobar$ be a $2n$-generic tame inertial $\F$-type and $\tau$ be a $2n$-generic tame inertial type.
Fix compatible $2n$-generic lowest alcove presentations for them, with corresponding elements $\tld{w}(\rhobar)$ and $\tld{w}(\tau)$. %
If $\tld{w}(\rhobar,\tau)\in \Adm^{\mathrm{reg}}(\eta)$ with factorization $\tld{w}_2^{-1}w_0 \tld{w}_1$ as in Remark \ref{rmk:factor}, then $\kappa \defeq F_{(\tld{w}_h^{-1}\tld{w}_2,\tld{w}(\rhobar)(\tld{w}_1)^{-1}(0))}$ is the unique Serre weight in $W^?(\rhobar) \cap \JH(\ovl{\sigma}(\tau))$ which maximizes the defect function $\delta_\rhobar$.
\end{prop}
\begin{proof}
First, the fact that $\tld{w}_1 \uparrow \tld{w}_h^{-1}\tld{w}_2$ implies that $\kappa\in W^?(\rhobar)$ by Proposition \ref{prop:W?}. 
Since $\tld{w}(\rhobar)(\tld{w}_1)^{-1}(0) = \tld{w}(\tau)(\tld{w}_2)^{-1}(0)$, $\kappa \in \JH(\ovl{\sigma}(\tau))$ by Proposition \ref{prop:JHbij}. 

Suppose that $(\tld{w},\omega)$ is a compatible lowest alcove presentation of an element $\sigma \in W^?(\rhobar) \cap \JH(\ovl{\sigma}(\tau))$.
By Proposition \ref{prop:intersect}, $\tld{w}_2^{-1}w_0 \tld{w}_1 = \tld{s}_2^{-1} s \tld{s}_1$ with $\tld{s}_1$, $\tld{s}_2\in \tld{\un{W}}^+$ and $\tld{s}_1 \uparrow \tld{w} \uparrow \tld{w}_h^{-1} \tld{s}_2$.
By Wang's theorem $\tld{s}_2 \leq \tld{w}_h\tld{w}$, and by Lemma \ref{lemma:gallery} $(\tld{w}_h \tld{w})^{-1}w_0\tld{s}_1 \geq \tld{s}_2^{-1}w_0\tld{s}_1$. So we have
\begin{align*} 
\delta_\rhobar(\sigma) = \ell(t_\eta) - \ell((\tld{w}_h \tld{w})^{-1}w_0\tld{s}_1) & \leq \ell(t_\eta)-\ell(\tld{s}_2^{-1}w_0\tld{s}_1) \\ 
& \leq \ell(t_\eta)-\ell(\tld{s}_2^{-1}s\tld{s}_1) = \ell(t_\eta)-\ell(\tld{w}_2^{-1}w_0 \tld{w}_1) = \delta_{\rhobar}(\kappa). 
\end{align*}
Equality implies that $\tld{s}_2 = \tld{w}_h \tld{w}$ and $s = w_0$.
By the uniqueness in Proposition \ref{prop:can:reg}, $\delta_\rhobar(\sigma) = \delta_\rhobar(\kappa)$ implies that $\sigma \cong \kappa$.
\end{proof}

\begin{proof}[Alternative proof of Theorem \ref{thm:genBM}(\ref{it:genBM:3})]
We will denote by $\rhobar$ both a continuous representation $G_K \ra \GL_n(\F)$ and the corresponding inertial $\F$-type obtained by restriction. 
We are given a minimal patching functor $M_\infty^{\rhobar}$ for $\rhobar$ and $\cS_{\Lambda,t}$, and thus a versal ring $R_\infty^{\rhobar}$ to $\cX_n$ at $\rhobar$. For a cycle $\cZ$ of $R_\infty^\rhobar$ of dimension $\dim R_\infty^{\rhobar}\times_{\cX_n} \cX_{n,\red}$, define its generic part $\cZ^{\gen}$ to be the cycle obtained from $\cZ$ by removing any components whose support do not belong to $\underset{\sigma' \textrm{generic}}{\bigcup}i^*_{\rhobar}(\cC_{\sigma'})$.
For any Serre weight $\sigma$ and tame type $\tau$, define $\cZ_\sigma(\rhobar)\defeq Z(M_\infty^\rhobar(\sigma))$ and $\cZ_\tau(\rhobar)\defeq Z(M_\infty^\rhobar(\ovl{\sigma}(\tau)))$. We observe:
\begin{itemize}
\item 
If $\sigma$ is generic, then $\cZ_\sigma(\rhobar)^{\gen}=\cZ_\sigma(\rhobar)$. This is because the support of $\cZ_\sigma(\rhobar)$ belongs to $\underset{\sigma \subset \JH(\ovl{\sigma}(\tau))}{\bigcap} i_\rhobar^*(\cX^{\eta,\tau})=\underset{\sigma \textrm{ covers } \kappa}{\bigcup} i^*_\rhobar (\cC_{\kappa})$, which consists of only generic components.
\item If $\tau$ is a tame type such that $\JH(\ovl{\sigma}(\tau))$ does not contain any generic weight, then $\cZ_{\tau}(\rhobar)^{\gen}=0$.
\item If $\tau$ is a tame type such that $\JH(\ovl{\sigma}(\tau))$ contains a generic weight, then $\cZ_{\tau}(\rhobar)=i^*_{\rhobar}(\cX^{\eta,\tau})$.
\end{itemize}
(For the first two items, we use that Theorem \ref{thm:EGmodp} applies to any $\cX^{\eta,\tau}$ containing $\rhobar$, since $\rhobar$ is $(6n-2)$-generic.) In particular, each $\cZ_{\tau}(\rhobar)^{\gen}$ only depends on the versal ring $R^{\rhobar}_\infty$, i.e.~  equals the $i^*_{\rhobar}$ of a cycle independent of the patching functor $M_\infty^\rhobar$.
 
We now show by induction on $\delta_\rhobar(\sigma)$ that for any $\sigma$, $\cZ_\sigma(\rhobar)^{\gen}$ depends only on $R_\infty^\rhobar$. This proves Theorem \ref{thm:genBM}(\ref{it:genBM:3}) in view of the first item above. By Proposition \ref{prop:WE}, it suffices to restrict our attention to weights in $W^?(\rhobar)$.

Choose a lowest alcove presentation $(\tld{w},\omega)$ of $\sigma$ and a compatible lowest alcove presentation of $\rhobar$.

If $\delta_\rhobar(\sigma) = 0$, then $\tld{w}(\tau) \defeq \tld{w}(\rhobar)t_{-w^{-1}(\eta)}$ corresponds to a compatible lowest alcove presentation of a tame type $\tau$, where $w \in\un{W}$ is the image of $\tld{w}$. 
By Corollary \ref{cor:extremeintersect}, $W^?(\rhobar) \cap \JH(\ovl{\sigma}(\tau)) = \{\sigma\}$.
By Proposition \ref{prop:WE}, $[\ovl{\sigma}(\tau):\sigma] \cZ_\sigma(\rhobar)^{\gen} = \cZ_\tau(\rhobar)^{\gen}$. (It is well-known that $[\ovl{\sigma}(\tau):\sigma]=1$, though we will not use this. In fact, this can be seen from the proof of Proposition \ref{prop:BMcoeff}.) This finishes the base case.

Now, suppose $\delta_\rhobar(\sigma) > 0$. If $\sigma$ corresponds to $(\tld{w},\tld{w}_1)$ in (\ref{eqn:W?}), we choose $\tau$ so that $\tld{w}(\tau) \defeq \tld{w}(\rhobar)(\tld{w}_h\tld{w})^{-1}w_0\tld{w}_1$. Then 
\begin{equation} \label{eq:recursive_BM}
[\ovl{\sigma}(\tau):\sigma] \cZ_\sigma(\rhobar)^{\gen} = \cZ_\tau(\rhobar)^{\gen} - \sum_{\kappa \in W^?(\rhobar) \cap \JH(\ovl{\sigma}(\tau))} [\ovl{\sigma}(\tau):\kappa] \cZ_\kappa(\rhobar)^{\gen}.
\end{equation}
(Again, $[\ovl{\sigma}(\tau):\sigma] =1$.)
Since the right hand side depends only on $R_\infty^\rhobar$ by induction and Proposition \ref{prop:defectlower}, so does the left hand side.
\end{proof}

\begin{rmk}
\label{rmk:algorithm_Zsigma}
Combining the above with Proposition \ref{prop:BMglobalize} gives the following recursive procedure to compute $\cZ_\sigma$ for generic $\sigma$: For any $\sigma'$ covered by $\sigma$, choose a tame $\rhobar$ lying on $\cC_{\sigma'}$ (for instance, those given by Proposition \ref{prop:rhobaroncomponents}(\ref{item:comp:obv})). This gives a defect function on $W^?(\rhobar)$, and the above proof allows us to recursively compute the coefficient of $\cC_{\sigma'}$ in $\cZ_\kappa$ for $\kappa\in W^?(\rhobar)$, and in particular the coefficient of $\cC_{\sigma'}$ in $\cZ_{\sigma}$.
Note that $\sigma \in W^?(\rhobar)$ by Proposition \ref{prop:covering}(\ref{it:covering:3}) (see Remark \ref{rmk:compTfixedpts}(\ref{it:compTfixedpts:2})) and Theorem \ref{thm:Tfixedpts}(\ref{it:Tfixedpts:2}).
\end{rmk}
\begin{prop}\label{prop:BMcycles_underdetermined}
Let $\Lambda \subset X_*(\un{T}^{\vee})$ be a finite set of dominant weights containing $0$.
Suppose there are two collection of effective cycles $\cZ_\sigma, \cZ'_\sigma \in \Z[\cX_{n,\mathrm{red}}]$ for each Serre weight $\sigma$ such that 
\begin{align*}
\cZ_{\lambda,\tau} &= \sum_{\sigma} [\ovl{\sigma}(\lambda,\tau):\sigma] \cZ_\sigma \\
\cZ_{\lambda,\tau} &= \sum_{\sigma} [\ovl{\sigma}(\lambda,\tau):\sigma] \cZ'_\sigma.
\end{align*}
for any $\lambda \in \Lambda$ and tame inertial type $\tau$ which is $P_{\cP,\Lambda,e}$-generic. Let $\sigma_0$ be a Serre weight such that:
\begin{itemize}
\item For any pair of tame types $\tau_1,\tau_2$ such that $\JH(\ovl{\sigma}(\tau_1))\cap \JH(\ovl{\sigma}(\tau_2))\neq \emptyset$ and $\kappa\in \JH(\ovl{\sigma}(\tau_1))$ for some $\kappa$ covered by $\sigma_0$, $\tau_2$ is $P_{\cP,\Lambda,e}$-generic.
\end{itemize}
Then for each semisimple $\rhobar\in \cX_n(\F)$, $\cZ_{\sigma_0}(\rhobar)=\cZ'_{\sigma_0}(\rhobar)$. In particular, $\cZ_{\sigma_0}=\cZ'_{\sigma_0}$.
\end{prop}
\begin{proof} 
In the proof, we will only consider cycles $\cZ_\sigma$ that occur in the given cycle equations.

We call a Serre weight $\sigma$ satisfying condition in the statement very generic. For a cycle $\cZ$ of $\cX_{n,\red}$, we define its very generic part $\cZ^{\mathrm{v.gen}}$ to be the cycle obtained by removing from $\cZ$ any component $\cC_{\sigma}$ such that $\sigma$ not very generic.

We make the following observations:
\begin{itemize}
\item Suppose $\sigma$ has the property that any tame type $\tau$ such that $\sigma\in \JH(\ovl{\sigma}(\tau))$ is $P_{\cP,\Lambda,e}$-generic. Then the support of $\cZ_\sigma$ belongs to $\underset{\sigma\in \JH(\ovl{\sigma}(\tau))}{\bigcap} \cX^{\eta,\tau}=\underset{\sigma \textrm{ covers } \kappa}{\bigcup} \cC_{\kappa}$ (where the equality follows from Theorem \ref{thm:EGmodp}). In particular, as in Remark \ref{rmk:algorithm_Zsigma}, for any tame $\rhobar$ such that $\cZ_{\sigma}(\rhobar)\neq 0$ we have $\sigma\in W^?(\rhobar)$.
\item Suppose $\rhobar$ is tame and $\cZ^{\mathrm{v.gen}}_\sigma(\rhobar)\neq 0$. Since $\cZ_\sigma$ occurs in the given cycle equations, we can find a $P_{\cP,\Lambda,e}$-generic tame type $\tau_1$ such that $\sigma\in \JH(\ovl{\sigma}(\tau_1))$. By Theorem \ref{thm:EGmodp}, $\cZ^{\mathrm{v.gen}}_\sigma(\rhobar)\neq 0$ implies $\kappa\in \JH(\ovl{\sigma}(\tau_1))$ for some very generic $\kappa$. The definition of very generic implies the previous item applies to $\sigma$, thus we learn that $\sigma\in W^?(\rhobar)$.
\item If $\sigma$ is very generic, then $\cZ_\sigma^{\mathrm{v.gen}}=\cZ_\sigma$.
\item If $\tau$ is a tame type such that $\cZ_{\tau}$ does not occur in the given cycle equations, then $\cZ^{\mathrm{v.gen}}_{\sigma}=0$ for any $\sigma\in\JH(\ovl{\sigma}(\tau))$.
\end{itemize}

We now fix a very generic $\sigma_0$ and $\rhobar$ such that $\cC_{\kappa}(\rhobar)\neq 0$ for some $\kappa$ covered by $\sigma_0$.
Given the above observations, we can repeat the argument in the alternative proof Theorem \ref{thm:genBM}(\ref{it:genBM:3}) to give a recursive formula for $\cZ^{\mathrm{v.gen}}_\sigma(\rhobar)$ in terms of $\cZ_\tau^{\mathrm{v.gen}}(\rhobar)$ for various tame types $\tau$.
But then $\cZ^{\prime,\mathrm{v.gen}}_\sigma(\rhobar)$ satisfies the same recursive formula, and 
hence $\cZ_{\sigma_0}(\rhobar)=\cZ'_{\sigma_0}(\rhobar)$.
\end{proof}
\begin{rmk} 
\label{rmk:BMcycles_underdetermined}
The condition on the Serre weight $\sigma_0$ in Proposition \ref{prop:BMcycles_underdetermined} is guaranteed by $Q$-genericity, for an appropriate polynomial $Q$ built out of $P_{\cP,\Lambda,e}$ (cf.~ Proposition \ref{prop:JHbij} and the proof of Lemma \ref{lemma:univpolytau}).
\end{rmk}

\clearpage{}%
\clearpage{}%
\section{Global applications}
\label{sec:GA}

\subsection{Serre weights for some definite unitary groups}\label{sec:unitary}

Let $F^+$ be a totally real field not equal to $\Q$, and let $F \subset \ovl{F}^+$ be a CM extension of $F^+$.
Denote the set of places in $F^+$ dividing $p$ by $S_p$.
We say that a finite place of $F^+$ (resp.~of $F$) is \emph{split} if it splits in $F$ (resp.~if its restriction to $F^+$ splits in $F$).
Suppose from now on that all places in $S_p$ are split.
Let $G_{/F^+}$ be a reductive group which is an outer form for $\GL_n$ which 
\begin{itemize}
\item splits over $F$; and
\item is definite at all archimedean places.
\end{itemize}
Recall from \cite[\S 7.1]{EGH} that $G$ admits a reductive model $\cG$ defined over $\cO_{F^+}[1/N]$, for some $N\in \N$ which is prime to $p$, together with an isomorphism
\begin{equation}
\label{iso integral}
\iota:\,\cG_{/\cO_{F}[1/N]} \stackrel{\iota}{\rightarrow}{\GL_n}_{/\cO_{F}[1/N]}
\end{equation}
which specializes to
$
\iota_w:\,\cG(\cO_{F^+_v})\stackrel{\sim}{\rightarrow}\cG(\cO_{F_w})\stackrel{\iota}{\rightarrow}\GL_n(\cO_{F_w})
$
for all split finite places $w$ in $F$ where $v$ is $w|_{F^+}$ here.
For each split place $v$ of $F^+$, we choose a place $\tld{v}$ of $F$ dividing $v$, and we let $\iota_v$ be the composition of $\iota_{\tld{v}}$ and the canonical isomorphism $\GL_n(\cO_{F_{\tld{v}}}) \cong \GL_n(\cO_{F^+_v})$ (suppressing the dependence on the choice of $\tld{v}$).
If $U = U_p U^{\infty,p} \leq G(\A_{F^+,p}^{\infty}) \times G(\A_{F^+}^{\infty,p})$ is a compact open subgroup and $W$ is a finite $\cO$-module endowed with a continuous action of $U_p$, then we define the space of algebraic automorphic forms on $G$ of level $U$ and coefficients in $W$ to be the (finite) $\cO$-module 
\begin{equation}
S(U,W) \defeq \left\{f:\,G(F^{+})\backslash G(\A^{\infty}_{F^{+}})\rightarrow W\,|\, f(gu)=u_p^{-1}f(g)\,\,\forall\,\,g\in G(\A^{\infty}_{F^{+}}), u\in U\right\}.
\end{equation}
We recall that the level $U$ is said to be \emph{sufficiently small} if for all $t \in G(\bA^{\infty}_{F^+})$, the order of the finite group $t^{-1} G(F^+) t \cap U$ is prime to $p$.
If $U$ is sufficiently small, then $S(U,-)$ defines an exact functor from finite $\cO$-modules with a continuous $U_p$-action to finite $\cO$-modules.
From now on we assume that $U$ is sufficiently small.

For a finite place $v$ of $F^+$ prime to $N$, we say that $U$ is \emph{unramified} at $v$ if one has a decomposition $U=\cG(\cO_{F^+_v})U^{v}$. %
Let $S$ be a finite set of finite places in $F^+$ containing $S_p$, all places dividing $N$, and all places at which $U$ is \emph{not} unramified.

Let $\cP_S$ be the set of split finite places $w$ of $F$ such that $v = w|_{F^+} \notin S$. 
For any subset $\cP\subseteq \cP_S$ of finite complement that is closed under complex conjugation, we write $\bT_{\cP}\defeq\cO[T^{(i)}_w,\,\,w\in\cP,\, 0 \leq i \leq n]$ for the universal Hecke algebra on $\cP$.
The space of algebraic automorphic forms $S(U,W)$ is endowed with an action of $\bT_{\cP}$, where $T_w^{(i)}$ acts by the usual double coset operator
\[
\iota_w^{-1}\left[ \GL_n(\cO_{F_w}) \left(\begin{matrix}
      \varpi_{w}\mathrm{Id}_i & 0 \cr 0 & \mathrm{Id}_{n-i} \end{matrix} \right)
\GL_n(\cO_{F_w}) \right].
\]

Suppose that $S(U,W)_{\fm}\neq 0$ (or equivalently $S(U,W\otimes_{\cO} \F)_{\fm}\neq 0$) where $\fm$ is the kernel of a homomorphism $\alpha: \bT_{\cP}\rightarrow \F$.
Let $\bT_{\cP}(U,W)$ be the image of $\bT_{\cP}$ in $\End_{\cO}(S(U,W))$---it is a semilocal ring.
If $Q$ is the (finite) set $\{w|_{F^+}:w\in \cP_S \setminus \cP\}$, then we also denote $\bT_{\cP}(U,W)$ by $\bT^Q(U,W)$.
Let $\alpha: \bT_{\cP} \onto \bT^Q(U,W)_{\fm}$ be the natural quotient map.
There is a Galois representation $r_{\fm} \defeq r(U,W)_{\fm}: G_{F^+,S} \ra \cG_n(\bT^Q(U,W)_{\fm})$, where $\cG_n$ is the group scheme over $\Z$ defined in \cite[\S 2.1]{CHT} (see also \S \ref{sec:galdef}), determined by the equations
\[
\det\left(1-r(U,W)_{\fm}|_{G_F}(\mathrm{Frob}_w)X\right)=\sum_{j=0}^n (-1)^j(\mathbf{N}_{F/\Q}(w))^{\binom{j}{2}}\alpha(T_w^{(j)})X^j
\]
for all $w\in \cP$.
Let $\rbar: G_{F^+} \ra \cG_n(\F)$ be the reduction $r_{\fm} \pmod{\fm}$.
We say that such a Galois representation $\rbar$ is \emph{automorphic} of level $U$ and coefficients $W$, and $\fm$ is the maximal ideal of $\bT_{\cP}$ corresponding to $\rbar$.
We say that $\rbar$ is \emph{automorphic} if $\rbar$ is automorphic of some level $U$ and some coefficients $W$.

Let $\cO_p$ be $\cO_{F^+} \otimes_{\Z} \Z_p \cong \prod\limits_{v \in S_p} \cO_{F^+_v}$.
Then the composition
\begin{equation}
\iota_p \defeq \prod_{v\in S_p} \iota_v: \cG(\cO_p) \cong \prod_{v\in S_p} \cG(\cO_{F^+_v}) \risom \prod_{v\in S_p} \GL_n(\cO_{F^+_v})
\end{equation}
gives an equivalence between $\cG(\cO_p)$-modules and $\prod\limits_{v\in S_p} \GL_n(\cO_{F^+_v})$-modules.
Let $k_v$ denote the residue field of $F^+_v$ and $\rG \defeq \prod\limits_{v\in S_p} \GL_n(k_v)$.
If $\sigma$ is a Serre weight of $\rG$, then $\sigma$ is naturally a $\prod\limits_{v\in S_p} \GL_n(\cO_{F^+_v})$-module by inflation.
We can now define what it means for a $\rbar$ as above to be automorphic of a particular weight and level.

\begin{defn}
\label{defn:mod:wght}
Let $U = \cG(\cO_p) U^{S_p}$ be a sufficiently small compact open subgroup of $G(\bA_{F^+}^\infty)$ and let $\sigma$ be a Serre weight for $\rG$.

We say that $\rbar$ is \emph{automorphic of weight} $\sigma$ and level $U$ or $\sigma$ is a \emph{modular} (\emph{Serre}) \emph{weight} for $\rbar$ at level $U$ if $\rbar$ is automorphic of level $U$ and coefficients $\sigma^\vee\circ\iota_p$, where $\sigma^\vee$ denotes the $\F$-dual of $\sigma$.
We say that $\rbar$ is automorphic of weight $\sigma$ or $\sigma$ is a \emph{modular} (\emph{Serre}) \emph{weight} for $\rbar$ if $\rbar$ is automorphic of weight $\sigma$ and some level $U$. 

Let $W(\rbar)$ be the set of modular Serre weights of $\rbar$.
Let $W_{\mathrm{gen}}(\rbar)$ be the subset of \emph{generic} Serre weights in $W(\rbar)$ (see Definition \ref{defn:genSW}).
\end{defn}

\noindent It is a standard fact that if $\rbar$ is automorphic, then $W(\rbar)$ is nonempty.
{
Indeed if $\rbar$ is automorphic, then we can assume by exactness of $S(U,-)_\fm$ that $\rbar$ is automorphic of level $U$ and coefficients $W$ where $W$ is an irreducible $\F[\cG(\cO_p)]$-module. 
Since the space of invariants of a pro-$p$ group acting continuously on an $\F$-vector space is nonzero, $W$ is of the form $\sigma^\vee \circ \iota_p$ for some Serre weight $\sigma$ of $\rG$. 
}

Fixing maps $\ovl{F}^+ \into \ovl{F}^+_v$ for each $v\in S_p$, the restriction of a continuous representation $\rbar: G_{F^+} \ra \cG_n(\F)$ gives a collection of continuous representations $(\rbar_v: G_{F^+_v} \ra \GL_n(\F))_{v\in S_p}$, which is equivalent to an $L$-homomorphism over $\F$ which we denote $\rbar_p: W_{\Qp} \ra$ $^L \un{G}(\F)$ where $G = \GL_n$.

\begin{defn} 
\label{defn:predicted:SW}
Given an $L$-parameter $\rhobar: G_{\Qp} \ra $ $^L \un{G}(\F)$, we say that $\sigma$ is a \emph{geometric} Serre weight of $\rhobar$ if the corresponding collection $(\rhobar_v)_{v\in S_p}$ lies on $\cC_\sigma$ (equivalently, $\rhobar_v$ lies on $\cC_{\sigma_v}$ for all $v\in S_p$ where $\sigma \cong \otimes_{v\in S_p} \sigma_v$).
We let $W^{\mathrm{g}}_{\mathrm{gen}}(\rhobar)$ be the set of geometric Serre weights of $\rhobar$ which are $3n-1$-deep. 

We let $W_{\mathrm{gen}}^{\mathrm{BM}}(\rhobar)$ be the set of \emph{generic} Serre weights such that $\rhobar$ lies in the support of $\cZ_\sigma$ (defined in Theorem \ref{thm:genBM}(\ref{it:genBM:2}), see Remark \ref{rmk:BMproduct}).
By Remark \ref{rmk:BMproduct}, 
\[
W_{\mathrm{gen}}^{\mathrm{BM}}(\rhobar) = \{\otimes_{v\in S_p} \sigma_v|\sigma_v \in W_{\mathrm{gen}}^{\mathrm{BM}}(\rhobar_v)\}.
\]
\end{defn}

\begin{rmk}\label{rmk:geomBM}
Proposition \ref{prop:BMcoeff} implies that any generic Serre weight in $W^{\mathrm{g}}_{\mathrm{gen}}(\rhobar)$ is contained in $W_{\mathrm{gen}}^{\mathrm{BM}}(\rhobar)$.
\end{rmk}

The following conjecture is based on \cite[Conjecture 6.9]{herzig-duke} and \cite[\S 2 and 9.2]{GHS}.

\begin{conj}
\label{conj:SWC-1}
Suppose that $\rbar: G_{F^+} \ra \cG_n(\F)$ is automorphic and that the inertial $L$-homomorphism $\rbar_p|_{I_{\Qp}}$ over $\F$ is tame and $2n$-generic.
Then $W(\rbar) = W^?(\rbar_p|_{I_{\Qp}})$.
\end{conj}

We can use Theorem \ref{thm:genBM}(\ref{it:genBM:2}) to make the following unconditional version of \cite[Conjecture 3.2.7]{GHS}.

\begin{conj}
\label{conj:SWC-2}
Suppose that $\rbar: G_{F^+} \ra \cG_n(\F)$ is automorphic.
Then $W_{\mathrm{gen}}(\rbar) = W_{\mathrm{gen}}^{\mathrm{BM}}(\rbar_p)$.
\end{conj}

\begin{thm}\label{thm:SWC}
There exists a polynomial $P(X_1,\ldots,X_n) \in \Z[X_1,\ldots,X_n]$, independent of $p$, such that if ($p\nmid 2n$ and) and
\begin{itemize}
\item $\rbar: G_{F^+} \ra \cG_n(\F)$ is automorphic;
\item $\rbar|_{G_F}(G_{F(\zeta_p)})$ is adequate; and that 
\item the inertial $L$-parameter $\rbar_p|_{I_{\Qp}}$ over $\F$ is tame and has a lowest alcove presentation $(s,\mu-\eta)$ such that $\mu$ is $P$-generic, i.e.~$P(\mu_{j,1},\ldots,\mu_{j,n}) \not\equiv 0 \pmod p$ for all $j\in \cJ$, where $\cJ = \Hom(F^+,E)$, %
\end{itemize}
then 
\[
W(\rbar) = W^{\mathrm{g}}_{\mathrm{gen}}(\rbar_p) = W^?(\rbar_p|_{I_{\Qp}}) \quad \textrm{and} \quad W_{\mathrm{gen}}(\rbar) = W_{\mathrm{gen}}^{\mathrm{BM}}(\rbar_p).
\]
\end{thm}

\begin{rmk}
\begin{enumerate}
\item The polynomial $P$ in Theorem \ref{thm:SWC} can be taken to be the product of the polynomials $P_{6n-2}$, $P_{2\eta,e}$, $P^{\eta_0}_{\eta,e}$, and $Q$ appearing, respectively, in Remark \ref{rmk:P-gen}, Theorem \ref{thm:stack_local_model}(\ref{it:stack_local_model:2}), equation (\ref{eqn:polysuper}) and the proof of Lemma \ref{lemma:geom?} below.

\item There exists $P$ so that the $P$-genericity hypothesis implies that $W^?(\rbar_p)$ contains only generic Serre weights. 
So $W_{\mathrm{gen}}(\rbar)$ in Theorem \ref{thm:SWC} could be replaced by $W(\rbar)$.
\end{enumerate}
\end{rmk}

\begin{rmk}
We describe a method to construct examples to which Theorem \ref{thm:SWC} applies.
Suppose that $p\nmid 2n$, $K/\Q_p$ is a finite unramified extension, and let $\rhobar: G_K \ra \GL_n(\F)$ be a semisimple continuous Galois representation such that $\rhobar|_{I_K}$ has a lowest alcove presentation $(s,\mu-\eta)$ with $\mu$ $P$-generic with $P$ as in Theorem \ref{thm:SWC}.
Then by \cite[Corollary A.7]{CEGGPS}, there exists a CM extension $F/F^+$ with $F^+ \neq \Q$ and a (potentially diagonalizably) automorphic representation $\rbar: G_{F^+} \ra \cG_n(\F)$ which is isomorphic to $\rhobar$ at all $p$-adic places and whose restriction $\rbar|_{G_F}(G_{F(\zeta_p)})$ is adequate. %
Then Theorem \ref{thm:SWC} applies to $\rbar$.
\end{rmk}

\begin{lemma}\label{lemma:geom?}
There exists a nonzero polynomial $P(X_1,\ldots,X_n) \in \Z[X_1,\ldots,X_n]$ such that if $\rhobar$ is a tame $L$-homomorphism over $\F$ such that $\rhobar|_{I_{\Qp}}$ has a lowest alcove presentation $(s,\mu-\eta)$ where $\mu$ is $P$-generic, then $W^{\mathrm{g}}_{\mathrm{gen}}(\rhobar) = W^?(\rhobar|_{I_{\Qp}})$.
\end{lemma}
\begin{proof}
The inclusion $W^{\mathrm{g}}_{\mathrm{gen}}(\rhobar) \subset W^?(\rhobar|_{I_{\Qp}})$ follows from Proposition \ref{prop:rhobaroncomponents}(\ref{item:comp:WE}) if $\rhobar$ is $4n$-generic. 
We now show the opposite inclusion.
Fix a set $R$ of representatives for the (finite) set $\tld{W}^+_1/X^0(T)$, and consider the (finite) product 
\[
Q(X_1,\ldots, X_n) \defeq \prod_{\tld{w} \in R} \prod_{\tld{w}_2 \uparrow \tld{w},\tld{w}_2 \in \tld{W}^+}\prod_{w\in W} P_{\tld{w}}(X+w\tld{w}_2^{-1}(0)),
\]
where $P_{\tld{w}}$ is as in Proposition \ref{prop:Tfixedpts} and $w\tld{w}_2^{-1}(0)$ is an element of $\Z^n$ under the usual identification $X^*(T) \cong \Z^n$.
If $\rhobar|_{I_K}$ has a lowest alcove presentation $(s,\mu-\eta)$ such that $\mu$ is $Q$-generic, then the compatible lowest alcove presentation for $\sigma \in W^{?}(\rhobar|_{I_K})$ from Proposition \ref{prop:W?} satisfies the hypothesis of Proposition \ref{prop:rhobaroncomponents}(\ref{item:comp:fixed}) so that $\rhobar\in \cC_\sigma$. 
We can therefore take $P = QP_{4n}$ (see \ref{rmk:P-gen}(\ref{it:P-gen:2})).
\end{proof}

We introduce notation for prime ideals in deformation rings corresponding to the irreducible components of
\[
\cX_{n,\mathrm{red}}^{F^+_p} = \prod_{v\in S_p, \F} \cX_{n,\mathrm{red}}^{F^+_v}.
\]
Recall that we index these irreducible components $\cC_\sigma$ by Serre weights $\sigma$ of $\rG$.
Let $\rhobar$ be a tame $L$-homomorphism over $\F$, and recall from \S \ref{sec:patch:ax} that 
\[
R_\rhobar \defeq \widehat{\bigotimes}_{v\in S_p,\cO} R_{\rhobar_v}^\square.
\]
Then there is a versal map $i_\rhobar: \Spf R_\rhobar \ra \prod_{v\in S_p, \Spf \cO} \cX_n^{F^+_v}$.
By Proposition \ref{prop:Tfixedpots_unibranch} and Remark \ref{rmk:MLM:sp:fib}(\ref{it:MLM:compmatch}), if $\rhobar$ is $QP_{4n}$-generic as in the proof of Lemma \ref{lemma:geom?} and $\sigma$ is $(3n-1)$-deep, then $i_\rhobar^*(\cC_\sigma)$ is an irreducible cycle (if nonzero).
In this case, we let $\fp_\sigma(\rhobar) \subset R_\rhobar$ denote the corresponding prime ideal.

\begin{lemma}\label{lemma:SWexist}
There exists a polynomial $P(X_1,\ldots,X_n) \in \Z[X_1,\ldots,X_n]$ such that if $\rhobar$ is a tame $L$-homomorphism over $\F$ such that $\rhobar|_{I_{\Qp}}$ has a lowest alcove presentation $(s,\mu-\eta)$ where $\mu$ is $P$-generic
and $M_\infty$ is a weak detectable patching functor for $\rhobar$, then 
\[
\{\sigma \mid M_\infty(\sigma) \neq 0\} = W^?(\rhobar|_{I_{\Qp}}) \quad \textrm{and} \quad \{\sigma \textrm{ generic}\mid M_\infty(\sigma) \neq 0\} = W_{\mathrm{gen}}^{\mathrm{BM}}(\rhobar).
\]
\end{lemma}
\begin{proof}
We claim that the result holds with $P$ taken to be the product of $P^{\eta_0}_{\eta,e}$ (see Theorem \ref{thm:stack_local_model}(\ref{it:stack_local_model:2}) and equation (\ref{eqn:polysuper})),
$P_{6n-2}$ (see Remark \ref{rmk:P-gen}), and $Q$ (from the proof of Lemma \ref{lemma:geom?}).
We have that $\{\sigma \mid M_\infty(\sigma) \neq 0\} \subset  W^?(\rhobar|_{I_{\Qp}})$ by Proposition \ref{prop:WE} (using that $P_{6n-2} \mid P$ and Remark \ref{rmk:P-gen}(\ref{it:P-gen:2})), so it suffices to show the opposite inclusion.

We first claim that if $\sigma \in W^?(\rhobar)$ is a Serre weight such that $\mathrm{Ann}_{R_\rhobar}M_\infty(\sigma) \subset \fp_{\sigma'}(\rhobar)$ for some Serre weight $\sigma'\in W^{\mathrm{g}}_{\mathrm{gen}}(\rhobar)$, then $\delta_{\rhobar|_{I_{\Qp}}}(\sigma') \leq \delta_{\rhobar|_{I_{\Qp}}}(\sigma)$ with $\delta_{\rhobar|_{I_{\Qp}}}$ defined in \S \ref{sec:BMbasis}.
Suppose that $\sigma \in W^?(\rhobar|_{I_{\Qp}})$ corresponds to $(\tld{w},\tld{w}_1)$ in (\ref{eqn:W?}) (with the lowest alcove presentation as in the statement of the theorem).
Then we let $\tau$ be the tame inertial $L$-parameter with a $2n$-generic lowest alcove presentation such that $\tld{w}(\tau) = \tld{w}(\rhobar|_{I_{\Qp}})(\tld{w}_h\tld{w})^{-1}w_0\tld{w}_1$. %
As in the alternative proof of Theorem \ref{thm:genBM}(\ref{it:genBM:3}) (\S \ref{sec:BMbasis}, this choice is made so that the set $W^?(\rhobar|_{I_{\Qp}})\cap \JH(\ovl{\sigma}(\tau))$ contains $\sigma$ and weights of strictly smaller $\rhobar|_{I_{\Qp}}$-defect than $\sigma$ (we say that $\tau$ is strictly defect lowering for $\rhobar$ and $\sigma$).
Theorem \ref{thm:irreducible_components_mod_p}(\ref{it:irreducible_components_mod_p:1}) implies (after taking products as in Remark \ref{rmk:MLM:sp:fib}(\ref{item:prod_mod_p_comp})) that the irreducible components of $\Spec R^\tau_\rhobar$ %
are 
\begin{equation} \label{eqn:intersect_geom?}
\{\cC_{\sigma'}(\rhobar)\mid \sigma' \in W^{\mathrm{g}}_{\mathrm{gen}}(\rhobar)\cap \JH(\ovl{\sigma}(\tau))\} = \{\cC_{\sigma'}(\rhobar)\mid \sigma' \in W^?(\rhobar|_{I_{\Qp}})\cap \JH(\ovl{\sigma}(\tau))\},
\end{equation}
where the equality uses Lemma \ref{lemma:geom?} (and that $QP_{4n}\mid P$).
If $\mathrm{Ann}_{R_\rhobar}M_\infty(\sigma) \subset \fp_{\sigma'}(\rhobar)$, then since $\mathrm{Ann}_{R_\rhobar}M_\infty(\ovl{\sigma}(\tau)) \subset \mathrm{Ann}_{R_\rhobar}M_\infty(\sigma)$, we conclude that $\sigma'\in \JH(\ovl{\sigma}(\tau))$.
The claim then follows from (\ref{eqn:intersect_geom?}) and that $\tau$ is strictly defect lowering for $\rhobar$ and $\sigma$.

We now establish the opposite inclusion: for $\sigma \in W^?(\rhobar|_{I_{\Qp}})$, we show that $M_\infty(\sigma) \neq 0$.
Choose $\tau$ in terms of $\sigma$ as in the previous paragraph.
Since $(\tld{w}_h\tld{w})^{-1}w_0\tld{w}_1\in \Adm(\eta)$ by Proposition \ref{prop:can:adm} and $P^{\eta_0}_{\eta,e} \mid P$, $(\eta,\tau) \in \cS_{0,\mathrm{t}}$ by the proof of Lemma \ref{lemma:univpolytau}.
Choosing an $\cO$-lattice $\sigma^\circ(\tau) \subset \sigma(\tau)$, combining that $M_\infty(\sigma^\circ(\tau))$ is maximal Cohen--Macaulay over $R_\infty(\tau)$, $R_\rhobar^\tau$ is a domain by Theorem \ref{thm:stack_local_model}(\ref{it:stack_local_model:2}) (cf.~Remark \ref{rmk:product_stack_local_model}), 
$M_\infty(\sigma^\circ(\tau))$ is nonzero by Proposition \ref{prop:patchnonzero}, and Theorem \ref{thm:irreducible_components_mod_p}(\ref{it:irreducible_components_mod_p:1}), we conclude that $\mathrm{Ann}_{R_\rhobar} M_\infty(\ovl{\sigma}^\circ(\tau))$ is contained in $\fp_\sigma(\rhobar)$ (which is a proper ideal by (\ref{eqn:intersect_geom?})).
Then $\mathrm{Ann}_{R_\rhobar} M_\infty(\sigma')$ is contained in $\fp_\sigma(\rhobar)$ for some $\sigma' \in W^?(\rhobar|_{I_{\Qp}})\cap \JH(\ovl{\sigma}(\tau))$. 
The claim in the previous paragraph (with the roles of $\sigma$ and $\sigma'$ reversed) implies that $\delta_{\rhobar|_{I_{\Qp}}}(\sigma) \leq \delta_{\rhobar|_{I_{\Qp}}}(\sigma')$.
Since $\tau$ is strictly defect lowering for $\rhobar$ and $\sigma$, $\sigma = \sigma'$.
We conclude that $M_\infty(\sigma) \neq 0$.

Finally, we claim that if $\sigma$ is generic, then $\sigma \in W^?(\rhobar)$ if and only if $\sigma \in W_{\mathrm{gen}}^{\mathrm{BM}}(\rhobar)$.
The forward implication follows from Lemma \ref{lemma:geom?} and Remark \ref{rmk:geomBM}.
We now show that $W_{\mathrm{gen}}^{\mathrm{BM}}(\rhobar)\subset W^?(\rhobar)$. Suppose that $\sigma \in W_{\mathrm{gen}}^{\mathrm{BM}}(\rhobar)$. Then for any minimal patching functor $M_\infty'$ for $\rhobar$ and $\cS_{0,\mathrm{t}}$, $M_\infty'(\sigma) \neq 0$ by Theorem \ref{thm:genBM}(\ref{it:genBM:2}) so that $\sigma \in W^?(\rhobar)$ by Proposition \ref{prop:WE}.
Alternatively, by the same result, $\rhobar \in \cC_{\sigma'}$ for some $\sigma'$ which $\sigma$ covers.
Then $\rhobar\in \cC_\sigma$ by Remark \ref{rmk:compTfixedpts}(\ref{it:compTfixedpts:2}) (see the proof of Proposition \ref{prop:rhobaroncomponents}).
\end{proof}

\begin{proof}[Proof of Theorem \ref{thm:SWC}]
The result follows from Lemmas \ref{lemma:SWexist} and \ref{lemma:patchSWC}.
\end{proof}

\subsection{A modularity lifting result}
\label{subsec:MLT}

\begin{thm} \label{thm:modularity_lifting}
Let $F/F^+$ be a CM extension, and let $r: G_F \ra \GL_n(E)$ be a continuous representation such that 
\begin{itemize}
\item $r$ is unramified at all but finitely many places;
\item $r$ is potentially crystalline at places dividing $p$ of type $(\lambda+\eta,\tau)$ where $\lambda \in (\Z^n_+)^{\Hom(F,E)}$ and $\tau$ is a tame inertial type that admits a lowest alcove presentation $(s,\mu-\eta)$ where $\mu$ is $P_{\lambda+\eta,e}$-generic;
\item $r^c \cong r^\vee \varepsilon^{1-n}$;
\item $\rbar$ is semisimple locally at places above $p$; 
\item $\rbar: G_{F(\zeta_p)} \ra \GL_n(\F)$ is an adequate subgroup and $\zeta_p\notin \ovl{F}^{\ker\ad\rbar}$; and
\item $\rbar \cong \rbar_\iota(\pi)$ for some $\pi$ a regular algebraic conjugate self-dual cuspidal (RACSDC) automorphic representation of $\GL_n(\A_F)$ of weight $\lambda$ so that $\sigma(\tau)$ is a $K$-type for $\pi$ at places dividing $p$.
\end{itemize}
Then $r$ is automorphic i.e.~$r \cong r_\iota(\pi')$ for some $\pi'$ a RACSDC automorphic representation of $\GL_n(\A_F)$ \emph{(}of weight $\lambda$ so that $\sigma(\tau)$ is a $K$-type for $\pi$ at places dividing $p$\emph{)}.
\end{thm}
\begin{proof}
This follows from Theorem \ref{thm:stack_local_model} from standard base change and Taylor--Wiles patching arguments {cf.~the proof of \cite[Proposition 6.0.2]{LLLM2}.}
\end{proof}

\begin{rmk}
\begin{enumerate}
\item
After possibly changing the polynomial $P_{\lambda+\eta,e}$ in Theorem \ref{thm:modularity_lifting}, the last condition on $r$ can be relaxed to only require that $\rbar \cong \rbar_\iota(\pi)$ for some RACSDC automorphic representation $\pi$ using Theorem \ref{thm:SWC} to incorporate a ``change of weight" result. 
\item Unfortunately, the inexplicit nature of $P_{\lambda+\eta,e}$ makes Theorem \ref{thm:modularity_lifting} rather impractical to apply.
\end{enumerate}
\end{rmk}

\newpage

\appendix

\section{Taylor--Wiles patching}
\label{app:TW:patch}

The goal of this section is to construct a patching functor from algebraic modular forms on a definite unitary group using the Taylor--Wiles method. 
This differs from most other constructions in that we allow arbitrary level while other constructions typically assume that the level away from $p$ is rather mild. 
For the purposes of automorphy lifting results, one can arrange for this assumption to hold using solvable base change. 
Since Theorem \ref{thm:SWC} is a characteristic $p$ result, we cannot use solvable base change as a reduction step.
Fortunately, the necessary modifications to account for level are not difficult.

\subsection{The result}

\begin{lemma}\label{lemma:patchSWC}
We use notation from \S \ref{sec:unitary}.
Assume that $p\nmid 2n$.
Let $\rbar: G_{F^+} \ra \cG_n(\F)$ be a continuous representation.
\begin{enumerate}
\item
\label{it:patchSWC:1}
If $\rbar$ is automorphic and $\rbar|_{G_F}(G_{F(\zeta_{p})})$ is adequate, then there exists a weak patching functor for $\rbar_p$ such that $M_\infty(\sigma) \neq 0$ if and only if $\sigma\in W(\rbar)$.
\item 
\label{it:patchSWC:2}
If furthermore, the inertial $L$-parameter $\rbar_p|_{I_{\Qp}}$ over $\F$ is tame and has a lowest alcove presentation $(s,\mu-\eta)$ such that $\mu$ is $P_{6n-2}P_{2\eta,e}$-generic, there exists $M_\infty$ as above such that $M_\infty$ is furthermore detectable. 
\end{enumerate}
\end{lemma}

\subsection{Patching functors and obvious weights}

Let $\rhobar$ be a tame $L$-homomorphism over $\F$.
Let $i_\rhobar: \Spf R_\rhobar \ra \prod_{v\in S_p,\Spf\cO}\cX_n^{F^+_v}$ be the versal map from \S \ref{sec:unitary}.
Recall that for any Serre weight $\sigma$ of $\rG$, we let $\cC_\sigma(\rhobar)$ be the irreducible cycle $i_\rhobar^*(\cC_\sigma)$ and let $\fp_\sigma(\rhobar) \subset R_\rhobar$ denote the prime ideal defining $\cC_\sigma(\rhobar)$ (and let $\fp_\sigma(\rhobar) = R_\rhobar$ if $\cC_\sigma(\rhobar) = 0$).
For an inertial type $\tau$, let $\fp_{\lambda,\tau}(\rhobar) \subset R_\rhobar$ denote the ideal defining $\Spec R_\rhobar^{\lambda+\eta,\tau}$.
We write $\fp_\tau(\rhobar)$ for $\fp_{0,\tau}(\rhobar)$.

\begin{lemma}\label{lemma:patchcover}
Let $\rhobar$ be an $L$-homomorphism over $\F$ and $M_\infty$ a weak patching functor for $\rhobar$.
If $\sigma_0$ is $(3n-1)$-deep and $\mathrm{Ann}_{R_\rhobar} M_\infty(\sigma_0) \subset \fp_\sigma(\rhobar)$, then $\sigma_0$ covers $\sigma$.
\end{lemma}
\begin{proof}
Suppose that $\mathrm{Ann}_{R_\rhobar} M_\infty(\sigma_0) \subset \fp_\sigma(\rhobar)$ and that $\tau$ is a $2n$-generic tame inertial $L$-parameter $\tau$ with $\sigma_0 \in \JH(\ovl{\sigma}(\tau))$.
Then by Definition \ref{minimalpatching}, 
\[
\fp_\sigma(\rhobar) \supset \mathrm{Ann}_{R_\rhobar} M_\infty(\sigma_0)\supset \mathrm{Ann}_{R_\rhobar} M_\infty(\ovl{\sigma}^\circ(\tau)) \supset \fp_\tau(\rhobar)+(\varpi).
\]
Remark \ref{rmk:MLM:sp:fib}(\ref{it:MLM:compmatch}) then implies that $\sigma \in \JH(\ovl{\sigma}(\tau))$.
We conclude that $\sigma_0$ covers $\sigma$.
\end{proof}

We say that a tame $L$-homomorphism $\rhobar$ over $\F$ is $P_{2\eta,e}$-generic if the inertial $L$-parameter $\tau$ with $\tld{w}(\rhobar,\tau) = t_{\eta+w_0(\eta)}$ is $P_{2\eta,e}$-generic.

\begin{prop}\label{prop:patchobv} 
Suppose that $\rhobar$ is a tame $L$-homomorphism over $\F$ such that $\rhobar|_{I_{\Qp}}$ has a lowest alcove presentation $(s,\mu-\eta)$ such that $\mu$ is $P_{2\eta,e}$-generic.
If $M_\infty$ is a weak patching functor, then $M_\infty(\sigma) \neq 0$ for every $\sigma\in W_{\obv}(\rhobar)$ in the highest $p$-restricted alcove.
\end{prop}
\begin{proof}
Let $\tau$ be the tame inertial $L$-parameter with a lowest alcove presentation $\eta$-compatible with that of $\rhobar$ such that $\tld{w}(\tau) = \tld{w}(\rhobar)t_{-\eta-w_0\eta}$.
If $M_\infty$ is a weak patching functor for $\rhobar$, then $M_\infty(\sigma^\circ(\eta,\tau))$ is nonzero for any lattice $\sigma^\circ(\eta,\tau) \subset \sigma(\eta,\tau)$ by Proposition \ref{prop:WE} and Lemma \ref{lemma:connectSW}. %
By assumption, $\tau$ has a lowest alcove presentation $(s,\mu-2\eta-w_0\eta)$ where $\mu-\eta-w_0\eta$ is (up to $X^0(\un{T})$) $P_{2\eta,e}$-generic, %
so that $\mathrm{Ann}_{R_\rhobar}M_\infty(\sigma^\circ(\eta,\tau)) = \fp_{\eta,\tau}(\rhobar)$ since $R^{2\eta,\tau}_\rhobar$ is a domain by Theorem \ref{thm:stack_local_model} and since $M_\infty(\sigma^\circ(\eta,\tau))$ is maximally Cohen--Macaulay over $R_\infty(\eta,\tau)$ by Definition \ref{minimalpatching}(\ref{it:minimalpatching:1}).
Since $M_\infty(\sigma^\circ(\eta,\tau))$ and $R^{2\eta,\tau}_\rhobar$ are $\cO$-flat, this implies that $\mathrm{Ann}_{R_\rhobar}M_\infty(\ovl{\sigma}^\circ(\eta,\tau)) = \fp_{\eta,\tau}(\rhobar)+ (\varpi)$.
In particular, $\mathrm{Ann}_{R_\rhobar}M_\infty(\ovl{\sigma}^\circ(\eta,\tau)) \subset \fp_\sigma(\rhobar)$ for any $\sigma \in \JH(\ovl{\sigma}(\eta,\tau))$ by Theorem \ref{thm:irreducible_components_mod_p}.
Then for any $\sigma \in \JH(\ovl{\sigma}(\eta,\tau))$, there exists a Serre weight $\sigma'$ such that  $\mathrm{Ann}_{R_\rhobar}M_\infty(\sigma')\subset \fp_{\sigma}(\rhobar)$.

We now take $\sigma \in W_{\obv}(\rhobar)$ in the highest $p$-restricted alcove.  %
Then $\rhobar \in \cC_\sigma$ by Proposition \ref{prop:rhobaroncomponents}(\ref{item:comp:obv}).
Lemma \ref{lemma:patchcover} implies that $M_\infty(\sigma') \neq 0$ for some $\sigma'$ which covers $\sigma$.
Then Proposition \ref{prop:covering}(\ref{it:covering:2}) implies that $\sigma' = \sigma$, so that $M_\infty(\sigma)\neq 0$.
\end{proof}

\subsection{Galois deformations}\label{sec:galdef}

We recall some some definitions from \S \ref{sec:deforms}.
Let $\cG_{/\cO}$ be a split (possibly disconnected) reductive group.
Let $\mathcal{C}_\cO$ be the category with objects Noetherian complete local $\cO$-algebras with residue field $\F$ and morphisms local $\cO$-algebra homomorphisms.
Given a topological group $\Gamma$, a continuous representation $\rbar: \Gamma \ra \cG(\F)$, and $(A,\fm_A) \in \mathcal{C}_\cO$, an $A$-valued lifting of $\rbar$ is a continuous representation $r_A: \Gamma \ra \cG(A)$ such that $\rbar \equiv r_A \pmod{\fm_A}$.
We say that two $A$-valued liftings are \emph{equivalent} if they are $\ker(\cG(A) \onto \cG(\F))$-conjugate.
An $A$-valued deformation of $\rbar$ is an equivalence class of $A$-valued liftings.
Given a $A$-valued lifting $r_A: \Gamma \ra \cG(A)$, let $\det r_A: \Gamma \ra \cG^{\mathrm{ab}}(A)$ denote its composition with the natural quotient map.
Note that $\det r_A$ only depends on the equivalence class of $r_A$.

An example of $\cG$ which will play an important role in what follows is the group scheme $\cG_n$ from \cite{CHT}, which is the (disconnected) split reductive group scheme over $\Z$ defined as the semidirect product
\[ 
(\GL_n \times \GL_1) \rtimes \{1,\jmath\} = \cG_n^0 \rtimes \{1,\jmath\},
\]
where $\jmath (g,a) \jmath = (a\, ^tg^{-1},a)$.
Let $\nu : \cG_n \rightarrow \GL_1$ be the homomorphism defined by $\nu(g,a) = a$ and $\nu(\jmath) = -1$.
Let $\cG_n^{\mathrm{ab}}$ be the quotient of $\cG_n$ by its derived subgroup.
Then $\cG_n^{\mathrm{ab}}$ is isomorphic to $\GL_1 \times \{1,\jmath\}$ (see \cite[\S 5.1]{BG}).
In the next sections, $\Gamma$ will be taken to be a Galois group.

\subsubsection{Local deformations} \label{sec:locdef}

Let $L$ be a nonarchimedean local field of characteristic zero.
For a Galois representation $\rhobar: G_L \ra \cG_n(\F)$, define the functor $D^\square_\rhobar: \mathcal{C}_\cO\ra\mathrm{Sets}$ by letting $D_\rhobar(A)$ be the set of $A$-valued liftings of $\rbar$.
Then $D^\square_\rhobar$ is represented by a ring $R_\rhobar^\square$, the $\cO$-lifting ring of $\rhobar$.
\begin{defn}
A local deformation problem for $\rhobar$ is a nontrivial subfunctor $D_\rhobar$ of $D^\square_\rhobar$ such that 
\begin{enumerate}
\item if $R_1 \onto R_0$ and $R_2 \onto R_0$ are two surjections in $\mathcal{C}_\cO$ and $R_3 = \ker(R_1 \times R_2 \ra R_0) \in \mathcal{C}_\cO$ with the natural ring structure, then $D_\rhobar(R_3)$ is identified with the equalizer of the diagram
\[
\begin{tikzcd}
D_\rhobar(R_1) \times D_\rhobar(R_2) \ar[shift left=.75ex]{r} \ar[shift right=.75ex]{r} & D_\rhobar(R_0);
\end{tikzcd}
\]
\item $D_\rhobar(A)$ is $\ker(\cG(A) \onto \cG(\F))$-conjugation invariant for all $A \in \mathcal{C}_\cO$;
\item the natural map $D_\rhobar(\varprojlim A_i) \risom \varprojlim D_\rhobar(A_i)$ is an isomorphism; and
\item if $i: A \into B$ is an injection in $\mathcal{C}_\cO$, then $r_A \in D_\rhobar(A)$ if and only if $i_*(r_A) \in D_\rhobar(B)$. 
\end{enumerate}
\end{defn}
\noindent Any local deformation problem $D_\rhobar$ is represented by a quotient $R_\rhobar$ of $R^\square_\rhobar$.

If $\xi: G_L \ra \cG_n^{\mathrm{ab}}(\cO)$ is a lift of $\det \rhobar: G_L \ra \cG_n^{\mathrm{ab}}(\F)$, define $D_\rhobar^{\xi,\prime}(A)$ to be the set of lifts $\rho_A$ such that $\det \rho_A=\xi$ (i.e.~$\xi$ composed with the map coming from the structure map $\cO \ra A$).
Then $D_\rhobar^{\xi,\prime}$ is a local deformation problem represented by a ring $R_\rhobar^{\xi,\prime}$.
Let $R_\rhobar^\xi$ be the maximal $\cO$-flat quotient of $R_\rhobar^{\xi,\prime}$, and let $D_\rhobar^\xi$ be the corresponding local deformation problem.

If $\rhobar(G_L)$ is contained in $\cG^0_n(\F)$, then denote the projection of a deformation $\rho_A$ to $\GL_n(A)$ by $\rho_A|$.
Then $\rho_A \mapsto \rho_A|$ induces a natural isomorphism $D_\rhobar^\xi \risom D^\square_{\rhobar|}$.

\subsubsection{Global deformations}\label{sec:globaldeforms}

Let $F^+$ be a totally real extension of $\Q$ and let $F\subset \ovl{F}^+$ be a CM extension of $F^+$. 
There is a natural inclusion $G_F \subset G_{F^+}$.
Let $S$ be a finite set of finite places of $F^+$.
Let $F(S)$ be the maximal extension of $F$ unramified outside $S$, and let $G_{F^+,S}$ be $\Gal(F(S)/F^+)$.
Let $\rbar: G_{F^+,S} \ra \cG_n(\F)$ be a representation which induces an isomorphism $G_{F^+}/G_F \risom \cG_n(\F)/\cG_n^0(\F)$.
(All representations $\rbar: G_{F^+,S} \ra \cG_n(A)$ below are assumed to induce the isomorphism $G_{F^+}/G_F \risom \cG_n(A)/\cG_n^0(A)$.)
Fix a lift $\xi: G_{F^+} \ra \cG_n^{\mathrm{ab}}(\cO)$ of $\det \rbar$.
Let $D_{\rbar}^{\square,\xi}$ denote the functor taking $A \in \mathcal{C}_\cO$ to the set of $A$-valued lifts with $\det r_A = \xi$, which is represented by a quotient $R_{\rbar}^{\square,\xi}$ of $R_{\rbar}^\square$.
We let $\rbar|_{G_F}$ denote the restriction of $\rbar$ to $G_F$ composed with the projection to $\GL_n(\F)$.
Suppose now that $\rbar|_{G_F}$ is absolutely irreducible.
Then the functor $D_{\rbar}^\xi$, taking $A \in \mathcal{C}_\cO$ to the set of equivalence classes in $D_{\rbar}^{\square,\xi}(A)$, is represented by a deformation ring $R_{\rbar}^\xi$.

For each place $v$ of $F^+$, we fix a map $\ovl{F}^+\into \ovl{F}^+_v$.
Then restriction gives an inclusion $G_{F^+_v} \into G_{F^+}$.
A global $\cG_n$-deformation datum is a tuple
\[
\cS = (F/F^+,S,\cO,\rbar,\xi,\{D_v\}_{v\in S}),
\]
where $F/F^+$, $S$, $\cO$, $\rbar$, and $\xi$ are as before, and $D_v$ corresponds to a local deformation problem for $\rbar_v \defeq \rbar|_{G_{F^+_v}}$ which is a subfunctor of $D_{\rbar_v}^{\xi_v}$ where $\xi_v \defeq \xi|_{G_{F^+_v}}$. %
For an $\cO$-algebra $A$, we say that a lifting $r_A: G_{F^+} \ra \cG_n(A)$ of $\rbar$ is of type $\cS$ if $\det r_A = \xi$ and $r_{A,v} \defeq r_A|_{G_{F^+_v}} \in D_v(A)$ for all $v\in S$.
We say that a deformation $[r_A]$ of $\rbar$ is of type $\cS$ if some (or equivalently any) lifting in the equivalence class is of type $\cS$.
Let $D^\square_\cS \subset D_{\rbar}^{\square,\xi}$ (resp.~$D_\cS \subset D_{\rbar}^\xi$) be the subfunctor consisting of liftings (resp.~deformations) of type $\cS$.
Then $D_\cS^\square$ (resp.~$D_\cS$) is represented by a quotient $R^\square_\cS$ of $R_{\rbar}^{\square,\xi}$ (resp.~a quotient $R_\cS$ of $R_{\rbar}^\xi$).

For $T \subset S$, an $A$-valued $T$-framed lifting of $\rbar$ of type $\cS$ is a tuple $(r_A,(\alpha_v)_{v\in T})$ where $r_A \in D^\square_\cS(A)$ and $\alpha_v \in \ker(\GL_n(A) \onto \GL_n(\F))$ for each $v\in T$.
If we let $\cO_T$ be $\cO[\![z_{v,i,j}]\!]_{v\in T,1\leq i,j\leq n}$, then the functor sending $A$ to the set of $A$-valued $T$-framed liftings of $\rbar$ of type $\cS$ is represented by the ring $R^{\square,\square_T}_\cS \cong R^\square_\cS \widehat{\otimes}_\cO \cO_T$.
We say that $A$-valued $T$-framed liftings $(r_A,(\alpha_v)_{v\in T})$ and $(r'_A,(\alpha'_v)_{v\in T})$ of $\rbar$ of type $\cS$ are \emph{equivalent} if for some $\beta \in \ker(\GL_n(A) \onto \GL_n(\F))$, $r'_A = \beta r_A \beta^{-1}$ and $\alpha'_v = \beta \alpha_v$ for all $v\in T$.
An $A$-valued $T$-framed deformation of $\rbar$ of type $\cS$ is an equivalence class of $A$-valued $T$-framed liftings of $\rbar$ of type $\cS$.
The functor of $T$-framed deformations of $\rbar$ of type $\cS$ is represented by a ring $R_\cS^{\square_T}$.
Taking equivalence classes gives a tautological map $R_\cS^{\square_T} \ra R_\cS^{\square,\square_T}$.
The maps sending $[(r_A,(\alpha_v)_{v\in T})]$ to $(\mathrm{Ad}(\alpha_v^{-1})r_{A,v})_{v\in T}$ induces a map $R^{\mathrm{loc}}_{\cS,T}\defeq \widehat{\otimes}_{v\in T,\cO} R_v \ra R_\cS^{\square_T}$ where $R_v$ denotes the ring representing $D_v$.

Fix a universal lifting $r_\cS: G_{F^+} \ra \cG_n(R_\cS)$ (or equivalently a section $\Spec R_\cS \ra \Spec R^\square_\cS$ of the natural map $\Spec R^\square_\cS \ra \Spec R_\cS$).
This induces a map $R^\square_\cS \widehat{\otimes}_\cO \cO_T \ra R_\cS \widehat{\otimes}_\cO \cO_T$ and the composition $R^{\square_T}_\cS \ra R^{\square,\square_T}_\cS \cong R^\square_\cS \widehat{\otimes}_\cO \cO_T \ra R_\cS \widehat{\otimes}_\cO \cO_T$ is an isomorphism.
(Indeed, since $\rbar|_{G_F}$ is absolutely irreducible and $p>2$, $\beta \in \ker(\GL_n(A) \onto \GL_n(\F))$ centralizes $\rbar$ if and only if $\beta = \id_n$.)

\subsubsection{Tangent spaces}

Given a representation $r:\Gamma \ra \cG_n(A)$, one naturally obtains an adjoint representation $\Gamma \ra \Aut_A(\Lie \cG_n(A))$. 
Note that $\Lie \cG_n \cong \mathfrak{gl}_n \times \mathfrak{gl}_1$.
Let $\ad\, r: \Gamma \ra \Aut_A(\mathfrak{gl}_n(A))$ be the representation obtained by the projection $\Lie \cG_n \onto \mathfrak{gl}_n$.

For $(A,\fm_A) \in \cO$, the reduced tangent space of $A$ is defined to be $\Hom_\cO(\fm_A/\fm_A^2,\F)$, which is naturally identified with the set of morphisms $A \ra \F[\eps]/\eps^2$ in $\mathcal{C}_\cO$.
In the setup of \S \ref{sec:locdef}, the reduced tangent space of $R^\square_\rhobar$ is naturally identified with both $D^\square_\rhobar(\F[\eps]/\eps^2)$ and $C^1(G_{F^+_v},\ad\, \rhobar)$.

Recall that $\rbar: G_{F^+,S} \ra \cG_n(\F)$ is a representation which induces an isomorphism $G_{F^+}/G_F \risom \cG_n(\F)/\cG_n^0(\F)$ and whose restriction $\rbar|_{G_F}$ is absolutely irreducible.
Fix a global $\cG_n$-deformation datum
\[
\cS = (F/F^+,S,\cO,\rbar,\xi,\{D_v\}_{v\in S}).
\]
For each $v\in S$, let $L_v \subset C^1(G_{F^+_v},\ad\, \rhobar)$ be the subspace corresponding to $D_v(\F[\eps]/\eps^2)$.

As before, let $T$ be a subset of $S$.
We define $H^i_{\cS,T}(G_{F^+,S},\ad\, \rbar)$ to be the cohomology of the complex
\[
C^i_{\cS,T}(G_{F^+,S},\ad\, \rbar) \defeq C^i(G_{F^+,S},\ad\, \rbar) \oplus \bigoplus_{v\in S} C^{i-1}(G_{F^+_v},\ad\, \rbar)/M_v^{i-1},
\]
where $M_v^i = 0$ unless $v\in S \setminus T$ and $i=0$ in which case $M_v^0 = C^0(G_{F^+_v},\ad\, \rbar)$ or $v\in S \setminus T$ and $i=1$ in which case $M_v^1 = L_v$.
The boundary map for the above complex maps $(\phi,(\psi_v)_{v\in S})$ to $(\partial \phi,(\phi|_{G_{F^+_v}}-\partial \psi_v)_{v\in S})$.

\begin{prop}\label{prop:localgenglobal}
There is a natural isomorphism 
\[
\Hom_\cO(\fm_{R^{\square_T}_\cS}/(\fm^2_{R^{\square_T}_\cS}+\fm_{R^{\mathrm{loc}}_{\cS,T}}),\F) \cong H^1_{\cS,T}(G_{F^+,S},\ad\, \rbar).
\]
\end{prop}

\subsubsection{Taylor--Wiles primes}\label{sec:TWprimes}
Let 
\[
\cS = (F/F^+,S,\cO,\rbar,\xi,\{D_v\}_{v\in S}),
\]
be a global $\cG_n$-deformation datum.
Let $Q$ be a set of split places in $F^+$ such that $\mathbf{N} v \equiv 1 \pmod{p}$ for all $v\in Q$, and let $\overline{\psi}_v$ be a generalized eigenspace for the projection of $\rbar(\Frob_v)$ to $\GL_n(\F)$ on which $\rbar(\Frob_v)$ acts semisimply.
Let $\overline{s}_v$ be the complementary $\rbar(\Frob_v)$-stable subspace.
For $v\in Q$, let $D_v(A)$ be the set of $A$-liftings which induce $G_{F^+_v}$-actions of $A^n$ which decompose as $s_v\oplus \psi_v$ lifting the decomposition over $\F$ such that $s_v$ is unramified and the inertial subgroup acts on $\psi_v$ by scalars.
Then $D_v$ is a local deformation problem by \cite[Lemma 4.2]{Thorne}, and we consider the global $\cG_n$-deformation datum
\[
\cS_Q = (F/F^+,S,\cO,\rbar,\xi,\{D_v\}_{v\in S \cup Q}).
\]

\begin{prop}\label{prop:TWprime}
Let $q_0\geq 0$ be an integer and 
\[
\cS = (F/F^+,S,\cO,\rbar,\xi,\{D_v\}_{v\in S}),
\]
be a global $\cG_n$-deformation datum such that $\rbar(G_{F^+(\zeta_p)})$ is adequate and $\xi(c_v) = -1$ for each $v \mid \infty$ where $c_v$ denotes complex conjugation at $v$.
Let $T \subset S$ be a finite set such that every place in $S\setminus T$ is splits in $F$ and that 
\[
\dim_{\F} L_v - \dim_{\F} H^0(G_{F^+_v},\ad\, \rbar) = 
\begin{cases}
[F^+_v:\Q_p]n(n-1)/2 & \textrm{if } v\mid p \\
0 & \textrm{if } v\nmid p.
\end{cases}
\]
Let $q$ be the larger of $\dim_{\F} H^1_{\mathcal{L}^\perp,T}(G_{F^+,S},\ad\, \rbar(1))$ and $q_0$ \emph{(}with $H^1_{\mathcal{L}^\perp,T}(G_{F^+,S},\ad\, \rbar(1))$ defined as in \cite[\S 2.3]{CHT}\emph{)}.
Then for any integer $N \geq 0$, we can find $(Q,(\overline{\psi}_v)_{v\in Q})$ where $Q$ is a set of places in $F^+$ which split in $F$ which is disjoint from $S$ and $\overline{\psi}_v$ is a nontrivial generalized eigenspace for $\rbar(\Frob_v)$ on which $\rbar(\Frob_v)$ acts semisimply for each $v\in Q$ such that
\begin{itemize}
\item $\# Q = q$;
\item $\mathbf{N} v \equiv 1 \pmod{p^N}$ for all $v\in Q$; and 
\item $R^{\square_T}_{\cS_Q}$ can be topologically generated over $R^{\mathrm{loc}}_{\cS,T} = R^{\mathrm{loc}}_{\cS_Q,T}$ by $q - \sum_{v \in T, v\mid p} [F^+_v:\Q_p]n(n-1)/2$ elements.
\end{itemize}
\end{prop}
\begin{proof}
This follows from \cite[Proposition 4.4]{Thorne}.
\end{proof}

We say that $(Q,(\overline{\psi}_v)_{v\in Q})$ in Proposition \ref{prop:TWprime} is a \emph{Taylor--Wiles datum of level} $N$ disjoint from $S$.

With $Q$ as above, let $\Delta_Q$ be $\prod\limits_{v\in Q} k_v^\times(p)$ where $k_v$ denotes the residue field of $F^+_v$ and $k_v^\times(p)$ denotes the maximal $p$-quotient of $k_v^\times$. 
(So $\Delta_Q$ is nontrivial if $Q$ is nonempty.)
Choose a universal lifting $r_{\cS_Q}$ and let $\psi_v$ be as above for each $v\in Q$.
For each $v\in Q$, the action of $k_v^\times$, thought of as a subgroup of $I_v^{\mathrm{ab}}$, acts on the summand $\psi_v$ and gives a character $k_v^\times \ra R_{\cS_Q}^\times$ which factors through $k_v^\times(p)$. 
Altogether, we have a map $\cO[\Delta_Q] \ra R_{\cS_Q}$.
Moreover, the natural map $R_{\cS_Q}/\mathfrak{a}_Q \ra R_{\cS}$ is an isomorphism, where $\mathfrak{a}_Q \subset \cO[\Delta_Q]$ denotes the augmentation ideal.
Similarly, $R_{\cS_Q}^{\square_T}/\mathfrak{a}_Q \ra R_{\cS}^{\square_T}$ is an isomorphism.

\subsection{Automorphic forms on definite unitary groups}  \label{sec:autforms}

For the reader's convenience, we recall notation from \S \ref{sec:unitary}.
Recall that $F^+$ is a totally real field not equal to $\Q$ and that $F \subset \ovl{F}^+$ is a CM extension of $F^+$.
The set of places in $F^+$ dividing $p$ is denoted $S_p$.
A finite place of $F^+$ (resp.~of $F$) is \emph{split} if it splits in $F$ (resp.~if its restriction to $F^+$ splits in $F$).
We assume that all places in $S_p$ are split.
Recall that $G_{/F^+}$ is an outer form for $\GL_n$ which 
\begin{itemize}
\item splits over $F$; and
\item is definite at all archimedean places.
\end{itemize}
Moreover, there is an $N\in \N$ prime to $p$ and a reductive model $\cG_{/\cO_{F^+}[1/N]}$ for $G$ with an isomorphism
\begin{equation}
\label{iso integral}
\iota:\,\cG_{/\cO_{F}[1/N]} \stackrel{\iota}{\rightarrow}{\GL_n}_{/\cO_{F}[1/N]}
\end{equation}
which specializes to
$
\iota_w:\,\cG(\cO_{F^+_v})\stackrel{\sim}{\rightarrow}\cG(\cO_{F_w})\stackrel{\iota}{\rightarrow}\GL_n(\cO_{F_w})
$
for all split finite places $w$ in $F$ where $v$ is $w|_{F^+}$ here.

In \S \ref{sec:globaldeforms}, we chose homomorphisms $\ovl{F}^+ \into \ovl{F}^+_v$, which induces a $v$-adic norm on $\ovl{F}^+$ (for the unique norm on $\ovl{F}^+_v$ extending any fixed norm on $F^+_v$ in the class of the place $v$).
Restriction to $F$ gives a place $\tld{v}$ dividing $v$ (that does not depend on the choice of the norm on $F^+_v$).
Changing the homomorphisms $\ovl{F}^+ \into \ovl{F}^+_v$, we assume without loss of generality, that $\tld{v}$ coincides with the choices in \S \ref{sec:unitary}.
We write $\iota_v$ be the composition of $\iota_{\tld{v}}$ and the canonical isomorphism $\GL_n(\cO_{F_{\tld{v}}}) \cong \GL_n(\cO_{F^+_v})$ (suppressing the dependence on the choice of $\tld{v}$).

If $U = U_p U^{\infty,p} \leq G(\A_{F^+,p}^{\infty}) \times G(\A_{F^+}^{\infty,p})$ is a compact open subgroup, and $W$ is a finite $\cO$-module endowed with a continuous action of $U_p$, then 
\begin{equation}
S(U,W) \defeq \left\{f:\,G(F^{+})\backslash G(\A^{\infty}_{F^{+}})\rightarrow W\,|\, f(gu)=u_p^{-1}f(g)\,\,\forall\,\,g\in G(\A^{\infty}_{F^{+}}), u\in U\right\}.
\end{equation}
From now on we assume that $U$ is sufficiently small, i.e.~for all $t \in G(\bA^{\infty}_{F^+})$, the order of the finite group $t^{-1} G(F^+) t \cap U$ is prime to $p$, so that $S(U,-)$ is exact.

We let $S$ be a finite set of finite places in $F^+$ containing $S_p$, places dividing $N$, and all places at which $U$ is not unramified; and we let $\cP_S$ be the set of split finite places $w$ of $F$ such that $v = w|_{F^+} \notin S$. 
For a subset $\cP\subseteq \cP_S$ of finite complement that is closed under complex conjugation, $\bT_{\cP}=\cO[T^{(i)}_w,\,\,w\in\cP,\, 0 \leq i \leq n]$ is the universal Hecke algebra on $\cP$.
Then $T_w^{(i)}\in \bT_{\cP}$ acts on $S(U,W)$ by the usual double coset operator
\[
\iota_w^{-1}\left[ \GL_n(\cO_{F_w}) \left(\begin{matrix}
      \varpi_{w}\mathrm{Id}_i & 0 \cr 0 & \mathrm{Id}_{n-i} \end{matrix} \right)
\GL_n(\cO_{F_w}) \right].
\]

Suppose that $S(U,W)_{\fm}\neq 0$ where $\fm$ is the kernel of a homomorphism $\alpha: \bT_{\cP}\rightarrow \F$.
Let $\bT_{\cP}(U,W)$ be the image of $\bT_{\cP}$ in $\End_{\cO}(S(U,W))$.
If $Q$ is the (finite) set $\{w|_{F^+}:w\in \cP_S \setminus \cP\}$, then we also denote $\bT_{\cP}(U,W)$ by $\bT^Q(U,W)$.
Let $\alpha: \bT_{\cP} \onto \bT^Q(U,W)_{\fm}$ be the natural quotient map.
Then there is a Galois representation $r_{\fm} \defeq r(U,W)_{\fm}: G_{F^+,S} \ra \cG_n(\bT^Q(U,W)_{\fm})$ determined by the equations
\[
\det\left(1-r(U,W)_{\fm}|_{G_F}(\mathrm{Frob}_w)X\right)=\sum_{j=0}^n (-1)^j(\mathbf{N}_{F/\Q}(w))^{\binom{j}{2}}\alpha(T_w^{(j)})X^j
\]
for all $w\in \cP$.
We denote the reduction $r_{\fm} \pmod{\fm}$ by $\rbar: G_{F^+} \ra \cG_n(\F)$.

Let $\cO_p$ be $\cO_{F^+} \otimes_{\Z} \Z_p \cong \prod\limits_{v \in S_p} \cO_{F^+_v}$.
Then the composition
\begin{equation}
\iota_p \defeq \prod_{v\in S_p} \iota_v: \cG(\cO_p) \cong \prod_{v\in S_p} \cG(\cO_{F^+_v}) \risom \prod_{v\in S_p} \GL_n(\cO_{F^+_v})
\end{equation}
gives an equivalence between $\cG(\cO_p)$-modules and $\prod\limits_{v\in S_p} \GL_n(\cO_{F^+_v})$-modules.
Let $k_v$ denote the residue field of $F^+_v$ and $\rG \defeq \prod\limits_{v\in S_p} \GL_n(k_v)$.
If $\sigma$ is a Serre weight of $\rG$, then $\sigma$ is naturally a $\prod\limits_{v\in S_p} \GL_n(\cO_{F^+_v})$-module by inflation.

We will need a local-global compatibility result for $r(U,W)_{\fm}$.
Let $\underline{G}$ be (the split group) $(\Res_{F^+ \otimes \Q_p/\Q_p} \GL_n)_{/E}$.
\begin{itemize}
\item %
Fix a highest weight $\lambda = (\lambda_v)_{v\in S_p}$ of $\un{G}$, which we also view as a coweight of the dual group $\un{G}^\vee$.
For $v\in S_p$, let $\tau_v$ be an $n$-dimensional Weil--Deligne inertial type for $F^+_v$.
Recall that there is a natural correspondence between local deformation problems for $\rhobar_v$ in $D_{\rhobar_v}^{\xi_v}$ and local deformation problems for $\rhobar_v|$.
Recall from \S \ref{sec:localdeforms} that $R_{\rhobar_v|}^{\lambda_v,\tau_v}$ represents a certain subfunctor, which we will denote $D_{\rhobar_v|}^{\lambda_v,\tau_v}$, of a potentially semistable deformation functor.
Let $D_{\rhobar_v}^{\lambda_v,\tau_v}$ be the local deformation problem corresponding to $D_{\rhobar_v|}^{\lambda_v,\tau_v}$. 
(Note that $D_{\rhobar_v}^{\lambda_v,\tau_v}$ does not depend on the choice of place $\tld{v}$.) \\
\item 
Let $m$ be a positive integer and $(Q,(\overline{\psi}_v)_{v\in Q})$ a Taylor--Wiles datum of level $m$ and disjoint from the union of $S_p$ and the set of places dividing $N$ (see \S \ref{sec:TWprimes}).
Let $d_v$ be the dimension of the generalized $\ovl{\psi}_v$-eigenspace.
Let $\fp_v$ be the standard (block upper triangular) parahoric corresponding to the partition $(n-d_v)+d_v$ of $n$ (suppressing the dependence on $\ovl{\psi}_v$).
Let $\fp_1^v$ be the kernel of the natural map $\fp^v \ra \GL_{d_v}(k_v) \overset{\det}{\ra} k_v^\times \ra k_v^\times(p)$, where $k_v$ is the residue field of $F^+_v$ and $k_v^\times(p) \defeq \Delta_v$ denotes the maximal $p$-quotient of $k_v^\times$.
Setting $U = U_Q U^Q$, let $U_0(Q)$ (resp.~$U_1(Q)$) be the compact open subgroup $(\prod\limits_{v\in Q} U_0(Q)_v) U^Q$ (resp.~$(\prod\limits_{v\in Q} U_1(Q)_v) U^Q$) where $U_0(Q)_v$ (resp.~$U_1(Q)_v$) is $\iota_v^{-1}(\fp^v)$ (resp.~$\iota_v^{-1}(\fp^v_1)$).
Let $D_v$ be the local deformation problem defined in \S \ref{sec:TWprimes}.
Note that for each $v\in Q$, the quotient $U_0(Q)_v/U_1(Q)_v$ is naturally identified with $\Delta_v$ so that $U_0(Q)/U_1(Q)$ is naturally identified with $\Delta_Q \defeq \prod_{v\in Q} \Delta_v$.
\end{itemize}

\begin{thm}\label{thm:localglobal}
Let $\xi$ be $\eps^{1-n} \delta^n_{F/F^+}$ where $\delta_{F/F^+}$ denotes the quadratic character of $G_{F^+}/G_F$.
Fix a dominant weight $\lambda = (\lambda_v)_{v\in S_p}$ of $\un{G}$ and, for each $v\in S_p$, an $n$-dimensional Weil--Deligne inertial type $\tau_v$ for $F^+_v$.
Let $\sigma(\tau_v)$ be as in Theorem \ref{thm:ILL}.
For $v\in S_p$, let $D_v$ be $D_{\rbar_v}^{\lambda_v,\tau_v}$. %
Let $W$ be an $\cO$-lattice in the $U_p$-module
\[
\bigotimes_{v\in S_p} \sigma(\lambda_v,\tau_v)^*  \circ \iota_v,
\]
where $(-)^*$ denotes the $E$-dual of an $E$-vector space.
Suppose that $S$ contains $S_p$, places dividing $N$, and all places where $U$ is not unramified. 

As in \S \ref{sec:TWprimes}, let $(Q,(\overline{\psi}_v)_{v\in Q})$ be a Taylor--Wiles datum disjoint from $S$.
Then there are a maximal ideal $\fm_Q\subset \bT_{\cP_S \setminus Q}$, maps $\alpha: \bT_{\cP_S \setminus Q} \ra \bT^Q(U,W)_{\fm_Q}$, $\alpha_0: \bT_{\cP_S \setminus Q} \ra \bT^\emptyset(U_0(Q),W)_{\fm_Q}$, and $\alpha_1: \bT_{\cP_S \setminus Q} \ra \bT^\emptyset(U_1(Q),W)_{\fm_Q}$, and Galois representations %
\begin{itemize}
\item $r(U,W)_{\fm_Q}: G_{F^+,S} \ra \cG_n(\bT^Q(U,W)_{\fm_Q})$, uniquely determined by the equations
\[
\det\left(1-r(U,W)_{\fm_Q}|_{G_F}(\mathrm{Frob}_w)X\right)=\sum_{j=0}^n (-1)^j(\mathbf{N}_{F/\Q}(w))^{\binom{j}{2}}\alpha(T_w^{(j)})X^j
\]
for all $w\in \cP_S \setminus Q$, of type 
\[
\cS \defeq (F/F^+,S,\cO,\rbar,\xi,\{D_v\}_{v\in S_p} \cup \{D_v^{\xi_v}\}_{v\in S\setminus S_p});
\]
\item $r(U_0(Q),W)_{\fm_Q}: G_{F^+,S} \ra \cG_n(\bT^\emptyset(U_0(Q),W)_{\fm_Q})$, uniquely determined by the equations
\[
\det\left(1-r(U_0(Q),W)_{\fm_Q}|_{G_F}(\mathrm{Frob}_w)X\right)=\sum_{j=0}^n (-1)^j(\mathbf{N}_{F/\Q}(w))^{\binom{j}{2}}\alpha_0(T_w^{(j)})X^j
\]
for all $w\in \cP_S \setminus Q$, of type 
\[
\cS = (F/F^+,S,\cO,\rbar,\xi,\{D_v\}_{v\in S_p} \cup \{D_v^{\xi_v}\}_{v\in S\setminus S_p});
\]
\item and $r(U_1(Q),W)_{\fm_Q}: G_{F^+,S} \ra \cG_n(\bT^\emptyset(U_1(Q),W)_{\fm_Q})$, uniquely determined by the equations
\[
\det\left(1-r(U_1(Q),W)_{\fm_Q}|_{G_F}(\mathrm{Frob}_w)X\right)=\sum_{j=0}^n (-1)^j(\mathbf{N}_{F/\Q}(w))^{\binom{j}{2}}\alpha_1(T_w^{(j)})X^j
\]
for all $w\in \cP_S \setminus Q$, of type 
\[
\cS_Q = (F/F^+,S \cup Q,\cO,\rbar,\xi,\{D_v\}_{v\in S_p} \cup \{D_v\}_{v\in Q} \cup \{D_v^{\xi_v}\}_{v\in S\setminus S_p}).
\]
\end{itemize}
\end{thm}
\begin{proof}
The construction of the Galois representations is as in the proof of \cite[Proposition 3.4.4]{CHT} using \cite[Theorem 7.2.1]{EGH}.
\end{proof}

If $U$ is sufficiently small, then with the natural action of $\Delta_Q$ on $S(U_1(Q),W)$, $S(U_1(Q),W)$ is a free $\cO[\Delta_Q]$-module and the image of $S(U_0(Q),W)$ in $S(U_1(Q),W)$ under the natural inclusion is identified with $S(U_1(Q),W)[\mathfrak{a}_Q]$.
Moreover, the induced action on $S(U_1(Q),W)_\fm$ coincides with the one given by the composition $\cO[\Delta_Q] \ra R_{\cS_Q} \ra \bT^\emptyset(U_1(Q),W)_{\fm} \ra \End_\cO(S(U_1(Q),W)_\fm)$. 

For each $v \in Q$, choose an element $\phi_v \in W_{F^+_v}$ lifting the geometric Frobenius element. 
Let $\varpi_v \in \cO_{F^+_v}$ be the uniformizer such that $\mathrm{Art}_{F^+_v}(\varpi_v)$ is the image of $\phi_v$ in $W_{F^+_v}^{\mathrm{ab}}$. 
Using the isomorphism $\iota_v: \cG(\cO_{F^+_v}) \risom \GL_n(F^+_v)$, we define $\pr_{\varpi_v} \in \End_{\cO}(S(U_i(Q),W)_{\fm_{Q}})$ as in \cite[Proposition 5.9]{Thorne} (suppressing the dependence on $U$ and $(Q,(\ovl{\psi}_v)_{v\in Q})$).
Then the operators $\pr_{\varpi_v}$ commute with each other and with the actions of $\cO[\Delta_{Q}]$ and $\bT^{\emptyset}(U_i(Q),W)_{\fm_Q}$ for $i=0,\, 1$. 
So letting $\pr = \prod_{v\in Q} \pr_{\varpi_v}$, $\pr(S(U_i(Q),W)_{\fm_Q})$ is well-defined for $i=0, 1$, $\pr(S(U_1(Q),W)_{\fm_Q})$ is a free $\cO[\Delta_Q]$-module, and the natural map
\[
\pr(S(U_0(Q),W)_{\fm_Q}) \risom \pr(S(U_1(Q),W)_{\fm_Q})[\mathfrak{a}_Q] 
\]
is an isomorphism.
Moreover, the image of the natural injection $S(U,W)_{\fm_Q} \ra S(U_0(Q),W)_{\fm_Q}$ is $\pr(S(U_0(Q),W)_{\fm_Q})$ as in the proof of \cite[Theorem 6.8]{Thorne}.

\subsection{The patching construction}\label{sec:patching}

We continue with the notation from \S \ref{sec:autforms}.
Let $\rbar: G_{F^+} \ra \cG_n(\F)$ be an automorphic Galois representation such that $\rbar(G_{F^+(\zeta_p)})$ is adequate (so that in particular $\rbar|_{G_F}$ is absolutely irreducible). 
By shrinking the level $U$, we can assume that $\rbar$ is automorphic of level $U = U_{S_p}U^{S_p}$ and coefficients $\cO$ with trivial $U_p$-action so that $(\prod_{v\in S_p} \cG(\cO_{F^+_v})) U^{S_p}$ is sufficiently small.
Let $S$ be a finite set of finite places of $F^+$ containing $S_p$, all places dividing $N$, and all places at which $U$ is not unramified.

Let $\cS$ be the global $\cG_n$-deformation datum
\[
\cS = (F/F^+,S,\cO,\rbar,\xi,\{D_v^{\xi_v}\}_{v\in S}).
\]
For each integer $m \geq 1$, let $(Q_m,(\ovl{\psi}_v)_{v\in Q_m})$ be as in Proposition \ref{prop:TWprime}.
For each $m$ and $v \in Q_m$, choose an element $\phi_v \in W_{F^+_v}$ lifting the geometric Frobenius element. 
Let $\varpi_v \in \cO_{F^+_v}$ be the uniformizer such that $\mathrm{Art}_{F^+_v}(\varpi_v)$ is the image of $\phi_v$ in $W_{F^+_v}^{\mathrm{ab}}$. 
Then for each $Q_m$, we define $\pr$ as in \S \ref{sec:autforms}.
For any open compact subgroup $K_p\subset U_{S_p}$ and integer $r>0$, we define 
\[
M_{m,K_p,r} \defeq \pr (S(K_p U_1(Q_m)^{S_p},W/\varpi^r)_{\fm_{Q_m}})^\vee/\mathfrak{a}_{Q_m}^r,
\]
where $(-)^\vee = \Hom^{\mathrm{cont}}_{\cO}(-,E/\cO)$ (with the compact open topology) denotes the Pontrjagin dual.
By Theorem \ref{thm:localglobal}, we have a (in fact surjective) map $R_{\cS_{Q_m}} \ra \bT^{Q_m}(K_p U_1(Q_m)^{S_p},W)_{\fm_{Q_m}}$.
Then $M_{m,K_p,r}$ is an $R_{\cS_{Q_m}}$-module, and we define
\[
M_{m,K_p,r}^\square \defeq M_{m,K_p,r}\otimes_{R_{\cS_{Q_m}}} R_{\cS_{Q_m}}^{\square_S}/\mathfrak{a}_S^r,
\]
where $\mathfrak{a}_S \subset \cO_S$ denotes the augmentation ideal of the formally smooth $\cO$-algebra $\cO_S$ defined in \S \ref{sec:globaldeforms}.
Let $\cO_\infty$ be $\cO[\![y_1,\ldots,y_q]\!]$. %
For each $m\in \N$ choose an ordering $v_1,\ldots,v_q$ of $Q_m$ and for each $v_i$ a generator $g_i$ of $\Delta_{v_i}$ which gives a surjection $\cO_\infty \onto \cO[\Delta_{Q_m}]$ mapping $y_i$ to $[g_i]-1 \in \cO[\Delta_{v_i}]$.
Let $S_\infty \defeq \cO_\infty \widehat{\otimes}_{\cO} \cO_S$ and let $\mathfrak{a}_\infty \subset S_\infty$ denote the augmentation ideal.
Then $S_\infty$ acts on $M_{m,K_p,r}^\square$ for all $N$, $K_p$, and $r$, and $M_{m,K_p,r}^\square/\mathfrak{a}_\infty$ is naturally identified with $S(K_p U^{S_p},W/\varpi^r)^\vee_\fm$.

We now patch our (dual) spaces of automorphic forms on $G$ as in \cite[\S 2]{CEGGPS}, in the language of ultrafilters following \cite[\S 9]{scholze}.
Choose a non-principal ultrafilter $\cF \subset 2^{\N}$.
Let $\mathbf{R} \defeq \prod_{m\in \N} \cO$ and $S_\cF\subset \mathbf{R}$ be the multiplicative set of idempotents $e_I = (e_{I,m})_{m\in \N}$ where $e_{I,m} = 1$ if $m\in I$ and $e_{I,m} = 0$ if $m\notin I$.
Then the diagonal map $\cO \ra \mathbf{R}$ induces an isomorphism
\[
\cO \cong \varprojlim_r S_\cF^{-1}\mathbf{R}/(\varpi^r),
\]
which gives a surjection $\mathbf{R} \onto \varprojlim_r S_\cF^{-1}\mathbf{R}/(\varpi^r) \cong \cO$.
Then we let
\[
M_\infty = \varprojlim_{K_p \subset U_{S_p},r} \cO \otimes_{\mathbf{R}} \prod_{m\in \N} M_{m,K_p,r}^\square.
\]
Through the diagonal map, $S_\infty$ acts on $M_\infty$ and 
\begin{equation}\label{eqn:control}
M_\infty/\mathfrak{a}_\infty \cong \varprojlim_{K_p \subset U_{S_p},r}S(K_p U^{S_p},W/\varpi^r)^\vee_\fm.
\end{equation}
Moreover, $(y_1,\ldots,y_q,z_1,\ldots,z_{n^2\# S})$ is an $M_\infty$-regular sequence, where $(z_1,\ldots,z_{n^2\# S})$ is any $\cO_S$-regular sequence.
Since $\cG(\cO_p)U^{S_p}$ is sufficiently small, $\varprojlim_{K_p \subset U_{S_p},r}S(K_p U^{S_p},W/\varpi^r)^\vee$ is a finite free $\cO[\![G_0(\Z_p)]\!]$-module.
By (\ref{eqn:control}), $M_\infty/\mathfrak{a}_\infty$ is a finitely generated projective and hence finitely generated maximal Cohen--Macaulay $\cO[\![G_0(\Z_p)]\!]$-module.
This implies that $M_\infty$ is a finitely generated maximal Cohen--Macaulay $S_\infty[\![G_0(\Z_p)]\!]$-module.
By \cite[Theorem 6.2]{venjakob}, $M_\infty$ is a finitely generated projective $S_\infty[\![G_0(\Z_p)]\!]$-module. 

Let $R_\infty$ be $R^{\mathrm{loc}}_{\cS,S}[\![x_1,\ldots,x_g]\!]$ where $g = q - [F:\Q]n(n-1)/2$.
Then for each $m\in \N$, we can and do choose surjections $R_\infty \onto R_{\cS_{Q_m}}^{\square_S}$ by Proposition \ref{prop:TWprime}.
We get a surjective map 
\[
R_\infty \ra \prod_{m\in \N} R_{\cS_{Q_m}}^{\square_S} \onto \cO \otimes_{\mathbf{R}} \prod_{m\in \N} R_{\cS_{Q_m}}^{\square_S},
\]
where the first map is the product of the above surjections (composed with the diagonal map).
Through this map, $M_\infty$ is an $R_\infty$-module.
The above $S_\infty$ action on $M_\infty$ factors through $\cO \otimes_{\mathbf{R}} \prod_{m\in \N} R_{\cS_{Q_m}}^{\square_S}$ as in \S \ref{sec:autforms}.
By formal smoothness of $S_\infty$, we can and do choose a lift 
\[
\begin{tikzcd}
\phantom{M} & R_\infty \ar[two heads]{d} \\
S_\infty \ar[dashed]{ur} \ar{r} & \cO \otimes_{\mathbf{R}} \prod_{m\in \N} R_{\cS_{Q_m}}^{\square_S}.
\end{tikzcd}
\]

Recall that we set $\cO_p\defeq \cO_{F^+} \otimes_{\Z} \Z_p \cong \prod\limits_{v \in S_p} \cO_{F^+_v}$.
Then $M_\infty$ has an natural $\cG(\cO_p)$-action (even a $\cG(F^+ \otimes_{\Q} \Qp)$-action though we will not use this), which can be thought of as a $\GL_n(\cO_p) \cong G_0(\Z_p)$-action via
\[
\iota_p \defeq \prod_{v\in S_p} \iota_v: \cG(\cO_p) \cong \prod_{v\in S_p} \cG(\cO_{F^+_v}) \risom \prod_{v\in S_p} \GL_n(\cO_{F^+_v}).
\]
Then we let 
\[
M_\infty(-) \defeq \Hom^{\mathrm{cont}}_{\cO[\![G_0(\Z_p)]\!]}(M_\infty,(-)^\vee)^\vee
\]
be the exact covariant functor from finite $\cO[G_0(\Z_p)]$-modules to finitely generated $R_\infty$-modules (finitely generated even over $S_\infty$).

\begin{proof}[Proof of Lemma \ref{lemma:patchSWC}]
In the construction of $M_\infty$ above, by shrinking $U^{S_p}$ we can assume without loss of generality that $W = \cO$ and that if $\rbar$ is modular of weight $\sigma$, then it is modular of weight $\sigma$ and level $\cG(\cO_p) U^{S_p}$.
We claim that $M_\infty$ constructed above is a weak patching functor.
Since $R_\infty \cong R_{\rbar_p} \widehat{\otimes}_{\cO} R^p$ where $R^p\defeq (\widehat{\bigotimes}_{v\in S\setminus S_p,\cO} R^{\xi_v}_v)[\![x_1,\ldots,x_g]\!]$ then $R^p$ is equidimensional by \cite[Theorem 3.3.3]{BG} or \cite[Theorem 1]{BoPa}.

To see that $M_\infty(-)$ is nonzero, note that
\begin{align*}
M_\infty(\Ind_{U_{S_p}}^{G_0(\Z_p)} \F)/\mathfrak{a}_\infty &\cong \Hom^{\mathrm{cont}}_{\cO[\![G_0(\Z_p)]\!]}(M_\infty/\mathfrak{a}_\infty,(\Ind_{U_{S_p}}^{G_0(\Z_p)} \F)^\vee)^\vee \\
&\cong \Hom^{\mathrm{cont}}_{\cO[\![G_0(\Z_p)]\!]}(\varprojlim_{K_p \subset U_{S_p},r}S(K_p U^{S_p},\cO/\varpi^r)^\vee_\fm,(\Ind_{U_{S_p}}^{G_0(\Z_p)} \F)^\vee)^\vee \\
&\cong \Hom^{\mathrm{cont}}_{\cO[\![G_0(\Z_p)]\!]}(\Ind_{U_{S_p}}^{G_0(\Z_p)} \F,\varinjlim_{K_p \subset U_{S_p}}S(K_p U^{S_p}, E/\cO)_\fm)^\vee \\
&\cong S(U,\F)^\vee_\fm \\
&\neq 0.
\end{align*}
If $\sigma$ is a Serre weight, then the same computation shows that $M_\infty(\sigma)/\mathfrak{a}_\infty \cong S(\cG(\cO_p) U^{S_p},\sigma^\vee)_\fm^\vee$.
By assumption, this latter space is nonzero if and only if $\sigma \in W(\rbar)$.
Furthermore, by Nakayama's lemma, $M_\infty(\sigma)$ is nonzero if and only if $M_\infty(\sigma)/\mathfrak{a}_\infty$ is nonzero.
We conclude that $M_\infty(\sigma)$ is nonzero if and only if $\sigma \in W(\rbar)$.

Let $\lambda\in X^*(\un{T})$ be a dominant weight and $\tau$ a Weil--Deligne inertial $L$-parameter.
For an $\cO$-lattice $\sigma^\circ(\lambda,\tau) \subset \sigma(\lambda,\tau)$, $M_\infty(\sigma^\circ(\lambda,\tau))$ is isomorphic to
\begin{align*}
 & \Hom^{\mathrm{cont}}_{\cO[\![G_0(\Z_p)]\!]}(M_\infty,\sigma^\circ(\lambda,\tau)^\vee)^\vee \\
\cong & \varprojlim_r \varprojlim_{K_p \subset U_{S_p}}\Hom_{G_0(\Z_p)}(\cO \otimes_{\mathbf{R}} \prod_{m\in \N} M_{m,K_p,r}^\square,(\sigma^\circ(\lambda,\tau)/\varpi^r)^\vee)^\vee \\
\cong & \varprojlim_r \varprojlim_{K_p \subset U_{S_p}}\Hom_{G_0(\Z_p)}(\sigma^\circ(\lambda,\tau)/\varpi^r,\cO \otimes_{\mathbf{R}} \prod_{m\in \N} \pr (S(K_p U_1(Q_m)^{S_p},\cO/\varpi^r)_{\fm_{Q_m}}) \otimes_{R_{\cS_{Q_m}}} R_{\cS_{Q_m}}^{\square_S}/\mathfrak{a}_S^r)^\vee \\
\cong & \varprojlim_r \cO \otimes_{\mathbf{R}} \prod_{m\in \N} \pr (S(\cG(\cO_p) U_1(Q_m)^{S_p}, (\sigma^\circ(\lambda,\tau)/\varpi^r)^\vee)_{\fm_{Q_m}}) \otimes_{R_{\cS_{Q_m}}} R_{\cS_{Q_m}}^{\square_S}/\mathfrak{a}_S^r)^\vee. 
\end{align*}
So the action of $R_\infty$ on $M_\infty(\sigma^\circ(\lambda,\tau))$ factors through $R_\infty(\lambda,\preceq\tau)$ by Theorem \ref{thm:localglobal}. 
Moreover, $M_\infty(\sigma^\circ(\lambda,\tau))$ is a maximal Cohen--Macaulay $S_\infty$-module, and therefore a maximal Cohen--Macaulay $R_\infty(\lambda,\preceq\tau)$-module since $\dim S_\infty = \dim R_\infty(\lambda,\preceq\tau)$ as can be seen from \cite[Theorem 3.3.4]{KisinPSS}.
Finally, if $\sigma \in \JH(\ovl{\sigma}^\circ(\lambda,\tau))$, then the $R_\infty$-action on $M_\infty(\sigma)$, a subquotient of $M_\infty(\ovl{\sigma}^\circ(\lambda,\tau))$ factors through ${R}_\infty(\lambda,\preceq\tau)/\varpi$.
Since $M_\infty(\sigma)$ is a maximal Cohen--Macaulay $S_\infty/\varpi$-module, it is a maximal Cohen--Macaulay $R_\infty(\lambda,\preceq\tau)/\varpi$ by dimension considerations.
This concludes the proof of (\ref{it:patchSWC:1}).

Proposition \ref{prop:patchobv} applied to $M_\infty(-)$ above implies that $\rbar$ is modular of a weight in $W_\obv(\rbar_p|_{I_{\Qp}})$.
Then \cite[Theorem 4.3.8]{LLL} implies that $W_\obv(\rbar_p|_{I_{\Qp}}) \subset W(\rbar)$.
This implies that $M_\infty(-)$ above is detectable, which establishes (\ref{it:patchSWC:2}).
\end{proof}

\section{A numerical example.}
\label{ex:failure:UB}

In this appendix, we work out a numerical example where the polynomial $P$ appearing in item \eqref{it:stack_local_model:2} of Theorem \ref{thm:stack_local_model} is made explicit.
For our example, we will choose $n=3$, $\cJ$ is a singleton, $\lambda=(3,1,0)$ and $\tld{z}=(23)t_{(2,1,1)}$. %

We first recall the scheme $\cU(\tld{z})^{\det,\leq 0}\to X=\bA^1$ from \S  \ref{sec:affine:charts}.
By Proposition \ref{prop:explicit_affine_chart}, the universal matrix $A\in \cU(\tld{z})^{\det,\leq 0}$ has the form 
\[
\begin{pmatrix}
(v-t)^2+d_{11}(v-t)+c_{11}&c_{12}&c_{13}\\
v(d_{21}(v-t)+c_{21})&c_{22} &(v-t)+c_{23}
\\v(d_{31}(v-t)+c_{31})&v&(v-t)d_{33}+c_{33}
\end{pmatrix}.
\]
so that $\cU^{\det,\leq 0}(\tld{z})$ is the quotient of $\Z[t,c_{11},d_{11},c_{12},c_{13},c_{21},d_{21},c_{22},c_{23},d_{31},c_{31},c_{33},d_{33}]$ subject to the equation 
\[\det A=-(v-t)^4.\]
The affine scheme $\cU(\tld{z})\cap \cS_X(\lambda)$ is obtained by imposing divisibility conditions of the minors on the universal matrix $A$ corresponding to $\lambda$ (and taking the underlying reduced subscheme).

We now turn to the universal monodromy condition \eqref{eq:universalnabla} as in \S \ref{subsec:UMLM}. In fact, we will work with its simplified version as explained in Remark \ref{rmk:setting_last_0}, so that our $\bf{a}\in\bA^3$ always belongs to the $\bA^2$ where the last coordinate is $0$, i.e. $\bf{a}=(a,b,0)$. In addition, we will only work over the open locus $V=\Spec \Z[a,b][\frac{1}{P(a,b)}]\subset \bA^2$ where 
\[P(a,b)=7!b(b-1)(a-1)(a-2)(a-b)(a-b-1)(a-b-2)\]
is invertible. This turns out to substantially simplify our considerations below, and is enough for our purposes, as any specialization we will eventually consider always occurs in $V$, due to the fact that the inertial types we consider will need to be at least $2$-generic.

Let $\cU(\tld{z},\lambda,\nabla)=(\cU(\tld{z})\times \bA^2) \cap \cM_X(\lambda,\nabla)$, an open affine of (the simplified variant of) $\cM_X(\lambda,\nabla)$.

\begin{prop}\label{prop:numerical_base_space}
\begin{enumerate} 
\item \label{it1:prop_numerical_base_space}
The scheme $\cU(\tld{z},\lambda,\nabla)\times _{\bA^2}V$ is represented by the quotient of the ring $\Z[t,a,b,c_{12},c_{13},d_{21},c_{22},d_{31},d_{33}][\frac{1}{P(a,b)}]$ by the ideal generated by
\begin{eqnarray*}
&&(a-2)c_{13}c_{22}-(a-b-2)c_{12}c_{22}d_{33}+c_{12}t(a-b),\\
&&b(a-1)(a-2)c_{13}d_{21}+(a-b-1)(a-1)(a-2)c_{12}d_{31}-b(a-1)(a-b-2)c_{22}d_{33}+\\
&&\quad t\big((a-b-1)(a-2)+(a-b)(a-1)b-2(a-1)(a-2)\big),\\
&&(a-1)(a-2)\big(c_{12}d_{21}d_{33}-c_{13}d_{21}-c_{22}d_{33}+t\big)+t(a-2)+(a-1)\big(-t(a-b)+(a-b-2)c_{22}d_{33}\big).
\end{eqnarray*}
\item
The irreducible components of $\cU(\tld{z},\lambda,\nabla)\times_{X\times \bA^2}  (\{0\}\times V)$ are given by
\begin{align*}
&(c_{22},c_{12}d_{33}-c_{13},bd_{21}d_{33}+(a-b-1)d_{31})
\\
&(c_{22},d_{21},c_{12})
\\
&(d_{31},c_{22},d_{21})\\
&(c_{22},c_{13},c_{12})\\
&(d_{33},c_{13},c_{12})\\
&(d_{33},d_{31},c_{13})\\
&(d_{31},c_{12}d_{21}-c_{22},(a-b-2)c_{12}d_{33}+(-a+2)c_{13})
\end{align*}
\item \label{it:prop_numerical_base_space}$\cU(\tld{z},\lambda,\nabla)\times _{\bA^2}V$ is an irreducible complete intersection, and 
$\cU(\tld{z},\lambda,\nabla)\times _{\bA^2}V\to X\times_\Z V$ is flat.
\item 
\label{it4:prop:numerical_base_space}
Let $H$ be the ideal generated by the $3\times 3$ minors of Jacobian matrix of $\cO(\cU(\tld{z},\lambda,\nabla)\times _{\bA^2}V)$ relative to $\Z[t,a,b][\frac{1}{P(a,b)}]$ with respect to the presentation \eqref{it1:prop_numerical_base_space}. Then $t^3\in H$.
\end{enumerate}

\end{prop}
\begin{proof} Let $R$ be the ring given by the presentation in the first item. We first observe that the equations in the first item indeed hold in 
$\cU(\tld{z},\lambda,\nabla)\times _{\bA^2}V$, and that the monodromy condition solves the variables $c_{11},d_{11},c_{21},c_{23},c_{31},c_{33}$ in terms of $t,c_{12},c_{13},d_{21},c_{22},d_{31},d_{33}$ (this uses the fact that its ring of functions if $t$-torsion free and that $P(a,b)$ is invertible). Thus, $\cU(\tld{z},\lambda,\nabla)\times _{\bA^2}V$ is a closed subcheme of $\Spec R$. Note that over $X^0$, this closed subscheme is all of $\Spec R[\frac{1}{t}]$.

Next, we observe that the minimal primes of $\Spec R/t$ are given by the list in the second item. In particular, this shows that each fiber of $\Spec R \to X\times V$ has codimension $3$ in the corresponding fiber of $\Spec\Z[t,a,b,c_{12},c_{13},d_{21},c_{22},d_{31},d_{33}][\frac{1}{P(a,b)}]\to X\times V$.
It follows that $\Spec R$ is a complete intersection, and that $\Spec R \to X\times V$ is flat \cite[\href{https://stacks.math.columbia.edu/tag/00R4}{Tag 00R4}]{stacks-project}. Since $R[\frac{1}{t}]$ is a regular domain (using the corresponding fact for $\cU(\tld{z},\lambda,\nabla)\times _{X\times \bA^2}X^0\times V$ cf Proposition \ref{prop:intersect_open_schubert_variety}), we conclude from the fact that $t$ is regular in $R$ that $\Spec R[\frac{1}{t}]$ is dense in $\Spec R$. It follows that $\cU(\tld{z},\lambda,\nabla)\times _{\bA^2}V=\Spec R$. This finishes the proof of the first three items. The last item follows from by a computation in Macaulay 2.
\end{proof}
We now define $\cU(\tld{z},\lambda,\nabla)^{\nm}$ to be the spectrum of the quotient
of $\Z[t,a,b,W,c_{12},c_{13},d_{21},c_{22},d_{31},d_{33}]$ by the ideal $\tld{I}$ generated by the following polynomials
\begin{eqnarray*}
&&Wc_{22}+td_{21},\\
&&(a-2)b(a-b)Wc_{12}+b(a-2)c_{13}d_{21}+(a-2)(b+1)(a-b-1)c_{12}d_{31}-b(a-b-2)c_{22}d_{33}-\\
&&\quad-t(b(b-a)+(a-2)),\\
&&(a-b-2)c_{12}c_{22}d_{33}-(a-2)c_{13}c_{22}-t(a-b)c_{12},\\
&&bc_{12}d_{21}d_{33}-(a-b)(a-1)Wc_{12}-(a-2)(a-b-1)c_{12}d_{31}-bc_{22}d_{33}+(b-2)t,
\\
&&(a-2)c_{12}d_{21}d_{33}+(a - 2) (b - a)Wc_{12}-(a-2)c_{13}d_{21}-(a - 2) (a - b - 1)c_{12}d_{31}-bc_{22}d_{33}+t(b-2),
\\
&&b(a(a-b)-a+b-1)Wd_{21}d_{33}+b(a(a-b-2)+b+1)d_{21}d_{31}d_{33}+(a - 1) (b - 1) (a - b)W^2+\\
&&\quad+(a - b - 1) (a (b - 1) - b)Wd_{31},
\\
&&b(a-2)c_{13}d_{21}d_{33}-(a - 2) (a - 1) (b - 1) (a - b)Wc_{13}-(a - 2) (a - b - 1) (a (b - 1) - b)c_{13}d_{31}+\\
&&\quad+b(-a+b+2)c_{22}d_{33}^2-tb(-a+b)d_{33}.
\end{eqnarray*}

We have the natural map $\cU(\tld{z},\lambda,\nabla)^{\nm}\to \cU(\tld{z},\lambda,\nabla)$ which is finite and birational, and hence identifies the former as a partial normalization of the latter (for the birationality, we note that the map is an isomorphism after inverting $c_{22}$, and in fact $\cO(\cU(\tld{z},\lambda,\nabla)^{\nm})$ is the subring $\cO(\cU(\tld{z},\lambda,\nabla))[\frac{-td_{21}}{c_{22}}]$ in the fraction field of $\cO(\cU(\tld{z},\lambda,\nabla))$.)
\begin{prop} \label{prop:numerical_base_change}
Suppose we are given $s:\Spec \cO\to X\times V \subset X\times \bA^2$, correspoding to $(-p,a,b)\in \cO^3$.
\begin{enumerate}
\item The base change $U(\tld{z},\lambda,\nabla_{(a,b,0)})=\cU(\tld{z},\lambda,\nabla)\times_{X\times \bA^2,s}\Spec \cO$ is $\cO$-flat.
\item \label{it2:prop_numerical_base_change}
The base change $U(\tld{z},\lambda,\nabla_{(a,b,0)})^{\nm}=\cU(\tld{z},\lambda,\nabla)^{\nm}\times_{X\times \bA^2,s}\Spec \cO$ is $\cO$-flat and normal.

In particular the base changed map $U(\tld{z},\lambda,\nabla_{(a,b,0)})^{\nm}\to U(\tld{z},\lambda,\nabla_{(a,b,0)})$ is the normalization map.
\item The pullback of each irreducible component of $U(\tld{z},\lambda,\nabla_{(a,b,0)})_{\F}$ along $U(\tld{z},\lambda,\nabla_{(a,b,0)})^{\mathrm{nm}}_{\F}\rightarrow U(\tld{z},\lambda,\nabla_{(a,b,0)})_{\F}$ decomposes into irreducible components according to Table \ref{table:Num_ex}
\end{enumerate}

\end{prop}
\begin{rmk} It follows from the first item that $U(\tld{z},\lambda,\nabla_{(a,b,0)})$ is the intersection of $M(\lambda,\nabla_{(a,b,0)})$ (cf Definition \ref{defn:alg:MLM}) with the affine open $U(\tld{z},\leql)\subset M(\leql)$ (cf \eqref{eq:chartsforlambda}). 
\end{rmk}
\begin{proof}
The first item immediately follows from Proposition \ref{prop:numerical_base_space}\eqref{it:prop_numerical_base_space}, while the last item is a direct computation in Macaulay 2.

We now establish the second item.
Let $I$ be the image of (the base change of) $\tld{I}$ inside  
$\F[W,c_{12},c_{13},d_{21},c_{22},d_{31},d_{33}]$ under the natural mod-$\varpi$ reduction map.
Using the fact that $7!b(b-1)(a-1)(a-2)(a-b)(a-b-1)(a-b-2)\in \F^\times$, we verify by running Buchberger's algorithm that $I$ admits the following Groebner basis with respect to the monomial order on $\F[W,c_{12},c_{13},d_{21},c_{22},d_{31},d_{33}]$ given by $W>c_{12}>c_{13}>d_{21}>c_{22}>d_{31}>d_{33}$:
\begin{eqnarray*}
&&Wc_{22},\\ &&
(a - 1) (a - 2) bc_{13}d_{21}+(a - 2) (a - 1) (a - b - 1)c_{12}d_{31}-(a - 1) b (a - b - 2)c_{22}d_{33},\\ &&
-(a - 2) (a-1) (a - b)Wc_{12}-(a - 2) (a-1) (a - b - 1)c_{12}d_{31},\\ &&
-(a - 1) b (-a + b + 1)c_{22}d_{31}d_{33},\\ &&
(a - 2) (a - 1) (a - b - 1)c_{12}d_{31}d_{33}+(a - 2) (a - 1)^2 (b - 1) (a - b)Wc_{13}+
\\&&\qquad+
(a - 2) (a - 1) (a - b - 1) (a (b - 1) - b)c_{13}d_{31},\\ &&
(a-b-2)c_{12}c_{22}d_{33}-(a-2)c_{13}c_{22},\\ &&
bc_{12}d_{21}d_{33}-(a - 1) (a - b)Wc_{12}-(a - 2) (a - b - 1)c_{12}d_{31}-bc_{22}d_{33},\\  &&
b ((a-b)(a-1)- 1)Wd_{21}d_{33}+(a - 1) b (a - b - 1)d_{21}d_{31}d_{33}+\\&&\qquad+(a - 1) (b - 1) (a - b)W^2+(a - b - 1) (a (b - 1) - b)Wd_{31},\\ &&
(a - 2) (a - 1) (a - b - 1)c_{13}c_{22}d_{31},\\ &&
(a - 2) (a - 1) (a - b - 1)c_{12}c_{22}d_{31},\\ &&
(a - 2) (a - 1)^2 (b - 1)(a - b)^2
W^2c_{13}+(a - 2) (a - 1) (2 a (b - 1) - 2 b + 1) (a - b - 1) (a - b)W
      c_{13}d_{31}+\\&&\qquad+(a - 2) (a - 1) (a (b - 1) - b) (-a + b + 1)^2c_{13}d_{31}^2,
      \end{eqnarray*}
and the leading monomial for each polynomial is the left-most term, except when $(a-b)((a-1)-1)$ vanishes the leading term of the $8$th generator is its second monomial (since its first monomial vanishes in this case).
 
A computation in Macaulay2 shows that the monomial scheme defined by the ideal of leading terms of $I$ 
\[\Spec \F[W,c_{12},c_{13},d_{21},c_{22},d_{31},d_{33}]/lead(I)\]
is Cohen-Macaulay. Since there is a flat Groebner degeneration from 
$U(\tld{z},\lambda_{(a,b,0)}),\nabla)^{\nm}_\F$ to this monomial scheme, we conclude that 
$U(\tld{z},\lambda,\nabla_{(a,b,0)})^{\nm}_\F$ is Cohen-Macaulay.

Next, we compute that the irreducible components of the special fiber of $U(\tld{z},\lambda,\nabla_{(a,b,0)})^{\mathrm{nm}}_\F$ are given by
\begin{eqnarray*}
&&(c_{22},(a-b)W+(a-b-1)d_{31},c_{12}d_{33}-c_{13},bd_{21}d_{33}+(a-b-1)d_{31},b
      c_{13}d_{21}+(a-b-1)c_{12}d_{31}),
      \\&&(d_{31},c_{22},d_{21},W),
      \\&&(d_{33},d_{31},c_{13},W),
      \\
      &&(d_{33},c_{13},c_{12},W), 
      \\
      &&(d_{31},W,c_{12}d_{21}-c_{22},(a-b-2)c_{12}d_{33}+(-a+2)
      c_{13},(a-2)c_{13}d_{21}+(-a+b+2)c_{22}d_{33}),
      \\
      &&(c_{22},d_{21},c_{12},(a - 1) (b - 1) (a - b)W+(a - b - 1) \big(a (b - 1) - b\big)d_{31}),
      \\
      &&(c_{22},c_{13},c_{12},b \big(a (a - b - 1) + b - 1\big)Wd_{21}d_{33}+(a - 1) b (a - b - 1)d_{21}d_{31}d_{33}+(a - 1) (b - 1) (a - b)W^2+\\&&\qquad+(a - b - 1) \big(a (b - 1) - b\big)Wd_{31}).
\end{eqnarray*}
From this, we see by inspection that $U(\tld{z},\lambda,\nabla_{(a,b,0)})^{\nm}_\F$ is generically reduced, and since it is also $S_1$ (since it is Cohen-Macaulay), we conclude it is reduced.

Next we show that $U(\tld{z},\lambda,\nabla_{(a,b,0)})^{\nm}$ is $\cO$-flat. For this, we first observe that $U(\tld{z},\lambda,\nabla_{(a,b,0)})^{\nm}$ is topologically flat over $\cO$. Now, if $U(\tld{z},\lambda,\nabla_{(a,b,0)})^{\nm}$ were not $\cO$-flat, we can find a global function $g$ which is not divisible by $\varpi$ and is $\varpi$-power torsion. Then $g$ must be nilpotent by topological flatness, but its reduction mod $\varpi$ then produces a non-zero nilpotent global function on $U(\tld{z},\lambda,\nabla_{(a,b,0)})^{\nm}_\F$, a contradiction.

Finally, since $U(\tld{z},\lambda,\nabla_{(a,b,0)})^{\nm}$ is $\cO$-flat with reduced special fiber, $U(\tld{z},\lambda,\nabla_{(a,b,0)})^{\nm}$ is normal \cite[Proposition 8.2]{PZ}.

\end{proof}

\begin{table}[H]
\captionsetup{justification=centering}
\caption[Foo content]{\textbf{Irreducible components of fibers
}
}
\label{table:Num_ex}
\centering
\adjustbox{max width=\textwidth}{
\begin{tabular}{| c | c | }
\hline
\hline
&\\
Irreducible components of $U(\tld{z},\lambda,\nabla_{(a,b,0)})_{\F}$& Irreducible components of the preimage\\
&\\
\hline
&\\
$(c_{22},c_{12}d_{33}-c_{13},bd_{21}d_{33}+(a-b-1)d_{31})$&$
\begin{array}{ll}(c_{22},(a-b)W+(a-b-1)d_{31},c_{12}d_{33}-c_{13},bd_{21}d_{33}+(a-b-1)d_{31},bc_{13}d_{21}+(a-b-1)c_{12}d_{31}), \\
(c_{22},c_{13},c_{12},(b-1)W+(-a+b+1)d_{31},bd_{21}d_{33}+(a-b-1)d_{31})\end{array}$\\
&\\
\hline
&\\
$(c_{22},d_{21},c_{12})$&$\begin{array}{ll}
(c_{22},d_{21},c_{12},(a - 1) (b - 1) (a - b)W+(a - b - 1) (a (b - 1) - b)d_{31})\\ (c_{22},d_{21},c_{13},c_{12},W)\end{array}$\\
&\\
\hline
&\\
$(d_{31},c_{22},d_{21})$&$(d_{31},c_{22},d_{21},W)$\\
&\\
\hline
&\\
$(c_{22},c_{13},c_{12})$&$(c_{22},c_{13},c_{12})$\\
&\\
\hline
&\\
$(d_{33},c_{13},c_{12})$&$\begin{array}{ll}(d_{33},c_{13},c_{12},W),\\
(d_{33},c_{22},c_{13},c_{12},(a - 1) (b - 1) (a - b)W+(a - b - 1) (a (b - 1) - b)d_{31})\end{array}$\\
&\\
\hline
&\\
$(d_{33},d_{31},c_{13})$&$(d_{33},d_{31},c_{13},W)$\\
&\\
\hline
&\\
$(d_{31},c_{12}d_{21}-c_{22},(a-b-2)c_{12}d_{33}+(-a+2)c_{13})$&$\begin{array}{ll}
(d_{31},W,c_{12}d_{21}-c_{22},(a-b-2)c_{12}d_{33}+(-a+2)c_{13},(a-2)c_{13}d_{21}+(-a+b+2)c_{22}d_{33}),
\\
(d_{31},c_{22},c_{13},c_{12},b (a (a - b - 1) + b - 1)d_{21}d_{33}+(a - 1) (b - 1) (a - b)W)
\end{array}$\\
&\\
\hline
\hline
\end{tabular}}
\caption*{
In each of the entries of the left column, the first ideal corresponds to the unique maximal dimensional component of the pre-image.
}
\end{table}

\begin{cor}\label{cor:numerical_unibranch} Suppose $(a,b)\in \cO^2$ such that 
\[b(b-1)(a-1)(a-2)(a-b)(a-b-1)(a-b-2)\in \cO^\times\]
and $p>7$. Then:
\begin{enumerate}
\item $U(\tld{z},\lambda,\nabla_{(a,b,0)})$ is unibranch at $\tld{z}\in U(\tld{z},\lambda,\nabla_{(a,b,0)})(\F)$.
\item $U(\tld{z},\lambda,\nabla_{(a,b,0)})$ is \emph{not} unibranch on a Zariski dense set of points on the irreducible component of its special fiber given by
\[(c_{22},c_{13},c_{12})\].
\end{enumerate}
\end{cor}
\begin{proof}
Let $\pi:U(\tld{z},\lambda,\nabla_{(a,b,0)})^{\nm}\to U(\tld{z},\lambda,\nabla_{(a,b,0)})$. Then $\pi$ is the normalization map by Proposition \ref{prop:numerical_base_change}\eqref{it2:prop_numerical_base_change}.

The first item immediately follows from the fact that $\pi^{-1}(\tld{z})$ is set-theoretically a singleton.

For the second item, let $C$ be the irreducible component of $U(\tld{z},\lambda,\nabla_{(a,b,0)})_\F$ cut out by the ideal $(c_{22},c_{13},c_{12})$.
Then according to Table \ref{table:Num_ex}, at the set theoretic level the map $\pi:\pi^{-1}(C)\to C$ identifies with
\begin{align*}
Z& \to \bA^3_\F\\
(d_{21},d_{31},d_{33},W)&\mapsto (d_{21},d_{31},d_{33}).
\end{align*}
where $Z\subset  \bA^4_\F$ is cut out by
\[
\text{\Small{$\frac{b((a-b)(a-1)- 1)}{(a - 1) (b - 1) (a - b)}Wd_{21}d_{33}+\frac{b (a - b - 1)}{(b - 1) (a - b)}d_{21}d_{31}d_{33}+\frac{(a - b - 1) (a (b - 1) - b)}{(a - 1) (b - 1) (a - b)}Wd_{31}+W^2$}}.
\]

This map is a double cover of $\bA^3_\F$ by an irreducible quadric, and hence is generically finite \'{e}tale of degree $2$. In particular, any point outside the branch locus will not be (geometrically) unibranch.
\end{proof}
\begin{cor} \label{cor:numerical_main_theorem}Let $n=3$, $K=\Qp$, $\lambda=(3,1,0)$ and $\tld{z}=(23)t_{(2,1,1)}$. 
Let $(s,\mu)$ be a $5$-generic lowest alcove presentation of a tame inertial type $\tau$.
Then both conclusions of Theorem \ref{thm:stack_local_model} holds when $\mu$ is $10$-deep in $C_0$, i.e.~with the polynomial in \emph{loc.cit.}~ taken to be 
$P(X,Y,Z)=\prod_{m=0}^{10} \big((X-Y-m)(Y-Z-m)(Z-X-m)\big)$

In particular, if $\mu$ is $10$-deep in $C_0$, then
\[\tld{\cX}^{\lambda,\tau}(\tld{z})\cong \tld{U}(\tld{z},\lambda, \nabla_{\bf{a}_\tau})\]
\end{cor}
\begin{proof} For the conclusion of the first part of Theorem \ref{thm:stack_local_model}, we need to choose the polynomial $P(X,Y,Z)$ to guarantee the Elkik approximation argument goes through. By Propositions \ref{prop:Elkik} and \ref{prop:monodromy_approximation} we need $\mu$ to be $m$-deep in $C_0$ for any $m$ such that
\[m-6+3>6\]
since the integer $r$ in Propositions \ref{prop:Elkik} is $3$ by Proposition \ref{prop:numerical_base_space}\eqref{it4:prop:numerical_base_space} and $h_{\lambda}=3$. In other words, we need $m\geq 10$, leading to the polynomial $P$ in the statement. However, in view of Corollary \ref{thm:stack_local_model}, this choice of $P$ already guarantees the unibranch property needed for the second part of Theorem \ref{thm:stack_local_model}.

Finally the last statement follows from the fact that $\tld{z}$ is not $\lambda'$-admissible for any $\lambda'<\lambda$.
\end{proof}
\begin{rmk} Corollary \ref{cor:numerical_unibranch} has only been stated when $\cJ$ is a singleton. However, it easily generalizes to the case of general $\cJ$ by taking products: the essential point is that in our situation taking products preserves the property of being reduced, and hence the product version of $U(\tld{z},\lambda,\nabla_{(a,b,0)})^{\nm}$ is still the normalization of the product version of $U(\tld{z},\lambda,\nabla_{(a,b,0)})$. In particular, the generelization of Corollary \ref{cor:numerical_main_theorem} to the case $K$ being a general unramified extension of $\Qp$ holds.
\end{rmk}
\clearpage{}%

\newpage
\bibliography{Biblio}
\bibliographystyle{amsalpha}

\end{document}